%% file: dis-graphs.tex
\documentclass{dis}

\usepackage{latexsym, bm}
\usepackage{amsmath, amssymb, amsfonts, amsthm, amscd}
\usepackage{cases}
\usepackage{enumerate, paralist}
\usepackage{appendix}

\usepackage{graphicx}
\usepackage{subfigure}

\hypersetup{
	hidelinks = true,
	colorlinks=true,
	linkcolor=blue,
	anchorcolor=blue,
	citecolor=blue,
	filecolor=blue,
	urlcolor=blue,
	menucolor=blue,
}

\usepackage[msc-links, alphabetic, abbrev, nobysame]{amsrefs}

\usepackage{aliascnt}

\theoremstyle{theorem}
\newtheorem{thm}{Theorem}[chapter]

\newtheorem*{thm*}{Theorem}

\newaliascnt{corollary}{thm}
\newtheorem{cor}[corollary]{Corollary}
\aliascntresetthe{corollary}

\newaliascnt{lemma}{thm}
\newtheorem{lem}[lemma]{Lemma}
\aliascntresetthe{lemma}

\newaliascnt{sublemma}{thm}
\newtheorem{slem}[sublemma]{Sublemma}
\aliascntresetthe{sublemma}

\theoremstyle{definition}
\newaliascnt{definition}{thm}
\newtheorem{defi}[definition]{Definition}
\aliascntresetthe{definition}

\newaliascnt{example}{thm}
\newtheorem{exa}[example]{Example}
\aliascntresetthe{example}

\newaliascnt{remark}{thm}
\newtheorem{rmk}[remark]{Remark}
\aliascntresetthe{remark}

\theoremstyle{theorem}
\newtheorem{thmA}{Theorem}

\theoremstyle{theorem}

\newaliascnt{corI}{thmI}

\aliascntresetthe{corI}

\usepackage{etoolbox}
\makeatletter
\patchcmd{\hyper@makecurrent}{%
	\ifx\Hy@param\Hy@chapterstring
	\let\Hy@param\Hy@chapapp
	\fi
}{%
	\iftoggle{inappendix}{%
		\@checkappendixparam{chapter}%
		\@checkappendixparam{section}%
		\@checkappendixparam{subsection}%
		\@checkappendixparam{subsubsection}%
		\@checkappendixparam{paragraph}%
		\@checkappendixparam{subparagraph}%
	}{}%
}{}{\errmessage{failed to patch}}

\newcommand*{\@checkappendixparam}[1]{%
	\def\@checkappendixparamtmp{#1}%
	\ifx\Hy@param\@checkappendixparamtmp
	\let\Hy@param\Hy@appendixstring
	\fi
}

\newtoggle{inappendix}
\togglefalse{inappendix}
\apptocmd{\appendix}{\toggletrue{inappendix}}{}{\errmessage{failed to patch}}
\@ifpackageloaded{appendix}
{
	\apptocmd{\appendices}{\toggletrue{inappendix}}{}{\errmessage{failed to patch}}
	\apptocmd{\subappendices}{\toggletrue{inappendix}}{}{\errmessage{failed to patch}}
}
{}
\makeatother 

\DeclareMathOperator*{\lip}{Lip}
\DeclareMathOperator*{\hol}{\mathrm{Hol}}
\DeclareMathOperator*{\graph}{\mathrm{Graph}}
\DeclareMathOperator*{\id}{\mathrm{id}}
\DeclareMathOperator*{\diam}{\mathrm{diam}}

\allowdisplaybreaks

\numberwithin{equation}{chapter}

\chaptersnewpage

\refnewpage

\begin{document}

\mathclass{Primary 37D10; Secondary 37D30, 57R30, 58B20, 55R65}
\thanks{The author would like to thank Prof. M. W. Hirsch for his useful suggestions. The author is also indebted to the referee and editors who have thoroughly revised the text and provided several helpful suggestions. This paper was part of the author's PhD thesis written under the supervision of Prof. Dongmei Xiao at Shanghai Jiao Tong University. This work was partially supported by NNSF of China (No. 11431008).}

\abbrevauthors{D. Chen}
\abbrevtitle{Existence and Regularity of Invariant Graphs}

\title{Existence and Regularity of Invariant Graphs for Cocycles in Bundles: partial hyperbolicity case}

\author{Deliang Chen}

\address{Department of Mathematics \\
	Shanghai Jiao Tong University \\
	Shanghai 200240, People's Republic of China}

\address{College of Mathematics and Physics \\
	Wenzhou University \\
	Wenzhou 325035, People's Republic of China \\
	E-mail: chernde@wzu.edu.cn}

\maketitledis
\setcounter{tocdepth}{2}
\tableofcontents
\begin{abstract}
We study the existence and regularity of invariant graphs for bundle maps (or bundle correspondences with generating bundle maps motivated by ill-posed differential equations) having some relative partial hyperbolicity on non-trivial and non-locally-compact bundles. The regularity includes (uniform) $ C^0 $ continuity, H\"older continuity and smoothness. To illustrate the power of our results and methods, a number of applications to both well-posed and ill-posed semilinear differential equations and abstract infinite-dimensional dynamical systems are given. These applications include the existence and regularity of different types of invariant foliations (laminations), including strong stable laminations and fake invariant foliations, the existence and regularity of holonomies for cocycles, the $ C^{k,\alpha} $ theorem and decoupling theorem, etc., in a general setting.
\end{abstract}
\makeabstract

\input{sect0.tex}

\input{sect1.tex}
\input{sect2.tex}
\input{sect3.tex}

\input{sect4.tex}

\input{sect5.tex}

\input{sect6.tex}

\begin{appendices}
	\setcounter{equation}{0}
	\input{app.tex}

\input{app0.tex}

\input{app1.tex}
\end{appendices}

\input{ref-graphs.bbl}

\end{document}

%% file: sect0.tex
\chapter{Introduction}

\section{Motivation}
Invariant manifold theory is one of the central topics in the theory of dynamical systems with some hyperbolicity. Classical contributions go back to A. Lyapunov \cite{Lia47} (originally published in Russian, 1892), J. Hadamard \cite{Had01}, \'E. Cotton \cite{Cot11}, O. Perron \cite{Per29}, V. A. Pliss \cite{Pli64}, A. Kelly \cite{Kel67}, N. Fenichel \cite{Fen72}, M. W. Hirsch , C. C. Pugh, M. Shub \cite{HPS77}, Ya. B. Pesin \cite{Pes77}, and others.
For new developments in infinite-dimensional dynamical systems see e.g. \cite{Rue82, Man83, LL10} (non-uniformly hyperbolic case), \cite{Hen81, CL88, BJ89, MS88, DPL88, MR09a} (different types of semilinear differential equations), \cite{BLZ98, BLZ99, BLZ00, BLZ08} (normally hyperbolic manifolds case), too numerous to list here.
Invariant manifolds can exist for some ill-posed differential equations (see e.g. \cite{EW91,Gal93,SS99,dlLla09,ElB12}), although they do not generate semiflows.
For significant applications of invariant manifold theory in ergodic theory, see e.g. \cite{PS97,Via08, AV10, BW10, BY17}.

This paper is part of a program to expand the scope of invariant manifold theory, following the work of Hirsch, Pugh and Shub \cite[Chapters 5--6]{HPS77} with a view towards making it applicable to both \emph{well-posed} and \emph{ill-posed} differential equations and abstract dynamical systems in \emph{non-compact} spaces (including non-locally-compact spaces). Here, we try to extend the classical theory to more general settings at least in the following directions by developing a number of new ideas and techniques.

\vspace{.5em}
\noindent{\textbf{Dynamical systems:}}
We focus on dynamical systems on bundles, i.e., \emph{cocycles}. Related works on invariant manifolds for cocycles are e.g. \cite{Hal61, Yi93, CY94, CL97}; but see also \cite[Chapter 6]{HPS77} and \cite{AV10, ASV13}. In terms of differential equations, there is a natural context we need to study cocycles. That is, when we look into the dynamical properties of a differential equation in a neighborhood of its invariant set, we usually study the linearization of the equation along this set and consider the equation as a linearized equation plus a `small' perturbation. In order to make our results applicable to ill-posed differential equations, we extend the notion of cocycle to \emph{bundle correspondence with generating bundle map} (see \autoref{correspondence} for definitions) originally due to Chaperon \cite{Cha08} and Akin \cite{Aki93}. By using this notion, the difficulty that cocycles may be non-invertible (especially in the infinite-dimensional setting) can be overcome without strain.

\vspace{.5em}
\noindent{\textbf{Bundles:}}
We do \emph{not} require the bundles to be trivial or the base spaces to be (locally) compact. Note that for infinite-dimensional dynamical systems, this situation happens naturally. For the existence results (see \autoref{graph}), the assumption on bundles is only that their fibers are complete metric spaces, which makes the results as general as possible. For the regularity results (see \autoref{stateRegularity}), the situation becomes more complicated, which is one of the major difficulties arising in this paper;
see \autoref{bundle} for how we deal with non-trivial bundles and \autoref{mf} for how we extend compact spaces or Banach spaces to non-compact and non-linear spaces in infinite dimensions; this treatment is essentially due to e.g. \cite{HPS77, PSW12, BLZ99, BLZ08, Cha04, Eld13, Ama15}.

\vspace{.5em}
\noindent{\textbf{Hyperbolicity:}}
We consider dynamical systems with some \emph{relative partial hyperbolicity}. There are many ways to describe hyperbolicity in different settings. Here we adopt the \emph{(A) (B) condition} (see \autoref{defAB} for definitions) originally due to \cite{MS88} (see also \cite{LYZ13, Zel14}) to describe hyperbolicity.
The definition of the (A) (B) condition concerns `dynamical systems' themselves and it can also make sense in the setting of metric spaces (in the spirit of \cite{Cha04, Cha08, Via08, AV10}). However, see \autoref{lem:c1} for the relation between the (A) (B) condition and the classical cone condition, where the latter only works for smooth dynamical systems on smooth manifolds.
As a first step, in this paper we only discuss the relative partial hyperbolicity in the \emph{uniform} sense (especially for the regularity results); our future work (see \cite{Che18d}) will address the non-uniform partial hyperbolicity in detail.

Using the (A) (B) condition has the advantage that we can give a unified approach to invariant manifolds in different hyperbolicity settings.
A more important thing is that some classical hyperbolicity assumptions (or spectral gap conditions) are not satisfied by some dynamics generated by differential equations but the (A) (B) condition can still be verified. Consult the important paper \cite{MS88}, where Mallet-Paret and Sell could successfully verify the (uniform) cone condition based on the principle of spatial average (see also \cite{Zel14}).

In summary, it seems that we first investigate \emph{the existence and regularity of invariant manifolds (or specifically invariant graphs) for bundle maps (or bundle correspondences with generating bundle maps) with some (relative) partial hyperbolicity described by the (A) (B) condition on non-trivial and non-compact bundles}. Due to the general settings, our results have a number of applications to, for instance, not only the classical dynamics generated by well-posed differential equations such as some parabolic PDEs, delay equations, age structured models, etc., but also to ill-posed differential equations such as elliptic PDEs on cylinders, mixed delay equations, and spatial dynamics generated by some PDEs; see \cite{Che18c} for more details. Meanwhile, our results give a unified approach to study the existence and regularity of different types of invariant manifolds and invariant foliations (laminations) including strong (un)stable laminations and fake invariant foliations, the existence and regularity of holonomies for cocycles, and others; see e.g. \autoref{foliations}.

\vspace{.5em}
\noindent{\textbf{Differential equations:}}
In \cite{Che18c}, we studied some classes of semilinear differential equations that can generate cocycles in the well-posed case or cocycle correspondences with generating cocycles in the ill-posed case, which can satisfy the (A) (B) condition (see \autoref{defAB}), so that our main results in \autoref{graph} and \autoref{stateRegularity} (as well as \autoref{foliations}) can be applied.
Our work (together with \cite{Che18c}) is an abstract generalization of \cite{Yi93, CY94, CL97} and the invariant manifolds of an equilibrium for the differential equations considered in \cite{MS88, BJ89, DPL88, MR09a, ElB12}, etc.

\section{Nontechnical overviews of main results, proofs and applications}
Let us give some nontechnical overviews of main results, proofs and applications for the reader's convenience.

\vspace{.5em}
\noindent{\textbf{Existence of invariant graphs.}}
Heuristically, our main results on the existence of invariant graphs for bundle maps (or bundle correspondences with generating bundle maps) may be summarized as follows (in a special case):

\begin{thmA}[Existence] \label{thmA00}
	Let $ H: X \times Y \to X \times Y $ be a bundle map (or a bundle correspondence with generating bundle map $ (F, G) $) over a map $ u $ where $ X, Y $ are bundles with fibers being complete metric spaces and with base space $ M $, and $ u: M \to M $. Let $ i = (i_X, i_Y) $ be a section of $ X \times Y $ which is invariant under $ H $. Assume $ H $ has some relative partial hyperbolicity described by the (A) (B) condition.
	If the spectral condition holds, then the following conclusions hold.
	\begin{enumerate}[(1)]
		\item There is a $ C^{0,1} $-fiber bundle map $ f: X \to Y $ over $ \id $ such that the graph of $ f $ is invariant under $ H $, i.e., $ \graph f_m \subset H^{-1}_m \graph f_{u(m)} $ for $ m \in M $ (or $ f $ satisfies \eqref{mainF}), and $ i \in \graph f $. In some sense, $ f $ is unique.

		\item The graph of $ f $ can be characterized by the asymptotic behavior of $ H $.
	\end{enumerate}
\end{thmA}

Similar results in the `unstable direction' $ Y \to X $ for $ H^{-1} $ also hold.

For more specific and general statements of our three existence results, see \autoref{thmA}, \autoref{thmB} and \autoref{thm:local} in \autoref{invariant}, where the section $ i $ can be non-invariant (originally motivated by \cite{Cha08}).
A characterization of invariant graphs is given in \autoref{properties}. For the `hyperbolic dichotomy' and `hyperbolic trichotomy' cases, see \autoref{thm:hyperTrivial} and \autoref{thm:pnormalTrivial}, respectively. The former can be applied e.g. to \emph{Anosov} dynamical systems (or more generally restrictions of Axiom A diffeomorphisms to hyperbolic basic sets), two-sided shifts of finite type, or more generally \emph{uniformly hyperbolic homeomorphisms} introduced in \cite{Via08} (see also \cite{AV10}). The latter can be applied e.g. to (partial) \emph{normally hyperbolic invariant manifolds} in a \emph{particular} context that the normal bundle of the (partial) normally hyperbolic manifold can embed into a trivial bundle and the dynamics can also extend to this trivial bundle maintaining the `hyperbolic trichotomy' (see \cite{HPS77, Eld13} and \autoref{pnormalhyper}).

In the `hyperbolic trichotomy' case, if $ H $ is an invertible bundle map, then $ H $ can be divided into three parts: a decoupled center part with two stable and unstable parts depending on the center part (see \autoref{decoupling}). This theorem can be regarded as a generalization of the Hartman--Grossman Theorem (see also \autoref{rmk:decoupling}).

Except the description of partial hyperbolicity, our existence results mentioned above are more general than the ones in the classical literature including e.g. \cite{Sta99, Cha04, Cha08, MS88, CY94, CL97, LYZ13} and \cite[Chapter 5]{HPS77} (about plaque families). No assumption on the base space $ M $ makes the existence results applicable to both deterministic and random dynamical systems. Continuity and measurability problems will be regarded as regularity problems; measurability problems will not be considered in this paper (see \cite{Che18d}) and \emph{(uniform) $ C^0 $ continuity}, \emph{H\"older continuity} and \emph{smoothness} problems will be investigated in \autoref{stateRegularity}.

Additionally, our existence results work not only for $ H $ being a cocycle having some hyperbolicity but also for $ H $ being a bundle correspondence with generating bundle map (see \autoref{basicCorr}). Loosely speaking, one in fact needs a subset of $ X \times Y $ (the graph of $ H $, $ \graph H $); then the `dynamical behavior' in some sense is the `iteration' of $ \graph H $. So it is important to give a method of describing $ \graph H $; classically $ \graph H $ is given by the cocycle (bundle map) and in our context it is given by the generating map.
We will introduce the precise notions of \emph{correspondence} and \emph{generating map} in \autoref{correspondence}, which we learned from \cite{Cha08}. We need these notions at least for two reasons.
\begin{asparaitem}[$ \bullet $]
	\item There are some differential equations (see e.g. \cite{EW91,Gal93,SS99,dlLla09,ElB12}) which are ill-posed and therefore cannot generate semiflows so that the classical theory of dynamical systems (especially invariant manifold theory) cannot be applied directly to these equations, but which may define \emph{cocycle correspondences with generating cocycles} (see \cite{Che18c}). So our general results can be applied to more types of differential equations; see \cite{Che18c} for details.

	\item Unlike finite-dimensional dynamics, the dynamics generated by differential equations in Banach space are not usually invertible, so one cannot derive `unstable results' from the proven `stable results'. Since `hyperbolic systems' (invertible or not) can define correspondences with generating maps, using \emph{dual correspondences} (see \autoref{dual}) one only needs to prove results in one direction. This idea was used in \cite{Cha08} to exhibit (pseudo) (un-)stable manifolds (foliations) in a very unified way.
\end{asparaitem}

\vspace{.5em}
\noindent{\textbf{Regularity of invariant graphs.}}
We will show that $ f $ obtained in \autoref{thmA00} has higher regularity including \emph{continuity}, \emph{H\"olderness} and \emph{smoothness}, once (i) more regularity of the bundle $ X \times Y $ (and the base space $ M $), (ii) more regularity of the maps $ u $, $ i $, $ F, G $, (iii) the spectral gap condition and additionally (iv) some technical assumption on the (almost) continuity of the functions in the (A) (B) condition are fulfilled.
Roughly, our main results on the regularity of invariant graphs for cocycles (or bundle correspondences with generating bundle maps) may be stated as follows.

\begin{thmA}[Regularity] \label{thmB00}
	Under the assumptions of \autoref{thmA00}, we have the following statements.
	\begin{enumerate}[(1)]
		\item \label{n1} If the fibers of $ X \times Y $ are Banach spaces and $ F_{m}(\cdot), G_{m}(\cdot) \in C^1 $ for each $ m $, then $ f_m(\cdot) \in C^1 $. Moreover, there is a $ K^1 \in L(\Upsilon_X^V, \Upsilon_Y^V) $ (see \eqref{HVspace}) over $ f $ satisfying \eqref{mainX} and $ D^vf = K^1 $. The case $ f_{m}(\cdot) \in C^{k,\alpha} $ is very similar. See \autoref{lem:leaf1}, \autoref{lem:leaf1a} and \autoref{lem:leafk}.

		Furthermore, if some continuity properties of the functions in the (A) (B) condition and the corresponding spectral gap conditions (which \textbf{differ} from item to item in the following) are assumed, then we get the following results on the regularity of $ f $. Here we omit specific assumptions on $ X, Y, M $ (in order to make sense of the different uniform properties of $ F, G, u $).

		\item \label{n2} Let $ F, G, i, u \in C^0 $. Then $ f \in C^0 $. See \autoref{lem:continuity_f}.

		\item \label{n3} Under the assumption of \eqref{n2}, if the fiber derivatives $ D^v F, D^v G \in C^{0} $, then $ K^1 \in C^0 $. See \autoref{lem:K1leafc}.

		\item \label{n4} If $ F, G \in C^{0,1} $ and $ u \in C^{0,1} $, then $ m \mapsto f_{m} $ is uniformly (locally) H\"older. See \autoref{lem:sheaf}.

		In addition, assume $ i $ is a $ 0 $-section and $ u \in C^{0,1} $. Then we draw the following further conclusions.
		\item \label{n5} If $ F, G \in C^{1,1} $, and a `better' spectral gap condition (than \eqref{n4}) holds, then $ m \mapsto f_{m} $ is also uniformly (locally) H\"older (in a better way). See also \autoref{lem:sheaf}.

		\item \label{n6} If $ m \mapsto DF_{m}(i^1(m)) $, $ DG_{m}(i^1(m)) $ are uniformly (locally) H\"older, where $ i^1(m) = ( i_X(m), i_Y(u(m)) ) $, then the distribution $ Df_m(i_X(m)) $ depends in a uniformly (locally) H\"older fashion on the base points $ m $. See \autoref{lem:base0}.

		\item \label{n7}
		If $ F, G, D^v F, D^v G \in C^{0,1} $, then $ m \mapsto Df_m(\cdot) $ is uniformly (locally) H\"older. See \autoref{lem:baseleaf}.

		\item \label{n8}
		Let $ u \in C^1 $ and $ F, G \in C^{1,1} $. Then $ f \in C^1 $. Moreover, there is a unique $ K \in L(\Upsilon^H_X, \Upsilon^V_Y)  $ (see \eqref{HVspace}) over $ f $ which is $ C^0 $ and satisfies \eqref{mainM} (or more precisely \eqref{mainMFW}) and $ \nabla f = K $ where $ \nabla f $ is the covariant derivative of $ f $. See \autoref{smoothbase}.

		\item \label{n9} Under the assumption of \eqref{n8}, $ x \mapsto \nabla_{m} f_m(x) $ is locally H\"older uniformly for $ m \in M $. See \autoref{lem:holversheaf}.

		\item \label{n10} Under the assumption of \eqref{n8} and $ u \in C^{1,1} $, $ m \mapsto \nabla_{m} f_m(\cdot) $ is uniformly locally H\"older. See \autoref{lem:final} (1).

		\item \label{n11} Under the assumption of \eqref{n8}, if $ F, G \in C^{2,1} $, $ u \in C^{1,1} $, and a `better' spectral gap condition (than \eqref{n10}) holds, then $ m \mapsto \nabla_{m} f_m(\cdot) $ is uniformly locally H\"older (in a better way). See \autoref{lem:final} (2).
	\end{enumerate}
\end{thmA}

For more specific and general statements of the regularity results, see the lemmas we listed above or \autoref{stateRegularity} where some uniform conditions on $ F, G, u, i $ and $ X, Y, M $ can be only around $ u(M) $ (i.e. the \emph{inflowing} case); see also \autoref{thmC}. In \autoref{bounded}, we also give the corresponding regularity results (without proof) when $ i $ is a `bounded' section, which was also studied in \cite{CY94}; this case is a complement to \autoref{thmB00}. In fact, the regularity results do not depend on the existence results (see \autoref{generalized}).

Our regularity results are uniform and recover many classical ones. The fiber smoothness of the invariant graph (i.e. \autoref{thmB00} \eqref{n1}) is well known in different settings including deterministic and random dynamical systems (see e.g. \cite{HPS77, CL97, FdlLM06, LL10}). Reducing to the case that $ M $ consists of one element, it gives a smoothness result on invariant ((strong) stable, center, pseudo-stable, etc.) manifolds of an equilibrium (cf. e.g. \cite{HPS77, Irw80, dlLW95, VvG87}).
Beside items \eqref{n2} \eqref{n3}, more continuity results on $ f $ are given in \autoref{continuityf}; see also \autoref{rmk:con} for a simple application.
The H\"older continuity of $ f $ like item \eqref{n4}, only requiring the Lipschitz continuity of the dynamics, was also discussed in a very special setting in \cite{Cha08, Sta99} (the metric space case) and \cite[Corollary 5.3]{Wil13} (the smooth space case).
Results like item \eqref{n5} are well known in different contexts; see e.g. \cite{HPS77, PSW97, LYZ13}. Results like those in item \eqref{n6} have been re-proved by many authors and go back to D. V. Anosov for hyperbolic systems on compact sets (see e.g. \cite{PSW12}); see also \cite{Has97, HW99}.
Conclusions like items \eqref{n7}--\eqref{n11} were studied e.g. in \cite{HPS77} (the normal hyperbolicity case), \cite{PSW97} (the partial hyperbolicity case; see also \autoref{holfoliation}) and \cite{CY94} (for cocycles in trivial bundles with base spaces being compact Riemannian manifolds). Very recently, in \cite[Lemma 3]{ZZJ14}, the authors obtained a result like item \eqref{n10} in a particular context; see also \cite[Theorem 1.3]{Sta99}.
Higher order smoothness of $ f $ is not deduced in the present paper.

For precise \emph{spectral gap conditions} in items \eqref{n2}--\eqref{n11} (sometimes called \emph{bunching conditions} (see \cite{PSW97}) or \emph{fiber bunching conditions} (see \cite{AV10, ASV13})), which vary with the different regularity assumptions on $ F, G $, look into the lemmas we listed in \autoref{thmB00}. Spectral gap conditions we give are general (due to the general assumptions on $ F, G $), and in some cases they are new. As is well known, the spectral gap conditions are in some sense essential to obtain higher regularity of $ f $; the sharpness of the spectral gap conditions in some cases was studied e.g. in \cite{HW99} (the H\"olderness of the distribution in item \eqref{n6}) and \cite{Zel14} (concerning differential equations).

Precise assumptions on $ X, Y, M, u, i $ are given in \autoref{settingOverview}.
Basically, we try to generalize the base space $ M $ to the non-compact (and non-measurable) case, and the bundles $ X, Y $ to the non-trivial case as well.
See \autoref{examples} for some relations (or examples) between our assumptions on $ X, Y, M $ and \emph{bounded geometry} (\cite{Eic91, Ama15, Eld13}) as well as the manifolds studied in \cite{BLZ99, BLZ08}. As a simple illustration, consider the following classical cases:
\begin{enumerate}[$ \bullet $]
	\item $ X \times Y = M \times X_0 \times Y_0 $, where $ M, X_0, Y_0 $ are Banach spaces;

	\item $ X, Y $ are smooth vector bundles and the base space $ M $ is a smooth compact Riemannian manifold.
\end{enumerate}
The assumptions on $ X, Y, M $ are in order to make sense of the uniform $ C^{k,\alpha} $ continuity of $ F, G, f, u $. Our way of describing the uniform properties of a bundle map, which in some sense is highly classical (see e.g. \cite[Chapter 6]{HPS77}, \cite[Section 8]{PSW12}, \cite[Section 2]{Cha04} and \cite{Eld13,Ama15}), is by using (a particular choice of) local representations of the bundle map with respect to preferred atlases; see \autoref{bundle} and \autoref{mf}. The uniform property of a bundle map is defined to be the uniform property of local representations. Here we mention that since our local representations are specially selected (see \autoref{def:regular}), the conclusions (and conditions) in \autoref{leaf} to \autoref{HolderivativeBase} are formally weak. But these will be strengthened as e.g. in \cite{HPS77,PSW12} once certain uniform properties of the transition maps in preferred atlases are assumed (which means that the higher regularity of the bundles is needed); see \autoref{lipbase} and \autoref{bundle} for details.
The ideas of how to extend the bundles (and base spaces) to more general settings are presented in \autoref{prelimiaries}.

The greatest difficulty in applying our existence and regularity results is how to verify the hyperbolicity condition, but this is not the content of this paper. We refer the readers to see \cite{CL99} for a detailed study of the spectral theory of linear cocycles in the infinite-dimensional setting; see also \cite{LL10, SS01, LP08} in the cocycle or evolution case, and \cite{LZ17, Zel14, NP00} in the `equilibrium' case. In \autoref{relAB0} and \autoref{relAB1} we given some relations between the (A) (B) condition and some classical hyperbolicity conditions; see also \cite{Che18c} for a discussion on the dichotomy and the (A) (B) condition in the context of differential equations.

\vspace{.5em}
\noindent{\textbf{A nontechnical overview of proofs.}}
In the theory of invariant manifolds, there are two fundamental methods: the \emph{Hadamard graph transform method} (due to \cite{Had01}) and the \emph{Lyapunov-Perron method} (due to \cite{Lia47, Cot11, Per29}), where the former is more geometric and the latter more analytic. These two methods have proved successful in establishing the existence and regularity of invariant manifolds, and in many cases both are feasible. We refer the readers to e.g. \cite{VvG87, CL88, CY94, CL97, Cha04, Cha08, MR09a, LL10} where the Lyapunov--Perron method was applied, and \cite{HPS77, MS88, PSW97, BLZ98, Sta99, LYZ13, Wil13} where the graph transform method was used.

In this paper, we employ the graph transform method to prove all the existence and regularity results. Intuitively, this method is to construct a \emph{graph transform} which for a given graph, say $ \mathcal{G}_1 $, yields a unique graph $ \mathcal{G}_2 $ such that $ \mathcal{G}_2 \subset H^{-1} \mathcal{G}_1 $; see the proof of \autoref{thmA00} for an example, and \autoref{graphTrans}. Unfortunately, usually this cannot be done for general $ H $ and arbitrary graphs. In our circumstances, $ H $ requires some hyperbolicity and the graphs, for our purpose, should be represented as functions $ X \to Y $. Now for the invariant graph $ \mathcal{G} $ of $ H $, one gets an equation like \eqref{mainF}. Applying the Banach Fixed Point Theorem to the graph transform on some appropriately chosen space, one can get the desired invariant graph. It seems that our approach to \autoref{thmA00} (i.e. \autoref{thmA}) based on the graph transform method is uniform, elementary and concise.

A more challenging task is to get the higher regularity of $ f $. Our strategy for proving the H\"olderness and smoothness of $ f $ is from `fiber-regularity' to `base-regularity' and from `low regularity' to `high regularity', that is, step by step in the order of items \eqref{n1}--\eqref{n11} in \autoref{thmB00}. (It is not so easy to prove the $ C^1 $ continuity of $ f $ if items \eqref{n1}--\eqref{n7} in \autoref{thmB00} are not obtained first.)
For the H\"older regularity of $ f $, we use an argument which at least in some special cases is classical (cf. e.g. \autoref{appexb} as a motivation), but very \emph{uniform} (see \autoref{argumentapp}). However, due to our general settings (especially no uniform boundedness of the fibers is assumed and the hyperbolicity is described in a relative sense), the expression of spectral gap conditions is a little complicated (see \autoref{absgc}) and some preliminaries are needed (see \autoref{Appbb}).
To prove the smoothness of $ f $, we employ a popular argument (see e.g. \autoref{lem:fdiifx} as an illustration); that is, one needs first to find the `variant equation' (see e.g. \eqref{mainX} or `\eqref{mainM}') satisfied by the `derivative' of $ f $ and then one tries to show the solution of this `variant equation' is indeed the `derivative' of $ f $. Again the graph transform method is used to solve the `variant equation'.
For the fiber-smoothness of $ f $ (i.e. \autoref{thmB00} \eqref{n1}), the `variant equation' is easy to find (i.e. \eqref{mainX} or ($ \dag $) in the proof of \autoref{lem:leafk} for the higher order fiber-smoothness), \emph{whereas} for the base-smoothness of $ f $ (i.e. \autoref{thmB00} \eqref{n8}) the situation becomes intricate.
We introduce an additional structure of a \emph{connection} (see \autoref{connections} for an overview) to give a derivative of a $ C^1 $ bundle map respecting base points, i.e., the \emph{covariant derivative} of $ f $. (Now `\eqref{mainM}' has a precise meaning, i.e., \eqref{mainMFW}.)
This approach to prove the smooth regularity of $ f $ may be new and different from the classical method (e.g. \cite{HPS77, PSW97}). Also, the proof given in this paper in some sense simplifies the classical one.

\vspace{.5em}
\noindent{\textbf{Application I: abstract dynamical systems.}}
In \autoref{section} and \autoref{foliations}, we give some applications of our existence and regularity results to abstract dynamical systems; more applications will appear in our future work (see e.g. \cite{Che18b, Che18}).
A direct application of our main results is the $ C^{k,\alpha} $ section theorem (\cite[Chapter 3, Chapter 6]{HPS77}, \cite[Theorem 3.2]{PSW97} and \cite[Theorem 10]{PSW12}); see \autoref{thm:sect1} and \autoref{thm:sect2}.

A second application is some results on invariant foliations contained in \autoref{foliations}, where the infinite-dimensional and non-compact settings are established and the dynamics is not necessarily invertible. The global version of invariant foliations for bundle maps, i.e., \autoref{thm:bundlemaps}, was re-proved by many authors (see e.g. \cite{HPS77, Sta99, Cha04, Cha08}). A basic application of different types of invariant foliations is to decouple the systems (see \autoref{decoupling}).
However, the local version of invariant foliations like fake invariant foliations, first introduced in \cite{BW10} (see also \cite{Wil13}), is not so well known. In \autoref{fake}, we exhibit fake invariant foliations in the infinite-dimensional setting (see \autoref{thm:fake1} and \autoref{thm:fake2}). We mention that unlike center foliations for partially hyperbolic systems which might not exist, fake invariant foliations always exist but are only locally invariant.

Strong (un)stable laminations (and foliations) were extensively studied e.g. by \cite{HPS77, Fen77, BLZ00, BLZ08} (for the normally hyperbolic case) and \cite{PSW97} (for the partially hyperbolic case). In \autoref{sslaminations}, we study the existence and H\"older continuity of strong stable laminations in the metric spaces setting; see \autoref{thm:ssl} and \autoref{corr:hol}, which also give the \emph{$ s $-lamination} (resp. $ u $-lamination) for maps (resp. invertible maps) (see e.g. \cite[Section 4.1]{AV10}). The corresponding result in the smooth spaces setting is given in \autoref{corr:smoothcase}. In general, one cannot expect a strong stable foliation is $ C^1 $; but if some bunching condition is assumed, a strong stable foliation is $ C^1 $ inside each leaf of the center-stable foliation (if it exists) (see e.g. \cite{PSW97, PSW00}). There is a \emph{difficulty} that usually the center-stable foliation is again non-smooth, which was dealt with in \cite{PSW97, PSW00}.
We re-prove this result by using our regularity results (see \autoref{holfoliation} for details).
A similar argument is also used in the proof of the regularity of fake invariant foliations (see \autoref{thm:fake2}) and strong stable foliations in center-stable manifolds (see \cite{Che18b}).

The existence and regularity of \emph{holonomies (over laminations) for cocycles} (see \autoref{holonomy}) are also discussed in \autoref{holonomyL}; these are mostly direct consequences of our main results (\autoref{thm:hol} and \autoref{corr:hol2}). The holonomies for cocycles were studied in \cite[Section 2]{Via08} and \cite[Section 5]{AV10}; see also \cite{ASV13}. We give a vast generalization (in a uniform way) so that one could apply it to infinite-dimensional dynamical systems.

\vspace{.5em}
\noindent{\textbf{Application II: differential equations.}}
To apply our existence and regularity results in \autoref{graph} and \autoref{stateRegularity} (as well as \autoref{foliations}) to semilinear differential equations, from the abstract point of view, one needs to show that \emph{differential equations can generate cocycles (in the well-posed case) or cocycle correspondences with generating cocycles (in the ill-posed case) satisfying the (A) (B) condition}. In \cite{Che18c}, we dealt with the relationship between the \emph{dichotomy} (or more precisely the exponential dichotomy of differential equations) and the (A) (B) condition. See \cite{Che18c} for applications of our main results (as well as \cite{Che18b, Che18}) to both well-posed and ill-posed differential equations.

\section{Structure of this paper}

The related notions of correspondence having a generating map are introduced in \autoref{correspondence} and the definition of the (A) (B) condition with its relevance is given in \autoref{hyper}.
In \autoref{graph}, we give our existence results together with their corollaries.
\autoref{prelimiaries} contains some preliminaries on bundles and manifolds such as the description of uniform $ C^{k,\alpha} $ continuity of bundle maps on appropriate types of bundles, casting them in a light suitable for our purpose to set up the regularity results.
In \autoref{stateRegularity}, we give the regularity results and their proofs.
Some applications of our main results are given in \autoref{application}. \autoref{Appbb} contains a key argument in the proof of our regularity results. \autoref{bundleII} continues \autoref{bundle}. \autoref{examples} provides some examples related to our (uniform) assumptions on manifolds and bundles.
Some miscellaneous topics such as a fixed point theorem under minimal conditions, Finsler manifolds in the sense of Neeb--Upmeier, Lipschitz characterizations in length spaces and bump functions (and blid maps), are provided in \autoref{misc}.

\vspace{.5em}
\noindent{\textbf{Guide to notation.}}
The following is a guide to the notation used throughout this paper.
\begin{enumerate}[$ \bullet $]
	\item $\lip f$: the Lipschitz constant of $f$;
	$\hol _{\alpha}f$: the $\alpha$-H\"{o}lder constant of $f$.

	\item $\mathbb{R}_{+} \triangleq \{x\in \mathbb{R}: x \geq 0\}$.

	\item $X(r) \triangleq \mathbb{B}_r = \{ x \in X: |x| <  r\}$ where $X$ is a Banach space.

	\item For a correspondence $ H: X \to Y $ (defined in \autoref{basicCorr}),
	\begin{asparaitem}
		\item $ \graph H $, the graph of $ H $,
		\item $ H(x) \triangleq \{ y: \exists (x, y) \in \graph H \} $,
		\item $A \subset H^{-1}(B)$, if $ \forall x \in A $, $ \exists y \in B $ such that $ y \in H(x) $,
		\item $ H^{-1}: Y \to X $, the inversion of $ H $, defined by $ (y,x) \in \graph H^{-1} \Leftrightarrow (x, y) \in \graph H $.
	\end{asparaitem}

	\item $H\sim(F,G)$: the correspondence $H: X_1 \times Y_1 \rightarrow X_2 \times Y_2$ has a generating map $(F,G)$, see \autoref{basicCorr}.

	\item $f(A) \triangleq \{ f(x): x \in A \}$ where $f$ is a map; $\graph f \triangleq \{ (x, f(x)): x \in X \}$.

	\item $ \diam A \triangleq \sup\{d(m,m'):m,m' \in A \} $: the diameter of $ A $, where $A$ is a subset of a metric space.

	\item $ X \otimes_u Y $, $ X \times Y $: the Whitney sum of $ X, Y $ through $ u $, defined in \autoref{bundleM}.

	\item $ \lambda^{(k)}(m) \triangleq  \lambda(u^{k-1}(m))\lambda(u^{k-2}(m)) \cdots \lambda(m) $ where $ \lambda: M \to \mathbb{R} $ over $ u $, defined in \autoref{bundlemap}.

	\item $ L_u(X,Y) \triangleq \bigcup_{m \in M} (m, L(X_m, Y_{u(m)})) $ where $ u: M \to N $ is a map and $ X, Y $ are vector bundles over $ M $, $ N $, respectively. Write $ K \in L(X,Y) $ (over $ u $) if $ K \in L_u(X, Y) $. See \autoref{vBundle}.

	\item $ \Upsilon^H_X, \Upsilon^V_X $: see \eqref{HVspace}. If $ K \in L(\Upsilon^H_X, \Upsilon^V_Y) $ or $ L(\Upsilon^V_X, \Upsilon^V_Y) $ over some map, we write $ K_{m} (x) = K_{(m,x)} $.

	\item $ Df_m(x) = D_x f_m(x) $: the derivative of $ f_m(x) $ with respect to $ x $; 
	
	$ D_1F_m(x,y) = D_x F_m(x,y) $, $ D_2F_m(x,y) = D_y F_m(x,y) $: the derivatives of $ F_m(x,y) $ with respect to $ x, y $, respectively.

	\item $ D^v f $: the fiber derivative of a bundle map $ f $ defined in \autoref{lip maps}.

	\item $ \nabla f $: the covariant derivative of a bundle map $ f $ defined in \autoref{coderivative}; $ \nabla_m f_m(x) $: the covariant derivative of a bundle map $ f $ at $ (m, x) $ defined in \autoref{coderivative}.

	\item $ f \in C^{k,\alpha} $: if $ D^i f $ is bounded for $ i = 1, 2, \ldots, k $ and $ D^k f \in C^{0,\alpha} $ (i.e. globally $ \alpha $-H\"{o}lder).

	\item $ \mathbb{G}(X) = \bigsqcup_{m} \mathbb{G}(X_m) $ is the ($ C^{k} $) \emph{Grassmann manifold} of $ X $ (see e.g. \cite{AMR88} for the definition of the Grassmann manifold of a Banach space), where $ X $ is a $ C^{k} $ vector bundle.

	\item $\widetilde{d} (A, z) \triangleq \sup_{\tilde{z} \in A} d(\tilde{z}, z) $ where $A$ is a subset of a metric space defined in \autoref{invariant}.

	\item $\sum_\lambda (X, Y) \triangleq \{ \varphi: X \rightarrow Y~ : \lip \varphi \leq \lambda \}$ where $X, Y$ are metric spaces, defined in \autoref{graphTrans}.

	\item $ a_n \lesssim b_n $, $ n \to \infty $ ($ a_n \geq 0, b_n > 0 $) means that $ \sup_{n\geq 0}b^{-1}_n a_n <\infty $.

	\item $ \mathcal{E}(u) $: defined in \autoref{def:lypnum}.

	\item $ \lambda\theta $, $ \max\{ \lambda, \theta \} $: defined by
	$ (\lambda\theta) (m) = \lambda(m)\theta(m) $,
	$ \max\{ \lambda, \theta \}(m) = \max\{ \lambda(m), \theta(m) \} $.

	\item $ \lambda^{*\alpha} \theta < 1 $, $ \theta < 1 $ where $ \lambda, \theta: M \to \mathbb{R}_+ $: see \autoref{absgc} and \autoref{Appbb}.

	\item $ [a] $: the largest integer no more than $ a $ where $ a \in \mathbb{R} $.

	\item $ r_\varepsilon(x) $: the radial retraction, see \eqref{radial}.

	\item {For convenience, we usually write the metric $d(x,y)$ as $|x-y|$}.
\end{enumerate}

\emph{All the metric spaces appearing in this paper are assumed to be complete}, unless otherwise mentioned.

%% file: sect1.tex
\chapter{Basic Notions: Bundle, Correspondence, Generating Map}\label{correspondence}

In this section, we collect some basic notions for the convenience of readers; in particular, we generalize the notion of correspondence having a generating map in \cite{Cha08} to different contexts. Throughout, we use the following notations for a map $ f $ between metric spaces:
\begin{enumerate}[$ \bullet $]
	\item $\lip f = \sup\{ d(f(x), f(y))/ d(x, y): x \neq y \}$, the Lipschitz constant of $f$;

	\item $\hol_{\alpha}f = \sup\{ d(f(x), f(y))/ d(x, y)^{\alpha}: x \neq y \}$, the $\alpha$-H\"{o}lder constant of $f$.
\end{enumerate}

\section{Basic notions relating to bundles}\label{metricfiber}
\subsection{Bundles with metric fibers}\label{bundleM}
$(X, M, \pi)$ is called a \emph{(set) bundle} if $ \pi: X\rightarrow M $ is a surjection, where $X, M$ are sets. If $M, \pi$ are not emphasized, we also call $X$ a bundle.
$X_m \triangleq \pi^{-1}(m)$ is called a \emph{fiber}, and $M, X, \pi$ are called the \emph{base space}, the \emph{total space}, and the \emph{projection}, respectively.
Elements of $X$ are sometimes written as $(m, x)$, where $x \in X_m$, in order to emphasize that $ x $ belongs to the fiber $ X_m $.
The most important and simplest bundle is the \emph{trivial bundle} $(M \times X, M, \pi)$, where $\pi(m,x)=m$.

If every fiber $X_m$ is a complete metric space, we say $(X, M, \pi)$ is a \textbf{(set) bundle with metric fibers}. \emph{Throughout this paper, all bundles are assumed to be bundles with metric fibers}.
Although the metrics $d_m$ of $X_m$ differ from each other, we will use the same symbol $d$ to indicate them. \emph{For convenience, we often write $d(x,y)$ as $|x-y|$}.

If $(X, M, \pi_1), (Y, M, \pi_2)$ are bundles with metric fibers, then $(X\times Y, M, \pi)$ is also a bundle with metric fibers, where $\pi((m,x), (m,y)) = m, x \in X_m, y \in Y_m$. The fibers are $X_m \times Y_m$, $ m \in M $, with the product metric, i.e.,
\begin{equation}\label{metric}
d_p((x,y), (x_1, y_1)) =
\begin{cases}
(d(x,x_1)^p + d(y,y_1)^p)^{\frac{1}{p}},  ~~1 \leq p < \infty, \\
\max\{d(x,x_1), d(y, y_1)\}, ~~p = \infty.
\end{cases}
\end{equation}

Let $ u : M \to N $ be a map. The Whitney sum of $ X, Y $ through $ u $, denoted by $ X \otimes_u Y $, is defined by
\begin{equation}\label{Whsum}
X \otimes_u Y \triangleq \{ (m, x, y): x \in X_m, y \in Y_{u(m)}, m \in M \}.
\end{equation}
This is a bundle with base space $ M $ and fibers $ X_m \times Y_{u(m)} $, $ m \in M $. If $ M = N $ and $ u = \id $, we also use the standard notation $ X \times Y = X \otimes_{\id} Y $.

\subsection{Bundle maps}\label{bundlemap}
If $(X, M_1, \pi_1), (Y, M_2, \pi_2)$ are bundles and $f: X \rightarrow Y$ is a map satisfying $\pi_2 f = u \pi_1$, where $u: M_1 \rightarrow M_2$ is a map, then we call $f$ a \textbf{bundle map over $u$}.
We write $f(m,x) = (u(m), f_m(x)), x \in X_m, f_m(x) \in Y_{u(m)}$, and call $f_m: X_m \rightarrow Y_{u(m)}$ a \textbf{fiber map}. Let $ M_1 = M_2 = M $ and $ X = Y $. The \textbf{$ k $th composition} of $ f $ is a bundle map over $ u^k $. We write $ f^k(m, x) = (u^k(m), f^{(k)}_m(x)) $. Usually, we write $ f^{(k)} $ rather than $ f^k $ for the $ k $th composition of $ f $.

For example, let $ \lambda: M \to \mathbb{R} $. Consider $ \lambda $ as a bundle map over $ u $, i.e. $ \lambda(m,x) = (u(m), \lambda(m)x) : M \times \mathbb{R} \to \mathbb{R} $. In this sense, we call $ \lambda $ a function of $ M \to \mathbb{R} $ over $ u $, and the notation $ \lambda^{(k)}(m) $ \emph{hereafter} means \label{functionu}
\[
\lambda^{(k)}(m) \triangleq \lambda^{(k)}_m (1) =  \lambda(u^{k-1}(m))\lambda(u^{k-2}(m)) \cdots \lambda(m).
\]

\subsection{Bundle charts}
Let $ X $ be a bundle. If $ m_0 \in U \subset M $ and $\varphi: U \times X_{m_{0}} \rightarrow X$, $ (m, x)\mapsto (m, \varphi_m (x))$, is such that $\varphi_m (x) \in X_m$ and $\varphi_m$ is a bijection, we call $(U, \varphi)$ a \textbf{(set) bundle chart} of $X$ at $m_0$ and $ U $ the \textbf{domain} of the bundle chart $ (U, \varphi) $. If we do not emphasize the domain $ U $, we also say $ \varphi $ is a bundle chart (at $ m_0 $).
If $\mathcal{A} =\{ (U_{\alpha}, \varphi^{\alpha}): \alpha \in \Lambda \} $ is a family of bundle charts of $ X $ such that $\bigcup_{\alpha \in \Lambda} U_{\alpha} = M$, where $ \Lambda $ is an index set, we usually call $ \mathcal{A} $ a \textbf{bundle atlas} (or \textbf{atlas} for short) of $ X $. For convenience, we also say that $ \mathcal{A} $ is a \textbf{bundle atlas on} $ M_1 \subset M $ if for every $ m_0 \in M_1 $, there is a bundle chart at $ m_0 $ belonging to $ \mathcal{A} $.

\subsection{Vector bundles}\label{vBundle}

A bundle is called a \emph{vector bundle} if each fiber is a Banach space. A $ C^{k} $ vector bundle is defined in \autoref{bundleP} \eqref{vtBundle} (see also \cite{AMR88}). A bundle atlas of a vector bundle is called a \emph{vector bundle atlas} if each bundle chart $ (U, \phi) $ belonging to this atlas is linear, i.e. $ x \mapsto \phi_m(x) $ is linear for each $ m \in U $. A bundle map between vector bundles is called a \emph{vector bundle map} if each fiber map is linear.

Denote $ L_{u}(X,Y) \triangleq \bigcup_{m \in M} (m, L(X_m, Y_{u(m)})) $, which is a vector bundle over $ M $, when $ X,Y $ are vector bundles over $ M, N $, respectively, and $ u: M \to N $. Note that if $ X, Y $ are $ C^k $ vector bundles and $ u $ is $ C^k $ ($ k \geq 0 $), then $ L_{u}(X,Y) $ is also a $ C^k $ vector bundle. If $ M = N $ and $ u = \id $, we use the standard notation $ L(X, Y) = L_{\id}(X, Y) $. Without causing confusion, write $ K \in L(X,Y) $ (over $ u $) if $ K \in L_u(X, Y) $.

\subsection{$0$-sections}\label{0-section}

\begin{defi}
	A map $i: M \rightarrow X$ is called a \emph{section} of a bundle $(X, M, \pi)$ if $i(m) \in X_m, \forall m \in M$.
	We say $ i $ is a \textbf{$\bm{0}$-section} of $X$ with respect to the bundle atlas $\mathcal{A}$ of $ X $ if for every $(U, \varphi) \in \mathcal{A}$ at $m_0$, $\varphi(m, i(m_0)) = (m, i(m))$ for $ m \in U$.
	When $X$ has a $0$-section $ i $, \emph{we will write $|x| \triangleq d(x, i(m))$ for $x \in X_m$}.
\end{defi}
For example, any vector bundle has a natural $0$-section with respect to any vector bundle atlas. The notion of $0$-section is important for us to give higher regularity of invariant graphs in the invariant section case (see \autoref{stateRegularity}); we think this is why \cite{Cha08} could not obtain the classical spectral gap condition for the regularity of invariant foliations in the trivial bundle case (see \cite[note, p. 1431]{Cha08} and \autoref{lem:sheaf}, \autoref{holsheaf}, \autoref{h3*}). In the bounded section case (see \autoref{bounded}), we only need the bounded section satisfying the weaker assumption (UB) on \autopageref{b3-}.

\section{Definitions: discrete case} \label{basicCorr}

Let $X, Y$ be sets. We say that there is a \textbf{correspondence} $H: X \rightarrow Y$ (see \cite{Cha08}) if there is a non-empty subset of $X \times Y$ called the \emph{graph} of $ H $ and denoted by $\graph H$.
There are some operations between correspondences:
\begin{enumerate}[$ \bullet $]
	\item (inversion) If $ H: X \rightarrow Y $ is a correspondence, define its inversion $ H^{-1} : Y \rightarrow X $ by
	\[
	(y,x) \in \graph H^{-1} ~\text{if and only if}~(x,y) \in \graph H.
	\]
	\item (composition) Let $H_1: X \rightarrow Y$, $H_2: Y \rightarrow Z$ be correspondences. Define $H_2 \circ H_1: X \rightarrow Z$ by
	\[
	\graph H_2 \circ H_1 = \{ (x,z): \exists y \in Y ~\text{such that}~ (x,y) \in \graph H_1, (y,z) \in \graph H_2 \}.
	\]
	If $ H: X \to X $, then as usual $ H^{(n)} \triangleq H\circ \cdots \circ H $ ($ n $ times).
\end{enumerate}
The following notations for a correspondence $ H: X \to Y $ will be used frequently:

\begin{enumerate}[$ \bullet $]
	\item Write
	\[
	H(x) \triangleq \{ y \in Y: \exists (x, y) \in \graph H \},~H(A) \triangleq \bigcup_{x \in A} H(x),
	\]
	if $ A \subset X $; allow $ H(x) = \emptyset $; if $ H(x) = \{y\} $, write $ H(x) = y $.

	\item So $ A \subset H^{-1}(B) $ means that $ \forall x \in A $, $ \exists y \in B $ such that $ y \in H(x) $ (i.e. $ x \in H^{-1}(y) $).

	\item If $ X = Y $, we say $ A \subset X $ is \emph{invariant} under $ H $ if $ A \subset H^{-1}(A) $.

	\item If $ A \subset H^{-1}(B) $, then $ H: A \to B $ can be regarded as a correspondence $ H|_{A \to B} $ defined by $ (x, y) \in \graph H|_{A \to B} \Leftrightarrow y \in H(x) \cap B $.

	\item We say $ H: A \to B $ (or $ H|_{A \to B} $) induces (or defines) a map (also denoted by $ H|_{A \to B} $), if $ \forall x \in A $, $ H(x) \cap B $ consists of only one element; if $ A = B $, we write $ H|_{A} = H|_{A \to B} $.
\end{enumerate}

The related concepts mentioned above also appeared in \cite{Aki93}, where the author used the term `relation' in place of `correspondence' and used it to unify some properties of topological dynamical systems (e.g. recurrence, attractor, chain recurrence, etc.).

Apparently, $ x \mapsto H(x) $ can be considered as a `multi-valued map'. But it does not make sense for us to look at it that way, since we only focus on the description of $ \graph H $. We say a correspondence $H: X_1 \times Y_1 \rightarrow X_2 \times Y_2$ has a \textbf{generating map} $(F,G)$, and write $H\sim(F,G)$, if there are maps $F: X_1 \times Y_2 \rightarrow X_2$, $G: X_1 \times Y_2 \rightarrow Y_1$, such that $(x_2,y_2) \in H(x_1,y_1) \Leftrightarrow y_1 = G(x_1,y_2), ~x_2 = F(x_1, y_2)$.

\begin{exa}\label{cgm}
	\begin{enumerate}[(a)]
		\item A map always induces a correspondence by using its graph, but it may not have a generating map. The following type of maps induce correspondences with generating maps.
		Let $H = (f,g): X_1 \times Y_1 \rightarrow X_2 \times Y_2$ be a map. Suppose for every $x_1 \in X_1$, $g_{x_1}(\cdot) \triangleq g(x_1, \cdot): Y_1 \rightarrow Y_2$ is a bijection. Let $G(x_1, y_2) = g^{-1}_{x_1}(y_2)$, $F(x_1, y_2) = f(x_1, G(x_1, y_2))$. Then $H \sim (F,G)$.

		A map having some hyperbolicity will be in this case, no matter if the map is bijective or not. This is an important observation for studying invariant manifolds for non-invertible maps. By transferring the maps to correspondences having generating maps, M. Chaperon \cite{Cha08} gave a uniform method of proving the existence of different types of invariant manifolds.
		\item \label{diff} Let $ X, Y $ be Banach spaces. Let $ T(t): X \to X $, $ S(-t): Y \to Y $, $ t \geq 0 $, be $ C_0 $ semigroups with generators $ A, -B $ respectively and $ |T(t)| \leq e^{\mu_s t} $, $ |S(-t)| \leq e^{-\mu_u t} $, $ \forall t \geq 0$. Take $ F_1 : X \times Y \to X $, $ F_2 : X \times Y \to Y $ and Lipschitz with $ \lip F_i \leq \varepsilon $. Then the time-one \emph{mild} solutions of the following equation induce a correspondence $ H : X \times Y \to X \times Y $ with generating map $ (F, G) $:
		\[ \tag{$ \divideontimes $}
		\begin{cases}
		\dot{x} = A x + F_1 (x, y), \\
		\dot{y} = B y + F_2 (x, y),
		\end{cases}
		\]
		or equivalently
		\[
		\begin{cases}
		x(t) = T(t)x_1 + \int_{0}^{t} T(t-s)F_1(x(s),y(s)) ~\mathrm{d} s, \\
		y(t) = S(t - 1)y_2 - \int_{t}^{1} S(t-s)F_2(x(s),y(s)) ~\mathrm{d} s,
		\end{cases}
		0 \leq t \leq 1.
		\]
		Define maps $ F, G $ as follows. The above equation has a unique $ C^0 $ solution $ (x(t), y(t)) $ ($ 0 \leq t \leq 1 $) such that $ x(0) = x_1, y(1) = y_2 $ (see e.g. \cite{ElB12, Che18c}); now let $ F(x_1, y_2) = x(1) $, $ G(x_1, y_2) = y(0) $. Using $ F, G $, one can define a correspondence $ H \sim (F, G) $.
		Moreover, if $ \mu_u - \mu_s - 2\varepsilon > 0 $, then there are constants $ \alpha, \beta > 0, \lambda_s = e^{\mu_s + \varepsilon}, \lambda_u = e^{-\mu_u + \varepsilon} $ such that $ \alpha \beta < 1 $ and $ H $ satisfies the (A)($ \alpha, \lambda_u $) (B)($ \beta, \lambda_s $) condition (see \autoref{AB}). For details, see \cite{Che18c}.
		Equation $ (\divideontimes) $ is usually \emph{ill-posed}, meaning that for a given $ (x_0, y_0) \in X \times Y $, there may be no (mild) solution $ (x(t), y(t)) $ satisfying $ (\divideontimes) $ with $ x(0) = x_0, y(0) = y_0 $.
	\end{enumerate}
\end{exa}

Let $(X, M, \pi_1), (Y, N, \pi_2)$ be bundles, and $u: M \rightarrow N$ a map. Suppose $H_m: X_m \times Y_m \rightarrow X_{u(m)} \times Y_{u(m)}$ is a correspondence for every $m \in M$. Using $H_m$, one can determine a correspondence $H: X \times Y \rightarrow X \times Y$ by $\graph H \triangleq \bigcup_{m \in M}(m, \graph H_m$), i.e. $(u(m), x_{u(m)}, y_{u(m)}) \in H(m, x_m, y_m) \Leftrightarrow (x_{u(m)}, y_{u(m)}) \in H_m(x_m, y_m)$. We call $H$ a \textbf{bundle correspondence over a map $u$}. Now, $ H^{-1} $, the inversion of $ H $, means $ H^{-1}_{m} \triangleq (H_{m})^{-1}: X_{u(m)} \times Y_{u(m)} \to X_{m} \times Y_{m} $, $ m \in M $.
If $H_m$ has a generating map $(F_m, G_m)$ for every $m \in M$, where $F_m: X_m \times Y_{u(m)} \rightarrow X_{u(m)}, ~G_m: X_m \times Y_{u(m)} \rightarrow Y_m$ are maps, we say $H$ has a \textbf{generating bundle map} $(F,G)$, and write $H \sim (F,G)$.

Let $ M = N $. We use $ H^{(k)} $ for the \textbf{$ k $th composition} of $ H $, defined by $ H^{(k)}_m = H_{u^{k-1}(m)} \circ H_{u^{k-2}(m)} \circ \cdots \circ H_m $. This is a bundle correspondence over $ u^k $.

\begin{exa}
	\begin{enumerate}[(a)]
		\item Let $H: X \times Y \rightarrow X \times Y$ be a bundle map over $u$, where $(X, M, \pi_1)$, $(Y, M, \pi_2)$ are bundles, and $H_m = (f_m, g_m): X_m \times Y_m \rightarrow X_{u(m)} \times Y_{u(m)}$. Suppose for every $m \in M$ and $x \in X_m$, $g_{m, x}(\cdot) \triangleq g_m(x, \cdot): Y_{m} \rightarrow Y_{u(m)}$ is a bijection. Let $G_m(x, y) = g^{-1}_{m,x}(y): X_m \times Y_{u(m)} \rightarrow Y_m$, $F_m(x, y) = f_m(x, G_m(x, y)): X_m \times Y_{u(m)} \rightarrow X_{u(m)}$. Then we have $H_m \sim (F_m,G_m)$ and $H \sim (F,G)$.
		\item Let $ A, B $ be the operators in \autoref{cgm} \eqref{diff}. Let $ Z = X \times Y $ and $ C = A \oplus B: Z \to Z $. Let $ M $ be a topological space and $ t: M \to M: \omega \mapsto t \omega $ a continuous semiflow. Let $ L : M \to L(Z, Z) $ be strongly continuous (i.e. $ (\omega, z) \mapsto L(\omega)z $ is continuous) and $ f(\cdot)(\cdot): M \times Z \to Z $ continuous. Assume for every $ \omega \in M $, $ \sup_{t \geq 0}|L(t\omega)| = \tau(\omega) < \infty $ and $ \sup_{t\geq 0 } \lip f(t\omega)(\cdot) < \infty $. Consider the equation
		\[
		\dot{z}(t) = Cz(t) + L(t\omega)z(t) + f(t\omega)z(t).
		\]
		For the above differential equation (which is usually ill-posed), the time-one mild solutions induce a bundle correspondence $ H $ over $ u(\omega) = 1\cdot \omega $ with generating bundle map. Moreover, under some \emph{uniform dichotomy} assumption (see e.g. \cite{LP08}), $ H $ will satisfy the (A) (B) condition (defined in \autoref{ABbundleC}). See \cite{Che18c}.
	\end{enumerate}
\end{exa}

\begin{lem}\label{lem:com}
	\begin{enumerate}[(1)]
		\item Let $ H_i : X_{i} \times Y_{i} \to X_{i+1} \times Y_{i+1}  $, $ i = 1,2 $ be correspondences with generating maps $ (F_i, G_i) $, $ i = 1, 2 $, respectively. Assume that 
		\[
		\sup_y \lip G_2(\cdot, y) \sup_x \lip F_1(x, \cdot) < 1.
		\]
		Then $ H_2 \circ H_1 $ also has a generating map $ (F, G) $.
		\item Let $ (X, M, \pi_1), (Y, M, \pi_2) $ be bundles with metric fibers, $ u : M \to M $ a map and $ H : X \times Y \to X \times Y $ a bundle correspondence over $ u $. Assume $ H \sim (F, G) $ and
		\[
		\sup_y \lip G_{u(m)}(\cdot, y) \sup_x \lip F_{m}(x, \cdot) < 1, ~m \in M.
		\]
		Then $ H^{(k)} $ also has a generating bundle map.
	\end{enumerate}
\end{lem}
\begin{proof}
	(2) is a direct consequence of (1), and the proof of (1) is simple. Let $ x_1 \in X_1 $, $ y_3 \in Y_3 $. Since $ \lip \sup_yG_2(\cdot, y) \sup_x F_1(x, \cdot) < 1 $, there is a unique $ y_2 = y_2(x_1,y_3) \in Y_2 $ such that $ y_2 = G_2(F_1(x_1,y_2), y_3) $. Let $ x_2 = x_2(x_1, y_3) = F_1(x_1, y_2) $. Set $ G(x_1,y_3) = G_1(x_1, y_2) $, $ F(x_1, y_3) = F_2(x_2, y_3) $. Now $ F, G $ are as desired.
\end{proof}

\section{Dual correspondence}\label{dual}

Let $ H: X_1 \times Y_1 \to X_2 \times Y_2 $ be a correspondence with generating map $ (F, G) $. The \textbf{dual correspondence} of $ H $, denoted by $ \widetilde{H} $, is defined as follows. Set $ \widetilde{X}_1 = Y_2 $, $ \widetilde{X}_2 = Y_1 $, $ \widetilde{Y}_1 = X_2 $, $ \widetilde{Y}_2 = X_1 $ and
\[
\widetilde{F}(\widetilde{x}_1, \widetilde{y}_2) = G(\widetilde{y}_2, \widetilde{x}_1), ~ \widetilde{G}(\widetilde{x}_1, \widetilde{y}_2) = F(\widetilde{y}_2, \widetilde{x}_1).
\]
Now $ \widetilde{H} : \widetilde{X}_1 \times \widetilde{Y}_1 \to \widetilde{X}_2 \times \widetilde{Y}_2 $ is uniquely determined by $(\widetilde{x}_2,\widetilde{y}_2) \in \widetilde{H}(\widetilde{x}_1,\widetilde{y}_1) \Leftrightarrow \widetilde{y}_1 = \widetilde{G}(\widetilde{x}_1,\widetilde{y}_2), ~\widetilde{x}_2 = \widetilde{F}(\widetilde{x}_1, \widetilde{y}_2)$, i.e., $ \widetilde{H} \sim (\widetilde{F}, \widetilde{G}) $. One can similarly define the \textbf{dual bundle correspondence} $ \widetilde{H} $ of the bundle correspondence $ H $ over $ u $ if $ u $ is invertible; $ \widetilde{H} $ now is over $ u^{-1} $.

$ \widetilde{H} $ and $ H $ are in some duality in the sense that $ \widetilde{H} $ can reflect some properties of $ H^{-1} $. For example, if $ H $ satisfies the (A)$(\alpha; \alpha', \lambda_u)$ (B)$(\beta; \beta', \lambda_s)$ condition (see \autoref{AB} below), then $ \widetilde{H} $ satisfies the (A)$(\beta; \beta', \lambda_s)$ (B)$(\alpha; \alpha', \lambda_u)$ condition; thus if one obtains `stable results' for $ H $ (see \autoref{invariant}), then one also obtains `unstable results' for $ H $ through the `stable results' for $ \widetilde{H} $.
In this paper, \emph{we only state the `stable results'}, leaving the corresponding `unstable' statements for the readers.

%% file: sect2.tex
\chapter{Hyperbolicity and the (A)(B) Condition}\label{hyper}

We focus on dynamical systems having some hyperbolicity described by the so called (A) (B) condition.
Our definition of the (A) (B) condition is for `dynamical systems' themselves and can be seen as a non-linear version of the cone condition, which is motivated by \cite{MS88} (see also \cite{LYZ13, Zel14} and \autoref{relAB1}).
The definitions of the (A) (B) condition are given in \autoref{defAB} and some relations between the (A) (B) condition and some classical hyperbolicity conditions are given in \autoref{relAB0} and \autoref{relAB1}. See also \cite{Che18c} for the relation between the (exponential) dichotomy and the (A) (B) condition in some classes of well-posed and ill-posed differential equations.

\section{Definitions} \label{defAB}

\subsection{(A) (B) condition for correspondences}

Let $X_i, Y_i, ~ i=1,2$, be metric spaces. \emph{For convenience, we write the metric $d(x,y)$ as $|x-y|$}.

\begin{defi}\label{AB}
	We say a correspondence $H: X_1 \times Y_1 \rightarrow X_2 \times Y_2$ satisfies the \textbf{(A) (B) condition, or (A)$(\bm{\alpha; \alpha', \lambda_u})$ (B)$(\bm{\beta; \beta', \lambda_s})$ condition}, if for all $(x_1, y_1) \times (x_2, y_2)$ and $(x'_1, y'_1) \times (x'_2, y'_2) \in \graph H$,
	\begin{enumerate}[(A)]
		\item (A1) if $|x_1 - x'_1| \leq \alpha |y_1 - y'_1|$, then $|x_2 - x'_2| \leq \alpha' |y_2 - y'_2|$;

		\noindent(A2) if $|x_1 - x'_1| \leq \alpha |y_1 - y'_1|$, then $ |y_1 - y'_1| \leq \lambda_u |y_2 - y'_2|$;

		\item (B1) if $|y_2 - y'_2| \leq \beta |x_2 - x'_2| $, then $ |y_1 - y'_1| \leq \beta' |x_1 - x'_1|$;

		\noindent(B2) if $|y_2 - y'_2| \leq \beta |x_2 - x'_2| $, then $ |x_2 - x'_2| \leq \lambda_s |x_1 - x'_1|$.
	\end{enumerate}
	If $\alpha = \alpha', ~\beta = \beta'$, we also speak about the \textbf{(A)$(\bm{\alpha, \lambda_u})$ (B)$(\bm{\beta, \lambda_s})$ condition}.
\end{defi}

In particular, if $ H \sim (F, G) $, then the maps $F,G$ satisfy the following Lipschitz conditions:
\begin{enumerate}[(A$'$)]
	\item (A1$'$) $\sup_{x}\lip F(x,\cdot) \leq \alpha'$,  (A2$'$) $\sup_{x}\lip G(x,\cdot) \leq \lambda_u$.
	\item (B1$'$) $\sup_{y}\lip G(\cdot,y) \leq \beta'$,  (B2$'$) $\sup_{y}\lip F(\cdot,y) \leq \lambda_s$.
\end{enumerate}
If $F,G$ satisfy the above Lipschitz conditions, then we say $H$ satisfies the \textbf{(A$\bm{'}$)($\bm{\alpha', \lambda_u}$) (B$\bm{'}$)($\bm{\beta', \lambda_s}$) condition}, or the \textbf{(A$\bm{'}$) (B$\bm{'}$) condition}. Similarly, we can define the (A$'$) (B) condition, (A) condition, and (A$'$) condition, etc., if $H$ satisfies (A$'$) (B), (A), (A$'$), etc., respectively.

\subsection{(A) (B) condition for bundle correspondences}\label{ABbundleC}

Let $(X, M, \pi_1), (Y, M, \pi_2)$ be bundles with metric fibers and $u: M \rightarrow M$ a map. Let
$H: X \times Y \rightarrow X \times Y$ be a bundle correspondence over $u$. We say $H$ satisfies the \textbf{(A) (B) condition}, or \textbf{(A)$\bm{(\alpha; \alpha', \lambda_u)}$ (B)$\bm{(\beta; \beta', \lambda_s)}$ condition}, if every $H_m \sim (F_m, G_m)$ satisfies the (A)$(\alpha(m); \alpha'(m), \lambda_u(m))$ (B)$(\beta(m); \beta'(m), \lambda_s(m))$ condition, where $\alpha, \alpha', \lambda_u, \beta, \beta', \lambda_s$ are functions $M \rightarrow \mathbb{R}_+$. Also, if $\alpha \equiv \alpha'$ and $\beta \equiv \beta'$, then we talk about the (A)$(\alpha, \lambda_u)$ (B)$(\beta, \lambda_s)$ condition.
The (A$'$)(B$'$) condition, (A$'$)(B) condition, or (A)(B$'$) condition for bundle correspondences are defined similarly to the case for correspondences.

In linear dynamical systems, spectral theory is well developed: see e.g. \cite{KH95,SS94,CL99,LP08,LZ17,Rue82,Man83,LL10}. The main objective of this theory is to characterize asymptotic properties of linear dynamical systems. Lyapunov numbers play an important role in measuring the average rate of separation of orbits starting from nearby initial points, and describing the local stability of orbits and chaotic behavior of dynamical systems. \emph{The functions $ \lambda_s $, $ \lambda_u $ appearing in our definition of the (A) (B) condition are related to Lyapunov numbers}. The spectral spaces are invariant under the dynamical system and every orbit starting from these spaces has the same asymptotic behavior. \emph{The fibers of the bundles $ X, Y $ in our setting are related to these spaces and the functions $ \alpha $, $ \beta $ in the (A) (B) condition describe how the fibers are approximately invariant}.
The robustness of the (A) (B) condition is not clear but since the Lipschitz condition (A$ ' $) (B$ ' $) implies the (A) (B) condition in some contexts (see \autoref{lem:a3} and \autoref{lem:cp}), the (A) (B) condition has some `open condition' property. We do not give the corresponding condition in a geometric way associated with tangent fields, but see \cite{MS88, Zel14} for some such results.

\begin{rmk}
	More generally, for every $ k $, let $ H^{(k)} $ be the $ k $th composition of $ H $. We say $ H $ satisfies the \textbf{($ \bm{\mathrm{A_1}} $) ($ \bm{\mathrm{B_1}} $) condition}, or the \textbf{($ \bm{\mathrm{A_1}} $)($ \bm{\alpha; \lambda_s; c} $) ($ \bm{\mathrm{B_1}} $)($ \bm{\beta; \lambda_s; c} $) condition}, if every $ H^{(k)}_m \sim (F^{(k)}_m, G^{(k)}_m) $ satisfies the (A)$(\alpha(m), c(m)\lambda^{k}_u(m))$ (B)$(\beta(m), c(m)\lambda^{k}_s(m))$ condition.
	In general, the (A) (B) does not imply ($ \mathrm{A_1} $) ($ \mathrm{B_1} $). However, when $ \lambda_s, \lambda_u $ are orbitally bounded relative to $ u $, i.e.,
	\[
	\sup_{N \geq 0} \lambda_s( u^N(m) ) < \infty,~
	\sup_{N \geq 0} \lambda_u(u^N(m)) < \infty,
	\]
	then the (A) (B) implies ($ \mathrm{A_1} $) ($ \mathrm{B_1} $). Indeed, one can use the \emph{sup Lyapunov numbers} of $ \{ \lambda^{(k)}_s(m) \} $, $ \{ \lambda^{(k)}_u(m) \} $ (see \autoref{def:lypnum}).
\end{rmk}

We now explain how to verify the (A) (B) condition.

\section{Relation between (A)(B) and (A$'$)(B$'$): Lipschitz case} \label{relAB0}

For simplicity, we only consider the case of correspondences.
In the following, assume that $X_i, Y_i$, $i=1,2$, are metric spaces, and $H: X_1 \times Y_1 \rightarrow X_2 \times Y_2$ is a correspondence with generating map $(F,G)$. The following lemma is just a consequence of the definition.
\begin{lem}\label{lem:a1}
	\begin{enumerate}[(1)]
		\item If $ H $ satisfies the (A)$(\alpha; \alpha', \lambda_u)$ (B)$(\beta; \beta', \lambda_s)$ condition, then $ \widetilde{H} $, the dual correspondence of $ H $ (see \autoref{dual}), satisfies the (A)$(\beta; \beta', \lambda_s)$ (B)$(\alpha; \alpha', \lambda_u)$ condition.
		\item If for all $n$, $H_n: X_n \times Y_n \rightarrow X_{n+1} \times Y_{n+1}$ satisfies the (A)$(\alpha, \lambda_{n, u})$ (B)$(\beta, \lambda_{n, s})$ condition, then $H_n \circ \cdots \circ H_1$ satisfies the (A)$(\alpha, \lambda_{1,u} \cdots \lambda_{n, u})$ (B)$(\beta, \lambda_{1,s} \cdots \lambda_{n,s})$ condition.
	\end{enumerate}
\end{lem}

The following lemma's condition was also used in \cite{Cha08} to obtain hyperbolicity. We will show that this condition implies (A) (B), so our main results recover \cite{Cha08}.
\begin{lem}[Lipschitz in $d_{\infty}$]\label{lem:ab}
	If for all $x_1, x'_1 \in X_1$ and $y_2, y'_2 \in Y_2$,
  	\begin{gather*}
  	|F(x_1, y_2)-F(x'_1, y'_2)| \leq \max\{ \lambda_s |x_1 - x'_1|,~ \alpha |y_2 - y'_2| \}, \\
  	|G(x_1, y_2)-G(x'_1, y'_2)| \leq \max\{ \beta |x_1 - x'_1|,~ \lambda_u |y_2 - y'_2| \},
  	\end{gather*}
   	and $\alpha\beta < 1$, $\lambda_s \lambda_u < 1$, then $H$ satisfies the (A)$(\alpha, \lambda_u)$ (B)$(\beta, \lambda_s)$ condition. In addition, if $ \alpha \beta < \lambda_s \lambda_u $, then $H$ satisfies the (A)$(c^{-1} \alpha; \alpha, \lambda_u)$ (B)$(c^{-1} \beta; \beta, \lambda_s)$ condition where $ c = \lambda_s \lambda_u < 1 $.
\end{lem}
\begin{proof}
	Let $(x_1, y_1) \times (x_2, y_2),~ (x'_1, y'_1) \times (x'_2, y'_2) \in \graph H$. If $|x_1 - x'_1| \leq \alpha|y_1 - y'_1|$, then
	\begin{align*}
	|y_1 - y'_1| \leq & \max\{ \beta |x_1 - x'_1|,~ \lambda_u |y_2 - y'_2| \}
	\leq \max\{ \alpha \beta |y_1 - y'_1|,~ \lambda_u |y_2 - y'_2| \} \\
	\leq & \lambda_u |y_2 - y'_2| \quad (\text{since}~ \alpha\beta < 1),
	\end{align*}
	and
  	\begin{align*}
  		|x_2 - x'_2| \leq & \max\{ \lambda_s |x_1 - x'_1|,~ \alpha |y_2 - y'_2| \}
  		\leq \max\{ \lambda_s \alpha |y_1 - y'_1|, ~\alpha |y_2 - y'_2| \} \\
  		\leq & \max\{ \lambda_s \lambda_u \alpha |y_2 - y'_2|, ~\alpha |y_2 - y'_2| \} = \alpha |y_2 - y'_2| \quad (\text{since}~ \lambda_s \lambda_u < 1).
  	\end{align*}
 	 The (A) condition is satisfied by $H$, similarly for (B) and the last statement.
\end{proof}

\begin{lem}[Lipschitz in $d_{1}$]\label{lem:a3}
  If $H$ satisfies the (A$'$)$(\widetilde{\alpha}, \widetilde{\lambda}_u)$ (B$'$)$(\widetilde{\beta}, \widetilde{\lambda}_s)$ condition, and $\widetilde{\lambda}_s \widetilde{\lambda}_u < c^2, ~ \widetilde{\alpha} \widetilde{\beta} < (c - \sqrt{\widetilde{\lambda}_s \widetilde{\lambda}_u})^2$, where $ 0 < c \leq 1 $, then $H$ satisfies the (A)$(\alpha; c\alpha, \lambda_u)$ (B)$(\beta; c\beta, \lambda_s)$ condition, where
  \[
  \begin{gathered}
  \alpha = \frac{b-\sqrt{b^2 - 4c\widetilde{\alpha}\widetilde{\beta}} }{2\widetilde{\beta}},~
  \beta = \frac{b-\sqrt{b^2 - 4c\widetilde{\alpha}\widetilde{\beta}} }{2\widetilde{\alpha}},\\
  \lambda_s = \frac{\widetilde{\lambda}_s}{1-\alpha\widetilde{\beta}},~
  \lambda_u = \frac{\widetilde{\lambda}_u}{1-\beta\widetilde{\alpha}},~
  b = c - \widetilde{\lambda}_s \widetilde{\lambda}_u + \widetilde{\alpha}\widetilde{\beta}.
  \end{gathered}
  \]
  Furthermore, $\alpha \beta < 1,~ \lambda_s \lambda_u < 1$.
\end{lem}
\begin{proof}
  Let $(x_1, y_1) \times (x_2, y_2),~ (x'_1, y'_1) \times (x'_2, y'_2) \in \graph H$. The (A$'$)$(\widetilde{\alpha}, \widetilde{\lambda}_u)$ (B$'$)$(\widetilde{\beta}, \widetilde{\lambda}_s)$ condition says that
  \[
  |x_2 - x'_2| \leq \widetilde{\lambda}_s |x_1 - x'_1| + \widetilde{\alpha} |y_2 - y'_2|,~ |y_1 - y'_1| \leq \widetilde{\beta} |x_1 - x'_1| + \widetilde{\lambda}_u |y_2 - y'_2|.
  \]
  Using this, if $|x_1 - x'_1| \leq \alpha |y_1 - y'_1|$, then $|x_1 - x'_1| \leq \frac{\alpha \widetilde{\lambda}_u}{1 - \alpha \widetilde{\beta}} |y_2 - y'_2|$. So
  \[
  |x_2 - x'_2| \leq ( \frac{\alpha \widetilde{\lambda}_s \widetilde{\lambda}_u}{1 - \alpha \widetilde{\beta} } + \widetilde{\alpha} ) |y_2 - y'_2|, ~
  |y_1 - y'_1| \leq  \frac{\widetilde{\lambda}_u}{1 - \alpha \widetilde{\beta}} |y_2 - y'_2|.
  \]
  $\alpha$ should satisfy $\frac{\alpha \widetilde{\lambda}_s \widetilde{\lambda}_u}{1 - \alpha \widetilde{\beta} } + \widetilde{\alpha} \leq c\alpha, ~ \alpha \widetilde{\beta} < 1$, or equivalently
  \[
  \alpha^2 c\widetilde{\beta} - \alpha b + \widetilde{\alpha} \leq 0, ~\alpha \widetilde{\beta} < 1,
  \]
  where $b = c - \widetilde{\lambda}_s \widetilde{\lambda}_u + \widetilde{\alpha}\widetilde{\beta}$.
  This can be satisfied if
  \[
  b^2 \geq 4 c\widetilde{\alpha} \widetilde{\beta}, ~
  \frac{b + \sqrt{b^2 - 4c \widetilde{\alpha} \widetilde{\beta}}} {2 \widetilde{\alpha}} < \frac{1}{\widetilde{\alpha}}, ~
  \frac{b - \sqrt{b^2 - 4c\widetilde{\alpha} \widetilde{\beta}}} {2 \widetilde{\alpha}} \leq \alpha \leq \frac{b + \sqrt{b^2 - 4c \widetilde{\alpha} \widetilde{\beta}}} {2 \widetilde{\alpha}}.
  \]
  We have $b^2 \geq 4 c\widetilde{\alpha} \widetilde{\beta}$ if and only if
  \[
  x^2 - 2(c + \widetilde{\lambda}_s \widetilde{\lambda}_u) x + (c - \widetilde{\lambda}_s \widetilde{\lambda}_u)^2 \geq 0,
  \]
  where $x = \widetilde{\alpha} \widetilde{\beta}$. So we need
  \[
  x < \frac{ 2(c + \widetilde{\lambda}_s \widetilde{\lambda}_u) - \sqrt{ 4(c + \widetilde{\lambda}_s \widetilde{\lambda}_u)^2 - 4 (c - \widetilde{\lambda}_s \widetilde{\lambda}_u)^2 } }{2} = (c - \sqrt{\widetilde{\lambda}_s \widetilde{\lambda}_u})^2.
  \]
  This is clearly the condition given in the lemma. $ (b + \sqrt{b^2 - 4c \widetilde{\alpha} \widetilde{\beta}}) / (2 \widetilde{\alpha}) < 1 / \widetilde{\alpha} $ is automatically satisfied by a simple computation. The same argument can be applied to the (B) condition.

  Finally, take $\alpha, \beta, \lambda_s, \lambda_u $ as in the lemma. All we need to show is $\alpha \beta < 1$, $\lambda_s \lambda_u < 1$, which is a simple computation.
\end{proof}

\begin{lem}\label{lem:a4}
  Let $H$ satisfy the (A$'$)$(\widetilde{\alpha}, \widetilde{\lambda}_u)$ (B$'$)$(\widetilde{\beta}, \widetilde{\lambda}_s)$ condition, and $ 0 \neq \widetilde{\lambda}_s \widetilde{\lambda}_u < 1$. Choose $\rho_1,\rho_2$ such that
  \[
  \widetilde{\lambda}_s \widetilde{\lambda}_u < \rho_2 < (\widetilde{\lambda}_s \widetilde{\lambda}_u)^{\frac{1}{2}}, ~
  1-\rho^{-1}_2 \widetilde{\lambda}_s \widetilde{\lambda}_u < \rho_1 < 1 - \rho_2.
  \]
  Let $\rho = \rho_1 + \rho_2~(<1)$.
  Suppose $\widetilde{\alpha} \widetilde{\beta} \leq \frac{1 - \rho^{-1}_2 \widetilde{\lambda}_s \widetilde{\lambda}_u }{\rho^{-1}_1}~ (<1)$. Then $H$ satisfies the (A)$(\alpha;\rho \alpha, \lambda_u)$ (B)$(\beta;\rho \beta, \lambda_s)$ condition, where
  \[
  \alpha = \rho^{-1}_1 \widetilde{\alpha}, ~\beta = \rho^{-1}_1 \widetilde{\beta}, ~\lambda_s = \frac{ \widetilde{\lambda}_s }{1-\alpha \widetilde{\beta}}, ~\lambda_u = \frac{ \widetilde{\lambda}_u }{1-\beta \widetilde{\alpha}}.
  \]
  Furthermore, $\alpha \beta < 1,~ \lambda_s \lambda_u < 1$.
\end{lem}
\begin{proof}
  Let $\alpha, \beta, \lambda_s, \lambda_u $ be as in the lemma. Then $\alpha \widetilde{\beta} = \beta \widetilde{\alpha} = \rho^{-1}_1 \widetilde{\alpha} \widetilde{\beta} \leq 1 - \rho^{-1}_2 \widetilde{\lambda}_s \widetilde{\lambda}_u$, and
  \[
  \alpha \beta = \rho^{-2}_1 \widetilde{\alpha} \widetilde{\beta} \leq \rho^{-1}_1(1 - \rho^{-1}_2 \widetilde{\lambda}_s \widetilde{\lambda}_u) < 1,
  ~\lambda_s \lambda_u = \frac{ \widetilde{\lambda}_s \widetilde{\lambda}_u }{ (1-\alpha \widetilde{\beta}) (1-\beta \widetilde{\alpha}) } \leq \frac{\rho^{2}_2}{\widetilde{\lambda}_s \widetilde{\lambda}_u} < 1.
  \]
  Let $(x_1, y_1) \times (x_2, y_2),~ (x'_1, y'_1) \times (x'_2, y'_2) \in \graph H$. Now, by the (A$'$)(B$'$) condition, if $|x_1 - x'_1| \leq \alpha |y_1 - y'_1|$, then
  \begin{gather*}
  	\mathrm{(A1)}~~
  	|x_2 - x'_2| \leq ( \frac{\alpha \widetilde{\lambda}_s \widetilde{\lambda}_u}{1 - \alpha \widetilde{\beta} } + \widetilde{\alpha} ) |y_2 - y'_2| \leq ( \rho_2 \alpha + \rho_1 \alpha) |y_2 - y'_2| = \rho \alpha |y_2 - y'_2|, \\
  	\mathrm{(A2)}~~
  	|y_1 - y'_1| \leq  \frac{\widetilde{\lambda}_u}{1 - \alpha \widetilde{\beta}} |y_2 - y'_2| = \lambda_s |y_2 - y'_2|.
  \end{gather*}
  The same argument gives the (B) condition.
\end{proof}

\begin{cor}\label{lem:cp}
  If for each $n$, $H_n: X_n \times Y_n \rightarrow X_{n+1} \times Y_{n+1}$ satisfies the (A)$(\alpha, \lambda_u)$ (B)$(\beta, \lambda_s)$ condition, and $\lambda_s \lambda_u < 1, \alpha \beta < \frac{1}{4}$, then for large $n \in \mathbb{N}$, there exist $\beta_1, \lambda'_s$, and $ \rho < \frac{1}{2}$, such that $H_n \circ \cdots \circ H_1$ satisfies the (A)$(\alpha, \lambda^{n}_u)$ (B)$(\beta_1; \rho \beta_1, \lambda'_s)$ condition. Furthermore, $\alpha \beta_1 < 1,~ \lambda'_s \lambda_u < 1$. If $\lambda_s < 1$, then we can also take $\lambda'_s < 1$.
\end{cor}
\begin{proof}
  Note that $H_n \circ \cdots \circ H_1$ satisfies the (A)$(\alpha, \lambda^{n}_u)$ (B)$(\beta, \lambda^{n}_s)$ condition by \autoref{lem:a1}, and $(\lambda_s \lambda_u)^n \rightarrow 0$, as $n \rightarrow \infty$.
  Since $\alpha \beta < \frac{1}{4}$, we see $1-(\alpha\beta)^{1/2} < 1-2\alpha\beta$.
  Choose large $n$ such that
  \[
  (\lambda_s \lambda_u)^{n/2} < \min\{1 - (\alpha \beta)^{1/2}, ~ 1/2\}.
  \]
  Let $\rho_2$ satisfy
  \[
  \frac{(\lambda_s \lambda_u)^n}{1-2\alpha \beta} < \rho_2 < \frac{(\lambda_s \lambda_u)^n}{1 - (\alpha \beta)^{1/2}}~ (< (\lambda_s \lambda_u)^{n/2}).
  \]
  Choose $\rho_1$ such that
  \[
  (1-\rho^{-1}_2 (\lambda_s \lambda_u)^n <)~ \frac{\alpha \beta}{1-\rho^{-1}_2 (\lambda_s \lambda_u)^n} < \rho_1 < \frac{1}{2}~ (< 1 - \rho_2).
  \]
  Now $\rho = \rho_1 + \rho_2 < \frac{1}{2}$ can be satisfied if $n$ is large. Then apply \autoref{lem:a4} to finish the proof.
\end{proof}

Next we give some maps which can satisfy the (A) (B) condition. Let $H=(f,g): X_1 \times Y_1 \rightarrow X_2 \times Y_2$ be a map, and let $g_x(\cdot) \triangleq g(x, \cdot): Y_2 \rightarrow Y_1$ be invertible. Then $H \sim (F,G)$, where $G(x,y) = g_x^{-1}(y),~F(x,y) = f(x,G(x,y))$. Note that if $\sup_{y}\lip g(\cdot, y) < \infty$ and $\sup_{x} \lip G(x,\cdot) < \infty$, then
\[
\sup_{y} \lip G(\cdot, y) \leq \sup_{y}\lip g(\cdot, y) \sup_{x} \lip G(x,\cdot).
\]

The following lemma is a direct consequence of \autoref{lem:a3}.
\begin{lem}
  Let $H=(f,g): X_1 \times Y_1 \rightarrow X_2 \times Y_2$ be a map, and let $g_x(\cdot) \triangleq g(x, \cdot): Y_1 \rightarrow Y_2$ be invertible. Suppose
  \begin{enumerate}[(a)]
    \item $\sup_x\lip f(x, \cdot) \leq \epsilon_1$, $\sup_x \lip g^{-1}_x(\cdot) \leq \lambda'_u$,
    \item $\sup_y\lip f(\cdot, y) \leq \lambda'_s$, $\sup_y\lip g(\cdot, y) \leq \epsilon_2$.
  \end{enumerate}
  Then $H$ satisfies the (A$'$)$(\widetilde{\alpha}, \widetilde{\lambda}_{u})$ (B$'$)$(\widetilde{\beta}, \widetilde{\lambda}_{s})$ condition, where $\widetilde{\alpha} = \epsilon_1 \lambda'_u, ~\widetilde{\beta} = \epsilon_2 \lambda'_u, ~ \widetilde{\lambda}_s = \lambda'_s + \epsilon_1 \lambda'_u, ~ \widetilde{\lambda}_u = \lambda'_u$.
  In particular, if $\lambda'_s \lambda'_u < 1$, $\epsilon_1 < \frac{1-\lambda'_s \lambda'_u}{(\lambda'_u)^2}$, $\epsilon_2 \leq \frac{ ( 1 - \sqrt{ (\lambda'_s + \epsilon_1)\lambda'_u } )^2 }{ 1-\lambda'_s \lambda'_u }$, then there exist $\alpha, \beta, \lambda_s, \lambda_u$, such that $H$ satisfies the (A)$(\alpha, \lambda_u)$ (B)$(\beta, \lambda_s)$ condition, and $\alpha \beta < 1$, $\lambda_s \lambda_u < 1$.
  Furthermore $\alpha, \beta \rightarrow 0$, $\lambda_s \rightarrow \lambda'_s$, $\lambda_u \rightarrow \lambda'_u$, as $\epsilon_1, \epsilon_2 \rightarrow 0$.
\end{lem}

Consider a special case. Assume $H=(f,g): X \times Y \rightarrow X \times Y$, $f(x,y) = A_s(x) + f'(x,y)$, $g(x,y) = A_u(y) + g'(x,y)$, where $X, Y$ are Banach spaces, $A_s: X \rightarrow X$ is Lipschitz, and $A_u: Y \rightarrow Y$ is invertible. If $\lip (A_s) \lip (A^{-1}_u) < 1$, and $\lip f'$, $\lip g'$ are small, then $H$ satisfies the (A)$(\alpha, \lambda_u)$ (B)$(\beta, \lambda_s)$ condition, for some $\alpha, \beta, \lambda_s, \lambda_u$, such that $\alpha \beta < 1$, $\lambda_s \lambda_u < 1$. Furthermore $\lambda_s \rightarrow \lip (A_s)$, $\lambda_u \rightarrow \lip (A^{-1}_u)$ as $\lip f'$, $\lip g' \rightarrow 0$.

The following lemma's assumption was used in \cite{Cha04} to obtain hyperbolicity. We will show that this condition implies the (A) (B) condition, so our main results will recover \cite{Cha04}.
\begin{lem}\label{lem:ac}
  Let $H=(f,g): X_1 \times Y_1 \rightarrow X_2 \times Y_2$ be a map, and let $g_x(\cdot) \triangleq g(x, \cdot): Y_1 \rightarrow Y_2$ be invertible. Assume $\lip f \leq k_0$ in $d_{\infty}$, i.e.,
  \[
  |f(x_1, y_1) - f(x'_1, y'_1)| \leq k_0 \max\{| x_1 - x'_1 |, | y_1 - y'_1 |\},~ \forall (x_1, y_1), (x'_1, y'_1) \in X_1 \times Y_1,
  \]
  and $\sup_{x}\lip g_x^{-1}(\cdot) \leq v_0, ~\sup_{y}\lip g_{(\cdot)}^{-1}(y) \leq \mu_0$.
  If $\mu_0 + v_0 k_0 < 1$, then $H$ satisfies the (A)$(\alpha, \lambda_u)$ (B)$(\beta, \lambda_s)$ condition, and $\lambda_s \lambda_u < 1$, where
  $$
  \alpha = \frac{ 2 v_0 k_0 }{ 1+\sqrt{1-4 k_0 v_0 \mu_0} }~ (<1), ~ \beta = \frac{ \mu_0 }{ 1-k_0 v_0 }~(< 1), ~ \lambda_s = k_0, ~ \lambda_u = \frac{ v_0 }{ 1-\alpha \mu_0 }.
  $$
  Moreover, if $v_0 + \mu_0 < 1$, then $\lambda_u < 1$; if $k_0 < 1$, then $\lambda_s < 1$.
\end{lem}
\begin{proof}
  Let $(x_1, y_1) \times (x_2, y_2),~ (x'_1, y'_1) \times (x'_2, y'_2) \in \graph H$. The lemma's assumption says
  $$
  |x_2 - x'_2| \leq k_0 \max\{ |x_1 - x'_1|, |y_1 - y'_1| \},~
  |y_1 - y'_1| \leq v_0 |y_2 - y'_2| + \mu_0 |x_1 - x'_1|.
  $$

  (B) condition. Let $\beta = \frac{ \mu_0 }{ 1-k_0 v_0 }$. Note that $\beta < 1$, since $\mu_0 + v_0 k_0 < 1$. Let $|y_2 - y'_2| \leq \beta |x_2 - x'_2|$. Then
  \begin{align*}
  &|y_2 - y'_2|
  \leq \beta k_0 \max\{ |x_1 - x'_1|, |y_1 - y'_1| \} \\
  \leq &
  \begin{cases}
  \beta k_0 |x_1 - x'_1|, &~\text{if}~|y_1 - y'_1| \leq |x_1 - x'_1|, \\
  \begin{aligned}
  \beta k_0 \max\{ |x_1 - x'_1|, & v_0 |y_2 - y'_2| + \mu_0 |x_1 - x'_1| \} \\
  & \leq \frac{\beta k_0 \mu_0}{1-\beta k_0 v_0} |x_1 - x'_1|,
  \end{aligned} 
  &~\text{if}~|x_1 - x'_1| \leq |y_1 - y'_1|.
  \end{cases}
  \end{align*}
  Thus,
  \begin{align*}
  & |y_1 - y'_1|
  \leq v_0 |y_2 - y'_2| + \mu_0 |x_1 - x'_1| \\
  \leq &
  \begin{cases}
  (\beta k_0 v_0 + \mu_0) |x_1 - x'_1| = \beta |x_1 - x'_1|,&~\text{if}~|y_1 - y'_1| \leq |x_1 - x'_1|, \\
  \begin{aligned}
  & (v_0 \frac{\beta k_0 \mu_0}{1-\beta k_0 v_0} + \mu_0) |x_1 - x'_1| \\
  & = \frac{\mu_0}{1-\beta k_0 v_0} |x_1 - x'_1| \leq \beta |x_1 - x'_1|, 
  \end{aligned}
  &~\text{if}~|x_1 - x'_1| \leq |y_1 - y'_1|,
  \end{cases}
  \end{align*}
  i.e., $|y_1 - y'_1| \leq \beta |x_1 - x'_1|$. Moreover,
  \begin{align*}
  |x_2 - x'_2| \leq & k_0 \max\{ |x_1 - x'_1|, |y_1 - y'_1| \} \\
  \leq & k_0 \max\{ |x_1 - x'_1|, \beta|x_1 - x'_1| \}  \leq k_0 |x_1 - x'_1|.
  \end{align*}

  (A) condition. Let $|x_1 - x'_1| \leq \alpha |y_1 - y'_1|$. Then $|x_1 - x'_1| \leq \frac{\alpha v_0}{1-\alpha \mu_0} |y_2 - y'_2|$. So
  \begin{align*}
  |y_1 - y'_1| \leq v_0 |y_2 - y'_2| + \mu_0 |x_1 - x'_1|
  \leq \frac{v_0}{1-\alpha \mu_0} |y_2 - y'_2|,
  \end{align*}
  and
  \begin{align*}
  |x_2 - x'_2| \leq & k_0 \max\{ |x_1 - x'_1|, |y_1 - y'_1| \} \\
  \leq & k_0 \max\{ \frac{\alpha v_0}{1-\alpha \mu_0}, \frac{v_0}{1-\alpha \mu_0} \}|y_2 - y'_2| \\
  \leq & \alpha |y_2 - y'_2|.
  \end{align*}
  $\alpha$ should satisfy $\frac{\alpha k_0 v_0}{1-\alpha \mu_0} \leq \alpha, ~\frac{k_0 v_0}{1-\alpha \mu_0} \leq \alpha$. Let us choose an appropriate value for $\alpha$. Under the condition $\mu_0 + v_0 k_0 < 1$, we have $2v_0 k_0 < 1 + (1 - 2 \mu_0)$ and $(1 - 2 \mu_0)^2 \leq 1 - 4 k_0 v_0 \mu_0$.
  We choose
  \[
  \alpha  = \frac{1-\sqrt{1-4\mu_0 v_0 k_0}}{2\mu_0} = \frac{2v_0 k_0}{1+\sqrt{1-4\mu_0 v_0 k_0}} \leq \frac{2v_0 k_0}{1+|1 - 2\mu_0|} < 1.
  \]
  So $\frac{k_0 v_0}{1-\alpha \mu_0} = \alpha < 1$ and the proof is complete.
\end{proof}

\section{Relation between the (A)(B) condition and the classical cone condition: $C^1$ case} \label{relAB1}

The cone condition was widely used in invariant manifold theory in classical references, e.g. \cite{BJ89,BLZ98,MS88,LYZ13,Zel14}.
Some ideas in the proof of the next theorem are motivated by \cite{LYZ13}. Since we have (b) $  \Rightarrow  $ (a) in the next theorem, our main results recover the results on invariant foliations obtained in \cite{LYZ13}.
Also, the abstract results on the existence and regularity of inertial manifolds in \cite{MS88} are consequences of our main results.

\begin{thm}\label{lem:c1}
  Assume that $X_i, Y_i$, $i=1,2$, are Banach spaces and $H \sim (F,G): X_1 \times Y_1 \rightarrow X_2 \times Y_2$ is a correspondence. Moreover, suppose $F,G \in C^{1}$ and $\alpha \beta < 1$.
  Then the following statements are equivalent:
  \begin{enumerate}[(a)]
    \item $H$ satisfies the (A)$(\alpha; \alpha', \lambda_u)$ condition and $\sup_{y} \lip G(\cdot, y) \leq \beta$.
    \item For every $(x_1,y_2) \in X_1 \times Y_2$, $(DF(x_1,y_2), DG(x_1, y_2))$ satisfies the (A)$(\alpha; \alpha', \lambda_u)$ condition and $|D_1G(x_1,y_2)| \leq \beta$.
  \end{enumerate}
\end{thm}
\begin{proof}
  First note that $\sup_{y} \lip G(\cdot, y) \leq \beta$ if and only if $\sup_{(x_1,y_2)}|D_1G(x_1,y_2)| \leq \beta$. (See also \autoref{length} for a general result.)

  (a) $\Rightarrow$ (b). Let $(x_1,y_1) \times (x_2, y_2) \in \graph H$. Let $y_0 = DG(x_1,y_2)(x_0, y'_0)$ and $x'_0 = DF(x_1, y_2)(x_0, y'_0)$ with $|x_0| \leq \alpha |y_0|$. There is no loss of generality in assuming $y_0 \neq 0$. We need to show (A1)~$|y_0| \leq \lambda_u |y'_0|$ and (A2)~$|x'_0| \leq \alpha' |y'_0|$.

  Take a linear $L: Y_1 \rightarrow X_1$ such that $Ly_0 = x_0$, $|L| \leq \alpha$. Take further $\widetilde{L} \in C^{1}(Y_1 \rightarrow X_1)$ such that $\widetilde{L}(y_1) = x_1$, $D\widetilde{L}(y_1) = L$ and $\lip \widetilde{L} \leq \alpha$.
  Let $y^* \in (Y_1)^*$ be such that $y^*(y_0) = |y_0|$ and $|y^*| \leq 1$, and set
  \[
  Ly = \frac{y^*(y)}{|y_0|} x_0,~ \widetilde{L}(y) = \frac{y^*(y-y_1)}{|y_0|} x_0 + x_1 = x_1 + L(y - y_1).
  \]

  Since $\alpha \beta < 1$, there exists a unique $\widehat{y}(\cdot)$ such that
  \begin{equation}
    G(\widetilde{L} \widehat{y} (y), y) = \widehat{y}(y),~
    F(\widetilde{L} \widehat{y} (y), y) \triangleq g(y).  \tag{*}
  \end{equation}
  Furthermore, by the (A) condition, we have $\lip \widehat{y} \leq \lambda_u$ (so $ \sup_{y}|D\widehat{y}(y)| \leq \lambda_u$), $\lip g \leq \alpha'$ (so $\Rightarrow \sup_{y}|D{g}(y)| \leq \alpha'$), and $\widehat{y}(y_2) = y_1$, $g(y_2) = x_2$.

  Since $F, G, \widetilde{L} \in C^1$, we have (taking the derivative of (*) at $y_2$)
  \[
  DG(x_1, y_2) (LD\widehat{y}(y_2), \id) = D\widehat{y}(y_2),~
  DF(x_1, y_2) (LD\widehat{y}(y_2), \id) = Dg(y_2).
  \]
  Since $y_0 = DG(x_1,y_2)(x_0, y'_0), ~x'_0 = DF(x_1, y_2)(x_0, y'_0), ~ Ly_0 = x_0$, we see $D\widehat{y}(y_2)y'_0 = y_0$ and $Dg(y_2)y'_0 = x'_0$. Thus, $|y_0| \leq \lambda_u |y'_0|$, $|x'_0| \leq \alpha' |y'_0|$.

  (b) $\Rightarrow$ (a). Let $(x_1, y_1) \times (x_2, y_2),~ (x'_1, y'_1) \times (x'_2, y'_2) \in \graph H$ and $|x_1 - x'_1| \leq \alpha |y_1 - y'_1|$. We need to show $|x_2 - x'_2| \leq \alpha' |y_2 - y'_2|$, and $|y_1 - y'_1| \leq \lambda_u |y_2 - y'_2|$. Without loss of generality, $|y_1 - y'_1| \neq 0$.

  Let $\widehat{L} \in C^1(Y_1 \rightarrow X_1)$ be such that $\lip \widehat{L} \leq \alpha$, $\widehat{L}y_1 = x_1$ and $\widehat{L}y'_1 = x'_1$. Let $y^* \in (Y_1)^*$ such that $y^*(y_1 - y'_1) = |y_1 - y'_1|$ and $|y^*| \leq 1$, and set
  \[
  \widehat{L}y = \frac{y^*(y - y'_1)}{|y_1 - y'_1|} (x_1 - x'_1) + x'_1.
  \]
  Since $\alpha \beta < 1$, we also have a unique $\widehat{y}(\cdot)$ such that
  \begin{equation}
    G(\widehat{L} \widehat{y} (y), y) = \widehat{y}(y),~
    F(\widehat{L} \widehat{y} (y), y) \triangleq \widehat{g}(y).  \tag{**}
  \end{equation}
  Furthermore, $\widehat{y} $ is differentiable (see e.g. \autoref{lem:fdiifx}). Thus taking the derivative of (**) at $y$, we obtain
  \[
  DG(x, y) (x_0, y'_0) = D\widehat{y}(y)y'_0,~
  DF(x, y) (x_0, y'_0) = D\widehat{g}(y)y'_0,
  \]
  where $x_0 \triangleq D\widehat{L}(\widehat{y}(y))D\widehat{y}(y)y'_0, ~x \triangleq \widehat{L}\widehat{y}(y)$.
  Since $|x_0| \leq \alpha |D\widehat{y}(y)y'_0|$, by the (A) condition, we have $|D\widehat{g}(y)y'_0| = |DF(x,y)(x_0, y'_0)| \leq \alpha' |y'_0|$ and $|D\widehat{y}(y)y'_0| = |DG(x, y) (x_0, y'_0)| \leq \lambda_u |y'_0|$ for any $y \in Y_2$.
  Thus, $\lip \widehat{y} \leq \lambda_u$ and $\lip \widehat{g} \leq \alpha'$.
  Since $\widehat{L}y_1 = x_1, ~\widehat{L}y'_1 = x'_1$, we have $\widehat{y}(y_2) = y_1$, $\widehat{y}(y'_2) = y'_1$, $\widehat{g}(y_2) = x_2$, $\widehat{g}(y'_2) = x'_2$. So the (A) condition holds and the proof is complete.
\end{proof}

We use the \emph{notation}
\begin{enumerate}[$ \bullet $]
	\item $X(r) \triangleq \{ x \in X: |x| <  r\}$ when $X$ is a Banach space.
\end{enumerate}
Consider some local results. First note that if $f: X(r) \rightarrow Y$ is Lipschitz, where $Y$ is a complete metric space, and $r< \infty$, then there is a unique $\widetilde{f}: \overline{X(r)} \rightarrow Y$ such that $\widetilde{f}|_{X(r)} = f$. Moreover, $\widetilde{f}$ is Lipschitz.
\begin{lem}\label{lem:cAB}
	Let $H \sim (F,G): X_1(r_1) \times Y_1(r'_1) \rightarrow X_2(r_2) \times Y_2 (r'_2)$ be a correspondence with $ F, G \in C^1 $. Consider the following conditions:
	\begin{enumerate}[(1)]
		\item $ (F,G) $ satisfies the (A)($\alpha; \alpha', \lambda_u$) condition and $\sup_{y_2}\lip G(\cdot, y_2) \leq \beta$. 
		\item For any $ (x_1, y_2) \in X_1(r_1) \times Y_2(r'_2) $, $(DF(x_1, y_2), DG(x_1, y_2))$ satisfies the (A)($\alpha; \alpha', \lambda_u$) condition and $ |D_1G(x_1, y_2)| \leq \beta $.
	\end{enumerate}
	Then (1) $\Rightarrow$ (2) if $\alpha \beta < \frac{1}{2}$; and (2) $\Rightarrow$ (1) if $\alpha \beta < 1$.
\end{lem}
\begin{proof}
  In the proof of \autoref{lem:c1} (a) $\Rightarrow$ (b), the major construction is the map $\widetilde{L}$ which is differentiable at $y_1$. Thus, some truncation is needed. Consider $\widetilde{L}(y) = x_1 + L(r_\varepsilon (y - y_1))$, where $x_1 \in X_1(r_1), ~y_1 \in Y_1(r'_1)$, $\varepsilon$ is small such that $|L|\varepsilon < r_1 - |x_1|$, and $r_\varepsilon$ is the \emph{radial retraction}, i.e.,
  \begin{equation}\label{radial}
  r_\varepsilon(x) =
  \begin{cases}
  x, &~ \text{if}~ |x| \leq \varepsilon, \\
  \varepsilon x / |x|, & ~ \text{if}~ |x| \geq \varepsilon.
  \end{cases}
  \end{equation}
  Now $\widetilde{L}: Y_1(r'_1) \rightarrow X_1(r_1)$, and it is differentiable at $y_1$. But $\lip \widetilde{L} \leq 2 \alpha$. That is why we need $\alpha \beta < \frac{1}{2}$. Using the same argument line for line as in the proof of \autoref{lem:c1} (a) $\Rightarrow$ (b), one gets (1) $\Rightarrow$ (2) if $\alpha \beta < \frac{1}{2}$.

  In the proof of \autoref{lem:c1} (b) $\Rightarrow$ (a), a key construction is the existence of $\widehat{L}$, which also makes sense in the local case. So the proof is complete.
\end{proof}

Using the (A$'$) (B$'$) condition, one can easily give more local results. The proof of the following lemma is straightforward (by using \autoref{lem:a3}), so we omit it.

\begin{lem}\label{lem:general}
  Let $(X, M, \pi_1), (Y, M, \pi_2)$ be bundles with fibers being Banach spaces, and $u: M \rightarrow M$ a map.
  Let $H: X \times Y \rightarrow X \times Y$ be a bundle correspondence over $u$ with generating bundle map $(F,G)$. Assume that $F_m(\cdot), G_m(\cdot) \in C^1$, and that the following uniform conditions hold:
\begin{enumerate}[(a)]
  \item $\sup_m |D_1F_m(0,0)| \leq \epsilon_1$, $\sup_m |D_2G_m(0,0)| \leq \lambda'_u(m)$.
  \item  $\sup_m |D_2F_m(0,0)| \leq \lambda'_s(m)$, $\sup_m |D_1G_m(0,0)| \leq \epsilon_2$.
  	\item $DF_m, DG_m$ are almost continuous at $(0,0)$ uniformly for $m$ in the following sense:
 	 \[
 	 \begin{gathered}
 	 \sup_m \sup_{|x|\leq r_1, |y| \leq r'_2} |DF_m(x,y) - DF_m(0,0)| \leq \epsilon'_1, \\
 	 \sup_m \sup_{|x|\leq r_1, |y| \leq r'_2} |DG_m(x,y) - DG_m(0,0)| \leq \epsilon'_2,
 	 \end{gathered}
 	 \]
	and $\epsilon'_1, \epsilon'_2 $ are sufficiently small as $r_1, r'_2 \rightarrow 0$.
  	\item $ \sup_m |F_m(0,0)| \leq \eta $, $ \sup_m |G_m(0,0)| \leq \eta $.
\end{enumerate}
If $ \sup_m \lambda'_s(m) \lambda'_u(m) < 1 $, and $ \eta $ is small, then there exist constants $ \alpha, \beta $, $ r_i, r'_i $, $ i = 1, 2 $, such that
\[
H_m \sim (F_m, G_m): X_m(r_1) \times Y_m(r'_1) \rightarrow X_{u(m)}(r_2) \times Y_{u(m)}(r'_2)
\]
satisfies the (A)$(\alpha, \lambda_u(m))$ (B)$(\beta, \lambda_s(m))$ condition,
where $ \lambda_s(m) = \lambda'_s(m) + \varsigma $, $ \lambda_u(m) = \lambda'_u(m) + \varsigma $, $ \varsigma > 0 $ is sufficiently small, and $\alpha, \beta \rightarrow 0$ as $\epsilon_1, \epsilon_2, r_1, r'_2, \eta \rightarrow 0$.
\end{lem}

\begin{cor}\label{lem:exampleH}
  Let $H: X \rightarrow X$ and $u: M \rightarrow M$ be maps, where $M \subset X$ and $X$ is a Banach space. Let $ \Pi^{s}_{m} $, $ m \in M $, be projections and $ \Pi^{u}_{m} = \id - \Pi^{s}_{m} $; $ X^{\kappa}_{m} = R(\Pi^{\kappa}_{m}) $, $ m \in M $, $ \kappa = s, u $. Suppose for every $m \in M$, $DH(m): X^s_m \oplus X^u_m \rightarrow X^s_{u(m)} \oplus X^u_{u(m)}$ satisfies the following:
  \begin{enumerate}[(a)]
    \item $ \Pi^{u}_{u(m)} DH(m): X^u_{m} \rightarrow X^u_{u(m)} $ is invertible; write
    \[
    \|\Pi^{u}_{u(m)} DH(m)|_{X^s_m}\| = \lambda'_s(m),~\|(\Pi^{u}_{u(m)}DH(m)|_{X^u_m})^{-1}\| = \lambda'_u(m);
    \]
    \item $ \|\Pi^{\kappa_1}_{u(m)} DH(m) \Pi^{\kappa_2}_{m}\| \leq \xi $, $ \kappa_1 \neq \kappa_2 \in \{s, u\} $;
    \item $ \sup_{m \in M} \|\Pi^{s}_{m}\| < \infty $, $ \sup_{m \in M}\lambda'_u(m) < \infty $;
    \item $DH$ is almost uniformly continuous around $ M $, meaning that the amplitude $ \mathfrak{A}(\epsilon) $ of $ DH $ in $ \mathbb{B}_{\epsilon}(M) = \{ x: d(x, M) < \epsilon \} $ can be sufficiently small when $ \epsilon \to 0 $;
    \item $|H(m) - u(m)| \leq \eta$.
  \end{enumerate}
  If $\eta, \xi$ are small and $\sup_m\lambda'_s(m) \lambda'_u(m) < 1$, then
  \begin{enumerate}[(1)]
  	\item there are small $r_i, r'_i$, $i = 1,2$, and maps $ F_m, G_m $, such that
  	\[
  	H_{m} \triangleq H(m+ \cdot) - u(m) \sim (F_m, G_m): X^s_m(r_1) \oplus X^u_m(r'_1) \rightarrow X^s_{u(m)}(r_2) \oplus X^u_{u(m)}(r'_2),
  	\]
  	satisfies the (A)$(\alpha, \lambda_u(m))$ (B)$(\beta, \lambda_s(m))$ condition,
  	where $\lambda_s(m) = \lambda'_s(m) + \varsigma$, $\lambda_u(m) = \lambda'_u(m) + \varsigma$, $ \varsigma > 0 $ is sufficiently small, and $\alpha, \beta \rightarrow 0$ as $r_1, r'_2, \eta, \xi \rightarrow 0$ and $ \mathfrak{A}(\epsilon) \to 0 $;
  	\item $ |F_{m}(0 ,0)| \leq K_1 \eta $ and $ |G_{m}(0 ,0)| \leq K_1 \eta $ for some $ K_1 > 0 $ independent of $ m ,\eta $;
  	\item \label{it:H3} there are maps $ \widetilde{F}_{m}, \widetilde{G}_{m} $ defined in all $ X^{s}_{m} \times X^{u}_{u(m)} $, $ m \in M $, such that 
  	\[
  	\widetilde{F}_{m}|_{X^s_m(r_1) \oplus X^u_{u(m)}(r'_2)} = F_{m} ~\text{and}~\widetilde{G}_{m}|_{X^s_m(r_1) \oplus X^u_{u(m)}(r'_2)} = G_{m}
  	\]
  	with $ \widetilde{H}_{m} \sim (\widetilde{F}_{m}, \widetilde{G}_{m}) $ satisfying the (A)$(\alpha, \lambda_u(m))$ (B)$(\beta, \lambda_s(m))$ condition.
  \end{enumerate}
\end{cor}
\begin{proof}
	Set $ \lambda^*_{u} = \sup_{m \in M}\lambda'_u(m) $, $ C_1 = \sup_{m \in M} \{\|\Pi^{s}_{m}\|, \|\Pi^{u}_{m}\|\} $, $ A_{m} = \Pi^{u}_{u(m)}DH(m)|_{X^u_m} $, and
	\[
	\begin{split}
	(f_{m}(x, y), g_{m}(x, y)) = H(m + & x + y) - u(m): \\
	& X^s_m \oplus X^u_m \rightarrow X^s_{u(m)} \oplus X^u_{u(m)}, ~(x, y) \in X^s_m \oplus X^u_m.
	\end{split}
	\]
	Here $ g_{m}(x, y) = \Pi^{u}_{u(m)} \{ H(m + x + y) - u(m) \} $.
	We first show there is a small $ r_1 > 0 $ (independent of $ m $) such that $ g^{-1}_{m}(x, \cdot)|_{X^{u}_{u(m)}(r_1)} $ exists for $ |x| \leq r_1 $, which is very standard by using e.g. the Banach fixed point theorem.
	By (d), 
	\[
	\mathfrak{A}(r) \triangleq \sup \{ \| DH(m_1) - DH(m_2) \|: |m_1 - m_2| \leq r, m_1, m_2 \in \mathbb{B}_{r}(M) \}
	\]
	is small when $ r \to 0 $.
	So there is a small $ r > 0 $ such that for all $ |x|, |y| \leq r$ and $ m \in M $,
	\[
	\| D_2g_{m}(x, y) - D_2g_{m}(0, 0) \| \leq \|\Pi^{u}_{u(m)}\| \cdot \| DH(m+x+y) - DH(m) \| \leq (\lambda^*_{u})^{-1} / 4. \tag{$ \ast $} \label{equ:pp}
	\]
	Let $ \eta \leq (C_1\lambda^*_{u})^{-1} r/4 $. Consider the following function:
	\[
	h_{m}(x, y, z) \triangleq -A^{-1}_{m}g_{m}(x, y) + y + z, ~ y, z \in X^{u}_{m}.
	\]
	By \eqref{equ:pp}, we then see that for $ |x| \leq r $ and $ y_1, y_2 \in X^{u}_{m}(r) $,
	\begin{align*}
	& |h_{m}(x, y_1, z) - h_{m}(x, y_2, z)| \\
	= & |-A^{-1}_{m} \int_{0}^{1} \{D_2g_{m}(x, sy_1 + (1-s)y_2) - D_2g_{m}(0, 0)\} ~\mathrm{d} s~ (y_1 - y_2)| \leq 1/4 |y_1 - y_2|,
	\end{align*}
	and for $ |x|, |y| \leq r $ and $ |z| \leq r/2 $,
	\begin{align*}
	|h_{m}(x, y, z)| & = | -A^{-1}_{m} \Pi^{u}_{u(m)} \{ H(m + x + y) - H(m) + H(m) - u(m) \} + y + z| \\
	& \leq |-A^{-1}_{m} \Pi^{u}_{u(m)} \{ H(m) - u(m) \}| \\
	& \quad + |-A^{-1}_{m} \int_{0}^{1}\{D_2g_{m}(x, sy) - D_2g_{m}(0, 0)\}~\mathrm{d}sy| + |z| \\
	& \leq r / 4 + r / 4 + |z| \leq r.
	\end{align*}
	Therefore, for each $ x,z $ such that $ |x| \leq r $ and $ |z| \leq r/2 $, there is a unique $ y(x, z) \in X^{u}_{m}(r) $ such that $ h_{m}(x, y(x, z), z) = y(x, z) $. Define $ G_{m}(x, z) = y(x, A^{-1}_{m}z) $, $ z \in X^{u}_{u(m)}((\lambda^*_{u})^{-1}/2 \cdot r) $, and $ r_1 =\min\{ r,  (\lambda^*_{u})^{-1}/2 \cdot r\} $. We find that $ g^{-1}_{m}(x, \cdot)|_{X^{u}_{u(m)}(r_1)} $ exists, and for $ |x| \leq r_1 $,
	\[
	g_{m}(x, G_{m}(x, z)) = z, ~ z \in X^{u}_{u(m)}(r_1).
	\]
	Let $ F_{m}(x, z) = f_{m}(x, G_{m}(x, z)) $. Then for some $ r_2 > 0 $, we have
	\[
	H_{m} \sim (F_{m}, G_{m}): X^{s}_{m}(r_1) \oplus X^{u}_{m}(r) \to X^{s}_{u(m)}(r_2) \oplus X^{u}_{u(m)}(r_1).
	\]
	Moreover, $ |G_{m}(0, 0)| \leq 2\lambda^*_{u}C_1 \eta $, and
	\[
	\lip G_{m}(\cdot, z) \leq (\epsilon + \xi)(1 - \epsilon)^{-1}, ~\lip G_{m}(x, \cdot) \leq (\epsilon + \lambda_{u}(m))(1 - \epsilon)^{-1},
	\]
	where $ \epsilon = \lambda^*_{u}C_1\mathfrak{A}(r) $. Similarly for $ \lip F_{m}(\cdot, z) \leq \lambda_{s}(m) + \epsilon_1 $ and $ \lip F_{m}(x, \cdot) \leq 2\lambda^*_{u}\xi + \epsilon_1 $ for some $ \epsilon_1 > 0 $ such that $ \epsilon_1 \to 0 $ as $ r \to 0 $; also $ |F_{m}(0, 0)| \leq K_1\eta $ for some $ K_1 > 0 $ (independent of $ \eta, m $).
	Define
	\begin{multline}\label{equ:rad}
	\widetilde{F}_{m}(x, y) = DF_{m}(0, 0)(x, y) \\
	+ F_{m}(0, 0) + (F_{m}(\cdot, \cdot) - F_{m}(0, 0) - DF_{m}(0, 0)) \circ (r_{r_1}(x), r_{r_1}(y)),
	\end{multline}
	and similarly for $ \widetilde{G}_{m}(x, y) $, where $ r_{\varepsilon}(\cdot) $ is the radial retraction given by \eqref{radial}; let $ \widetilde{H}_{m} \sim (\widetilde{F}_{m}, \widetilde{G}_{m}) $.
	Using \autoref{lem:a3}, we have the (A) (B) condition for $ H_{m} $ and $ \widetilde{H}_{m} $ as given in the corollary.
	The proof is complete.
\end{proof}

The above corollary shows how this paper's main results can be applied to some classical settings; see e.g. \autoref{foliations}. Here, $ X $ can be a Riemannian manifold having {bounded geometry} (see e.g. \autoref{defi:bounded}), or a uniformly regular Riemannian manifold on $ M $ (see e.g. \autoref{UR}) including the smooth compact Riemannian manifold with $ M $ being far from the boundary, or more generally, $ X $ is $ C^{0,1} $-uniform around $ M $ (see assumption $ (\blacksquare) $ on \autopageref{C1MM}).

To apply our existence results in \autoref{invariant} to the local case such as the one given in the above results, one usually needs to use the radial retraction \eqref{radial} to truncate the generating maps so that the local conditions can become global (e.g. \autoref{lem:exampleH} \eqref{it:H3}); and to apply the regularity results in \autoref{stateRegularity}, instead one needs to use smooth bump functions or \emph{blid maps} (see \autoref{bump}).
If $ \sup_m\lambda_s(m) < 1 $ (or $ \sup_m\lambda_u(m) < 1 $), \autoref{thm:local} and \autoref{thm:LocalRegularity} can be applied directly in some contexts.

%% file: sect3.tex
\chapter{Existence of Invariant Graphs}\label{graph}

Our existence results are stated in \autoref{invariant} and the proofs are presented in \autoref{graphTrans} and \autoref{proofs}. More characterizations and corollaries are given in \autoref{properties} and \autoref{corollaries0}.

\section{Invariant graphs: statements of the discrete case}\label{invariant}
\emph{For convenience, we write the metric $d(x,y)$ as $|x-y|$}. A bundle with metric fibers means its fibers are complete metric spaces.

\begin{thm}[First existence theorem]\label{thmA}
	Let $(X, M, \pi_1), (Y, M, \pi_2)$ be bundles with metric fibers and $u: M \rightarrow M$ a map. Let $H: X \times Y \rightarrow X \times Y$ be a bundle correspondence over $u$ with generating bundle map $(F,G)$. Take $ \varepsilon_1(\cdot): M \to \mathbb{R}_+ $.
	Assume that the following assumptions hold:
	\begin{enumerate}[(i)]
		\item ($\varepsilon$-pseudo-stable section) $ i = (i_X, i_Y): M \to X \times Y $ is an $\varepsilon$-\textbf{pseudo-stable section} of $H$, i.e.,
		\[
		\begin{gathered}
		| i_X(u(m)) - F_m( i_X(m), i_Y(u(m)) ) | \leq \eta(u(m)), \\
		| i_Y(m) - G_m( i_X(m), i_Y(u(m)) ) | \leq \eta(m),
		\end{gathered}
		\]
		where $ \eta: M \rightarrow \mathbb{R}_+ $, with $ \eta(u(m)) \leq \varepsilon(m) \eta(m) $ and $ 0 \leq \varepsilon(m) \leq \varepsilon_1(m) $ for $ m \in M $.
		\item $H$ satisfies the (A$'$)$(\alpha, \lambda_u)$ (B)$(\beta; \beta', \lambda_s)$ condition, where $\alpha, \beta, \beta', \lambda_u, \lambda_s$ are functions $M \rightarrow \mathbb{R}_+$, and in addition,
		\begin{enumerate}[(a)]
			\item (angle condition) $ \sup_m \alpha(m) \beta'(u(m)) < 1 $, $ \beta'(u(m)) \leq \beta(m)$ for all $ m \in M $, $ \sup_m \{ \alpha(m), \beta(m) \} < \infty $,
			\item (spectral condition) $ \sup_m \frac{ \lambda_u(m) \lambda_s(m) + \lambda_u(m) \varepsilon_1(m) }{1-\alpha(m) \beta'(u(m))} < 1 $.
		\end{enumerate}
		\item[$\mathrm{(ii)'}$] \textbf{Or} $H$ satisfies the (A)$(\alpha;\alpha', \lambda_u)$ (B)$(\beta;\beta', \lambda_s)$ condition, where $\alpha, \beta, \alpha', \beta', \lambda_u, \lambda_s$ are functions $M \rightarrow \mathbb{R}_+$, and in addition,
		\begin{enumerate}[(a)$'$]
			\item (angle condition) $\sup_m \alpha'(m) \beta'(u(m)) < 1$, $ \alpha'(m) \leq \alpha(u(m))$, $ \beta'(u(m)) \leq \beta(m)$ for all $ m \in M  $, $ \sup_m \{ \alpha(m), \beta(m) \} < \infty $,
			\item (spectral condition) $ \sup_m \lambda_s(m) \lambda_u(m) < 1 $, $ \sup_m \lambda_u(m)\varepsilon_1(m) < 1 $.
		\end{enumerate}
	\end{enumerate}
	Then there is a unique bundle map $ f: X \to Y $ over $\id$ such that the following hold:
	\begin{enumerate}[(1)]
		\item $\lip f_m \leq \beta'(m) $, $| f_m(i_X(m))  - i_Y(m)| \leq K \eta(m)$ for some constant $ K \geq 0 $.
		\item $\graph f \subset H^{-1} \graph f$. More precisely, $ (x_m(x), f_{u(m)}(x_m(x)) ) \in H_m(x, f_m(x)) $ for all $ x \in X_m $, where $x_m(\cdot): X_m \to X_{u(m)} $ is such that 
		\[
		\lip x_m(\cdot) \leq \lambda_s(m), \quad| x_m(i_X(m))  - i_X(u(m))| \leq K_0\eta(u(m))
		\]
		for all $m \in M$ and some constant $ K_0 $.
		\item $f$ does not depend on the choice of $ \eta' = \eta $ as long as $ \eta'(u(m)) \leq \varepsilon'(m) \eta'(m) $, $ \varepsilon'(m) \leq \varepsilon_{1}(m) $ and $ \eta(m) \leq \eta'(m) $, for all $ m \in M $.
	\end{enumerate}
\end{thm}

\begin{rmk}
	There are two cases which frequently occur in practical applications.
	\begin{enumerate}[(1)]
		\item The invariant-stable case: $ \varepsilon_1 = 0 $. In particular, $ i $ is an invariant section of $ H $ (i.e. $\eta \equiv 0$ or $ i(u(m)) \in H_m(i(m)) $, $ m \in M $). We usually apply this case to obtain some invariant foliations; see e.g. \autoref{foliations}.

		\item The bounded-stable case: $ \varepsilon_1 = 1 $. For instance, if $ M $ is a single point set, then all the constant sections belong to this case. Note that in this case, $ \eta $ may not be a bounded function. In particular, $ \sup_m\eta(m) < \infty $.
	\end{enumerate}
\end{rmk}

We use the \emph{notation}
\begin{enumerate}[$ \bullet $]
	\item $\widetilde{d} (A, z) \triangleq \sup_{\tilde{z} \in A} d(\tilde{z}, z) $ where $A$ is a subset of a metric space.
\end{enumerate}

\begin{thm}[Second existence theorem]\label{thmB}
	Let $(X, M, \pi_1), (Y, M, \pi_2)$ be bundles with metric fibers and $u: M \rightarrow M$ a map.
	Let $H: X \times Y \rightarrow X \times Y$ be a bundle correspondence over $u$ with generating bundle map $(F,G)$. Take $ \varepsilon_1(\cdot): M \to \mathbb{R}_+ $.
	Assume that the following statements hold:
	\begin{enumerate}[(i)]
		\item ($\varepsilon$-$Y$-bounded-section) $ i = (i_X, i_Y): M \to X \times Y $ is an $\varepsilon$-$Y$-\textbf{bounded-section} of $H$, i.e.,
		\[
		\widetilde{d}(G_m( X_m, i_Y(u(m)) ) , i_Y(m)) \leq \eta(m),
		\]
		where $ \eta: M \rightarrow \mathbb{R}_+ $ with $ \eta(u(m)) \leq \varepsilon(m) \eta(m) $ and $ 0 \leq \varepsilon(m) \leq \varepsilon_1(m) $ for $ m \in M $.
		\item $H$ satisfies the (A$'$)$(\alpha, \lambda_u)$ (B)$(\beta; \beta', \lambda_s)$ condition, where $\alpha, \lambda_u, \beta, \beta', \lambda_s$ are functions $M \rightarrow \mathbb{R}_+$, and in addition,
		\begin{enumerate}[(a)]
			\item (angle condition) $ \sup_m \alpha(m) \beta'(u(m)) < 1 $, $ \beta'(u(m)) \leq \beta(m)$ for all $ m \in M $, $ \sup_m \{\alpha(m), \beta(m)\} < \infty $,
			\item (spectral condition) $ \sup_m \frac{ \lambda_u(m) \varepsilon_1(m) }{1-\alpha(m) \beta'(u(m))} < 1 $.
		\end{enumerate}
		\item [$\mathrm{(ii)'}$] \textbf{Or} $H$ satisfies the (A)$(\alpha;\alpha', \lambda_u)$ (B)$(\beta;\beta', \lambda_s)$ condition, where $\alpha, \beta, \alpha', \beta', \lambda_u, \lambda_s$ are functions $M \rightarrow \mathbb{R}_+$, and in addition,
		\begin{enumerate}[(a)$'$]
			\item (angle condition) $ \sup_m \alpha'(m) \beta'(u(m)) < 1 $, $ \alpha'(m) \leq \alpha(u(m))$, $\beta'(u(m)) \leq \beta(m) $ for all $ m \in M  $, $ \sup_m \{\alpha(m), \beta(m)\} < \infty $,
			\item (spectral condition) $ \sup_m \lambda_u(m) \varepsilon_1(m) < 1 $.
		\end{enumerate}
	\end{enumerate}
	Then there is a unique bundle map $ f: X \to Y $ over $\id$ such that the following hold:
	\begin{enumerate}[(1)]
		\item $ \lip f_m \leq \beta'(m) $, $\widetilde{d}( f_m(X_m), i_Y(m)) \leq K \eta(m) $, $ m \in M $, where $ K \geq 0 $ is a constant.
		\item $\graph f \subset H^{-1} \graph f$. More precisely, $ (x_m(x), f_{u(m)}(x_m(x)) ) \in H_m(x, f_m(x)) $ for all $ x \in X_m $, where $x_m(\cdot): X_m \to X_{u(m)} $ is such that $\lip x_m(\cdot) \leq \lambda_s(m) $ for all $m \in M$. In addition, if
		\[
		\widetilde{d}(F_m(X_{m}, i_Y(u(m))) , i_X(u(m))) \leq \eta(u(m)),~ m \in M,
		\]
		then $\widetilde{d}( x_m(X_m), i_X(u(m)) ) \leq K_0 \eta(u(m)),~ m \in M$, for some constant $ K_0 $.
		\item $f$ does not depend on the choice of $ \eta' = \eta $ as long as $ \eta'(u(m)) \leq \varepsilon'(m) \eta'(m) $, $ \varepsilon'(m) \leq \varepsilon_{1}(m) $ and $ \eta(m) \leq \eta'(m) $, for all $ m \in M $.
	\end{enumerate}
\end{thm}

\begin{rmk}
	A very special case of \autoref{thmB} is that $ \sup_m\eta(m) < \infty $ (in this case take $ \varepsilon_1 \equiv 1 $). Now $ f $ is unique in the following sense: If $ i' $ is a section such that $ \sup_m d( i_Y(m), i'_Y(m) )  < \infty $, and for $ i' $ there is a bundle map $ f': X \to Y $ over $ \id $ satisfying (a)  $\lip f'_m \leq \beta'(m) $, $\sup_m \widetilde{d}( f'_m(X_m), i'_Y(m) ) < \infty$, and (b) $ \graph f' \subset H^{-1} \graph f' $, then $ f' = f $.

	Suppose $ H $ satisfies condition (ii) or (ii)$ ' $ in \autoref{thmB} with $ \varepsilon_1 \equiv 1 $, and one of the following situations holds:
	\begin{enumerate}[(1)]
		\item $\sup_m \diam X_m < \infty$, and $i = (i_X, i_Y)$ is a section of $X \times Y$ such that 
		\[
		\sup_m |G_m( i_X(m), i_Y(u(m)) ) - i_Y(m)| < \infty;
		\]
		\item $\sup_m \diam Y_m < \infty$, and $i = (i_X, i_Y)$ is any section of $X \times Y$;
		\item $X_m, Y_m$ are Banach spaces, $G_m(x,y) = B_m y + g'_m(x,y)$, where $B_m: Y_{u(m)} \to Y_{m} $ is a linear operator, and $g'_m$ are bounded uniformly for $m \in M$. Let $i = 0 $ be the $0$-section. This case was also studied in e.g. \cite{CY94, CL97}.
	\end{enumerate}
	 Then we get the conclusions in \autoref{thmB}. In (2) or (3), it is obvious that condition (i) in \autoref{thmB} is satisfied. And in (1), use the fact that $\sup_m\sup_{y}\lip G_m(\cdot,y) <\infty$.
\end{rmk}

\begin{rmk}
	It is instructive to see the difference between \autoref{thmA} and \autoref{thmB} in simple settings where $ M $ consists of one element and $ H $ is the time-one solution map of the following equation in $ \mathbb{R}^n $:
	\[
	\dot{z} = Az + h(z),
	\]
	where $ A \in \mathbb{R}^{n \times n} $ with $ \sigma(A) \cap i\mathbb{R} \neq \emptyset $ and $ \lip h $ is sufficiently small. In this case, (a) if $ h(0) = 0 $, then the existence of a stable, center-stable or pseudo-stable manifold of $ 0 $ for $ H $ is a direct consequence of \autoref{thmA}, and (b) if $ h $ is a bounded function (but $ H $ may have no equilibrium), then the existence of a center-stable (or pseudo-stable) manifold for $ H $ can be derived from \autoref{thmB} (or \autoref{thmA} for $ \varepsilon = 1 $); in the latter case, if there is no equilibrium, then $ H $ may have no stable manifold. While, using \autoref{thmB}, one can obtain some non-resonant manifolds due to the fact that no spectral gap condition is needed (see also \cite{Cha02, dlLla97} for the `higher order' case).
\end{rmk}

We state a local version of the existence result in the strong stable case. In this case, we do not need to truncate the system, but the thresholds in the angle condition and spectral condition are a little different from those in \autoref{thmA} and \autoref{thmB}.
\begin{thm}[Third existence theorem]\label{thm:local}
	Let $(X, M, \pi_1), (Y, M, \pi_2)$ be two bundles with fibers being Banach spaces and $u: M \rightarrow M$ a map. Let $ \varepsilon_1(\cdot): M \to \mathbb{R}_+ $ and $ i = 0: M \to X \times Y $.
	For every $ m \in M $, suppose that $ H_m \sim (F_m,G_m): X_m(r_1) \times Y_m(r'_1) \rightarrow X_{u(m)}(r_2) \times Y_{u(m)}(r'_2) $ is a correspondence, where $ r_i, r'_i $, $ i = 1, 2 $, are independent of $ m \in M $, and
	\[
	F_m: X_m(r_1) \times Y_{u(m)}(r'_2) \to X_{u(m)}(r_2), ~G_m: X_m(r_1) \times Y_{u(m)}(r'_2) \to Y_m(r'_1).
	\]

	\begin{enumerate}[(i)]
		\item ($\varepsilon$-pseudo-stable section) $ i $ is an $\varepsilon$-{pseudo-stable section} of $H$, i.e.,
		\[
		| F_m( 0, 0 ) | \leq \eta(u(m)), ~ | G_m( 0, 0 ) | \leq \eta(m),
		\]
		where $ \eta: M \rightarrow \mathbb{R}_+ $, with $ \eta(u(m)) \leq \varepsilon(m) \eta(m) $ and $ 0 \leq \varepsilon(m) \leq \varepsilon_1(m) $ for all $ m \in M $.
		\item For every $ m \in M $, $ H_{m} $ satisfies the (A$ ' $)($ \alpha(m) $, $ \lambda_u(m) $) (B)($ \beta(m) $; $ \beta'(m) $, $ \lambda_s(m) $) condition and assume the following conditions hold:
		\begin{enumerate}[(a)]
			\item (angle condition) $ \sup_m \alpha(m) \beta'(u(m)) < 1 / 2 $, $ \beta'(u(m)) \leq \beta(m)$ for all $ m \in M $, $ \sup_m \{ \alpha(m), \beta(m) \} < \infty $,
			\item (spectral condition) $ \sup_m \frac{ \lambda_u(m) \lambda_s(m) + \lambda_u(m) \varepsilon_1(m) }{1-\alpha(m) \beta'(u(m))} < 1 $, $ \sup_m \lambda_s(m) < 1 $.
		\end{enumerate}

		\item[$\mathrm{(ii)'}$] \textbf{Or} for every $ m \in M $, $ H_{m} $ satisfies the (A)($ \alpha(m) $; $ \alpha'(m) $, $ \lambda_u(m) $) (B)($ \beta(m) $; $ \beta'(m) $, $ \lambda_s(m) $) condition and assume the following conditions hold:
		\begin{enumerate}[(a$ ' $)]
			\item (angle condition) $ \sup_m \alpha'(m) \beta'(u(m)) < 1 / 2 $, $ \alpha'(m) \leq \alpha(u(m)) $, $ \beta'(u(m)) \leq \beta(m) $ for all $ m \in M $, $ \sup_m \{ \alpha(m), \beta(m) \} < \infty $,
			\item (spectral condition) $ \sup_m \lambda_s(m) \lambda_u(m), \sup_m \lambda_u(m)\varepsilon_1(m), \sup_m\lambda_s(m) < 1 $.
		\end{enumerate}
	\end{enumerate}
	If $ \eta_0 > 0 $ is small and $ \sup_m\eta(m) \leq \eta_0 $, then there is a small $ \sigma_0 > 0 $ such that there are unique maps $ f_m: X_m(\sigma_0) \to Y_m $, $ m \in M $, satisfying the following:
	\begin{enumerate}[(1)]
		\item $\lip f_m \leq \beta'(m) $, $| f_m(0) | \leq K \eta(m)$ for some constant $ K \geq 0 $.
		\item $ \graph f \subset H^{-1} \graph f$. More precisely, $ (x_m(x), f_{u(m)}(x_m(x)) ) \in H_m(x, f_m(x)) $ for all $ x \in X_m $, where $x_m(\cdot): X_m(\sigma_0) \to X_{u(m)}(\sigma_0) $ is such that $\lip x_m(\cdot) \leq \lambda_s(m)$, $| x_m(0) | \leq K_0\eta(u(m))$ for $m \in M$ and some constant $ K_0 $.
		\item $f$ does not depend on the choice of $ \eta' = \eta $ as long as $ \eta'(u(m)) \leq \varepsilon'(m) \eta'(m) $, $ \varepsilon'(m) \leq \varepsilon_{1}(m) $ and $ \eta(m) \leq \eta'(m) \leq \eta_0 $, for all $ m \in M $.
	\end{enumerate}
\end{thm}

A similar local version of \autoref{thmB} can be stated, which is omitted.

\section{Graph transform} \label{graphTrans}

We use the the following notations:
\begin{enumerate}[$ \bullet $]
	\item $\sum_\lambda (X, Y) \triangleq \{ \varphi: X \rightarrow Y~ : \lip \varphi \leq \lambda \}$ where $X, Y$ are metric spaces;

	\item $\graph f \triangleq \{ (x, f(x)): x \in X \}$ where $f: X \rightarrow Y$ is a map.
\end{enumerate}
In lemmas \ref{lem:gr}--\ref{lem:bas0} below, assume that $ X_i, Y_i $, $ i=1,2 $, are metric spaces and $ H: X_1 \times Y_1 \rightarrow X_2 \times Y_2 $ is a correspondence with generating map $ (F,G) $ satisfying the (A$'$)$(\alpha, \lambda_u)$ (B)$(\beta; \beta', \lambda_s)$ condition.

\begin{lem}\label{lem:gr}
	Let $f_2 \in \sum_{\hat{\beta}} (X_2, Y_2)$, and $ \alpha \hat{\beta} < 1,~\hat{\beta} \leq \beta $. Then there exist unique $f_1 \in \sum_{\beta'} (X_1, Y_1)$ and $x_1(\cdot) \in \sum_{\lambda_s} (X_1, X_2)$ such that $ (x_1(x), f_2(x_1(x))) \in H(x, f_1(x)) $, $ x \in X_1 $, i.e.,
	\[
	F(x, f_2(x_1(x))) = x_1(x), ~ G(x, f_2(x_1(x))) = f_1(x).
	\]
\end{lem}
\begin{proof}
	Consider the fixed point equation, $F(x, f_2(\hat{x})) = \hat{x}$. Because $\alpha \hat{\beta} < 1$, there is a unique $\hat{x} \triangleq x_1(x)$ satisfying that equation. Since $\hat{\beta} \leq \beta$, by the (B) condition, we have $f_1 \in \sum_{\beta'} (X_1, Y_1)$ and $x_1(\cdot) \in \sum_{\lambda_s} (X_1, X_2)$.
\end{proof}

\begin{lem}\label{lem:bas1}
	Under the assumption of \autoref{lem:gr}, let $ (\hat{x}_i, \hat{y}_i) \in X_i \times Y_i $, $ i =1,2 $, satisfy
	\[
	|F(\hat{x}_1, \hat{y}_2) - \hat{x}_2| \leq \eta_1, |G(\hat{x}_1, \hat{y}_2) - \hat{y}_1| \leq \eta_2,~|f_2(\hat{x}_2) - \hat{y}_2| \leq C_2.
	\]
	Then we have the following estimates:
	\[
	|f_1(\hat{x}_1) - \hat{y}_1| \leq \lambda_u \frac{\hat{\beta}\eta_1 + C_2}{1-\alpha \hat{\beta}} + \eta_2, ~|x_1(\hat{x}_1) - \hat{x}_2| \leq  \frac{\alpha C_2 + \eta_1
	}{1-\alpha \hat{\beta}}.
	\]
	In particular, if $(x_2, y_2) \in H(x_1, y_1)$ and $y_2 = f_2(x_2)$, then $y_1 = f_1(x_1)$, $x_2 = x_1(x_1)$.
\end{lem}

\begin{proof}
	This follows from the following computations:
	\begin{align*}
	|x_1(\hat{x}_1) - \hat{x}_2| & \leq |F( \hat{x}_1, f_2(x_1(\hat{x}_2)) ) - F(\hat{x}_1, \hat{y}_2)| + |F(\hat{x}_1, \hat{y}_2) - \hat{x}_2| \\
	& \leq \alpha |f_2(x_1(\hat{x}_2)) - \hat{y}_2| + \eta_1,
	\end{align*}
	and
	\begin{align*}
	|f_2(x_1(\hat{x}_2)) - \hat{y}_2| & \leq |f_2(x_1(\hat{x}_2)) - f_2(\hat{x}_2)| + |f_2(\hat{x}_2) - \hat{y}_2| \\
	&\leq \hat{\beta} |x_1(\hat{x}_1) - \hat{x}_2| + C_2  \\
	&\leq \alpha \hat{\beta} |f_2(x_1(\hat{x}_2)) - \hat{y}_2| + \hat{\beta} \eta_1 + C_2 .
	\end{align*}
	Thus we have $|f_2(x_1(\hat{x}_2)) - \hat{y}_2|  \leq \frac{\hat{\beta}\eta_1 + C_2}{1-\alpha \hat{\beta}}$. Furthermore,
	\begin{align*}
	|f_1(\hat{x}_1) - \hat{y}_1| & \leq |G( \hat{x}_1, f_2(x_1(\hat{x}_2)) ) - G(\hat{x}_1, \hat{y}_2)| + |G(\hat{x}_1, \hat{y}_2) - \hat{y}_1| \\
	&\leq \lambda_u |f_2(x_1(\hat{x}_2)) - \hat{y}_2| + \eta_2 \\
	&\leq \lambda_u \frac{\hat{\beta}\eta_1 + C_2}{1-\alpha \hat{\beta}} + \eta_2, \\
	|x_1(\hat{x}_1) - \hat{x}_2| & \leq \alpha \frac{\hat{\beta}\eta_1 + C_2}{1-\alpha \hat{\beta}} + \eta_1 = \frac{\alpha C_2 + \eta_1
	}{1-\alpha \hat{\beta}}.
	\end{align*}
	The proof is complete.
\end{proof}

The proof of the following estimates is straightforward.
\begin{lem}\label{lem:bas0}
	Assume $ \alpha \hat{\beta} < 1,~\hat{\beta} \leq \beta $.
	Let $f_2, f'_2 \in \sum_{\hat{\beta}} (X_2, Y_2)$. We have unique $f_1,f'_1 \in \sum_{\beta'} (X_1, Y_1)$ and $x_1(\cdot),x'_1(\cdot) \in \sum_{\lambda_s} (X_1, X_2)$ such that $\graph f_1 \subset H^{-1} \graph f_2$, $\graph f'_1$ $\subset H^{-1} \graph f'_2$. Then
	\begin{align*}
	|x_1(x) - x'_1(x)| & \leq \alpha |f_2(x_1(x)) - f'_2(x'_1(x))|, \\
	|f_2(x_1(x)) - f'_2(x'_1(x))| & \leq |f_2(x_1(x)) - f'_2(x_1(x))| + \hat{\beta} |x_1(x) - x'_1(x)|,\\
	|f_2(x_1(x)) - f'_2(x'_1(x))| & \leq  \frac{1}{1- \alpha \hat{\beta}} |f_2(x_1(x)) - f'_2(x_1(x))|,\\
	|f_1(x) - f'_1(x)| & \leq  \lambda_u |f_2(x_1(x)) - f'_2(x'_1(x))| \leq \frac{\lambda_u}{1 - \alpha \hat{\beta}} |f_2(x_1(x)) - f'_2(x_1(x))|.
	\end{align*}
\end{lem}

\section{Invariant graphs: proofs} \label{proofs}

\subsection{Proof of \autoref{thmA}}

Define a metric space as
\begin{align*}
E_\infty \triangleq \{f : X \to Y & ~\text{a bundle map over}~ \id:\lip f_m \leq \beta'(m), \\
&| f_m( i_X(m) ) - i_Y(m) | \leq C_2(m), \forall m \in M \},
\end{align*}
with metric
\[
\widehat{d}(f, f') \triangleq \sup_{m \in M} \sup_{x \in X_m} \frac{d( f_m(x), f'_m(x) )}{\max\{ d(x, i_X(m)), C_1(m) \}},
\]
where $ C_1(m) =  K_2 \eta(m),~ C_2(m) = K_1 \eta(m)$.
The constants $ K_1, K_2 $ satisfy
\[
K_1 \geq  \frac{\overline{\lambda}_1 \hat{\beta} + 1}{1-\overline{\lambda}_1},  ~
K_2 \triangleq  \dfrac{\hat{\alpha} K_1 + 1}{1 - \sup_m \alpha(m) \beta'(u(m))},
\]
where $ \overline{\lambda}_1 \triangleq \sup_m \frac{ \lambda_u(m) \varepsilon_1(m) }{1 - \alpha(m) \beta'(u(m))} < 1 $, $ \sup_m \alpha(m) \leq \hat{\alpha} < \infty $, $ \sup_m \beta(m) \leq \hat{\beta} < \infty $.
Note that the constant bundle map $ f_m(x) = i_Y(m) $, $ x \in X_m $, belongs to $ E_\infty $.

First we show the metric is well defined.
\begin{slem}
	The metric $ \widehat{d} $ is well defined, and the space $E_\infty$ is complete under the metric $ \widehat{d} $.
\end{slem}
\begin{proof}
	It suffices to show $ \widehat{d}(f, f') < \infty $ for $ f, f' \in E_\infty $. This is easy from
	\begin{align*}
		& |f_m(x) - f'_m(x)| \\
		\leq & |f_m(x) - f_m(i_X(m))| + | f_m(i_X(m)) - f'_m(i_X(m)) | + |f'_m(i_X(m)) - f'_m(x)|\\
		\leq & 2\hat{\beta} |x - i_X(m)| + 2 C_2(m),
	\end{align*}
	and $ \sup_m \frac{C_2(m)}{C_1(m)} < \infty$.
	One can use the standard argument to show the completeness of $E_\infty$, so we omit it.
\end{proof}

In the following, we first assume $ H $ satisfies condition (ii) in \autoref{thmA}.

For every $ f \in E_\infty $, since $ f_{u(m)} \in \sum_{\beta'(u(m))}(X_{u(m)}, Y_{u(m)}) $, $\sup_m \alpha(m) \beta'(u(m)) < 1$, $ \beta'(u(m)) \leq \beta(m) $, by \autoref{lem:gr},  there are unique $ \widetilde{f}_m \in \sum_{\beta'(m)}(X_{m}, Y_{m}) $ and $ x_m(\cdot) \in \sum_{\lambda_s(m)}(X_m, X_{u(m)}) $ such that
\[
(x_m(x), f_{u(m)}(x_m(x)) ) \in H_m(x, \widetilde{f}_m(x)), ~x \in X_{m}.
\]
Since $ | f_{u(m)} ( i_X(u(m)) ) - i_Y(u(m)) | \leq C_2(u(m)) $, $ | i_X(u(m)) - F_m( i_X(m), i_Y(u(m)) ) | \leq \eta(u(m))$, and $ | i_Y(m) - G_m( i_X(m), i_Y(u(m)) ) | \leq \eta(m) $, by \autoref{lem:bas1}, we have
\begin{align*}
| \widetilde{f}_m (i_X(m)) - i_Y(m) | & \leq \lambda_u(m) \frac{ \beta'(u(m)) \eta(u(m)) + C_2(u(m)) }{1- \alpha(m) \beta'(u(m))} + \eta(m) \\
& \leq  \overline{\lambda}_1 \hat{\beta} \eta(m) + \overline{\lambda}_1 C_2(m)  + \eta(m) \leq C_2(m),
\end{align*}
and
\[
|x_m(i_X(m)) - i_X(u(m))| \leq \frac{\hat{\alpha} C_2(u(m)) + \eta(u(m))}{1 - \sup_m \alpha(m) \beta'(u(m))} \leq C_1(u(m)).
\]
Let $ \widetilde{f}(m, x) = (m, \widetilde{f}_m(x)) $. Then $ \widetilde{f} \in E_\infty $. Define the graph transform $ \varGamma $ as
\[
\varGamma: E_\infty \to E_\infty, f \mapsto \widetilde{f}.
\]
Note that $ \graph \widetilde{f} \subset H^{-1} \graph f $.
\begin{slem}
	The graph transform $ \varGamma $ is Lipschitz under the metric $ \widehat{d} $, and
	\[
	\lip \varGamma \leq \sup_m \dfrac{ \lambda_u(m) \lambda_s(m) + \lambda_u(m) \varepsilon_1(m) }{1-\alpha(m) \beta'(u(m))} < 1.
	\]
\end{slem}
\begin{proof}
	Let $ f, f' \in E_\infty $, $ \widetilde{f} = \varGamma f $, $ \widetilde{f}' = \varGamma f' $, with $ x_m(\cdot) $, $ x'_m(\cdot) $. Note first that by \autoref{lem:bas0}, 
	\[
	| \widetilde{f}_m(x) - \widetilde{f}'_m(x) | \leq \frac{\lambda_u(m)}{1-\alpha(m) \beta'(u(m))} | f_{u(m)}(x_m(x)) -f'_{u(m)}(x_m(x)) |,
	\]
	and
	\begin{align*}
		|x_m(x) - i_X(u(m)) | & \leq | x_m(x) - x_m(i_X(m)) | + |x_m(i_X(m)) - i_X(u(m)) | \\
	& \leq \lambda_s(m) | x - i_X(m) | + C_1(u(m)).
	\end{align*}
	Thus,
	\begin{align*}
		&~ \frac{| \widetilde{f}_m(x) - \widetilde{f}'_m(x) |}{\max\{  |x - i_X(m)|, C_1(m) \}} \\
		\leq &~ \frac{\lambda_u(m)}{1-\alpha(m) \beta'(u(m))} \frac{| f_{u(m)}(x_m(x)) -f'_{u(m)}(x_m(x)) |}{\max\{ |x_m(x) - i_X(u(m))|, C_1(u(m)) \}} \\
		& \quad \cdot \frac{\max\{ |x_m(x) - i_X(u(m))|, C_1(u(m)) \}}{\max\{  |x - i_X(m)|, C_1(m) \}} \\
		\leq &~ \frac{ \lambda_u(m) \lambda_s(m) + \lambda_u(m) \varepsilon_1(m) }{1-\alpha(m) \beta'(u(m))} \frac{| f_{u(m)}(x_m(x)) -f'_{u(m)}(x_m(x)) |}{\max\{ |x_m(x) - i_X(u(m))|, C_1(u(m)) \}},
	\end{align*}
	and the proof is complete.
\end{proof}

Since $ \varGamma $ is contractive, we see there is a unique fixed point of $ \varGamma $ in $ E_\infty $, denoted by $ f $. By the construction of $ \varGamma $, we have $ \graph f \subset H^{-1} \graph f $, showing conclusions (1), (2) in \autoref{thmA}. The constant $ K $ in (2) can be taken as $ K = \frac{\overline{\lambda}_1 \hat{\beta} + 1}{1-\overline{\lambda}_1} $.

$ f $ is unique in the sense that if $ f':X \to Y $ is any bundle map over $ \id $ that satisfies
\begin{enumerate}[(a{0})]
	\item $ \lip f'_m \leq \beta'(m) $, $ | f'_m(i_X(m))  - i_Y(m)| \leq K' \eta'(m) $, where $ K' \geq 0 $ is a constant, and
	\item $ \graph f' \subset H^{-1} \graph f' $,
\end{enumerate}
then $ f' = f $, where $ \eta' $ satisfies that $ \eta'(u(m)) \leq \varepsilon'(m) \eta'(m) $, $ \varepsilon'(m) \leq \varepsilon_{1}(m) $ and $ \eta(m) \leq \eta'(m) $ for all $ m \in M $.
This is easy from the unique fixed point of $ \varGamma $ in $ E_\infty $. Without loss of generality, we can assume $ K' \leq K_1 $. The construction of $ E_\infty \triangleq E_\infty(\eta) $ depends on $ \eta $. However, $ E_\infty(\eta) \subset E_\infty(\eta') $. So $ f, f' \in E_\infty(\eta') $. The uniqueness of the fixed point of $ \varGamma $ in $ E_\infty(\eta') $ shows $ f = f' $.

Next, we show that if $ H $ satisfies condition (ii)$ ' $ in \autoref{thmA}, then the conclusions also hold.
Consider $ H^{(k)} $, the $ k $th composition of $ H $.
\begin{slem}
	$ H^{(k)} $ is well defined, and is a bundle correspondence over $u^k$ with generating map $ (F^{(k)},G^{(k)}) $, satisfying (A)$(\alpha;\alpha'_{k}, \lambda^{(k)}_u)$ (B)$(\beta_{k};\beta', \lambda^{(k)}_s)$ condition, where $ \alpha'_{k}(m) = \alpha'(u^{k-1}(m)) $, $ \beta_{k}(m) = \beta(u^{k-1}(m)) $,
	\[
	\begin{gathered}
	\lambda^{(k)}_u(m) = \lambda_u(u^{k-1}(m)) \cdot \lambda_u(u^{k-2}(m)) \cdots \lambda_u(m),\\ \lambda^{(k)}_s(m) = \lambda_s(u^{k-1}(m)) \cdot \lambda_s(u^{k-2}(m)) \cdots \lambda_s(m).
	\end{gathered}
	\]
	Moreover, $ H_{u^{k}(m)} \circ H^{(k)}_{m} = H^{(k+1)}_m $, $ k = 1,2,\ldots $.
\end{slem}
\begin{proof}
	Since $ H_m \sim (F_m, G_m) : X_m \times Y_m \to X_{u(m)} \times Y_{u(m)} $, and $\alpha'(m) \beta'(u(m)) < 1$, by \autoref{lem:com} and \autoref{lem:a1}, we know that $ H_m^{(2)} \triangleq H_{u(m)} \circ H_m \sim (F_m^{(2)}, G_m^{(2)}): X_m \times Y_m \to X_{u^2(m)} \times Y_{u^2(m)} $ satisfies the (A)$(\alpha(m);\alpha'_{2}(m), \lambda^{(2)}_u(m))$ (B)$(\beta_{2}(m);\beta'(m), \lambda^{(2)}_s(m))$ condition, and $ \alpha'_{2}(m) \beta'(u^2(m)) = \alpha'(u(m)) \beta'(u^2(m)) < 1 $. Using $ H_m^{(2)} $, one can define $ H^{(2)}: X \times Y \to X \times Y $, a bundle correspondence over $ u^2 $ with generating bundle map $(F^{(2)}, G^{(2)})$, i.e., $ \graph H^{(2)} = \bigcup_{m \in M} (m, \graph H^{(2)}_m) $. $H^{(k)}$ can be defined inductively by using the composition of $H^{(k-1)}$ and $H$. Obviously $ H_{u^{k}(m)} \circ H^{(k)}_{m} = H^{(k+1)}_m $, $ k = 1,2,\ldots $, by our construction.
\end{proof}

Fix any $ m_0 \in M $. Let $ \widehat{M}_{m_0} = \{ u^n(m_0): n \geq 0 \} $. Hence $ H^{(k)} $ can be regarded as a bundle correspondence $ X|_{\widehat{M}_{m_0}} \times Y|_{\widehat{M}_{m_0}} \to X|_{\widehat{M}_{m_0}} \times Y|_{\widehat{M}_{m_0}} $ over $ u^{k} $, denoted by $ H^{(k)}|_{\widehat{M}_{m_0}} $. Define a function $ \hat{\varepsilon}_1(\cdot) $ over $ u $ as
\[
\hat{\varepsilon}_1(m) = \max\{ \varepsilon(m), \lambda_{s}(m) \}, ~m \in M.
\]
Set
\[
\lambda \triangleq \sup_m \lambda_s(m) \lambda_u(m), ~
\lambda_1 \triangleq \sup_m \varepsilon(m) \lambda_u(m), ~
\gamma \triangleq \sup_m \alpha'(m) \beta'(u(m)).
\]

\begin{slem}
	The section $ i=(i_X, i_Y) $ is an $ \hat{\varepsilon}_1^{(k)} $-pseudo-stable section of $ H^{(k)}|_{\widehat{M}_{m_0}} $, i.e.,
	\begin{gather*}
	| i_X(u^{k}(m)) - F^{(k)}_m( i_X(m), i_Y(u^k(m)) ) | \leq \eta^{(m_0)}_{k}(u^{k}(m)), \\
	| i_Y(m) - G^{(k)}_m( i_X(m), i_Y(u^k(m)) ) | \leq \eta^{(m_0)}_{k}(m),
	\end{gather*}
	for all $ m \in \widehat{M}_{m_0} $, where $ \eta^{(m_0)}_{k}(u^i(m_0)) = c_k \hat{\varepsilon}^{(i)}_1(m_0) \eta(m_0) $, $ i \geq 0 $, and $ c_k \geq 1 $ is a constant independent of $ m_0 $.
\end{slem}
\begin{proof}
	This is a direct consequence of \autoref{lem:bas1}. We only consider the case $ k=2 $. Let $ \hat{y}_m $ be the unique point satisfying $ \hat{y}_m = G_{u(m)}( F_m(i_X(m), \hat{y}_m), i_Y(u^2(m)) ) $. Then
	\[
	|\hat{y}_m - i_Y(u(m))| \leq \frac{ (\beta'(u(m)) + 1) \eta(u(m)) }{1 - \beta'(u(m)) \alpha'(m)},
	\]
	and
	\begin{align*}
		|G^{(2)}_m(i_X(m), i_Y(u^2(m))) - i_Y(m) | & = |G_m(i_X(m), \hat{y}_m) - i_Y(m) | \\
		& \leq \lambda_u(m) |  \hat{y}_m  - i_Y(u^2(m)) | + \eta(m)  \\
		& \leq \lambda_u(m) \frac{ (\beta'(u(m)) + 1) \eta(u(m)) }{1 - \beta'(u(m)) \alpha'(m)} + \eta(m).
	\end{align*}
	Also,
	\begin{align*}
	&~ |F^{(2)}_m(i_X(m), i_Y(u^2(m) ) ) - i_X(u^2(m)) | \\
	= &~ | F_{u(m)} ( F_m(i_X(m), \hat{y}_m), i_Y(u^2(m)) ) - i_X(u^2(m)) | \\
	\leq &~ \lambda_s(u(m)) \frac{ (\alpha'(m) + 1)\eta(u(m))}{1 - \beta'(u(m)) \alpha'(m)} + \eta(u^2(m)).
	\end{align*}
	If $ m = u^{j}(m_0) $, then $ \eta(m) \leq \hat{\varepsilon}^{(j)}_{1}(m_0)\eta(m_0) $.
	So we can choose
	\[
	c_2 = \max\left\{ \dfrac{\lambda_1 (\hat{\beta} + 1)}{1 - \gamma} + 1, \dfrac{\hat{\alpha} + 1}{1 - \gamma} + 1 \right\},
	\]
	completing the proof.
\end{proof}

Since $ \lambda, \lambda_1, \gamma < 1 $, we can choose large $ k $ (independent of $ m_0 $) such that $ 2\lambda^k + \lambda_1^k < 1 - \gamma $.
So
\[
\sup_m \frac{ \lambda_s^{(k)}(m) \lambda_u^{(k)}(m) + \hat{\varepsilon}^{(k)}_1(m) \lambda_u^{(k)}(m) }{ 1 - \alpha'_{k}(m) \beta'(u^{k}(m)) } < 1.
\]
So far, we have shown $ H^{(k)}|_{\widehat{M}_{m_0}} $ satisfies condition (ii) with the $ \hat{\varepsilon}^{(k)}_1 $-pseudo-stable section $ i $. Using what we have proven, we obtain a unique bundle map $ f^{k, (m_0)}: X |_{\widehat{M}_{m_0}} \to Y|_{\widehat{M}_{m_0}} $ over $ \id $ such that
\begin{enumerate}[(a{1})]
	\item $ \lip f^{k, (m_0)}_{m} \leq \beta'(m) $, $ | f^{k, (m_0)}_m(i_X(m))  - i_Y(m)| \leq K'_1 \eta^{(m_0)}_k(m) $, where $ K'_1 \geq 0 $ (independent of $ m_0 $),
	\item $ \graph f^{k, (m_0)}_m \subset (H^{(k)}_{m})^{-1} \graph f^{k, (m_0)}_{u^{k}(m)} $, $ m \in \widehat{M}_{m_0} $,
	\item $ f^{k, (m_0)} $ does not depend on the choice of $ \eta^0(\cdot) = \eta^{(m_0)}_k(\cdot) $ as long as $ \eta^0(u^{k}(m)) \leq \hat{\varepsilon}^{(k)}_1(m) \eta^0(m) $, $ \eta^{(m_0)}_{k}(m) \leq \eta^0(m) $, for all $ m \in \widehat{M}_{m_0} $.
\end{enumerate}

Note that for $ H^{(k)} $, its corresponding graph transform is $ \varGamma ^k = \varGamma \circ \cdots \circ \varGamma $ ($ k $ times). Since $ f^{k, (m_0)} $ is the unique fixed point of $ \varGamma^k $, we have $ \varGamma f^{k, (m_0)} = f^{k, (m_0)} $, i.e., 
\[
\graph f^{k, (m_0)}_{m} \subset H^{-1}_m \graph f^{k, (m_0)}_{u(m)}, ~m \in \widehat{M}_{m_0}.
\]

Set $ f_{m_0} = f^{k, (m_0)}_{m_0} $, $ m_0 \in M $. Note that $ f^{k, (m_0)}_{u(m_0)} = f^{k,(u(m_0))}_{u(m_0)} $. Indeed, $ \{f^{k, (m_0)}_{m'}: m' \in \widehat{M}_{u(m_0)} \} $ also fulfills (a1) and (b1) for the case where $ m_0 $ is replaced by $ u(m_0) $ with $ \eta^0(\cdot) = \eta^{(m_0)}_k(\cdot) |_{\widehat{M}_{u(m_0)}} $ instead of $ \eta^{(u(m_0))}_k(\cdot) $. Also, note that 
\[
|f^{k, (m_0)}_{m_0}(i_X(m_0)) - i_Y(m_0)| \leq K'_1c_{k}\eta(m_0).
\]
This shows $ f $ satisfies conclusions (1) and (2) in \autoref{thmA}. The uniqueness of $ f $ follows from the uniqueness of $ f: X |_{\widehat{M}_{m_0}} \to Y|_{\widehat{M}_{m_0}} $ for $ H^{(k)}|_{\widehat{M}_{m_0}} $.
\qed

\subsection{Proof of \autoref{thmB}}

The proof is essentially the same as that of \autoref{thmA}. We sketch the main steps.
The metric space is now defined by
\begin{align*}
E^b_\infty \triangleq \{f : X \to Y ~&\text{a bundle map over}~ \id:~ \lip f_m \leq \beta'(m), \\
&\widetilde{d}( f_m(X_m), i_Y(m)) \leq K_1 \eta(m), ~\forall m \in M \},
\end{align*}
with metric
\[
\widehat{d} (f,f') = \sup_{m \in M, \eta(m) \neq 0} \sup_{x \in X_m} \frac{d(f_m(x) , f'_m(x) )}{\eta(m)},
\]
for some constant $ K_1 > 0 $. The metric $ \widehat{d} $ is well defined, and $ E^b_\infty $ is complete under this metric.

Assume $ H $ satisfies condition (ii) in \autoref{thmB}. The graph transform $ \varGamma $ now also has $ \varGamma E^b_\infty \subset E^b_\infty $, by simple computations. It is also Lipschitz with Lipschitz constant no more than $ \sup_m \frac{ \lambda_u(m)\varepsilon_1(m) }{1-\alpha(m)\beta'(u(m))} < 1 $. From these, we clearly obtain the results.
When $ H $ satisfies condition (ii)$ ' $ in \autoref{thmB}, consider $ H^{(k)} $ instead of $ H $ for large $ k $; then we have $ \lip \varGamma^k < 1$, competing the proof. \qed

\subsection{Proof of \autoref{thm:local}}

The proof is almost identical to that of \autoref{thmA}. We only consider the case where the (A$ ' $) (B) condition holds. The metric space is now defined by
\begin{align*}
E^0_\infty \triangleq \{f : X(\sigma_0) \to Y ~\text{a bundle map over}~  &\id:\lip f_m \leq \beta'(m), \\
& | f_m( 0 ) | \leq K_1 \eta(m), \forall m \in M \},
\end{align*}
where $ \sigma_0, K_1 > 0 $ will be chosen later. The metric is the same as in $ E_\infty $.
Take the following constants:
\begin{gather*}
\gamma = \sup_m \alpha(m) \beta'(u(m)) < 1 / 2, ~
\overline{\lambda}_1 \triangleq \sup_m \frac{ \lambda_u(m) \varepsilon_1(m) }{1 - \alpha(m) \beta'(u(m))} < 1,\\
\hat{\alpha} = \sup_m \alpha(m),~\hat{\beta} =  \sup_m \beta(m),\\
K_1 = \frac{\overline{\lambda}_1 \hat{\beta} + 1}{1-\overline{\lambda}_1}, ~
K_2 =  \frac{\hat{\alpha} K_1 + 1}{1 - 2\gamma}, ~
\overline{\lambda}_s = \sup_{m} \lambda_s(m) < 1, ~
r < \min\{ r_i, r'_i: i = 1, 2 \} / 2.
\end{gather*}
Let $ \eta_0 = \sup_{m} \eta(m) $ be small and choose $ \sigma_0 > 0 $ such that
\[
\frac{K_2\eta_0}{1 - \overline{\lambda}_s} \leq \sigma_0 < r, ~\hat{\beta} \sigma_0 + K_1 \eta_0 < r.
\]
So $ \overline{\lambda}_s \sigma + K_2 \eta_0 < \sigma_0 $ if $ \sigma < \sigma_0 $; and $ f_m(X_{m}(\sigma_0)) \subset Y_{m}(r) $ if $ f \in E^0_{\infty} $.

Let $ f \in E^0_{\infty} $ and $ f^1_m(x) = f_m(r_{\sigma_0}(x)) $, where $ r_{\sigma_0}(\cdot) $ is the radial retraction (see \eqref{radial}). Note that $ \lip f^1_m \leq 2\beta'(m) $.
Consider the following fixed point equation for each $ x \in X_{m}(\sigma_0) $:
\[
F_m( x, f^1_{u(m)}( y ) ) = y, ~y \in X_{u(m)}.
\]
Since $ \alpha(m) \beta'(u(m)) < 1 /2 $, for each $ x \in X_{m}(\sigma_0) $ the above equation has a fixed point $ y = x_{m}(x) \in X_{u(m)} $. As $ |f^1_{u(m)}(0)| = |f_{u(m)}(0)| \leq K_1 \eta_0 $ and $ |F_m(0,0)| \leq \eta_0 $, we see that $ |x_m(0)| \leq K_2 \eta_0 $.
Let us show
\begin{slem}
	$ \lip x_{m}(\cdot)|_{X_{m}(\sigma_0)} \leq \lambda_s(m) $.
\end{slem}
\begin{proof}
	First note that $ \lip x_{m}(\cdot)|_{X_{m}(\sigma_0)} \leq \frac{\overline{\lambda}_s}{1 - 2 \gamma} $. So we can choose an $ \varepsilon_0 > 0 $ small such that $ \frac{\overline{\lambda}_s}{1 - 2 \gamma} \varepsilon_0 + K_2 \eta_0 \leq \sigma_0 $, yielding $ x_{m}(X_{m}(\varepsilon_0)) \subset X_{u(m)}(\sigma_0) $ and $ F_m( x, f_{u(m)}( x_{m}(x) ) ) = x_{m}(x) $. By the (B) condition, we know $ \lip x_{m}(\cdot)|_{X_{m}(\varepsilon_0)} \leq \lambda_s(m) $. Set
	\[
	\sigma_1 = \sup\{ \sigma \leq \sigma_0: \lip x_{m}(\cdot)|_{X_{m}(\sigma)} \leq \lambda_s(m) \}.
	\]
	If $ \sigma_1 < \sigma_0 $, then $ \lip x_{m}(\cdot)|_{X_{m}(\sigma)} \leq \lambda_s(m) $ for any $ \sigma < \sigma_1 $, which again shows 
	\[
	\lip x_{m}(\cdot)|_{\overline{X_{m}(\sigma_1)}} \leq \lambda_s(m).
	\]
	Since $ \sigma_1 < \sigma_0 $, we have $ \overline{\lambda}_s \sigma_1 + K_2\eta_0 < \sigma_0 $. Thus, we can choose an $ \varepsilon > 0 $ small such that
	\[
	\overline{\lambda}_s \sigma_1 + K_2\eta_0 +  \frac{\overline{\lambda}_s}{1 - 2 \gamma} \varepsilon  \leq \sigma_0,
	\]
	which implies that $ x_m(X_{m}(\sigma_1 + \varepsilon)) \subset X_{u(m)} (\sigma_0) $ and $ \lip x_{m}(\cdot)|_{X_{m}(\sigma_1 + \varepsilon)} \leq \lambda_s(m) $, contradicting the definition of $ \sigma_1 $. So $ \sigma_1 = \sigma_0 $, showing the result.
\end{proof}

In particular, by the above sublemma, $ x_{m}(X_{m}(\sigma_0)) \subset X_{u(m)}(\sigma_0) $. Now we can define
\[
\widetilde{f}_{m} (x) = G_m(x, f_{u(m)}(x_m(x))), ~x \in X_{m}(\sigma_0),
\]
and $ \widetilde{f}(m,x) = (m, \widetilde{f}_{m} (x)) $. Next, following the same steps as in the proof of \autoref{thmA}, one can see that $ \widetilde{f} \in E^{0}_{\infty} $ and $ \varGamma: f \mapsto \widetilde{f} $ is contractive. The proof is complete.
\qed

\section{More characterizations}\label{properties}

Let $ H \sim (F, G) : X \times Y \to X \times Y $ and $ i: M \to X \times Y $ be as in \autoref{thmA}. Let $ f $ be the bundle map obtained in \autoref{thmA}.

Below, we use the following notations.
\begin{enumerate}[$ \bullet $]
	\item $ W^s(i) \triangleq W^s \triangleq \graph f $, $ W^s(i(m)) \triangleq W^s(m) \triangleq \graph f_m $.

	\item $ h^s_m \triangleq x_m(\cdot) : X_m \to X_{u(m)} $. This gives a bundle map $ h^s: X \to X $ over $ u $ such that $ h(m,x) = (m, h^s_m(x)) $. Also, note that $ \lip h^s_m \leq \lambda_s(m) $.

	\item $ \{ z_n = (x_n, y_n): n = m, u(m), u^2(m), \cdots \} $ is called a \textbf{forward orbit} of $H$ from $ m $ if $ z_{u^k(m)} \in H_{u^{k-1}(m)}(z_{u^{k-1}(m)}) $, $ k = 1,2,3,\ldots $. Backward orbit can be defined similarly.

	\item Let $ P^X (m, x, y) = (m, P^X_m(x, y)) = (m, x) $ for $ (x, y) \in X_m \times Y_m $; and similarly for $ P^Y,~ P^Y_m $.

	\item $ a_n \lesssim b_n $, $ n \to \infty $ ($ a_n \geq 0, b_n > 0 $) if $ \sup_{n\geq 0}b^{-1}_n a_n <\infty $. \label{notation1}
\end{enumerate}

The bundle map $ f $ has the properties listed below.

\begin{enumerate}[(1)]
	\item $ H_m: W^s(m) \to W^s(u(m)) $ defines a Lipschitz map $ H^s_m $. This gives a bundle map $ H^s: X \times Y \to X \times Y $ over $ u $ such that $ H^s (m, z) = (m, H^s_m(z)) $. We have $ P^X \circ H^s = h^s \circ P^X $.

	\item Let $ z'_m = (x'_m, y'_m) \in W^s(m) $. Then there is a unique forward orbit $ \{ z_n = (x_n, y_n): n = m, u(m), u^2(m), \ldots \} $ of $ H $ from $ m $ such that $ z_m = z'_m $. Furthermore,
	\begin{equation}\label{c01}
	|i_X(u^k(m)) - x'_{u^k(m)} | \leq K_0 \sum_{i=1}^{k} \eta(u^{i-1}(m)) \lambda^{(k-i)}_s(u^{i}(m)) + \lambda^{(k)}_s(m) |i_X(m) - x'_m|.
	\end{equation}
	Therefore,
	\[
	| i(u^k(m)) - z_{u^k(m)} | \leq \widetilde{K} (\sum_{i=1}^{k} \eta(u^{i-1}(m)) \lambda^{(k-i)}_s(u^{i}(m)) + \lambda^{(k)}_s(m) |i_X(m) - x'_m|)
	\]
	for some constant $\widetilde{K} > 0$.

	\begin{proof}
		Compute
		\begin{align*}
		&|i_X(u^k(m)) - x'_{u^k(m)} | \\
		= & |i_X(u^k(m)) - x_{u^{k-1}(m)} ( x'_{u^{k-1}(m)} ) | \\
		\leq & | i_X(u^k(m)) - x_{u^{k-1}(m)} ( i_X(u^{k-1}(m)) ) | \\
		& + | x_{u^{k-1}(m)} ( i_X(u^{k-1}(m)) ) - x_{u^{k-1}(m)} ( x'_{u^{k-1}(m)} ) | \\
		\leq & K_0 \eta(u^{k-1}(m)) + \lambda_s(u^{k-1}(m)) | i_X(u^{k-1}(m)) - x'_{u^{k-1}(m)} |,
		\end{align*}
		showing \eqref{c01}.
	\end{proof}

	\item If $ \{ z_n = (x_n, y_n): n = m, u(m), u^2(m), \cdots \} $ is a forward orbit of $H$ from $ m $ such that
	\begin{equation}\label{c11}
	| z_{u^{k}(m)} - i(u^{k}(m)) | \lesssim \varepsilon'^{(k)}(m), ~k \to \infty,
	\end{equation}
	where $ \varepsilon'(\cdot) : M \to \mathbb{R}_+ $ over $ u $, $ \sup_m \varepsilon'(m)\lambda_u(m)\vartheta(m) < 1 $ ($ \vartheta(m) = 1 $ if (ii)$ ' $ holds, and $ \vartheta(m) = (1- \alpha(m)\beta'(u(m)))^{-1} $ otherwise), then $ z_n \in W^s(n) $, $ n = m, u(m), u^2(m), \ldots $.
	\begin{proof}
		First consider the case where condition (ii)$ ' $ in \autoref{thmA} holds.
		Let $ \{ z''_n = (x''_n, y''_n) \in W^s(n) , n = m, u(m), u^2(m), \cdots \} $ be a forward orbit of $ H $ from $ m $ such that $ z''_m = (x_m, f_m(x_m)) \in W^s(m) $. Compute
		\begin{align*}
		&~ |y_m - f_m(x_m)|\\
		\leq &~ \lambda^{(k)}_u(m) |y_{u^k(m)} - y''_{u^k(m)}| \quad \text{(by (A) condition)} \\
		\leq &~ \lambda^{(k)}_u(m) \left\lbrace  | y_{u^k(m)} - i_Y(u^k(m)) | + | i_Y(u^k(m)) - f_{u^k(m)}(i_X(u^k(m))) | \right. \\
		&\quad\quad \quad+ \left. | f_{u^k(m)}(i_X(u^k(m))) - f_{u^k(m)}(x''_{u^k(m)}) |  \right\rbrace  \\
		\leq &~ \lambda^{(k)}_u(m) \left\lbrace  C\varepsilon'^{(k)}(m) + K \eta(u^k(m)) + \hat{\beta}( K_0  \sum_{i=1}^{k} \eta(u^{i-1}(m)) \lambda^{(k-i)}_s(u^{i}(m)) \right. \\
		&\quad\quad\quad+ \left. \lambda^{(k)}_s(m) |i_X(m) - x'_m|) \right\rbrace \\
		\leq &~ \widetilde{C} \left\lbrace  \hat{\lambda}^k + \hat{\lambda}_1^k \eta(m) + \sum_{i=1}^{k} \hat{\lambda}_1^i \lambda^{k-i} + \lambda^k |i_X(m) - x'_m| \right\rbrace ,
		\end{align*}
		where in the third inequality we use \eqref{c11}, \autoref{thmA} and  \eqref{c01}, and where
		\[
		\lambda = \sup_m \lambda_s(m) \lambda_u(m) < 1,~\hat{\lambda} = \sup_m \varepsilon'(m)\lambda_u(m) < 1, ~\hat{\lambda}_1 = \sup_{m} \varepsilon(m)\lambda_u(m) < 1,
		\]
		and $ \widetilde{C} > 0 $ is a constant. Letting $ k \to \infty $, we have $ y_m = f_m(x_m) $.

		Next consider the case where condition (ii) in \autoref{thmA} holds. Then
		\begin{align*}
		|y_m - f_m(x_m)| & \leq \frac{\lambda_{u}(m)}{1 - \alpha(m)\beta'(u(m))} |y_{u(m)} - f_{u(m)}(x_{u(m)})| \\
		& \leq \widetilde{\lambda}^{(k)}_{u} (m) |y_{u^{k}(m)} - f_{u^k(m)}(x_{u^{k}(m)})| \\
		& \leq \widetilde{\lambda}^{(k)}_{u} (m) \left\{ |y_{u^{k}(m)} - i_Y(u^k(m))| +  |i_Y(u^k(m)) - f_{u^k(m)}(i_{X}(u^k(m)))| \right. \\
		& \quad \quad + \left. |f_{u^k(m)}(i_{X}(u^k(m))) - f_{u^k(m)}(x_{u^{k}(m)})| \right\}\\
		& \leq \widetilde{C} \left\{\widetilde{\lambda}^{(k)}_{u} (m) \varepsilon'^{(k)}(m) +  \widetilde{\lambda}^{(k)}_{u} (m) \eta(u^{k}(m)) \right\},
		\end{align*}
		where $ \widetilde{\lambda}_u(m) = \lambda_{u}(m)\vartheta(m) $ (over $ u $). Letting $ k \to \infty $, we see that $ y_m = f_m(x_m) $, completing the proof.
	\end{proof}

	\item In particular,
	\begin{align*}
	\{ j: M \to X & \times Y : j ~\text{is an invariant section of $ H $ such that}~ \\
	&|j(m) - i(m) | \leq K' \eta'(m) ~\text{for some constant} ~ K' > 0,  ~\forall m \in M  \} \subset W^s,
	\end{align*}
	where $ \eta'(u(m)) \leq \varepsilon'(m) \eta'(m), ~\forall m \in M $, and $ \sup_m \varepsilon'(m)\lambda_u(m)\vartheta(m) < 1 $  ($ \vartheta(\cdot) $ is given in (3)).

	\item
	We have the following characterization of $ W^s(m) $.
	\begin{align*}
	W^s(m) = \{ z_m = (x_m, & y_m) \in X_m \times Y_m: ~\text{there is a forward orbit}~ \{ z_k \}_{k \geq 0} ~ \text{of} ~H \\
	& \text{such that}~ z_0 = z_m,~ \sup_{k \geq 0} (\hat{\varepsilon}^{(k)}(m))^{-1} | z_k - i(u^k(m)) | < \infty \},
	\end{align*}
	where $ \hat{\varepsilon}(m) \geq \varepsilon(m) $, $ \lambda_s(m) + \varsigma < \hat{\varepsilon}(m) < (\lambda_u(m)\vartheta(m))^{-1} - \varsigma $ (for sufficiently small $ \varsigma > 0 $), where $ \vartheta(\cdot) $ is given in (3).
	This follows from (2) (3).
\end{enumerate}

	Similar results hold for the bundle map $ f $ obtained in \autoref{thmB}. The `unstable results' (in the direction $ Y \to X $) can be obtained through the dual bundle correspondence of $ H $ (see \autoref{dual}).

\section{Corollaries} \label{corollaries0}

\subsection{Hyperbolic dichotomy} \label{dichotomy}

\begin{defi}\label{pseudoSection}
	Let $ H \sim (F, G): X \times Y \to X \times Y $ be a bundle correspondence over $ u $ with generating bundle map $ (F, G) $. Let $ i = (i_X, i_Y): M \to X \times Y $ satisfy
	\[
	| i_X(u(m)) - F_m( i_X(m), i_Y(u(m)) ) | \leq \eta(u(m)), ~ | i_Y(m) - G_m( i_X(m), i_Y(u(m)) ) | \leq \eta(m),
	\]
	where $ \eta: M \rightarrow \mathbb{R}_+ $.
	Take $ \varepsilon(\cdot): M \to \mathbb{R}_+ $. Then $ i $ is called an $\varepsilon$-\textbf{pseudo-stable section} of $ H $ if $ \eta(u(m)) \leq \varepsilon(m) \eta(m),~\forall m \in M $; an $\varepsilon$-\textbf{pseudo-unstable section} of $ H $ if $ \eta(m) \leq \varepsilon(m) \eta(u(m)),~\forall m \in M $; and an $\varepsilon$-\textbf{pseudo-invariant section} of $ H $ if it is both an $\varepsilon$-{pseudo-stable section} and an $\varepsilon$-{pseudo-unstable section} of $ H $. In particular, a $0$-pseudo-invariant section is an invariant section of $ H $, i.e., $ i(u(m)) \in H_m(i(m)) $ for all $ m \in M $.
\end{defi}
Note that if $ u $ is invertible, then $ i $ is an $ \varepsilon $-pseudo-unstable section of $ H $ if and if only $ i $ is an $ \varepsilon \circ u^{-1} $-pseudo-stable section of the dual bundle correspondence $ \widetilde{H} $ of $ H $ (see \autoref{dual}). The following theorem is a restatement of the previous results.

\begin{thm}\label{thm:hyperTrivial}
	Let $(X, M, \pi_1), (Y, M, \pi_2)$ be bundles with metric fibers and $u: M \rightarrow M$ an \emph{invertible map}.
	Let $H: X \times Y \rightarrow X \times Y$ be a bundle correspondence over $u$ with generating bundle map $(F,G)$.
	Assume $ i = (i_X, i_Y), j=(j_X, j_Y): M \to X \times Y $ are an $ \varepsilon_1 $-pseudo-stable section and an $ \varepsilon_2 $-pseudo-unstable section of $ H $, respectively, where $ \varepsilon_1(\cdot), \varepsilon_2 (\cdot): M \to \mathbb{R}_+ $ (over $ u $ and $ u^{-1} $ respectively).
	Suppose $H$ satisfies the (A)$(\alpha;\alpha', \lambda_u)$ (B)$(\beta;\beta', \lambda_s)$ condition, where $\alpha, \beta, \alpha', \beta', \lambda_u, \lambda_s$ are \emph{bounded} functions of $M \rightarrow \mathbb{R}_+$. In addition,
	\begin{enumerate}[(a)]
		\item (angle condition) $\sup_m \alpha'(m) \beta'(u(m)) < 1$, $ \alpha'(m) \leq \alpha(u(m))$, $ \beta'(u(m)) \leq \beta(m)$ for all $ m \in M  $,

		\item  (spectral condition) $ \sup_m \lambda_u(m) \varepsilon_1(m) < 1 $, $ \sup_m \lambda_s(m) \varepsilon_2(m) < 1 $, and 
		
		$ \sup_m \lambda_s(m) \lambda_u(m) < 1 $.
	\end{enumerate}

	Then there are $ W^{s}_m(i) $, $ W^{u}_m(i) $, $ m \in M $ such that the following statements hold:
	\begin{enumerate}[(1)]
		\item $ W^{s}_m(i) $, $ W^{u}_m(j) $ are represented as Lipschitz graphs of $ X_m \to Y_m $ and $ Y_m \to X_m $, respectively, such that $ W^s_m(i) \subset H^{-1}_m W^s_{u(m)}(i) $ and $ W^u_m(j) \subset H_{u^{-1}(m)} W^u_{u^{-1}(m)}(j) $.

		\item $ W^s_m(i) \cap W^u_m(j) = \{ k(m) \} $. Therefore, $ k $ is an invariant section of $ H $, i.e., $ k(u(m)) \in H_m(k(m)) $.

		\item (shadowing section) We have
		\[
		|k(u^k(m)) - i(u^k(m)) | \lesssim {\hat{\varepsilon}_1}^{(k)}(m),~|k(u^{-k}(m)) - j(u^{-k}(m))| \lesssim {\hat{\varepsilon}_2}^{(k)}(m),~ k \to \infty,
		\]
		where $ \hat{\varepsilon}_i(\cdot) $, $ i = 1,2 $, are functions $ M \to \mathbb{R}_+ $ over $ u $ and $ u^{-1} $, respectively, such that $ \varepsilon_1 (m) \leq \hat{\varepsilon}_1(m) $, $ \varepsilon_2(m) \leq \hat{\varepsilon}_2(m) $, $ \lambda_s(m) + \varsigma < \varepsilon_1(m) < \lambda^{-1}_u(m) - \varsigma $, $ \lambda_u(m) + \varsigma < \varepsilon_2(m) < \lambda^{-1}_s(m) - \varsigma $, $ m \in M $, and $ \varsigma > 0 $ is sufficiently small.

		\item An invariant section $ s $ of $ H $ belongs to $ W^s(i) $ if and only if for every $ m \in M $, $ |s(u^k(m)) - i(u^k(m)) | \lesssim {\hat{\varepsilon}_1}^{(k)}(m) $, $ k \to \infty $. Similarly, $ s $ belongs to $ W^u(j) $ if and only if for every $ m \in M $, $ |s(u^{-k}(m)) - i(u^{-k}(m)) | \lesssim {\hat{\varepsilon}_2}^{(k)}(m) $, $ k \to \infty $. Here $ \hat{\varepsilon}_i(\cdot) $, $ i = 1,2 $, are the functions in \textnormal{(3)} and $ W^s(i) = \bigcup_{m \in M} W^{s}_{m} (i) $, $ W^u(i) = \bigcup_{m \in M} W^{u}_{m} (j) $.

	\end{enumerate}
\end{thm}
In applications, we usually take $ i,j $ to be invariant sections of $ H $ or $ 1 $-pseudo-invariant sections.

\subsection{Hyperbolic trichotomy}

\begin{enumerate}
	\item [(HT)] Let $ u: M \to M $ be an invertible map. Let $ (X^\kappa, M, \pi) $, $ \kappa = s,c,u $, be three (set) bundles with metric fibers. Let  $ X = X^s \times X^c \times X^u $. Suppose $ H : X \to X $ is a bundle correspondence over $ u $ having generating bundle maps $ (F^{cs}, G^{cs}) $ and $ (F^{cu}, G^{cu}) $, where
	\[
	F^{cs}_m : X^{cs}_m \times X^u_{u(m)} \to X^{cs}_{u(m)}, ~G^{cs}_m : X^{cs}_m \times X^u_{u(m)} \to X^u_{m},
	\]
	\[
	F^{cu}_m : X^{s}_m \times X^{cu}_{u(m)} \to X^{s}_{u(m)}, ~G^{cu}_m : X^{s}_m \times X^{cu}_{u(m)} \to X^{cu}_{m}.
	\]

	Assume $ H \sim (F^{cs}, G^{cs}) $ satisfies the (A)$(\alpha_u, \lambda_u)$ (B)$(\beta_{cs}, \lambda_{cs})$ condition, and $ H \sim (F^{cu}, G^{cu}) $ satisfies (A)$(\alpha_{cu}, \lambda_{cu})$ (B)$(\beta_s, \lambda_s)$ where $ \alpha_{\kappa_1}, \lambda_{\kappa_1} $, $ \beta_{\kappa_2}, \lambda_{\kappa_2} $, $ \kappa_1 = u,cu $, $ \kappa_2 = s, cs $, are \emph{bounded} functions $M \rightarrow \mathbb{R}_+$ with the following properties:
\begin{enumerate}
	\item [(a)] (angle condition) $ \sup_m \alpha_u(m) \beta_{cs}(u(m)) < 1 $, $\sup_m \alpha_{cu}(m) \beta_{s}(u(m)) < 1 $; $ \alpha_{\kappa_1}(m) \leq \alpha_{\kappa_1}(u(m))$, $ \beta_{\kappa_2}(u(m)) \leq \beta_{\kappa_2}(m) $, $ m \in M $, $ \kappa_1 = u,cu $, $ \kappa_2 = s,cs $;
	\item [(a$ ' $)] $ \alpha_{cu}(m) \leq 1 $, $ \beta_{cs}(m) \leq 1 $, $ \beta_{cs}(m) \leq \beta_s(m) $, $ \alpha_{cu}(m) \leq \alpha_u(m) $, $ m \in M $, $ \sup_m\alpha_u(m) \beta_s(m) < 1 $;
	\item [(b)] (spectral condition) $ \sup_m \lambda_u(m) \lambda_{cs}(m) < 1 $, $ \sup_m\lambda_u(m) < 1 $; and 
	
	$ \sup_m \lambda_{cu}(m) \lambda_{s}(m) < 1 $, $ \sup_m\lambda_s(m) < 1 $.
\end{enumerate}
\end{enumerate}

Let $ X^{\kappa_1\kappa_2} = X^{\kappa_1} \times X^{\kappa_2} $ with fibers equipped with $ d_\infty $ metrics (see \eqref{metric}) where $ \kappa_1, \kappa_2 \in \{ s,c,u \} $, $ \kappa_1 \neq \kappa_2 $. Let $ P^{\kappa} : X \to X^\kappa $ be the natural bundle projection, $ \kappa \in \{  s, c, u, cs, cu \} $.

\begin{thm}\label{thm:pnormalTrivial}
	Let \textnormal{(HT)} hold. Assume there is a section $ i = (i_s, i_c, i_u) : M \to X^s \times X^c \times X^u $ which is a $ 1 $-pseudo-invariant section of $ H $ (see \autoref{pseudoSection}). Let $ \varsigma > 0 $ be sufficiently small. Choose any functions $ \hat{\varepsilon}_i(\cdot) $, $ i = 1,2 $, over $ u $ and $ u^{-1} $, respectively, such that $ \lambda_{cs}(m) + \varsigma < \hat{\varepsilon}_1(m) < \lambda^{-1}_u(m) - \varsigma $, and $ \lambda_{cu}(m) + \varsigma < \hat{\varepsilon}_2(m) < \lambda^{-1}_s(m) - \varsigma $, $ \hat{\varepsilon}_i(m) \geq 1 $, $ i =1,2 $, $ m \in M $, and any functions $ \tilde{\varepsilon}_i(\cdot) $, $ i = 1,2 $, over $ u $ and $ u^{-1} $, respectively, such that $ \lambda_s(m) + \varsigma < \tilde{\varepsilon}_1(m) < \lambda^{-1}_{cu}(m) - \varsigma $, $ \lambda_u(m) + \varsigma < \tilde{\varepsilon}_2(m) < \lambda^{-1}_{cs}(m) - \varsigma $, $ m \in M $. Then the following assertions hold:
	\begin{enumerate}[(1)]
		\item \label{x1} For any $ m \in M $, there is a Lipschitz graph $ W^{cs}_m $ (resp. $ W^{cu}_m $) of $ X^{cs}_m \to X^s_m $ (resp. $ X^{cu}_m \to X^u_m $) with the Lipschitz constant no more than $ \beta_{cs}(m) $ (resp. $ \alpha_{cu}(m) $) such that $ W^{cs}_m \subset H^{-1}_mW^{cs}_{u(m)} $ (resp. $ W^{cu}_m \subset H_{u^{-1}(m)}W^{cu}_{u^{-1}(m)} $). Moreover,
		\begin{align*}
		W^{cs}_m = \{ z_m \in X_m:~ & \text{there is a forward orbit $ \{ z_k : k=0, 1, 2, \ldots \} $ of $ H $ such that} \\
		& z_0 = z_m,~ \sup_{k\geq 0} (\hat{\varepsilon}^{(k)}_1(m))^{-1} | z_k - i(u^k(m)) | < \infty \},\\
		W^{cu}_m = \{ z_m \in X_m:~ & \text{there is a backward orbit $ \{ z_{-k} : k=0, 1, 2, \ldots \} $ of $ H $ such that} \\
		& z_0 = z_m,~ \sup_{k\geq 0} (\hat{\varepsilon}^{(k)}_2(m))^{-1} | z_{-k} - i(u^{-k}(m)) | < \infty \},
		\end{align*}

		\item \label{x2} $ W^c_m = W^{cs}_m \cap W^{cu}_m $ is a Lipschitz graph of $ X^{c}_m \to X^{s}_m \times X^{u}_m $ such that $  W^c_m \subset H^{-1}_m W^{c}_{u(m)} $, $ W^c_m \subset H_{u^{-1}(m)} W^{c}_{u^{-1}(m)} $, $ m \in M $. In addition,
		\begin{align*}
		W^{cs}_m = & \{ z \in X_m: \exists \{ z_k \}_{k \geq 0}, z_{k+1} \in H_{u^{k}(m)} (z_k), z_0 = z, d(z_k, W^c_{u^k(m)}) \to 0, k \to \infty \}\\
		= & \{ z \in X_m: \exists \{ z_k \}_{k \geq 0}, z_{k+1} \in H_{u^{k}(m)} (z_k), z_0 = z, \sup_{k\geq 0}d(z_k, W^c_{u^k(m)}) < \infty \}, \\
		W^{cu}_m = & \{ z \in X_m: \exists \{ z_{-k} \}_{k \geq 0}, z_{-k+1} \in H_{u^{-k}(m)} (z_{-k}), z_0 = z, \\
		& \qquad \qquad \qquad \qquad \qquad \qquad d(z_{-k}, W^c_{u^{-k}(m)}) \to 0, k \to \infty \}\\
		= & \{ z \in X_m: \exists \{ z_{-k} \}_{k \geq 0}, z_{-k+1} \in H_{u^{-k}(m)} (z_{-k}), z_0 = z,\\
		& \qquad \qquad \qquad \qquad \qquad \qquad \qquad \sup_{k\geq 0}d(z_{-k}, W^c_{u^{-k}(m)}) < \infty \}.
		\end{align*}

		\item \label{x3} (exponential tracking and shadowing orbits) If $ \{ z_{k} \}_{k\geq 0} $ is a forward orbit of $ H $ from $ m $, then there is a forward orbit $ \{ \overline{z}_{k} \}_{k \geq 0} \subset W^{cu} $ ($ \overline{z}_0 \in W^{cu}_m $) such that $ |z_{k} - \overline{z}_{k}| \lesssim \tilde{\varepsilon}_{1}^{(k)}(m) $, $ k \to \infty $; similarly, if $ \{ z_{-k} \}_{k\geq 0} $ is a backward orbit of $ H $ from $ m $, then there is a backward orbit $ \{ \overline{z}_{-k} \}_{k \geq 0} \subset W^{cs} $ ($ \overline{z}_0 \in W^{cs}_m $) such that $ |z_{-k} - \overline{z}_{-k}| \lesssim \tilde{\varepsilon}_{2}^{(k)}(m) $, $ k \to \infty $.

		\item \label{x4}
		Let $ H^{cs} = H|_{W^{cs}} $, $ H^{-cu} = H^{-1}|_{W^{cu}} $, $ H^{c} = H|_{W^{c}} $ be the bundle maps over $ u, u^{-1}, u $ induced by $ H: W^{cs} \to W^{cs} $, $ H^{-1}: W^{cu} \to W^{cu} $, and $ H: W^{c} \to W^{c} $, respectively.

		There are bundle maps $ h^{cs}: X^{cs} \to X^{cs} $, $ h^{-cu}: X^{cu} \to X^{cu} $, $ h^{c}: X^{c} \to X^{c} $ over $ u $, $ u^{-1}, u $, respectively, with $ \lip h^{cs}_m, \lip h^{s}_m \leq \lambda_{cs}(m) $ and $ \lip h^{-cu}_m, \lip (h^{c}_m)^{-1} \leq \lambda_{cu}(m) $ such that the following diagrams are commutative:
		\[
		\CD
		W^{cs} @> H^{cs}>> W^{cs} \\
		@V P^{cs} VV @V P^{cs} VV  \\
		X^{cs} @> h^{cs} >> X^{cs} \\
		@V \pi VV @V \pi VV \\
		M @> u >> M
		\endCD
		\quad\quad
		\CD
		W^{cu} @> H^{-cu}>> W^{cu} \\
		@V P^{cu} VV @V P^{cu} VV  \\
		X^{cu} @> h^{-cu} >> X^{cu} \\
		@V \pi VV @V \pi VV \\
		M @> u^{-1} >> M
		\endCD
		\quad\quad
		\CD
		W^{c} @> H^{c}>> W^{c} \\
		@V P^{c} VV @V P^{c} VV  \\
		X^{c} @> h^{c} >> X^{c} \\
		@V \pi VV @V \pi VV \\
		M @> u >> M
		\endCD
		\]

		\item \label{x5} There are Lipschitz graphs $ W^{ss}_z $ (resp. $ W^{uu}_z $), $ z \in W^{cs}_m $ (resp. $ z \in W^{cu}_m $) with Lipschitz constants no more than $ \beta_s(m) $ (resp. $ \alpha_u(m) $) such that $ z \in W^{ss}_z $ (resp. $ z \in W^{uu}_z $) and $ H^{cs}_mW^{ss}_z \subset W^{ss}_{H^{cs}_m(z)} $ (resp. $ H^{-cu}_mW^{uu}_z \subset W^{uu}_{H^{-cu}_m(z)} $), $ m \in M $.

		If $ z \in W^{cs}_m $, then $ z' \in W^{ss}_z $ if and only if $ |(H^{cs})^{(k)}_m(z') - (H^{cs})^{(k)}_m(z)| \lesssim \tilde{\varepsilon}_1^{(k)}(m) $, $ k \to \infty  $; similarly if $ z \in W^{cu}_m $, then $ z' \in W^{uu}_z $ if and only if $ |(H^{-cu})^{(k)}_m(z') - (H^{-cu})^{(k)}_m(z)| \lesssim \tilde{\varepsilon}_2^{(k)}(m) $, $ k \to \infty  $.

		Now $ W^{cs}_m $, $ W^{cu}_m $ can be regarded as bundles over $ W^{c}_m $ with fibers $ W^{ss}_z $, $ W^{uu}_z $, $ z \in W^c_m $, respectively. So are $ W^{cs} $, $ W^{cu} $, $ W^{cs} \times W^{cu} \cong X^s \times X^c \times X^u $ over $ W^{c} $.
	\end{enumerate}
\end{thm}

For the regularity of $ W^{cs} $, $ W^{cu} $, and $ W^{ss}_z $, $ z \in W^{cs} $, $ W^{uu}_z $, $ z \in W^{cu} $, see \autoref{stateRegularity} and \autoref{thm:bundlemaps}.
\begin{proof}
	\textbf{(I)}
	Since $ i = (i_s, i_c, i_u) : M \to X^s \times X^c \times X^u $ is a $ 1 $-pseudo-invariant section of $ H $, by \autoref{thmA}, there are invariant Lipschitz graphs $ W^{cs} = \graph f^{cs} $, $ W^{cu} = \graph f^{cu} $ of $ H $ such that $ W^{cs} \subset HW^{cs} $ and $ W^{cu} \subset H^{-1}W^{cu} $, where $ f^{cs}: X^{cs} \to X^{u} $ over $ \id $ with $ \lip f^{cs}_m \leq \beta_{cs}(m) $ and $ f^{cu}: X^{cu} \to X^{s} $ over $ \id $ with $ \lip f^{cu}_m \leq \alpha_{cu}(m) $.
	Note that
	\begin{equation}\label{repreS}
	H^{cs} = H|_{W^{cs}}: W^{cs} \to W^{cs},~ (m,x^{cs}, f^{cs}_m(x^{cs})) \mapsto (u(m), h^{cs}_m(x^{cs}), f^{cs}_{u(m)}(h^{cs}_m(x^{cs})),
	\end{equation}
	is a bundle map over $ u $, and similarly
	\begin{multline}\label{repreU}
	H^{-cu} = H^{-1}|_{W^{cu}}: W^{cu} \to W^{cu},~ (m,x^{cu}, f^{cu}_m(x^{cu})) \mapsto \\
	( u^{-1}(m), h^{-cu}_m(x^{cu}), f^{cu}_{u^{-1}(m)}(h^{-cu}_m(x^{cu})) ),
	\end{multline}
	is a bundle map over $ u^{-1} $. Note that $ \lip h^{cs}_m \leq \lambda_{cs}(m) $ and $ \lip h^{-cu}_m \leq \lambda_{cu}(m) $. Combining this with the characterizations in \autoref{properties}, we see that conclusions \eqref{x1} and \eqref{x4} hold.

	\textbf{(II)} Under $ \sup_m\alpha_{cu}(m) \beta_{cs}(m) < 1 $, one has $ h^{cs} \sim (\widetilde{F}^s, \widetilde{G}^s): X^{s} \times X^{c} \to X^{s} \times X^{c} $ and $ h^{-cu} \sim (\widetilde{F}^u, \widetilde{G}^u): X^{c} \times X^{u} \to X^{c} \times X^{u} $. Take $ h^{-cu} $ as an example. Let $ F^{cs}_m = (F^{cs1}_m, F^{cs2}_m) $, where $ F^{cs1}_m : X^{cs}_m \times X^u_{u(m)} \to X^{c}_{u(m)} $ and $ F^{cs2}_m : X^{cs}_m \times X^u_{u(m)} \to X^{s}_{u(m)} $. For any $ (y_1, z_2) \in X^{c}_m \times X^{u}_{u(m)} $, we see there is only one $ z_1 \in X^u_m $ such that
	\[
	z_1 = G^{cs}_m(f^{cu}_m(y_1,z_1), z_2) \triangleq \widetilde{G}^u_m(y_1, z_2);
	\]
	define $ \widetilde{F}^u_m(y_1,z_2) \triangleq F^{cs1}_m(f^{cu}_m(y_1,z_1), z_2) $.
	Moreover, the following statements hold:
	\begin{enumerate}[(a)]
		\item If $ \alpha_{cu}(m) \leq 1 $, $ \beta_{cs}(m) \leq 1 $, $ \beta_{cs}(m) \leq \beta_s(m) $, $ m \in M $, then $ h^{cs} \sim (\widetilde{F}^s, \widetilde{G}^s) $ satisfies the (A)$(\alpha_{cu}, \lambda_{cu})$ (B)$(\beta_{s}, \lambda_{s})$ condition;
		\item If $ \alpha_{cu}(m) \leq 1 $, $ \beta_{cs}(m) \leq 1 $, $ \alpha_{cu}(m) \leq \alpha_u(m) $, $ m \in M $, then $ h^{-cu} \sim (\widetilde{F}^u, \widetilde{G}^u) $ satisfies the (A)$(\alpha_{u}, \lambda_{u})$ (B)$(\beta_{cs}, \lambda_{cs})$ condition.
	\end{enumerate}
	The proof is direct, so we omit it. Now for $ h^{cs} $, by \autoref{thmA} (or \autoref{thm:bundlemaps}), there are Lipschitz maps $ f^{ss}_{z_m}: X^s_m \to X^c_m $, $ z_m \in X^s_m \times X^c_m $ with $ \lip f^{ss}_{z_m} \leq \beta_s(m) $ such that $ z_m \in \graph f^{ss}_{z_m} $ and $ h^{cs}_{z_m} \graph f^{ss}_{z_m} \subset \graph f^{ss}_{h^{cs}_m(z_m)} $. Note that by the characterization (see \autoref{properties}), $ \graph f^{ss}_{z'_m} \cap \graph f^{ss}_{z''_m} = \emptyset $ or $ \graph f^{ss}_{z'_m} = \graph f^{ss}_{z''_m} $; in particular $ \bigsqcup_{z_m \in X^s_m \times X^u_m} \graph f^{ss}_{z_m} $ foliates $ X^s_m \times X^u_m $.
	Now using $ f^{ss}_{(\cdot)} $, one obtains $ W^{ss}_z $, $ z \in W^{cs}_m \triangleq \graph f^{cs}_m $ such that $ z \in W^{ss}_z $, $ H^{cs}_m W^{ss}_z \subset W^{ss}_{H^{cs}_m(z)} $ and $ W^{ss}_z $ is a Lipschitz graph of $ X^s_m \to X^{cu}_m $. $ \bigsqcup_{z \in W^{cs}_m} \graph W^{ss}_{z} $ foliates $ W^{cs}_m $.
	Similarly, one gets $ W^{uu}_z $, $ z \in W^{cu}_m \triangleq \graph f^{cu}_m $ such that $ z \in W^{uu}_z $, $ H^{-cu}_m W^{uu}_z \subset W^{uu}_{H^{-cu}_m(z)} $ and $ W^{uu}_z $ is a Lipschitz graph of $ X^u_m \to X^{cs}_m $. $ \bigsqcup_{z \in W^{cu}_m} \graph W^{uu}_{z} $ foliates $ W^{cu}_m $.

	Since $ \sup_m\alpha_{cu}(m) \beta_{cs}(m) < 1 $, one has $ W^c \triangleq W^{cs} \cap W^{cu} = \graph f^c $, where
	\[
	f^c_m = (f^{1c}_m, f^{2c}_m) : X^c_m \to X^s_m \times X^u_m,
	\]
	with $ \lip f^{1c}_m \leq \alpha_{cu}(m) $, $ \lip f^{2c}_m \leq \beta_{cs}(m) $. Note that $ W_c \subset H^{\pm 1} W_c $. So
	\begin{multline}\label{repreC}
	H^c = H|_{W^c} : W^c \to W^c,
	(m, f^{1c}_m(x^c), x^c, f^{2c}_m(x^c)) \mapsto \\
	(u(m), f^{1c}_{u(m)}(h^{c}_m(x^c)), h^{c}_m(x^c), f^{2c}_{u(m)}(h^{c}_m(x^c))),
	\end{multline}
	where $ h^c_m : X^c_m \to X^{c}_{u(m)} $ is a bi-Lipschitz map with 
	\[
	\lip h^c_m \leq \lambda_{cs}(m),~ \lip (h^{c}_m)^{-1} \leq \lambda_{cu}(m).
	\]
	This also shows that if $ X^c_m $, $ m \in M $, are all more than one point, then $ \lambda_{cs}(m) \lambda_{cu}(m) \geq 1 $; in particular, we can assume
	\[
	\lambda_s(m) < \lambda^{-1}_{cu}(m) \leq \lambda_{cs}(m) < \lambda^{-1}_u(m).
	\]

	We show in fact $ W^{cs} $, $ W^{cu} $ are bundles over $ W^c $ with fibers $ W^{ss}_z $, $ W^{uu}_z $, $ z \in W^c $, respectively. Let $ W^c_m = W^{cs}_m \cap W^{cu}_m $. Due to $ \sup_m\alpha_u(m) \beta_s(m) < 1 $, one can define $ \pi^s_c(m,z) $ as the unique point in $ W^{ss}_z \cap W^c_m $, where $ z \in W^{cs}_m $. So $ (W^{cs}, W^c, \pi^s_c) $ is a bundle, called the \emph{strong stable fiber bundle}. Similarly, one defines $ \pi^{u}_c $, and $ (W^{cu}, W^c, \pi^u_c) $ is a bundle called the \emph{strong unstable fiber bundle}.

	We have shown that conclusion \eqref{x5} and the first part of conclusion \eqref{x2} hold.

	\textbf{(III)} Finally, we show the characterization of $ W^{cs} $, $ W^{cu} $ by using $ W^c $ and conclusion (3). Let $ \{ z_{k} \}_{k\geq 0} $ be a forward orbit of $ H $ from $ m $, and $ M_1 = \{u^k(m): k \in \mathbb{N} \} $. Applying \autoref{thmA} to $ X^1 \triangleq X|_{M_1} $, $ H|_{X^1} $, $ M_1 $, $ u $ with an invariant section $ \overline{i} $ defined by $ \overline{i}(u^{k}(m)) = z_k $, one gets a Lipschitz graph $ W^s_{u^k(m)}(\overline{i}) $ of $ X^{s}_{u^k(m)} \to X^{cu}_{u^k(m)} $, which is invariant under $ H $, with Lipschitz constant no more than $ \beta_s(u^k(m)) $. Let $ \{\overline{z}_k\} = W^s_{u^k(m)}(\overline{i}) \cap W^{cu}_{u^k(m)} $. One can easily check that $ \{\overline{z}_k\}_{k\geq 0} $ is a forward orbit of $ H $ and $ |z_{k} - \overline{z}_{k}| \lesssim \varepsilon_{1}^{(k)}(m) $, $ k \to \infty $ (see e.g. \autoref{dichotomy}). We have proved conclusion \eqref{x3}. Let us show that $ W^{cs}_m $ has the characterization in conclusion \eqref{x2}.

	(i) If $ z_m \in W^{cs}_m $, then there is a forward orbit $ \{ z_{k} : k=0, 1, 2, \ldots \} $ of $ H $ such that $ z_0 = z_m $. For this orbit $ \{z_k\}_{k\geq 0} $, one has $ \{ \overline{z}_k \}_{k\geq 0} \subset W^{cu} $ such that $ |z_k - \overline{z}_{k}| \to 0 $. This also shows that by the characterization of $ W^{cs} $, $ \{\overline{z}_{k}\} \subset W^{cs} $. Hence, $ \{\overline{z}_{k}\} \subset W^c $, which gives $ d(z_k, W^c_{u^k(m)}) \leq |z_k - \overline{z}_{k}| \to 0 $.

	(ii) If there is a forward orbit $ \{ z_{k} \}_{k \geq 0} $ of $ H $ such that 
	\[
	z_0 = z_m \in X_m~ \text{and}~\sup_{k\geq 0}d(z_k, W^c_{u^k(m)}) < \infty,
	\]
	then we need to show $ z_m \in W^{cs}_m $.
	For this forward orbit $ \{z_k\} $, there is a forward orbit $ \{ \overline{z}_{k} \}_{k \geq 0} \subset W^{cu} $ ($ \overline{z}_0 \in W^{cu}_m $) such that $ \sup_{k\geq 0}|z_{k} - \overline{z}_{k}| < \infty $. Write $ \overline{z}_{k} = ({x}^s_{k}, {x}^c_{k}, {x}^u_{k}) $. First we show $ \sup_{k \geq 0}|{x}^u_{k} - f^{2c}_{u^k(m)} ({x}^c_k)| < \infty $.

	Take $ \hat{z}_k = (f^{1c}_{u^k(m)} ({x}^c_k), {x}^c_k, f^{2c}_{u^k(m)} ({x}^c_k) ) \in W^c_{u^k(m)} $. Since $ \sup_{k\geq 0}d(z_k, W^c_{u^k(m)}) < \infty $, there is $ \{ \tilde{z}_k = (f^{1c}_{u^k(m)} (\tilde{x}^c_k), \tilde{x}^c_k, f^{2c}_{u^k(m)} (\tilde{x}^c_k) ) \} \subset W^c $ such that $ \sup_{k \geq 0}|z_k - \tilde{z}_k| < \infty $. So $ \sup_{k \geq 0}|\overline{z}_k - \tilde{z}_k| < \infty $, and in particular $ \sup_{k \geq 0}|{x}^c_k - \tilde{x}^c_k| < \infty $. Thus, $ \sup_{k \geq 0}|\tilde{z}_{k} - \hat{z}_k| < \infty $, which yields
	\[
	\sup_{k \geq 0}|\overline{z}_{k} - \hat{z}_k| \leq \sup_{k \geq 0}|\overline{z}_{k} - \tilde{z}_k| + \sup_{k \geq 0}|\tilde{z}_{k} - \hat{z}_k| < \infty,
	\]
	and so $ \sup_{k \geq 0}|{x}^u_{k} - f^{2c}_{u^k(m)} ({x}^c_k)| < \infty $.
	Set
	\[
	(H^c)^{(k)}_m (\hat{z}_0) = ({x^s_k}', {x^c_k}', {x^u_k}') \in W^c.
	\]
	Then by the (A) condition for $ h^{-cu} $, one has for $ k \geq 1 $,
	\begin{multline*}
	|{x^c_k}' - x^c_k| \leq \alpha_u(u^{k-1}(m)) |{x^u_k}' - x^u_k| \\
	\leq \alpha_u(u^{k-1}(m)) |{x^u_k}' - f^{2c}_{u^k(m)} ({x}^c_k)| + \alpha_u(u^{k-1}(m)) |f^{2c}_{u^k(m)} ({x}^c_k) - x^u_k|,
	\end{multline*}
	yielding 
	\[
	|{x^c_k}' - x^c_k| \leq \frac{\alpha_u(u^{k-1}(m))}{1 - \alpha_u(u^{k-1}(m))\beta_{cs}(u^k(m))} |f^{2c}_{u^k(m)} ({x}^c_k) - {x}^u_{k}|.
	\]
	Thus, $ \sup_{k \geq 0}|{x^c_k}' - x^c_k| < \infty $. This implies that $ \sup_{k \geq 0} (\hat{\varepsilon}^{(k)}_1(m))^{-1} |i(u^{k}(m)) - z_k| < \infty  $, giving $ z_m \in W^{cs}_m $ (by \eqref{x1}). The proof is complete.
\end{proof}

\begin{rmk}[(Partial) normal hyperbolicity] \label{pnormalhyper}
	Let $ H: Z \to Z $ be a smooth map and $ M \subset Z $, $ H(M) \subset M $, where $ Z $ is a smooth Finsler manifold and $ M $ is a submanifold of $ Z $. Assume $ M $ is invariant and normally hyperbolic with respect to $ H $ \cite{HPS77,Fen72}. Then we have the center-(un)stable manifolds $ W^{cs}(M) $, $ W^{cu}(M) $ of $ M $ for $ H $, and $ M = W^{cs}(M) \cap W^{cu}(M) $. A small $ C^1 $ perturbation $ \tilde{H} $ of $ H $ corresponds to $ \tilde{W}^{cs} $, $ \tilde{W}^{cu} $ and $ \tilde{M} = \tilde{W}^{cs} \cap \tilde{W}^{cu} $. Now $ \tilde{M} $ is diffeomorphic to $ M $, which means that $ M $ is persistent under small $ C^1 $ perturbations. In general, \autoref{thm:pnormalTrivial} cannot be used to derive the existence of the above invariant manifolds. However, \emph{there is a situation where our results can be applied}: The normal bundle $ \mathcal{N} $ of $ M $ in $ Z $ can embed into a trivial bundle $ M \times Y $ where $ Y $ is a Banach space, and $ H $ can extend to $ M \times Y $, denoted by $ \hat{H} $, so that $ M $ is also normally hyperbolic for $ \hat{H} $. Now applying \autoref{thm:pnormalTrivial} to $ \hat{H} $, $ M \times Y $, one gets invariant manifolds of $ \hat{H} $, then pulls back all these manifolds to $ Z $ to get the desired manifolds for $ H $. This can be done for example when $ Z $ is a smooth compact Riemannian manifold and $ M $ is a compact submanifold of $ Z $ (see e.g. \cite{HPS77}), or $ Z $ is a smooth Riemannian manifold having bounded geometry (see \autoref{defi:bounded}) and $ M $ is a complete immersed submanifold of $ Z $ (see e.g. \cite{Eld13}). (Lipschitz) trivialization of vector bundles is important in \cite{HPS77, Eld13} for the study of normal hyperbolicity. Also, note that $ C^0 $ trivialization is insufficient, for one needs the persistence of normal hyperbolicity in the new trivial vector bundle.

	A notion of hyperbolicity between normal hyperbolicity and partial hyperbolicity, namely \emph{partial normal hyperbolicity} introduced in \cite{Che18}, is of importance due to both its theoretical and practical applications (see e.g. \cite{CLY00,BC16,LLSY16}). Partial normal hyperbolicity can make sense for sets \cite{CLY00a, BC16, Che18}. Analogously, if the normal bundle of a partially normally hyperbolic manifold can embed into a trivial bundle and the dynamics can extend to the trivial bundle maintaining hyperbolic trichotomy, then \autoref{thm:pnormalTrivial} can be applied directly.

	In \cite{Che18b, Che18}, the associated results of normal hyperbolicity and partial normal hyperbolicity for maps and semiflows (see e.g. the references listed above) are extended to correspondences in Banach spaces (so that these results can be applied to some ill-posed differential equations).

\end{rmk}

In the following, we show how to use invariant foliations to decouple the bundle map $ H $ when $ H $ is invertible.
\begin{cor}[Decoupling theorem]\label{decoupling}
	In the setting of \autoref{thm:pnormalTrivial}, assume $ H $ is invertible (i.e., for very $ m \in M $, $ H_m $ is an invertible map). Then $ H $ can be decoupled as $ h^s \times h^c \times h^u $ with bundle maps $ h^s : X^{cs} \to X^s $, $ h^u: X^{cu} \to X^u $ and $ h^c : X^c \to X^c $ over $ u $, i.e., there is an invertible bundle map $ \Pi: X \to X $ over $ \id $ such that the following diagram is commutative:
	\[
	\CD
	X^s \times X^c \times X^u @> H>> X^s \times X^c \times X^u \\
	@V \Pi VV @V \Pi VV  \\
	X^s \times X^c \times X^u @> h^s \times h^c \times h^u >> X^s \times X^c \times X^u
	\endCD
	\]
\end{cor}
\begin{proof}
	By \autoref{thm:bundlemaps}, we have the following different invariant foliations: $ W^{\kappa_1}(z) $, $ z \in X $, $ \kappa_1 \in \{ s, u, cs, cu \} $, such that
	\begin{enumerate}[(a)]
		\item $ W^{\kappa_1}(z) $ is a Lipschitz graph of $ X^{\kappa_1}_m \to X^{\kappa_2}_m $ with Lipschitz constant no more than $ \varpi_{\kappa_1} $, if $ \pi(z) = m  $, where $ \varpi_{\kappa_1} = \beta_{\kappa_1}(m) $ if $ \kappa_1 = s, cs $, and $ \varpi_{\kappa_1} = \alpha_{\kappa_1}(m) $ if $ \kappa_1 = u, cu $, and where $ \kappa_2 = scu - \kappa_1 $ (meaning that the letter $ \kappa_1 $ is deleted from $ scu $; e.g., if $ \kappa_1 = s $, then $ \kappa_2 = cu $).
		\item $ z \in W^{\kappa_1}(z) $, $ HW^{\kappa_1}(z) = W^{\kappa_1}(H(z)) $, where $ \kappa_1 \in \{ s, u, cs, cu \} $.
	\end{enumerate}

	We also need all the invariant foliations obtained in \autoref{thm:pnormalTrivial}, i.e., $ W^{c\kappa} $ and $ W^{\kappa\kappa}_z $, $ z \in W^{c\kappa} $, $ \kappa = s, u $.
	Note that if $ z \in W^{cs} $, then $ W^{ss}_z = W^{s}(z) $; similarly if $ z \in W^{cu} $, then $ W^{uu}_z = W^{u}(z) $. Also, $ W^{s} $ (resp. $ W^{u} $) subfoliates each leaf of $ \bigsqcup_{z\in X} W^{cs}(z) $ (resp. $ \bigsqcup_{z\in X} W^{cu}(z) $) (which means that e.g. $ \{ W^{s}(z'): z' \in W^{cs}(z) \} $ is a `subfoliation' of $ W^{cs}(z) $ for any $ z \in X $). These follow from the characterizations of these invariant foliations (see \autoref{properties}).

	Define the following maps. For $ m_c \in W^c $,
	\[
	\pi^s_{m_c}: z \mapsto W^{cu}(z) \cap W^s(m_c), ~ \pi^u_{m_c}: z \mapsto W^{cs}(z) \cap W^u(m_c),
	\]
	\[
	\pi^s(m_c,z) = (m_c, \pi^s_{m_c}(z)), ~ \pi^u(m_c,z) = (m_c, \pi^u_{m_c}(z)).
	\]
	Consider $ W^{cs} $ and $ W^{cu} $ as bundles over $ W^c $. Let $ H^s $ be another representation of $ H^{cs} $ as a bundle map over $ H^c $, i.e.,
	\[
	H^{s}: (m_c, z) \mapsto (H^c(m_c), H^{cs}(z)), ~  z \in W^{ss}_{m_c}, ~ m_c \in W^c.
	\]
	Similarly,
	\[
	H^{u}: (m_c, z) \mapsto (H^c(m_c), H^{cu}(z)), ~  z \in W^{uu}_{m_c}, ~ m_c \in W^c,
	\]
	where $ H^{cu} = (H^{-cu})^{-1} $.
	By the invariance of the foliations, we have
	\[
	\pi^\kappa (H(m_c), H(z)) = ( H(m_c), \pi^\kappa_{H(m_c)} (H(z)) ) = H^\kappa(m_c, \pi^\kappa_{m_c}(z)),~ \kappa = s, u.
	\]
	In particular, $ \pi^\kappa \circ (H^c, H) = H^\kappa \circ \pi^\kappa $, $ \kappa = s, u $.

	Let us give the precise meaning of $ X^s \times X^c \times X^u = X \cong W^{cs} \times W^{cu} $. The most important thing is that we need to find a way to track the base points, which can be done as follows. Let
	\[
	\widetilde{\pi}_c: X \to W^c, z \mapsto W^{u}(W^s(z) \cap W^{cu}) \cap W^c,
	\]
	and
	\[
	\widetilde{\pi} (z) = (m_c, \pi^s_{m_c}z, \pi^u_{m_c}z) : X \to W^{cs} \times W^{cu},
	\]
	where $ m_c = \widetilde{\pi}_c(z) $. By the invariance of those foliations, we see $ \widetilde{\pi}_c(H(z)) = H^c(\widetilde{\pi}_c(z)) $. What we have shown is the following commutative diagram:
	\[
	\CD
	X^s \times X^c \times X^u @> H>> X^s \times X^c \times X^u \\
	@V \widetilde{\pi} VV @V \widetilde{\pi} VV  \\
	W^{cs} \times W^{cu} @> H^s \times H^u >> W^{cs} \times W^{cu} \\
	@V \pi_c VV @V \pi_c VV \\
	W^c @> H^c >> W^c
	\endCD
	\]
	where $ \pi_c $ is the projection of the bundle $ (W^{cs} \times W^{cu}, W^c, \pi_c) $.

	Represent $ H^{cs} $, $ H^{cu} $ in $ X $. Let $ m_c = (m, z_c) \in W^c $, $ z_c \in W^c_m $. Note that $ P^c_m: W^c_m \cong X^c_m $. Since $ H^{cs}: W^{ss}_{m_c} \to W^{ss}_{H(m_c)} $ and $ W^{ss}_{m_c} \cong X^s_m $ (through $ P^s_m $), we have a map $ h^s_m(P^c_m z_c, \cdot): X^s_m \to X^s_{u(m)} $ induced by $ H^{cs} $ (similar to \eqref{repreS}). Also, we have another map $ h^u_m(P^c_m z_c, \cdot): X^u_m \to X^u_{u(m)} $ induced by $ H^{cu} $ (similar to \eqref{repreU}).
	By our construction, we get
	\[
	P^{s}_{u(m)} H^{cs}_m (z) = h^s_m(P^c_m z_c, P^{s}_m z), ~ z \in W^{ss}_{m_c}, ~m_c = (m, z_c) \in W^c;
	\]
	analogously,
	\[
	P^{u}_{u(m)} H^{cu}_m (z) = h^u_m(P^c_m z_c, P^{u}_m z), ~ z \in W^{uu}_{m_c}, ~m_c = (m, z_c) \in W^c.
	\]
	Now we have the following commutative diagram:
	\[
	\CD
	W^{cs} \times W^{cu} @> H^s \times H^u >> W^{cs} \times W^{cu} \\
	@V \widehat{\pi} VV @V \widehat{\pi} VV \\
	X^s \times X^c \times X^u @> h^s \times h^c \times h^u >> X^s \times X^c \times X^u \\
	@V P^c VV @V P^c VV  \\
	X^c @> h^c >> X^c
	\endCD
	\]
	where $ h^c $ is defined in \eqref{repreC} and
	\[
	\widehat{\pi} (m_c, z_s, z_u) = (P^c_m m_c, P^s_m z_s, P^u_m z_u): W^{cs} \times W^{cu} \to X,
	\]
	where $ (z_s, z_u) \in W^{ss}_{m_c} \times W^{uu}_{m_c}, ~m_c = (m, z_c) \in W^c $. Combining the above two commutative diagrams, we have
	\[
	\CD
	X^s \times X^c \times X^u @> H>> X^s \times X^c \times X^u \\
	@V \widetilde{\pi} VV @V \widetilde{\pi} VV  \\
	W^{cs} \times W^{cu} @> H^s \times H^u >> W^{cs} \times W^{cu} \\
	@V \widehat{\pi} VV @V \widehat{\pi} VV \\
	X^s \times X^c \times X^u @> h^s \times h^c \times h^u >> X^s \times X^c \times X^u
	\endCD
	\]
	Therefore we decouple $ H $ as $ h^s \times h^c \times h^u $ with $ h^s : X^{cs} \to X^s $, $ h^u: X^{cu} \to X^u $ and $ h^c : X^c \to X^c $ if we show $ \widetilde{\pi} $ and $ \widehat{\pi} $ are invertible.

	That $ \widehat{\pi} $ is invertible is easy since
	\[
	\widehat{\pi}^{-1}_m (x^s, x^c, x^u) = (z_c, z_s, z_u),
	\]
	where $ z_c = (f^{1c}_m(x^c), x^c, f^{2c}_m(x^c)) $, $ x^{cs} = (x^s, f^{ss}_{(m,z_c)}(x^s)) $, $ z_s = (x^{cs}, f^{cs}_m(x^{cs}) ) $, and similarly for $ z_u $.

	Next we show $ \widetilde{\pi} $ is invertible. Let $ (z_s, z_u) \in W^{ss}_{m_c} \times W^{uu}_{m_c} $, $ m_c = (m,z_c) \in W^c $. We need to find $ a \in X_m $ such that
	\[
	W^{u}(W^s(a) \cap W^{cu}) \cap W^c = z_c, ~W^{cu}(a) \cap W^{ss}_{m_c} = z_s, ~W^{cs}(a) \cap W^{uu}_{m_c} = z_u.
	\]
	Set $ b = W^{cs}(z_u) \cap W^{uu}_{m_c} $, $ a = W^{cu}(z_s) \cap W^s(b) $. We show $ a $ is the desired choice. (Note that here we identify $ (m, x) = x $, if $ x \in X_m $.)

	(i) Since $ a \in W^{cu}(z_s) $, one has $ W^{cu}(a) \cap W^{ss}_{m_c} = z_s $. (ii) Since $ b \in W^{cs}(z_u) $, by the subfoliation property one gets $ W^{s}(b) \subset W^{cs}(z_u) $. Thus, $ a \in W^{s}(b) \subset W^{cs}(z_u) $, which yields $ W^{cs}(a) \cap W^{uu}_{m_c} = z_u $. (iii) As $ a \in W^s(b) $, one has $ W^s(a) = W^s(b) $. And $ b \in W^{cu}_m $ (note that $ W^{uu}_{m_c} \subset W^{cu}_m $), so $ W^s(a) \cap W^{cu}_m = W^s(b)\cap W^{cu}_m = b \in W^{uu}_{m_c} $. Therefore, $ z_c = W^{u}(W^s(a) \cap W^{cu}) \cap W^c $. The proof is finished.
\end{proof}

\begin{rmk}\label{rmk:decoupling}
	\begin{asparaenum}[(a)]
		\item The regularity of $ \Pi $ in \autoref{decoupling} relies on the regularity of the foliations $ W^{\kappa}(z) $, $ z \in X $, $ \kappa \in \{ s, u, cs, cu \} $, and $ W^{cs} $, $ W^{cu} $, which we will study in \autoref{stateRegularity}. A simple result is that if $ X^{\kappa} $, $ \kappa \in \{ s, c, u \} $, are $ C^0 $ bundles and $ u $, $ H $ are continuous in addition to all the functions in the (A) (B) condition being (almost) continuous, then $ \Pi $ is homeomorphic. If $ H $ is bi-Lipschitz in the fiber topology, then $ \Pi $ is at least bi-H\"older in the fiber topology (see e.g. \autoref{thm:bundlemaps}). We do not give a detailed statement on the regularity of $ \Pi $ here.
		\item The choice of $ W^{cs} $, $ W^{cu} $ is not unique. Different choices of $ W^{cs} $, $ W^{cu} $ give different decoupled systems $ h^s \times h^c \times h^u $. It is important for us to give a further refined decoupled system of $ H $. The existence of a $ 1 $-pseudo-invariant section of $ H $ and the condition $ \lambda_s, \lambda_u < 1 $ are in order to ensure the existence of $ W^{cs} $, $ W^{cu} $. Instead, if there is an invariant section of $ H $, the condition $ \lambda_s < 1 $, $ \lambda_u < 1 $ can be removed. All we need is to use \autoref{thmA} (or \autoref{thmB}) to produce invariant graphs $ W^{cs} $, $ W^{cu} $ such that $ W^{cs} \times W^{cu} $ is a bundle over $ W^{cs} \cap W^{cu} $.

		\item The existence of $ \Pi $ heavily relies on the existence of the foliations $ W^{\kappa}(z) $, $ z \in X $, $ \kappa \in \{ s, u, cs, cu \} $, which may not exist if $ H $ is not invertible. In general, one cannot expect the decoupling theorem holds for a bundle correspondence or a non-invertible bundle map. However, in \cite{Lu91}, the reader can find the decoupling theorem also for some special non-invertible maps.

		\item The corresponding result for flows is the same, which is omitted here. Except the linearization of $ h^{cs}, h^{cu} $, the decoupling theorem was also obtained by many authors. The most famous result is the Hartman--Grossman Theorem (see e.g. \cite{KH95}) with its generalizations to Banach spaces for hyperbolic equilibriums. For non-hyperbolic equilibriums, it was perhaps first proved by F. Taken \cite{Tak71}; K. Palmer also proved this case for nonautonomous continuous systems (the corresponding nonautonomous discrete systems are the case $ M = \mathbb{Z} $, $ u(n) = n+1 $); see \cite{KP90} and the references therein. See also \cite{Irw70, Tak71} (hyperbolic periodic orbits), \cite{dMel73} (hyperbolic compact sets), and \cite{PS70} (normally hyperbolic compact invariant manifolds).
	\end{asparaenum}
\end{rmk}

%% file: sect4.tex
\chapter{Uniform Properties of Bundles, Bundle Maps, Manifolds and Foliations}\label{prelimiaries}

To set up our regularity results in \autoref{stateRegularity}, some definitions and notations are needed which we gather below. We hope that by doing so, we can make it clear how we deal with the spaces that lack compactness, high smoothness and boundedness; the main ideas of these treatments come from e.g. \cite{HPS77, PSW12, BLZ99, BLZ08, Cha04, Eld13, Ama15}. Meanwhile, we try to clarify certain concepts, facts and also the connection between the hypotheses in \autoref{settingOverview} and the classical references (e.g. \cite{HPS77}).

A purpose of this chapter is to explain the notion of uniform $ C^{k,\alpha} $ ($ k = 0,1, 0 \leq \alpha \leq 1 $) continuity of a bundle map that respects fibers and base points in appropriate bundles with uniform properties. This is done in an extremely natural way, that is, we represent the bundle map in local bundle charts belonging to preferred bundle atlases. In order to do this, (uniformly) locally metrizable spaces are introduced in \autoref{locallyM} to make sense of uniform $ C^{0,\alpha} $ continuity. A quick review of the fiber (or leaf) topology and connections is given in \autoref{immersed} and \autoref{connections}, respectively. Different types of continuity of bundle maps in different classes of bundles are introduced in \autoref{bundle} (and \autoref{bundleII}). The relevant notions of uniform manifold and foliation appear in \autoref{mf}. Some examples related to our (uniform) assumptions on bundles and manifolds are given in \autoref{examples}.

\section{Locally metrizable spaces}\label{locallyM}

Let us consider a type of topological spaces which have a \emph{uniform topological structure} but may not be globally metrizable. That the space is \emph{not} assumed to admit a metric is sometimes important for us when we deal with, e.g., immersed manifolds (see \autoref{immeresedM}), foliations with leaf topology or bundles with fiber topology (these manifolds may not be metrizable or separable, see \autoref{fake}), or non-manifolds (for instance laminations with leaf topology, see \autoref{holonomyL}).

\begin{defi}\label{defi:locallyM}
	Let $ M $ be a Hausdorff topological space. $ M $ is called a \textbf{locally metrizable space} (associated with an open cover $ \{ U_m: m \in M \} $) if the following conditions hold: (a) $ U_m $ is open and $ m \in U_m $, $ m \in M $; (b) every $ U_m $ is a metric space with metric $ d_m $ (possibly incomplete). Here in $ U_m $ the metric topology is the same as the subspace topology induced from $ M $. Write $ U_m(\epsilon) \triangleq \{ m'\in U_m: d_m(m',m) < \epsilon \} $. For convenience, set $ U_{m}(\infty) = U_{m} $. Also, we denote the \emph{$ \varepsilon $-neighborhood} of $ M_1 $ by $ M^{\varepsilon}_1 = \bigcup_{m \in M_1} U_{\varepsilon}(m) $, where $ M_1 \subset M $.
\end{defi}

\begin{exa}\label{lmsE}
	\begin{enumerate}[(a)]
		\item \label{lmsE1} Any metric space is locally metrizable.

		\item \label{lmsE2} Any $ C^k $ ($ k = 0, 1, 2, \ldots, \infty, \omega $) manifold $ M $ is locally metrizable. For example, take an atlas $ \{ (U_{\alpha}, \phi^{\alpha}): \alpha \in \Lambda \} $ of $ M $ where $ \phi^{\alpha}: U_{\alpha} \to X^{\alpha} $ is a $ C^k $ homeomorphism and $ X^{\alpha} $ is a Banach space with norm $ |\cdot|_{\alpha} $. For any $ m \in M $, choose an $ \alpha(m) \in \Lambda $ such that $ m \in U_{\alpha(m)} \triangleq U_{m} $. The metric $ d_m $ in $ U_{m} $ is defined by $ d_m(x, y) = |\phi^{\alpha(m)}(x) - \phi^{\alpha(m)}(y)|_{\alpha(m)} $. In particular, any $ C^{k} $ \emph{foliation} (see e.g. \cite{HPS77, AMR88}) with leaf topology is a locally metrizable space which is usually neither separable nor metrizable.

		\item \label{lmsE3} Any $ C^0 $ topology bundle (see \autoref{bundleP} \eqref{topologyBundle}) over a locally metrizable space is a locally metrizable space. The local product structure will make the bundle a locally metrizable space. The details are the following. Take an open cover $ \{U_{m}: m \in M \} $ of the base space making it a locally metrizable space. For every $ m \in M $, one can choose small $ \epsilon_{m} > 0 $ such that $ \{ (U_{m}(\epsilon_{m}), \phi^{m}): m \in M \} $ is a $ C^0 $ bundle atlas of the bundle. Let $ V_{m} = \phi^m(U_{m} \times X_{m}) $ with the metric defined by $ d((m_1,x), (m_2, y)) = d_{m}(m_1, m_2) + d_{m}((\phi^{m}_{m_1})^{-1}x, (\phi^{m}_{m_2})^{-1}y) $, where $ d_{m} $ is the metric in $ X_m $ or $ U_{m} $. We will discuss this type of bundles in \autoref{bundle}.

		\item \label{lmsE4} Any $ C^{k} $ ($ k \geq 1 $) Finsler manifold $ M $ with Finsler metric in each component of $ M $ (see \autoref{finsler}) is a locally metrizable space. Denote its components by $ V_{\alpha}, \alpha \in \Lambda $. Note that $ V_{\alpha} $ is open in $ M $ (as $ M $ is locally connected). The metric in $ V_{\alpha} $ is the Finsler metric.

		\item \label{lmsE5} Analogously, any (set) bundle with metric fibers (see \autoref{metricfiber}) equipped with fiber topology (see \autoref{immersed}) or any lamination with metric leaves (see \autoref{holonomyL}) endowed with leaf topology is a locally metrizable space.
	\end{enumerate}
\end{exa}

See also \autoref{immersionI} for immersed manifolds in Banach spaces. See \autoref{immersed} for a discussion of the leaf (or fiber) topology. Below, we will talk about some properties of maps between locally metrizable spaces.

In \autoref{ucontinuous} and \autoref{hcontinuous} below, let $ M, N $ be locally metrizable spaces associated with open covers $ \{ U_{m}: m \in M \} $ and $ \{ V_{n}: n \in N \} $, respectively.

\begin{defi}[Uniform continuity]\label{ucontinuous}
	Let $ g: M \to N $ and $ M_1 \subset M $. Define the \textbf{amplitude} of $ g $ around $ M_1 $ (with respect to the open covers $ \{ U_{m} \} $ and $ \{ V_{n} \} $) as
	\[
	\mathfrak{A}_{M_1}(\sigma) = \sup \{ d_{g(m_0)}(g(m), g(m_0)): m \in U_{m_0}(\sigma), m_0 \in M_1 \},
	\]
	where $ d_{g({m_0})} $ is the metric in $ V_{g(m_0)} $. Here for convenience, if $ g(m) \notin V_{g(m_0)} $, let 
	\[
	d_{g(m_0)}(g(m), g(m_0)) = \infty.
	\]
	We say $ g $ is
	\begin{enumerate}[(a)]
		\item \textbf{uniformly continuous around $ M_1 $} if $ \mathfrak{A}_{M_1}(\sigma) \to 0 $ as $ \sigma \to 0 $;
		\item \textbf{$ \varepsilon $-almost uniformly continuous around $ M_1 $} if $ \mathfrak{A}_{M_1}(\sigma) \leq \varepsilon $ as $ \sigma \to 0 $;
		\item continuous (or $ C^0 $) if for every $ m_0 \in M $, $ \mathfrak{A}_{m_0}(\sigma) \to 0 $ as $ \sigma \to 0 $;
		\item \textbf{$ \varepsilon $-almost continuous} if for every $ m_0 \in M $, $ \mathfrak{A}_{m_0}(\sigma) \leq \varepsilon $ as $ \sigma \to 0 $.
	\end{enumerate}
	If $ M_1 = M $, the words `around $ M_1 $' are usually omitted.
\end{defi}

Obviously, $ g $ is $ C^0 $ in the sense of the above definition if and only if $ g: M \to N $ is $ C^0 $ when $ M, N $ are considered as topological spaces. Note that when $ g $ is $ \varepsilon $-almost (uniformly) continuous, it may not be continuous. In \autoref{stateRegularity} (and \autoref{application}), we will frequently need that some functions are $ \varepsilon $-almost (uniformly) continuous around $ M_1 $. Consider a situation which can illustrate why almost (uniform) continuity is needed (see also \cite{BLZ08}): $ g: X \to Y $ is $ C^0 $ where $ X, Y $ are infinite-dimensional Banach spaces, and $ M_1 \subset X $ is compact. Usually, $ g $ may not be uniformly continuous in any $ U_{\epsilon}(M_1) = \{ m': d(m', M_1) < \epsilon \} $ but the amplitude of $ g $ in $ U_{\epsilon}(M_1) $ can be sufficiently small if $ \epsilon $ is small.

\begin{defi}[H\"older continuity]\label{hcontinuous}
	Let $ g: M \to N $ and $ M_1 \subset M $. Define the \textbf{uniform (local) $ \theta $-H\"older constant} of $ g $ around $ M_1 $ (with respect to the open covers $ \{ U_{m} \} $ and $ \{ V_{n} \} $) as
	\[
	\hol_{\theta, M_1}(\sigma) \triangleq \hol_{\theta, M_1, \sigma}(g) = \sup \left\{ \frac{d_{g(m_0)}(g(m), g(m_0))}{d^{\theta}_{m_0}(m, m_0)} : m \in U_{m_0}(\sigma), m_0 \in M_1 \right\},
	\]
	where $ d_{m_0} $ and $ d_{g({m_0})} $ are the metrics in $ U_{m_0} $ and $ V_{g(m_0)} $, respectively. We say $ g $ is \textbf{uniformly (locally) $ \theta $-H\"older around $ M_1 $} if $ \hol_{\theta, M_1}(\sigma) < \infty $ when $ \sigma \to 0 $. If for every $ m_0 \in M_1 $, $ \hol_{\theta, m_0}(\sigma) < \infty $ as $ \sigma \to 0 $, then we say $ g $ is \textbf{(locally) $ \theta $-H\"older around $ M_1 $}.
	If $ M_1 = M $, the words `around $ M_1 $' are usually omitted. Often, $ 1 $-H\"older = Lipschitz.
\end{defi}

Obviously, if $ g $ is uniformly (locally) H\"older around $ M_1 $, then it is uniformly continuous around $ M_1 $. Even if $ M $, $ N $ are metric spaces, the above definition of (local) $ \theta $-H\"olderness in general does not imply $ g $ is actually H\"older. (For example, a $ C^1 $ map in some set has $ 0 $ derivative but it might not be a constant.) However, if the metrics are length metrics, then it turns out to be true; see \autoref{length} for details, where we use `upper scalar uniform $ \alpha $-H\"older constant' instead of `uniform (local) $ \theta $-H\"older constant'.

The reader might want to know, when we know a property of $ g $ in $ U_{m_0}(\epsilon_1) \to V_{g(m_0)} $, what happens to $ g $ in $ U_{m_0}(\epsilon_1) \to V_{\hat{m}_0}(\epsilon_2) \subset V_{g(m_0)} $. Let us consider some compatible property of metrics in locally metrizable spaces, which is essentially the same as the definition of Finsler structure in the sense of Neeb--Upmeier (see \autoref{finsler}).
\begin{defi}\label{def:ulms}
	We say the metrics in a locally metrizable space $ M $ with an open cover $ \{ U_{m} \} $ are \emph{uniformly compatible} on $ M_1 \subset M $ if there are a constant $ \sigma > 0 $ (possibly $ \sigma = \infty $) and a constant $ \Xi \geq 1 $ such that if $ m, m' \in U_{m_1}(\sigma) \cap U_{m_2}(\sigma) \neq \emptyset $, $ m_1, m_2 \in M_1 $, then
	\[
	\Xi^{-1} d_{m_2}(m, m') \leq d_{m_1}(m, m') \leq \Xi d_{m_2}(m, m'),
	\]
	where $ d_{m} $ is the metric in $ U_{m} $. If $ M_1 = M $, we also say $ M $ is a \emph{uniformly locally metrizable space}.
\end{defi}
From this, we find that if $ M, N $ are uniformly locally metrizable spaces and $ g: M \to N $ is uniformly (locally) $ \theta $-H\"older, then for small $ \varepsilon_0 > 0 $, there is a constant $ L > 0 $ such that if $ m_1, m_2 \in U_{m_0}(\varepsilon_0) $, $ \hat{m}_0 \in U_{g(m_0)}(\varepsilon_0) $, one has
\[
d_{\hat{m}_0}(g(m_1), g(m_2)) \leq Ld_{m_0}(m_1, m_2).
\]
For this reason, the uniform continuity and (uniform) H\"older continuity of $ g $ are defined in terms of $ g: U_{m}(\sigma) \to V_{g(m)} $. Uniformly locally metrizable spaces appear naturally in applications. For instance, in \autoref{lmsE} \eqref{lmsE1} \eqref{lmsE5}, the spaces are already uniformly locally metrizable.

\begin{exa}
	\begin{enumerate}[(a)]
		\item We continue \autoref{lmsE} \eqref{lmsE2}. We assume (i) each component of $ M $ is modeled on a Banach space $ X_{\gamma} $ with a norm $ |\cdot|_{\gamma} $, and (ii) there is an atlas $ \{ (U_{\alpha}, \phi^{\alpha}) \} $ in each component satisfying 
		\[
		\phi^{\alpha} (U_{\alpha}) = X_{\gamma}(1) = \{ x \in X_{\gamma}: |x|_{\gamma} < 1 \} ~\text{and}~\sup_{U_{\alpha} \cap U_{\beta} \neq \emptyset} \lip \phi^{\alpha} \circ (\phi^{\beta})^{-1}(\cdot) < \infty
		\]
		(see also \autoref{UR}). Then $ M $ is a uniformly locally metrizable space.

		\item We continue \autoref{lmsE} \eqref{lmsE3}. For short, let `u.l.m.s' stand for uniformly locally metrizable space. Any $ C^0 $ topology bundle over a u.l.m.s that has \emph{$ \varepsilon $-almost local $ C^{0,1} $-fiber trivializations} (see \autoref{def:lipbundle}) is a u.l.m.s. In particular, any $ C^0 $ vector bundle over a u.l.m.s having \emph{$ C^0 $ uniform Finsler structure} (see \autoref{finsler} \eqref{finB}) is a u.l.m.s, and so is the tangent bundle of a $ C^{k} $ Finsler manifold in the sense of \emph{Neeb--Upmeier weak uniform} or \emph{Palais} (see \autoref{finsler} \eqref{finD}) or of a Riemannian manifold.

		\item For a locally metrizable space $ M $ with an open cover $ \{ U_{m} \} $, if there are a constant $ C \geq 1 $ and a metric $ d $ in $ M $ such that $ C^{-1} d_m(m_1, m_2) \leq d(m_1, m_2) \leq C d_{m}(m_1, m_2) $, where $ m_1, m_2 \in U_{m} $ and $ d_m $ is the metric in $ U_m $, then $ M $ is uniformly locally metrizable. See also \autoref{regManifold}.
	\end{enumerate}
\end{exa}

It might happen that $ U_{m}(\epsilon_1) = U_{m}(\epsilon_2) $ if $ \epsilon_1 > \epsilon_2 $ but that is not what we want. So we need the following notion of locally uniform size neighborhoods which is important for applying our main results (see \autoref{application}). A related notion is the injectivity radius of the exponential map in Riemannian manifolds; see also \autoref{regManifold}, \autoref{bgemetry} and \autoref{immersionII}.
\begin{defi}\label{uniformSize}
	Let $ M $ be a locally metrizable space associated with an open cover $ \{ U_{m}: m \in M \} $. We say $M$ has \textbf{locally uniform size neighborhoods} on $M_1$, if there is an $\varepsilon_1 >0$ such that $\overline{U_{m}(\varepsilon_1)} \subset U_{m}$ for all $ m \in M_1$, where the closure is taken in the topology of $M$.
	Note that if there is such an $\varepsilon_1 >0$, then $\overline{U_{m}(\varepsilon)} \subset U_{m}$ is also satisfied for all $\varepsilon < \varepsilon_1$.
\end{defi}

\section{Fiber (or leaf) topology}\label{immersed}
 \label{def:immersed}

Take a bundle $ X $ over $ M $. Assume the fibers of $ X $ are (Hausdorff) topological spaces. Let $ \mathcal{B}_s \triangleq \bigcup_{m \in M}  \{X_m \cap V: V ~\text{is open in}~ X_{m}\}$. Then the unique topology such that $ \mathcal{B}_s $ is a topology subbase is called the \textbf{fiber topology} of $ X $.
\begin{exa}
	\begin{enumerate}[(a)]
		\item ($ C^k $ bundle and its fiber topology) Let $ (X, M, \phi) $ be a $ C^k $ bundle (see \autoref{fiberbundle}). Then it is a $ C^k $ manifold locally modeled on $ T_mM \times T_xX_m $. If we consider the fiber topology of $ X $, then $ X $ is also a $ C^k $ manifold but locally modeled on $ T_xX_m $; the fiber topology does not even make $ X $ a $ C^0 $ bundle.

		\item (Leaf topology for foliation) Let $ M $ be a (Hausdorff) topological space. Assume there are mutually disjoint subsets of $ M $, $ U_{\alpha} $, $ \alpha \in \Lambda $, such that $ M = \bigcup_{\alpha}U_{\alpha} $. Consider $ M $ as a bundle over $ \Lambda $. The fiber topology in $ M $ is called the \textbf{leaf topology} of $ M $. In particular, we can give a \emph{foliation} (see e.g. \autoref{regFoliation}) or a \emph{lamination} (see e.g. \autoref{holonomyL} or \cite{HPS77}) the leaf topology if the mutually disjoint subsets are taken as the leaves of the foliation or lamination.
	\end{enumerate}
\end{exa}

We will use the fiber (or leaf) topology frequently when applying our main results; see e.g. \autoref{foliations}.

\section{Connections: mostly review} \label{connections}

\subsection{$ C^k $ bundles}

\begin{defi}\label{fiberbundle}
	$ (X, M, \pi) $ is called a $ \bm{C^k} $ \textbf{fiber bundle} (or for short $ \bm{C^k} $ \textbf{bundle}, $ k=0,1,2,\ldots,\infty,\omega $) if the following hold (see e.g. \cite{AMR88}):
	\begin{enumerate}[(a)]
		\item $ (X, M, \pi) $ is a bundle with $ M $ a $ C^k $ Banach manifold and every fiber $ X_m $ a paracompact $ C^k $ Banach manifold (or for simplicity a Banach space).
		\item There is a family of bundle charts of $ X $, which is called a \emph{$ C^k $ bundle structure} and denoted by $ \mathcal{A}_1 = \{(U_\alpha, \varphi^\alpha): \alpha \in \Lambda \} $, where $ \Lambda $ is an index, such that
		\begin{enumerate}[(1)]
			\item (open cover) $ U_\alpha, \alpha \in \Lambda $, are open in $ M $, and cover $ M $,
			\item ($ C^k $ transition maps) if $(U_{\alpha}, \varphi^{\alpha}), (U_{\beta}, \varphi^{\beta}) \in \mathcal{A}_1$ are bundle charts at $m_0, m_1$, respectively, and $U_{\alpha} \cap U_{\beta} \triangleq U_{\alpha\beta} \neq \emptyset $, then $\varphi^{\alpha, \beta} \triangleq (\varphi^{\beta})^{-1} \circ \varphi^{\alpha}: U_{\alpha\beta} \times X_{m_0} \rightarrow U_{\alpha\beta} \times X_{m_1}$ is $ C^k $,
			\item (maximal) $ \mathcal{A}_1 $ is maximal in the sense that it contains all possible pairs $ (U_\alpha, \varphi^\alpha) $ satisfying (1)--(2).
		\end{enumerate}
	\end{enumerate}

	Any $ \mathcal{A}_0 $ satisfying (1)--(2) is called a $ C^k $ \emph{bundle atlas} of $ X $. A $ C^k $ bundle atlas of $ X $ gives a unique $ C^k $ bundle structure and a unique $ C^k $ structure on $ X $ to make $ X $ a $ C^k $ bundle and a $ C^k $ Banach manifold, locally modeled on $ T_{m}M \times T_xX_{m} $, such that the given $ C^k $ bundle atlas becomes a subset of these structures. Note that $ M $ is now a submanifold of $ X $.
\end{defi}

\subsection{Connections}

In order to give a tangent representation of a $ C^1 $ bundle map respecting base points, we need to know how $ T_{(m,x)}X $ represents as $ T_mM \times T_xX_m $. So an additional structure of $ X $ is needed, named a connection in $ X $. See \cite{Kli95} in the setting of vector bundles.

\emph{From now on, we assume $ X $ is a $ C^k $ bundle with a $ C^k $ ($ k \geq 1 $) bundle atlas $ \mathcal{A}_0 $}.
There are two natural $ C^{k-1} $ vector bundles associated with $ X $. Let
\begin{equation}\label{HVspace}
\Upsilon^V_X \triangleq \bigcup_{(m,x) \in X} (m,x) \times T_xX_m, ~\Upsilon^H_X \triangleq \bigcup_{(m,x) \in X} (m,x) \times T_mM.
\end{equation}
$ \Upsilon_X \triangleq \Upsilon^H_X \times \Upsilon^V_X $ is a $ C^{k-1} $ vector bundle with base space $ X $ and fibers $ T_mM \times T_xX_m $.
Write the natural fiber projections
$ \Pi^h_{(m,x)} : T_mM \times T_xX_m \to T_mM $, $ \Pi^v_{(m,x)} : T_mM \times T_xX_m \to T_xX_m $, and associated bundle projections $ \Pi^h : \Upsilon_X \to \Upsilon^H_X $, $ \Pi^v: \Upsilon_X \to \Upsilon^V_X $.
Another vector bundle is the tangent bundle of $ X $, i.e., $ TX $. Note that locally $ T_{(m,x)} X \cong T_mM \times T_xX_m $.

In the following, we will identify $ T_mM = T_mM \times \{ 0 \} $ and $ T_xX_m = \{ 0 \} \times T_xX_m $.

\begin{defi}\label{def:connection}
	$ \mathcal{C}: TX \to \Upsilon_X $ is called a $ C^{k-1} $ \textbf{connection} of $ X $ if the following conditions hold:
	\begin{enumerate}[(a)]
		\item (global representation) $ \mathcal{C} $ is a $ C^{k-1} $ vector bundle isomorphism over $ \id $, i.e., $ \mathcal{C} $ is a $ C^{k-1} $ bundle map over $ \id $ and for every $ (m,x) \in X $, $ \mathcal{C}_{(m,x)} :T_{(m,x)}X \to T_mM \times T_xX_m $ is a vector isomorphism.
		\item (local representation) $ \mathcal{C} $ satisfies the local representation
		\begin{multline*}
		\mathcal{C}_{(m',\varphi_{m'}(x'))}  D_{(m',x')}\varphi =\left(  \begin{matrix}
		\id & 0 \\
		\varGamma_{(m',x')} & D\varphi_{m'}(x')
		\end{matrix} \right) : \\
		T_{m'}M \times T_{x'}X_m \to  T_{m'}M \times T_{\varphi_{m'}(x')}X_{m'},
		\end{multline*}
		for every bundle chart $ (U,\varphi) \in \mathcal{A}_0 $ at $ m $, $ (m',x') \in U \times X_{m} $. $ \varGamma_{(m',x')}: T_{m'}M \to T_{\varphi_{m'}(x')}X_{m'} $ (or $ \widehat{\varGamma}_{(m',x')} = (D\varphi_{m'}(x'))^{-1} \varGamma_{(m',x')} : T_{m'}M \to T_{x'}X_m $) is often called the \textbf{Christoffel map} (or Christoffel symbol) in the bundle chart $ \varphi $.
	\end{enumerate}
	Denote the \emph{horizontal space} of $ TX $ at $ (m,x) $ by $ \mathcal{H}_{(m,x)} \triangleq \mathcal{C}^{-1}_{(m,x)} (T_mM \times \{0\}) $, and the \emph{vertical space} of $ TX $ at $ (m,x) $ by $ \mathcal{V}_{(m,x)} \triangleq \mathcal{C}^{-1}_{(m,x)} (\{0\} \times T_xX_m) $. Then $ T_{(m,x)}X = \mathcal{H}_{(m,x)} \oplus \mathcal{V}_{(m,x)} $.

	Furthermore, if $ X $ is a vector bundle and for every $ m \in M $, $ x \in X_m $, $ \mathcal{C}_m(\cdot, a) \triangleq \mathcal{C}_{(m,\cdot)}a: X_m \to X_m $ is linear, then we call $ \mathcal{C} $ a linear connection of $ X $. In this case, the Christoffel map $ \varGamma_{(m,x)} : T_mM \to X_m $ is linear in $ x $, and is often written as $ \varGamma_{m} (\cdot, \cdot) \triangleq \varGamma_{(m, \cdot)}(\cdot) : T_mM \times X_m \to X_m  $. Now $ \varGamma_{m} \in L(T_mM \times X_m; X_m) $.
\end{defi}

\begin{lem}\label{lem:vs}
	$ \mathcal{V} = \ker D \pi $, i.e., $ \mathcal{V}_{(m,x)} = \ker D_{(m,x)}\pi $. In particular, if $ \mathcal{C}^1, \mathcal{C}^2 $ are connections of $ X $, then $ \Pi^v \mathcal{C}^1 = \Pi^v \mathcal{C}^2 $.
\end{lem}
\begin{proof}
	Choose a bundle chart $ (U,\varphi) \in \mathcal{A}_0 $ at $ m $. Since $ \pi \circ \varphi : (m,x) \mapsto m $, we have $ D_{(m,x)}\pi D_{(m,x')}\varphi = (\id, 0) $, where $ x = \varphi_m(x') $. Further, $ v \in \ker D_{(m,x)}\pi $ if and only if 
	\[
	(D_{(m,x')}\varphi)^{-1} v = (0, x_1) \in T_mM \times T_{x'}X_m,
	\]
	and if and only if $ v = (D_{(m,x')}\varphi) (0, x_1) $ for some $ x_1 \in T_{x'}X_m $. As $ \mathcal{C}_{(m,x)} D_{(m,x')}\varphi \{0\} \times T_{x'}X_m = \{0\} \times T_xX_m $, we have $ v \in \ker D_{(m,x)}\pi $ if and only if $ \mathcal{C}_{(m,x)} v \in \{0\} \times T_{x}X_m $, i.e., $ v \in \mathcal{V}_{(m,x)} $.
\end{proof}
From the above lemma, we see that a connection is described by how the horizontal spaces are chosen continuously, and this is what the classical definition of connection says; see \cite{Kli95}.

\begin{exa}
	\begin{enumerate}[(a)]
		\item If $ M $ is a $ C^k $ manifold, then we can (and will) choose $ \mathcal{C} = \id $.
		\item There is a canonical connection $ \mathcal{C} = \left(  \begin{matrix}
		\id & 0 \\
		0 & \id
		\end{matrix} \right) $ for the trivial bundle.
		\item If $ X $, $ Y $ are $ C^k $ bundles over the same base space with $ C^{k-1} $ connections $ \mathcal{C}^X $, $ \mathcal{C}^Y $, then $ X \times Y $ has a natural connection $ \mathcal{C}^X \times \mathcal{C}^Y : T(X \times Y) = TX \times TY \to \Upsilon_{X \times Y} $. We always use this product connection for the product bundle.
		\item If $ M $ is a smooth Riemannian manifold, then there is a connection of $ TM $ called the \emph{Levi-Civita connection} (see \cite{Kli95}).
		\item If $ X $ is a $ C^1 $ bundle over a paracompact $ C^1 $ manifold $ M $, then $ X $ has a $ C^0 $ (and locally $ C^{0,1} $) connection. The proof is the same as the construction of a linear connection of a vector bundle (see e.g. \cite[Theorem 1.5.15]{Kli95}).
	\end{enumerate}
\end{exa}

\subsection{Covariant derivatives}\label{coderivative}

Let $ (X, M, \pi_1) $, $ (Y, N, \pi_2) $ be $ C^k $ bundles, and $ f: X \to Y $ a bundle map over a map $ u: M \to N  $. We write $ f \in C^k(X, Y) $ if $ f $ is $ C^k $ when $ X, Y $ are regarded as $ C^k $ manifolds, and in this case, one can easily see that $ u \in C^k(M,N) $.

\begin{defi}\label{defi:coderivative}
	Suppose that $ X,Y $ have $ C^{k-1} $ connections $ \mathcal{C}^X, \mathcal{C}^Y $, respectively, and that $ f: X \to Y $ is a $ C^1 $ bundle map over a map $ u: M \to N  $. Define its \textbf{covariant derivative} as
	\[
	\nabla_mf_m(x) \triangleq  \Pi^v_{f(m,x)}\mathcal{C}^Y_{f(m,x)} Df(m,x) (\mathcal{C}^X_{(m,x)})^{-1} : T_mM = T_mM \times \{ 0 \} \to T_{f_m(x)} Y_{u(m)}.
	\]
\end{defi}
Hereafter, the symbol $ \nabla $ always stands for the covariant derivative of bundle maps. See \cite{Kli95} for the covariant derivative of a section, or more specially, a vector field, and other equivalent definitions of covariant derivative.

Note that if the bundle map $ f \in C^k(X,Y) $ and the connections are $ C^{k-1} $, then
\[
\nabla f: \Upsilon^H_X \to \Upsilon^V_Y, ~(m,x, v) \mapsto (f(m,x), \nabla_m f_m(x)v),
\]
is a $ C^{k-1} $ vector bundle map over $ f $, i.e., $ \nabla f \in L_{f}(\Upsilon^H_X, \Upsilon^V_Y) $.
Here are some properties of the covariant derivative.
\begin{lem}\label{lem:chain}
	Let $ X_i $ be $ C^k $ $ (k \geq 1) $ bundles with respective connections $ \mathcal{C}^i $, $ i = 1,2,3 $. Suppose $ f: X_1 \to X_2 $, $ g: X_2 \to X_3 $ are $ C^1 $ bundle maps over $ u_1, u_2 $, respectively.
	Then
	\begin{enumerate}[(1)]
		\item $ \Pi^h_{f(m,x)} \mathcal{C}^2_{f(m,x)} Df(m,x)  (\mathcal{C}^1_{(m,x)})^{-1} | T_mM = Du_1(m) $,
		\item $ \Pi^v_{f(m,x)} \mathcal{C}^2_{f(m,x)} Df(m,x)  (\mathcal{C}^1_{(m,x)})^{-1} | T_xX_m = Df_m(x) $,
		\item (\textbf{chain rule}) 
		\[
		\nabla_m (g_{u_1(m)} \circ f_m)(x) = \nabla_{u_1(m)}g_{u_1(m)}(f_m(x)) Du_1(m) + Dg_{u_1(m)}(f_m(x)) \nabla_m f_m(x).
		\]
	\end{enumerate}
\end{lem}
\begin{proof}
	(1) and (2) are easily seen to be true from the local representation of connections, and (3) is a consequence of (1) and (2). Here are the details.
	Take bundle charts $ (U_i,\varphi^i) $ of $ X_i $ at $ m_i $, $ i = 1,2 $, and $ x \in X_{m_1} $. Set $ x' = f_m(x) $, $ \varphi^1_m(x_1) = x $, $ \varphi^2_{m'}(x_2) = x' $, $ m' = u_1(m) $.
	Then
	\begin{align*}
	& \ \mathcal{C}^2_{f(m,x)} Df(m,x)  (\mathcal{C}^1_{(m,x)})^{-1}   \\
	= &  \mathcal{C}^2_{(m',x')} D_{(m',x_2)}\varphi^2 D_{(m,x_1)}( (\varphi^2)^{-1} \circ f \circ \varphi^1 ) D_{(m,x)} (\varphi^1)^{-1} ( \mathcal{C}^1_{(m,x)})^{-1}  \\
	= & \left(  \begin{matrix}
	\id & 0 \\
	* & D\varphi^2_{m'}(x_2)
	\end{matrix} \right) \left(  \begin{matrix}
	Du_1(m) & 0 \\
	* & D((\varphi^2_{m'})^{-1} \circ f_m \circ \varphi^1_m )(x_1)
	\end{matrix} \right) \left(  \begin{matrix}
	\id & 0  \\
	* & D(\varphi^1_m)^{-1}(x)
	\end{matrix} \right).
	\end{align*}
	This gives (1) and (2). For (3), let $ u_1(m) = m_2, u_2(m_2) = m_3 $, $ f_m(x) = x_2', g_{m_2}(x_2') = x_3' $. Now,
	\begin{align*}
	& \nabla_m (g_{u_1(m)} \circ f_m)(x) \\
	= & \Pi^v_{(m_3, x_3')} \mathcal{C}^3_{(m_3, x_3')} D(g \circ f)(m,x) (\mathcal{C}^1_{(m,x)})^{-1} | T_mM \\
	= & \Pi^v_{(m_3, x_3')} \mathcal{C}^3_{(m_3, x_3')} Dg(m_2,x_2')Df(m,x) (\mathcal{C}^1_{(m,x)})^{-1} | T_mM \\
	= & \Pi^v_{(m_3, x_3')} \mathcal{C}^3_{(m_3, x_3')} Dg(m_2,x_2') (\mathcal{C}^2_{(m_2,x_2')})^{-1} \Pi^h_{(m_2,x_2')}\mathcal{C}^2_{(m_2,x_2')} Df(m,x) (\mathcal{C}^1_{(m,x)})^{-1} | T_mM \\
	& + \Pi^v_{(m_3, x_3')} \mathcal{C}^3_{(m_3, x_3')} Dg(m_2,x_2')  (\mathcal{C}^2_{(m_2,x_2')})^{-1} \Pi^v_{(m_2,x_2')} \mathcal{C}^2_{(m_2,x_2')} Df(m,x) (\mathcal{C}^1_{(m,x)})^{-1} | T_mM \\
	= & \nabla_{u_1(m)}g_{u_1(m)}(f_m(x)) Du_1(m) + Dg_{u_1(m)}(f_m(x)) \nabla_m f_m(x).
	\end{align*}
\end{proof}

\subsection{Normal bundle charts}

\begin{defi}\label{def:normal}
	Let $ X $ be a $ C^1 $ bundle with a $ C^0 $ connection $ \mathcal{C} $.
	A $ C^k $ ($ k \geq 1 $) bundle chart $ (V, \psi) $ at $ m $ is called a $ C^k $ \textbf{normal bundle chart} (with respect to $ \mathcal{C} $) if $ \nabla_m \psi_m(x) = 0 $ for every $ x \in X_m $. A bundle atlas $ \mathcal{A} $ is \emph{normal} (with respect to $ \mathcal{C} $) on $ M_1 $ if every bundle chart $ (V, \psi) \in \mathcal{A} $ at $ m \in M_1 $ is normal.
\end{defi}

In this paper, we use normal bundle charts only to simplify the calculations in \autoref{stateRegularity} (see in particular the proof of \autoref{smoothbase}). Our assumptions in \autoref{stateRegularity} are usually about $ C^1 $ normal bundle charts and the connections are assumed to be $ C^0 $. However, if one needs to construct normal bundle charts from given bundle charts (see \autoref{app:normal} for a discussion), higher regularity is needed.

\begin{lem}
	Suppose there is a section $ i $ of $ X $, which is a $ 0 $-section (see \autoref{0-section}) with respect to a family of $ C^k $ normal bundle charts $ \mathcal{A}'_0 $ on $ M_1 $. Then:
	\begin{enumerate}[(1)]
		\item $ i $ is $ C^k $ in a neighborhood of $ M_1 $ and $ \nabla_{m} i(m) = 0 $ for $ m \in M_1 $.
		\item For any $ C^1 $ bundle map $ g:  X \to Y $ over $ u $ such that $ i $ is invariant under $ g $ in $ u^{-1}(M_1) $, we have $ \nabla_{m}g_{m}(i(m)) = 0 $, $ m \in M_1 $.
	\end{enumerate}
\end{lem}
\begin{proof}
	(1) Since $ \varphi_{m'}(i(m)) = i(m') $, where $ (U, \varphi) \in \mathcal{A}'_0 $ is a normal bundle chart at $ m $, it follows that $ i $ is $ C^k $ in $ U $, and $ \nabla_{m}i(m) = \nabla_{m} \varphi_{m}(i(m)) = 0 $.
	
	(2) Note that $ g_m(i(m)) = i(u(m)) $. Thus, $ \nabla_mg_m(i(m)) = \nabla_{m'}i(m')Du(m) = 0 $, where $ m' = u(m) \in M_1 $.
\end{proof}

\section{Bundles and bundle maps with uniform properties: part I}\label{bundle}

\subsection{Local representations}\label{representation}

Take a bundle atlas $\mathcal{A} =\{ (U_{\alpha}, \varphi^{\alpha}): \alpha \in \Lambda \} $ of $ X $. Take any bundle charts $(U_{\alpha}, \varphi^{\alpha}), (U_{\beta}, \varphi^{\beta}) \in \mathcal{A}$ at $m_0, m_1$ respectively such that $U_{\alpha} \cap U_{\beta} \triangleq U_{\alpha\beta} \neq \emptyset$. Then we call the following map a \textbf{transition map} (at $ m_0, m_1 $ with respect to $ \mathcal{A} $):
\[
\varphi^{\alpha, \beta} \triangleq (\varphi^{\beta})^{-1} \circ \varphi^{\alpha}: U_{\alpha\beta} \times X_{m_0} \rightarrow U_{\alpha\beta} \times X_{m_1}.
\]
If every domain $ U_{\alpha} $ has property $P_0$ and every transition map has property $P$, we say the bundle $X$ has property $(P, P_0)$ with respect to $\mathcal{A}$.
Consider the following examples.
\begin{exa}\label{bundleP}
	\begin{enumerate}[(a)]
		\item If $ \Lambda $ only contains one element, $ X $ is usually called a \emph{trivial bundle}.
		\item \label{topologyBundle} If $ M $ is a topological space, let $ P_0 $ mean openness and $ P $ continuity; then $ X $ is called a \emph{$ C^0 $ topology bundle}. As usual, the bundle atlas can be maximal. For convenience, a bundle chart belonging to this maximal bundle atlas is called a $ C^0 $ bundle chart, and any bundle atlas in this maximal bundle atlas is called a $ C^0 $ bundle atlas.
		\item \label{vtBundle} Assume each fiber of $ X $ is a Banach space and the bundle charts in $ \mathcal{A} $ are linear (i.e. $ x \mapsto \varphi^\alpha_m(x) $ is an invertible linear operator). Let $ P_0 $ mean openness and $ P $ continuity in $ U \to L(X_{m_0}, X_{m_1}) $; then $ X $ is called a \emph{$ C^0 $ vector bundle}. In addition, if $ M $ is a $ C^k $ manifold and $ P $ means $ C^k $ ($ k=0,1,2,\ldots, \infty, \omega $) in $ U \to L(X_{m_0}, X_{m_1}) $, then $ X $ is called a \emph{$ C^k $ vector bundle}. See e.g. \cite{AMR88}. We refer the readers to see the definition of $ C^r $-uniform (Banach) bundle in \cite[Chapter 6]{HPS77}.
		A slightly more general type of bundle than vector bundle with some uniform property, like $ C^r $-uniform (Banach) bundle, are summarized in \autoref{rmk:vector}; see also \autoref{vectorB}.
	\end{enumerate}
\end{exa}

Let $(X, M, \pi_1), (Y, N, \pi_2)$ be bundles and $F: X \rightarrow Y$  a bundle map over $u$. If for every $(U, \varphi) \in \mathcal{A}$ at $ m_0 $ and $(V, \psi) \in \mathcal{B}$ at $ m'_0 $ we have $ W \triangleq U \cap u^{-1}(V) \neq \emptyset $, then we call the following map a \textbf{local representation} of $ F $ (at $ m_0, m'_0 $ with respect to $\mathcal{A}, \mathcal{B}$):
\[
\widehat{F}_{m_0, m'_0} (m, x) \triangleq  \psi^{-1}_{u(m)} \circ F_m \circ \varphi_m(x): W \times X_{m_0} \rightarrow  Y_{m'_0}. \tag{$ \sharp $}\label{sharp}
\]
If every local representation of $ F $ has property $P_1$, we say that $F$ has property $P_1$ with respect to $ \mathcal{A}, \mathcal{B} $.
For example, if $ X $, $ Y $ are $ C^0 $ topology bundles, and $ P_1 $ means $ C^0 $ continuity, then we say $ F $ is $ C^0 $. Note that in this case since all the transition maps are $ C^0 $, one just needs to know $ u $ is $ C^0 $ and every map $ \widehat{F}_{m_0, u(m_0)} $ is continuous at $ m_0 $; then $ F $ is $ C^0 $. Take another example. If $ X $, $ Y $ are $ C^r $-uniform (Banach) bundles with preferred $ C^r $-uniform atlases $ \mathcal{A}, \mathcal{B} $ respectively (see \cite[Chapter 6]{HPS77}), and $ P_1 $ means $ C^r $-uniform (i.e. all the local representations of $ F $ are uniformly $ C^r $ equicontinuous), then we say $ F $ is $ C^r $-uniform with respect to $ \mathcal{A}, \mathcal{B} $. Also in this case, it suffices to know the maps $ \widehat{F}_{m_0, u(m_0)} $, $ m_0 \in M $, are uniformly $ C^r $ equicontinuous (i.e. $ |D^{i}\widehat{F}_{m_0, u(m_0)}(m,x') - D^{i}\widehat{F}_{m_0, u(m_0)}(m_0,x)| \to 0  $, $ 0 \leq i \leq r $, uniformly for $ |x|, |x'| \leq 1 $, $ m_0 \in M $ as $ (m,x') \to (m_0,x) $) and $ u $ is uniformly $ C^r $ ---then $ F $ is $ C^r $-uniform.
This is an important observation which significantly simplifies the proof of regularity results in \autoref{stateRegularity}.
Based on this motivation, we usually consider the local representations of $ F $ at $ m_0, u(m_0) $ (\emph{not} $ m_0 $ and the arbitrary choice of $ m'_0 \in N $). Moreover, for every $ m $, we usually require that there is at most one bundle chart $ \varphi $ at $ m $ (but possibly with different domains) belonging to the bundle atlas, and also without loss of generality, $ \varphi_{m} = \id $.

\begin{defi}\label{def:regular}
	A \textbf{regular} bundle atlas $ \mathcal{A} $ of a bundle $ (X,M,\pi_1) $ on $ M_1 $ means that if $ (U, \phi), (V, \varphi) $ are bundle charts at $ m \in M_1 $ belonging to $ \mathcal{A} $, then $ \phi = \varphi $ in $ U \cap V $ and $ \phi_m = \id $. For convenience, when $ M $ is a topological space, we also say a bundle atlas $ \mathcal{A} $ is \textbf{open regular} on $ M_1 $ if it is regular, and the domains of the bundle charts at $ m_0 \in M_1 $ belonging to $ \mathcal{A} $ form a neighborhood basis at $ m_0 $.

	If $ (X,M,\pi_1), (Y,N,\pi_2) $ have preferred regular bundle atlases $ \mathcal{A} $ and $ \mathcal{B} $ on $ M_1 $, $ u(M_1) $, respectively, where $ u: M \to N $, $ M_1 \subset M $, then the \textbf{regular} local representations of a map $ F: X \to Y $ over $ u $ on $ M_1 $ with respect to $ \mathcal{A}, \mathcal{B} $ are $ \widehat{F}_{m_0, u(m_0)} $ (defined by \eqref{sharp}), $ m_0 \in M_1 $. We usually write $ \widehat{F}_{m_0} = \widehat{F}_{m_0, u(m_0)} $ and call $ \widehat{F}_{m_0}(\cdot, \cdot) $ a \emph{(regular) local representation} or a \emph{(regular) vertical part} of $ F $ (at $ m_0 $ with respect to $ \mathcal{A}, \mathcal{B} $). If $ X = M = N $, $ u = \id $, then $ F $ is a section of $ M \to Y $; in this case, $ \widehat{F}_{m_0}(\cdot, \cdot) $ is also called a \emph{principal part} of $ F $ (at $ m_0 $ with respect to $ \mathcal{B} $).
\end{defi}
The same terminology is used for manifolds if they are considered as bundles with zero fibers. Unless otherwise specified, \emph{all bundle atlases and local representations are taken to be regular in the present paper}. In \autoref{stateRegularity}, the conditions and conclusions are stated about regular local representations with respect to preferred regular bundle atlases. Once the transition maps with respect to regular bundle atlases have more regularity (which means that the bundles and manifolds are more regular), the conclusions, in a way, become classical as in e.g. \cite{HPS77}. For instance, see \autoref{lipbase}.

\begin{defi}\label{def:uSize}
	A bundle $ X $ over a locally metrizable space $ M $ associated with an open cover $ \{ U_m \} $ is said to have \emph{uniform size trivializations} on $ M_1 \subset M $ with respect to $ \mathcal{A} $ if there is a $ \delta > 0 $ such that for every $ m_0 \in M_1 $, there is a bundle chart $ (U_{m_0}(\delta), \varphi^{m_0}) $ of $ X $ at $ m_0 $ (and $ \varphi^{m_0}_{m_0} = \id $). If a (regular) bundle atlas $ \mathcal{A} $ contains such bundle charts, we also say $ \mathcal{A} $ on $ M_1 $ (or the bundle charts in $ \mathcal{A} $ on $ M_1 $) has \emph{uniform size domains}.
\end{defi}
In the following definitions of uniform property for bundles and bundle maps, uniform size trivializations are assumed. However, only when $ M $ has locally uniform size neighborhoods on $ M_1 $ (see \autoref{uniformSize}), the uniform property would make more sense. So a more meaningful definition of uniform size trivializations is that further assume $ M $ has locally uniform size neighborhoods on $ M_1 $. However, we do not assume this first.
Such (regular) bundle atlas in \autoref{def:uSize} in some sense plays a similar role to the `plaquation' used in \cite[Chapter 6]{HPS77}.

\subsection{Fiber-regularity of bundle maps}\label{fiberR}

We use the following notation:
\begin{enumerate}[$ \bullet $]
	\item $ \diam A \triangleq \sup\{d(m,m'):m,m' \in A \} $, the diameter of a subset $ A $ of a metric space.
\end{enumerate}

Consider the following description of fiber-regularity of bundle maps.
Let $(X, M, \pi_1)$ and $(Y, N, \pi_2)$ be bundles with metric fibers, $u: M \rightarrow N$ a map and $ f: X \rightarrow Y $ a bundle map over $ u $.
Since the fiber maps $ f_m: X_m \rightarrow Y_{u(m)} $ depend on $ m \in M $, there are at least two cases which have to be distinguished: pointwise dependence or uniform dependence on $ m \in M $. We list some fiber-regularities below:
\begin{enumerate}
	\item[($ C^0 $)] $ f $ is said to be \emph{$ C^{0} $-fiber} (resp. uniformly continuous-fiber or equicontinuous-fiber) if $ f_m(\cdot) $, $ m \in M $, are $ C^0 $ (resp. uniformly continuous or equicontinuous). For the equicontinuous-fiber case, we usually say $ f_m(\cdot) $, $ m \in M $, are uniformly continuous uniformly for $ m \in M $.

	\item[($ C^{0,\alpha} $)] $ f $ is said to be \emph{$ C^{0, \alpha} $-fiber} if $ f_m \in C^{0,\alpha}(X_m, Y_{u(m)}) $ for all $ m \in M $, and \emph{uniformly $ C^{0, \alpha} $-fiber} if $ \sup_{m} \hol_{\alpha} f_m(\cdot) < \infty $. For the uniformly $ C^{0, \alpha} $-fiber case, we say $ f_m(\cdot) $ are ($ \alpha $-)H\"older uniformly for $ m \in M $.

	\item[($ C^{k,\alpha} $, $ k \geq 1 $)] For simplicity, in this case, the fibers of $X, Y$ are assumed to be Banach spaces (or open subsets of Banach spaces); see also \autoref{rmk:general}. As in the $ C^{0,\alpha} $ case, $ f $ is said to be \emph{$ C^{k, \alpha} $-fiber} if $ f_m \in C^{k,\alpha}(X_m, Y_{u(m)}) $ for all $ m \in M $. If the $C^{k, \alpha}$ norms of $ f_m $, $ m \in M $, are uniformly bounded, we say $f$ is \emph{uniformly $C^{k, \alpha}$-fiber}, or $ f_m(\cdot) $ are \emph{$ C^{k,\alpha} $ uniformly for $ m \in M $}. If for any \emph{fiber-bounded set} $ A $, i.e., $ \sup_{m \in M}\diam A_{m} < \infty $, the $C^{k, \alpha}$ norms of $ f_m(\cdot)|_{A_m} $, $ m \in M $, are uniformly bounded, we say $f$ is \emph{uniformly locally $C^{k, \alpha}$-fiber}, or $ f_m(\cdot) $ are \emph{locally $ C^{k,\alpha} $ uniformly for $ m \in M $}.
\end{enumerate}

If $ f $ is $ C^1 $-fiber, then we use $ Df_m(x) $ (or $ D_xf_m(x) $, $ Df_m $) to represent the derivative of the fiber map $ f_m: X_m \to Y_{u(m)} $ (at $ x \in X_m $).
Denote the \emph{fiber derivative} of $ f $ by $ D^v f $, which is a vector bundle map $ \Upsilon_X^V \to \Upsilon_Y^V $ (see \eqref{HVspace}) over $ f $, and is defined by
\[
D^v f: (m, x, v) \mapsto (f(m,x), Df_m(x)v),
\]
where $ Df_m(x): T_xX_m \to T_{f_m(x)}Y_{u(m)} $.

\subsection{Uniform $ C^{0,1} $-fiber bundles} \label{uniform lip bundle}

If a bundle $ X $ is a $ C^0 $ topology bundle (see \autoref{bundleP} \eqref{topologyBundle}), one would like to know how the metrics in the fibers change continuously, i.e., the metrics in fibers are compatible with the topology in the bundle. Motivated by the definition of Finsler structure (see e.g. \autoref{finsler}), we give the following definitions.

\begin{defi}[$ C^0 $-fiber bundle]
	The metrics in the fibers of a $ C^0 $ topology bundle $ X $ (see \autoref{bundleP} \eqref{topologyBundle}) are said to be $ C^0 $ if for every $ C^0 $ bundle chart $ (U, \varphi) $ at $ m $, $ U \times X_{m} \times X_{m} \to \mathbb{R}_+, (m',x,y) \mapsto d_{m'}(\varphi_{m'}(x), \varphi_{m'}(y)) $, is $ C^0 $, where $ d_m $ is the metric in $ X_{m} $. In this case, $ X $ is called a \textbf{$ C^0 $-fiber bundle}.
\end{defi}

\begin{defi}[$ C^1 $-fiber bundle]\label{def:C1fiber}
	Suppose $ X $ is a $ C^0 $ topology bundle (see \autoref{bundleP} \eqref{topologyBundle}) with the fibers being (paracompact) Banach manifolds. If there is a $ C^0 $ bundle atlas $ \mathcal{A} $ of $ X $ such that
	\begin{enumerate}
		\item[$ \bullet $] for every bundle chart $ \varphi^{m_0} $ the map $ x \mapsto \varphi^{m_0}_m(x) $ is $ C^1 $, and for every transition map $ \varphi^{m_0, m_1} = (\varphi^{m_1})^{-1} \circ \varphi^{m_0} $ with respect to $ \mathcal{A} $, $ (m,x) \mapsto D\varphi^{m_0, m_1}_{m}(x) $ is continuous,
	\end{enumerate}
	then we say $ X $ has \emph{local $ C^1 $-fiber topology trivializations} with respect to $ \mathcal{A} $ and $ \mathcal{A} $ is a \emph{$ C^1 $-fiber topology} bundle atlas; this type of bundle atlas can be maximal; for this reason we also call $ X $ a \emph{$ C^1 $ topology bundle}. In addition, if $ X $ is also a $ C^0 $-fiber bundle, $ X $ is called a \textbf{$ C^1 $-fiber bundle}.
\end{defi}

\begin{defi}[$ C^{0,1} $-fiber bundle]\label{def:lipbundle}
	Let $ M $ be a locally metrizable space associated with an open cover $ \{ U_{m}: m \in M \} $ (see \autoref{defi:locallyM}). Let $ X $ be a bundle over $ M $ with metric fibers. Let $ M_1 \subset M $. Assume there is a preferred (\emph{open regular}) bundle atlas $ \mathcal{A} $ (see \autoref{def:regular}) of $ X $ on $ M_1 $ with the following properties:
	If $ (V_{m_0}, \varphi^{m_0}) \in \mathcal{A} $ at $ m_0 \in M_1 $, then
	\begin{enumerate}[(i)]
		\item $ V_{m_0} \subset U_{m_0} $ and $ V_{m_0} $ is open; so without loss of generality, one can assume $ V_{m_0} = U_{m_0}(\epsilon_{m_0}) $ for some $ \epsilon_{m_0} > 0 $;
		\item $ \lip(\varphi^{m_0}_m)^{\pm 1}(\cdot) \leq \eta_{m_0}(d_{m_0}(m,m_0)) + 1 $, $ m \in V_{m_0} $, where $ \eta_{m_0}: \mathbb{R}_{+} \rightarrow \mathbb{R}_{+} $, and $ d_{m_0} $ is the metric in $ U_{m_0} $.
	\end{enumerate}
	\begin{enumerate}[(a)]
		\item Let $ M_1 = M $. If every $ \eta_{m_0} $ satisfies $ \limsup_{\delta \to 0}\eta_{m_0}(\delta) \leq \varepsilon $, then we say $ X $ has \textbf{$ \varepsilon $-almost local $ C^{0,1} $-fiber trivializations} with respect to $ \mathcal{A} $.
		If $ X $ is a $ C^0 $ topology bundle with $ C^0 $ atlas $ \mathcal{A} $ and every $ \eta_{m_0} $ is continuous with $ \eta_{m_0}(0) = 0 $, then $ X $ is called a \textbf{$ C^{0,1} $-fiber bundle} (with respect to $ \mathcal{A} $).
		\item Assume $ X $ has uniform size trivializations on $ M_1 $ with respect to $ \mathcal{A} $ (see \autoref{def:uSize}), and $ \eta_{m_0} = \eta $, where $ \eta: \mathbb{R}_{+} \rightarrow \mathbb{R}_{+} $.
		If $ \limsup_{\delta \to 0}\eta(\delta) \leq \varepsilon $, we say $ X $ has \textbf{$ \varepsilon $-almost uniform $ C^{0,1} $-fiber trivializations} on $ M_1 $ with respect to $ \mathcal{A} $.
		If $ \eta $ is continuous and $ \eta(0) = 0 $, we say $ X $ has \emph{uniform $ C^{0,1} $-fiber trivializations} on $ M_1 $ with respect to $ \mathcal{A} $; in addition if $ M_1 = M $ and $ X $ is a $ C^0 $ topology bundle with $ C^0 $ atlas $ \mathcal{A} $, then we call $ X $ a \emph{uniform $ C^{0,1} $-fiber bundle} (with respect to $ \mathcal{A} $).
	\end{enumerate}
\end{defi}

\begin{defi}[$ C^{1,1} $-fiber bundle]\label{def:C11bundle}
	We continue \autoref{def:lipbundle}. In addition, assume the fibers of $ X $ are Banach spaces (or open subsets of Banach spaces) for simplicity (see also \autoref{rmk:general}) and the following property is satisfied:
	\begin{enumerate}
		\item[(iii)] $ x \mapsto \varphi^{m_0}_m(x) $ is $ C^1 $, and $ \lip D(\varphi^{m_0}_m)^{\pm 1}(\cdot) \leq \eta^\infty_{m_0}(d_{m_0}(m,m_0)) $, $ m \in V_{m_0} $, where $\eta^\infty_{m_0}: \mathbb{R}_{+} \rightarrow \mathbb{R}_{+}$.
	\end{enumerate}
	\begin{enumerate}[(a)]
		\item Let $ M_1 = M $. Assume every $ \eta^\infty_{m_0} $ satisfies $ \limsup_{\delta \to 0} \eta^\infty_{m_0}(\delta) < \infty $. If $ X $ has {$ \varepsilon $-almost local $ C^{0,1} $-fiber trivializations} with respect to $ \mathcal{A} $, then we say $ X $ has \textbf{$ \varepsilon $-almost local $ C^{1,1} $-fiber trivializations} with respect to $\mathcal{A}$.
		If $ X $ is a $ C^{0,1} $-fiber and $ C^{1} $-fiber bundle, then $ X $ is called a \textbf{$ C^{1,1} $-fiber bundle} (with respect to $\mathcal{A}$).
		\item Let $ \eta^\infty_{m_0} = \eta^\infty $ where $\eta^\infty: \mathbb{R}_{+} \rightarrow \mathbb{R}_{+}$ and $ \limsup_{\delta \to 0} \eta^\infty(\delta) < \infty $.
		If $ X $ has $ \varepsilon $-almost uniform $ C^{0,1} $-fiber trivialization on $ M_1 $ with respect to $ \mathcal{A} $, then we say $ X $ has \textbf{$ \varepsilon $-almost uniform $ C^{1,1} $-fiber trivializations} on $ M_1 $ with respect to $ \mathcal{A} $.
		If $ X $ has uniform $ C^{0,1} $-fiber trivializations on $ M_1 $ with respect to $ \mathcal{A} $, then we say $ X $ has \emph{uniform $ C^{1,1} $-fiber trivializations} on $ M_1 $ with respect to $ \mathcal{A} $; in addition if $ M_1 = M $ and $ X $ is also a $ C^1 $-fiber bundle, then $ X $ is called a \emph{uniform $ C^{1,1} $-fiber bundle} (with respect to $ \mathcal{A} $).
	\end{enumerate}
\end{defi}

\noindent \emph{Notation warning}:
Note the difference between a $ C^k $ fiber bundle (or $ C^k $ bundle, see \autoref{fiberbundle}) and a $ C^{k} $-fiber bundle. The latter has `-' to emphasize that the $ C^k $ regularity is about the fibers, not about the base space.

Now we have the following relational graph, where $ a \to b $ means if $ X $ is $ a $, then $ X $ is $ b $:
\[
\begin{CD}
\text{uniform $ C^{1,1} $-fiber bundle} @>>> \text{uniform $ C^{0,1} $-fiber bundle}\\
@VVV  @VVV \\
\text{$ C^{1,1} $-fiber bundle} @>>> \text{$ C^{0,1} $-fiber bundle}\\
@VVV  @VVV \\
\text{$ C^{1} $-fiber bundle} @>>> \text{$ C^{0} $-fiber bundle}\\
@VVV  @VVV \\
\text{$ C^{1} $ topology bundle} @>>> \text{$ C^{0} $ topology bundle}\\
\end{CD}
\]

\begin{exa}[Vector bundle case]
	Consider a $ C^0 $ vector bundle $ X $ (modeled on $ \mathbb{E} \times \mathbb{F} $) over a Finsler manifold $ M $ with Finsler metrics in each component of $ M $.
	Let $ \mathcal{A} $ be a $ C^0 $ (regular) vector bundle atlas of $ X $. Now $ X $ is a $ C^1 $ topology bundle.
	Fix a $ C^0 $ Finsler structure in $ X $ (see \autoref{finsler} \eqref{finA}). Now the Finsler structure gives each fiber $ X_m $ a norm, and it is $ C^0 $ meaning $ X $ is a $ C^0 $-fiber bundle (and so a $ C^1 $-fiber bundle).
	For vector bundles, \autoref{def:C11bundle} does not give any information as vector bundle charts are linear and $ \eta^{\infty}_{m_0} \equiv 0 $. Let us consider what \autoref{def:lipbundle} actually says in this case.
	That $ X $ has {$ \varepsilon $-almost local $ C^{0,1} $-fiber trivializations} with respect to $ \mathcal{A} $ means the Finsler structure in $ X $ is $ C^0 $ uniform with constant $ \Xi = 1 + \varepsilon $ (see \autoref{finsler} \eqref{finB}).
	That $ X $ is a $ C^{0,1} $-fiber bundle means the Finsler structure in $ X $ is uniformly $ C^0 $ (see \autoref{finsler} \eqref{finA}); in particular, if $ X = TM $ this means $ TM $ is a Finsler manifold in the sense of Palais (see \autoref{finsler} \eqref{finD}).

	Let $ M_1 \subset M $. We assume the bundle charts in $ \mathcal{A} $ on $ M_1 $ have uniform size domains (see \autoref{def:uSize}), that is, there is a $ \delta > 0 $ such that $ (U_{m_0}(\delta), \varphi^{m_0}) \in \mathcal{A} $, $ m_0 \in M_1 $, where $ U_{m_0}(\delta) = \{ m' \in U_{m_0}: d(m', m_0) < \delta \} $.
	That $ X $ has $ \varepsilon $-almost uniform $ C^{0,1} $-fiber trivializations on $ M_1 $ with respect to $ \mathcal{A} $ means that there is a $ \delta' > 0 $ ($ \delta' < \delta $) such that for any $ m_0 \in M_1 $,
	\[
	1 - \varepsilon < |(\varphi^{m_0}_m)^{\pm 1}| < 1 + \varepsilon,
	\]
	provided $ m \in U_{m_0}(\delta') $, where $ |\varphi^{m_0}_m| $ is the norm of $ \varphi^{m_0}_m: X_{m_0} \to X_{m} $. And $ X $ has uniform $ C^{0,1} $-fiber trivializations on $ M_1 $ with respect to $ \mathcal{A} $ if 
	\[
	\lim_{\epsilon \to 0}\sup_{m_0 \in M_1}\sup_{m \in U_{m_0}(\epsilon)}|(\varphi^{m_0}_m)^{\pm 1}| = 1.
	\]
	The notion of uniform $ C^{1,1} $-fiber bundle is related to a vector bundle having \emph{bounded geometry} (see \autoref{vectorB}). The uniform property of trivializations of $ X $ on $ M_1 $ is vital for our H\"older regularity results (see \autoref{stateRegularity}), in order to overcome the lack of compactness and high smoothness.

\end{exa}

\autoref{lip maps} below can be skipped on the first reading.

\subsection{Base-regularity of bundle maps, $ C^{0,1} $-uniform bundle: H\"older case}\label{lip maps}

We give a description of H\"older continuity respecting the base points for a bundle map $ f $, i.e., $ m \mapsto f_{m}(x) $. In general, it is meaningless to talk about the H\"older continuity of $ m \mapsto f_{m}(x) $ if we do not know the H\"older continuity of the fiber maps $ x \mapsto f_{m}(x) $. A natural way to describe the H\"older continuity of $ m \mapsto f_{m}(x) $ is to use the local representations of $ f $ (see e.g. \cite{HPS77, PSW12}). The local representations of $ f $ might not be H\"older with respect to $ x $ even if $ f $ is H\"older-fiber, so a natural setting for the bundles is that they have some local $ C^{0,1} $-fiber trivializations. However, we do not assume this at first. We give the details as follows.
\begin{defi}\label{def:boundedf}
	Let $ X $ be a bundle with metric fibers over $ M $ and $ M_1 \subset M $. For a subset $ A \subset X $, we write $ A_m = A \cap X_{m} $. Then $ A $ is said to be \emph{bounded-fiber} on $ M_1 $ if $ \sup_{m_0 \in M_1}\diam A_{m_0} < \infty $.

	A function $ c: X \to \mathbb{R}_+ $ ($ c_{m}: X_{m} \to \mathbb{R}_+ $) is said to be \emph{bounded on $ M_1 $ on bounded-fiber sets} if for {bounded-fiber set} $ A $ on $ M_1 $, $ \sup_{m_0 \in M_1}\sup_{x \in A_{m_0}}c_{m_0}(x) < \infty $. For example,
	\begin{enumerate}[(a)]
		\item a \emph{uniformly bounded} function, i.e. $ \sup_{m_0 \in M_1}\sup_{x \in X_{m_0}}c_{m_0}(x) < \infty $;
		\item $ c_{m_0}(x) \leq c_0(|x|) $, where $c_{0}: \mathbb{R}_+ \rightarrow \mathbb{R}_+$, $ i_X : M \to X $ is a fixed section of $ X $ and $ |x| = |x|_{m_0} = d_{m_0}(x,i_X(m_0)) $. A more concrete example of $ c_0 $ is
		\begin{equation*}
		c_{c,\gamma}(a) = M_0 (c + a^{\gamma}), ~a \in \mathbb{R}_+,
		\end{equation*}
		where $ c = 0 $ or $ 1 $, $ \gamma \geq 0 $, $ M_0 > 0 $.
	\end{enumerate}
\end{defi}

\begin{defi}\label{vsII}
	Let $ M, N $ be locally metrizable spaces associated with open covers $ \{ U_{m}: m \in M \} $ and $ \{ V_{n}: n \in N \} $, respectively. The metrics in $ U_m $ and $ V_{n} $ are $ d^1_m $ and $ d^2_n $, respectively.
	Let $ X, Y $ be $ C^0 $ topology bundles over $ M, N $ with $ C^0 $ (\emph{open regular}) bundle atlases $ \mathcal{A} $, $ \mathcal{B} $, respectively (see \autoref{def:regular}).
	Let $ u: M \to N $ be $ C^0 $.
	Let $ f: X \to Y $ be a bundle map over $ u $, and
	\begin{equation*}
	\widehat{f}_{m_0}: U_{m_0}(\varepsilon_{m_0}) \times X_{m_0} \to Y_{u(m_0)},
	\end{equation*}
	a (regular) local representation of $ f $ at $ m_0 $ with respect to $ \mathcal{A}, \mathcal{B} $ (see \autoref{def:regular}), where $ \varepsilon_{m_0} > 0 $ is small.
	Assume
	\[
	|\widehat{f}_{m_0} (m,x) - \widehat{f}_{m_0} (m_0,x)| \leq c_{m_0}(x) d^1_{m_0}(m,m_0)^{\theta}, ~ m \in U_{m_0}(\varepsilon_{m_0}), \tag{$ \blacktriangle $}
	\]
	where $c_{m_0}: X_{m_0} \rightarrow \mathbb{R}_+$, $ m_0 \in M_1 $, and $ M_1 \subset M $.

	\begin{asparaenum}[(a)]
		\item We say (the vertical part of) \emph{$ f $ depends in a (locally) $ C^{0,\theta} $ fashion on the base points} around $ M_1 $, or $ m \mapsto f_m(\cdot) $ is \emph{(locally) $ C^{0,\theta} $} around $ M_1 $ with respect to $ \mathcal{A} $, $ \mathcal{B} $.
		\item \label{vs:bb} Assume $ u $ is uniformly continuous around $ M_1 $ (see \autoref{hcontinuous}).
		Suppose $ X, Y $ have \emph{uniform size trivializations} on $ M_1 $ with respect to $ \mathcal{A}, \mathcal{B} $ respectively (see \autoref{def:uSize}), i.e., one can choose $ \varepsilon_{m_0} $ such that $ \inf_{m_0 \in M_1} \varepsilon_{m_0} > 0 $.
		If $ c_{m_0} $, $ m_0 \in M_1 $, are bounded on $ M_1 $ on bounded-fiber sets (see \autoref{def:boundedf}), then we say (the vertical part of) \emph{$ f $ depends in a uniformly (locally) $ C^{0,\theta} $ fashion on the base points around $ M_1 $ {`uniformly for bounded-fiber sets'}}, or more intuitively, $ m \mapsto f_m(\cdot) $ is \emph{uniformly (locally) $ \theta $-H\"older around $ M_1 $ {`uniformly for bounded-fiber sets'}}, with respect to $ \mathcal{A}, \mathcal{B} $. Usually, the words in `...' are omitted, especially when $ c_{m_0} $ is the class of functions in \autoref{def:boundedf} (a) (b).
		\item \label{vs:cc} (vector case) Under case \eqref{vs:bb}, let in addition $ X, Y, \mathcal{A}, \mathcal{B}, f $ be vector and $ |c_{m_0}(x)| \leq M_0 |x| $ (where $ M_0 $ is a constant independent of $ m_0 $). Note that $ m \mapsto f_{m} $ can be regarded naturally as a section $ M \to L_{u}(X, Y) $. We also say $ m \mapsto f_{m} $ is \emph{uniformly (locally) $ \theta $-H\"older around $ M_1 $} (with respect to $ \mathcal{A}, \mathcal{B} $).
		\item The words `around $ M_1 $' in (a) and (b) are often omitted if $ M_1 = M $. Also $ \theta $-H\"older = $ C^{0,\theta} $, $ 1 $-H\"older = $ C^{0,1} $ = Lipschitz. We will omit the letter `$ \theta $' if we do not emphasize the H\"older degree $ \theta $. Furthermore, without causing confusion, the words `with respect to $ \mathcal{A}, \mathcal{B} $' will be omitted if $ \mathcal{A}, \mathcal{B} $ are predetermined. The same terminology will be used for a section.
	\end{asparaenum}
\end{defi}

\begin{rmk}\label{rmk:holder}
	\begin{enumerate}[(a)]
		\item $ c_{m_0} $ belongs to the class of functions in \autoref{def:boundedf} (b), for example, when $ X, Y $ are vector bundles with vector bundle atlases $ \mathcal{A}, \mathcal{B} $ and $ f $ is a vector bundle map (usually $ c_{0} = c_{0,1} $ in \autoref{def:boundedf}); see also \autoref{stateRegularity}. If the fibers of $ X $ (or $ Y $) are uniformly bounded (i.e. $ \sup_{m \in M} \diam X_{m} < \infty $), for example $ X = M $ or $ X $ is a disk bundle, then $ c_{m_0} $ belongs to the class of functions in  \autoref{def:boundedf} (a); see also \autoref{section}. In the latter case, in \cite[Section 8]{PSW12}, $ f $ is also said to have \emph{$ \theta $-bounded vertical shear} (on $ M_1 $) with respect to $ \mathcal{A} $.

		\item If, in addition, the fiber maps of $f$ are Lipschitz and $ X, Y $ have {$ \varepsilon $-almost local $ C^{0,1} $-fiber trivializations} with respect to $ \mathcal{A}, \mathcal{B} $, then
		\[
		\lip \widehat{f}_{m_0}(m,\cdot) \leq \varepsilon d^1_{m_0}(m,m_0) + \varepsilon d^2_{u(m_0)}(u(m), u(m_0)) + \lip f_{m_0}(\cdot), ~m \in U_{m_0}(\varepsilon_{m_0}).
		\]
		This fact will be frequently used in \autoref{stateRegularity}.

		\item \label{h:cc} See \autoref{representation} for a reason why it is not necessary to consider a condition like
		\[
		|\widehat{f}_{m_0} (m,x) - \widehat{f}_{m_0} (m',x)| \leq c_{m_0}(x) d^1_{m_0}(m,m')^{\theta}, ~ m, m' \in U_{m_0}(\varepsilon_{m_0});
		\]
		see also \autoref{lipbase} in concrete settings and the following discussion.
	\end{enumerate}
\end{rmk}

Until now, we do not make a higher regularity assumption on the base space, except pointwise $ C^0 $ continuity.

\begin{defi}[$ C^{0,1} $-uniform bundle]\label{def:C01Uniform}
	Let $ M $ be a \emph{uniformly} locally metrizable space associated with an open cover $ \{U_m\} $ (see \autoref{def:ulms}) and $ M_1 \subset M $. Let $ d_m $ be the metric in $ U_{m} $. Let $ X $ be a bundle over $ M $ with an (open regular) bundle atlas $ \mathcal{A} $. Assume $ X $ has uniform size trivializations on $ M_1 $ with respect to $ \mathcal{A} $ (see \autoref{def:uSize}). Then $ \mathcal{A} $ is said to be \textbf{$ C^{0,1} $-uniform} around $ M_1 $, if the following assumption holds:
	Choose a small $ \delta > 0 $ (in \autoref{def:uSize}), and take any bundle charts $ (U_{m_0}(\delta), \varphi^{m_0}), (U_{m_1}(\delta), \varphi^{m_1}) \in \mathcal{A}$ at $ m_0, m_1 \in M_1 $, respectively. The transition map $ \varphi^{m_0,m_1} $ with respect to $ \mathcal{A} $, i.e.,
	\[
	\varphi^{m_0,m_1} = (\varphi^{m_1})^{-1} \circ \varphi^{m_0}: (W_{m_0, m_1}(\delta), d_{m_0}) \times X_{m_0} \to (W_{m_0, m_1}(\delta), d_{m_1}) \times X_{m_1},
	\]
	where $ W_{m_0, m_1}(\delta) = U_{m_0}(\delta) \cap U_{m_1}(\delta) \neq \emptyset $, satisfies
	\[
	\lip \varphi^{m_0,m_1}_{(\cdot)} (x) \leq c^1_{m_0, m_1}(x),
	\]
	and $ c^1_{m_0, m_1}(x) $, $ m_1 \in M_1 $, are bounded on $ M_1 $ on bounded-fiber sets (see \autoref{def:boundedf}).

	If $ \mathcal{A} $ is \emph{$ C^{0,1} $-uniform} around $ M_1 $ and $ X $ has $ \varepsilon $-almost uniform $ C^{0,1} $-fiber trivializations on $ M_1 $ with respect to $ \mathcal{A} $, then we say $ X $ has \textbf{$ \varepsilon $-almost $ C^{0,1} $-uniform trivializations} on $ M_1 $ with respect to $ \mathcal{A} $; in addition, if $ M_1 = M $ and $ \varepsilon = 0 $, we also call $ X $ a \emph{$ C^{0,1} $-uniform bundle} (with respect to $ \mathcal{A} $).
\end{defi}

By a simple computation as in \autoref{lipbase}, one can show that  in $ C^{0,1} $-uniform bundles, the definition of uniform $ \theta $-H\"older continuity of vertical part respecting the base points given in \autoref{vsII} will imply a stronger form of uniform $ \theta $-H\"older continuity as in e.g. \cite[Section 8]{PSW12}.
\begin{lem}\label{lem:uniformC01}
	Assume $ (X, M, \pi_1), (Y, N, \pi_2) $ have $ \varepsilon $-almost $ C^{0,1} $-uniform trivializations on $ M^{\epsilon_1}_1 $, $ u(M^{\epsilon_1}_1) $ with respect to preferred $ C^{0,1} $-uniform bundle atlases $ \mathcal{A}, \mathcal{B} $ respectively (see \autoref{def:C01Uniform}), where $ M_1 \subset M $ and $ M^{\epsilon_1}_1 $ is the $ \epsilon_1 $-neighborhood of $ M_1 $. Suppose $ f: X \to Y $ is uniformly $ C^{0,1} $-fiber (see \autoref{fiberR}) over a map $ u $ which is uniformly (locally) Lipschitz around $ M_1 $ (see \autoref{hcontinuous}). Assume $ f $ depends in a uniformly (locally) $ C^{0,\theta} $ fashion on the base points around $ M^{\epsilon_1}_1 $ uniformly for bounded-fiber sets (see \autoref{vsII}). Then there is a (small) $ \delta > 0 $ such that for any bundle charts $ (U_{m_0}(\delta), \varphi^{m_0}) \in \mathcal{A} $ at $ m_0 \in M_1 $, $ (V_{m'_0}(\delta), \phi^{m'_0}) \in \mathcal{B} $ at $ m'_0 \in u(M_1) $, $ W_{m_0, m'_0}(\delta) \triangleq U_{m_0}(\delta) \cap u^{-1}(V_{m'_0}(\delta)) \neq \emptyset $, the local representation of $ f $ (at $ m_0, m'_0 $),
	\begin{multline*}
	(\phi^{m'_0})^{-1} \circ f \circ \varphi^{m_0} (m, x) = (u(m), \widehat{f}_{m_0, m'_0}(m,x)): \\
	(W_{m_0, m'_0}(\delta), d^1_{m_0})  \times X_{m_0} \rightarrow (V_{m'_0}(\delta), d^2_{m'_0}) \times Y_{m'_0},
	\end{multline*}
	where $ d^1_{m_0}, d^2_{m'_0} $ are the metrics in $ U_{m_0}(\delta), V_{m'_0}(\delta) $ respectively, satisfies
	\[
	|\widehat{f}_{m_0, m'_0} (m,x) - \widehat{f}_{m_0, m'_0} (m',x)| \leq c^2_{m_0, m'_0}(x) d^1_{m_0}(m,m')^{\theta}, ~ m, m' \in W_{m_0, m'_0}(\delta),
	\]
	and $ c^2_{m_0, m'_0} $, $ m'_0 \in u(M_1) $, are bounded on $ M_1 $ on bounded-fiber sets (see \autoref{def:boundedf}). Moreover, if $ c_{m_0} $ (in \autoref{vsII}) and $ c^{1}_{m_0,m_1} $ (in \autoref{def:C01Uniform}) are the class of functions in \autoref{def:boundedf} (a) (b), so is $ c^2_{m_0, m'_0} $.
\end{lem}

For other types of uniform properties of a bundle map such as uniformly $ C^0 $ base-regularity of $ C^0 $ bundle map, base-regularity of $ C^1 $-fiber bundle map, base-regularity of $ C^1 $ bundle map, and some extensions, which can be discussed very analogously, see \autoref{bundleII}.

\section{Manifolds and foliations}\label{mf}

Having the preliminaries in \autoref{bundle} at hand, we can discuss some uniform properties of manifolds and foliations in this subsection, which can be skipped on the first reading until \autoref{foliations}.

\subsection{Uniform manifolds}\label{regManifold}

Let us introduce a class of uniform manifolds, which generalize Banach spaces and compact Riemannian manifolds, or a Banach-manifold-like version of Riemannian manifolds having \emph{bounded geometry} (see \autoref{defi:bounded}) and \emph{uniformly regular} Riemannian manifolds (see \autoref{UR}).

\begin{enumerate}[$ (\blacksquare) $]\label{C1MM}
	\item Let $ M $ be a $ C^1 $ Finsler manifold (see \autoref{finsler}) with Finsler metric $ d $ in each component of $ M $. Let $ M_1 \subset M $. Suppose $ M $ has \emph{locally uniform size neighborhoods} on $ M_1 $ (see \autoref{uniformSize}), i.e., there is an $ \varepsilon_{1} < \epsilon' $ such that $ \overline{U_{m_0}(\varepsilon_{1})} \subset U_{m_0}(\epsilon') $, $ m_0 \in M_1 $, where the closure is taken in the topology of $ M $ and $ U_{m_0}(\epsilon) = \{ m': d(m', m_0) < \epsilon \} $. Assume there is a constant $ \Xi = \Xi(\epsilon') \geq 1 $ such that for every $ m_0 \in M_1 $, there is a $ C^1 $ local chart $ \chi_{m_0} : U_{m_0}(\epsilon') \to T_{m_0}M $ with $ \chi_{m_0} (m_0) = 0 $ and $ D\chi_{m_0} (m_0) = \id $, satisfying
	\begin{equation}\label{equ:uniform}
	\sup_{m' \in U_{m_0}(\epsilon') } |D\chi_{m_0}(m')| \leq \Xi,~
	\sup_{m' \in U_{m_0}(\epsilon') } |(D\chi_{m_0}^{-1})(\chi_{m_0}(m'))| \leq \Xi.
	\end{equation}
	That is, $ TM $ has \emph{$ (\Xi-1) $-almost uniform $ C^{0,1} $-fiber trivializations} (see \autoref{def:lipbundle}) on $ M_1 $ with respect to a $ C^0 $ \emph{canonical} bundle atlas $ \mathcal{M} $ of $ TM $ given by
	\[
	\mathcal{M} = \{ (U_{m_0}(\epsilon), (\id \times D\chi_{m_0}^{-1} (\chi_{m_0}(\cdot))) ): m_0 \in M_1, 0 < \epsilon < \epsilon' \}.
	\]
\end{enumerate}
Usually, we say $ M $ with $ M_1 $ satisfies $ (\blacksquare) $ or $ M $ is \textbf{$ \Xi $-$ C^{0,1} $-uniform} around $ M_1 $ (with respect to $ \mathcal{M} $).
In addition if $ \Xi \to 1 $ as $ \epsilon' \to 0 $, we say $ M $ is \textbf{$ C^{0,1} $-uniform} around $ M_1 $ (with respect to $ \mathcal{M} $). This type of manifolds will be used as natural base spaces in \autoref{foliations}.

\begin{lem}\label{lem:usd}
	For a $ C^1 $ Finsler manifold $ M $, if there is a local chart $ \chi_{m_0}: U_{m_0}(\epsilon') \to T_{m_0}M $ such that $ \chi_{m_0} (m_0) = 0 $ and \eqref{equ:uniform} holds, then there is a $ \delta_{m_0} > 0 $ such that $ T_{m_0}M(\delta_{m_0}) \subset \chi_{m_0} (U_{m_0}(\epsilon')) $ and
	\[
	\Xi^{-1}|x_1 - x_2| \leq d(\chi^{-1}_{m_0}(x_1), \chi^{-1}_{m_0}(x_2)) \leq \Xi |x_1 - x_2|, ~x_1, x_2 \in T_{m_0}M(\delta_{m_0}).
	\]
	Suppose there is an $ \varepsilon_{1} < \epsilon' $ such that $ \overline{U_{m_0}(\varepsilon_{1})} \subset U_{m_0}(\epsilon') $, where the closure is taken in the topology of $ M $. Then we can take $ \delta_{m_0} = \varepsilon_1/\Xi $, and vice versa.
\end{lem}
\begin{proof}
	The first conclusion is easy; see e.g. \cite[Lemma 2.4]{JS11}. We prove the second conclusion.
	Since $ \chi_{m_0} (U_{m_0}(\epsilon')) $ is open in $ T_{m_0}M $, we know there is a $ \sigma > 0 $ such that $ T_{m_0}M(\sigma) \subset \chi_{m_0} (U_{m_0}(\epsilon')) $.
	If $ \sigma = \varepsilon_{1} / \Xi $, the proof is finished. So assume $ \sigma < \varepsilon_{1} / \Xi $.
	Let $ y \in T_{m_0}M $ with $ |y| < \sigma $. Consider
	\[
	D = \{ t \in [1, \varepsilon_{1} / (\Xi \sigma)]: \chi^{-1}_{m_0}(sy) \in U_{m_0}(\epsilon'), ~ s \in [0,t] \}, ~
	t_0 = \sup D.
	\]
	Note that as $ \sup_{m' \in U_{m_0}(\epsilon') } |(D\chi_{m_0}^{-1})(\chi_{m_0}(m'))| \leq \Xi $, we can deduce that if $ t \in D $, then $ \chi^{-1}_{m_0}(ty) \in U_{m_0}(\varepsilon_1) $. Indeed, taking $ \gamma(s) = \chi^{-1}_{m_0}(sty) $, $ 0 \leq s \leq 1 $, we see that
	\[
	d(\chi^{-1}_{m_0}(ty), m_0) \leq \int_{0}^{1} |\gamma'(s)|_{\gamma(s)} ~\mathrm{d} s = \int_{0}^{1} |(D\chi_{m_0}^{-1})(sty)ty| ~\mathrm{d} s < \Xi \cdot \varepsilon_{1} / (\Xi \sigma) \cdot \sigma = \varepsilon_{1}.
	\]
	Suppose $ t_0 < \varepsilon_{1} / (\Xi \sigma) $. Then $ \chi^{-1}_{m_0}(t_0y) \in \overline{U_{m_0}(\varepsilon_{1})} \subset U_{m_0}(\epsilon') $, i.e., $ t_0 \in D $. However, this yields a contradiction since $ \chi^{-1}_{m_0}(t_0y) $ is an interior point of $ U_{m_0}(\epsilon') $ and $ t_0 $ is the supremum of $ D $. Thus $ t_0 = \varepsilon_{1} / (\Xi \sigma) $. This shows $ T_{m_0}M(\varepsilon_{1} / \Xi) \subset \chi_{m_0} (U_{m_0}(\epsilon')) $. `Vice versa' is obvious.
\end{proof}

The above lemma shows that if $ M $ is a $ C^{1} $ Finsler manifold in the sense of Neeb--Upmeier weak uniform (see \autoref{finsler} \eqref{finD}), then the \emph{uniformly} locally metrizable space $ M $ can be characterized locally. Take a local chart $ \chi_{m}: V_{m} \to T_{m}M $ at $ m $ that satisfies \eqref{equ:uniform}, where $ V_{m} = \chi^{-1}_{m} (T_mM(\epsilon_{m})) $, with the metric $ d_m $ in $ V_m $ given by $ d_{m}(m_1, m_2) = |\chi_{m}(m_1) - \chi_{m}(m_2)|_{m} $, where $ |\cdot|_m $ is the norm of $ T_mM $ induced by the Finsler structure of $ TM $. Then $ M $ has \emph{locally uniform size neighborhoods} on $ M_1 $ (see \autoref{uniformSize}) if and only if $ \inf_{m \in M_1} \epsilon_{m} > 0 $.

\begin{defi}\label{def:manifoldMap}
	Let $ (\blacksquare) $ hold. Let $ u: M \to M $ and $ u(M) \subset M_1 $.
	We say $ u $ is \textbf{uniformly (locally) $ C^{0,1} $} (resp. \textbf{uniformly $ C^0 $}) around $ M_1 $ if $ M $ is considered as a locally metrizable space and $ u $ is uniformly (locally) $ C^{0,1} $ (resp. \textbf{uniformly $ C^0 $}) around $ M_1 $ in the sense of \autoref{hcontinuous}.
	We say $ u $ is \textbf{uniformly (locally) $ C^{1,1} $} around $ M_1 $ with respect to $ \mathcal{M} $ if $ u $ is uniformly (locally) $ C^{0,1} $ around $ M_1 $ and $ m \mapsto Du(m) $ is uniformly (locally) $ C^{0,1} $ around $ M_1 $ with respect to $ \mathcal{M}, \mathcal{M} $ in the sense of \autoref{vsII} \eqref{vs:cc}, that is, for
	\[
	\widehat{Du}_{m_0} (m)v \triangleq D\chi_{u(m_0)}(u(m)) Du(m)D\chi_{m_0}^{-1}(\chi_{m_0}(m))v , ~(m,v) \in U_{m_0}(\epsilon') \times T_{m_0}M,
	\]
	one has
	\[
	\|\widehat{Du}_{m_0} (m) - Du(m_0)\| \leq C d(m,m_0), m \in U_{m_0}(\epsilon_1),m_0 \in M_1,
	\]
	for some small $ \epsilon_1 > 0 $ and some constant $ C > 0 $ (independent of $ m_0 $).
	We say $ Du $ is \textbf{uniformly $ C^{0} $} around $ M_1 $ with respect to $ \mathcal{M} $ if $ m \mapsto Du(m) $ is uniformly $ C^{0} $ around $ M_1 $ with respect to $ \mathcal{M}, \mathcal{M} $ in the sense of \autoref{def:ucontinuity} \eqref{uc:cc}.
	Similarly, if $ M $ with $ M_1 $ and $ N $ with $ N_1 $ both satisfy $ (\blacksquare) $, and $ u: M \to N $ with $ u(M_1) \subset N_1 $, then we say $ u $ is uniformly (locally) $ C^{1,1} $ or $ C^{0,1} $ around $ M_1 $.
\end{defi}

In the above definition, the assumption that $ M $ has locally uniform size neighborhoods on $ M_1 $ (see \autoref{uniformSize}) can be removed.

\begin{defi}[$ C^{0,1} $-uniform and $ C^{1,1} $-uniform manifold]\label{def:C11um}
	(i) If $ M $ with $ M_1 = M $ satisfies $ (\blacksquare) $, then we say $ M $ is a \emph{$ \Xi $-$ C^{0,1} $-uniform manifold}; if in addition $ \Xi \to 1 $ as $ \epsilon' \to 0 $, we say $ M $ is a \emph{$ C^{0,1} $-uniform manifold}.

	(ii) We say a $ C^1 $ Finsler manifold $ M $ is \emph{$ \Xi $-$ C^{1,1} $-uniform} around $ M_1 $ if $ (\blacksquare) $ holds and $ \mathcal{M} $ is $ C^{0,1} $-uniform around $ M_1 $ in the sense of \autoref{def:C01Uniform}, with $ c^1_{m_0, m_1} $ there satisfying $ c^1_{m_0, m_1}(x) \leq C|x| $ for some $ C > 0 $ independent of $ m_0, m_1 $.
	If $ M_1 = M $ (resp. $ \Xi \to 1 $ as $ \epsilon' \to 0 $), then the words `around $ M_1 $' (resp. `$ \Xi $-') will be omitted.
\end{defi}

If $ M, N $ are $ C^{1,1} $-uniform manifolds, then a map $ u: M \to N $ that is uniformly (locally) $ C^{1,1} $ in the sense of \autoref{def:manifoldMap} is also $ C^{1,1} $ in the classical sense of \autoref{lem:uniformC01}; see also \cite{HPS77, Eld13}.

We continue to give a slightly more general definition of a $ C^{0,1} $-uniform manifold than \autoref{def:C11um}. (A pointwise version was also defined in \cite{Pal66}.) This type of manifolds in some cases are important as some smooth (and Lipschitz) approximations can be made to weaken the requirement of high smooth regularity of manifolds. For example, any $ C^1 $ compact embedded submanifold of a smooth (finite-dimensional) Riemannian manifold satisfies the following definition; see \cite[Theorem 6.9]{BLZ08} for a proof of the case where the $ C^1 $ compact embedded submanifold is in a Banach space.
\begin{defi}\label{def:uManifold}
	Let $ M $ be a $ C^1 $ Finsler manifold (see \autoref{finsler}) with Finsler metric $ d $ in each component of $ M $. Let $ M_1 \subset M $. We say $ M $ is \emph{$ \Xi $-$ C^{0,1} $-uniform} (resp. \emph{strongly $ \Xi $-$ C^{0,1} $-uniform}) around $ M_1 $ if (a), (b)(i), (c) (resp. (a), (b)(ii), (c)) hold, where $ \Xi \geq 1 $.
	\begin{enumerate}[(a)]
		\item (Base space) Suppose $ M $ has \emph{locally uniform size neighborhoods} on $ M_1 $ (see \autoref{uniformSize}), i.e., there is an $ \varepsilon_{1} < \epsilon' $ such that $ \overline{U_{m_0}(\varepsilon_{1})} \subset U_{m_0}(\epsilon') $, $ m_0 \in M_1 $, where the closure is taken in the topology of $ M $ and $ U_{m_0}(\epsilon) = \{ m': d(m', m_0) < \epsilon \} $.

		\item (Approximate tangent bundle) Let $ X $ be a $ C^0 $ vector bundle over $ M $ endowed with a $ C^0 $ Finsler structure (see \autoref{finsler}). Let $ \mathcal{A} $ be a (regular) $ C^0 $ vector bundle atlas of $ X $. (i) $ X $ has $ (\Xi - 1) $-almost uniform $ C^{0,1} $-fiber trivializations on $ M_1 $ with respect to $ \mathcal{A} $ (see \autoref{def:lipbundle}) (resp. (ii) $ X $ has $ (\Xi - 1) $-almost $ C^{0,1} $-uniform trivializations on $ M_1 $ with respect to $ \mathcal{A} $ (see \autoref{def:C01Uniform})).

		\item (Compatibility) For certain $ \varepsilon > 0 $ and each $ m_0 \in M_1 $, there is a bi-Lipschitz map (local chart) $ \chi_{m_0}: U_{m_0}(\varepsilon) \to X_{m_0} $ such that $ \chi_{m_0}(0) = 0 $, $ \lip \chi_{m_0}(\cdot)|_{U_{m_0}(\varepsilon)} \leq \chi(\varepsilon) $ and $ \lip \chi^{-1}_{m_0}(\cdot)|_{ \chi_{m_0} (U_{m_0}(\varepsilon)) } \leq \chi(\varepsilon) $, where $ \chi: \mathbb{R}_+ \to \mathbb{R}_+ $ and $ \chi(\varepsilon) \leq \Xi $ as $ \varepsilon \to 0 $.
	\end{enumerate}
	If $ M_1 = M $ (resp. $ \Xi = 1 $), then the words `around $ M_1 $' (resp. `$ \Xi $-') will be omitted.
\end{defi}
In some sense $ TM \approx X $ (if $ \chi(\varepsilon) - 1 > 0 $ is small). In fact, the base space $ M $ can be only a uniformly locally metrizable space (see \autoref{def:ulms}) with locally uniform size neighborhoods on $M_1$ (see \autoref{uniformSize}); in this case, (c) also implies $ M $ is a $ C^{0,1} $ manifold in the sense of \cite{Pal66}.
Also, by \autoref{lem:usd}, there is an $ \varepsilon_{1} > 0 $ such that $ X_{m_0} (\varepsilon_1) \subset \chi_{m_0} (U_{m_0}(\varepsilon)) $, $ m_0 \in M $.

\subsection{Uniformly H\"older foliations} \label{regFoliation}

A $ C^0 $ \textbf{foliation} $ \mathcal{L} $ of a $ C^1 $ manifold $ M $, is a division of $ M $ into disjoint (immersed) submanifolds called leaves of $ \mathcal{L} $ with the following properties (see also \cite{HPS77, AMR88}):
\begin{enumerate}[(a)]
	\item Each leaf of $ \mathcal{L} $ is a connected $ C^1 $ injectively immersed submanifold of $ M $. The unique leaf through $ m $ is denoted by $ \mathcal{L}_m $.
	\item There are two ($ C^0 $) subbundles $ E, F $ of $ TM $ such that $ E \oplus F = TM $. For each $ m \in M $, there is a homomorphism $ \varphi_m: U_{m} \to T_{m}M $, where $ U_{m} $ is a neighborhood of $ m $ in $ M $, such that $ \varphi_{m}(m) = 0 $ and $ \varphi_m (\mathcal{L}^c_{m'}) \subset E_{m} + y' $, where $ \phi(m') = (x', y') \in E_m \times F_{m} $, $ m' \in U_{m} $, and $ \mathcal{L}^c_{m'} $ is the component of $ U_{m} \cap \mathcal{L}_{m'} $ containing $ m' $.
\end{enumerate}
In particular, $ T_m \mathcal{L}_m = E_m $, so we denote $ T\mathcal{L} = E $. The bundles $ E, F $ are called the \emph{tangent bundle} and \emph{normal bundle} of the foliation $ \mathcal{L} $, respectively. $ \mathcal{L}^c_{m} $ is called a \emph{plaque} of $ \mathcal{L} $ at $ m $; it is an embedded submanifold of $ M $. The map $ \varphi_m $ is called a foliation chart at $ m $. Note that $ \varphi_m $ gives a local chart of the plaque at $ m $, i.e., $ \varphi_{m}: \mathcal{L}^c_{m} \to T_{m}\mathcal{L}_m = E_m $.
Sometimes, we also say $ M $ is ($ C^0 $) \emph{foliated} by $ \mathcal{L} $.
If for each $ m $, there is a $ C^1 $ foliation chart at $ m $, then we say $ \mathcal{L} $ is a $ C^1 $ foliation. See also \cite{AMR88} for more details on $ C^1 $ foliations; note that $ C^1 $ foliations are completely different from $ C^0 $ foliations.
Similarly, if the foliation charts can be chosen locally $ C^{k,\theta} $, we say $ \mathcal{L} $ is locally $ C^{k,\theta} $ or $ M $ is \emph{$ C^{k,\theta} $ foliated} by $ \mathcal{L} $.

Let us describe some uniform properties of a foliation. Here we essentially follow \cite{HPS77,PSW97}, but the plaques are represented in an approximate tangent bundle.

Let $ M $ be a $ C^1 $ Finsler manifold and $ M_1 \subset M $.
Let $ M $ with $ M^{\epsilon_1}_1 = \bigcup_{m_0 \in M_1} U_{m_0}(\epsilon_1) $ (for some $ \epsilon_1 > 0 $) satisfy \autoref{def:uManifold} (a, bi, c); the notations thereof will be used below.
Assume $ X = X^c \oplus X^h $ with $ X^c, X^h $ being two $ C^0 $ (vector) subbundles of $ X $. So we have projections $ \Pi^{\kappa}_{m} $, $ m \in M $, such that $ R(\Pi^{\kappa}_{m}) = X^{\kappa}_{m} $, $ \kappa = c, h $, $ \Pi^{c}_{m} + \Pi^{h}_{m} = \id $ and $ m \mapsto \Pi^{\kappa}_{m} $ is $ C^0 $. We assume $ m \mapsto |\Pi^{\kappa}_{m}| $ is $ \varepsilon_0 $-almost uniformly continuous around $ M^{\epsilon_1}_1 $ (see \autoref{ucontinuous}) with small $ \varepsilon_0 > 0 $; a special case is that $ m \mapsto \Pi^{\kappa}_{m} $ is uniformly $ C^0 $ around $ M^{\epsilon_1}_1 $ (see also \autoref{upVector}).
In this case, there are two natural (vector) bundle atlases $ \mathcal{A}^c, \mathcal{A}^h $ of $ X^c, X^h $ induced by $ \mathcal{A} $, respectively, such that $ X^{\kappa} $ has $ \Xi_1 $-almost uniform $ C^{0,1} $-fiber trivializations on $ M^{\epsilon_1}_1 $ with respect to $ \mathcal{A}^{\kappa} $, $ \kappa = c,h $; i.e., for certain $ \epsilon > 0 $, if $ (U_{m}(\epsilon), \varphi^m) \in \mathcal{A} $, then $ (U_{m}(\epsilon), {^{\kappa}\varphi^{m}}) \in \mathcal{A}^{\kappa} $, where
\[
{^{\kappa}\varphi^{m}} (m',x) = (m', \Pi^{\kappa}_{m'}\varphi^{m}_{m'}\Pi^{\kappa}_{m}x).
\]

The plaque at $ m $ in $ U_{m}(\epsilon) $ is denoted by $ \mathcal{L}_{m}(\epsilon) $.
To characterize the (uniform) H\"older continuity of $ \mathcal{L} $ respecting base points around $ M_1 $, as before, we need $ \mathcal{L}_{m}(\epsilon) $ to have local representations with uniform size domains and the representation depending on $ m $ in a H\"older fashion. Suppose for every $ m_0 \in M^{\epsilon_1}_1 $, there is a $ C^1 $ map $ f_{m_0}: X^c_{m_0} (\varepsilon_{m_0}) \to X^h_{m_0} $, such that $ \graph f_{m_0} \subset \chi_{m_0}( \mathcal{L}_{m_0}(\epsilon')) $ and $ f_{m_0}(0) = 0 $; this map always exists if in some sense $ X^c \approx E $ and $ X^h \approx F $.
The following terms will be used in \autoref{foliations}.

\begin{asparaitem}
	\item We say $ \mathcal{L} $ is \textbf{uniformly (locally) ($ \theta $-)H\"older} (respecting base points) around $ M_1 $ or $ m \mapsto \mathcal{L}_{m}(\epsilon) $ is \emph{uniformly (locally) $ C^{0,\theta} $} around $ M_1 $ (in $ C^0 $-topology on bounded sets) if $ f_{m_0}(\cdot) $ are $ C^{0,1} $ uniformly for $ m_0 \in M^{\epsilon_1}_1 $ (see \autoref{fiberR}), $ \inf_{m_0 \in M^{\epsilon_1}_1} \varepsilon_{m_0} > \epsilon $, and $ m_0 \mapsto f_{m_0} $ is uniformly (locally) ($ \theta $-)H\"older around $ M_1 $ with respect to $ \mathcal{A}^{c},
	\mathcal{A}^{h} $ in the sense of \autoref{vsII} \eqref{vs:bb}. If in addition $ M_1 = M $, then we say $ \mathcal{L} $ is a \emph{uniformly (locally) ($ \theta $-)H\"older foliation}.

	In the infinite-dimensional setting, the leaves of a `foliation' might not be $ C^1 $, but the above terminology also makes sense. We allow each leaf of $ \mathcal{L} $ to be only $ C^{0,1} $ and the map $ f_{m_0} $, the representation of the plaque at $ m_0 $, to be only $ C^{0,1} $.

	\item Assume each map in \autoref{def:uManifold} (c) is $ C^1 $. We say $ m \mapsto \mathcal{L}_{m}(\epsilon) $ is $ C^0 $ (resp. \emph{uniformly $ C^0 $}, \emph{uniformly (locally) $ C^{0,\theta} $}) \emph{around $ M_1 $ in $ C^1 $-topology on bounded sets} if (i) $ \inf_{m_0 \in M^{\epsilon_1}_1} \varepsilon_{m_0} > \epsilon $, (ii) $ f_{m_0}(\cdot) $ is $ C^{1,1} $ uniformly for $ m_0 \in M^{\epsilon_1}_1 $ (see \autoref{fiberR}), and (iii) $ m \mapsto f_{m} $ is $ C^0 $ (resp. uniformly $ C^0 $, uniformly (locally) $ C^{0,\theta} $) around $ M_1 $ in $ C^1 $-topology on bounded sets (with respect to $ \mathcal{A}^{c},
	\mathcal{A}^{h} $) (see \autoref{def:uD} (d)). In particular, in this case, if we take $ X = TM $ and $ M $ is $ C^{0,1} $-uniform (or $ C^{1,1} $-uniform) around $ M_1 $ in the sense of \autoref{def:C11um} with $ T\mathcal{L} = E = X^c, F = X^h $, then $ m \mapsto E_{m}: M \to \mathbb{G}(X) $ is $ C^0 $ (resp. uniformly $ C^0 $, uniformly (locally) $ C^{0,\theta} $) around $ M_1 $.

	\item
	Assume $ X $ is a $ C^1 $ bundle with a $ C^0 $ connection $ \mathcal{C} $ and also $ X^{c}, X^{h} $ are $ C^1 $ subbundles of $ X $. (So $ X^c, X^h $ have natural $ C^0 $ connections induced by $ \mathcal{C} $). And assume the maps in \autoref{def:uManifold} (c) are $ C^{1,1} $ uniformly for $ m_0 \in M^{\epsilon_1}_{1} $.
	We say $ \mathcal{L} $ is \textbf{uniformly (locally) $ C^{1,\theta} $} (respecting base points) around $ M_1 $ if (i) it is $ C^1 $ and $ m \mapsto \mathcal{L}_{m}(\epsilon) $ is uniformly (locally) $ C^{0,1} $ around $ M_1 $ in $ C^1 $-topology on bounded sets, and (ii) $ x \mapsto \nabla_{m_0} f_{m_0}(x) $ is $ C^{0,1} $ uniformly for $ m \in M^{\epsilon_1}_1 $ (see \autoref{fiberR}) and $ m_0 \mapsto \nabla_{m_0}f_{m_0}(\cdot) $ is uniformly ($ \theta $-)H\"older around $ M_1 $ with respect to $ \mathcal{A}^{c},
	\mathcal{A}^{h} $ in the sense of \autoref{def:cd}. In addition, if $ M_1 = M $, $ \mathcal{L} $ is called a \emph{uniformly (locally) $ C^{1,\theta} $ foliation}.
\end{asparaitem}

If $ M $ is a $ C^{1,1} $-uniform manifold (see \autoref{def:C11um}) with $ TM $ being $ C^{1,1} $-uniform (see \autoref{def:C11Uniform}) (resp. a strongly $ C^{0,1} $-uniform manifold, see \autoref{def:uManifold}), then the definition of uniform $ C^{1,\theta} $ (resp. $ C^{0,\theta} $) continuity of $ \mathcal{L} $ respecting base points is classical as in e.g. \cite{Fen77, HPS77, PSW97}. We refer the reader to \cite{PSW97} for more discussion of the regularity of foliations.

Let $ \mathcal{L}, \mathcal{F} $ be $ C^0 $ foliations of $ M $. We say $ \mathcal{L} $ is a \emph{subfoliation} of $ \mathcal{F} $ (or $ \mathcal{L} $ \emph{subfoliates} $ \mathcal{F} $) if $ \mathcal{L}_m \subset \mathcal{F}_m $ for every $ m \in M $ and $ T\mathcal{L} $ is a subbundle of $ T\mathcal{F} $. This also means $ \mathcal{F}_m $ is ($ C^0 $) foliated by $ \mathcal{L} \cap \mathcal{F}_m $. For convenience, if $ \mathcal{F}_m $ is $ C^{k,\theta} $ foliated by $ \mathcal{L} \cap \mathcal{F}_m $ for each $ m \in M $, then we also say \emph{each leaf of $ \mathcal{F} $ is $ C^{k,\theta} $ foliated by $ \mathcal{L} $}, or \emph{$ \mathcal{L} $ is a $ C^{k,\theta} $ foliation inside each leaf of $ \mathcal{F} $}. In this case, it is equivalent to $ \mathcal{L} $ $ C^{k,\theta} $ foliating $ \mathcal{F} $ when $ \mathcal{F} $ is endowed with the leaf topology (see \autoref{def:immersed}); now $ \mathcal{F} $ becomes a $ C^{k,\theta} $ manifold locally modeled on $ T_m\mathcal{F}_m $.

%% file: sect5.tex
\chapter{Regularity of Invariant Graphs}\label{stateRegularity}

The overviews of our regularity results and the settings are given in \autoref{resultOverview} and \autoref{settingOverview}, respectively, and the detailed statements of the H\"olderness and smoothness regularities with their proofs are presented in \autoref{leaf} to \autoref{HolderivativeBase}. In \autoref{lipbase}, we give more classical Lipschitz (or H\"older) results for the case of respecting base points under stronger assumptions. The continuity regularity of invariant graphs is considered in \autoref{continuityf}. The corresponding results for the bounded section case are stated in \autoref{bounded} without proofs. At the end, some generalized regularity results in a local version are given in \autoref{generalized}.

\section{Statements of the results: invariant section case}\label{resultOverview}

Throughout \autoref{leaf} to \autoref{HolderivativeBase}, we work in the following basic setting.
\begin{enumerate}[$ \bullet $]
	\item Let $ (X, M, \pi_1) $, $ (Y, M, \pi_2) $, $ u : M \to M $, $ H \sim (F,G) $, $ i: M \to X \times Y $ be as in \autoref{thmA} and $ f $ the bundle map obtained in \autoref{thmA} when $ i $ is an \textbf{invariant section} of $ H $. Since in \autoref{thmA} condition (ii) implies condition (ii)$ ' $ (if we consider $ H^{(n)} $ for large $ n $ instead of $ H $), \emph{in the following proofs, we only deal with the case where condition (ii) holds}.

	\item Since we only focus on partial hyperbolicity in the uniform sense, we add an additional assumption on the spectral condition, i.e., the functions in the (A$ ' $)(B) (or (A)(B)) condition are \emph{bounded} (except \emph{\autoref{lem:leaf1}}).
\end{enumerate}

In the following, we will show $ f $ has higher regularity, once (a) more regularity of the bundle $ X \times Y $, (b) more regularity of the maps $ u $, $ i $, $ F, G $, (c) the spectral gap condition and (d) some technical assumption on the continuity of the functions in (A)(B) (or (A$ ' $)(B)) condition are satisfied.
We will consider the following regularity properties of $ f $:
(a) continuity of $ (m,x) \mapsto f_m(x) $ (\autoref{lem:continuity_f});
(b) smoothness of $ x \mapsto f_m(x) $ (\autoref{lem:leaf1});
(c) H\"{o}lder continuity of $ m \mapsto f_m(x) $ (\autoref{lem:sheaf});
(d) continuity of $ (m,x) \mapsto Df_m(x) $ (\autoref{lem:K1leafc} and \autoref{lem:baseK1});
(e) H\"{o}lder continuity of $ x \mapsto Df_m(x) $ (\autoref{lem:leaf1a} and \autoref{lem:leafk});
(f) H\"{o}lder continuity of $ m \mapsto Df_m(x) $ (\autoref{lem:base0} and \autoref{lem:baseleaf});
(g) smoothness of $ m \mapsto f_m(x) $ (and so $ (m,x) \mapsto f_m(x) $) (\autoref{smoothbase});
(h) continuity of $ (m,x) \mapsto \nabla_mf_m(x) $ (\autoref{lem:leafK} and \autoref{baseK});
(i) H\"{o}lder continuity of $ x \mapsto \nabla_mf_m(x) $ (\autoref{lem:holversheaf});
and finally (j) H\"{o}lder continuity of $ m \mapsto \nabla_mf_m(x) $ (\autoref{lem:final}).
(b) (e) (i) are fiber-regularities of $ f $ and the others are base-regularities of $ f $.

Since $ \graph f \subset H^{-1} \graph f $, we have
\begin{equation}\label{mainF}
	\begin{cases}
	F_m( x, f_{u(m)}(x_m(x)) ) = x_m(x), \\
	G_m( x, f_{u(m)}(x_m(x)) ) = f_m(x).
	\end{cases}
\end{equation}
Consider the corresponding variation equations in some reasonable sense,
\begin{equation}\label{mainX}
	\begin{cases}
	DF_m( x, f_{u(m)}(x_m(x)) ) ( \id, K^1_{u(m)}(x_m(x))R^1_m(x) ) = R^1_m(x),  \\
	DG_m( x, f_{u(m)}(x_m(x)) ) ( \id, K^1_{u(m)}(x_m(x))R^1_m(x) ) = K^1_m(x),
	\end{cases}
\end{equation}
and
\begin{equation}\label{mainM}
	\left\lbrace 
	\begin{split}
	D_mF_m( x, y ) + D_2F_m(x, y) ( & K_{u(m)}(x_m(x))Du(m) \\
	& + Df_{u(m)} (x_m(x)) R_m(x) ) = R_m(x), \\
	D_mG_m( x, y ) + D_2G_m(x, y) ( & K_{u(m)}(x_m(x))Du(m) \\
	& + Df_{u(m)} (x_m(x)) R_m(x) ) = K_m(x),
	\end{split} 
	\right. 
\end{equation}
where $ y = f_{u(m)}(x_m(x)) $; see \eqref{mainMFW} for the explicit meaning of \eqref{mainM}.

\begin{rmk}[Abbreviation of spectral gap condition]\label{absgc}
	To simplify the writing, we make the following abbreviations.
	Set $ \vartheta(m) = (1 - \alpha(m)\beta'(u(m)))^{-1} $ if \autoref{thmA} (ii) holds, and $ \vartheta(m) = 1 $ if \autoref{thmA} (ii)$ ' $ holds.

	Let $ \lambda : M \to \mathbb{R}_+ $. The notation $ \bm {\lambda < 1} $ means that $ \sup_m \vartheta(m) \lambda(m) < 1 $.
	Moreover, if $ \theta: M \to \mathbb{R}_+ $ and $ \theta < 1 $, then $ \bm{\lambda^{*\alpha} \theta} < 1 $ has the following meanings in different settings (see \autoref{def:lypnum} for the meaning of the notations $ \bm{\lambda^*}  $ (\emph{sup Lyapunov numbers} of $ \{ \lambda^{(n)} \} $) and $ \bm{\mathcal{E}(u)} $):
	\begin{enumerate}[(i)]
		\item $ \sup_m  \lambda^\alpha(m) \vartheta(m)\theta(m) < 1 $ or $ \sup_m  (\lambda^\alpha \vartheta\theta)^*(m) < 1 $ if $ x_m(\cdot) $ are bounded uniformly for $ m \in {M} $ (in particular when $ X_m $ is bounded uniformly for $ m \in M $) in Lemmas \ref{lem:leaf1a} and \ref{lem:holversheaf}, and additionally $ u $ is a bounded function (in particular when $ M $ is bounded) in Lemmas \ref{lem:sheaf}, \ref{lem:base0}, \ref{lem:baseleaf} and \ref{lem:final};

		\item $ \sup_m  \lambda^\alpha(m) \vartheta(m)\theta(m) < 1 $ or $ \sup_m  (\lambda^\alpha \vartheta\theta)^*(m) < 1 $ if $ \vartheta\theta \in \mathcal{E}(u) $ or $ \alpha = 1 $;
		\item $ \sup_m \lambda^{*\alpha}(m) (\vartheta\theta)^*(m) < 1 $ otherwise.
	\end{enumerate}

	\noindent \emph{Additional notations}: $ \lambda\theta $ and $ \max\{ \lambda, \theta \} $ are defined by
	\[
	(\lambda\theta) (m) = \lambda(m)\theta(m), ~\max\{ \lambda, \theta \}(m) = \max\{ \lambda(m), \theta(m) \}.
	\]

	\noindent \emph{Suggestion}: At the \emph{first} reading, the readers may think the notation
	$ \lambda^{*\alpha} \theta < 1 $ means
	\[
	\sup_m  \sup_{N \geq 0}\lambda(u^N(m)) \sup_{N \geq 0} \vartheta(u^N(m)) \theta(u^N(m)) < 1;
	\]
	or for simplicity, assume all the functions in the (A$ ' $)(B) (or (A)(B)) condition are \emph{constants}.
\end{rmk}

Let us take a quick glimpse of what our regularity results say in a \emph{not} very sharp and general setting.
For the meaning of $ C^{k,\alpha} $ regularity of bundle maps, see \autoref{bundle}.

\begin{enumerate}[({E}1)]
	\item (about $ M $) Let $ M $ be a $ C^1 $ Finsler manifold and $ M^0_1 \subset M $. Let $ M_1 $ be the $ \varepsilon $-neighborhood of $ M^0_1 $ ($ \varepsilon > 0 $). Take a $ C^1 $ atlas $ \mathcal{N} $ of $ M $. Let $ \mathcal{M} $ be the canonical bundle atlas of $ TM $ induced by $ \mathcal{N} $. Assume $ M $ is \emph{$ (1+\zeta) $-$ C^{1,1} $-uniform} around $ M_1 $ with respect to $ \mathcal{M} $ (see \autoref{def:C11um}) where $ \zeta > 0 $; see also \autoref{bgemetry} for examples.

	\item (about $ X \times Y $) $ (X, M, \pi_1) $, $ (Y, M, \pi_2) $ are $ C^1 $ bundles endowed with $ C^0 $ connections $ \mathcal{C}^X, \mathcal{C}^Y $ which are \emph{uniformly (locally) Lipschitz} around $ M_1 $ (see \autoref{lipcon} or \autoref{lipconC}). Take $ C^1 $ \emph{normal} bundle atlases $ \mathcal{A} $, $ \mathcal{B} $ of $ X, Y $, respectively. Assume $ (X, M, \pi_1) $ and $ (Y, M, \pi_2) $ have \emph{$ \zeta $-almost $ C^{1,1} $-uniform trivializations} on $ M_1 $ with respect to $ \mathcal{A}, \mathcal{B} $ (and $ \mathcal{M} $), respectively; see \autoref{def:C11Uniform} and also \autoref{bgemetry} for examples.
	The fibers $ X_m, Y_m $ are Banach spaces (see also \autoref{fibersG}).

	\item (about $ i $) The section $ i $ is a $ 0 $-section of $ X \times Y $ with respect to $\mathcal{A \times B}$; see \autoref{0-section}.

	\item (about $ u $) $ u : M \to M $ is a $ C^1 $ map such that (i) $ |Du(m)| \leq \mu(m) $ where $ \mu : M \to \mathbb{R}_+ $ and $ \sup_{m \in M_1} \mu (m) < \infty $, (ii) $ m \mapsto |Du(m)| $ is $ \zeta $-almost uniformly continuous around $ M_1 $ (see \autoref{ucontinuous}), and (iii) $ u(M) \subset M^0_1 $.

	\item [(E5)] The functions in the (A$ ' $)(B) (or (A)(B)) condition are $ \zeta $-almost uniformly continuous around $ M_1 $ and $ \zeta $-almost continuous; see \autoref{ucontinuous}.

\end{enumerate}

\begin{thm}\label{thmC}
	Assume \textnormal{(E1)--(E5)} hold with $ \zeta > 0 $ small depending on the following spectral gap conditions. Let $ H: X \times Y \rightarrow X \times Y $ be a bundle correspondence over $ u $ with generating bundle map $ (F,G) $. Let $ f $ be given in \autoref{thmA} when $ i $ is an \textbf{invariant section} of $ H $. Then we have the following results on the regularity of $ f $. (In the following, $ 0 < \alpha, \beta \leq 1 $.)
	\begin{enumerate}[(1)]
		\item $ f_m(i_X(m)) = i_Y(m) $. If $ F, G $ are continuous, so is $ f $.

		\item\label{xx1} Assume for every $ m \in M $, $ F_{m}(\cdot), G_{m}(\cdot) $ are $ C^1 $. Then so is $ f_m (\cdot) $. Moreover, if the fiber derivatives (see \autoref{fiberR}) $ D^v F, D^v G $ are $ C^{0} $, so is $ D^vf $. Also, there is a unique $ K^1 \in L(\Upsilon_X^V, \Upsilon_Y^V) $ over $ f $ satisfying \eqref{mainX} and $ D^vf = K^1 $.

		\item Under \eqref{xx1}, in addition, suppose (i) $ DF_{m}(\cdot), DG_{m}(\cdot) $ are $ C^{0,\gamma} $ uniformly for $ m \in M $, and
		(ii) $ \lambda^{*\gamma\alpha}_s \lambda_s \lambda_u < 1 $.
		Then $ Df_m(\cdot) $ is $ C^{0,\gamma\alpha} $ uniformly for $ m $.

		\item Suppose
		(i) $ F, G $ are uniformly (locally) $ C^{0,1} $ around $ M_1 $ (i.e. \eqref{ligFG} holds),
		and (ii) $  (\max\{ \lambda^{-1}_s, 1 \} \mu)^{*\alpha} \lambda_s \lambda_u < 1 $.
		\textbf{Or} suppose
		(i$ ' $) $ F, G $ are uniformly (locally) $ C^{1,1} $ around $ M_1 $ (see \autoref{rmk:hFG}),
		and (ii$ ' $) $ (\frac{\mu}{\lambda_s})^{*\alpha} \lambda_s \lambda_u < 1 $.
		Then $ m \mapsto f_m(x) $ is uniformly (locally) $ \alpha $-H\"older around $ M_1 $.

		\item Suppose
		(i) $ m \mapsto DF_{m}(i_X(m), i_Y(u(m))) $, $ m \mapsto DG_{m}(i_X(m), i_Y(u(m))) $ are uniformly (locally) $ C^{0,\gamma} $ around $ M_1 $,
		and (ii) $  \mu^{* \gamma \alpha} \lambda_s \lambda_u < 1 $.
		Then $ m \mapsto Df_m(i_X(m)) $ is uniformly (locally) $ C^{0,\gamma \alpha} $ around $ M_1 $.

		\item Suppose
		(i) $ D^vF, D^vG $ are uniformly (locally) $ C^{0,1} $ around $ M_1 $ (i.e. estimates \eqref{lipzz} hold),
		and (ii) $ \lambda^2_s \lambda_u \mu^\alpha < 1 $, $ \lambda^{*\beta}_s \lambda_s \lambda_u < 1 $, $ \mu^{*\alpha} \lambda_s \lambda_u < 1 $.
		Then $ m \mapsto Df_m(x) $ is uniformly (locally) $ \alpha \beta $-H\"older around $ M_1 $.

		\item\label{zz1} Suppose
		(i) $ F, G $ are $ C^{1,1} $ around $ M_1 $ (see \autoref{rmk:hFG}) and $ C^1 $ in $ X \times Y $,
		and (ii) $ \lambda_s \lambda_u \mu < 1 $.
		Then $ f $ is $ C^1 $, $ \nabla_m f_m(i_X(m)) = 0 $ for all $ m \in M_1 $, and there is a constant $ C $ such that $ |\nabla_m f_m(x)| \leq C |x| $ for all $ x \in X_m $, $ m \in M_1 $. Also, there is a unique $ K \in L(\Upsilon^H_X, \Upsilon^V_Y)  $ over $ f $ satisfying \eqref{mainM} (or more precisely \eqref{mainMFW}) and $ \nabla f = K $ (the covariant derivative of $ f $, see \autoref{defi:coderivative}).

		Moreover, if an additional gap condition holds: $ \lambda_s^{*\beta} \lambda_s \lambda_u < 1 $, $ \max\{ 1, \lambda_s \}^{*\alpha} \lambda_s \lambda_u \mu < 1 $, then $ \nabla_m f_m(\cdot) $ is locally $ \alpha \beta $-H\"older uniformly for $ m \in M $.

		\item Under \eqref{zz1}, assume $ u $ is uniformly (locally) $ C^{1,1} $ around $ M_1 $ (see \autoref{def:manifoldMap}). Suppose
		$ \lambda^2_s \lambda_u < 1 $, $ \lambda^2_s \lambda_u \mu < 1 $, $ \max\{ \frac{\mu}{\lambda_s}, \mu \}^{*\alpha} \lambda_s \lambda_u \mu < 1 $.
		\textbf{Or} suppose
		(i$ ' $) $ TM $ has a $ 0 $-almost $ C^{1,1} $-uniform trivialization on $ M_1 $ with respect to $ \mathcal{M} $,
		(ii$ ' $) $ F, G $ are uniformly (locally) $ C^{2,1} $ around $ M_1 $,
		and (iii$ ' $) $ \lambda_s < 1 $, $ \mu^{*\alpha} \lambda_s \lambda_u \mu < 1 $.
		Then $ m \mapsto \nabla_m f_m(x) $ is uniformly (locally) $ \alpha $-H\"older around $ M_1 $.
	\end{enumerate}
\end{thm}

\begin{proof}
	See Lemmas \ref{lem:continuity_f}, \ref{lem:leaf1}, \ref{lem:leaf1a}, \ref{lem:K1leafc}, \ref{lem:sheaf}, \ref{lem:base0}, \ref{lem:baseleaf}, \ref{smoothbase}, \ref{lem:holversheaf}, \ref{lem:final}, and \autoref{lipbase}.
\end{proof}

\section{Settings for regularity of invariant graphs: an overview}\label{settingOverview}

We need some basic assumptions on the base space $ M $ and the bundles $ X, Y $ to study the regularity of $ f $. The assumptions will change from lemma to lemma according to different regularities of $ f $. The assumptions for the invariant section $ i $ and the map $ u $ are also given here.

\begin{enumerate}
	\item [(H1)] (about $ M $) \textbf{(i)} Let $ M $ be a locally metrizable space (associated with an open cover $ \{ U_m: m \in M \} $); see \autoref{defi:locallyM}. The metric in $ U_{m} $ is denoted by $ d_m(\cdot,\cdot) = |\cdot - \cdot| $. Let $ M_1 \subset M $.

	\noindent\textbf{(ii)} $ M $ is a $ C^1 $ Finsler manifold (see \autoref{finsler}). The norm of $ T_mM $ is denoted by $ |\cdot|_m $.
\end{enumerate}

\begin{enumerate} [({H1}a)]
	\item Let (H1) (i) hold.

	\item \label{baseSpaceAA} (\emph{uniform compatibility of (i) (ii)}) Let (H1) hold.  There exist positive constants $ \Theta_1, \Theta_2 $ such that for every $ m_0 \in M_1 $, there is a constant $ \epsilon'_{m_0} >0 $ and a $ C^1 $ local chart $ \chi_{m_0}: U_{m_0}(\epsilon'_{m_0}) \to T_{m_0}M $, $ \chi_{m_0}(m_0) = 0 $, such that the norm $ |\cdot|_{m_0} $ of $ T_{m_0}M $ with $ D\chi_{m_0}(m_0) = \id $ satisfies
	\begin{equation}\label{A11}
	\Theta_1d_{m_0}(m_1,m_2)  \leq  |\chi_{m_0}(m_1) - \chi_{m_0}(m_2)|_{m_0}  \leq \Theta_2 d_{m_0}(m_1,m_2),
	\end{equation}
	for all $ m_1,m_2 \in U_{m_0}(\epsilon'_{m_0}) $.

	\item \label{baseSpaceAS}
	Let $ M $ be a $ C^1 $ Finsler manifold with Finsler metric $ d $ in its components.
	Assume there are constants $ \Xi, \epsilon' > 0 $ and $ C^1 $ local charts $ \chi_{m_0}: U_{m_0}(\epsilon') \to T_{m_0}M $ at $ m_0 \in M_1 $ satisfying $ D\chi_{m_0}(m_0) = \id $, and
	\[
	\sup_{m' \in U_{m_0}(\epsilon')}|D\chi_{m_0}(m')| \leq \Xi,~ \sup_{m' \in U_{m_0}(\epsilon')}|D\chi_{m_0}^{-1}(\chi_{m_0}(m'))| \leq \Xi.
	\]
	Note that now (H1b) holds with $ d_m = d $ (where $ \Theta_1 = \Xi^{-1} $, $ \Theta_2 = \Xi $) and $ \inf_{m_0 \in M_1}\epsilon'_{m_0} > 0 $.
\end{enumerate}

\begin{enumerate}
	\item[(H2)] (about $ X \times Y $) Let $ \epsilon_2 > 0 $ (possibly $ \epsilon_2 = \infty $) and
	\begin{align*}
	\mathcal{A} = \{ (U_{m_0}(\epsilon), \varphi^{m_0}) ~&\text{a bundle chart of $ X $ at}~ m_0: \\
	& \varphi^{m_0}_{m_0} = \id, m_0 \in M_1, 0 < \epsilon \leq \epsilon_2 \}  \subset \mathcal{A}', \\
	\mathcal{B} = \{ (U_{m_0}(\epsilon), \phi^{m_0}) ~&\text{a bundle chart of $ Y $ at}~ m_0: \\
	& \phi^{m_0}_{m_0} = \id, m_0 \in M_1, 0 < \epsilon \leq \epsilon_2 \}  \subset \mathcal{B}',
	\end{align*}
	where $ \mathcal{A}', \mathcal{B}' $ are bundle atlases of $ X, Y $ on $ M $, respectively.

	\noindent\textbf{(i)} The fibers $ X_m, Y_m $ of the bundles $ X, Y $ are Banach spaces (see also \autoref{fibersG}).

	\noindent\textbf{(ii)} $ X, Y $ are $ C^1 $ bundles with $ C^0 $ connections $ \mathcal{C}^X, \mathcal{C}^Y $, respectively (see also \autoref{connections}). $ \mathcal{A}' $, $ \mathcal{B}' $ are $ C^1 $ and $ \mathcal{A}, \mathcal{B} $ are \emph{normal} with respect to $ \mathcal{C}^X, \mathcal{C}^Y $ respectively (see \autoref{def:normal}).

	\noindent\textbf{(iii)} $ X, Y $ are $ C^{0} $ topology bundles with $ C^{0} $ bundle atlases $ \mathcal{A}', \mathcal{B}' $ (see e.g. \autoref{bundleP} \eqref{topologyBundle}).

	\noindent\textbf{(iv)} $ X, Y $ are $ C^{1} $ topology bundles with $ C^{1} $-fiber bundle atlases $ \mathcal{A}', \mathcal{B}' $; see \autoref{def:C1fiber}.

	\noindent\textbf{(v)} Let $ \varepsilon > 0 $ be small (depending on the following spectral gap conditions). (1) $ X, Y $ have $ \varepsilon $-almost \emph{uniform} $ C^{0,1} $-fiber trivializations on $M_1$ with respect to $ \mathcal{A}, \mathcal{B} $, respectively; (2) $ X, Y $ have $ \varepsilon $-almost \emph{uniform} $ C^{1,1} $-fiber trivializations on $M_1$ with respect to $ \mathcal{A}, \mathcal{B} $, respectively; (1$ ' $) $ X, Y $ have $ \varepsilon $-almost local $ C^{0,1} $-fiber trivializations with respect to $ \mathcal{A}', \mathcal{B}' $, respectively; (2$ ' $) $ X, Y $ have $ \varepsilon $-almost local $ C^{1,1} $-fiber trivializations with respect to $ \mathcal{A}', \mathcal{B}' $, respectively. See \autoref{uniform lip bundle}.
\end{enumerate}

\begin{enumerate}[({H2}a)]
	\item Let (H2)(v1) hold.

	\item Let (H2)(i, v2) hold.

	\item Let (H2)(i, ii, v1) hold.

	\item Let (H2)(i, ii, v2) hold.
\end{enumerate}

\begin{enumerate}
	\item [(H3)] (about $ i $) The section $ i $ is a $ 0 $-section of $ X \times Y $ with respect to $\mathcal{A \times B}$; see \autoref{0-section}. In this case, we use the \emph{notations}
	\[
	|(x,y)| \triangleq \max\{ |x|, |y| \}, ~|x| = d(x,i_X(m)), ~|y| = d(y,i_Y(m)),
	\]
	if $(x,y) \in X_m \times Y_m$.
	\item [(H3$ ' $)] $ i: M \to X \times Y $ is continuous.
\end{enumerate}

\begin{enumerate}
	\item [(H4)] (about $ u $) $ u(M) \subset M_1 $ and $ u: M \to M $ is continuous.
	\item [(H4$ ' $)] $ u: M \to M $ is uniformly continuous around $ M_1 $; see \autoref{ucontinuous}.
\end{enumerate}
\begin{enumerate}[({H4}a)]
	\item  (\emph{Lipschitz continuity of $ u $}) Let (H4) and (H1a) hold. There exist an $ \varepsilon_1 > 0 $ (possibly $ \varepsilon_1 = \infty $) and a function $ \mu: M \to \mathbb{R}_+ $ such that $ \sup_{m_0 \in M_1} \mu(m_0) \triangleq \hat{\mu} < \infty $ and for every $m_0 \in M_1 $,
	$ u(U_{m_0}(\mu^{-1}(m_0) \varepsilon_1)) \subset U_{u(m_0)} $,
	\[
	d_{u(m_0)}(u(m),u(m_0)) \leq \mu(m_0) d_{m_0}(m,m_0), ~m \in U_{m_0}(\mu^{-1}(m_0) \varepsilon_1).
	\]

	\item  (\emph{smoothness of $ u $}) Let (H4a) and (H1) hold. Assume $ u $ is $ C^1 $ and $ |Du(m_0)| \leq \mu(m_0) $ for every $ m_0 \in M_1 $, where $ |Du(m)| $ is defined by
	\[
	|Du(m)| = \sup\{ |Du(m)x|_{u(m)}: x \in T_mM, |x|_m \leq 1 \}.
	\]

	\item  (\emph{Lipschitz continuity of $ Du $}) Let (H4b) and (H1c) hold. $ Du $ is uniformly (locally) Lipschitz around $ M_1 $ in the sense of \autoref{def:manifoldMap}; that is, for the local representation
	\[
	\widehat{Du}_{m_0} (m)v \triangleq D\chi_{u(m_0)}(u(m)) Du(m)D\chi_{m_0}^{-1}(\chi_{m_0}(m))v , ~(m,v) \in U_{m_0}(\epsilon'_1) \times T_{m_0}M,
	\]
	one has
	\[
	\|\widehat{Du}_{m_0} (m) - Du(m_0)\| \leq C_0 d(m,m_0), m \in U_{m_0}(\epsilon'_1), m_0 \in M_1,
	\]
	for some small $ \epsilon'_1 < \hat{\mu}^{-1} \varepsilon_{1} $ and some constant $ C_0 > 0 $ (independent of $ m_0 $).
\end{enumerate}

\begin{enumerate}
	\item [(H5)] The functions in the (A$ ' $)(B) (or (A)(B)) condition are $ \varepsilon $-almost uniformly continuous around $ M_1 $ and $ \varepsilon $-almost continuous in $ M $, where $ \varepsilon > 0 $ is small (depending on the following spectral gap conditions); see \autoref{ucontinuous}.
	\item [(H5$ ' $)] The functions in (A$ ' $)(B) (or (A)(B)) condition are $ \varepsilon $-almost continuous in $ M $, where $ \varepsilon > 0 $ is small (depending on the following spectral gap conditions); see \autoref{ucontinuous}.
\end{enumerate}

We give some comments on the above settings.

\begin{rmk}
	\begin{enumerate}[(a)]
		\item For (H1a), see also \autoref{locallyM}; we emphasize again no requirement that $ M $ admits a metric is made. For (H1c), see also \autoref{regManifold} and \autoref{examples} for some examples; here, in contrast to assumption ($ \blacksquare $) (on \autopageref{C1MM}), we do not assume $ M $ has \emph{locally uniform size neighborhoods} on $ M_1 $ (see \autoref{uniformSize}).

		\item For (H2), $ X \times Y $ has a natural bundle atlas on $ M_1 $, i.e., $ \mathcal{A \times B} \triangleq \{ (U_{m_0}(\epsilon), \varphi^{m_0} \times \phi^{m_0}): m_0 \in M_1, 0 < \epsilon \leq \epsilon_2 \} $. If (H2a) (resp. (H2b)) holds, then $ X \times Y $ also has $ \varepsilon $-almost uniform $ C^{0,1} $- (resp. $ C^{1,1} $-)fiber trivializations on $M_1$ with respect to $ \mathcal{A} \times \mathcal{B} $; so does $ X \otimes_u Y $ (see \eqref{Whsum}) with respect to $ \mathcal{A} \otimes_u \mathcal{B} \triangleq \{ (U_{m_0}(\epsilon), \varphi^{m_0} \times \phi^{u(m_0)}): m_0 \in M_1, 0 < \epsilon \leq \epsilon_2 \} $ if in addition (H4a) holds. If (H2) (ii) holds and $ X \times Y $ is equipped with the product connection $ \mathcal{C}^X \times \mathcal{C}^Y $, then $ \mathcal{A \times B} $ is also normal on $ M_1 $ with respect to $ \mathcal{C}^X \times \mathcal{C}^Y $.

		\item For (H4b), from assumptions (H1b) and (H4a), we have $ |Du(m_0)| \leq \frac{\Theta_2}{\Theta_1} \mu(m_0) $. So (H4b) really only says $ u \in C^1 $ if we use $ \frac{\Theta_2}{\Theta_1} \mu $ instead of $ \mu $ to characterize the `spectral gap condition' in the following subsections. If one chooses a `better' metric in $ U_{m_0} $, then $ \frac{\Theta_2}{\Theta_1} $ could be sufficiently close to $ 1 $ (or equal to $ 1 $). For example, in many applications, we take $ d_{m_0}(m_1, m_2) = |\chi_{m_0}(m_1) - \chi_{m_0}(m_2)|_{m_0} $. But this is not a `global' metric. Another choice is the Finsler metric (see \autoref{length}). Furthermore, if $ M $ is a Finsler manifold in the sense of Palais (see \autoref{finsler}), then the Finsler metric in each component of $ M $ is a \emph{length metric} (see \autoref{length}) and $ |Du(m_0)| \leq \mu(m_0) $.

		\item For (H5) and (H5$ ' $), the functions may not even be continuous.
	\end{enumerate}
\end{rmk}

\begin{rmk}\label{fibersG}
	Actually, assumption (H2) (i) (i.e. the fibers $ X_m, Y_m $ of the bundles $ X $, $ Y $ are Banach spaces) is not general. This can be weakened in some sense but the fibers of $ X $, $ Y $ have to be paracompact $ C^1 $ Banach manifolds. We need this assumption at least for two purposes: \textbf{(i)} We need that $ (DF(z), DG(z)) $ satisfies the (A$ ' $)(B) (or (A)(B)) condition which is implied by that $ (F, G) $ satisfies the (A$ ' $)(B) (or (A)(B)) condition (see \autoref{lem:c1}); \textbf{(ii)} some bounded geometry property of the fibers (see e.g. \autoref{bgemetry}) is needed in order to get some uniform estimates.

	For \textbf{(i)}, due to the local result \autoref{lem:cAB}, evidently $ (DF(z), DG(z)) $ satisfies the (A$ ' $)(B) (or (A)(B)) condition once $ (F, G) $ satisfies the (A$ ' $)(B) (or (A)(B)) condition with a restriction on the functions $ \alpha, \beta $, and $ X_m $, $ Y_m $ are complete $ C^1 $ Finsler manifolds \emph{in the sense of Palais} (see \autoref{finsler}) with Finsler metrics (note that in this case \autoref{liplength} holds). If the metrics of the fibers are not \emph{length metrics}, one can assume directly that $ (DF(z), DG(z)) $ satisfies the (A$ ' $)(B) (or (A)(B)) condition or use a stronger Lipschitz condition instead (see e.g. \autoref{lem:ab} and \autoref{lem:a3}), which is similar to what we do for $ u $ (by assuming (H4a) and (H4b) separately).

	For \textbf{(ii)}, if $ X_m, Y_m $ are complete Riemannian manifolds having bounded geometry (see \autoref{bgemetry}), then the results in this section indeed hold (but the statements are more complex). As we do for $ M $, one can generalize $ X_m, Y_m $ to the Banach-manifold-like setting (for example, $ X_m, Y_m $ are complete connected $ C^1 $ Finsler manifolds which are $ C^{1,1} $-uniform, see \autoref{def:C11um}).

	In this paper, we do not give a detailed analysis of the generalization of (H2) (i). All the main ideas have being given when dealing with the generalization of the base space $ M $.
\end{rmk}

\begin{rmk}[A few different generalizations]
	\begin{enumerate}[(a)]
		\item For the H\"older continuity of $ m \mapsto f_m(x) $, $ K^1_m(x) $, $ K_m(x) $, there is a simpler case with `better' spectral gap conditions: there is an $ \varepsilon_1 > 0 $ ($ \varepsilon_1 = \infty $ is possible) such that for any $ m_0 \in M_1 $, $ (U_{m_0}(\varepsilon_1), \varphi^{m_0}) \in \mathcal{A} $, $ (U_{m_0}(\varepsilon_1), \phi^{m_0}) \in \mathcal{B} $ (in (H2)) and $ u(U_{m_0}(\varepsilon_1)) \subset U_{u(m_0)} (\varepsilon_1) $ (in (H4a)); see \autoref{rmk:simple}. This situation will appear, e.g., in the trivial bundles; see also \autoref{thm:bundlemaps}, \autoref{thm:fake1} and \autoref{thm:fake2}.
		\item In fact the regularity results do not depend on the existence results; see \autoref{generalized}.
	\end{enumerate}
\end{rmk}

\begin{rmk}[Uniform $ C^{k,\gamma} $ continuity of $ F, G $]\label{rmk:hFG}
	Let $ \widehat{F}_{m_0}(\cdot,\cdot), \widehat{G}_{m_0}(\cdot,\cdot) $ be the (regular) local representations of $ F, G $ at $ m_0 \in M_1 $ with respect to $ \mathcal{A}, \mathcal{B} $ under (H1a), (H2a), (H4a); see \eqref{localrepre} or \autoref{def:regular}. Also write $ \widehat{D}_m F_{m_0}(\cdot,\cdot), \widehat{D}_m G_{m_0}(\cdot,\cdot) $ as the (regular) local representations of $ \nabla F, \nabla G $ at $ m_0 \in M_1 $ with respect to $ \mathcal{A}, \mathcal{B}, \mathcal{M} $ under (H1c), (H2d), (H4b) and $ F,G $ being $ C^1 $; see \eqref{basediff00} or \autoref{tensorH}.
	For convenience, we use the following notions. Consider the estimates below:
	\begin{gather}\label{ligFG}
	\max\{ |\widehat{F}_{m_0}(m_1, z) - \widehat{F}_{m_0}(m_0, z)|, |\widehat{G}_{m_0}(m_1, z) - \widehat{G}_{m_0}(m_0, z)| \} \leq M_0 |m_1 - m_0|, \\
	\begin{cases}\label{lipzz}
	\begin{split}
	\max\{ | D_z\widehat{F}_{m_0}(m_1, z) - D_z\widehat{F}_{m_0}(m_0, z) |,  & | D_z\widehat{G}_{m_0}(m_1, z) - D_z\widehat{G}_{m_0}(m_0, z) | \} \\
	& \leq M_0 |m_1 - m_0|,
	\end{split}\\
	\max\{ | D F_{m}(z_1) - D F_{m}(z_2)|, |D G_{m}(z_1) - D G_{m}(z_2) | \} \leq M_0 |z_1 - z_2|,
	\end{cases}\\
	\begin{cases}\label{ligmm}
	\begin{split}
	\max\{ |\widehat{D}_m F_{m_0}(m_1, z) - \widehat{D}_m F_{m_0}(m_0, z)|,  & |\widehat{D}_m G_{m_0}(m_1, z) - \widehat{D}_m G_{m_0}(m_0, z)| \} \\
	& \leq M_0 |m_1 - m_0| ,
	\end{split}\\
	\max\{ | \nabla_m F_{m}(z_1) - \nabla_m F_{m}(z_2)|, |\nabla_m G_{m}(z_1) - \nabla_m G_{m}(z_2) | \} \leq M_0 |z_1 - z_2|,
	\end{cases}
	\end{gather}
	for all $ m_1 \in U_{m_0} (\epsilon') $, $ z \in X_{m_0} \times Y_{u(m_0)} $, $ m_0 \in M_1 $, and $ z_1, z_2 \in X_{m} \times Y_{u(m)} $, $ m \in M $, where $ \epsilon' > 0 $ is small and $ M_0 > 0 $ independent of $ m_0 \in M_1 $. Note that we have assumed $ F, G $ are uniformly $ C^{0,1} $-fiber (see \autoref{fiberR}).
	\begin{enumerate}[(i)]
		\item Under (H1a), (H2a), (H4a), we say $ F, G $ are uniformly (locally) $ C^{0,1} $ around $ M_1 $ if \eqref{ligFG} holds.
		\item Under (H1a), (H2b), (H4a), we say $ D^vF, D^vG $ are uniformly (locally) $ C^{0,1} $ around $ M_1 $ if \eqref{lipzz} holds.
		\item Under (H1c), (H2d), (H4b), we say $ \nabla F, \nabla G $ are uniformly (locally) $ C^{0,1} $ around $ M_1 $ if \eqref{ligmm} holds.
		\item Under (H1c), (H2d), (H4b), we say $ F, G $ are uniformly (locally) $ C^{1,1} $ around $ M_1 $ if \eqref{ligFG}--\eqref{ligmm} hold.
		\item Similarly, $ C^{0,\gamma} $ and $ C^{1,\gamma} $ can be defined.
	\end{enumerate}
\end{rmk}

\emph{In the following, the constants $ M_0 $, $ C $ are independent of $ m \in M $ or $ m_0 \in M_1 $. Also in the proofs, the concrete value of the symbol $ \widetilde{C} $ will change from line to line, but it is independent of $ m \in M $ or $ m_0 \in M_1 $.}

\section{Smooth leaves: $ C^{k,\alpha} $ continuity of $ x \mapsto f_m(x) $} \label{leaf}

\subsection{$ C^1 $ leaves: smoothness of $ x \mapsto f_m(x) $}

\begin{lem}\label{lem:leaf1}
	Assume the following conditions hold:
	\begin{enumerate}[(a)]
		\item The fibers of the bundles $ X, Y $ are Banach spaces (see also \autoref{fibersG}).
		\item For every $ m \in M $, $ F_m(\cdot), G_m(\cdot) $ are $ C^1 $.
		\item (spectral gap condition) $ \lambda_s \lambda_u < 1 $; see \autoref{absgc}.
	\end{enumerate}
	Then we have the following conclusions:
	\begin{enumerate}[(1)]
		\item There exists a unique vector bundle map $ K^1 \in L(\Upsilon_X^V, \Upsilon_Y^V) $ (see \eqref{HVspace}) over $ f $ such that $ |K^1_m(x)| \leq \beta'(m) $, for all $ m \in M $ and \eqref{mainX} holds.
		\item ($ C^1 $ smooth leaves) For every $ m \in M $, $ f_m(\cdot): X_m \to Y_m $ is $ C^1 $, and $ Df_m(x) = K^1_m(x) $ (i.e. $ D^v f = K^1 $).
	\end{enumerate}
\end{lem}

Hereafter, for $ K^1 \in L_{f}(\Upsilon_X^V, \Upsilon_Y^V) $ we write $ K^1_m(x) = K^1_{(m,x)}  $.
\begin{proof}
	Define a metric space
	\begin{align*}
	E^L_1 \triangleq \{  K^1 \in L(\Upsilon_X^V, & \Upsilon_Y^V)~ \text{a vector bundle map over}~ f : \\
	& |K^1_m(x)| \leq \beta'(m), \forall x \in X_m,
	x \mapsto K^1_m(x) ~\text{is continuous},  m \in M \},
	\end{align*}
	with metric
	\[
	d_1 (K^1, K^{1}{'}) \triangleq \sup_{m \in M} \sup_{x \in X_m} |K^1_m(x) - K^{1}_{m}{'}(x)|.
	\]
	Note that $ \sup_m \beta(m) \leq \hat{\beta} $. So $ d_1 $ is well defined and $ (E^L_1, d_1) \neq \emptyset $ is complete.

	Also, for each $ (x,y) \in X_m \times Y_{u(m)} $, the map $ (DF_m(x,y), DG_m(x,y)) $ satisfies the (A$'$) ($ \alpha(m) $, $ \lambda_u(m) $), (B) ($ \beta'(u(m)) $; $ \beta'(m) $, $ \lambda_s(m) $) condition when $ (F_m, G_m) $ satisfies the (A$'$) ($ \alpha(m) $, $ \lambda_u(m) $) (B) ($ \beta'(u(m)) $; $ \beta'(m) $, $ \lambda_s(m) $) condition (here $ \beta'(u(m)) \leq \beta(m) $); see e.g. \autoref{lem:c1}.

	Since $ \alpha(m) \beta'(u(m)) < 1 $, given a $ K^1 \in E^L_1 $, there is a unique $ R^1_m(x) \in L(X_m, X_{u(m)}) $ satisfying
	\[
	DF_m( x, f_{u(m)}(x_m(x)) ) ( \id, K^1_{u(m)}(x_m(x))R^1_m(x) ) = R^1_m(x),
	\]
	and $ |R^1_m(x)| \leq \lambda_s(m) $ (by (B)). Moreover, $ x \mapsto R^1_m(x) $ is continuous; see e.g. \autoref{lem:fconx}. Now let
	\[
	\widetilde{K}^1_m(x) \triangleq DG_m( x, f_{u(m)}(x_m(x)) ) ( \id, K^1_{u(m)}(x_m(x))R^1_m(x) ).
	\]
	By (B), we see $ \widetilde{K}^1 \in E^L_1 $.

	Consider the graph transform $ \varGamma^1: E^L_1 \to E^L_1, K^1 \mapsto \widetilde{K}^1 $.
	\begin{slem}\label{slem:TT}
		$ \varGamma^1 $ is Lipschitz, and $ \lip \varGamma^1 \leq \sup_m \dfrac{\lambda_s(m) \lambda_u(m)}{1-\alpha(m)\beta'(u(m))} < 1  $.
	\end{slem}
	\begin{proof}
		Let $ \widetilde{K}^1 = \varGamma^1 K^1, ~ \widetilde{K}^1{'} = \varGamma^1 K^1{'} $. By \autoref{lem:bas0}, 
		\begin{align*}
		| \widetilde{K}^1_m(x) - \widetilde{K}^{1}_m{'}(x) | & \leq \frac{\lambda_u(m)}{1-\alpha(m) \beta'(u(m))} | K^1_{u(m)} (x_m(x))R^1_m(x) - K^1_{u(m)}{'} (x_m(x))R^1_m(x) | \\
		& \leq \frac{\lambda_s(m)\lambda_u(m)}{1-\alpha(m) \beta'(u(m))} | K^1_{u(m)} (x_m(x)) - K^1_{u(m)}{'} (x_m(x)) |,
		\end{align*}
		completing the proof.
	\end{proof}
	Therefore there is a unique $ K^1 \in E^L_1 $ satisfying \eqref{mainX}. We will show $ Df_m(x) = K^1_m(x) $.
	Set
	\[
	Q(m, x', x) \triangleq f_m(x') - f_m(x) - K^1_m(x)(x' - x).
	\]

	\begin{slem}\label{slem:TS}
		$ \sup\limits_m \sup\limits_{x \in X_m} \limsup\limits_{x' \to x} \frac{|Q(m,x',x)|}{|x'-x|} < \infty $, and
		\[
		\sup_{x \in X_m} \limsup_{x' \to x} \dfrac{|Q(m,x',x)|}{|x'-x|} \leq \frac{\lambda_s(m)\lambda_u(m)}{1-\alpha(m) \beta'(u(m))} \sup_{x \in X_{u(m)}} \limsup_{x' \to x} \dfrac{|Q(u(m),x',x)|}{|x'-x|}.
		\]
	\end{slem}
	\begin{proof}
		The first inequality is apparently true. Let us consider the second.

		We write $ |g_m(x',x)| \leq O_m(1) $ if $ \limsup_{x' \to x}|g_m(x',x)| = 0 $.
		We compute
		\begin{align*}
		& |x_m(x') - x_m(x) - R^{1}_m(x)(x'-x)| \\
		= & \left| F_{m}(x', f_{u(m)}(x_m(x'))) - F_{m}(x, f_{u(m)}(x_m(x))) \right. \\
		& \qquad \left. - DF_m(x,f_{u(m)}(x_m(x))) \left\{ ( x'-x, f_{u(m)}(x_m(x')) - f_{u(m)}(x_m(x)) )\right\} \right. \\
		& \left. + D_2F_m(x,f_{u(m)}(x_m(x))) \left\lbrace  f_{u(m)}(x_m(x')) - f_{u(m)}(x_m(x)) \right.\right.\\
		& \qquad \left. \left. - K^1_{u(m)}(x_m(x)) R^1_m(x)(x' -x) \right\rbrace \right|   \\
		\leq & O_m(1)|x'-x| + \alpha(m) | Q(u(m),x_m(x'),x_m(x)) | \\
		& + \alpha(m) \beta'(u(m)) |x_m(x') - x_m(x) - R^{1}_m(x)(x'-x)|.
		\end{align*}
		Similarly,
		\begin{align*}
		|Q(m, x', x)| \leq & O_m(1)|x'-x| + \lambda_u(m)| Q(u(m),x_m(x'),x_m(x))| \\
		& + \lambda_u(m) \beta'(u(m)) |x_m(x') - x_m(x) - R^1_m(x)(x'-x)| \\
		\leq & O_m(1)|x'-x| + \frac{\lambda_u(m)}{1-\alpha(m) \beta'(u(m))} |Q( u(m), x_m(x'), x_m(x) )|.
		\end{align*}
		Thus,
		\begin{align*}
		& \sup_{x \in X_m} \limsup_{x' \to x} \dfrac{|Q(m,x',x)|}{|x'-x|} \\
		\leq & \frac{\lambda_u(m)}{1-\alpha(m) \beta'(u(m))} \sup_{x \in X_m} \limsup_{x' \to x} \frac{|Q( u(m), x_m(x'), x_m(x) )|}{| x_m(x') - x_m(x) |}\frac{|x_m(x')-x_m(x)|}{|x'-x|} \\
		\leq & \frac{\lambda_s(m)\lambda_u(m)}{1-\alpha(m) \beta'(u(m))} \sup_{x \in X_{u(m)}} \limsup_{x' \to x} \dfrac{|Q(u(m),x',x)|}{|x'-x|}.
		\end{align*}
		The proof is complete.
	\end{proof}
	Using the above sublemma and $ \theta \triangleq \sup_m \frac{\lambda_s(m) \lambda_u(m)}{1-\alpha(m)\beta'(u(m))} < 1 $, we have
	\begin{multline*}
	\sup_{m}\sup_{x \in X_m} \limsup_{x' \to x} \dfrac{|Q(m,x',x)|}{|x'-x|} \\
	\leq \sup_{m}\frac{\lambda_s(m)\lambda_u(m)}{1-\alpha(m) \beta'(u(m))} \sup_{x \in X_{u(m)}} \limsup_{x' \to x} \dfrac{|Q(u(m),x',x)|}{|x'-x|} \\
	\leq \theta \sup_{m}\sup_{x \in X_m} \limsup_{x' \to x} \dfrac{|Q(m,x',x)|}{|x'-x|} < \infty,
	\end{multline*}
	yielding $ \lim\limits_{x' \to x} \frac{|Q(m,x',x)|}{|x'-x|} = 0 $. The proof of \autoref{lem:leaf1} is finished.
\end{proof}

\begin{rmk}\label{00diff}
	A careful examination of the above proof shows that if $ F_m, G_m $ are only differentiable at $ (i_X(m), i_Y(u(m))) $, then $ f_m $ is differentiable at $ i_X(m) $; see also \autoref{generalized}.
\end{rmk}

\subsection{$ C^{1,\alpha} $ leaves: H\"{o}lderness of $ x \mapsto Df_m(x) = K^1_m(x) $}

\begin{lem}\label{lem:leaf1a}
	Under the conditions of \autoref{lem:leaf1}, assume in addition that
	\begin{enumerate}[(a)]
		\item $ DF_m(\cdot), DG_m(\cdot) $ are $ C^{0,\gamma} $ uniformly for $ m \in M $, i.e.,
		\[
		|DF_m(z_1) - DF_m(z_2)| \leq M_0 |z_1 - z_2|^{\gamma}, ~|DG_m(z_1) - DG_m(z_2)| \leq M_0 |z_1 - z_2|^{\gamma},
		\]
		for all $ z_1, z_2 \in X_m \times Y_{u(m)} $ and $ m \in M $, where $ 0 < \gamma \leq 1 $;
		\item (spectral gap condition) $ \lambda_s \lambda_u < 1 $, $ \lambda^{*\gamma\alpha}_s \lambda_s \lambda_u < 1 $, where $ 0 <\alpha \leq 1 $ (see \autoref{absgc}).
	\end{enumerate}
	Then $ K^1_m(\cdot) $ is $ C^{0,\gamma\alpha} $ uniformly for $ m $, i.e.,
	\[
	|K^1_m(x_1) - K^1_m(x_2)| \leq C |x_1 - x_2|^{\gamma\alpha},~\forall x_1, x_2 \in X_m.
	\]
	In particular, if $ F_m(\cdot), G_m(\cdot) \in C^{1,1} $ uniformly for $ m $, and $ \lambda_s\lambda_u < 1, \lambda^{2}_s \lambda_u < 1 $, then $ K^1_m(\cdot) \in C^{0,1} $ uniformly for $ m $.
\end{lem}
\begin{proof}
	Since $ K^1 $ satisfies \eqref{mainX}, we have
	\begin{align*}
	& |K^1_m(x_1) -K^1_m(x_2) |  \\
	 = & | DG_m( x_1, f_{u(m)}(x_m(x_1)) )( \id, K^1_{u(m)}(x_m(x_1))R^1_m(x_1) ) \\
	& -  DG_m( x_2, f_{u(m)}(x_m(x_2)) )( \id, K^1_{u(m)}(x_m(x_2))R^1_m(x_2)  ) | \\
	 \leq & |  ( DG_m( x_1, f_{u(m)}(x_m(x_1)) ) - DG_m( x_2, f_{u(m)}(x_m(x_2)) )  )( \id, K^1_{u(m)}(x_m(x_1))R^1_m(x_1) )|  \\
	& + | D_2G_m( x_2, f_{u(m)}(x_m(x_2)) ) ( K^1_{u(m)}(x_m(x_1))R^1_m(x_1) - K^1_{u(m)}(x_m(x_2))R^1_m(x_2) ) | \\
	\leq & \widetilde{C} |x_1 - x_2|^\gamma + \lambda_u(m) | K^1_{u(m)}(x_m(x_1))R^1_m(x_1) - K^1_{u(m)}(x_m(x_2))R^1_m(x_2) | \\
	\leq & \widetilde{C} |x_1 - x_2|^\gamma +  \lambda_s(m)\lambda_u(m) | K^1_{u(m)}(x_m(x_1)) - K^1_{u(m)}(x_m(x_2)) | \\
	& + \lambda_u(m)\beta'(u(m)) |R^1_m(x_1) - R^1_m(x_2)|.
	\end{align*}
	Similarly,
	\begin{multline*}
	|R^1_m(x_1) - R^1_m(x_2)| \leq \widetilde{C} |x_1 - x_2|^\gamma \\
	+ \alpha(m)\lambda_s(m) |K^1_{u(m)}(x_m(x_1)) - K^1_{u(m)}(x_m(x_2))|
	+ \alpha(m)\beta'(u(m)) |R^1_m(x_1) - R^1_m(x_2)|.
	\end{multline*}
	Thus,
	\begin{equation}\label{k11}
	|K^1_m(x_1) -K^1_m(x_2)| \leq \widetilde{C} |x_1 - x_2|^\gamma + \dfrac{\lambda_s(m) \lambda_u(m)}{1-\alpha(m)\beta'(u(m))} |K^1_{u(m)}(x_m(x_1)) - K^1_{u(m)}(x_m(x_2))|.
	\end{equation}
	Now using the argument in \autoref{Appbb} (see \autoref{argumentapp} \eqref{lem1}), we obtain the result.
\end{proof}

\subsection{$ C^k $ leaves: higher order smoothness of $ x \mapsto f_m(x) $}
\begin{lem}\label{lem:leafk}
	Under the assumptions of \autoref{lem:leaf1}, further assume $ F $, $ G $ are $ C^k $-fiber and uniformly $ C^{k-1,1} $-fiber (see \autoref{lip maps}) and $ \lambda^k_s \lambda_u < 1 $. Then $ f $ is $ C^k $-fiber and uniformly $ C^{k-1,1} $-fiber.
\end{lem}
\begin{proof}(sketch).
	The proof is essentially the same as in the $ C^1 $-fiber case by induction (note that we also have $ \lambda^i_s \lambda_u < 1 $, $ i = 1, \ldots, k $). We give a sketch here. We use the notation $ L^k(Z_1, Z_2) \triangleq L(\underbrace{Z_1 \times \ldots \times Z_1}_{k}, Z_2) $ if $ Z_i $ are vector bundles over $ M_{i} $, $ i = 1,2 $. Assume \autoref{lem:leafk} holds for $ k-1 $ ($ k \geq 2 $). Taking the $ k $th order derivative of \eqref{mainF} with respect to $ x $ informally, one has
	\[
	\begin{cases}
	\begin{split}
	W^k_{1,m}(x) + D_2F_m( x, y ) \left\{ K^k_{u(m)} (x_m(x)) \right. & (R^1_m)^k(x) \\
	& + \left. K^1_{u(m)} (x_m(x)) R^k_m(x)  \right\} = R^k_m(x), 
	\end{split}\\
	\begin{split}
	W^k_{2,m}(x) + D_2G_m( x, y ) \left\{ K^k_{u(m)} (x_m(x)) \right. & (R^1_m)^k(x) \\
	& + \left. K^1_{u(m)} (x_m(x)) R^k_m(x)  \right\} = K^k_m(x),
	\end{split}
	\end{cases} \tag{\dag}
	\]
	where $ y = f_{u(m)}( x_m(x) ) $, $ K^k \in L^k(\Upsilon_X^V, \Upsilon_Y^V) $ over $ f $, $ R^{k} \in L^k(\Upsilon_X^V, \Upsilon_X^V) $ over $ x_{(\cdot)}(\cdot) $, and $ (R^1_m)^k(x) = (R^1_m(x), \cdots, R^1_m(x)) $ ($ k $ components). 
	$ W^k_{i,m} $, $ i = 1,2 $, consist of a finite sum of terms which can be explicitly calculated with the help of the Fa\`{a} di Bruno formula (see e.g. \cite{MR09a, FdlLM06}); the non-constant factors are $ D^j F_m $, $ D^j G_m $ ($ 1 \leq j \leq k $), $ D^j f_{m} $ ($ 1 \leq j < k $), $ D^j x_m $ ($ 1 \leq j < k $). Note that $ W^k_{i,m} $, $ i = 1,2 $, are bounded uniformly for $ m \in M $ by induction. Since $ \lambda^k_s \lambda_u < 1 $, there exists a unique $ K^k \in L^k(\Upsilon_X^V, \Upsilon_Y^V) $ over $ f $ satisfying ($\dag $) (see \autoref{slem:TT} for a similar proof). One can further show $ D^k f_m (x) = K^k_m(x) $ by an analogous argument to that in \autoref{slem:TS}.
\end{proof}

\section{H\"{o}lder vertical parts: H\"{o}lderness of $ m \mapsto f_m(x) $} \label{hverticalsheaf}

In order to make sense of the H\"{o}lder continuity of $ m \mapsto f_m(x) $, we need some (uniform) assumptions on the bundles $ X, Y $ and the base space $ M $. The natural settings are (H1), (H2a), (H3), (H4a), (H5).

Regard $ F: X \otimes_u Y \to X $ as a bundle map over $ u $ and $ G: X \otimes_u Y \to Y $ as a bundle map over $ \id $; also $ f: X \to Y $ over $ \id $, and $ x_{(\cdot)}(\cdot) : X \to X $ over $ u $.

Consider the local representations of $ F, G $, $ f, x_{(\cdot)}(\cdot) $ at $ m_0 \in M_1 $ with respect to $ \mathcal{A}, \mathcal{B} $, i.e.,
\begin{equation}\label{localrepre}
\begin{cases}
\begin{split}
\widehat{F}_{m_0} (m, x, y) \triangleq (\varphi^{u(m_0)}_{u(m)})^{-1} \circ & F_m \circ ( \varphi^{m_0}_m(x) , \phi^{u(m_0)}_{u(m)}(y) ) : \\
& U_{m_0} (\mu^{-1}(m_0) \varepsilon_1) \times X_{m_0} \times Y_{u(m_0)} \to X_{u(m_0)} ,
\end{split}\\
\begin{split}
\widehat{G}_{m_0} (m, x, y) \triangleq (\phi^{m_0}_{m})^{-1} \circ G_m \circ & ( \varphi^{m_0}_m(x) , \phi^{u(m_0)}_{u(m)}(y) ) : \\
&U_{m_0} (\mu^{-1}(m_0) \varepsilon_1) \times X_{m_0} \times Y_{u(m_0)} \to Y_{m_0} ,
\end{split}\\
\widehat{f}_{m_0}(m, x) \triangleq (\phi^{m_0}_m)^{-1} \circ f_m \circ \varphi^{m_0}_m(x) :  U_{m_0} (\mu^{-1}(m_0) \varepsilon_1) \times X_{m_0} \to Y_{m_0} , \\
\widehat{x}_{m_0}(m, x) \triangleq (\varphi^{u(m_0)}_{u(m)})^{-1} \circ x_{m} \circ \varphi^{m_0}_m(x) : U_{m_0} (\mu^{-1}(m_0) \varepsilon_1) \times X_{m_0} \to X_{u(m_0)}.
\end{cases}
\end{equation}
Then
\begin{equation}\label{mainLocal}
	\begin{cases}
		\widehat{F}_{m_0} ( m, x, \widehat{f}_{u(m_0)} (u(m), \widehat{x}_{m_0} (m,x) ) ) = \widehat{x}_{m_0} (m, x), \\
		\widehat{G}_{m_0} ( m, x, \widehat{f}_{u(m_0)} (u(m), \widehat{x}_{m_0} (m,x) ) ) = \widehat{f}_{m_0} (m, x),
	\end{cases}
\end{equation}
for $ m \in U_{m_0} (\hat{\mu}^{-2} \varepsilon_1) $, $ x \in X_{m_0} $, $ m_0 \in M_1 $.

If $ i $ is an invariant section of $ H $ and also a $ 0 $-section of $ X \times Y $, then $ \widehat{f}_{m_0} (m, i_X(m_0)) = i_Y(m_0) $, $ \widehat{x}_{m_0} (m, i_X(m_0)) = i_X(u(m_0)) $, and also
\[
\begin{cases}
	\widehat{F}_{m_0} (m, i_X(m_0), i_Y(u(m_0))) = i_X(u(m_0)), \\
	\widehat{G}_{m_0} (m, i_X(m_0), i_Y(u(m_0))) = i_Y(m_0).
\end{cases}
\]
Hereafter, if $ m_1 \in U_{m_0} $, then $ |m_1 - m_0| $ means the distance between $ m_1, m_0 $ in $ U_{m_0} $, i.e., $ |m_1 - m_0| = d_{m_0}(m_1, m_0) $, where $ d_{m_0} $ is the metric in $ U_{m_0} $ (see (H1) (i)).

\begin{lem}\label{lem:sheaf}
	Assume the following conditions hold:
	\begin{enumerate}[(a)]
		\item Let \textnormal{(H1a), (H2a), (H3), (H4a), (H5)} hold.
		\item $ F, G $ satisfy
		\begin{equation}\label{cc2}
		\begin{cases}
		|\widehat{F}_{m_0}(m_1, z) - \widehat{F}_{m_0}(m_0, z) | \leq M_0 |m_1 - m_0|^\gamma |z|^\zeta, \\
		|\widehat{G}_{m_0}(m_1, z) - \widehat{G}_{m_0}(m_0, z)| \leq M_0 |m_1 - m_0|^\gamma |z|^\zeta,
		\end{cases}
		\end{equation}
		for all $ m_1 \in U_{m_0} (\mu^{-1}(m_0) \varepsilon_1), z \in X_{m_0} \times Y_{u(m_0)} $, $ m_0 \in M_1 $, where $ 0< \gamma \leq 1 $, $ \zeta \geq 0 $.
		\item (spectral gap condition) $ \lambda_s \lambda_u < 1 $, $  (\max\{ \lambda^{\zeta - 1}_s, 1 \} \mu^{\gamma})^{*\alpha} \lambda_s \lambda_u < 1 $, where $ 0 < \alpha \leq 1 $; see \autoref{absgc}.
	\end{enumerate}
	If $ \varepsilon^*_1 \leq \hat{\mu}^{-2} \varepsilon_1 $ is small, then
	\begin{equation}\label{cc3}
	|\widehat{f}_{m_0}(m_1, x) - \widehat{f}_{m_0}(m_0, x)| \leq C |m_1 - m_0|^{\gamma \alpha} |x|^{\zeta \alpha + 1 - \alpha},
	\end{equation}
	for all $ m_1 \in U_{m_0} (\varepsilon^*_1), x \in X_{m_0}, m_0 \in M_1 $ with 
	\[
	|m_1 - m_0|^\gamma |x|^\zeta \leq \hat{r}\min\{ |x|, |x|^{\zeta c -(c-1)} \},
	\]
	where the constant $ C $ depends on the constant $ \hat{r} > 0 $ but not on $ m_0 \in M_1 $, $ c > 1 $ and $ \hat{r} $ does not depend on $ m_0 \in M_1 $.
	In particular, if $ \gamma = \zeta = 1 $ and $ \lambda_s \lambda_u < 1 $, $ \lambda_s \lambda_u  \mu  < 1 $, then
	\[
	|\widehat{f}_{m_0}(m_1, x) - \widehat{f}_{m_0}(m_0, x)| \leq C |m_1 - m_0| |x|.
	\]
\end{lem}
See also \autoref{lipbase} and \autoref{rmk:holder} \eqref{h:cc}.

\begin{rmk}\label{holsheaf}
	Consider some cases where the above condition on $ F, G $ can be satisfied.
	First, under (H1c), (H2d), (H3), if \eqref{ligFG} holds and 
	\[
	\begin{gathered}
	|\nabla_{m} \varphi^{m_0}_m(x)| \leq M_0 |x|^{\zeta}, |\nabla_{m} \phi^{m_0}_m(x)| \leq M_0 |x|^{\zeta}, \\
	|\nabla_{m}F_{m}(z)| \leq M_0|z|^{\zeta}, |\nabla_{m}G_{m}(z)| \leq M_0|z|^{\zeta},
	\end{gathered}
	\]
	for $ x \in X_{m_0} $, $ z \in X_{m} \times Y_{u(m)} $, $ m \in U_{m_0}(\varepsilon_{1}) $, $ m_0 \in M_1 $,
	then the estimates \eqref{cc2} hold with $ \gamma = 1 $. This is quite easy.
	It suffices to consider $ F $. Let $ m_0 \in M_1 $, $ \chi_{m_0}(m_1) = x_1 \in T_{m_0}M $. We compute
	\begin{align*}
	& |\widehat{F}_{m_0}(m_1, z) - \widehat{F}_{m_0}(m_0, z)| \\
	= & |\widehat{F}_{m_0}(\chi^{-1}_{m_0}(x_1), z) - \widehat{F}_{m_0}(\chi^{-1}_{m_0}(0), z)|
	\leq \sup_{v}|D_v \widehat{F}_{m_0}(\chi^{-1}_{m_0}(v),z)| |x_1 - x_0|_{m_0}  \\
	\leq & \sup_{v}|D_m\widehat{F}_{m_0}(\chi^{-1}_{m_0}(v),z)D\chi^{-1}_{m_0}(v)| |x_1 - x_0|_{m_0}
	\leq \widetilde{C}_0 |x_1 - x_0|_{m_0} |z|^{\zeta} \\
	\leq & \widetilde{C}_0 |m_1 - m_0||z|^{\zeta},
	\end{align*}
	where \eqref{kk1} (see below) is used.
	\begin{enumerate}[(a)]
		\item If $ F, G $ are uniformly (locally) $ C^{0,1} $ (i.e. estimates \eqref{ligFG} hold), then \eqref{cc2} is satisfied for $ \gamma = 1 $, $ \zeta = 0 $. If $ \lambda_s < 1 $, then the spectral gap condition reads: $ \lambda_s \lambda_u < 1 $, $  (\frac{\mu}{\lambda_s})^{*\alpha} \lambda_s \lambda_u< 1  $. This case was also discussed in e.g. \cite{Sta99, Cha08, Wil13}.

		\item A well known case is the following: If $ D_m\widehat{F}_{m_0}(m, \cdot) $, $ D_m\widehat{G}_{m_0}(m, \cdot) \in C^{0,\beta} $ uniformly for $ m_0, m $, then the estimates \eqref{cc2} hold for $ \gamma = 1 $, $ \zeta = \beta $. In particular, in this case with $ \beta = 1 $ and the spectral gap condition: $ \lambda_s \lambda_u < 1 $, $\mu^{*\alpha} \lambda_s \lambda_u  < 1  $, the uniform $ \alpha $-H\"{o}lder continuity of $ m \mapsto f_m(x) $ is obtained. We mention that if $ \lambda_s < 1 $, the spectral gap condition in this case is better than (a); that is, \emph{the higher the regularity of $ F, G $, the better the spectral gap condition}.
	\end{enumerate}
\end{rmk}

\begin{rmk}\label{h3*}
	When $ \zeta = 0 $, the assumption (H3) can be replaced by the assumption that $ i $ is Lipschitz around $ M_1 $ with respect to $ \mathcal{A} \times \mathcal{B} $ in the following sense:
	\begin{enumerate}
		\item [$ \mathrm{(B3)} $] Let
		\[
		\widehat{i}^{m_0}(m) = (\widehat{i}^{m_0}_X(m), \widehat{i}^{m_0}_Y(m)) = ((\varphi^{m_0}_m)^{-1}(i_X(m)), (\phi^{m_0}_m)^{-1}(i_Y(m)) ).
		\]
		Then $ m \mapsto \widehat{i}^{m_0}(m): U_{m_0}(\epsilon_0) \to X_{m_0} \times Y_{m_0}  $ is Lipschitz with Lipschitz constant no more than $ c_0 $ for a fixed $ \epsilon_0 $ and every $ m_0 \in M_1 $, where $ c_0 $ is independent of $ m_0 $.
	\end{enumerate}
	Assumption $ \mathrm{(H3)} $ actually means that $ i $ is Lipschitz with Lipschitz constant $ 0 $, i.e., $ i $ is locally constant.
	Now under the assumptions of \autoref{lem:sheaf} with $ \mathrm{(H3)} $ replaced by $ \mathrm{(B3)} $ and $ \zeta = 0 $, we also find that $ f $ is H\"older respecting the base points. For a proof, see \autoref{argumentapp} \eqref{lem2}.
	See also \cite{Cha08} for the same result in the trivial bundle case.
\end{rmk}

\begin{proof}[Proof of \autoref{lem:sheaf}]
	Take $ m_0 \in M_1 $.
	First note that by (H5), (H2a), if $ m \in U_{m_0} (\hat{\mu}^{-2} \varepsilon_1) $ and $ \varepsilon_1 $ is small,  then we can choose $ \alpha'', \beta'', \lambda''_s, \lambda''_u $, such that $ 1-\alpha''(m)\beta''(u(m)) > 0 $, and
	\[
	\lambda''_s \lambda''_u < 1,~  (\max\{ (\lambda''_s)^{\zeta - 1}, 1 \} \mu^{\gamma})^{ *\alpha }  \lambda''_s \lambda''_u < 1,
	\]
	and the following inequalities hold:
	\begin{equation}\label{AB''}
	\begin{cases}
		\lip \widehat{F}_{m_0}(m, x, \cdot) \leq (1+\eta(d(u(m),u(m_0)))) \alpha(m) (1+\eta(d(u(m),u(m_0)))) \leq \alpha''(m_0),\\
		\lip \widehat{G}_{m_0} (m,x,\cdot) \leq (1+\eta(d( m,m_0   ))) \lambda_u(m) (1+\eta(d(u(m),u(m_0)))) \leq \lambda''_u(m_0), \\
		\lip \widehat{G}_{m_0} (m,\cdot, y) \leq (1+\eta(d( m,m_0   ))) \beta'(m) (1+\eta(d( m,m_0   ))) \leq \beta''(m_0),\\
		\lip \widehat{F}_{m_0}(m,\cdot,y) \leq (1+\eta(d(u(m),u(m_0)))) \lambda_s(m) (1+\eta(d( m,m_0   ))) \leq \lambda''_s(m_0), \\
		\lip \widehat{f}_{m_0}(m, \cdot) \leq (1+\eta(d( m,m_0 ))) \beta'(m) (1+\eta(d( m,m_0   ))) \leq \beta''(m_0), \\
		\lip \widehat{x}_{m_0}(m, \cdot) \leq (1+\eta(d(u(m),u(m_0)))) \lambda_s(m) (1+\eta(d( m,m_0 ))) \leq \lambda''_s(m_0),
	\end{cases}
	\end{equation}
	where $ \eta(\cdot) $ is the function in the definition of uniform $ C^{0,1} $-fiber trivializations for both $ X, Y $ (see \autoref{def:lipbundle}).

	Note that as $ \widehat{f}_{m_0} (m, i(m_0)) = i(m_0) $, $ \widehat{x}_{m_0} (m, i(m_0)) = i(u(m_0)) $, we have
	\[
	|\widehat{f}_{m_0} (m, x)| \leq \beta''(m_0)|x|,~|\widehat{x}_{m_0} (m, x)| \leq \lambda''_s(m_0) |x| .
	\]
	Since $ \widehat{f}_{m_0}(m, x) $ satisfies \eqref{mainLocal}, we have
	\begin{align*}
	& |\widehat{f}_{m_0}(m_1,x) - \widehat{f}_{m_0}(m_0,x)| \\
	\leq & | \widehat{G}_{m_0} (m_1, x, \widehat{f}_{u(m_0)} (u(m_1), \widehat{x}_{m_0} (m_1,x) ) ) - \widehat{G}_{m_0} (m_0, x, \widehat{f}_{u(m_0)} (u(m_1), \widehat{x}_{m_0} (m_1,x) ) ) | \\
	& + | \widehat{G}_{m_0} (m_0, x, \widehat{f}_{u(m_0)} (u(m_1), \widehat{x}_{m_0} (m_1,x) ) ) - \widehat{G}_{m_0} (m_0, x, \widehat{f}_{u(m_0)} (u(m_0), \widehat{x}_{m_0} (m_0,x) ) ) | \\
	\leq & \widetilde{C} |m_1 - m_0|^\gamma|x|^\zeta + \lambda''_u(m_0) |  \widehat{f}_{u(m_0)} (u(m_1), \widehat{x}_{m_0} (m_1,x) ) - \widehat{f}_{u(m_0)} (u(m_0), \widehat{x}_{m_0} (m_0,x) ) | \\
	\leq & \widetilde{C} |m_1 - m_0|^\gamma|x|^\zeta +  \lambda''_u(m_0) |  \widehat{f}_{u(m_0)} (u(m_1), \widehat{x}_{m_0} (m_1,x) ) - \widehat{f}_{u(m_0)} (u(m_0), \widehat{x}_{m_0} (m_1,x) ) | \\
	& + \lambda''_u(m_0) \beta''(u(m_0)) | \widehat{x}_{m_0} (m_1,x) - \widehat{x}_{m_0} (m_0,x) |.
	\end{align*}
	Similarly,
	\begin{multline*}
	| \widehat{x}_{m_0} (m_1,x) - \widehat{x}_{m_0} (m_0,x) |
	\leq   \widetilde{C} |m_1 - m_0|^\gamma|x|^\zeta + \alpha''(m_0) |  \widehat{f}_{u(m_0)} (u(m_1), \widehat{x}_{m_0} (m_1,x) ) - \\
	\widehat{f}_{u(m_0)} (u(m_0), \widehat{x}_{m_0} (m_1,x) ) |
	+ \alpha''(m_0) \beta''(u(m_0)) | \widehat{x}_{m_0} (m_1,x) - \widehat{x}_{m_0} (m_0,x) |.
	\end{multline*}
	Thus,
	\begin{multline}\label{basef}
	|\widehat{f}_{m_0}(m_1,x) - \widehat{f}_{m_0}(m_0,x)|
	\leq \widetilde{C} |m_1 - m_0|^\gamma|x|^\zeta \\
	+ \frac{\lambda''_u(m_0)}{1-\alpha''(m_0)\beta''(u(m_0))} |  \widehat{f}_{u(m_0)} (u(m_1), \widehat{x}_{m_0} (m_1,x) ) - \widehat{f}_{u(m_0)} (u(m_0), \widehat{x}_{m_0} (m_1,x) ) |.
	\end{multline}
	Using the argument in \autoref{Appbb} (see \autoref{argumentapp} \eqref{lem2}), we get the results.
\end{proof}

\section{H\"{o}lder distributions: H\"{o}lderness of $ m \mapsto Df_m(x) = K^1_m(x) $}
(Based on the H\"{o}lder continuity of $ m \mapsto f_m(x) $ and $ x \mapsto K^1_m(x) $.)

Let $ K^1 \in L(\Upsilon_X^V, \Upsilon_Y^V) $ be the bundle map over $ f $ obtained in \autoref{lem:leaf1}.
In order to make sense of the H\"{o}lder continuity of $ m \mapsto Df_m(x) = K^1_m(x) $, we need (H2b) instead of (H2a).
Now taking the derivative of \eqref{mainLocal} with respect to $ x $, we have
\begin{equation}\label{localleaf1}
\begin{cases}
	D_z\widehat{F}_{m_0} ( m, z ) ( \id,
	\widehat{K}^1_{u(m_0)} (u(m), \widehat{x}_{m_0} (m,x) ) \widehat{R}^1_{m_0} (m,x) ) = \widehat{R}^1_{m_0} (m,x), \\
	D_z\widehat{G}_{m_0} ( m, z ) ( \id,
	\widehat{K}^1_{u(m_0)} (u(m), \widehat{x}_{m_0} (m,x) ) \widehat{R}^1_{m_0} (m,x) ) = \widehat{K}^1_{m_0} (m,x),
\end{cases}
\end{equation}
where $ z = (x, \widehat{f}_{u(m_0)} (u(m), \widehat{x}_{m_0} (m,x) ) ) $, and $ \widehat{K}^1 $, $ \widehat{R}^1 $ are the local representations of $ K^1, R^1 $ with respect to $ \mathcal{A}, \mathcal{B} $, i.e.,
\begin{equation}\label{K1R1}
\begin{cases}
\begin{split}
\widehat{K}^1_{m_0} (m, x) \triangleq ( D \phi^{m_0}_{m} )^{-1} ( f_m( & \varphi^{m_0}_{m}(x) ) ) K^1_m( \varphi^{m_0}_{m}(x) ) D\varphi^{m_0}_{m}(x): \\
& U_{m_0}(\mu^{-1}(m_0)\varepsilon_1) \times X_{m_0} \to L (X_{m_0}, Y_{m_0}) ,
\end{split}\\
\begin{split}
\widehat{R}^1_{m_0} (m, x) \triangleq ( D\varphi^{u(m_0)}_{u(m)} )^{-1} ( x_m( &  \varphi^{m_0}_{m}(x) ) ) R^1_m(\varphi^{m_0}_{m}(x)) D\varphi^{m_0}_{m}(x): \\
& U_{m_0}(\mu^{-1}(m_0)\varepsilon_1) \times X_{m_0} \to L (X_{m_0}, X_{u(m_0)}).
\end{split}
\end{cases}
\end{equation}
Note that $ \widehat{K}^1_{m_0} (m, x) =
D_x \widehat{f}_{m_0} ( m, x ) $, $ \widehat{R}^1_{m_0} (m, x) = D_x \widehat{x}_{m_0} (m,x) $.

First consider a special case, i.e., $ m \mapsto K^1_{m}(i_X(m)) $.
\begin{lem}\label{lem:base0}
	Let the conditions in \autoref{lem:leaf1} hold. In addition, assume the following:
	\begin{enumerate}[(a)]
		\item Let \textnormal{(H1a), (H2b), (H3), (H4a), (H5)} hold.

		\item $ DF_m(i_X(m), i_Y(u(m))) $, $ DG_m(i_X(m), i_Y(u(m))) $ are uniformly (locally) $ \gamma $-H\"{o}lder \linebreak around $ M_1 $ in the following sense:
		\begin{align*}
		| D_z\widehat{F}_{m_0}(m_1, i^{1}(m_0)) - D_z\widehat{F}_{m_0}(m_0, i^{1}(m_0)) | \leq M_0 |m_1 - m_0|^\gamma, \\
		| D_z\widehat{G}_{m_0}(m_1, i^{1}(m_0)) - D_z\widehat{G}_{m_0}(m_0, i^{1}(m_0)) | \leq M_0 |m_1 - m_0|^\gamma,
		\end{align*}
		for all $ m_1 \in U_{m_0} (\mu^{-1}(m_0) \varepsilon_1) $, $ m_0 \in M_1 $, where $ i^{1}(m_0) = (i_X(m_0), i_Y(u(m_0))) $ and $ 0 < \gamma \leq 1 $.

		\item (spectral gap condition) $ \lambda_s \lambda_u < 1 $, $  \mu^{* \gamma \alpha} \lambda_s \lambda_u < 1 $, where $ 0 < \alpha \leq 1 $; see \autoref{absgc}.

	\end{enumerate}
	If $ \varepsilon^*_1 \leq \hat{\mu}^{-2} \varepsilon_1 $ is small, then 
	\begin{equation}\label{k1optimal}
	| \widehat{K}^1_{m_0} (m_1, i_X{(m_0)}) - \widehat{K}^1_{m_0} (m_0, i_X{(m_0)}) | \leq C |m_1 - m_0|^{\gamma\alpha},
	\end{equation}
	for all $ m_1 \in U_{m_0} (\varepsilon^*_1) $, $ m_0 \in M_1 $.

\end{lem}

\begin{lem}\label{lem:baseleaf}
	Let the conditions in \autoref{lem:leaf1} hold. In addition, assume the following:
	\begin{enumerate}[(a)]
		\item Let \textnormal{(H1a), (H2b), (H3), (H4a), (H5)} hold.

		\item $ DF, DG $ are $ C^{0,1} $ around $ M_1 $, i.e., the estimates \eqref{lipzz} hold.
		Moreover, the estimates \eqref{cc2} in \autoref{lem:sheaf} hold for $ \gamma = \zeta = 1 $; see also \autoref{holsheaf}.

		\item (spectral gap condition) $ \lambda_s \lambda_u < 1 $, $ \lambda^2_s \lambda_u \mu^\alpha < 1 $, $ \lambda^{*\beta}_s \lambda_s \lambda_u < 1 $, $ \mu^{*\alpha} \lambda_s \lambda_u < 1 $, where $ 0 < \alpha, \beta \leq 1 $; see \autoref{absgc}.
	\end{enumerate}
	If $ \varepsilon^*_1 \leq \hat{\mu}^{-2} \varepsilon_1 $ is small, then 
	\begin{equation}\label{bleaf}
	| \widehat{K}^1_{m_0} (m_1, x) - \widehat{K}^1_{m_0} (m_0, x) | \leq C\left\lbrace |m_1 - m_0|^\alpha (|x|+1) + (|m_1 - m_0|^\alpha |x|)^\beta \right\rbrace ,
	\end{equation}
	for all $ m_1 \in U_{m_0} (\varepsilon^*_1), x \in X_{m_0} $, $ m_0 \in M_1 $, where the constant $ C $ depends on the constant $ \varepsilon^*_1 $ but not on $ m \in M $.
\end{lem}

One can discuss the case where $ \widehat{F}_{m_0}(\cdot), \widehat{G}_{m_0}(\cdot) \in C^{1,\gamma} $, but it is more complicated.

\begin{proof}[Proof of \autoref{lem:base0} and \autoref{lem:baseleaf}]
	As in the proof of \autoref{lem:sheaf}, when we let $ \varepsilon_1 $ be small and $ m \in U_{m_0} (\hat{\mu}^{-2} \varepsilon_1) $, we can choose $ \alpha'', \beta'', \lambda''_s, \lambda''_u $, such that
	the following spectral gap condition holds (see \autoref{absgc} for the actual meaning where $ \alpha'', \beta'' $ replace $ \alpha, \beta' $):
	\[
	\lambda''_s \lambda''_u < 1,  (\lambda''_s)^2 \lambda''_u \mu^\alpha < 1 ,  (\lambda''_s)^{*\beta} \lambda''_s \lambda''_u < 1, \mu^{*\alpha} \lambda''_s \lambda''_u  < 1,
	\]
	\eqref{AB''} holds, and in addition
	\begin{equation}\label{kkr}
	\begin{cases}
	\sup_{x \in X_m} | \widehat{K}^1_{m_0}(m, x) | \leq (1+\eta(d( m,m_0 ))) \beta'(m) (1+\eta(d( m,m_0   ))) \leq \beta''(m_0), \\
	\sup_{x \in X_m} | \widehat{R}^1_{m_0}(m, x) | \leq (1+\eta(d(u(m),u(m_0)))) \lambda_s(m) (1+\eta(d( m,m_0 ))) \leq \lambda''_s(m_0).
	\end{cases}
	\end{equation}

	Let $ y_1 = \widehat{f}_{u(m_0)} (u(m_1), \widehat{x}_{m_0} (m_1,x) ), ~ y_2 = \widehat{f}_{u(m_0)} (u(m_0), \widehat{x}_{m_0} (m_0,x) ) $, $ \Delta_1 = |y_1 - y_2| $. By \eqref{localleaf1}, 
	\begin{align*}
	& | \widehat{K}^1_{m_0} (m_1, x) - \widehat{K}^1_{m_0} (m_0, x) |  \\
	\leq & | (  D_z\widehat{G}_{m_0} ( m_1, x, y_1 ) - D_z\widehat{G}_{m_0} ( m_1, x, y_2 ) ) ( \id, \widehat{K}^1_{u(m_0)} (u(m_1), \widehat{x}_{m_0} (m_1,x) ) \widehat{R}^1_{m_0} (m_1,x) ) | \\
	& + | ( D_z\widehat{G}_{m_0} ( m_1, x, y_2 ) - D_z\widehat{G}_{m_0} ( m_0, x, y_2 ) ) \\
	& \qquad ( \id, \widehat{K}^1_{u(m_0)} (u(m_1), \widehat{x}_{m_0} (m_1,x) ) \widehat{R}^1_{m_0} (m_1,x) ) | \\
	& + | D_y\widehat{G}_{m_0} ( m_0, x, y_2 ) (\widehat{K}^1_{u(m_0)} (u(m_1), \widehat{x}_{m_0} (m_1,x) ) \widehat{R}^1_{m_0} (m_1,x) \\
	& \qquad - \widehat{K}^1_{u(m_0)} (u(m_0), \widehat{x}_{m_0} (m_0,x) ) \widehat{R}^1_{m_0} (m_0,x)) | \\
	\leq & \widetilde{C}\Delta_1 + \widetilde{C} |m_1 - m_0| + \lambda''_s(m_0)\lambda''_u(m_0) |\widehat{K}^1_{u(m_0)} (u(m_1), \widehat{x}_{m_0} (m_1,x) ) \\
	& \qquad - \widehat{K}^1_{u(m_0)} (u(m_0), \widehat{x}_{m_0} (m_0,x) ) |
	 + \lambda''_u(m_0)\beta''(u(m_0)) | \widehat{R}^1_{m_0} (m_1,x) - \widehat{R}^1_{m_0} (m_0,x) |.
	\end{align*}
	Similarly,
	\begin{multline*}
	|\widehat{R}^1_{m_0} (m_1,x) - \widehat{R}^1_{m_0} (m_0,x)|
	\leq \widetilde{C}\Delta_1 + \widetilde{C} |m_1 - m_0| \\
	+ \lambda''_s(m_0)\alpha''(m_0) \left|\widehat{K}^1_{u(m_0)} (u(m_1), \widehat{x}_{m_0} (m_1,x) ) - \widehat{K}^1_{u(m_0)} (u(m_0), \widehat{x}_{m_0} (m_0,x) ) \right| \\
	+ \alpha''(m_0)\beta''(u(m_0)) | \widehat{R}^1_{m_0} (m_1,x) - \widehat{R}^1_{m_0} (m_0,x) |.
	\end{multline*}
	Now we have
	\begin{multline*}
	| \widehat{K}^1_{m_0} (m_1, x) - \widehat{K}^1_{m_0} (m_0, x) |
	\leq \widetilde{C}\Delta_1 + \widetilde{C} |m_1 - m_0| \\
	+ \theta''_1(m_0) |\widehat{K}^1_{u(m_0)} (u(m_1), \widehat{x}_{m_0} (m_1,x) ) - \widehat{K}^1_{u(m_0)} (u(m_0), \widehat{x}_{m_0} (m_0,x) ) |,
	\end{multline*}
	where $ \theta''_1(m_0) =  \frac{\lambda''_s(m_0)\lambda''_u(m_0)}{1-\alpha''(m_0)\beta''(u(m_0))} < 1 $.

	For the proof of \autoref{lem:base0}, the estimate becomes
	\begin{multline}\label{K110}
	| \widehat{K}^1_{m_0} (m_1, i_X(m_0)) - \widehat{K}^1_{m_0} (m_0, i_X(m_0)) | \\
	\leq \widetilde{C} |m_1 - m_0|^\gamma + \theta''_1(m_0) |\widehat{K}^1_{u(m_0)} (u(m_1), i_X(m_0) ) - \widehat{K}^1_{u(m_0)} (u(m_0), i_X(m_0) ) |.
	\end{multline}
	Using the argument in \autoref{Appbb} (see \autoref{argumentapp} \eqref{lem3}), one can obtain \autoref{lem:base0}.

	For the general case, we need the H\"{o}lder continuity of $ m \mapsto f_m(x) $ and $ x \mapsto K^1_m(x) $, i.e., \autoref{lem:sheaf} and \autoref{lem:leaf1a}. There is a constant $ \varepsilon'_1 > 0 $ such that (a), (b) below hold.

	(a) By \autoref{lem:leaf1a} and (H2b), we have
	\[
	| \widehat{K}^1_{m_0} (m, x_1) - \widehat{K}^1_{m_0} (m, x_2) | \leq \widetilde{C} |x_1 - x_2|^\beta,~ | \widehat{R}^1_{m_0} (m, x_1) - \widehat{R}^1_{m_0} (m_, x_2) | \leq \widetilde{C} |x_1 - x_2|^\beta,
	\]
	for $ m \in U_{m_0} (\varepsilon'_1) $, $ x_1, x_2 \in X_{m_0} $, where the constant $ \widetilde{C} $ only depends on $ \varepsilon'_1 $.

	(b) Meanwhile, by \autoref{lem:sheaf},
	\[
	|\widehat{f}_{m_0}(m_1, x) - \widehat{f}_{m_0}(m_0, x)| \leq \widetilde{C} |m_1 - m_0|^{\alpha} |x|,~|\widehat{x}_{m_0}(m_1, x) - \widehat{x}_{m_0}(m_0, x)| \leq \widetilde{C} |m_1 - m_0|^{\alpha} |x|,
	\]
	for $ m_1 \in U_{m_0} (\varepsilon'_1), x \in X_{m_0} $, where the constant $ \widetilde{C} $ only depends on $ \varepsilon'_1 $.
	Now we see that
	\begin{multline}\label{K111}
	| \widehat{K}^1_{m_0} (m_1, x) - \widehat{K}^1_{m_0} (m_0, x) |
	\leq \widetilde{C} \left\{  |m_1 - m_0|^\alpha |x| + (|m_1 - m_0|^{\alpha}|x|)^{\beta} \right\} + \widetilde{C} |m_1 - m_0| \\
	+ \theta''_1(m_0) |\widehat{K}^1_{u(m_0)} (u(m_1), \widehat{x}_{m_0} (m_1,x) ) - \widehat{K}^1_{u(m_0)} (u(m_0), \widehat{x}_{m_0} (m_1,x) ) |.
	\end{multline}
	Apply the argument in \autoref{Appbb} (see \autoref{argumentapp} \eqref{lem3}) to obtain the result.
\end{proof}

\section{Smoothness of $ m \mapsto f_m(x) $ and H\"{o}lderness of $ x \mapsto \nabla_mf_m(x) = K_m(x) $} \label{derivativeBase}

\subsection{Smoothness of $ m \mapsto f_m(x) $}
(Based on the $ C^1 $ smoothness of $ f_m(\cdot) $ and the Lipschitz continuity of $ m \mapsto f_m(x) $.)

In the following, we need some smoothness properties of the bundles $ X, Y $ in order to make sense of the smoothness of $ m \mapsto f_m(x) $. The natural assumption is that $ X, Y $ are $ C^1 $ bundles with $ C^0 $ \emph{connections} (see \autoref{connections}).
Now we can interpret \eqref{mainM} precisely by using the notion of \emph{covariant derivative} (see \autoref{defi:coderivative}), i.e.,
\begin{equation}\label{mainMFW}
\begin{cases}
\begin{split}
\nabla_mF_m( x, y ) + D_2F_m(x, y) ( & K_{u(m)}(x_m(x))Du(m) \\
& + Df_{u(m)} (x_m(x)) R_m(x) ) = R_m(x), 
\end{split}\\
\begin{split}
\nabla_mG_m( x, y ) + D_2G_m(x, y) ( & K_{u(m)}(x_m(x))Du(m) \\
& + Df_{u(m)} (x_m(x)) R_m(x) ) = K_m(x),
\end{split}
\end{cases}
\end{equation}
where $ y = f_{u(m)}(x_m(x)) $, $ K \in L(\Upsilon^H_X, \Upsilon^V_Y) $ is a vector bundle map over $ f $, and $ R \in L(\Upsilon^H_X, \Upsilon^V_X ) $ is a vector bundle map over $ x_{(\cdot)}(\cdot) $. Hereafter for $ K \in L(\Upsilon^H_X, \Upsilon^V_Y) $ (over $ x_{(\cdot)}(\cdot) $), we write $ K_{m} (x) = K_{(m,x)} $.

\begin{lem}\label{smoothbase}
	Assume the following conditions hold:
	\begin{enumerate}[(a)]
		\item Let \textnormal{(H1b), (H2c), (H3), (H4b), (H5)} hold.
		\item $ F, G $ are $ C^1 $ and
		\begin{equation}\label{ccc0}
		\begin{cases}
		|\widehat{F}_{m_0}(m_1, z) - \widehat{F}_{m_0}(m_0, z) | \leq M_0 |m_1 - m_0| |z|,\\
		|\widehat{G}_{m_0}(m_1, z) - \widehat{G}_{m_0}(m_0, z)| \leq M_0 |m_1 - m_0| |z|,
		\end{cases}
		\end{equation}
		for all $ m_1 \in U_{m_0} (\mu^{-1}(m_0) \varepsilon_1), z \in X_{m_0} \times Y_{u(m_0)} $, $ m_0 \in M_1 $.
		\item (spectral gap condition) $ \lambda_s \lambda_u < 1 $, $ \lambda_s \lambda_u \mu < 1 $; see also \autoref{absgc}.
	\end{enumerate}
	Then the following conclusions hold:
	\begin{enumerate}[(1)]
		\item There exists a unique $ C^0 $ vector bundle map $ K \in L(\Upsilon^H_X, \Upsilon^V_Y)  $ over $ f $ such that
		\[
		K_m(i_X(m)) = 0, ~\text{for}~ m \in M_1, ~\sup_{m \in M_1}\sup_{x \neq i_X(m)} \frac{|K_m(x)|}{|x|} < \infty,
		\]
		and it satisfies \eqref{mainMFW}. In addition, if
		\[
		\nabla_m F_m (i_X(m), i_Y(u(m))) = 0 ,~\nabla_m G_m (i_X(m), i_Y(u(m))) = 0,
		\]
		for all $ m \in M $, then $ K_m(i_X(m)) = 0 $ for all $ m \in M $.
		\item $ f $ is $ C^1 $ and $ \nabla_m f_m (x) = K_m(x) $ for $ m \in M, x \in X_m $.
	\end{enumerate}
\end{lem}

\begin{rmk}\label{weakFG}
	\begin{enumerate}[(a)]
		\item Since $ \mathcal{A},~ \mathcal{B} $ consist of normal bundle charts, by the \emph{chain rule} (see \autoref{lem:chain}), we have
		\begin{multline}\label{kk1}
		D_{m} \widehat{F}_{m_0} (m,z)|_{m = m_0}
		=  \nabla_{u(m_0)} (\varphi^{u(m_0)}_{u(m_0)})^{-1}(x'') Du(m_0)
		+ D(\varphi^{u(m_0)}_{u(m_0)})^{-1}(x'') \\
		\cdot \left\{  \nabla_{m_0} F_{m_0}(z')  +  DF_{m_0}(z') ( \nabla_{m_0}\varphi^{m_0}_{m_0} (x), \nabla_{u(m_0)}\phi^{u(m_0)}_{u({m_0})}(y)Du(m_0) ) \right\} \\
		= \nabla_{m_0} F_{m_0}(z'),
		\end{multline}
		where $ z = (x,y) $, $ z' =  (\varphi^{m_0}_{m_0}(x) , \phi^{u(m_0)}_{u(m_0)}(y)) $, $ x'' = F_m(z') $. Similarly,
		\begin{equation}\label{kk2}
		D_{m} \widehat{G}_{m_0} (m,z)|_{m = m_0}  = \nabla_{m_0} G_{m_0}(z').
		\end{equation}

		\item One can replace \eqref{ccc0} by (see also \autoref{holsheaf}):
		\begin{gather*}
		|\widehat{F}_{m_0}(m_1, z) - \widehat{F}_{m_0}(m_0, z) | \leq M_0 |m_1 - m_0| |z|^{\gamma}, \\
		|\widehat{G}_{m_0}(m_1, z) - \widehat{G}_{m_0}(m_0, z)| \leq M_0 |m_1 - m_0| |z|^{\gamma},
		\end{gather*}
		and the spectral gap condition by $ \lambda_s \lambda_u < 1 $, $ \lambda^{\gamma}_s \lambda_u \mu < 1 $, $ \lambda_s < 1 $ (or $ \lambda_s \lambda_u < 1 $, $  \max\{ \lambda^{\gamma - 1}_s, 1 \} \mu \lambda_s \lambda_u < 1 $), where $ 0 < \gamma \leq 1 $. Then also $ f \in C^1 $ but
		\[
		\sup_{m \in M_1}\sup_{x \neq i_X(m)} \frac{|\nabla_m f_m(x)|}{|x|^{\gamma}} < \infty.
		\]
		The proof is very similar to the case $ \gamma = 1 $. Note that under $ \lambda^{\gamma}_s \lambda_u \mu < 1 $, $ \lambda_s < 1 $, by using \autoref{lem:sheaf}, one gets
		\[
		\sup_{m_0 \in M_1}\sup_{x \in X_{m_0}}\limsup_{m_1 \to m_0}\frac{|\widehat{f}_{m_0}(m_1, x) - \widehat{f}_{m_0}(m_0, x)|}{|m_1 - m_0||x|^\gamma} \leq C.
		\]
		Now we can use the same steps as in the proof of \autoref{smoothbase} with some nature minor modifications.
	\end{enumerate}
\end{rmk}

\begin{proof}[Proof of \autoref{smoothbase}]
	First note that the covariant derivatives of $ F, G $ satisfy
	\begin{equation}\label{ccc}
	|\nabla_{m_0} F_{m_0}(z)| \leq C_0|z|,~ |\nabla_{m_0} G_{m_0}(z)| \leq C'_0|z|,
	\end{equation}
	for all $ z \in X_{m_0} \times Y_{u(m_0)} $ and $ m_0 \in M_1 $, where $ C'_0 > 0 $ is independent of $ m_0 $.
	This is quite simple as the following shows. Let $ m_0 \in M_1 $, $ \chi_{m_0}(m_1) = x_1 \in T_{m_0}M $. Then
	\[
	\frac{|\widehat{F}_{m_0}(\chi^{-1}_{m_0}(x_1), z) - \widehat{F}_{m_0}(\chi^{-1}_{m_0}(0), z)|}{|x_1|}
	\leq \Theta^{-1}_1 M_0|z|,
	\]
	and so
	\[
	|\nabla_{m_0} F_{m_0}(z)| = |D_{m} \widehat{F}_{m_0} (m,z)|_{m = m_0} D\chi^{-1}_{m_0}(0)| \leq \Theta^{-1}_1 M_0|z|
	\]
	by \eqref{kk1}, and similarly for $ |\nabla_{m_0} G_{m_0}(z)| $.

	\textbf{(I)} Let $ e: M_1 \to M $ be the inclusion map, i.e., $ e(m) = m $. Consider the pull-back bundles $ e^*X, e^*Y $. Via two natural inclusion maps of $ e^*X \to X $, $ e^*Y \to Y $, induced by $ e $, we get natural pull-back vector bundles $ e^*TX, e^*TY $. Using $ e^*X, e^*Y $ and $ e^*TX, e^*TY $ instead of $ X,Y $, $ TX, TY $, we can, without loss of generality, assume $ M_1 = M $. \emph{So first let} $ M_1 = M $. \emph{The existence of $ K $ suffices under the assumptions in \autoref{lem:leafK}} below. So in this step, we in fact prove the conclusion (1) of \autoref{lem:leafK}.

	Define a metric space
	\[
	E^L_0 \triangleq \{ K \in L(\Upsilon^H_X, \Upsilon^V_Y) ~\text{a $ C^0 $ vector bundle map over}~  f: |K_m(x)| \leq C_1 |x| ~\text{if}~m \in M_1 \},
	\]
	with metric
	\[
	d_0(K,K') \triangleq \sup_m \sup_{x \neq i_X(m)} \frac{|K_m(x) - K_m'(x)|}{|x|},
	\]
	where
	\begin{equation*}
	C_1 = \frac{C_0}{1- \theta_0},~\theta_0 = \sup_{m \in M_1}\frac{\lambda_s(m) \lambda_u(m) \mu(m)}{1-\alpha(m)\beta'(u(m))} < 1,
	\end{equation*}
	and $ C_0 $ is given by \eqref{eee} below.
	Note that in our notation $ |x| = d_m(x,i_X(m)) $ if $ x \in X_m $.
	The metric space $ (E^L_0, d_0)  \neq \emptyset $ is complete. (Note that in the local representations, convergence in $ E^L_0 $ is uniform for $ m $ and $ x $ belonging to any bounded-fiber set.)

	Since $ \alpha(m)\beta'(u(m)) < 1 $, given a $ K \in E^L_0 $, there is a unique $ R \in L(\Upsilon^H_X, \Upsilon^V_X) $ over $ x_{(\cdot)}(\cdot) $ satisfying
	\begin{equation}\label{K1e}
	\nabla_mF_m( x, y ) + D_2F_m(x, y) ( K_{u(m)}(x_m(x))Du(m) + Df_{u(m)} (x_m(x)) R_m(x) ) = R_m(x),
	\end{equation}
	where $ y = f_{u(m)}(x_m(x)) $; and define
	\[
	\widetilde{K}_m(x) \triangleq \nabla_mG_m( x, y) + D_2G_m(x, y) ( K_{u(m)}(x_m(x))Du(m) + Df_{u(m)} (x_m(x)) R_m(x) ).
	\]

	Consider the graph transform $ \varGamma^0: K \mapsto \widetilde{K} $.
	\begin{slem}
		$ \varGamma^0 $ is a Lipschitz map $ E^L_0 \to E^L_0 $ with 
		\[
		\lip \varGamma^0 \leq \sup_{m \in M_1} \frac{\lambda_s(m) \lambda_u(m) \mu(m)}{1-\alpha(m)\beta'(u(m))} < 1.
		\]
	\end{slem}
	\begin{proof}
		Let $ \varGamma^0(K) = \widetilde{K} $, $ \varGamma^0(K') = \widetilde{K}' $. Since
		\begin{multline*}
		|\widetilde{K}_m(x) -\widetilde{K}_m'(x) | \leq | \{K_{u(m)} (x_m(x)) -  K_{u(m)}' (x_m(x))\} Du(m)  | \\
		+ \lambda_u(m)\beta'(u(m)) | R_m(x) - R_m'(x) |,
		\end{multline*}
		and
		\begin{multline*}
		|R_m(x) - R_m'(x)| \leq \alpha(m) | \{K_{u(m)} (x_m(x)) - K_{u(m)}' (x_m(x))\} Du(m)  | \\
		+ \alpha(m)\beta'(u(m)) | R_m(x) - R_m'(x) |,
		\end{multline*}
		we have
		\begin{equation}\label{kkk}
		| \widetilde{K}_m(x) -\widetilde{K}_m'(x) | \leq \frac{\lambda_u(m) \mu(m)}{1-\alpha(m) \beta'(u(m))} | K_{u(m)} (x_m(x)) - K_{u(m)}' (x_m(x)) |.
		\end{equation}
		If $ K \in E^L_0 $, then
		(a) $ \widetilde{K} \in L(\Upsilon^H_X, \Upsilon^V_Y) $ by construction,
		(b) $ \widetilde{K} $ is $ C^0 $,
		and (c) $ |\widetilde{K}_m(x)| \leq C_1 |x| $. Thus, $ \widetilde{K} \in E^L_0 $.

		To show (b), using a suitable local representation of \eqref{K1e} (see e.g. \eqref{K1elocal} below), the fact that $ f $, $ K^1 $ are continuous proved in \autoref{lem:continuity_f} and \autoref{lem:K1leafc} below, and assumption (H5$ ' $), one can easily see that $ \widetilde{K} $ is $ C^0 $ (see also the proof that $ K $ is $ C^0 $ in $ L(\Upsilon^H_X, \Upsilon^V_Y) $ below).

		We show (c) as follows. Letting $ K' = 0 $, by \eqref{ccc}, we have
		\begin{multline}\label{eee}
		|\varGamma^0(0)| \leq M_0 |z| + \lambda_u(m) \beta'(u(m)) |R_m'(x)| \\
		\leq M_0 |z| + \lambda_u(m) \beta'(u(m))   (1-\alpha(m)\beta'(u(m)))^{-1} M_0|z| \leq C_0 |x|,
		\end{multline}
		where $ z = (x,f_{u(m)}(x_m(x)) ) $, and $ C_0 $ is a constant independent of $ m $.
		Then by \eqref{kkk}, 
		\[
		|\widetilde{K}_m(x)| \leq |\varGamma^0(0)| +  \frac{\lambda_u(m)\mu(m)}{1-\alpha(m)\beta'(u(m))} |K_{u(m)} (x_m(x)) | \leq C_0 |x| + \theta_0 C_1 |x| \leq C_1 |x|.
		\]
		Inequality \eqref{kkk} also gives the estimate of the Lipschitz constant of $ \varGamma^0 $, and the proof is complete.
	\end{proof}
	Therefore, there is a unique $ K \in E^L_0 $ satisfying \eqref{mainMFW} when $ M_1 = M $. Next since $ u(M) \subset M_1 $, using \eqref{mainMFW}, one can uniquely define $ K $ in all of $ X $ (not only in $ e^*X $).

	We show $ K: \Upsilon^H_X \to \Upsilon^V_Y $ is a $ C^0 $ vector bundle map. This is direct from the local representation of \eqref{K1e}. We give the details. Let $ m_0 \in M $. Choose $ C^1 $ local charts of $ M $ at $ m_0, u(m_0) $, e.g., $ \xi_{m_0}: U_0 \to T_{m_0}M $ and $ \xi_{u(m_0)}: V \to T_{u(m_0)}M $, such that $ u(U_0) \subset V \subset U_{u(m_0)} $. Note that $ u(m_0) \in M_1 $. Moreover, take $ C^1 $ bundle charts $ (U_0, \varphi^0) $, $ (U_0, \phi^0) $ of $ X, Y $ at $ m_0 $, respectively, as well as $ (U_{u(m_0)}, \varphi^{u(m_0)} ) \in \mathcal{A} $, $ (U_{u(m_0)}, \phi^{u(m_0)} ) \in \mathcal{B} $. Using the above charts, one can give the local representations of $ K, R $:
	\begin{equation}\label{baseKR}
	\begin{cases}
	\begin{split}
	\widehat{K}'_{m_0} (m,x) \triangleq  D(\phi^{0}_{m})^{-1} ( f_m(\varphi^{0}_m(x)) )  K_{m} (\varphi^{0}_m & (x)) D\xi^{-1}_{m_0}(\xi_{m_0}(m)): \\
	& U_{0} \times X_{m_0} \to L(T_{m_0}M,Y_{m_0}),
	\end{split}\\
	\begin{split}
	\widehat{K}_{u(m_0)} (m,x) \triangleq  D(\phi^{u(m_0)}_{u(m)})^{-1} ( f_m( \varphi^{u(m_0)}_m & (x)) ) K_{m} (\varphi^{u(m_0)}_m (x)) D\xi^{-1}_{u(m_0)}(\xi_{u(m_0)}(m)): \\
	& u(U_0) \times X_{u(m_0)} \to L(T_{u(m_0)}M, Y_{u(m_0)}),
	\end{split}\\
	\begin{split}
	\widehat{R}_{m_0} (m,x) \triangleq  D ( \varphi^{u(m_0)}_{u(m)} )^{-1}  ( x_m(\varphi^0_m(x) ) )  R_m ( \varphi^{0}_m & (x) ) D\xi_{m_0}^{-1}(\xi_{m_0}(m)):\\
	& U_0 \times X_{m_0} \to L(T_{m_0}M, X_{u(m_0)}),
	\end{split}
	\end{cases}
	\end{equation}
	as well as $ \nabla F, \nabla G, F, G, x_{(\cdot)}(\cdot), Du$, i.e.,
	\begin{multline}
	\widehat{D}_m F_{m_0}(m, x, y)  \triangleq D( \varphi^{u(m_0)}_{u(m)} )^{-1}( F_m( \varphi^{0}_{m}(x), \phi^{u(m_0)}_{u(m)}(y) ) ) \\
	\cdot \nabla_m F_m( \varphi^{0}_{m}(x), \phi^{u(m_0)}_{u(m)}(y)  ) D\xi^{-1}_{m_0}(\xi_{m_0}(m)): \\
	U_0 \times X_{m_0} \times Y_{u(m_0)} \to L(T_{m_0}M, Y_{u(m_0)}), \label{basediff1}
	\end{multline}
	\begin{multline}
	\widehat{D}_m G_{m_0}(m, x, y)  \triangleq  D( \varphi^{0}_{m} )^{-1}( G_m( \varphi^{0}_{m}(x), \phi^{u(m_0)}_{u(m)}(y) ) )  \\
	\cdot \nabla_m G_m( \varphi^{0}_{m}(x), \phi^{u(m_0)}_{u(m)}(y)  ) D\xi^{-1}_{m_0}(\xi_{m_0}(m)): \\
	U_0 \times X_{m_0} \times Y_{u(m_0)} \to L(T_{m_0}M, Y_{m_0}) , \label{basediff2}
	\end{multline}
	\begin{equation}\label{basediff}
	\begin{cases}
	\widehat{F}'_{m_0} (m,x,y) \triangleq (\varphi^{u(m_0)}_{u(m)})^{-1} \circ F_m(\varphi^0_m(x), \phi^{u(m_0)}_{u(m)} (y) ): ~ U_{0} \times X_{m_0} \times Y_{u(m_0)} \to X_{u(m_0)} , \\
	\widehat{G}'_{m_0} (m,x,y) \triangleq (\varphi^{0}_{m})^{-1} \circ G_m(\varphi^0_m(x), \phi^{u(m_0)}_{u(m)} (y) ): ~ U_{0} \times X_{m_0} \times Y_{u(m_0)} \to X_{m_0} , \\
	\widehat{x}'_{m_0}(m, x) \triangleq (\varphi^{u(m_0)}_{u(m)})^{-1} \circ x_{m} (\varphi^{0}_m(x)) :~ U_{0} \times X_{m_0} \to X_{u(m_0)}, \\
	\widehat{Du}_{m_0} (m)v \triangleq D\xi_{u(m_0)}(u(m)) Du(m)D\xi_{m_0}^{-1}(\xi_{m_0}(m))v:~ U_0 \times T_{m_0}M \to T_{u(m_0)}M.
	\end{cases}
	\end{equation}
	Then by a simple computation,
	\begin{equation}\label{K1elocal}
	\begin{cases}
	\begin{split}
	\widehat{D}_m F_{m_0}( m, x, \hat{y}' ) + D_y \widehat{F}'_{m_0} ( & m, x, \hat{y}' ) \left\{ \widehat{K}_{u(m_0)} ( u(m),\hat{x}' ) \widehat{Du}_{m_0}(m) \right.\\
	&\left. + D_x \widehat{f}_{u(m_0)} (u(m), \hat{x}') \widehat{R}_{m_0} (m,x) \right\} = \widehat{R}_{m_0} (m,x),
	\end{split}\\
	\begin{split}
	\widehat{D}_m G_{m_0}( m, x, \hat{y}' ) + D_y \widehat{G}'_{m_0} & ( m, x, \hat{y}' ) \left\{  \widehat{K}_{u(m_0)} ( u(m),\hat{x}' ) \widehat{Du}_{m_0}(m) \right.\\
	&\left. + D_x \widehat{f}_{u(m_0)} (u(m), \hat{x}') \widehat{R}_{m_0} (m,x) \right\} = \widehat{K}'_{m_0} (m,x) ,
	\end{split}
	\end{cases}
	\end{equation}
	where $ \hat{x}' = \widehat{x}'_{m_0} (m,x), \hat{y}' = \widehat{f}_{u(m_0)} ( u(m), \hat{x}' ) $.
	Note that $ f $, $ K^1 $ are continuous as proved in \autoref{lem:continuity_f} and \autoref{lem:K1leafc} below, and so $ \widehat{D}_m F_{m_0} $, $ \widehat{D}_m G_{m_0} $, $ D_y \widehat{F}'_{m_0} $, $ D_y \widehat{G}'_{m_0} $, and $ \widehat{K}^1_{m_0} (m,x) = D_x \widehat{f}_{m_0} (m,x) $ are all continuous for every $ m_0 \in M $.

	By assumption (H5$ ' $), we get for $ m \in U_0 $,
	\[
	|D_y \widehat{F}'_{m_0} (m, x, \hat{y}' ) D_x \widehat{f}_{u(m_0)} (u(m), \hat{x}') | \approx \alpha(m_0)\beta'(u(m_0)) < 1,
	\]
	where $ U_0 $ may be smaller if necessary. Since $ \widehat{K}_{u(m_0)} (u(\cdot),\cdot) $ is continuous at $ (m_0,x) $ for all $ x \in X_{u(m_0)} $, we see $ \widehat{R}_{m_0}(\cdot,\cdot) $ is continuous at $ (m_0, x) $, and so is $ \widehat{K}'_{m_0}(\cdot,\cdot) $, which implies that $ K $ is $ C^0 $.

	\textbf{(II)} In the following, we need to show $ f $ is $ C^1 $ and $ \nabla_m f_m(x) = K_m(x) $. Working in bundle charts, let us consider \eqref{mainLocal} and the local representation of \eqref{mainMFW}.
	Let
	\[
	\begin{gathered}
	\widehat{K}_{m_0} (x) = D(\phi^{m_0}_{m_0})^{-1} (y') K_{m_0} (\varphi^{m_0}_{m_0}(x)) = K_{m_0}(x),\\
	\widehat{R}_{m_0} (x) = D(\varphi^{u(m_0)}_{u(m_0)})^{-1} (x'') R_{m_0} (\varphi^{m_0}_{m_0}(x)) = R_{m_0}(x),
	\end{gathered}
	\]
	where $ y' = f_{m_0}(\varphi^{m_0}_{m_0}(x)) $, $ x'' = x_{m_0}(\varphi^{m_0}_{m_0}(x)) $, $ m_0 \in M_1 $. Using \eqref{kk1} and \eqref{kk2}, we know
	\begin{equation*}
	\begin{cases}
	\begin{split}
	D_m \widehat{F}_{m_0} (m,x,y) & |_{m = m_0} + D_y\widehat{F}_{m_0} (m_0,x,y) \\
	& \cdot \left\lbrace  \widehat{K}_{u(m_0)} (\hat{x}) Du(m_0) + D_x \widehat{f}_{u(m_0)} (u(m_0), \hat{x} )\widehat{R}_{m_0}(x) \right\rbrace  = \widehat{R}_{m_0}(x),
	\end{split}\\
	\begin{split}
	D_m \widehat{G}_{m_0} (m,x,y) & |_{m = m_0} + D_y\widehat{G}_{m_0} (m_0,x,y) \\
	& \cdot \left\lbrace \widehat{K}_{u(m_0)} (\hat{x}) Du(m_0) + D_x \widehat{f}_{u(m_0)} (u(m_0), \hat{x} )\widehat{R}_{m_0}(x) \right\rbrace = \widehat{K}_{m_0}(x),
	\end{split}
	\end{cases}
	\end{equation*}
	where $ y = \widehat{f}_{u(m_0)} (u(m_0), \hat{x} ) $, $ \hat{x} = \widehat{x}_{m_0}(m_0,x) $.
	Define
	\[
	Q^0 (m', m_0, x) \triangleq \widehat{f}_{m_0} (m',x) - \widehat{f}_{m_0}(m_0,x) - \widehat{K}_{m_0}(x) (m' - m),
	\]
	where $ m' \in U_{m_0}(\hat{\mu}^{-2} \varepsilon_1 ) $, $ m_0 \in M_1 $. $ \varepsilon_1 $ is taken from the proof of \autoref{lem:sheaf}, as also are $ \alpha'', \beta'' $, $ \lambda''_s, \lambda''_u $, such that $ 1-\alpha''(m)\beta''(u(m)) > 0 $,
	\[
	\lambda''_s \lambda''_u < 1,~  \mu \lambda''_s \lambda''_u < 1,
	\]
	and \eqref{AB''}, \eqref{kkr} hold.

	\noindent \textbf{\emph{Convention}}.
	Here, we identify $ U_{m_0}(\hat{\mu}^{-2} \varepsilon_1 ) $ with the Banach space $ T_{m_0}M $ by the local chart $ \chi_{m_0}: U_{m_0}(\hat{\mu}^{-2} \varepsilon_1 ) \to T_{m_0}M $ given in (H1b); that is, we identify $ x' $ with $ m' $ if $ \chi_{m_0}(m') = x' $. For example, the local representation of $ \widehat{f}_{m_0}(m',x) $ by $ \chi_{m_0}: U_{m_0}(\epsilon_{m_0}') \to T_{m_0}M $ is $ \widehat{f}_{m_0}(\chi^{-1}_{m_0}(x'),x) $, but we also write it as $ \widehat{f}_{m_0}(m',x) $. The metric in $ T_{m_0}M $ is induced from $ d_{m_0} $, i.e., $ |x_1 - x_2| = d_{m_0}(\chi^{-1}_{m_0}{(x_1)}, \chi^{-1}_{m_0}{(x_2)}) $. Further, $ \varepsilon_1 $ may be taken small and may even depend on $ m $. This does not cause any problem since at last we will let $ \varepsilon_1 \to 0 $, i.e., $ U_{m_0} \ni m' \to m_0 $.

	We will write $ |g(m', m_0, x)| \leq O_{x}(|m'-m_0|) $ if 
	\[
	\limsup_{m' \to m_0} \frac{|g(m',m_0, x)|}{|m'-m_0|} = 0.
	\]
	\begin{slem}
		$ \sup\limits_{m_0 \in M_1}\sup\limits_{x \in X_{m_0}}\limsup\limits_{m' \to m_0} \frac{|Q^0 (m', m_0, x)|}{|m'-m_0||x|} < \infty $ and
		\[
		|Q^0 (m', m_0, x)| \leq O_{x}(|m'-m_0|) + \frac{\lambda''_u(m_0)}{1-\alpha''(m_0)\beta''(u(m_0))} |Q^0( u(m'), u(m_0), \widehat{x}_{m_0}(m',x) )|.
		\]
	\end{slem}
	\begin{proof}
		Under the conditions given in the lemma, by \autoref{lem:sheaf}, we have
		\begin{equation}\label{fff}
		\left\{
		\begin{split}
		\sup\limits_{m_0 \in M_1}\sup\limits_{x \in X_{m_0}}\limsup\limits_{m_1 \to m_0}\frac{|\widehat{x}_{m_0}(m_1, x) - \widehat{x}_{m_0}(m_0, x)|}{|m_1 - m_0||x|} \leq C , \\
		\sup\limits_{m_0 \in M_1}\sup\limits_{x \in X_{m_0}}\limsup\limits_{m_1 \to m_0}\frac{|\widehat{f}_{m_0}(m_1, x) - \widehat{f}_{m_0}(m_0, x)|}{|m_1 - m_0||x|} \leq C .
		\end{split}\right.
		\end{equation}
		Now the first inequality follows from \eqref{fff}, \eqref{A11} and the above construction of $ K $. (That the constants $ \Theta_1, \Theta_2 $ in (H1b) are independent of $ m_0 $ is used exactly here.) Next, let us consider the second inequality. We compute
		\begin{align*}
		& |Q^0 (m', m_0, x)|  \\
		= & \left|\widehat{G}_{m_0} ( m', x, \widehat{f}_{u(m_0)} (u(m'), \widehat{x}_{m_0} (m',x) ) ) \right. \\
		& \left. - \widehat{G}_{m_0} ( m_0, x, y )  - D_m \widehat{G}_{m_0} (m,x,y)|_{m = m_0} (m' - m_0)\right. \\
		& \left. - D_y\widehat{G}_{m_0} (m_0,x,y) ( \widehat{K}_{u(m_0)} (\hat{x}) Du(m_0) + D_x \widehat{f}_{u(m_0)} (u(m_0), \hat{x} )\widehat{R}_{m_0}(x) ) (m' - m_0) \right| \\
		\leq & |\square_1| + |\square_2|,
		\end{align*}
		where $ y = \widehat{f}_{u(m_0)} (u(m_0), \hat{x} ) $, $ \hat{x} = \widehat{x}_{m_0}(m_0,x) $, and
		\begin{align*}
		\square_1 \triangleq & \widehat{G}_{m_0} ( m', x, \widehat{f}_{u(m_0)} (u(m'), \widehat{x}_{m_0} (m',x) ) ) - \widehat{G}_{m_0} ( m', x, \widehat{f}_{u(m_0)} (u(m_0), \hat{x} ) ) \\
		& - D_y\widehat{G}_{m_0} (m_0,x,y) ( \widehat{K}_{u(m_0)} (\hat{x}) Du(m_0) + D_x \widehat{f}_{u(m_0)} (u(m_0), \hat{x} )\widehat{R}_{m_0}(x) ) (m' - m_0) ,\\
		\square_2 \triangleq & \widehat{G}_{m_0} ( m', x, y ) - \widehat{G}_{m_0} ( m_0, x, y ) - D_m \widehat{G}_{m_0} (m,x,y)|_{m = m_0} (m' - m_0).
		\end{align*}
		Note that $ | \square_2 | \leq O_{x}(|m' - m_0|) $ and
		\begin{align*}
		\square_1 = & \int_{0}^{1} D_y\widehat{G}_{m_0} (m',x,y_t) ~\mathrm{d}t \left\{ \widehat{f}_{u(m_0)} (u(m'), \widehat{x}_{m_0} (m',x) ) - \widehat{f}_{u(m_0)} (u(m_0), \hat{x} ) \right\} \\
		& -  D_y\widehat{G}_{m_0} (m_0,x,y) ( \widehat{K}_{u(m_0)} (\hat{x}) Du(m_0) + D_x \widehat{f}_{u(m_0)} (u(m_0), \hat{x} )\widehat{R}_{m_0}(x) ) (m' - m_0) \\
		= & \int_{0}^{1} \{ D_y\widehat{G}_{m_0} (m',x,y_t) - D_y\widehat{G}_{m_0} (m_0,x,y) \} ~\mathrm{d}t \left\{ \widehat{f}_{u(m_0)} (u(m'), \widehat{x}_{m_0} (m',x) ) \right. \\
		& \qquad \left.- \widehat{f}_{u(m_0)} (u(m_0), \hat{x} ) \right\} \\
		& + D_y\widehat{G}_{m_0} (m_0,x,y)  \left\lbrace \widehat{f}_{u(m_0)} (u(m'), \widehat{x}_{m_0} (m',x) ) - \widehat{f}_{u(m_0)} (u(m_0), \hat{x} ) \right. \\
		& \qquad \left. - \widehat{K}_{u(m_0)} (\hat{x}) Du(m_0) (m' - m_0) - D_x \widehat{f}_{u(m_0)} (u(m_0), \hat{x} )\widehat{R}_{m_0}(x) (m' - m_0) \right\rbrace \\
		\triangleq & \textcircled{1} + D_y\widehat{G}_{m_0} (m_0,x,y)\{ \textcircled{2} \},
		\end{align*}
		where $ y_t = t \widehat{f}_{u(m_0)} (u(m'), \widehat{x}_{m_0} (m',x) ) + (1-t) \widehat{f}_{u(m_0)} (u(m_0), \hat{x} ) $.
		By \eqref{fff} and the Lipschitz continuity of $ u(\cdot), \widehat{f}_{m_0}(m, \cdot), \widehat{x}_{m_0}(m,\cdot) $, we have
		\begin{multline*}
		|\textcircled{1}| \triangleq \left|\int_{0}^{1} \{ D_y\widehat{G}_{m_0} (m',x,y_t) - D_y\widehat{G}_{m_0} (m_0,x,y) \} ~\mathrm{d}t \left\{ \widehat{f}_{u(m_0)} (u(m'), \widehat{x}_{m_0} (m',x) ) \right. \right. \\
		\left. \left. - \widehat{f}_{u(m_0)} (u(m_0), \hat{x} ) \right\} \right| \leq O_{x}(|m' - m_0|) .
		\end{multline*}
		Consider
		\begin{align*}
		\textcircled{2} \triangleq & \widehat{f}_{u(m_0)} (u(m'), \widehat{x}_{m_0} (m',x) ) - \widehat{f}_{u(m_0)} (u(m_0), \hat{x} ) \\
		& - \widehat{K}_{u(m_0)} (\hat{x}) Du(m_0) (m' - m_0) - D_x \widehat{f}_{u(m_0)} (u(m_0), \hat{x} )\widehat{R}_{m_0}(x) (m' - m_0) \\
		= & \widehat{f}_{u(m_0)} ( u(m'), \widehat{x}_{m_0} (m',x) ) - \widehat{f}_{u(m_0)} ( u(m_0), \widehat{x}_{m_0} (m',x) ) \\
		& \quad \quad - \widehat{K}_{u(m_0)} ( \widehat{x}_{m_0}(m',x) ) \{ u(m') - u(m_0) \} \\
		& + \widehat{f}_{u(m_0)} ( u(m_0), \widehat{x}_{m_0} (m',x) ) - \widehat{f}_{u(m_0)} (u(m_0), \hat{x} ) \\
		& \quad \quad  - D_x \widehat{f}_{u(m_0)} (u(m_0), \hat{x} ) \{ \widehat{x}_{m_0} (m',x) - \widehat{x}_{m_0}(m_0,x) \} \\
		& + \left\{ \widehat{K}_{u(m_0)} ( \widehat{x}_{m_0}(m',x) ) - \widehat{K}_{u(m_0)} (\hat{x}) \right\} ( u(m') - u(m_0) ) \\
		& + \widehat{K}_{u(m_0)} (\hat{x}) \left\{ u(m') - u(m_0) - Du(m_0)(m' - m_0) \right\} \\
		& + D_x \widehat{f}_{u(m_0)} ( u(m_0), \hat{x} ) \left\{ \widehat{x}_{m_0} (m',x) - \widehat{x}_{m_0}(m_0,x) - \widehat{R}_{m_0}(x)(m' - m_0) \right\} \\
		\triangleq & Q^0( u(m'), u(m_0), \widehat{x}_{m_0}(m',x) ) + \textcircled{3} + \textcircled{4} + \textcircled{5} + D_x \widehat{f}_{u(m_0)} ( u(m_0), \hat{x} ) \{\textcircled{6}\}.
		\end{align*}
		Since $ (m,x)  \to \widehat{x}_{m_0}(m,x) $ is continuous (see e.g. \autoref{lem:continuity_f} below), $ f_m(\cdot) $ is $ C^1 $, and \eqref{fff} holds, we have
		\begin{align*}
		|\textcircled{3}| & \triangleq | \widehat{f}_{u(m_0)} ( u(m_0), \widehat{x}_{m_0} (m',x) ) - \widehat{f}_{u(m_0)} (u(m_0), \hat{x} ) \\
		& \quad - D_x \widehat{f}_{u(m_0)} (u(m_0), \hat{x} ) \{ \widehat{x}_{m_0} (m',x) - \widehat{x}_{m_0}(m_0,x) \} | \\
		& \leq O_{x}(|m'-m_0|).
		\end{align*}
		Since $ \widehat{x}_{m_0}(\cdot) $ is continuous, $ K $ is $ C^0 $, and $ u $ is Lipschitz, we have
		\[
		|\textcircled{4}| \triangleq | \left\{ \widehat{K}_{u(m_0)} ( \widehat{x}_{m_0}(m',x) ) - \widehat{K}_{u(m_0)} (\hat{x}) \right\} ( u(m') - u(m_0) ) | \leq O_{x}(|m'-m_0|).
		\]
		As $ u $ is $ C^1 $ and \eqref{A11}, we have
		\[
		|\textcircled{5}| \triangleq | \widehat{K}_{u(m_0)} (\hat{x}) \{ u(m') - u(m_0) - Du(m_0)(m' - m_0) \} | \leq O_{x}(|m'-m_0|).
		\]
		Thus,
		\begin{align*}
		& |Q^0 (m', m_0, x)| \\
		\leq & O_x(|m'-m_0|) + \lambda''_u(m_0) |Q^0( u(m'), u(m_0), \widehat{x}_{m_0}(m',x) )| + \lambda''_u(m_0)\beta''(u(m_0))  |\textcircled{6}|.
		\end{align*}

		The same computation gives
		\begin{align*}
		|\textcircled{6}| \triangleq & |  \widehat{x}_{m_0} (m',x) - \widehat{x}_{m_0}(m_0,x) - \widehat{R}_{m_0}(x)(m' - m_0) | \\
		\leq & O_{x}(|m'-m_0|) + \alpha''(m_0) |Q^0( u(m'), u(m_0), \widehat{x}_{m_0}(m',x) )| + \alpha''(m_0)\beta''(u(m_0)) |\textcircled{6}|.
		\end{align*}
		Therefore,
		\[
		|Q^0 (m', m_0, x)| \leq O_{x}(|m'-m_0|) + \frac{\lambda''_u(m_0)}{1-\alpha''(m_0)\beta''(u(m_0))} |Q^0( u(m'), u(m_0), \widehat{x}_{m_0}(m',x) )|.
		\]
		The proof is complete.
	\end{proof}
	Using the above sublemma, we obtain
	\[
	\sup_{m_0 \in M_1}\sup_{x \in X_{m_0}}\limsup_{m' \to m_0} \frac{|Q^0 (m', m_0, x)|}{|m'-m_0||x|} \leq \theta''_0 \sup_{m_0 \in M_1}\sup_{x \in X_{m_0}} \limsup_{m' \to m_0} \frac{|Q^0 (m', m_0, x)|}{|m'-m_0||x|},
	\]
	where $ \theta''_0 = \sup_{m \in M_1}\frac{\lambda''_s(m) \lambda''_u(m) \mu''(m)}{1-\alpha''(m)\beta''(u(m))} < 1 $, which yields $ \limsup\limits_{m' \to m_0} \frac{|Q^0 (m', m_0, x)|}{|m'-m_0|} = 0 $, i.e., $ \widehat{f}_{m_0} (\cdot,x) $ is differentiable at $ m_0 \in M_1 $ and
	\begin{equation}\label{fk}
	D_m \widehat{f}_{m_0} (m,x)|_{m = m_0} = \widehat{K}_{m_0}(x).
	\end{equation}

	Finally, we need to show $ f $ is $ C^1 $ everywhere. Let $ m', m_0 \in M $, and choose $ C^1 $ bundle charts $ (U_0, \varphi^0) $, $ (U_0, \phi^0) $ of $ X, Y $ at $ m' $, respectively, such that 
	\[
	u(U_0) \subset U_{u(m_0)}(\mu^{-1}(u(m_0))\varepsilon_1).
	\]
	Note that $ u(m_0) \in M_1 $. Consider $ F, G, f, x_{(\cdot)}(\cdot) $ in local bundle charts $ \varphi^0: U_0 \times X_{m'} \to X $, $ \phi^0: U_0 \times Y_{m'} \to Y $, $ \varphi^{u(m_0)}: U_{u(m_0)}(\epsilon_2) \times X_{u(m_0)} \to X $, $ \phi^{u(m_0)}: U_{u(m_0)}(\epsilon_2) \times Y_{u(m_0)} \to Y $. That is,
	\begin{equation}\label{local11}
	\begin{cases}
	\begin{split}
	\widehat{F}'_{m_0} (m, x, y) \triangleq (\varphi^{u(m_0)}_{u(m)})^{-1} \circ F_m \circ ( \varphi^{0}_m(x) , & \phi^{u(m_0)}_{u(m)}(y) ) : \\
	& U_{0} \times X_{m'} \times Y_{u(m_0)} \to X_{u(m_0)},
	\end{split}\\
	\begin{split}
	\widehat{G}'_{m_0} (m, x, y) \triangleq (\phi^{0}_{m})^{-1} \circ G_m \circ ( \varphi^{0}_m(x) , \phi^{u(m_0)}_{u(m)} & (y) ) : \\
	& U_{0} \times X_{m'} \times Y_{u(m_0)} \to Y_{m'} ,
	\end{split}\\
	\widehat{f}'_{0}(m, x) \triangleq (\phi^{0}_m)^{-1} \circ f_m \circ  \varphi^{0}_m(x) :  U_{0} \times X_{m'} \to Y_{m'} , \\
	\begin{split}
	\widehat{f}_{u(m_0)}(m, x) \triangleq (\phi^{u(m_0)}_{m})^{-1} & \circ f_{u(m)} \circ ( \varphi^{u(m_0)}_{m}(x) ) : \\
	& U_{u(m_0)} (\mu^{-1}(u(m_0)) \varepsilon_1) \times X_{u(m_0)} \to Y_{u(m_0)} ,
	\end{split}\\
	\widehat{x}'_{m_0}(m, x) \triangleq (\varphi^{u(m_0)}_{u(m)})^{-1} \circ x_{m} (\varphi^{0}_m(x)) : U_{0} \times X_{m'} \to X_{u(m_0)}. \\
	\end{cases}
	\end{equation}
	Then (see also \eqref{mainLocal})
	\begin{equation}\label{localeq}
	\begin{cases}
	\widehat{F}'_{m_0} ( m, x, \widehat{f}_{u(m_0)} (u(m), \widehat{x}'_{m_0} (m,x) ) ) = \widehat{x}'_{m_0} (m, x), \\
	\widehat{G}'_{m_0} ( m, x, \widehat{f}_{u(m_0)} (u(m), \widehat{x}'_{m_0} (m,x) ) ) = \widehat{f}'_{0} (m, x).
	\end{cases}
	\end{equation}
	Due to (H5), we can let $ U_0 $ be much smaller such that if $ m \in U_0 $, then
	\[
	\lip \widehat{F}'_{m_0}(m, x, \cdot) \leq (1+\eta(d(u(m),u(m_0)))) \alpha(m) (1+\eta(d(u(m),u(m_0)))) \leq \alpha''(m_0),
	\]
	and $ \alpha''(m_0)\beta''(u(m_0)) < 1 $,
	where $ \beta''(u(m_0)) $ is given by
	\begin{multline*}
	\lip \widehat{f}_{u(m_0)}(u(m), \cdot) \leq (1+\eta(d( u(m),u(m_0) ))) \beta'(u(m))  \\
	\cdot (1+\eta(d( u(m),u(m_0)   ))) \leq \beta''(u(m_0)).
	\end{multline*}

	Due to $ \lip \widehat{F}'_{m_0} ( m, x, \widehat{f}_{u(m_0)}(u(m), \cdot) ) \leq \alpha''(m_0)\beta''(u(m_0)) < 1 $ and the differentiability of $ \widehat{f}_{u(m_0)} (\cdot, x ) $  at $ u(m_0) $, using \autoref{lem:fdiifx}, we know that $ \widehat{f}'_{0} (\cdot, x) $ is differentiable at $ m_0 $. Since $ m_0 $ can be taken in a small neighborhood of $ m' $, $ \widehat{f}'_{0} (\cdot, x) $ is differentiable in a neighborhood of $ m' $ and then $ (m, x) \mapsto D_m\widehat{f}'_{0} (m, x) $ is continuous (as $ K $ is $ C^0 $ and \eqref{fk}). Moreover, $ (m, x) \mapsto D_x\widehat{f}'_{0} (m, x) $ is continuous, hence $ f $ is differentiable and hence $ C^1 $.

	Now $ \nabla_{m} f_m(x) = K_m(x) $ if $ m \in M_1 $, as $ \varphi^{m}, \phi^{m} $ are normal bundle charts and we have \eqref{fk}. For $ m \in M \backslash M_1 $, this follows from \eqref{mainMFW}, $ \nabla_{u(m)} f_{u(m)}(x) = K_{u(m)}(x) $, and taking the covariant derivative of \eqref{mainF}. The proof is thus finished.
\end{proof}

\subsection{H\"{o}lderness of $ x \mapsto \nabla_mf_m(x) = K_m(x)$}
(Based on the H\"{o}lder continuity of $ x \mapsto Df_m(x) = K^1_m(x) $.)

It suffices to consider the unique vector bundle map $ K $ satisfying \eqref{mainMFW}. Since we do not need $ \nabla_m f_m(x) = K_m(x) $, some assumptions in \autoref{smoothbase} can be weakened; see the proof of \autoref{smoothbase}.

\begin{lem}\label{lem:holversheaf}
	Assume the following conditions hold:
	\begin{enumerate}[(a)]
		\item Let \textnormal{(H1), (H2)(i, ii, v1$ ' $), (H3), (H4b), (H5$ ' $)} hold.
		\item 
		\begin{enumerate}[(i)]
			\item $ F, G $ are $ C^1 $.
			$|\nabla_m F_m(z)| \leq M_0|z|$, $ |\nabla_m G_m(z)| \leq M_0|z|$ for all $ m \in M_1 $ and $ z \in X_m \times Y_{u(m)} $.
			\item $ \nabla_mF_m(\cdot) $, $ \nabla_mG_m(\cdot) $, $ DF_m(\cdot) $ and $ DG_m(\cdot) $ are $ C^{0,\gamma} $ uniformly for $ m $, i.e.,
			\begin{equation*}
			\begin{cases}
			\max\{|\nabla_mF_m(z_1) - \nabla_mF_m(z_2)|,  |\nabla_mG_m(z_1) - \nabla_mG_m(z_2)|\} \leq M_0 |z_1 - z_2|^{\gamma}, \\
			\max\{|DF_m(z_1) - DF_m(z_2)|,  |DG_m(z_1) - DG_m(z_2)| \} \leq M_0 |z_1 - z_2|^{\gamma},
			\end{cases}
			\end{equation*}
			for all $ z_1, z_2 \in X_m \times Y_{u(m)} $, $ m \in M $, where $ 0 < \gamma \leq 1 $.
		\end{enumerate}
		\item (spectral gap condition) $ \lambda_s \lambda_u < 1 $, $ \lambda_s \lambda_u \mu < 1 $, $ \lambda_s^{*\beta} \lambda_s \lambda_u < 1 $, and 
		
		$ \max\{ \lambda^{-(1-\gamma)}_s, \lambda^{\gamma}_s, \lambda^{\beta}_s \}^{*\alpha} \lambda_s \lambda_u \mu < 1 $, where $ \gamma \geq \beta $; see also \autoref{absgc}.
	\end{enumerate}
	Then the $ C^0 $ vector bundle map $ K \in L(\Upsilon^H_X, \Upsilon^V_Y) $ over $ f $ satisfying \eqref{mainMFW} has the following H\"older property:
	\[
	|K_m(x_1) - K_m(x_2)| \leq C  \left\lbrace ( |x_1 - x_2|^{\gamma} (1+|x_1|) )^{\alpha} + (|x_1 - x_2|^\beta |x_1|)^\alpha \right\rbrace (|x_1| + |x_2|)^{1-\alpha}
	\]
	if $ |x_1 - x_2|^{\gamma} (1+|x_1|) + |x_1 - x_2|^\beta |x_1| \leq \hat{r} (|x_1| + |x_2|) $, where the constant $ C $ depends on the constant $ \hat{r} > 0 $ but not $ m \in M $.
\end{lem}
\begin{proof}
	First,under these conditions, by \autoref{lem:leaf1a}, we have
	\[
	Df_m(x) = K^1_m(x), ~|K^1_m(x_1) - K^1_m(x_2)| \leq \widetilde{C} |x_1 - x_2|^\beta.
	\]
	Also, $ \sup_{m \in M_1}\sup_{x \neq i_X(m)} |K_m(x)|/|x| < \infty $ by \autoref{smoothbase} (1).
	Using \eqref{mainMFW}, one gets
	\begin{align*}
	& K_m(x_1) - K_m(x_2) \\
	= & \nabla_mG_m( x_1, f_{u(m)}(x_m(x_1)) ) - \nabla_mG_m( x_2, f_{u(m)}(x_m(x_2)) ) \\
	& + \left\lbrace  D_2G_m(x_1, f_{u(m)}(x_m(x_1))) - D_2G_m(x_2, f_{u(m)}(x_m(x_2))) \right\rbrace  \\
	& \quad \cdot ( K_{u(m)}(x_m(x_1))Du(m) + Df_{u(m)} (x_m(x_1)) R_m(x_1) ) \\
	& + D_2G_m(x_2, f_{u(m)}(x_m(x_2))) \left\{ K_{u(m)}(x_m(x_1))Du(m) + Df_{u(m)} (x_m(x_1)) R_m(x_1)\right. \\
	& \quad  \left.- K_{u(m)}(x_m(x_2))Du(m) - Df_{u(m)} (x_m(x_2)) R_m(x_2) \right\}.
	\end{align*}
	Now for $ m \in M $ we have 
	\begin{multline*}
	|K_m(x_1) - K_m(x_2)| \leq \widetilde{C} |x_1 - x_2|^\gamma (1+|x_1|) + \widetilde{C} |x_1 - x_2|^\beta |x_1| \\
	+ \lambda_u(m) \mu(m) | K_{u(m)}(x_m(x_1)) - K_{u(m)}(x_m(x_2)) | + \lambda_u(m) \beta'(u(m)) |R_m(x_1) - R_m(x_2)|,
	\end{multline*}
	and similarly
	\begin{multline*}
	|R_m(x_1) - R_m(x_2)| \leq \widetilde{C} |x_1 - x_2|^\gamma (1+|x_1|) + \widetilde{C} |x_1 - x_2|^\beta |x_1| \\
	+ \alpha(m) \mu(m) | K_{u(m)}(x_m(x_1)) - K_{u(m)}(x_m(x_2)) | + \alpha(m) \beta'(u(m)) |R_m(x_1) - R_m(x_2)|.
	\end{multline*}
	Hence,
	\begin{multline}\label{Kxx}
	|K_m(x_1) - K_m(x_2)| \leq \widetilde{C} |x_1 - x_2|^\gamma (1+|x_1|) + \widetilde{C} |x_1 - x_2|^\beta |x_1| \\
	+ \frac{\lambda_u(m)\mu(m)}{1-\alpha(m) \beta'(u(m))}| K_{u(m)}(x_m(x_1)) - K_{u(m)}(x_m(x_2)) |.
	\end{multline}
	Apply the argument in \autoref{Appbb} to conclude the proof (see \autoref{argumentapp} \eqref{lem4}).
\end{proof}

Under the conditions in \autoref{smoothbase}, we know that if $ F, G $ are uniformly $ C^{1,1} $, then from this lemma, $ \nabla_m f_m(\cdot) $ is at least locally H\"older uniformly for $ m \in M $.

\section{H\"{o}lderness of $ m \mapsto \nabla_mf_m(x) = K_m(x) $}\label{HolderivativeBase}
(Based on all the previous lemmas.)

Let $ K \in L(\Upsilon^H_{X}, \Upsilon^V_{Y}) $ over $ f $ be obtained in \autoref{smoothbase}. In particular, $ K $ and $ R \in L(\Upsilon^H_X, \Upsilon^V_X ) $ over $ x_{(\cdot)}(\cdot) $ satisfy \eqref{mainMFW}. Write $ \widehat{D}_m F_{m_0}(\cdot,\cdot), \widehat{D}_m G_{m_0}(\cdot,\cdot) $, $ \widehat{K}_{m_0}(\cdot,\cdot), \widehat{R}_{m_0}(\cdot,\cdot) $ as the local representations of $ \nabla_m F_m, \nabla_m G_m $, $ K, R $ at $ m_0 \in M_1 $ with respect to $ \mathcal{A}, \mathcal{B}, \mathcal{M} $ (see e.g. \eqref{basediff1} \eqref{basediff2} \eqref{baseKR} but using the local bundle charts belonging to $ \mathcal{A}, \mathcal{B}, \mathcal{M} $). We rewrite them explicitly here for the convenience of the reader (see also \autoref{tensorH} and \autoref{summary}):

\begin{equation}\label{basediff00}
\begin{cases}
\begin{split}
\widehat{D}_m F_{m_0}(m, x,  y)  \triangleq D( \varphi^{u(m_0)}_{u(m)} & )^{-1}( F_m( z ) ) \nabla_m F_m( z  ) D\chi^{-1}_{m_0}(\chi_{m_0}(m)): \\
& U_{m_0}(\epsilon') \times X_{m_0} \times Y_{u(m_0)} \to L(T_{m_0}M, Y_{u(m_0)}) ,
\end{split}\\
\begin{split}
\widehat{D}_m G_{m_0}(m, x,  y)  \triangleq  D( \varphi^{m_0}_{m} )^{-1}( & G_m( z ) )   \nabla_m G_m( z  ) D\chi^{-1}_{m_0}(\chi_{m_0}(m)): \\
& U_{m_0}(\epsilon') \times X_{m_0} \times Y_{u(m_0)} \to L(T_{m_0}M, Y_{m_0}) ,
\end{split}
\end{cases}
\end{equation}
where $ z = (\varphi^{m_0}_{m}(x), \phi^{u(m_0)}_{u(m)}(y)) $;
\begin{equation}\label{baseKR0}
\begin{cases}
\begin{split}
\widehat{K}_{m_0} (m,x) \triangleq  D(\phi^{m_0}_{m})^{-1} ( f_m(\varphi^{m_0}_m(x)) )  K_{m} & (\varphi^{m_0}_m(x)) D\chi^{-1}_{m_0}(\chi_{m_0}(m)): \\
& U_{m_0}(\epsilon') \times X_{m_0} \to L(T_{m_0}M, Y_{m_0}),
\end{split}\\
\begin{split}
\widehat{R}_{m_0} (m,x) \triangleq  D ( \varphi^{u(m_0)}_{u(m)} )^{-1}  ( x_m(\varphi^{m_0}_m (x &)) ) R_m ( \varphi^{m_0}_m(x) ) D\chi_{m_0}^{-1}(\chi_{m_0}(m)):\\
& U_{m_0}(\epsilon') \times X_{m_0} \to L(T_{m_0}M, X_{u(m_0)}).
\end{split}
\end{cases}
\end{equation}

By analogy with \eqref{K1elocal}, we know that
\begin{equation}\label{KK1local}
\begin{cases}
\begin{split}
\widehat{D}_m F_{m_0}(m, x, \hat{y} ) + D_y \widehat{F}_{m_0}  & (m, x, \hat{y} ) \left\{ \widehat{K}_{u(m_0)} ( u(m),\hat{x} ) \widehat{Du}_{m_0}(m) \right. \\
&\left. + \widehat{K}^1_{u(m_0)} (u(m),\hat{x}) \widehat{R}_{m_0} (m,x) \right\} = \widehat{R}_{m_0} (m,x),
\end{split}\\
\begin{split}
\widehat{D}_m G_{m_0}(m, x, \hat{y} ) + D_y \widehat{G}_{m_0}  & (m, x, \hat{y} ) \left\{  \widehat{K}_{u(m_0)} ( u(m),\hat{x} ) \widehat{Du}_{m_0}(m) \right. \\
&\left. + \widehat{K}^1_{u(m_0)} (u(m),\hat{x}) \widehat{R}_{m_0} (m,x) \right\} = \widehat{K}_{m_0} (m,x) ,
\end{split}
\end{cases}
\end{equation}
where $ \hat{x} = \widehat{x}_{m_0} (m,x) $, $ \hat{y} = \widehat{f}_{u(m_0)} ( u(m), \hat{x} ) $, $ \widehat{K}^1_{u(m_0)} (u(m),\hat{x}) = D_x \widehat{f}_{u(m_0)} (u(m),\hat{x}) $.

\begin{lem}\label{lem:final}
	Assume the following conditions hold:
	\begin{enumerate}[(a)]
		\item Let \textnormal{(H1c), (H2d), (H3), (H4c), (H5)} hold.
		\item $ F, G $ are $ C^1 $ and uniformly (locally) $ C^{1,1} $ around $ M_1 $ in the following sense: \eqref{ccc0} (see also \autoref{holsheaf}), \eqref{lipzz} ($ D^vF, D^vG \in C^{0,1} $) and \eqref{ligmm} ($ \nabla F, \nabla G \in C^{0,1} $) hold.
		\item (spectral gap condition) $ \lambda_s \lambda_u < 1 $, $ \lambda^2_s \lambda_u < 1 $, $ \lambda_s \lambda_u \mu < 1 $, $ \lambda^2_s \lambda_u \mu < 1 $, and
		
		$ \max\{ \frac{\mu}{\lambda_s}, \mu \}^{*\alpha} \lambda_s \lambda_u \mu < 1 $, where $ 0 < \alpha \leq 1 $. See \autoref{absgc}.
	\end{enumerate}
	Then there exists $ \epsilon'_1 > 0 $ small such that the following hold:

	\begin{enumerate}[(1)]
		\item We have
		\begin{equation}\label{finaleq}
		|\widehat{K}_{m_0}(m_1, x) - \widehat{K}_{m_0}(m_0, x)| \leq C |m_1 - m_0|^\alpha (|x|+1)^\alpha |x|^{1-\alpha},
		\end{equation}
		for all $ m_1 \in U_{m_0} (\epsilon'_1) $, $ x \in X_{m_0} $, $ m_0 \in M_1 $.
		The constant $ C $ depends on the constant $ \epsilon'_1 > 0 $ but not on $ m_0 \in M_1 $.
		\item Suppose $ \lambda_s < 1 $ and $ M $ is $ C^{1,1} $-uniform around $ M_1 $ (see \autoref{def:C11um}). If $ F, G $ satisfy the additional estimates
		\begin{align*}
		| \widehat{D}_m F_{m_0}(m_1, z) - \widehat{D}_m F_{m_0}(m_0, z) | \leq M_0 |m_1 - m_0||z|^\zeta, \\
		| \widehat{D}_m G_{m_0}(m_1, z) - \widehat{D}_m G_{m_0}(m_0, z) | \leq M_0 |m_1 - m_0||z|^\zeta,
		\end{align*}
		for all $ m_1 \in U_{m_0} (\epsilon') $, $ z_1, z_2 \in X_{m_0} \times Y_{u(m_0)} $, $ m_0 \in M_1 $, where $ 0 < \zeta \leq 1 $, and (a `better' spectral gap condition) $ \lambda_s \lambda_u < 1 $, $ \lambda_s \lambda_u \mu < 1 $, $ \ (\mu/\lambda^{1-\zeta}_s)^{*\alpha} \lambda_s \lambda_u \mu < 1 $, where $ 0 < \alpha \leq 1 $, then
		\[
		| \widehat{K}_{m_0}(m_1, x) - \widehat{K}_{m_0}(m_0, x) | \leq C |m_1 - m_0|^\alpha (|x|+|x|^\zeta)^\alpha |x|^{1-\alpha}
		\]
		for all $ m_1 \in U_{m_0} (\epsilon'_1) $, $ m_0 \in M_1 $, $ x \in X_{m_0} $.
		The case $ \zeta = 1 $ is satisfied, e.g., when $ F, G \in C^{2,1} $, $ \lambda_s \lambda_u < 1 $, $ \lambda_s \lambda_u \mu < 1 $, $ \mu^{*\alpha} \lambda_s \lambda_u \mu < 1 $, $ 0 < \alpha \leq 1 $; now
		\[
		| \widehat{K}_{m_0}(m_1, x) - \widehat{K}_{m_0}(m_0, x) | \leq C |m_1 - m_0|^\alpha |x|.
		\]
	\end{enumerate}
\end{lem}

The result like (1) was also obtained in \cite[Lemma 3, p. 82]{ZZJ14} for invertible maps in trivial bundles, where the authors used the Perron method; see also \cite[Theorem 1.3]{Sta99}. The result like (2) is well known in different settings; see e.g. \cite{PSW97}.

\begin{proof}[Proof of \autoref{lem:final}]
	We need the desired Lipschitz continuity of $ m \mapsto f_m(x) $, $ x \mapsto K^1_m(x), K_m(x) $ and $ m \mapsto K^1_m(x) $.
	\begin{enumerate}[(a)]
		\item From \autoref{lem:leaf1a} and \autoref{lem:holversheaf}, under the spectral gap condition $ \lambda_s \lambda_u < 1 $, $ \lambda^{2}_s \lambda_u < 1 $, $ \lambda_s \lambda_u \mu < 1 $, we see that
		\[
		|K^1_m(x_1) - K^1_m(x_2)| \leq C_0 |x_1 - x_2|, ~|K_m(x_1) - K_m(x_2)| \leq C_0 |x_1 - x_2|(|x_1|+1).
		\]
		\item From \autoref{lem:sheaf}, under \eqref{ccc0} and the spectral gap condition: $ \lambda_s \lambda_u < 1 $, $ \lambda_s \lambda_u \mu < 1 $, we get
		\[
		|\widehat{f}_{m_0}(m_1, x) - \widehat{f}_{m_0}(m_0, x)|\leq C_0 |m_1 - m_0||x|
		\]
		for every $ m_1 \in U_{m_0} (\epsilon_*) $, where $ \epsilon_* $ is small and independent of $ m_0 \in M_1 $.
		\item Similarly, from \autoref{lem:baseleaf}, under the spectral gap condition $ \lambda_s \lambda_u < 1 $, $ \lambda^{2}_s \lambda_u < 1 $, $ \lambda_s \lambda_u \mu < 1 $, $ \lambda^2_s \lambda_u \mu < 1 $, we have
		\[
		| \widehat{K}^1_{m_0}(m_1,x) - \widehat{K}^1_{m_0}(m_0,x) | \leq C_0 |m_1 - m_0| (|x| + 1)
		\]
		for $ m_1 \in U_{m_0}(\epsilon_*) $.
	\end{enumerate}

	First observe that by \eqref{KK1local},
	\begin{align*}
	& \widehat{K}_{m_0} (m_1,x) -  \widehat{K}_{m_0} (m_0,x) \\
	= &  \widehat{D}_m F_{m_0}(m_1, x, \hat{y}_1 ) - \widehat{D}_m F_{m_0}(m_0, x, \hat{y}_0 ) + \left\lbrace  D_y \widehat{G}_{m_0} ( m_1, x, \hat{y}_1 ) - D_y \widehat{G}_{m_0} ( m_0, x, \hat{y}_0 )    \right\rbrace \\
	& \quad \times \left\{ \widehat{K}_{u(m_0)} ( u(m_1), \hat{x}_1 ) \widehat{Du}_{m_0}(m_1) +  \widehat{K}^1_{u(m_0)} ( u(m_1), \hat{x}_1 )  \widehat{R}_{m_0} (m_1,x) \right\} \\
	& + D_y \widehat{G}_{m_0} ( m_0, x, \hat{y}_0 ) \left\lbrace  \widehat{K}_{u(m_0)} ( u(m_1), \hat{x}_1 ) \widehat{Du}_{m_0}(m_1) +  \widehat{K}^1_{u(m_0)} ( u(m_1), \hat{x}_1 )  \widehat{R}_{m_0} (m_1,x) \right. \\
	& \quad - \left. \widehat{K}_{u(m_0)} ( u(m_0), \hat{x}_0 ) \widehat{Du}_{m_0}(m_0) -  \widehat{K}^1_{u(m_0)} ( u(m_0), \hat{x}_0 )  \widehat{R}^0_{m_0} (m_0,x) \right\rbrace,
	\end{align*}
	for $ m_1 \in U_{m_0}(\epsilon_*) $, where $ \hat{x}_i = \widehat{x}_{m_0} (m_i,x)$, $ \hat{y}_i = \widehat{f}_{u(m_0)} ( u(m_i), \hat{x}_i ) $, $ i = 0,1 $. Here $ \epsilon_* $ is taken still smaller, so that the functions $ \alpha'', \beta'', \lambda''_s, \lambda''_u $ in the proof of \autoref{lem:sheaf}, satisfy \eqref{AB''} \eqref{kkr}, and the function $ \mu'' $ satisfies (by (H4c)) $ |\widehat{Du}_{m_0}(m_1)| \leq \mu''(m_0) $
	for $ m_1 \in U_{m_0}(\epsilon_*) $, and
	in addition these functions fulfill the spectral gap condition,
	\[
	\lambda''_s \lambda''_u < 1, \lambda''^2_s \lambda''_u < 1,  \lambda''_s \lambda''_u \mu'' < 1 , \lambda''^2_s \lambda''_u \mu'' < 1 ,  \max\{\mu''/\lambda''_s, \mu'' \}^{*\alpha} \lambda''_s \lambda''_u \mu'' < 1.
	\]
	We compute
	\begin{align*}
	& |\widehat{K}_{m_0} (m_1,x) -  \widehat{K}_{m_0} (m_0,x) | \\
	\leq & \widetilde{C} |m_1 - m_0|( 1 + |x| )
	+ \lambda''_u(m_0) | \widehat{K}_{u(m_0)} ( u(m_1), \hat{x}_1 ) \widehat{Du}_{m_0}(m_1) \\
	& \qquad - \widehat{K}_{u(m_0)} ( u(m_0), \hat{x}_0 ) \widehat{Du}_{m_0}(m_0) | \\
	& + \lambda''_u(m_0) |  \widehat{K}^1_{u(m_0)} ( u(m_1), \hat{x}_1 )  \widehat{R}_{m_0} (m_1,x) - \widehat{K}^1_{u(m_0)} ( u(m_0), \hat{x}_0 )  \widehat{R}_{m_0} (m_0,x) | \\
	\leq & \widetilde{C} |m_1 - m_0|( 1 + |x| ) + \lambda''_u(m_0) \mu''(m_0) | \widehat{K}_{u(m_0)} ( u(m_1), \hat{x}_1 ) - \widehat{K}_{u(m_0)} ( u(m_0), \hat{x}_1 ) | \\
	& + \lambda''_u(m_0) \beta''(u(m_0)) | \widehat{R}_{m_0} (m_1,x) - \widehat{R}_{m_0} (m_0,x) |.
	\end{align*}
	Similarly,
	\begin{align*}
	& |\widehat{R}_{m_0} (m_1,x) -  \widehat{R}_{m_0} (m_0,x) | \\
	\leq & \widetilde{C} |m_1 - m_0|( 1 + |x| )  + \alpha''(m_0) \mu''(m_0) | \widehat{K}_{u(m_0)} ( u(m_1), \hat{x}_1 ) - \widehat{K}_{u(m_0)} ( u(m_0), \hat{x}_1 ) | \\
	& + \alpha''(m_0) \beta''(u(m_0)) | \widehat{R}_{m_0} (m_1,x) - \widehat{R}_{m_0} (m_0,x) |.
	\end{align*}
	Thus,
	\begin{multline}\label{k011}
	|\widehat{K}_{m_0} (m_1,x) -  \widehat{K}_{m_0} (m_0,x)| \leq \widetilde{C} |m_1 - m_0|( 1 + |x| ) \\
	+ \frac{ \lambda''_u(m_0) \mu''(m_0) }{1-\alpha''(m_0)\beta''(u(m_0))} |  \widehat{K}_{u(m_0)} ( u(m_1), \hat{x}_1 ) - \widehat{K}_{u(m_0)} ( u(m_0), \hat{x}_1 ) |.
	\end{multline}

	Now we can apply the argument in \autoref{Appbb} to conclude the proof of (1) (see \autoref{argumentapp} \eqref{lem5}). The proof of (2) is the same as that of (1), where the following estimate is needed:
	\begin{align*}
	& |\widehat{K}_{m_0} (m_1,x) -  \widehat{K}_{m_0} (m_0,x)| \\
	\leq & \widetilde{C} |m_1 - m_0|( |x|^{\zeta} + |x| )  + \lambda''_u(m_0) \mu''(m_0) | \widehat{K}_{u(m_0)} ( u(m_1), \hat{x}_1 ) - \widehat{K}_{u(m_0)} ( u(m_0), \hat{x}_1 ) | \\
	& + \lambda''_u(m_0) \beta''(u(m_0)) | \widehat{R}_{m_0} (m_1,x) - \widehat{R}_{m_0} (m_0,x) |. 
	\end{align*}
\end{proof}

\begin{rmk}[Lipschitz continuity of connection] \label{lipconC}
	How about the $ C^{1,\gamma} $ continuity of $ f $? We need the Lipschitz continuity of the connections $ \mathcal{C}^X, \mathcal{C}^Y $ (see also \autoref{lipcon}). Let $ (U_{m_0}(\epsilon'), \varphi^{m_0}) \in \mathcal{A} $, and
	\begin{align*}
	\widehat{D}_m \varphi^{m_0} (m,x) & \triangleq (D\varphi^{m_0}_{m})^{-1} (\varphi^{m_0}_m(x)) \nabla_m \varphi^{m_0}_m(x) D\chi^{-1}_{m_0}(\chi_{m_0}(m)) \\
	& = \widehat{\varGamma}^{m_0}_{(m,x)} D\chi^{-1}_{m_0}(\chi_{m_0}(m)) :~ U_{m_0}(\epsilon'_{m_0}) \times X_{m_0} \to L(T_{m_0}M, X_{m_0}),
	\end{align*}
	where $ \widehat{\varGamma}^{m_0}_{(m,x)} $ is the Christoffel map in the bundle chart $ \varphi^{m_0} $ (see \autoref{def:connection}).
	We say the connection $ \mathcal{C}^X $ is \emph{uniformly (locally) Lipschitz} around $ M_1 $ with respect to $ \mathcal{A}, \mathcal{M} $ if
	\begin{gather*}
	\sup_{m_0 \in M}\sup_{m \in U_{m_0}(\epsilon')}\lip_{x} \widehat{\varGamma}^{m_0}_{(m,\cdot)} < \infty, ~
	|\widehat{D}_{m} \varphi^{m_0} (m_1,x) | \leq C|x| |m_1 - m_0|,~ m_1 \in U_{m_0}(\epsilon'_{m_0}),
	\end{gather*}
	where $ C > 0 $. (Note what this means for linear connections.)
	We have
	\begin{align*}
	D_{m} \widehat{f}_{m_0} (m,x)  = \nabla_{m} (\phi^{m_0}_{m})^{-1}(x'')  + D(\phi^{m_0}_{m})^{-1}(x'') \left\lbrace  \nabla_{m} f_{m}(x')
	+ Df_{m}(x') \nabla_{m}\varphi^{m_0}_{m} (x) \right\rbrace,
	\end{align*}
	where $ x' = \varphi^{m_0}_m(x) $, $ x'' = f_m(x') $. From this, one can deduce the H\"older continuity of $ m \mapsto D_{m} \widehat{f}_{m_0} (m,x) $, i.e.,
	\[
	m \mapsto D_{m} \widehat{f}_{m_0} (m,x) D\chi^{-1}_{m_0}(\chi_{m_0}(m))
	\]
	is H\"older, and in addition $ x \mapsto D_{m} \widehat{f}_{m_0} (m,x)  $ is (uniformly) Lipschitz from the $ C^{0,1} $ continuity of $ x \mapsto \nabla_m f_{m}(x), Df_{m}(x) $. Thus, by using the H\"older continuity of $ (m,x) \mapsto Df_{m}(x), \nabla_{m} f_{m}(x) $ (and the uniform Lipschitz continuity of $ \mathcal{C}^X, \mathcal{C}^Y $), the $ C^{1,\gamma} $ continuity of $ f $ is well understood.
\end{rmk}

\section{Appendix. Lipschitz continuity respecting base points}\label{lipbase}

One can describe the H\"older continuity of $ m \mapsto f_m(x), K^1_m(x), K_m(x) $ in a more classical way if some Lipschitz property of the transition maps is assumed.
Let $ M^\epsilon_1 = \bigcup_{m_0 \in M_1} U_{m_0}(\epsilon) $. Use $ M^{\epsilon_1}_1 $ instead of $ M_1 $ in all assumptions in \autoref{hverticalsheaf} to \autoref{HolderivativeBase} for some small $ \epsilon_1 > 0 $.

\begin{enumerate} [$ (\bullet) $]
	\item When (H1a) or (H1b) holds, we assume $ M $ is a \emph{uniformly} locally metrizable space (see \autoref{def:ulms}); note that if (H1c) holds, then $ M $ is already a uniformly locally metrizable space.
	Moreover, assume $ X, Y $ have \emph{$ \varepsilon $-almost $ C^{0,1} $-uniform trivializations} (i.e. ($ \Re $) below holds) if (H1a) holds, or \emph{$ \varepsilon $-almost $ C^{1,1} $-fiber-uniform trivializations} (i.e. ($ \Re $) and ($ \Im $) below hold) if (H1b) or (H1c) holds, on $ M^{\epsilon_1}_1 $ with respect to bundle atlases $ \mathcal{A}, \mathcal{B} $ (in (H2)), respectively; $ M $ is \emph{$ C^{1,1} $-uniform} around $ M_1 $ (i.e. ($ \aleph $) below holds) if (H1c) holds. (See also \autoref{bundleII}.) There is an $ \epsilon_0 > 0 $ such that the transition maps with respect to $ \mathcal{A}, \mathcal{B}, \mathcal{M} $,
	\begin{align*}
	\varphi^{m_0,m_1} = (\varphi^{m_1})^{-1} \circ \varphi^{m_0} : (W_{\epsilon_0}, d_{m_0}) \times X_{m_0} \to (W_{\epsilon_0}, d_{m_1}) \times X_{m_1} ,\\
	\phi^{m_0,m_1} = (\phi^{m_1})^{-1} \circ \phi^{m_0} : (W_{\epsilon_0}, d_{m_0}) \times Y_{m_0} \to (W_{\epsilon_0}, d_{m_1}) \times Y_{m_1}, \\
	\chi^{m_0,m_1} (m') = D\chi_{m_1}(m')(D\chi_{m_0}(m'))^{-1}: W_{\epsilon_{m_0}} \to L(T_{m_0}M, T_{m_1}M),
	\end{align*}
	where $ W_{\epsilon_0} = U_{m_0}(\epsilon_0) \cap U_{m_1}(\epsilon_0) \neq \emptyset $, satisfy the following Lipschitz conditions:
	\begin{gather*}
	\lip \varphi^{m_0,m_1}_{(\cdot)} (x) \leq \hat{\eta}^{X,0}_{m_0}(x),~
	\lip \phi^{m_0,m_1}_{(\cdot)} (y) \leq \hat{\eta}^{Y,0}_{m_0}(y), \tag{$ \Re $} \\
	\lip D_x\varphi^{m_0,m_1}_{(\cdot)} (x) \leq \hat{\eta}^{X,1}_{m_0}(x),
	\lip D_y\phi^{m_0,m_1}_{(\cdot)} (y) \leq \hat{\eta}^{Y,1}_{m_0}(y), \tag{$ \Im $} \\
	\lip\chi^{m_0,m_1} (\cdot) \leq M_0,\tag{$ \aleph $}
	\end{gather*}
	where the functions $ \hat{\eta}^{X,i}_{m_0}(\cdot), ~\hat{\eta}^{Y,i}_{m_0}(\cdot): X_{m_0} \to \mathbb{R}_{+} $, $ i = 0,1 $, satisfy
	\[
	\hat{\eta}^{X,0}_{m_0}(x) \leq M_0 |x|,~
	\hat{\eta}^{X,1}_{m_0}(x) \leq M_0,~
	\hat{\eta}^{Y,0}_{m_0}(y) \leq M_0 |y|,~
	\hat{\eta}^{Y,1}_{m_0}(y) \leq M_0,
	\]
	where $ M_0 > 0 $ is a constant independent of $ m_0 \in M^{\epsilon_1}_1, x,y $.
\end{enumerate}
Take small $ \epsilon'_2 < \epsilon_1 $. Now we have the following statements:
\begin{enumerate}[(a)]
	\item In \autoref{lem:sheaf}, \eqref{cc3} can be changed to
	\[
	|\widehat{f}_{m_0}(m_1, x) - \widehat{f}_{m_0}(m_2, x)| \leq C |m_1 - m_2|^{\gamma \alpha} |x|^{\zeta \alpha + 1 - \alpha},
	\]
	for all $ m_1, m_2 \in U_{m_0} (\varepsilon^*_1), x \in X_{m_0}, m_0 \in M^{\epsilon'_2}_1 $, provided 
	\[
	|m_i - m_0|^\gamma |x|^\zeta \leq \hat{r}\min\{ |x|, |x|^{\zeta c -(c-1)} \}, ~i = 1,2.
	\]
	\item In \autoref{lem:base0}, \eqref{k1optimal} can be taken as
	\[
	| \widehat{K}^1_{m_0} (m_1, i_X{(m_0)}) - \widehat{K}^1_{m_0} (m_2, i_X{(m_0)}) | \leq C |m_1 - m_2|^{\gamma\alpha},
	\]
	for all $ m_1, m_2 \in U_{m_0} (\varepsilon^*_1) $, $ m_0 \in M^{\epsilon'_2}_1 $.

	In \autoref{lem:baseleaf}, \eqref{bleaf} is taken as
	\[
	| \widehat{K}^1_{m_0} (m_1, x) - \widehat{K}^1_{m_0} (m_2, x) | \leq C\left\lbrace |m_1 - m_2|^\alpha (|x|+1) + (|m_1 - m_2|^\alpha |x|)^\beta \right\rbrace,
	\]
	for all $ m_1, m_2 \in U_{m_0} (\varepsilon^*_1), x \in X_{m_0} $, $ m_0 \in M^{\epsilon'_2}_1 $.

	\item In \autoref{lem:final}, \eqref{finaleq} is now strengthened to
	\[
	|\widehat{K}_{m_0}(m_1, x) - \widehat{K}_{m_0}(m_2, x)| \leq C |m_1 - m_2|^\alpha (|x|+1)^\alpha |x|^{1-\alpha},
	\]
	for all $ m_1, m_2 \in U_{m_0} (\epsilon'_1) $, $ x \in X_{m_0} $, $ m_0 \in M^{\epsilon'_2}_1 $.
\end{enumerate}
Here $ |m_1 - m_2| \triangleq d_{m_0}(m_1, m_2) $, where $ d_{m_0} $ is the metric in $ U_{m_0} $.

\begin{proof}
	As an example, we only consider the case $ m \mapsto f_m(x) $ in addition to that in \autoref{lem:sheaf}, $ \alpha = 1 $, $ \gamma = 1 $, $ \zeta = 1 $.
	Let $ m'_0 \in U_{m_0}(\epsilon_*) $, $ m_0 \in M^{\epsilon'_2}_1 $, $ \epsilon_* < \epsilon_2 $, where $ \epsilon_*, \epsilon'_2 (< \epsilon_2) $ are chosen such that $ U_{m_0}(\epsilon_*) \subset M^{\epsilon_2}_1 $ for all $ m_0 \in M^{\epsilon'_2}_1 $. Note that we have the following local change of coordinates:
	\[
	\widehat{f}_{m_0}(m,x) = \phi^{m'_0, m_0}_{m} \circ \widehat{f}_{m'_0} (m, \varphi^{m_0, m'_0}_{m}(x)).
	\]
	Let $ m_1, m'_0 \in U_{m_0}(\epsilon_*) $. We compute
	\begin{align*}
	& |\widehat{f}_{m_0}(m_1, x) - \widehat{f}_{m_0}(m'_0, x)|\\
	= & \left| \phi^{m'_0, m_0}_{m_1} \circ \widehat{f}_{m'_0} (m_1,  \varphi^{m_0, m'_0}_{m_1}(x)) - \phi^{m'_0, m_0}_{m'_0} \circ \widehat{f}_{m'_0}  (m'_0, \varphi^{m_0, m'_0}_{m'_0}(x)) \right| \\
	\leq & \left| \phi^{m'_0, m_0}_{m_1} \circ \widehat{f}_{m'_0}  (m_1, \varphi^{m_0, m'_0}_{m_1}(x)) - \phi^{m'_0, m_0}_{m'_0} \circ \widehat{f}_{m'_0} (m_1, \varphi^{m_0, m'_0}_{m_1}(x)) \right| \\
	& + \left| \phi^{m'_0, m_0}_{m'_0} \circ \widehat{f}_{m'_0}  (m_1, \varphi^{m_0, m'_0}_{m_1}(x)) - \phi^{m'_0, m_0}_{m_0} \circ \widehat{f}_{m'_0} (m'_0, \varphi^{m_0, m'_0}_{m'_0}(x))  \right| \\
	\leq & M_0 |\widehat{f}_{m'_0} (m_1, \varphi^{m_0, m'_0}_{m_1}(m_1, x))| |m_1 - m'_0|_{m'_0} + \lip \phi^{m'_0, m_0}_{m'_0}(\cdot)\\
	& \qquad \cdot \left| \widehat{f}_{m'_0} (m_1, \varphi^{m'_0, m_0}_{m_1}(x)) - \widehat{f}_{m'_0} (m'_0, \varphi^{m_0, m'_0}_{m'_0}(x)) \right| \\
	\leq & \widetilde{C} |m_1 - m'_0|_{m_0} |x| + \lip \phi^{m'_0, m_0}_{m'_0}(\cdot) \left\{ \left| \widehat{f}_{m'_0} (m_1, \varphi^{m_0, m'_0}_{m_1}(x)) - \widehat{f}_{m'_0} (m'_0, \varphi^{m_0, m'_0}_{m_1}(x)) \right| \right. \\
	& \qquad \left. + \left| \widehat{f}_{m'_0} (m'_0, \varphi^{m_0, m'_0}_{m_1}(x)) - \widehat{f}_{m'_0} (m'_0, \varphi^{m_0, m'_0}_{m'_0}(x)) \right| \right\}\\
	\leq & C_0 |m_1 - m'_0| |x|,
	\end{align*}
	for $ m_1, m'_0 \in U_{m_0}(\epsilon_*) $, where $ \epsilon_* $ is taken smaller than $ \epsilon^*, \epsilon', \epsilon_2 $, independent of $ m_0 \in M_1 $, and $ C_0 > 0 $ is independent of $ m_0 \in M_1 $.
\end{proof}

\section{Continuity properties of $ f $} \label{continuityf}

For a bundle map $ g: X \to Y $, one can talk about its fiber-continuity (i.e. $ x \mapsto g_m(x) $) and base-continuity (i.e. $ m \mapsto g_m(x) $). The $ C^0 $ continuity of $ x \mapsto g_m(x) $ or $ (m,x) \mapsto g_m(x) $  is easy. Let us consider the continuity in the uniform sense. The following different types of continuity usually arise in applications; see also \autoref{uniformC} and \autoref{tensorV}.
\begin{enumerate}[(a)]
	\item (The uniform $ C^0 $-fiber case) (i) For each $ m \in M $, $ x \mapsto g_m(x) $ is uniformly continuous; (ii) $ x \mapsto g_m(x) $ is uniformly continuous uniformly for $ m $, i.e., $ x \mapsto g_m(x) $, $ m \in M $, are equicontinuous.
	\item (The uniform $ C^0 $-base case) (i) $ m \mapsto g_m(\cdot) $ is \emph{continuous} or \emph{uniformly continuous} in \emph{$ C^0 $-topology on bounded sets or in the whole space}. This means for the local representation $ \widehat{g} $ of $ g $ with respect to the bundle atlases of $ \mathcal{A}, \mathcal{B} $, one has $ \mathcal{L} |\widehat{g}_{m_0}(m',x) - \widehat{g}_{m_0}(m_0,x)| = 0 $, where $ \mathcal{L} $ stands for the following four different limits, respectively:
	\begin{equation}\label{equ:limits}
	\left\{
	\begin{gathered}
	\mathcal{L}^{base}_1: \lim_{r \to \infty} \sup_{m_0 \in M} \limsup_{m' \to m_0} \sup_{|x|\leq r}, ~\mathcal{L}^{base}_{1,u}: \lim_{r \to \infty} \limsup_{\epsilon \to 0} \sup_{m_0 \in M_1} \sup_{m' \in U_{m_0}(\epsilon)} \sup_{|x|\leq r},\\
	\mathcal{L}^{base}_2: \sup_{m_0 \in M} \limsup_{m' \to m_0} \sup_{x \in X_{m_0}}, ~\mathcal{L}^{base}_{2,u}: \limsup_{\epsilon \to 0} \sup_{m_0 \in M_1} \sup_{m' \in U_{m_0}(\epsilon)} \sup_{x \in X_{m_0}}.
	\end{gathered}
	\right.
	\end{equation}
	(ii) $ m \mapsto g_m(\cdot) $ is \emph{continuous} or \emph{uniformly continuous} in \emph{$ C^1 $-topology on bounded sets or in the whole space}. This means that $ \mathcal{L} |\widehat{g}_{m_0}(m',x) - \widehat{g}_{m_0}(m_0,x)| = 0 $ and $ \mathcal{L}|\widehat{D}_xg_{m_0}(m',x) - \widehat{D}_xg_{m_0}(m_0,x)| = 0 $, where $ \widehat{D}_xg_{m_0} $ is the local representative of the fiber derivative $ D_xg $, i.e., $ \widehat{D}_xg_{m_0} = D_x \widehat{g}_{m_0} $.
\end{enumerate}

In the following, we will consider the continuity of $ (m,x) \mapsto f_m(x) $, $ K^1_m(x), K_m(x) $. Unlike the H\"older case, the assumptions on $ M $, $ X \times Y $ (as well as $ u $) can be pointwise.
We write $ \widehat{F}_{m_0}, \widehat{G}_{m_0},  \widehat{f}_{m_0}, \widehat{x}_{m_0} $ for the local representations of $ F, G $, $ f, x_{(\cdot)}(\cdot) $ at $ m_0 \in M_1 $ with respect to $ \mathcal{A}, \mathcal{B} $; see \eqref{localrepre}.

We say $ m' \to m_0 \in M_1 $ is in the \emph{uniform} sense if the limit $ \lim_{m' \to m_0} $ is replaced by $ \limsup_{\epsilon \to 0} \sup_{m_0 \in M_1} \sup_{m' \in U_{m_0}(\epsilon)} $.

\begin{lem}[Continuity of $ f $]\label{lem:continuity_f}
	Assume the following conditions hold:
	\begin{enumerate}[(a)]
		\item Let \textnormal{(H1a), (H2)(iii,v1$ ' $), (H3$ ' $), (H4), (H5$ ' $)} hold.
		\item $ F, G $ are continuous and (spectral gap condition) $ \lambda_s \lambda_u < 1 $; see \autoref{absgc}.
	\end{enumerate}
	Then
	\begin{enumerate}[(1)]
		\item  $ f : X \to Y $ is continuous. Moreover, the stronger the continuity of $ F,G $, the stronger that of $ f $, i.e.,
		\item If for any $ r > 0 $,
		\begin{equation}\label{ukk}
		\left\{
		\begin{split}
		\lim_{m' \to m_0} \sup_{0 \neq |z - {i}^1_{m_0}| \leq r} \frac{|\widehat{F}_{m_0}(m', z) - \widehat{F}_{m_0}(m_0, z)|}{|z - {i}^1_{m_0}|} = 0,\\
		\lim_{m' \to m_0} \sup_{0 \neq |z - {i}^1_{m_0}| \leq r} \frac{|\widehat{G}_{m_0}(m', z) - \widehat{G}_{m_0}(m_0, z)|}{|z - {i}^1_{m_0}|} = 0,
		\end{split}
		\right.
		\end{equation}
		where $ {i}^1_{m_0} = ({i}_X({m_0}), {i}_Y(u(m_0))) $, $ z \in X_{m_0} \times Y_{u(m_0)} $, $ m_0 \in M_1 $, then
		\begin{equation}\label{fkk}
		\lim_{m' \to m_0} \sup_{0 \neq |x - {i}_X({m_0})| \leq r} \frac{|\widehat{f}_{m_0}(m', x) - \widehat{f}_{m_0}(m_0, x)|}{|x - {i}_X({m_0})|} = 0.
		\end{equation}
		If \eqref{ukk} holds for $ r = \infty $, so does \eqref{fkk}.

		\item Assume that $ i: M \to X \times Y $ is uniformly continuous around $ M_1 $ (see \autoref{ucontinuous}) and \textnormal{(H3$ '' $) (H4$ ' $)} hold. Then if \eqref{ukk} holds with $ m' \to m_0 \in M_1 $ in the uniform sense, so does \eqref{fkk}.
	\end{enumerate}
\end{lem}
\begin{proof}
	We first show $ f $ is continuous at $ (m_0,x) $, $ m_0 \in M_1 $, $ x \in X_{m_0} $. It suffices to show for fixed $ x $, $ m \mapsto \widehat{f}_{m_0}(m, x) $ is continuous at $ m_0 $, since $ x \mapsto \widehat{f}_{m_0}(m, x) $ is Lipschitz locally uniformly for $ m $. By (H5$ ' $), we have $ \limsup_{m' \to m_0 } |\kappa(m') - \kappa(m_0)| \leq \varepsilon $ where $ \kappa(\cdot) $ is taken to be the function in the (A$ ' $)(B) (or (A)(B)) condition; set $ \kappa''(m_0) = \kappa(m_0) + \varepsilon $.
	From the computation in the proof \autoref{lem:sheaf}, we have
	\begin{multline*}
	\limsup_{m' \to m_0} |\widehat{f}_{m_0}(m',x) - \widehat{f}_{m_0}(m_0,x)| \leq \frac{\lambda''_u(m_0)}{1-\alpha''(m_0)\beta''(u(m_0))} \\ \cdot \limsup_{m' \to m_0} |  \widehat{f}_{u(m_0)} (u(m'), \widehat{x}_{m_0} (m_0,x) ) - \widehat{f}_{u(m_0)} (u(m_0), \widehat{x}_{m_0} (m_0,x) ) |.
	\end{multline*}

	Set $ \theta = \sup_{m_0 \in M_1}\frac{\lambda''_s(m_0) \lambda''_u(m_0)}{1-\alpha''(m_0)\beta''(u(m_0))} < 1 $ (by taking $ \varepsilon $ small). Then
	\begin{align*}
	& \sup_{m_0 \in M_1} \sup_{x \in X_{m_0}} \limsup_{m' \to m_0} \frac{|\widehat{f}_{m_0}(m',x) - \widehat{f}_{m_0}(m_0,x)|}{|x - {i}_X(m_0)|} \\
	\leq & \theta \sup_{m_0 \in M_1} \sup_{x \in X_{m_0}} \limsup_{m' \to m_0} \frac{| \widehat{f}_{u(m_0)} (u(m'), \widehat{x}_{m_0} (m_0,x) ) - \widehat{f}_{u(m_0)} (u(m_0), \widehat{x}_{m_0} (m_0,x) ) |}{|\widehat{x}_{m_0} (m_0,x) - {i}_X(u(m_0)) |} \\
	\leq & \theta \sup_{m_0 \in M_1} \sup_{x \in X_{m_0}} \limsup_{m' \to m_0} \frac{|\widehat{f}_{m_0}(m',x) - \widehat{f}_{m_0}(m_0,x)|}{|x - {i}_X(m_0)|}.
	\end{align*}
	Note that
	\[
	\sup_{m_0 \in M_1} \sup_{x \in X_{m_0}} \limsup_{m' \to m_0} \frac{|\widehat{f}_{m_0}(m',x) - \widehat{f}_{m_0}(m_0,x)|}{|x - {i}_X(m_0)|} < \infty,
	\]
	which follows from
	\[
	|\widehat{f}_{m_0}(m',x) - \widehat{f}_{m_0}(m_0,x)| \leq |\widehat{f}_{m_0}(m',x) - {i}_Y(m_0)| + |{i}_Y(m_0) - \widehat{f}_{m_0}(m_0,x)|,
	\]
	and
	\begin{align*}
	& \limsup_{m' \to m_0} |\widehat{f}_{m_0}(m',x) - {i}_Y(m_0)| \\
	\leq & \limsup_{m' \to m_0} \left\lbrace | \widehat{f}_{m_0}(m',x) - \widehat{f}_{m_0}(m', \widetilde{i}^{m_0}_X(m'))  |  +  | \widehat{f}_{m_0}(m', \widetilde{i}^{m_0}_X(m')) - {i}_Y(m_0) |  \right\rbrace  \\
	\leq & \beta'(m_0) \limsup_{m' \to m_0} |x - \widetilde{i}^{m_0}_X(m') | + \limsup_{m' \to m_0} | (\phi^{m_0}_{m'})^{-1} (i_Y(m')) - {i}_Y(m_0) |\\
	\leq & \beta'(m_0) | x - {i}_X(m_0) |,
	\end{align*}
	where $ \widetilde{i}^{m_0}_X(m') = (\varphi^{m_0}_{m'})^{-1}(i_X(m')) $. Since $ \theta < 1 $, we have
	\[
	\lim_{m' \to m_0} |\widehat{f}_{m_0}(m',x) - \widehat{f}_{m_0}(m_0,x)| = 0.
	\]

	Thus we have shown $ f $ is continuous at $ (m_0, x) $, $ m_0 \in M_1 $. We need to show $ f $ is continuous everywhere. This follows by the same argument as in the last part proof of \autoref{smoothbase}.

	Let $ m_0 \in M $, and choose $ C^0 $ bundle charts $ (U_0, \varphi^0) $, $ (U_0, \phi^0) $ at $ m_0 $ of $ X, Y $, respectively, such that $ u(U_0) \subset V_{u(m_0)} \subset U_{u(m_0)} $. Note that $ u(m_0) \in M_1 $. Consider $ F, G, f, x_{(\cdot)}(\cdot) $ in local bundle charts $ \varphi^0: U_0 \times X_{m_0} \to X $, $ \phi^0: U_0 \times Y_{m_0} \to Y $, $ \varphi^{u(m_0)}: V_{u(m_0)} \times X_{u(m_0)} \to X $, $ \phi^{u(m_0)}: V_{u(m_0)} \times Y_{u(m_0)} \to Y $. That is (see also \eqref{local11}),
	\begin{equation*}
	\begin{cases}
	\begin{split}
	\widehat{F}'_{m_0} (m, x, y) \triangleq (\varphi^{u(m_0)}_{u(m)})^{-1} \circ F_m \circ ( \varphi^{0}_m(x) , \phi^{u(m_0)}_{u(m)} & (y) ) : \\
	& U_{0} \times X_{m_0} \times Y_{u(m_0)} \to X_{u(m_0)}  ,
	\end{split}\\
	\widehat{G}'_{m_0} (m, x, y) \triangleq (\phi^{0}_{m})^{-1} \circ G_m \circ ( \varphi^{0}_m(x) , \phi^{u(m_0)}_{u(m)}(y) ) : U_{0} \times X_{m_0} \times Y_{u(m_0)} \to Y_{m_0} ,\\
	\widehat{f}'_{0}(m, x) \triangleq (\phi^{0}_m)^{-1} \circ f_m \circ ( \varphi^{0}_m(x) ) :  U_{0} \times X_{m_0} \to Y_{m_0} , \\
	\widehat{f}_{u(m_0)}(m, x) \triangleq (\phi^{u(m_0)}_{u(m)})^{-1} \circ f_{u(m)} \circ ( \varphi^{u(m_0)}_{u(m)}(x) ) :  U_{0} \times X_{u(m_0)} \to Y_{u(m_0)} ,\\
	\widehat{x}'_{m_0}(m, x) \triangleq (\varphi^{u(m_0)}_{u(m)})^{-1} \circ x_{m} (\varphi^{0}_m(x)) : U_{0} \times X_{m_0} \to X_{u(m_0)}. \\
	\end{cases}
	\end{equation*}
	Then we see \eqref{localeq} holds.

	By (H5$ ' $), we can take $ U_0 $ smaller such that if $ m \in U_0 $, then
	\[
	\lip \widehat{F}'_{m_0} ( m, x, \widehat{f}_{u(m_0)}(u(m), \cdot) )  < 1.
	\]
	As $ \widehat{f}_{u(m_0)} (\cdot, \cdot ) $ is continuous at $ (u(m_0),x) $, by \autoref{lem:fconx} and \eqref{localeq}, $ \widehat{f}'_{0} (\cdot, \cdot) $ is continuous at $ (m_0, x) $, and thus the proof of (1) is complete.

	To prove (2), consider the following limits (the uniformity of $ m' \to m_0 $ is the same):
	\[
	\lim_{r \to \infty} \sup_{m_0 \in M_1}\limsup_{m' \to m_0 } \sup_{|x - \widehat{i}_X(m_0)| \leq r}  \frac{|\widehat{f}_{m_0}(m',x) - \widehat{f}_{m_0}(m_0,x)|}{|x - {i}_X(m_0)|},
	\]
	\[
	\sup_{m_0 \in M_1}\limsup_{m' \to m_0 } \sup_{x}  \frac{|\widehat{f}_{m_0}(m',x) - \widehat{f}_{m_0}(m_0,x)|}{|x - {i}_X(m_0)|}.
	\]
	Also, note that by \eqref{ukk},
	\[
	\sup_{m_0 \in M_1}  \limsup_{m' \to m_0} \sup_{x \in X_{m_0}}\frac{|\widehat{f}_{m_0}(m',x) - \widehat{f}_{m_0}(m_0,x)|}{|x - {i}_X(m_0)|} < \infty.
	\]
	Using the same argument as in (1), one completes the proof.
\end{proof}

We write $ \widehat{K}^1_{m_0}(\cdot, \cdot), \widehat{R}^1_{m_0}(\cdot, \cdot) $ for the local representations of $ K^1, R^1 $ at $ m_0 \in M_1 $ with respect to $ \mathcal{A}, \mathcal{B} $, respectively; see \eqref{K1R1}.

\begin{lem}[Continuity of $ K^1_m(\cdot) $]\label{lem:K1leafc}
	Assume the following conditions hold:
	\begin{enumerate}[(a)]
		\item Let \textnormal{(H1a), (H2)(i, iv, v1$ ' $), (H3$ ' $), (H4), (H5$ ' $)} hold.
		\item Assume $ F, G $ are $ C^0 $ and $ C^1 $-fiber with continuous fiber derivatives $ D^vF, D^vG $ and (spectral gap condition) $ \lambda_s \lambda_u < 1 $; see \autoref{absgc}.
	\end{enumerate}
	Then:
	\begin{enumerate}[(1)]
		\item The bundle map $ K^1 \in L_{f}(\Upsilon^V_X, \Upsilon^V_Y) $ obtained in \autoref{lem:leaf1} is $ C^0 $.

		\item If for every $ m \in M $, $ DF_m(\cdot), DG_m(\cdot) $ are uniformly continuous, then so is $ K^1_m(\cdot) $.
		If $ DF_m(\cdot), DG_m(\cdot) $, $ m \in M $, are equicontinuous, then so are $ K^1_m(\cdot) $, $ m \in M $.
	\end{enumerate}
\end{lem}
\begin{proof}
	(1) That $ K^1 \in L(\Upsilon^V_X, \Upsilon^V_Y) $ is $ C^0 $ means $ \widehat{K}^1_{m_0}(\cdot, \cdot) $, the local representation of $ K^1 $ at $ m_0 $ (see e.g. \eqref{K1R1}), is continuous at $ (m_0,x) $ for $ m_0 \in M $, $ x \in X_{m_0} $. Let
	\[
	E^{L,c}_1 = \{ K^1 \in E^L_1: K^1 ~\text{is $ C^0 $} \}.
	\]
	Then one can easily verify that $ E^{L,c}_1 $ is closed in $ E^L_1 $ (using the fact that the convergence in $ E^L_1 $ is uniform) and $ \varGamma^{1} $, the graph transform defined in the proof of \autoref{lem:leaf1}, satisfies $ \varGamma E^{L,c}_1 \subset E^{L,c}_1 $ (by \eqref{localleaf1}, \autoref{lem:fconx} and the continuity of $ f $). Now we can deduce that $ K^1 \in E^{L,c}_1 $.

	(2) The proof is essentially the same as in the proof of \autoref{lem:continuity_f}, by using the computation in the proof of \autoref{lem:leaf1a} and considering the following limits, respectively:
	\[
	\sup_{m \in M} \lim_{r \to 0^+ } \sup_{|x_1 - x_2| \leq r} |K^1_m(x_1) - K^1_m(x_2)|, ~
	\lim_{r \to 0^+ } \sup_{m \in M} \sup_{|x_1 - x_2| \leq r} |K^1_m(x_1) - K^1_m(x_2)|.
	\]
	Note that $ \sup_{m \in M} \sup_{x_1, x_2} |K^1_m(x_1) - K^1_m(x_2)| < \infty $. So the above limits are both finite, and finally equal to $ 0 $.
\end{proof}

\begin{lem}[Continuity of $ m \mapsto K^1_m(x) $]\label{lem:baseK1}
	Assume the following conditions hold:
	\begin{enumerate}[(a)]
		\item Let \textnormal{(H1a), (H2)(i, iv, v2$ ' $), (H3), (H4), (H5$ ' $)} hold.
		\item Assume $ F, G $ are $ C^0 $ and (spectral gap condition) $ \lambda_s \lambda_u < 1 $; see also \autoref{absgc}.
	\end{enumerate}
	\begin{enumerate}[(1)]
		\item Suppose $ DF_m(\cdot), DG_m(\cdot) $, $ m \in M $, are uniformly continuous (resp. equicontinuous with additionally \textnormal{(H2)(v2), (H4$ ' $), (H5)} holding).
		If $ m \mapsto DF_{m}(\cdot), DG_{m}(\cdot) $ are continuous (resp. uniformly continuous) around $ M_1 $ in $ C^0 $-topology on bounded sets (or in the whole space), so is $ m \mapsto K^1_{m}(\cdot) $; see \autoref{def:uD}.

		\item Assume \textnormal{(H4$ ' $), (H5)} instead of \textnormal{(H4), (H5$ ' $)}. If $ m \mapsto DF_m(i^1(m)), DG_m(i^1(m)) $ are uniformly continuous around $ M_1 $ (see \autoref{ucontinuous}), so is $ m \mapsto K^1_m(i_X(m)) $, where $ i^1(m) = (i_X(m), i_Y(u(m))) $.
	\end{enumerate}
\end{lem}
\begin{proof}
	Note that we have assumed $ i $ is a $ 0 $-section. We first show that, under (1), \eqref{ukk} holds. This is easy from
	\[
	\widehat{F}_{m_0}(m', z) - \widehat{F}_{m_0}(m_0, z) = \int^{1}_{0} \left\lbrace  D_z\widehat{F}_{m_0}(m', z_t) - D_z\widehat{F}_{m_0}(m_0, z_t) \right\rbrace ~\mathrm{d}t~ (z - i^1(m_0)),
	\]
	where $ z_t = tz+(1-t)i^1(m_0) $, $ i^1(m_0) = (i_X(m_0), i_Y(u(m_0))) $, and similarly for $ \widehat{G}_{m_0} $. So by \autoref{lem:continuity_f}, $ \widehat{f}_{m_0} $ is continuous at $ (m_0,x) $ uniformly for $ x $ belonging to any bounded set of $ X_{m_0} $. Using the above fact, \autoref{lem:K1leafc} (by the assumption (b) on $ F,G $), and the computation in the proof of \autoref{lem:baseleaf} (i.e. \eqref{localleaf1}), we see that
	\begin{align*}
	& \limsup_{m' \to m_0} \sup_{|x| \leq r} |\widehat{K}^1_{m_0}(m',x) - \widehat{K}^1_{m_0}(m_0,x)| \\
	\leq & \frac{\lambda_s(m_0)\lambda_u(m_0)}{1-\alpha(m_0)\beta'(u(m_0))} \limsup_{m' \to m_0} \sup_{|x| \leq r} |\widehat{K}^1_{u(m_0)}(m',\widehat{x}_{m_0}(m_0,x)) - \widehat{K}^1_{u(m_0)}(m_0,\widehat{x}_{m_0}(m_0,x))|.
	\end{align*}
	Thus,
	\begin{align*}
	& \limsup_{r \to \infty} \sup_{m_0 \in M_1} \limsup_{m' \to m_0} \sup_{|x| \leq r} |\widehat{K}^1_{m_0}(m',x) - \widehat{K}^1_{m_0}(m_0,x)| \\
	\leq & \theta \limsup_{r \to \infty} \sup_{m_0 \in M_1} \limsup_{m' \to m_0} \sup_{|x| \leq r} |\widehat{K}^1_{m_0}(m',x) - \widehat{K}^1_{m_0}(m_0,x)|
	\end{align*}
	where $ \theta = \sup_{m_0 \in M_1}\frac{\lambda_s(m_0) \lambda_u(m_0)}{1-\alpha(m_0)\beta'(u(m_0))} < 1 $, which gives for any $ r > 0 $,
	\[
	\limsup_{m' \to m_0} \sup_{|x| \leq r} |\widehat{K}^1_{m_0}(m',x) - \widehat{K}^1_{m_0}(m_0,x)| = 0.
	\]
	The proof of the case $ Z^1_{m_0} = X_{m_0} \times Y_{u(m_0)} $ is similar by using $ \mathcal{L}^{base}_{2} $ instead of $ \mathcal{L}^{base}_{1} $; for the uniform case, use the limits $ \mathcal{L}^{base}_{1,u} $, $ \mathcal{L}^{base}_{2,u} $ instead of $ \mathcal{L}^{base}_{1} $, $ \mathcal{L}^{base}_{2} $ (see \eqref{equ:limits}). The proof of (2) is much easier than that of (1) by taking $ x = i_X(m_0) $ in the above analysis.
\end{proof}

\begin{lem}[Continuity of $ K_m(\cdot) $] \label{lem:leafK}
	Assume the following:
	\begin{enumerate}[(a)]
		\item Let \textnormal{(H1), (H2)(i, ii, v1$ ' $), (H3), (H4b), (H5$ ' $)} hold.
		\item Let $ F, G $ be $ C^1 $ and satisfy
		\begin{equation*}
		|\nabla_{m_0} F_{m_0}(z)| \leq M_0|z|,~ |\nabla_{m_0} G_{m_0}(z)| \leq M_0|z|,
		\end{equation*}
		for all $ {m_0} \in M_1 $ and $ z \in X_{m_0} \times Y_{u({m_0})} $, where $ M_0 $ is a constant independent of $ m $.
		\item (spectral gap condition) $ \lambda_s \lambda_u < 1 $, $ \lambda_s \lambda_u \mu < 1 $; see also \autoref{absgc}.
	\end{enumerate}
	Then the following results hold:
	\begin{enumerate}[(1)]
		\item There exists a unique $ C^0 $ vector bundle map $ K \in L_f(\Upsilon^H_X, \Upsilon^V_Y) $ such that 
		\[
		\sup_{m \in M_1}\sup_{x}\frac{|K_m(x)|}{|x|} < \infty
		\]
		and \eqref{mainMFW} is satisfied.

		\item
		If $ DF_m, DG_m, \nabla_{m}F_m, \nabla_mG_m $ satisfy
		\begin{equation}\label{DFDG1}
		\mathcal{L} \frac{|DF_m(z_1)- DF_m(z_2)|}{|z_1|+|z_2|} = 0,~
		\mathcal{L} \frac{|DG_m(z_1)- DG_m(z_2)|}{|z_1|+|z_2|} =0,
		\end{equation}
		\begin{equation}\label{DFDG2}
		\mathcal{L} \frac{|\nabla_mF_m(z_1)- \nabla_mF_m(z_2)|}{|z_1|+|z_2|} =0,~
		\mathcal{L} \frac{|\nabla_mG_m(z_1)- \nabla_mG_m(z_2)|}{|z_1|+|z_2|} =0,
		\end{equation}
		then
		\[
		\mathcal{L} \frac{|K_m(x_1)- K_m(x_2)|}{|x_1|+|x_2|} =0,
		\]
		where $ \mathcal{L} $ stands for the following two limits, respectively: 
		\[
		\sup\limits_{m \in M_1}\lim\limits_{r \to 0^+} \sup\limits_{|z_1 - z_2| \leq r },~\lim\limits_{r \to 0^+}\sup\limits_{m \in M_1} \sup\limits_{|z_1 - z_2| \leq r }.
		\]
		(For $ K $, use $ x_1, x_2 $ instead of $ z_1, z_2 $ in the supremum.) In particular, this means $ K_m(\cdot) $, $ m \in M_1 $ are uniformly continuous and equicontinuous, respectively.
	\end{enumerate}
\end{lem}

\begin{proof}
	The conclusion (1) has already been proved in \autoref{smoothbase}.
	To prove (2), we can use almost the same strategy as in proving \autoref{lem:K1leafc}. Consider the limits
	\[
	\mathcal{L} \frac{|K_m(x_1)- K_m(x_2)|}{|x_1|+|x_2|}.
	\]
	Using the continuity of $ K^1 $ (by \autoref{lem:K1leafc} and noting that \eqref{DFDG1} implies the condition on $ DF_m, DG_m $ in \autoref{lem:K1leafc} (2)), the continuity of $ f $ (by \autoref{lem:continuity_f}) and the computation in the proof of \autoref{lem:holversheaf} (by \eqref{mainMFW}), one can deduce that
	\[
	\mathcal{L} \frac{|K_m(x_1)- K_m(x_2)|}{|x_1|+|x_2|} \leq \theta_0 \mathcal{L} \frac{|K_m(x_1)- K_m(x_2)|}{|x_1|+|x_2|},
	\]
	where $ \theta_0 = \sup_{m \in M_1}\frac{\lambda_s(m)\lambda_u(m)\mu(m)}{1-\alpha(m)\beta'(u(m))} < 1 $.
	Observe that 
	\[
	\sup_{m \in M_1}\sup_{x_1,x_2} \frac{|K_m(x_1)- K_m(x_2)|}{|x_1|+|x_2|} < \infty,
	\]
	which gives the results.
\end{proof}

We write $ \widehat{D}_{m_0} F_{m_0}(\cdot,\cdot)  $, $ \widehat{D}_{m_0} G_{m_0}(\cdot,\cdot) $, $ \widehat{K}_{m_0}(\cdot,\cdot) $, $\widehat{R}_{m_0}(\cdot,\cdot) $ for the local representations of $ \nabla F $, $ \nabla G $, $ K $, $ R $ (at $ m_0 \in M_1 $) with respect to $ \mathcal{A}, \mathcal{B} $, and any $ C^1 $ local charts $ \xi_{m_0}, \xi_{u(m_0)} $ of $ M $ at $ m_0 $, $ u(m_0) $, such that $ D\xi_{m_0}(m_0) = \id, D\xi_{u(m_0)}(u(m_0)) = \id $. See e.g.  \eqref{basediff1}, \eqref{basediff2}, \eqref{basediff}, \eqref{baseKR}. If (H1c) is assumed, then $ \xi_{m_0}, \xi_{u(m_0)} $ are replaced by $ \chi_{m_0}, \chi_{u(m_0)} $; see \eqref{basediff00} and \eqref{baseKR0}.

\begin{lem}[Continuity of $ m \mapsto K_m(x) $] \label{baseK}
	Under the assumptions in \autoref{lem:leafK} with \textnormal{(H2)(v1$ ' $)} replaced by \textnormal{(H2)(v2$ ' $)}, and with \eqref{DFDG1} and \eqref{DFDG2} satisfied for 
	\[
	\mathcal{L} = \sup\limits_{m \in M_1}\lim\limits_{r \to 0^+} \sup\limits_{|z_1 - z_2| \leq r },
	\]
	the following results hold:
	\begin{enumerate}[(1)]
		\item If $ \nabla F, \nabla G $ satisfy for any $ m_0 \in M_1 $, $ r > 0 $,
		\begin{equation}\label{aabb}
		\left\{
		\begin{split}
		\lim_{m' \to m_0}\sup_{|z| \leq r}\frac{|\widehat{D}_{m_0} F_{m_0}(m',z) - \widehat{D}_{m_0} F_{m_0}(m_0,z)|}{|z|} = 0, \\
		\lim_{m' \to m_0}\sup_{|z| \leq r}\frac{|\widehat{D}_{m_0} G_{m_0}(m',z) - \widehat{D}_{m_0} G_{m_0}(m_0,z)|}{|z|} = 0,
		\end{split}
		\right.
		\end{equation}
		and the $ C^0 $ continuity in the bounded-fiber sets case in \autoref{lem:baseK1} (1) is satisfied,
		then
		\begin{equation}\label{KKK111}
		\lim_{m' \to m_0}\sup_{|x| \leq r}\frac{|\widehat{K}_{m_0}(m',x) - \widehat{K}_{m_0}(m_0,x)|}{|x|} = 0.
		\end{equation}
		\item If $ \nabla F, \nabla G $ satisfy \eqref{aabb} for $ r = \infty $, and the $ C^0 $ continuity in the whole space case in \autoref{lem:baseK1} (1) is satisfied, then $ K $ also satisfies \eqref{KKK111} for $ r = \infty $.
		\item Suppose that \textnormal{(H1c), (H2d), (H5)} hold and $ Du $ is uniformly continuous around $ M_1 $, that \eqref{DFDG1} and \eqref{DFDG2} are satisfied for the limit $ \lim\limits_{r \to 0^+}\sup\limits_{m \in M_1} \sup\limits_{|z_1 - z_2| \leq r } $, and that the uniform continuity in the bounded fiber sets (resp. in the whole space) case in \autoref{lem:baseK1} (1) is satisfied. If  \eqref{aabb} holds with $ m' \to m_0 \in M_1 $ in the uniform sense for any $ r > 0 $ (resp. $ r = \infty $), so does \eqref{KKK111}.
	\end{enumerate}
\end{lem}
\begin{proof}
	We only consider (1); the others are similar. Under the assumptions of this lemma, we can use \autoref{lem:continuity_f} (2), \autoref{lem:K1leafc} (2), \autoref{lem:baseK1} (1), \autoref{lem:leafK} (2) to obtain the desired continuity of $ f, K^1, K_m(\cdot) $.
	Consider the following limit:
	\[
	\limsup_{r \to \infty}\sup_{m_0 \in M_1}\lim_{m' \to m_0}\sup_{|x| \leq r}\frac{|\widehat{K}_{m_0}(m',x) - \widehat{K}_{m_0}(m_0,x)|}{|x|}.
	\]
	By a similar computation to the proof of \autoref{lem:final} (but this time using \eqref{K1elocal}), one concludes that
	\begin{align*}
	& \limsup_{r \to \infty}\sup_{m_0 \in M_1}\lim_{m' \to m_0}\sup_{|x| \leq r}\frac{|\widehat{K}_{m_0}(m',x) - \widehat{K}_{m_0}(m_0,x)|}{|x|} \\
	\leq & \theta_0\limsup_{r \to \infty}\sup_{m_0 \in M_1}\lim_{m' \to m_0}\sup_{|x| \leq r}\frac{|\widehat{K}_{m_0}(m',x) - \widehat{K}_{m_0}(m_0,x)|}{|x|},
	\end{align*}
	where $ \theta_0 = \sup_{m \in M_1}\frac{\lambda_s(m)\lambda_u(m)\mu(m)}{1-\alpha(m)\beta'(u(m))} < 1 $. And note that
	$ \sup_{m_0 \in M_1}\sup_{x} \frac{|K_{m_0}(x)|}{|x|} < \infty $.
	Now combine all the above facts to get \eqref{KKK111}.
\end{proof}

The condition on $ F, G $ in \autoref{baseK} can be satisfied if $ F, G $ are $ C^{1,1} $. This is actually used in some classical results; see e.g. \cite{HPS77}. Also, note that the assumptions in the $ C^0 $ continuity results are much weaker than those in the H\"older continuity results.
Now we have extended the results in \cite[Chapter 5 (about plaque families), Chapter 6]{HPS77} to our general settings.

\begin{rmk}\label{rmk:con}
	A basic application of the continuity results is the study of dynamical systems with parameters. Consider a special case. Let $ u = \id: M \to M $, $ X = M \times X_0 $, $ Y =  M \times Y_0 $ (the trivial bundle case), where $ M, X_0, Y_0 $ are all Banach spaces. Now $ m \mapsto H_m \sim (F_m, G_m) $ can be viewed as parameter-dependent correspondences. The associated invariant graphs $ f_m $, $ m \in M $, obtained in \autoref{thmA}, depend on the parameter $ m $. \autoref{lem:continuity_f} (2) for $ r = \infty $ and \autoref{lem:baseK1} (1) give us the continuous dependence of $ m \mapsto f_m(\cdot) $ and $ m \mapsto Df_m(\cdot) $, i.e., $ \lim_{m \to m_0} |f_m - f_{m_0}|_{C^1_b(X_0, Y_0)} = 0 $. Also, \autoref{baseK} (2) gives the smooth dependence of $ m \mapsto f_m(\cdot) $. To see this, consider
	\begin{align*}
	\lim_{m \to m_0} \frac{\| f_m - f_{m_0} - K_{m_0} (m - m_0) \|}{|m - m_0|} = & \lim_{m \to m_0} \sup_{x}\frac{| f_m(x) - f_{m_0}(x) - K_{m_0}(x) (m - m_0) |}{|m - m_0|} \\
	\leq & \int_{0}^{1} \lim_{m \to m_0} \sup_{x}| K_{m_t}(x) - K_{m_0}(x)| ~\mathrm{d}t,
	\end{align*}
	where $ m_t = t m + (1-t)m_0 $. In the above argument we technically assume $ F_m, G_m $, $ m \in M $, are bounded in order to let $ f_m \in C_b (X_0, Y_0) $.

	We can also take $ \mathrm{Lip}_0(X_0, Y_0) = \{g: X_0 \to Y_0 ~\text{is Lipschitz:}~ g(0) = 0 \} $ instead of $ C_b(X_0, Y_0) $ (assuming the section $ i = 0 $). However, in this case the results do not give the smooth dependence of $ m \mapsto f_m(\cdot) $ in $ \mathrm{Lip}_0(X_0, Y_0) $. This can be achieved if we consider a higher order differential of $ f $.
\end{rmk}

\section{Appendix. Regularity of invariant graphs: bounded section case} \label{bounded}

In this appendix, we make the following assumptions.
\begin{enumerate}[$ \bullet $]
	\item Let $ (X, M, \pi_1) $, $ (Y, M, \pi_2) $, $ u : M \to M $, $ H \sim (F,G) $, $ i: M \to X \times Y $ be as in \autoref{thmA}. Let $ f $ be the bundle map obtained in \autoref{thmA} with $ i $ a \textbf{$ 1 $-pseudo-stable section} of $ H $ and $ \sup_m \eta(m) < \infty $.

	In general, the uniform assumption $ \sup_m \eta(m) < \infty $ can sometimes be weakened. For example, we can only assume that $ \sup_{n \geq 0}\eta(u^{n}(m)) < \infty $ ($ \forall m \in M $) and the function $ \eta(\cdot) $ is locally bounded. However, in this case, the conditions on $ F, G $ should pointwise depend on $ m $ in some exponential way. In this paper, we do not consider this situation.

	\item We add an additional assumption that the functions in the (A)(B) (or (A$ ' $)(B)) condition are bounded. Note that in the $ 1 $-pseudo-stable section case, the spectral condition in \autoref{thmA} is $ \lambda_s\lambda_u < 1 $, $ \lambda_u < 1 $, where we use the abbreviation in \autoref{absgc}.
\end{enumerate}

Like the case that $ i $ is an invariant section of $ H $, one can also show that $ f $ has more regularity.
Since the proofs of the relevant results are very similar to those in \autoref{leaf} to \autoref{HolderivativeBase}, we only give the statements and omit all the proofs. The conditions on $ M, X, Y $, and $ u $ are the same as in the invariant section case. The main different assumptions are about $ F, G $ and the spectral gap conditions. Also, the section $ i $ need not be a $ 0 $-section, but is instead assumed to be uniformly (locally) bounded (see the assumption (UB) below).

We use the same abbreviations for the spectral gap conditions (see \autoref{absgc}). Also, \emph{the constants $ M_0 $, $ C $, appearing in the following, are independent of $ m \in M $ and $ m_0 \in M_1 $.}

\vspace{.5em}
\noindent{\textbf{Smooth leaves: $ C^{k,\alpha} $ continuity of $ x \mapsto f_m(x) $.}}
The results of Lemmas \ref{lem:leaf1}, \ref{lem:leaf1a} and \ref{lem:leafk} also hold  when $ i $ is a $ 1 $-pseudo-stable section of $ H $. But there is no result like \autoref{00diff}.

\vspace{.5em}
\noindent{\textbf{H\"{o}lder vertical parts: H\"{o}lderness of $ m \mapsto f_m(x)$.}}
We use the notations of \autoref{hverticalsheaf}, $ \widehat{F}_{m_0} $, $ \widehat{G}_{m_0} $, $ \widehat{f}_{m_0} $, $ \widehat{x}_{m_0} $, for the local representations of $ F, G $, $ f, x_{(\cdot)}(\cdot) $ at $ m_0 \in M_1 $ with respect to $ \mathcal{A}, \mathcal{B} $; see \eqref{localrepre}.
Also, let
\begin{multline*}
\widehat{i}^{m_0}(m) \triangleq (\widehat{i}^{m_0}_X(m), \widehat{i}^{m_0}_Y(m)) = ((\varphi^{m_0}_m)^{-1}(i_X(m)), (\phi^{m_0}_m)^{-1}(i_Y(m)) ) : \\
U_{m_0}(\epsilon) \to X_{m_0} \times Y_{m_0},
\end{multline*}
be the local representation of $ i $ (at $ m_0 $) with respect to $ \mathcal{A} \times \mathcal{B} $.
For brevity, we also write $|x| \triangleq d(x, i_X(m))$ if $ x \in X_{m} $, $|y| \triangleq d(y, i_Y(m))$ if $ y \in Y_{m} $, and $ |z| = |(x,y)| = d((x,y), (i_X(m), i_{Y}(m)) $ if $ z = (x,y) \in X_{m} \times Y_{m} $. Note that $ \widehat{i}^{m_0}_X(m_0) = i_{X}(m_0) $, $ \widehat{i}^{m_0}_Y(m_0) = i_{Y}(m_0) $. The assumption on $ i $ is the following:

\begin{enumerate}
	\item [\textbf{(UB)}] \label{b3-} There are $ \epsilon_0 , \hat{c}_1 > 0 $ such that $ | \widehat{i}^{m_0}(m) - \widehat{i}^{m_0}(m_0) | \leq \hat{c}_1 $ for all $ m \in U_{m_0}(\epsilon_0) $ and $ m_0 \in M_1 $.
\end{enumerate}

\begin{lem}\label{lem1:sheaf}
	Assume the following conditions hold:
	\begin{enumerate}[(a)]
		\item Let  \textnormal{(H1a), (H2a), (UB), (H4a), (H5)}  hold.
		\item $ F, G $ are uniformly (locally) $ \gamma $-H\"{o}lder around $ M_1 $ in the following sense:
		\begin{align}
		|\widehat{F}_{m_0}(m_1, z) - \widehat{F}_{m_0}(m_0, z) | \leq M_0 |m_1 - m_0|^\gamma (|z|^\zeta + \hat{c}_0), \label{ccc2f}\\
		|\widehat{G}_{m_0}(m_1, z) - \widehat{G}_{m_0}(m_0, z)| \leq M_0 |m_1 - m_0|^\gamma (|z|^\zeta + \hat{c}_0), \label{ccc2g}
		\end{align}
		for all $ m_1 \in U_{m_0} (\mu^{-1}(m_0) \varepsilon_1) $, $ z \in X_{m_0} \times Y_{u(m_0)} $, $ m_0 \in M_1 $, where $ 0< \gamma \leq 1 $, $ \zeta \geq 0 $, $ \hat{c}_0 > 0 $.
		\item (spectral gap condition) $ \lambda_u < 1 $, $ \lambda_s \lambda_u < 1 $, $ (\max\{ 1, \lambda^\zeta_s \} \mu^\gamma)^{*\alpha} \lambda_u < 1 $, and
		\[
		(\max \{ 1, 1/\lambda_s, 1/\lambda^{1-\zeta}_s \} \mu^\gamma )^{*\alpha} \lambda_s \lambda_u < 1,
		\]
		where $ 0 < \alpha \leq 1 $; see also \autoref{absgc}.
	\end{enumerate}
	If $ \varepsilon^*_1 \leq \hat{\mu}^{-2} \varepsilon_1 $ is small, then 
	\[
	|\widehat{f}_{m_0}(m_1, x) - \widehat{f}_{m_0}(m_0, x)| \leq C |m_1 - m_0|^{\gamma \alpha} (|x|^{\zeta} + 1)^\alpha (|x| + 1)^{1-\alpha},
	\]
	for every $ m_1 \in U_{m_0} (\varepsilon^*_1), x \in X_{m_0}, m_0 \in M_1 $ under $ |m_1 - m_0|^\gamma (|x|^\zeta + 1) \leq \hat{r}\min\{ |x|+1, (|x|^\zeta + 1)^c (|x|+1)^{-(c-1)} \} $, where $ C $ depends on $ \hat{r} > 0 $ but not on $ m_0 \in M_1 $, $ c > 1 $, and $ \hat{r} $ does not depend on $ m_0 \in M_1 $.

\end{lem}

\begin{rmk}
	\begin{enumerate}[(a)]
		\item Note that $ \hat{c}_0 = 0 $ is equivalent to $ i $ being invariant. So we only consider $ \hat{c}_0 > 0 $, and in this case we can assume $ \hat{c}_0 = 1 $.
		Two special cases are important: $ \zeta = 0 $ and $ \zeta = 1 $.
		For the case $ \gamma = 1 $, $ \zeta = 1 $, under the spectral gap condition $ \lambda_u < 1 $, $ \lambda_s \lambda_u < 1 $, $ \lambda_u \mu < 1 $, $ \lambda_s \lambda_u \mu < 1 $, we have for all $ m_1 \in U_{m_0} (\varepsilon^*_1) $,
		\begin{equation*}
		|\widehat{f}_{m_0}(m_1, x) - \widehat{f}_{m_0}(m_0, x)| \leq C|m_1 - m_0|(|x| + 1).
		\end{equation*}

		\item
		Under conditions (a), (b) in \autoref{lem1:sheaf},
		with $\sup_m \widetilde{d}( f_m(X_m), i_Y(m) ) < \infty$ and a `better' spectral gap condition, $ \lambda_u < 1 $, $ \lambda_s \lambda_u < 1 $, $ (\max\{ 1, \lambda^\zeta_s \} \mu^\gamma)^{*\alpha} \lambda_u < 1 $, we get
		\[
		|\widehat{f}_{m_0}(m_1, x) - \widehat{f}_{m_0}(m_0, x)| \leq C |m_1 - m_0|^{\gamma \alpha} (|x|^{\zeta} + 1)^\alpha,
		\]
		for every $ m_1 \in U_{m_0} (\varepsilon^*_1), x \in X_{m_0}, m_0 \in M_1 $ under $ |m_1 - m_0|^\gamma (|x|^\zeta + 1) \leq \hat{r} $, where the constant $ C $ depends on $ \hat{r} > 0 $ but not on $ m_0 \in M_1 $, and $ \hat{r} $ does not depend on $ m_0 \in M_1 $.

		\item Consider the trivial bundle $ X\times Y  = M \times X_0 \times Y_0 $ where $ X_0, Y_0 $ are Banach spaces.
		A special condition where $ F, G $ satisfy \textnormal{(b)} is the following: Let
		\[
		F_m(x,y) = A(m)x + f(m,x,y),~~G_m(x,y) = B(m)y + g(m,x,y),
		\]
		where $ A(m), B(m) \in L(X_0, Y_0) $, $ f : M \times X_0 \times Y_0 \to X_0 $, $ g: M \times X_0 \times Y_0 \to Y_0 $. Assume $ A(\cdot), B(\cdot) \in C^{0,1} $, $ f(\cdot,x,y), g(\cdot,x,y)  $ are $ C^{0,\gamma} $ uniformly for $ x,y $, and $ \sup_{m,y}|f(m,0,y)| \leq \eta_0 $, $ \sup_{m,x}|g(m,x,0)| \leq \eta_0 $ for some constant $ \eta_0 > 0 $. Take the section $ i = 0 $. See also similar results in \cite{CY94} for cocycles in the finite-dimensional setting.
	\end{enumerate}
\end{rmk}

\vspace{.5em}
\noindent{\textbf{H\"{o}lder distribution: H\"{o}lderness of $ m \mapsto Df_m(x) = K^1_m(x) $.}}
There is no result like \autoref{lem:base0}.
\begin{lem}
	Assume the conditions in \autoref{lem:baseleaf} hold with \textnormal{(H3)} and \eqref{cc2} in \autoref{lem:sheaf} replaced by \textnormal{(UB)} and \eqref{ccc2f} \eqref{ccc2g} in \autoref{lem1:sheaf} (for $ \gamma = 1, \zeta = 1 $), respectively. In addition, we need the following spectral gap condition \textnormal{(c)$ ' $} instead of \textnormal{(c)}:
	\begin{enumerate}

		\item [\textnormal{(c)}$ ' $] (spectral gap condition) $ \lambda_u < 1 $, $ \lambda_s \lambda_u < 1 $, $ \mu^\alpha \lambda^2_s \lambda_u < 1 $, $ \lambda_s^{*\beta} \lambda_s \lambda_u < 1 $, $ \mu^{*\alpha} \lambda_u < 1 $; see also \autoref{absgc}.

	\end{enumerate}
	Then the results in \autoref{lem:baseleaf} all hold.
\end{lem}
\begin{rmk}
	Assume the conditions in \autoref{lem:baseleaf} hold, but replace \textnormal{(H3)} by \textnormal{(UB)}, and \eqref{cc2} in \autoref{lem:sheaf} by
	\begin{equation*}
	|\widehat{F}_{m_0}(m_1, z) - \widehat{F}_{m_0}(m_0, z) | \leq M_0 |m_1 - m_0|,~|\widehat{G}_{m_0}(m_1, z) - \widehat{G}_{m_0}(m_0, z)| \leq M_0 |m_1 - m_0|,
	\end{equation*}
	with $\sup_m \widetilde{d}( f_m(X_m), i_Y(m) ) < \infty$. In addition, we need the following spectral gap condition: $ \lambda_u < 1 $, $ \lambda_s \lambda_u < 1 $, $ \lambda^{1+\beta}_s\lambda_u < 1 $, $ \lambda_u\mu^\alpha < 1 $, $ \lambda_s\lambda_u\mu^\alpha < 1 $. If $ \varepsilon^*_1 \leq \hat{\mu}^{-2} \varepsilon_1 $ is small, then 
	\[
	| \widehat{K}^1_{m_0} (m_1, x) - \widehat{K}^1_{m_0} (m_0, x) | \leq C |m_1 - m_0|^{\alpha\beta},\quad \forall m_1 \in U_{m_0} (\varepsilon^*_1), x \in X_{m_0}, m_0 \in M_1.
	\]
\end{rmk}

\vspace{.5em}
\noindent{\textbf{Smoothness of $ m \mapsto f_m(x) $ and H\"{o}lderness of $ x \mapsto \nabla_mf_m(x) = K_m(x) $}.}
\begin{lem}\label{smoothbase1}
	Assume the conditions in \autoref{smoothbase} hold, but replace \textnormal{(H3)} by \textnormal{(UB)}, and conditions \textnormal{(b) (c)} by the following conditions \textnormal{(b)$ ' $ (c)$ ' $} or \textnormal{(b)$ '' $ (c)$ '' $}.
	\begin{enumerate}
		\item [\textnormal{(b)}$ ' $]\textnormal{(i)} $ F, G \in C^1 $.  \textnormal{(ii)} \eqref{ccc2f} \eqref{ccc2g} hold for $ \gamma = 1 $, $ \zeta = 1 $.
		\item [\textnormal{(c)}$ ' $](spectral gap condition) $ \lambda_u < 1 $, $ \lambda_s \lambda_u < 1 $, $ \lambda_u \mu < 1 $, $ \lambda_s \lambda_u \mu < 1 $; see also \autoref{absgc}.

		\item [\textnormal{(b)}$ '' $]\textnormal{(i)}  $ F, G \in C^1 $.  \textnormal{(ii)} \eqref{ccc2f} \eqref{ccc2g} hold for $ \gamma = 1 $, $ \zeta = 0 $. \textnormal{(iii)} In addition, 
		\[
		\sup_m \widetilde{d}( f_m(X_m), i_Y(m) ) < \infty.
		\]

		\item [\textnormal{(c)}$ '' $](spectral gap condition) $ \lambda_u < 1 $, $ \lambda_s \lambda_u < 1 $, $ \lambda_u \mu < 1 $; see also \autoref{absgc}.
	\end{enumerate}
	Then the following conclusions hold:
	\begin{enumerate}[(1)]
		\item There exists a unique $ C^0 $ vector bundle map $ K \in L(\Upsilon^H_X, \Upsilon^V_Y) $ over $ f $ that satisfies \eqref{mainMFW}, and if \textnormal{(b)$ ' $ (c)$ ' $} hold then $ |K_m(x)| \leq C (|x|+ 1) $, while if \textnormal{(b)$ '' $ (c)$ '' $} hold then $ |K_m(x)| \leq C $, for all $ x \in X_m $, $ m \in M_1 $.

		\item $ f $ is $ C^1 $ and $ \nabla f = K $.
	\end{enumerate}
\end{lem}

\begin{lem}
	Assume the conditions in \autoref{lem:holversheaf} hold, but replace \textnormal{(H3)} by \textnormal{(UB)}, and conditions \textnormal{(b) (ii) (c)} by the following \textnormal{(b)$ ' $ (ii) (c)$ ' $} or \textnormal{(b)$ '' $ (ii) (c)$ '' $}.
	\begin{enumerate}
		\item [\textnormal{(b)$ ' $}] (ii) $ |\nabla_m F_m(z)| \leq M_0(|z|+1),~|\nabla_m G_m(z)| \leq M_0(|z|+1) $.

		\item [\textnormal{(c)$ ' $}] (spectral gap condition) Let condition \textnormal{(c)$ ' $} in \autoref{smoothbase1} hold and $ \lambda^{*\beta}_s \lambda_s \lambda_u < 1 $, $ (\lambda_s^{\gamma})^{ *\alpha}\lambda_s \lambda_u \mu < 1 $; see also \autoref{absgc}.

		\item [\textnormal{(b)$ '' $}] (ii) $ |\nabla_m F_m(z)| \leq M_0 $, $ |\nabla_m G_m(z)| \leq M_0 $.

		\item [\textnormal{(c)$ '' $}] (spectral gap condition) Let condition \textnormal{(c)$ '' $} in \autoref{smoothbase1} hold and $ \lambda^{*\beta}_s \lambda_s \lambda_u < 1 $, $ (\lambda_s^{\gamma})^{ *\alpha} \lambda_u \mu < 1 $; see also \autoref{absgc}.
	\end{enumerate}
	If $ 0 < \alpha \leq 1 $, then the $ C^0 $ vector bundle map $ K \in L(\Upsilon^H_X, \Upsilon^V_Y) $ over $ f $ satisfying \eqref{mainMFW} has the following H\"older property:
	\begin{enumerate}[(1)]
		\item If \textnormal{(b)$ ' $ (c)$ ' $} hold, then
		\begin{multline*}
		|K_m(x_1) - K_m(x_2)| \\
		\leq C  \left\lbrace ( |x_1 - x_2|^{\gamma} (1+|x_1|) )^{\alpha} + (|x_1 - x_2|^\beta (1+|x_1|))^\alpha \right\rbrace (|x_1| + |x_2| + 1)^{1-\alpha},
		\end{multline*}
		for all $ m \in M_1 $, under $ (|x_1 - x_2|^{\gamma} + |x_1 - x_2|^\beta) (1+|x_1|)\leq \hat{r} (|x_1| + |x_2| + 1) $, where the constant $ C $ depends on $ \hat{r} > 0 $ but not on $ m \in M $.

		\item If \textnormal{(b)$ '' $ (c)$ '' $} hold, then
		\[
		|K_m(x_1) - K_m(x_2)| \leq C (|x_1 - x_2|^{\gamma\alpha} + |x_1 - x_2|^{\beta\alpha} ).
		\]
	\end{enumerate}
\end{lem}

\vspace{.5em}
\noindent{\textbf{H\"{o}lderness of $ m \mapsto \nabla_mf_m(x) = K_m(x) $.}}
\begin{lem}\label{final1}
	Let \autoref{lem:final} \textnormal{(a)} hold with \textnormal{(H3)} replaced by \textnormal{(UB)}. In addition, assume the following conditions \textnormal{(b)$ ' $ (c)$ ' $} or \textnormal{(b)$ '' $ (c)$ '' $} hold.
	\begin{enumerate}
		\item [\textnormal{(b)$ ' $}] $ F, G $ are $ C^1 $. Furthermore, \eqref{ccc2f} \eqref{ccc2g} hold for $ \gamma = 1, \zeta = 1 $, as well as \eqref{lipzz} and \eqref{ligmm}.

		\item [\textnormal{(c)$ ' $}] (spectral gap condition) $ \lambda_u < 1 $, $ \lambda_s \lambda_u < 1 $, $ \lambda_s^2\lambda_u < 1 $, $ \lambda_u \mu < 1 $, $ \lambda_s^2 \lambda_u \mu < 1 $, $ \max\{ \mu/\lambda_s, \mu \}^{*\alpha} \lambda_s \lambda_u \mu < 1 $; see also \autoref{absgc}.

		\item [\textnormal{(b)$ '' $}] $ F, G $ are $ C^1 $; \eqref{ccc2f} \eqref{ccc2g} hold for $ \gamma = 1 $, $ \zeta = 0 $; \eqref{lipzz} ($ DF, DG \in C^{0,1} $) and \eqref{ligmm} ($ \nabla F, \nabla G \in C^{0,1} $) hold. Also, $\sup_m \widetilde{d}( f_m(X_m), i_Y(m) ) < \infty$.

		\item [\textnormal{(c)$ '' $}] (spectral gap condition) $ \lambda_u < 1 $, $ \lambda_s \lambda_u < 1 $, $ \lambda_s^2\lambda_u < 1 $, $ \lambda_u \mu < 1 $, $ \lambda_s \lambda_u \mu < 1 $, $ \mu^{* \alpha}\lambda_u\mu  < 1 $; see also \autoref{absgc}.
	\end{enumerate}
	If $ 0 < \alpha \leq 1 $, then there exists a small $ \epsilon'_1 > 0 $ such that the following hold:
	\begin{enumerate}[(1)]
		\item Under conditions \textnormal{(b)$ ' $ (c)$ ' $}, we have
		\[
		|\widehat{K}_{m_0}(m_1, x) - \widehat{K}_{m_0}(m_0, x)| \leq C |m_1 - m_0|^\alpha (|x|+1)
		\]
		for all $ m_1 \in U_{m_0} (\epsilon'_1) $, $ m_0 \in M_1 $. The constant $ C $ depends on $ \epsilon'_1 $ but not on $ m_0 \in M_1 $.

		\item Under conditions \textnormal{(b)$ '' $ (c)$ '' $}, we have
		\[
		|\widehat{K}_{m_0}(m_1, x) - \widehat{K}_{m_0}(m_0, x)| \leq C |m_1 - m_0|^\alpha,
		\]
		for all $ m_1 \in U_{m_0} (\epsilon'_1) $, $ m_0 \in M_1 $. The constant $ C $ depends on $ \epsilon'_1 $ but not on $ m_0 \in M_1 $.
	\end{enumerate}

\end{lem}

\vspace{.5em}
\noindent{\textbf{Continuity properties of $ f $.}}
The results on the continuity of $ f $ are essentially the same as in \autoref{continuityf} with minor changes. In order not to increase the length of the paper, the statements are all omitted and left to the readers.

\begin{rmk}
	We have already considered the regularities of $ f $ obtained in \autoref{thmA} for the two extreme cases, i.e., $ \varepsilon = 0 $ and $ \varepsilon = 1 $. The results of \autoref{lem:leaf1} and \autoref{lem:leaf1a} also hold for all $ 0 \leq \varepsilon \leq 1 $ without changing the proofs. Similar results such as the H\"older continuity and smoothness respecting base points should also hold when $ i $ is an $\varepsilon$-pseudo-stable section, where $ 0 < \varepsilon < 1 $.
	The readers can also consider the regularities of $ f $ obtained in \autoref{thmB}, which are very similar to the results listed in this appendix in the case of $\sup_m \widetilde{d}( f_m(X_m), i_Y(m) ) < \infty$.
\end{rmk}

\section{A local version of the regularity results}\label{generalized}

Generally speaking, the regularity results do not depend on the existence results. That is, if there is a graph which is (locally) invariant under $ H $, then one can directly investigate the regularities of this graph; our proofs given in \autoref{leaf} to \autoref{HolderivativeBase} have indicated this. For a simple motivation, see also \autoref{lem:fconx} and \autoref{lem:fdiifx}; this idea was also used in \cite{BLZ99} to prove the smoothness of invariant manifolds. In the following, we will present a detailed statement of the conditions such that the regularity results hold in a local version for our applications; see \autoref{thm:fake1}, \autoref{thm:fake2}, and \cite{Che18b, Che18} as well. In this way, one may comprehend what we really need in our proofs.

Take a bundle correspondence $ H \sim (F, G): X \times Y \to X \times Y $ over $ u $.
Let $ X^2_m \subset X^1_m \subset X_m $. Assume there are functions $ f_m : X_m \to Y_m $, and $ x_m(\cdot): X_m \to X_{u(m)} $, $ m \in M $, such that
\[
\graph f_m|_{X^1_m} \subset H^{-1}_m \graph f_{u(m)} |_{X^2_{u(m)}},
\]
i.e., $ x_m(X^1_m) \subset X^2_{u(m)} $, and for all $ x \in X^1_m $,
\[
\begin{cases}
F_m(x, f_{u(m)}(x_m(x))) = x_m(x), \\
G_m(x, f_{u(m)}(x_m(x))) = f_m(x).
\end{cases}
\]
Take $ Y^1_m \subset Y_m $ such that $ f_m(X^1_m) \subset Y^1_m $. Consider the following assumptions:

\begin{enumerate}[(I)]
	\item $ H_m \sim (F_m, G_m): X^1_m \times Y^1_m \to X^1_{u(m)} \times Y^1_{u(m)} $ satisfies the (A$ ' $)($ \alpha(m) $, $ \lambda_s(m) $) (B$ ' $)($ \beta'(m) $, $ \lambda_u(m) $) condition.

	\item For each $ m \in M $ and each $ z \in X^{1}_m \times Y^1_{u(m)} $, $ (DF_m(z), DG_m(z)): X_m \times Y_m \to X_{u(m)} \times Y_{u(m)} $ satisfies (A$ ' $)($ \alpha(m) $, $ \lambda_s(m) $) (B)($ \beta'(u(m)) $; $ \beta'(m) $, $ \lambda_u(m) $) condition.

	\item For every $ m \in M $, there is a neighborhood $ \mathring{X}^1_m $ of $ X^1_m $ in $ X_m $ such that
	\[
	\lip f_{m}|_{\mathring{X}^1_m} \leq \beta'(m), ~\lip x_{m} |_{\mathring{X}^1_m} \leq \lambda_s(m).
	\]
	If (H3) is assumed, then we also suppose that $ |f_m(x)| \leq \beta'(m)|x| $ and $ |x_m(x)| \leq \lambda_s(m)|x| $ for all $ x \in X^1_m $, $ m \in M_1 $.

	\item For $ m \in M $, there are $ \epsilon_{m}'' > 0 $, $ X^0_{m} $, $ Y^0_{m} $ such that $ X^2_{m} \subset X^0_{m} \subset X^1_{m} $, $ f_{m}(X^1_{m}) \subset Y^0_{m} \subset Y^1_{m} $, and moreover
	\[
	X^2_m \subset \varphi^{m_0}_{m} (X^0_{m_0}) \subset X^1_m, ~f_{m}(X^1_{m}) \subset \phi^{m_0}_m(Y^0_{m_0}) \subset Y^1_{m},~m \in U_{m_0}(\epsilon_{m_0}''), m_0 \in M,
	\]
	where $ (U_{m_0}(\epsilon_{m_0}''), \varphi^{m_0}) \in \mathcal{A}' $, $ (U_{m_0}(\epsilon_{m_0}''), \phi^{m_0}) \in \mathcal{B}' $ (in (H2)), and $ \inf_{m_0 \in M_1}\epsilon_{m_0}'' > 0 $.
\end{enumerate}

Let $ X^1 = \{(m,x): x \in X^1_m, m \in M \} $ and $ Y^1 = \{(m,y): y \in Y^1_m, m \in M \} $ be the subbundles of $ X $ and $ Y $, respectively.
For simplicity, write (see \eqref{HVspace})
\begin{equation*}
\Upsilon^V_{X^1} \triangleq \bigcup_{(m,x) \in X^1} (m,x) \times X_m, ~\Upsilon^H_{X^1} \triangleq \bigcup_{(m,x) \in X^1} (m,x) \times T_mM.
\end{equation*}
Then the results in \autoref{leaf}--\autoref{continuityf} also hold with some modifications, namely, \emph{$ X_{m_0} \times Y_{u(m_0)} $ is replaced by $ X^0_{m_0} \times Y^0_{u(m_0)} $, $ X_{m_0} $ by $ X^0_{m_0} $, $ X_m \times Y_{u(m)} $ by $ X^1_m \times Y^1_{u(m)} $, and $ X_m $ by $ X^1_m $}; in the following \autoref{thm:LocalRegularity} we make this convention.
\begin{thm}\label{thm:LocalRegularity}
	\begin{enumerate}[(a)]
		\item ($ x \mapsto f_m(x) $) Under (I)--(III) and the same conditions as in \linebreak \autoref{lem:leaf1} with (b) replaced by (b$ ' $): $ F_m, G_m $ are differentiable on $ X^1_m \times Y^1_{u(m)} $ and $ X^1_m \times Y^1_{u(m)} \ni z \mapsto DF_m(z), DG_m(z) $ are $ C^0 $, the map $ f_m $ is differentiable on $ X^1_m $. Moreover, there exists a unique vector bundle map $ K^1 \in L(\Upsilon_{X^1}^V, \Upsilon_{Y^1}^V) $ over $ f $, such that $ |K^1_m(x)| \leq \beta'(m) $, $ m \in M $, $ X^1_m \ni x \mapsto K^1_m(x) $ is $ C^0 $, and it satisfies \eqref{mainX} with $ Df_m(x) = K^1_m(x) $, $ x \in X^1_m $.

		\item ($ x \mapsto K^1_m(x) $) Let (I)--(III) hold. (i) Then \autoref{lem:leaf1a} holds.

		(ii) Suppose $ F, G $ are uniformly $ C^{k-1,1} $-fiber in a neighborhood of $ X^1 \otimes_{u} Y^1 $ (see \autoref{lip maps}) and $ \sup_{m}\lip D^{k-1} f_{m}(\cdot)|_{\mathring{X}^1_m} < \infty $. If $ \lambda_s \lambda_u < 1, \lambda^k_s \lambda_u < 1 $, and $ F, G $ are $ C^k $-fiber on $ X^1 \otimes_{u} Y^1 $, then $ f $ is $ C^k $-fiber on $ X^1 $ (\autoref{lem:leafk}).

		\item ($ m \mapsto f_m(x) $) Let (I), (III), (IV) hold. Then \autoref{lem:sheaf} holds.

		\item ($ m \mapsto K^1_m(x) $) Let (I)--(IV) hold. Then \autoref{lem:baseleaf} holds.

		\item \label{ee0} ($ C^1 $ smoothness of $ f $) (i) Under (I)--(III) and the same conditions as in \autoref{lem:leafK} but assuming $ F, G $ are $ C^1 $ in a neighborhood of $ X^1 \otimes_{u} Y^1 $ instead of in $ X \otimes_{u} Y $, there exists a unique $ C^0 $ vector bundle map $ K \in L(\Upsilon^H_{X^1}, \Upsilon^V_{Y^1}) $ (over $ f $) such that $ \sup_{m \in M_1}\sup_{x}|K_m(x)|/|x| < \infty $ and it satisfies \eqref{mainMFW} for all $ x \in X^1_m $, $ m \in M $.

		(ii) Let (I)--(IV) hold. Assume all the conditions in \autoref{smoothbase} hold but $ F, G $ are $ C^1 $ in a neighborhood of $ X^1 \otimes_{u} Y^1 $ instead of in $ X \otimes_{u} Y $; then $ f $ is differentiable in $ X^2 = \bigcup_{m \in M} X^2_m $ with $ \nabla_{m} f_m(x) = K_m(x) $ if $ x \in X^2_m $ and $ m \in M $, where $ K $ is given by (i).

		\item ($ x \mapsto K_m(x) $) Let (I)--(III) hold. Assume $ F, G $ are $ C^1 $ in a neighborhood of $ X^1 \otimes_{u} Y^1 $ instead of condition (b) (i) in \autoref{lem:holversheaf}. Then \autoref{lem:holversheaf} holds.

		\item ($ m \mapsto K_m(x) $) Let (I)--(IV) hold. Assume $ F, G $ are $ C^1 $ in a neighborhood of $ X^1 \otimes_{u} Y^1 $ instead of condition (b) (i) in \autoref{lem:final}. Then \autoref{lem:final} holds.
		A similar result to \autoref{lem:final} (2) also holds.

		\item We leave similar modifications for the continuities of $ f $ in \autoref{continuityf} to the readers, and the results in \autoref{bounded} as well.
	\end{enumerate}
\end{thm}
\begin{proof}
	Without any changes except some notations (e.g. $ X_{m_0} \times Y_{u(m_0)} $ replaced by $ X^0_{m_0} \times Y^0_{u(m_0)} $, $ X_{m_0} $ by $ X^0_{m_0} $, $ X_m \times Y_{u(m)} $ by $ X^1_m \times Y^1_{u(m)} $, and $ X_m $ by $ X^1_m $), the proofs given in \autoref{leaf} to \autoref{HolderivativeBase} also show this theorem holds.
\end{proof}

\begin{rmk}
	In \autoref{leaf}--\autoref{HolderivativeBase}, we in fact take $ X^i_m = X_m $, $ i = 0,1,2 $, and $ Y^i_m = Y_m $, $ i = 0, 1 $. In this case, (IV) is redundant. The conditions in \autoref{thmA} imply (I) (III), and also (II) if $ F_m(\cdot), G_m(\cdot) \in C^1 $ (see \autoref{lem:c1}).

	In many cases, \textbf{(i)} we might take $ X^i_m = \{i_X(m)\} $ (and $ Y^1_m = \{i_Y(m)\} $) if $ i: M \to X \times Y $ is an invariant section of $ H $. Now (IV) holds if (H3) is assumed, i.e., $ i $ is a $ 0 $-section of $ X \times Y $ with respect to $ \mathcal{A} \times \mathcal{B} $. For example, see \autoref{00diff}, \autoref{lem:base0} and \autoref{lem:baseK1} (2).

	\textbf{(ii)} Since in some situations, we use bump functions or blid maps (see \autoref{bump}) to truncate the systems, the fibers of $ X, Y $ are usually Banach spaces (i.e. (H2) (i) holds). In this case, we generally take $ X^i_m = X_m(r_i) $ ($ r_2 < r_0 < r_1 $) and $ Y^i_m = Y_m(\delta_i) $ ($ \delta_0 < \delta_1 $). Now (IV) is guaranteed by $ X, Y $ having $ \varepsilon $-almost uniform $ C^{0,1} $-fiber trivializations on $ M_1 $ with respect to $ \mathcal{A} $, $ \mathcal{B} $. If $ \sup_{m}\lambda_s(m) < 1 $ and $ i = 0 $ is an invariant section of $ H $, then $ x_m(X^1_m) \subset X^2_{u(m)} $ is satisfied. Thus, we can apply \autoref{thm:LocalRegularity} to deduce the regularities of the invariant graph obtained in \autoref{thm:local}. (Also note that the conditions of \autoref{thm:local} have implied (I)--(III) hold; for (II), see also \autoref{lem:cAB}.) Here it is worth noting that although we use the radial retraction (see \eqref{radial}), the systems would be non-smooth in all $ X \times Y $ but they are smooth in $ X^1 \times Y^1 $. In general Banach spaces, radial retraction always exists but smooth and Lipschitz bump functions might not, unlike in $ \mathbb{R}^n $ or separable Hilbert spaces. \autoref{thm:LocalRegularity} is sometimes important for infinite-dimensional dynamical systems.
	We refer readers to \autoref{foliations} and \cite{Che18b} for more applications of cases \textbf{(i) (ii)}.

	Finally, when we work in Riemannian manifolds, the local bundle charts in $ \mathcal{A} $ and $ \mathcal{B} $ are usually taken as parallel translations along geodesics, so they are isometric (see also \cite[Section 1.8]{Kli95} for more results on the local isometric trivialization, particularly Theorem 1.8.23 thereof). For (IV), in this situation it may happen $ X^2_{m} = X^0_m = X^1_m = X_m(r) $, $ Y^0_m = Y^1_m = Y_m(\delta) $.
\end{rmk}

%% file: sect6.tex
\chapter{Some Applications}\label{application}

In the following, we give some applications of our main results, namely the $ C^{k,\alpha} $ section theorem (\autoref{section}), and the existence and regularity of (un)strong stable foliations, fake invariant foliations and holonomies over laminations for bundle maps (\autoref{foliations}). For applications to differential equations, see \cite{Che18c}.

\section{$ C^{k,\alpha} $ section theorem} \label{section}

The $ C^{k,\alpha} $ section theorem, which appeared in \cite[Chapters 3, 6]{HPS77}, \cite[Theorem 3.2]{PSW97} and \cite[Theorem 10]{PSW12} in different settings, has wide applications in invariant manifold theory, e.g. it is used to prove the regularity of invariant manifolds (and invariant foliations). As a first application of our paper's main results obtained in \autoref{invariant} and \autoref{stateRegularity}, we will give a $ C^{k,\alpha} $ section theorem in more general settings.

\begin{enumerate}[({HH}1)]
	\item Let $ M $ be a locally metrizable space (associated with an open cover $ \{ U_m: m \in M \} $); see \autoref{defi:locallyM}. The metric in $ U_{m} $ is denoted by $ d_m $. Let $ M_1 \subset M $.

	\item Let $ (X, M, \pi) $ be a $ C^{0} $ topology bundle (see \autoref{bundleP} \eqref{topologyBundle}) with metric fibers. Let
	\begin{multline*}
	\mathcal{A} = \{ (U_{m_0}(\epsilon), \varphi^{m_0}) ~\text{a $ C^0 $ bundle chart of $ X $ at}~ m_0:\\ \varphi^{m_0}_{m_0} = \id, m_0 \in M_1, 0 < \epsilon < \epsilon' \},
	\end{multline*}
	be a bundle atlas of $ X $ on $ M_1 $. $ (X, M, \pi) $ has $ \varepsilon $-almost uniform $ C^{0,1} $-fiber trivializations on $M_1$ with respect to $ \mathcal{A} $, where $ \varepsilon > 0 $ is small; see \autoref{uniform lip bundle}.

	\item (Lipschitz continuity of $ h^{-1} $) Let $ h : M_1 \to M $ be a $ C^0 $ invertible map. Assume there exist an $ \varepsilon_1 > 0 $ and a function $ \mu: M \to \mathbb{R}_+ $ such that 

	(i) for every $ m_0 \in M_1 $, $ h^{-1}(U_{m_0}(\mu^{-1}(m_0) \varepsilon_1)) \subset U_{h^{-1}(m_0)} $,
	\[
	d_{h^{-1}(m_0)} (h^{-1}(m), h^{-1}(m_0)) \leq \mu(m_0) d_{m_0}(m, m_0),~m \in U_{m_0}(\mu^{-1}(m_0) \varepsilon_1),
	\]

	(ii) $ \sup_{m_0 \in M_1} \mu(m_0)  < \infty $.
\end{enumerate}

\begin{thm}[$ C^{k,\alpha} $ section theorem I]\label{thm:sect1}
	Let \textnormal{(HH1), (HH2), (HH3)} hold. Let $ f: X \to X $ be a $ C^0 $ bundle map over $ h $. Assume that the fibers of $ X $ are uniformly bounded, i.e. $ \sup_m\diam X_m < \infty $, and that $ \lip f_m \leq \lambda(m) $, where $ \lambda : M \to \mathbb{R}_+ $ is $ \varepsilon $-almost uniformly continuous around $ M_1 $ and $ \varepsilon $-almost continuous (see \autoref{ucontinuous}), and $ \sup_m\lambda(m) < 1 $. If $ \varepsilon $ is sufficiently small, then the following assertions hold:
	\begin{enumerate}[(1)]
		\item There is a unique section $ \sigma : M \to X $ which is invariant under $ f $. Also $ \sigma $ is $ C^0 $.
		\item In addition, if $ f_m $ depends in a uniformly Lipschitz fashion on the base point $ m $ (around $ M_1 $ with respect to $ \mathcal{A} $) and $ \sup_m \mu^{\theta}(m) \lambda (m) < 1 $ (bunching condition), then $ \sigma $ is uniformly (locally) $ C^{0,\theta} $ (around $ M_1 $).
	\end{enumerate}
\end{thm}

Here we explain some notions.
\begin{enumerate}[(a)]
	\item That $ f_m $ depends in a uniformly Lipschitz fashion on the base point $ m $ around $ M_1 $ (see also \cite[Section 2.5]{AV10} using this expression without the word `uniformly') means that $ \widehat{f} $, the local representation of $ f $ at $ m_0 \in M_1 $ with respect to $ \mathcal{A} $, i.e.,
	\[
	\widehat{f}_{m_0}(m, x) \triangleq (\varphi^{m_0}_{m})^{-1} \circ f_{h^{-1}(m)} \circ ( \varphi^{h^{-1}(m_0)}_{h^{-1}(m)}(x) ) :  U_{m_0} (\mu^{-1}(m_0) \varepsilon_1) \times X_{h^{-1}(m_0)} \to X_{m_0},
	\]
	where $ \varphi^{m_0}, \varphi^{h^{-1}(m_0)} \in \mathcal{A} $ are two bundle charts at $ m_0, h^{-1}(m_0) $, respectively, satisfies
	\[
	|\widehat{f}_{m_0}(m, x) - \widehat{f}_{m_0}(m_0, x)| \leq M_0 d_{m_0}(m, m_0),
	\]
	where the constant $ M_0 $ does not depend on $ m_0 \in M_1 $ (see \autoref{vsII} and \autoref{rmk:holder}); see also \cite{PSW12}.

	\item Similarly, that $ \sigma $ is uniformly (locally) $ C^{0,\theta} $ (around $ M_1 $) means that $ \widehat{\sigma} $, the local representation of $ \sigma $ at $ m_0 \in M_1 $ with respect to $ \mathcal{A} $, i.e.,
	\[
	\widehat{\sigma}_{m_0}(m) \triangleq (\varphi^{m_0}_{m})^{-1} \circ \sigma (m) :  U_{m_0} (\mu^{-1}(m_0) \varepsilon_1) \to X_{m_0},
	\]
	satisfies
	\[
	|\widehat{\sigma}_{m_0}(m) - \widehat{\sigma}_{m_0}(m_0)| \leq C_0 d_{m_0}(m, m_0)^{\theta},
	\]
	where the constant $ C_0 > 0 $ does not depend on $ m_0 \in M_1 $.
\end{enumerate}

\begin{rmk}\label{rmk:bounded}
	The assumption that the fibers of $ X $ are uniformly bounded is only for simplicity. One can replace this assumption by \textbf{(a)} \emph{there is a section $ i: M \to X $ such that $ \sup_m d( f_m(i(m)), i(h(m)) ) < \infty $}, where $ d $ is the metric in $ X_m $. The existence result now follows from \autoref{thmA}, and others are the same; the regularity results would need that \textbf{(b)} \emph{$ i $ satisfies assumption (UB) of \autopageref{b3-}}. The uniqueness now means that if there is an invariant section $ \sigma' : M \to X $ of $ f $ such that $ \sup_m d(\sigma'(m), i(m)) < \infty $, then $ \sigma' = \sigma $. Note that \emph{no} requirement that $ X $ has a continuous (or H\"older) section is made.
\end{rmk}

\begin{proof}[Proof of \autoref{thm:sect1}]
	Let $ \overline{M} = M $, $ \overline{X}_{m} = \{ m \} $, $ \overline{Y}_m = X_m $, $ u = h^{-1} $, and
	\[
	G_m(x, y) = f_{h^{-1}(m)}(y) : \overline{X}_m \times \overline{Y}_{u(m)} \to \overline{Y}_m, ~
	F_m (x, y) = u(m) : \overline{X}_m \times \overline{Y}_{u(m)} \to \overline{X}_{u(m)}.
	\]
	Then $ H \sim (F, G) $. Note that $ H $ satisfies the (A)($ \alpha $, $ \lambda $) (B)($ 0 $, $ 0 $) condition for any positive constant $ \alpha < 1 $. Now \autoref{thmA} or \autoref{thmB} gives (1); that $ \sigma $ is $ C^0 $ can be proved similarly to \autoref{lem:continuity_f}. (2) is a direct consequence of \autoref{lem1:sheaf}.
\end{proof}

\begin{enumerate}[(HE1)]
	\item Let $ M $ be a $ C^1 $ Finsler manifold with Finsler metric $ d $ in its components, which satisfies (H1c) in \autoref{settingOverview}; see also \autoref{examples} for some examples.

	\item $ (X, M, \pi) $ is a $ C^{1} $ bundle with a $ C^{0} $ connection $ \mathcal{C}^X $. Take a $ C^1 $ (regular) normal bundle atlas $ \mathcal{A} $ of $ X $.
	Assume $ (X, M, \pi) $ has $ \varepsilon $-almost uniform $ C^{1,1} $-fiber trivializations on $ M_1 $ with respect to $ \mathcal{A} $ (see \autoref{uniform lip bundle}), where $ \varepsilon > 0 $ is sufficiently small.
	The fibers of $ X $ are Banach spaces (see also \autoref{fibersG}).

	\item $ h : M_1 \to M $ is a $ C^1 $ invertible map. Assume $ h $ satisfies the following:
	\begin{enumerate}[(i)]
		\item There is a function $ \mu : M \to \mathbb{R}_+ $ such that $ |Dh^{-1}(m)| \leq \mu(m) $.
		\item $ m \mapsto |Dh^{-1}(m)| $ is $ \varepsilon $-almost uniformly continuous around $ M_1 $; see \autoref{ucontinuous}.
		\item $ \sup_{m \in M_1} \mu (m) < \infty $.
	\end{enumerate}
\end{enumerate}

\begin{thm}[$ C^{k,\alpha} $ section theorem II]\label{thm:sect2}
	Let the assumptions in \autoref{thm:sect1} hold with \textnormal{(HH1), (HH2), (HH3)} replaced by \textnormal{(HE1), (HE2), (HE3)}. (Also the assumption that the fibers of $ X $ are uniformly bounded can be replaced by \autoref{rmk:bounded} \textbf{(a) (b)}). Assume $ f $ is $ C^1 $. Let $ \sigma $ be the unique invariant section of $ f $. Then the following results hold:
	\begin{enumerate}[(1)]
		\item If $ f_m $ depends in a uniformly Lipschitz fashion on the base points (around $ M_1 $ with respect to $ \mathcal{A} $) and the bunching condition holds, $ \sup_m \mu(m) \lambda (m) < 1 $, then $ \sigma $ is $ C^{1} $ (and uniformly (locally) $ C^{0,1} $ around $ M_1 $).
		\item In addition, assume that $ D^vf, \nabla f $ are uniformly (locally) $ C^{0,1} $ around $ M_1 $, $ h^{-1} $ is uniformly (locally) $ C^{1,1} $ around $ M_1 $ and the bunching condition $ \sup_m \mu^{1+\theta}(m) \lambda (m) < 1 $ holds. Then $ m \mapsto \nabla_m\sigma(m) $ is uniformly (locally) $ C^{0,\theta} $ (around $ M_1 $).
	\end{enumerate}
\end{thm}

	Some interpretations of the conditions in the above theorem are needed. We use the notations in assumption $ (\blacksquare) $ (on \autopageref{C1MM}).
	\begin{enumerate}[(a)]
		\item We say that $ D^vf, \nabla f $ are uniformly (locally) $ C^{0,1} $ around $ M_1 $, which means that (i) $ Df_{m}(\cdot), \nabla_{m} f_{m}(\cdot) \in C^{0,1} $ uniformly for $ m \in M $, and (ii) for the local representations $ D_x\widehat{f}_{m_0}, \widehat{D}_{m} f_{m_0} $ of $ D^vf, \nabla f $ with respect to $ \mathcal{A}, \mathcal{M} $, one has
		\begin{multline*}
		\max\left\{ |D_x\widehat{f}_{m_0}(m_1,x) - D_x\widehat{f}_{m_0}(m_0,x) |, |\widehat{D}_{m} f_{m_0}(m_1,x) - \widehat{D}_{m} f_{m_0}(m_0,x) | \right\} \\ 
		\leq M_0d(m_1,m_0)
		\end{multline*}
		for all $ m_1 \in U_{m_0} (\mu^{-1}(m_0) \varepsilon_1) $, $ x \in X_{m_0} $, $ m_0 \in M_1 $, where $ M_0 $ is independent of $ m_0 \in M_1 $; see also \autoref{rmk:hFG}.
		\item That $ h^{-1} $ is uniformly (locally) $ C^{1,1} $ around $ M_1 $ means that for the local representation $ \widehat{Dh^{-1}}_{m_0} $ of $ Dh $ at $ m_0 $, defined by
		\[
		\widehat{Dh^{-1}}_{m_0} (m)v \triangleq D\chi_{h^{-1}(m_0)}(h^{-1}(m)) Dh^{-1}(m)D\chi_{m_0}^{-1}(\chi_{m_0}(m))v,
		\]
		where $ (m,v) \in U_{m_0}(\epsilon') \times T_{m_0}M $, one has 
		\[
		\|\widehat{Dh^{-1}}_{m_0} (m) - Dh^{-1}(m_0)\| \leq C_0 d(m,m_0),~m \in U_{m_0}(\epsilon'_1), m_0 \in M_1,
		\]
		for some small $ \epsilon'_1 > 0 $ and some constant $ C_0 > 0 $ (independent of $ m_0 $); see also \autoref{def:manifoldMap}.
		\item Similarly, that $ m \mapsto \nabla_m\sigma(m) $ is uniformly (locally) $ C^{0,\theta} $ (around $ M_1 $) means that for the local representations $ \widehat{D}_{m} \sigma_{m_0} $ of $ \nabla \sigma $ with respect to $ \mathcal{A}, \mathcal{M} $, one has
		\[
		| \widehat{D}_{m}\sigma_{m_0}(m_1) - \nabla_{m_0}\sigma(m_0) | \leq C_0 d(m_1, m_0)^{\theta}, ~m_1 \in U_{m_0}(\epsilon''), m_0 \in M_1,
		\]
		for some small $ \epsilon'' > 0 $ and some constant $ C_0 $ independent of $ m_0 \in M_1 $.
		\item See also \autoref{lipbase} (with $ \hat{\eta}^{X,0}_{m_0}(\epsilon_0, x) , \hat{\eta}^{X,1}_{m_0}(\epsilon_0, x)  $ smaller than a fixed constant in this case) and \autoref{lipcon} for a classical way of describing the $ C^{1,1} $ continuity of $ \sigma $ when the transition maps (with respect to $ \mathcal{A} $) and $ \mathcal{C}^X $ are uniformly (locally) Lipschitz.
	\end{enumerate}

\begin{proof}
	Using the notations in the proof of \autoref{thm:sect1} and considering $ \overline{X}_m $ as a zero space, now from \autoref{smoothbase1} and \autoref{final1} we obtain the results.
\end{proof}

\section{Invariant foliations}\label{foliations}

Invariant foliations have been investigated by many authors: see e.g. \cite{HPS77, Fen72, PSW97, PSW12, CY94} in the finite-dimensional setting, \cite{CLL91, BLZ00, BLZ99, Rue82, LL10} in the infinite-dimensional setting, \cite{Cha08} even in metric spaces and many others in different settings, too numerous to list here. The settings include different hyperbolicity assumptions (uniform hyperbolicity, non-uniform hyperbolicity, partial hyperbolicity, etc.), different types of `dynamics' (maps, semiflows, cocycles, random dynamical systems, etc.), and different state spaces (compact Riemannian manifolds, Riemannian manifolds having bounded geometry, $ \mathbb{R}^n $, Banach spaces, metric spaces, etc.). Invariant foliations (and invariant manifolds) are extremely useful in the study of qualitative properties of a dynamical system near invariant sets. And they also provide some coordinates such that the dynamical system can be decoupled (see e.g. \autoref{decoupling} and \cite{Lu91}).
Roughly, strong (un)stable foliations exist in partially hyperbolic systems, but center foliations do not (see e.g. \cite{RHRHTU12}). The notion of \emph{dynamical coherence} is important in the study of center foliations (see \cite{HPS77}). To avoid assuming dynamical coherence, Burns and Wilkinson \cite{BW10} introduced `fake invariant foliations' (see also \autoref{fake} below) which are locally invariant.

The regularity of invariant foliations is of crucial importance. For example, the H\"older and smooth regularity of foliations (or more precisely holonomy maps) are not only vital in the study of stably ergodic behavior (see e.g. \cite{PS97}) but also important for higher smooth linearization (see e.g. \cite{ZZJ14} and \autoref{rmk:decoupling}). Usually, the smoothness of leaves of strong (un)stable foliations is the same as that of dynamical systems (see e.g. \cite{HPS77}). However, the foliations may only be H\"older. The absolute continuity of invariant foliations usually requires the H\"older continuity of the foliations as preparations; see e.g. \cite{LYZ13, BY17a} in the infinite-dimensional setting.

As a second application of our results, we present some results on strong (un)stable foliations, fake invariant foliations and holonomies over laminations for bundle maps, which are extremely useful in the study of partially hyperbolic dynamics.
\emph{For some notions related to foliations we will use, see \autoref{regFoliation}}.

\subsection{Invariant foliations for bundle maps}

The following result was re-proved by many authors: see e.g. \cite{HPS77, Sta99, Cha04, Cha08} and \cite{CHT97,CLL91, Hal61, Yi93, CY94, CL97}.

\begin{thm}[Invariant foliations: global version]\label{thm:bundlemaps}
	Let $ (X, M, \pi_1) $, $ (Y, M, \pi_2) $ be bundles with metric fibers. Let $ u: M \to M $ be a map. Let $ H \sim (F,G): X \times Y \to X \times Y $ be a bundle map over $ u $ satisfying (i) the (A$'$)$(\alpha, \lambda_u)$ (B)$(\beta; \beta', \lambda_s)$ condition (\textbf{or} (i$ ' $) the (A)$(\alpha, \lambda_u)$ (B)$(\beta; \beta', \lambda_s)$ condition), where $\alpha, \lambda_u, \beta, \beta', \lambda_s$ are \emph{bounded} functions $ M \rightarrow \mathbb{R}_+ $. Set $ \vartheta(m) = (1 - \alpha(m)\beta'(u(m)))^{-1} $ if (i) holds, and $ \vartheta(m) = 1 $ if (i$ ' $) holds. In addition,
	\begin{enumerate}[(a)]
		\item (angle condition) $ \sup_m \alpha(m) \beta'(u(m)) < 1 $, $ \beta'(u(m)) \leq \beta(m)$, $ \forall m \in M $,
		\item (spectral condition) $ \sup_m  \lambda_s(m) \lambda_u(m) \vartheta(m)  < 1 $.
	\end{enumerate}

	\textbf{Existence}. There exist unique $ W^s_m \triangleq \bigsqcup_{q \in X_m \times Y_m} W^s_m(q) $, $ m \in M $, such that the following (1), (2) hold.
	\begin{enumerate}[(1)]
		\item $ q \in W^s_m(q) $ and each $ W^s_m(q) $ is a Lipschitz graph of $ X_m \to Y_m $ with Lipschitz constant no more than $ \beta'(m) $;
		\item $ H_m W^s_m(q) \subset W^s_{u(m)}(H_m(q)) $, $ q \in X_m \times Y_m $, $ m \in M $.

		In addition,
		\item $ W^s_m(q) = \{ p \in X_m \times Y_m: d_{u^n(m)}( H^{(n)}_m(p), H^{(n)}_m(q) ) \lesssim \varepsilon^{(n)}(m), n \to \infty \} $, where $ \varepsilon(\cdot) $ is a function $ M \to \mathbb{R}_+ $ satisfying $ \lambda_s(m) + \varsigma < \varepsilon(m) < (\lambda_u(m)\vartheta(m))^{-1} - \varsigma $ for all $ m \in M $, where $ \varsigma $ is small and $ d_m $ is the metric in $ X_m \times Y_m $;
		\item for each $ m \in M $, $ W^s_m $ is a $ C^0 $ `foliation' of $ X_m \times Y_m $; in particular, if $ W^s_m(q_1) \cap W^s_m(q_2) \neq \emptyset $, then $ q_1 \in W^s_m(q_2) $.

		\noindent\textbf{Regularity}.
		\item If $ \sup_m \lip H_m(\cdot) < \infty $, then $ W^s_m $ is a uniformly (locally) H\"older `foliation' (uniformly for $ m $).
		\item Moreover, suppose $ X_m, Y_m $, $ m \in M $, are Banach spaces (see also \autoref{fibersG}).
		\begin{enumerate}[(i)]
			\item If $ F_m(\cdot), G_m(\cdot) $ are $ C^1 $ for all $ m \in M $, then every $ W^s_m(q) $ is a $ C^1 $ graph. If $ F_m(\cdot), G_m(\cdot) $ are $ C^{1, \gamma} $ ($ \gamma > 0 $) uniformly for $ m $, then every $ W^s_m(q) $ is a $ C^{1,\zeta} $ graph (uniformly for $ m, q $) where $ 0< \zeta \leq \gamma $; if in addition $ \sup_m \lip H_m(\cdot) < \infty $, then the tangent distribution $ T_q W^s_m(q) $ depends in a uniformly (locally) H\"older fashion on $ q \in X_m \times Y_m $, where the H\"older constants are independent of $ m $.

			\item If $ \sup_m \lip H_m(\cdot) < \infty $, $ F_m(\cdot), G_m(\cdot) \in C^{1,1} $ uniformly for $ m $, and the following bunching condition holds:
			\[
			\sup_m \lip H_m(\cdot) \lambda_s(m) \lambda_u(m) \vartheta(m) < 1,
			\]
			then $ W^s_m $ is a $ C^1 $ foliation of $ X_m \times Y_m $. In addition, assume (a) $ H_m(\cdot) \in C^{1,1} $ uniformly for $ m $, and (b) $ \sup_m\lambda_s(m) < 1 $, or (b$ ' $)
			\[
			\sup_m\lambda^2_s(m) \lambda_u(m) \vartheta(m) < 1, ~\sup_m \lip H_m(\cdot) \lambda^2_s(m) \lambda_u(m) \vartheta(m) < 1.
			\]
			Then the foliation $ W^s_m $ is uniformly (locally) $ C^{1,\varsigma} $ for some $ \varsigma > 0 $ (uniformly for $ m $).
		\end{enumerate}
	\end{enumerate}
\end{thm}

\begin{rmk}
	\begin{enumerate}[(a)]
		\item The H\"older exponent $ \theta $ in item (5) can be chosen as $ ((\mu / \lambda_s)^*)^\theta \vartheta \lambda_s \lambda_u < 1 $ (see \autoref{absgc}), where $ \mu(m) = \lip H_m(\cdot) $. If $ \sup_m \diam H_m(X_m \times Y_m) < \infty $ or $ \theta = 1 $, then $ ((\mu / \lambda_s)^*)^\theta \lambda_s \lambda_u \vartheta < 1 $ could be read as
		\[
		\sup_m(\mu(m) / \lambda_s (m) )^\theta \lambda_s(m) \lambda_u(m) \vartheta(m) < 1.
		\]
		\item Let $ W^s_m(q) = \graph f_{(m,q)} $, where $ f_{(m,q)}: X_m \to Y_m $. That the foliation $ W^s_m $ is uniformly (locally) $ C^{k,\gamma} $ ($ \gamma \geq 0 $) uniformly for $ m $ means that
		\[
		(q,x) \mapsto f_{(m,q)} (x) :( X_m \times Y_m ) \times X_m \to Y_m
		\]
		is (locally) $ C^{k,\gamma} $, and if $ \gamma > 0 $, the (locally) $ C^{k,\gamma} $ constant is less than a fixed constant independent of $ m $. $ W^{s}_m (q) $ is a $ C^{k,\gamma} $ graph (uniformly for $ m, q $) if $ f_{(m,q)} (\cdot) \in C^{k,\gamma} $ with $ \sup_{m,q}|f_{(m,q)} (\cdot)|_{C^{k,\gamma}} < \infty $ (if $ \gamma > 0 $) and $ (q,x) \mapsto D^{i}f_{(m,q)}(x) $ being $ C^0 $, $ i=0,1,\ldots,k $.
		\item In the first statement of item (6) (i), suppose further that $ DF_m(\cdot), DG_m(\cdot) $, $ m \in M $, are uniformly continuous; then $ q \mapsto f_{(m,q)}(\cdot) $ is uniformly continuous in $ C^1 $-topology on bounded subset of $ X_m $, i.e. for any bounded subset $ A \subset X_m $, one has
		\[
		\sup_{x \in A} | f_{(m,q')}(x) - f_{(m,q)}(x) | + \sup_{x \in A} \| Df_{(m,q')}(x) - Df_{(m,q)}(x) \| \to 0,
		\]
		as $ q' \to p $, uniformly for $ p $. This follows from \autoref{lem:continuity_f} and \autoref{lem:baseK1}.
		\item We do not consider the continuity and smoothness of $ m \mapsto W^s_m(q) $. This is what our general results deal with, so see \autoref{stateRegularity} for a detailed study. A simple result is that if $ M $ is a topological space, $ X, Y $ are $ C^{0,1} $-fiber bundles (see \autoref{def:lipbundle}), and $ u, H $ (also $\alpha, \lambda_u, \beta, \beta', \lambda_s$) are continuous, then $ (m,q) \mapsto W^s_m(q) $ is continuous pointwise, i.e., $ (m,q,x) \mapsto f_{(m,q)}(x) $ is continuous (as a bundle map $ (X \times Y) \times X \to Y $ over $ \id $).
	\end{enumerate}
\end{rmk}

\begin{proof}[Proof of \autoref{thm:bundlemaps}]
	Let $ \widehat{M} = X \times Y $ be a bundle over $ M $, where the topology of $ \widehat{M} $ is given by the fiber topology (see \autoref{immersed}). Define the bundle $ \widehat{X} \times \widehat{Y} $ over $ \widehat{M} $ as $ (X \times Y) \times (X \times Y) $. Consider the bundle map $ \widehat{H} $ over $ H $, defined by
	$ \widehat{H}(m,p, q) = (H(m,p), H_m(q))  $, where $ (m,p) \in \widehat{M} $, $ q \in \widehat{X}_{(m,p)} \times \widehat{Y}_{(m,p)} = X_m \times Y_m $.
	Set $ i(m,p) = p: \widehat{M} \to \widehat{X} \times \widehat{Y} $. Then $ i $ is an invariant section of $ \widehat{H} $. Now the existence of $ W^s_m $ such that (1) (2) hold follows from \autoref{thmA} applied to $ \widehat{H} $, $ \widehat{X} \times \widehat{Y} $, $ H $, $ \widehat{M} $, $ i $; (3) follows from \autoref{properties} and (4) is a consequence of \autoref{lem:continuity_f} and (3); (5) follows from \autoref{h3*}.

	The first and second assertions in item (6) (i) follow from \autoref{lem:leaf1} and \autoref{lem:leaf1a}, respectively. The third statement in item (6) (i) follows from \autoref{lem:base0}. The results in item (6) (ii) are consequences of \autoref{smoothbase} and \autoref{lem:final}. Here note that $ F_m(\cdot), G_m(\cdot) \in C^{1} $ implies $ H_m(\cdot) \in C^1 $.
\end{proof}

For a local version, one can use \autoref{lem:general} and radial retractions (or smooth bump functions or blid maps (see \autoref{bump}) if one needs smooth results). See also \autoref{thm:fake1} below for a local version for maps. The parallel results for continuous cocycles, which we do not give here, can be deduced from the bundle map case; see also \cite{Che18c}.

\subsection{Strong stable laminations}\label{sslaminations}

Let us consider the local version of \autoref{thmA} about maps, as well as that of the relevant results in \autoref{properties}. The parallel results for semiflows also hold. The results can hold for Lipschitz maps in metric spaces (which include two-sided shifts of finite type and restrictions of Axiom A diffeomorphisms to hyperbolic basic sets) and smooth maps in Finsler (Banach) manifolds. First we re-prove a classical result in the finite-dimensional setting to show how our main results can be applied.

\noindent\textbf{\emph{(Setting)}}: Let $ f: M \to M $ be a $ C^{1,1} $ map. For simplicity, assume $ M $ is a smooth compact Riemannian manifold without boundary and $ f $ is invertible. Assume $ f $ is partially hyperbolic with $ TM = X^{s} \oplus X^{c} \oplus X^{u} $ where $ X^{s} $, $ X^{c} $, $ X^{u} $ are three continuous vector subbundles of $ TM $ and invariant under $ Df $, and $ Df $ in $ X^{s} $ and $ X^{u} $ uniformly contracts and expands, respectively. Let $ T^\kappa f = Df|_{X^\kappa} $, $ \kappa = s, c, u $. Then
\[
\| T^s_m f \| < 1, ~ \| (T^u_m f)^{-1} \| < 1,~ \| T^s_m f \| \cdot \| (T^c_m f)^{-1} \| < 1, ~\| T^c_m f \| \cdot \| (T^u_m f)^{-1} \| < 1, ~ m \in M.
\]
The strong stable (resp. unstable) foliation $ W^{s} $ (resp. $ W^{u} $) exists uniquely such that $ T_mW^{s}_m = X^{s}_m $ and $ f(W^{s}_m) \subset W^{s}_{f(m)} $ (resp. $ T_mW^{u}_m = X^{u}_m $ and and $ f^{-1}(W^{u}_m) \subset W^{u}_{f^{-1}(m)} $), where $ W^s_m $ (resp. $ W^u_m $) is the leaf of $ W^s $ (resp. $ W^{u} $) through $ m $; see \cite{HPS77} or \autoref{corr:smoothcase} below.
Assume there are ($ f $-invariant) foliations $ W^{cs} $ and $ W^{cu} $ such that $ TW^{cs} = X^{s} \oplus X^{c} $, $ TW^{cu} = X^{c} \oplus X^{u} $. Under this context, $ H $ is said to be \emph{dynamically coherent} (see \cite{HPS77}) and it also yields the integrability of $ X^c $. By the characterization of $ W^s $, we see $ W^s $ subfoliates $ W^{cs} $. In \cite{PSW97, PSW00}, Pugh, Shub and Wilkinson showed the following:
\begin{thm}[\cite{PSW97, PSW00}]\label{holfoliation}
	Under the above setting, if the following center bunching condition holds:
	\[
	\| T^{c}_mf \| \cdot \|T^{s}_m f\| \cdot \| (T^{c}_m f)^{-1} \| < 1,~m \in M,
	\]
	then $ W^{s} $ is a $ C^1 $ foliation inside each leaf of $ W^{cs} $. In fact, it is also $ C^{1,\zeta} $ ($ \zeta > 0 $) under the additional assumption $ \| T^c_m f \|^2 \cdot \| (T^u_m f)^{-1} \| < 1, ~ m \in M $. (A similar fact holds about $ W^u $, $ W^{cu} $ as well.)
\end{thm}

\begin{proof}
	Since $ M $ is a smooth compact boundaryless Riemannian manifold, there is a $ \delta > 0 $ such that for every $ k \in \mathbb{N} $, the Riemannian metric up to $ k $th order derivatives and the Christoffel symbols in the normal coordinates up to $ k $th order derivatives are all uniformly bounded in normal coordinates of radius $ \delta $ around each $ m \in M $, with bounds independent of $ m \in M $. In particular, for some small $ \delta > 0 $, the normal coordinate or exponential map $ \exp_m $ at $ m $,
	\[
	\exp_m : T_mM(\delta) \to U_{m}(\delta) = \{ m' \in M: d(m', m) < \delta \},
	\]
	satisfies (i) $ |D \exp^{\pm 1}_m(\cdot)| \leq C $ (so $ |\exp^{-1}_m(m_1)| \leq Cd(m,m_1) $), and (ii)
	\begin{equation}\label{translation}
	|\exp^{-1}_{m_1} \exp_m (z_0 + z) - P^{m}_{m_1} z | \leq C d(m_1, m)|z|,
	\end{equation}
	for some constant $ C > 0 $ independent of $ m, m_1 $, where $ z_0 = \exp^{-1}_m m_1 $, $ m_1 \in U_{m}(\delta/4) $, $ z \in T_mM(\delta / 4) $, and $ P^{m}_{m_1}: T_{m}M \to T_{m_1} M $ is the parallel translation along the geodesic connecting $ m, m_1 $. For the convenience of the readers, we prove \eqref{translation} below.
	What we really need here is that $ M $ has $ 3 $th order \emph{bounded geometry} (see \autoref{defi:bounded}).

	\begin{proof}[Proof of \eqref{translation}]
		For the Christoffel symbol $ \varGamma^{m_1}(z) \in L(T_{m_1}M, T_{m_1}M) $, $ z \in  T_{m_1}M(\delta / 4) $ in the local chart $ \exp^{-1}_{m_1} $ (see also \autoref{def:connection}), $ z \mapsto \varGamma^{m_1}(z) $ is $ C^{1,1} $ uniformly for $ m_1 \in M $. The constructions of geodesic and parallel translation are through solving the following two ODEs, respectively:
		\[
		u'' = -\varGamma^{m_1}(u)(u',u'), ~y' = -\varGamma^{m_1}(x_{a})(x'_a,y),
		\]
		where $ x_a(t) $, $ 0 \leq t\leq 1 $, is the representation of the geodesic connecting $ \exp_{m_1}a $ and $ m_1 $ with $ x_{a}(0) = \exp_{m_1}a $ in the local chart $ \exp^{-1}_{m_1} $. Denote the two solutions associated with the above ODEs by $ u(t)(z,v) $, $ y(t)(a)v $, respectively, where $ u(0)(z,v) = z, u'(0)(z,v) = v $, $ y(0)(a)v = v $. Then
		\begin{gather*}
		u(1)(\exp^{-1}_{m_1}m, D\exp^{-1}_{m_1}(m) (z_0 + z)) = \exp^{-1}_{m_1} \exp_m (z_0 + z), \\
		y(1)(\exp^{-1}_{m_1}m,D\exp^{-1}_{m_1}(m)z) = P^{m}_{m_1}z.
		\end{gather*}
		Since $ z \mapsto \varGamma^{m_1}(z) $ is $ C^{1,1} $, we see $ (z, v) \mapsto u(1)(z,v) $ is $ C^{1,1} $, and so $ a \mapsto x_{a}(\cdot) $ and $ a \mapsto y(1)(a) \in L(T_{m_1}M, T_{m_1}M) $ are $ C^{0,1} $, uniformly for $ m_1 $. Let $ z_1 = \exp^{-1}_{m_1}m $. Note that $ y(1)(0,v) = v $, $ u(1)(0, v) = v $, and $ u(1)(z_1, D\exp^{-1}_{m_1}(m) (z_0)) = 0 $.
		So
		\begin{align*}
		& |u(1)(z_1, D\exp^{-1}_{m_1}(m) z_0 + v) - y(1)(z_1)v| \\
		= & \left| \int_{0}^{1} D_{2} u(1) (z_1, D\exp^{-1}_{m_1}(m)z_0 + tv) v- y(1)(z_1)v ~\mathrm{d}t \right| \\
		\leq & \int_{0}^{1} | D_{2} u(1) (z_1, D\exp^{-1}_{m_1}(m)z_0 + tv) v - D_{2} u(1) (0, D\exp^{-1}_{m_1}(m)z_0 + tv)v | \\
		& \quad + | y(1)(z_1)v - y(1)(0)v | ~\mathrm{d}t
		\leq \widetilde{C}|z_1||v|.
		\end{align*}
		Letting $ v = D\exp^{-1}_{m_1}(m)z $, we finish the proof.
	\end{proof}

	Let $ \mathbb{G}(TM) $ be the Grassmann manifold of $ TM $. Let $ \Pi^{\kappa}_m \in L(T_mM, T_mM) $ be the projections such that $ R(\Pi^{\kappa}_m) = X^{\kappa}_m $ and $ \Pi^s_m + \Pi^{c}_m + \Pi^{u}_m = \id $, $ m \in M $, $ \kappa = s, c, u $. Since $ m \mapsto X^{\kappa}_m \in \mathbb{G}(TM) $ is $ C^0 $, by the smooth approximation of $ X^{\kappa} $ in $ \mathbb{G}(TM) $, there are projections $ \widehat{\Pi}^{\kappa}_m \in L(T_mM, T_mM) $, $ R(\widehat{\Pi}^{\kappa}_m) = \widehat{X}^\kappa_m $, $ \kappa = s, c, u $, such that $ m \mapsto \widehat{X}^\kappa_m $ is $ C^2 $ (and so $ C^{1,1} $) and $ \sup_m|\Pi^{\kappa}_m - \widehat{\Pi}^{\kappa}_m| \leq \varepsilon $ for sufficiently small $ \varepsilon > 0 $.

	\textbf{(I)} (base space) Let $ \widehat{M} $ ($ = M $) be $ W^{cs} $ endowed with the leaf topology (see \autoref{immersed}). Since each leaf of $ W^{cs} $ is a $ C^1 $, boundaryless and injectively immersed submanifold of the compact manifold $ M $, we see that the $ C^1 $ Finsler manifold $ \widehat{M} $ satisfies (H1b) (see \autoref{settingOverview}); the details are given as follows. (Here note that in general $ \widehat{M} $ is non-separable and so is not metrizable, but each leaf of $ W^{cs} $ is a Riemannian manifold.)
	As the open cover of $ \widehat{M} $ we take $ \widehat{U}_m = W^{cs}_m $, $ m \in M $. There is an $ \epsilon_0 > 0 $ such that the small plaque at $ m $ with radius $ \epsilon_0 $, denoted by $ \widehat{U}_m(\epsilon_0) $ (the component of $ W^{cs}_m \cap U_{m}(\epsilon_0) $ containing $ m $), is a $ C^1 $ embedded submanifold of $ M $ which can be represented as $ \exp^{-1}_m \widehat{U}_m(\epsilon_0) = \graph g_m $, where $ g_m : \widehat{X}^{cs}_m \to \widehat{X}^u_m $ with $ \sup_m \lip g_m|_{\widehat{X}^{cs}_m(\epsilon_0)} $ small and $ x \mapsto Dg_m(x) $, $ m \in M $, being equicontinuous. See also \cite[Chapters 6--7]{HPS77}. So
	\[
	\chi_{m}: \widehat{U}_{m}(\epsilon_0) \to \widehat{X}^{cs}_m \approx X^{cs}_m,
	\]
	defined by $ \chi^{-1}_{m} (x) = \exp_m(x+g_m(x)) $, satisfies $ \sup_m \sup|D\chi^{\pm1}_{m}(\cdot)| < \infty $, and $ m' \mapsto D\chi^{\pm 1}_{m}(m') $, $ m \in M $, are equicontinuous; through the map $ \Pi^{cs}_{m}: \widehat{X}^{cs}_m \to X^{cs}_m $, we can assume $ D\chi_{m}(m) = \id $.

	\textbf{(II)} (base map) Since $ W^{cs} $ is invariant under $ f $, i.e. $ f(W^{cs}_m) \subset W^{cs}_{f(m)} $, where $ W^{cs}_m $ is the leaf of $ W^{cs} $ through $ m $, we see that $ f $ induces a map in $ \widehat{M} $, also denoted by $ f: \widehat{U}_{m} \to \widehat{U}_{f(m)} $, which is $ C^1 $ with $ \lip f|_{\widehat{U}_{m} (\epsilon_0)} \triangleq \mu(m) $ approximately equal to $ \|T^{cs}_mf\| = \|T^c_mf\| $ if $ \epsilon_0 $ is small.

	\textbf{(III)} (bundles) $ \widehat{X}^{\kappa} $ ($ \kappa = s, c, u $) can be considered naturally as a bundle over $ \widehat{M} $ which is $ C^1 $. As a natural bundle atlas $ \mathcal{B}^{\kappa} $ for $ \widehat{X}^{\kappa} $ we take
	\[
	{^\kappa \varphi^{m_0}}: \widehat{U}_{m_0}(\delta) \times \widehat{X}^{\kappa}_{m_0} \to \widehat{X}^{\kappa}, ~(m,x) \mapsto (m, \widehat{\Pi}^{\kappa}_m P^{m_0}_m x).
	\]
	Since $ M \to \mathbb{G}(TM) $, $ m \mapsto \widehat{\Pi}^{\kappa}_m $, $ m \mapsto P^{m_0}_m $, are $ C^2 $, there exists a natural $ C^0 $ connection $ \mathcal{C}^{\kappa} $ for $ \widehat{X}^{\kappa} $ such that $ 	\mathcal{C}^{\kappa} D( {^\kappa \varphi^{m_0}} ) (m_0,x) = \id $, i.e. $ {^\kappa \varphi^{m_0}} $ is normal with respect to $ \mathcal{C}^{\kappa} $.
	(Here also note that $ \nabla_{m_0} P^{m_0}_{m_0} = 0 $ and what is more this connection $ \mathcal{C}^{\kappa} $ is uniformly Lipschitz in the sense of \autoref{lipcon}.) That is, $ \widehat{X}^{s} \times \widehat{X}^{cu} $ satisfies (H2c) (see \autoref{settingOverview}).

	\textbf{(IV)} (bundle map) Let
	\[
	\widehat{f}_m = \exp^{-1}_{f(m)} \circ f \circ \exp_m: T_m M(\delta) \to T_{f(m)} M.
	\]
	Since $ D\widehat{f}_m(0) = T_mf $, we see that $ \widehat{f}_m \sim (F_m, G_m) $, where $ F_m: \widehat{X}^{s}_m(r_1) \oplus \widehat{X}^{cu}_{f(m)}(r'_2) \to X^{s}_{f(m)}(r_2) $ and $ G_m: \widehat{X}^{s}_m(r_1) \oplus \widehat{X}^{cu}_{f(m)}(r'_2) \to X^{cu}_{m}(r'_1) $,
	satisfies the (A) ($\alpha $, $\lambda_u(m)$) (B) ($\beta$, $\lambda_{s}(m)$) condition (see \autoref{lem:general}),
	where $ \lambda_{s}(m) = \|T^{s}_mf\| + \zeta $, $ \lambda_u(m) = \|(T^{c}_mf)^{-1}\| + \zeta $, and $ \zeta $, $ \alpha $, $ \beta \rightarrow 0 $ as $ r_1, r'_2 \rightarrow 0 $.

	By using smooth bump functions, one finds $ \widehat{F}_m, \widehat{G}_m $, defined in all $ \widehat{X}^{s}_m \oplus \widehat{X}^{cu}_{f(m)} $ such that
	\begin{gather}
	\widehat{F}_m|_{\widehat{X}^{s}_m(r) \oplus \widehat{X}^{cu}_{f(m)}(r)} = F_m, ~\widehat{G}_m|_{\widehat{X}^{s}_m(r) \oplus \widehat{X}^{cu}_{f(m)}(r)} = G_m, \label{mmm}\\
	\widehat{F}_m(\widehat{X}^{s}_m \oplus \widehat{X}^{cu}_{f(m)}) \subset \widehat{X}^{s}_{f(m)}(r_2), ~\widehat{G}_m(\widehat{X}^{s}_m \oplus \widehat{X}^{cu}_{f(m)}) \subset \widehat{X}^{cu}_m(r'_1), \notag
	\end{gather}
	and $ \widehat{H} \sim (\widehat{F}_m, \widehat{G}_m):\widehat{X}^{s}_m \oplus \widehat{X}^{cu}_{m} \to \widehat{X}^{s}_{f(m)} \oplus \widehat{X}^{cu}_{f(m)}  $ satisfies the (A)$(\alpha, \lambda_u(m))$ (B)$(\beta, \lambda_{s}(m))$ condition (see e.g. \autoref{lem:exampleH} \eqref{it:H3}), where $ r < \min\{ r_i, r'_i: i = 1, 2 \} / 2 $ is small.
	Let $ \widetilde{F}_m(\cdot,\cdot,\cdot), \widetilde{G}_m(\cdot,\cdot,\cdot) $ be the local representations of $ \widehat{F}_m, \widehat{G}_m $ with respect to the bundle atlases $ \mathcal{B}^s $, $ \mathcal{B}^{c} \times \mathcal{B}^{u} $.
	Although $ \widehat{F}: \widehat{X}^s \otimes_f \widehat{X}^{cu} \to \widehat{X}^s $ and $ \widehat{G}: \widehat{X}^s \otimes_f \widehat{X}^{cu} \to \widehat{X}^{cu} $ are $ C^1 $ and $ \widehat{F}_m(\cdot), \widehat{G}_m(\cdot) $ are $ C^{1,1} $ uniformly for $ m $ as $ f \in C^{1,1} $, in general, $ \widetilde{F}_m, \widetilde{G}_m $ are not $ C^{1,1} $ since $ \widehat{M} $ is only $ C^1 $. Let us show:
	\begin{slem}
		The local representations $ \widetilde{F}_{m_0}(\cdot,\cdot,\cdot) $ and $ \widetilde{G}_{m_0}(\cdot,\cdot,\cdot) $ satisfy \eqref{ccc0} in \autoref{smoothbase}.
	\end{slem}
	\begin{proof}
		The local representations $ \widetilde{F}_{m_0}(\cdot,\cdot,\cdot), \widetilde{G}_{m_0}(\cdot,\cdot,\cdot) $ can be written as
		\[
		\begin{cases}
		\begin{split}
		\widetilde{F}_{{m}_0} ({m}, x, y) \triangleq (^{s}\varphi^{{f}({m}_0)}_{{f}({m})})^{-1} \circ  \widehat{F}_{m} & \circ ( {^{s}\varphi^{m_0}_{m}}(x) , {^{cu}\varphi^{f(m_0)}_{f(m)}}(y) ) :\\
		& \widehat{U}_{m_0} (\epsilon_1) \times \widehat{X}^{s}_{m_0}(\hat{r}) \times \widehat{X}^{cu}_{f(m_0)}(\hat{r}) \to \widehat{X}^{s}_{f(m_0)} ,
		\end{split}\\
		\begin{split}
		\widetilde{G}_{{m}_0} ({m}, x, y) \triangleq (^{cu}\varphi^{{m}_0}_{m})^{-1} \circ  \widehat{G}_{m} & \circ ( {^{s}\varphi^{m_0}_{m}}(x) , {^{cu}\varphi^{f(m_0)}_{f(m)}}(y) ) : \\
		& \widehat{U}_{m_0} (\epsilon_1) \times \widehat{X}^{s}_{m_0}(\hat{r}) \times \widehat{X}^{cu}_{f(m_0)}(\hat{r}) \to \widehat{X}^{cu}_{m_0},
		\end{split}
		\end{cases}
		\]
		where $ \epsilon_1 < (\sup_{m_0}\mu(m_0))^{-1}r $ is small and $ \hat{r} < \delta / 4 $ is small.
		Let $ m_0 \in M $, $ m_1 \in \widehat{U}_{m_0}(\epsilon_1) $, and $ \exp^{-1}_{m_0} m_1 = x^s_0 + x^{cu}_0 $, $ \exp^{-1}_{f(m_0)} f(m_1) = x^s_1 + x^{cu}_1 $, where $ x^{\kappa}_0 \in \widehat{X}^{\kappa}_{m_0} $, $ x^{\kappa}_1 \in \widehat{X}^{\kappa}_{f(m_0)} $, $ \kappa = s, cu $.
		Note that since $ \epsilon_1 $ is small and \eqref{mmm} holds, we see that $ \widehat{F}_{m_0}( x^s_0, x^{cu}_1) = x^s_1 $ and $ \widehat{G}_{m_0}(x^s_0, x^{cu}_1) = x^{cu}_0 $.
		Let $ z = (x, y) \in X^{s}_{m_0}(\hat{r}) \times X^{cu}_{f(m_0)}(\hat{r}) $, and
		\[
		\begin{gathered}
		\hat{y} = \widehat{G}_{m_0}(x, y), ~y'_1 = \widehat{G}_{m_0}(x + x^s_0, y + x^{cu}_1) - x^{cu}_0, \\
		\hat{x} = \widehat{F}_{m_0}(x, y), ~x'_2 = \widehat{F}_{m_0}(x + x^s_0, y + x^{cu}_1) - x^s_1, \\
		y_1 = \widehat{G}_{{m}_1} ({^{s}\varphi^{m_0}_{m_1}}(x), {^{cu}\varphi^{f(m_0)}_{f(m_1)}}(y) ), ~
		x_2 = \widehat{F}_{{m}_1} ({^{s}\varphi^{m_0}_{m_1}}(x), {^{cu}\varphi^{f(m_0)}_{f(m_1)}}(y) ).
		\end{gathered}
		\]
		We need to show $ |\hat{x} - (^{s}\varphi^{f(m_0)}_{f(m_1)})^{-1} x_2| + |\hat{y} - (^{cu}\varphi^{m_0}_{m_1})^{-1} y_1 | \leq \widetilde{C} d(m_1, m_0)|z| $ for some constant $ \widetilde{C} $ independent of $ m_0, m_1 $. (In the following, the value of $ \widetilde{C} $ will vary from line to line.)
		As $ \widehat{F}_m(\cdot), \widehat{G}_m(\cdot) $ are $ C^{1,1} $, we get
		\[
		|\hat{y} - y'_1| + |\hat{x} - x'_2| \leq \widetilde{C}(|x^s_0|+|x^{cs}_1|)|z| \leq \widetilde{C}d(m_1,m_0)|z|.
		\]
		Set
		\begin{gather*}
		\exp^{-1}_{m_1} \exp_{m_0} (x + x^s_0 + y'_1 + x^{cu}_0) = \overline{x}^s_1 + \overline{x}^{cu}_1, \\
		\exp^{-1}_{f(m_1)} \exp_{f(m_0)} (x'_2 + x^s_1 + y + x^{cu}_1) = \overline{x}^s_2 + \overline{x}^{cu}_2,
		\end{gather*}
		where $ \overline{x}^{\kappa}_1 \in \widehat{X}^{\kappa}_{m_1} $, $ \overline{x}^{\kappa}_2 \in \widehat{X}^{\kappa}_{f(m_1)} $, $ \kappa = s, cu $. Then by \eqref{translation}, we have
		\begin{gather*}
		|\overline{x}^s_1 + \overline{x}^{cu}_1 - {^{s}\varphi^{m_0}_{m_1}}(x) - {^{cu}\varphi^{m_0}_{m_1}}(y'_1)|  \leq \widetilde{C} d(m_1, m_0) (|x|+|y'_1|)  \leq \widetilde{C} d(m_1, m_0) (|x|+|y|), \\
		\begin{split}
		|\overline{x}^s_2 + \overline{x}^{cu}_2 - {^{s}\varphi^{f(m_0)}_{f(m_1)}}(x'_2) - {^{cu}\varphi^{f(m_0)}_{f(m_1)}}(y)|  & \leq \widetilde{C} d(m_1, m_0) (|x'_2|+|y|) \\
		& \leq \widetilde{C} d(m_1, m_0) (|x|+|y|),
		\end{split}
		\end{gather*}
		yielding
		\begin{align*}
		|\overline{x}^{cu}_1 - y_1| & = |\widehat{G}_{{m}_1} (\overline{x}^s_1, \overline{x}^{cu}_2 ) - \widehat{G}_{{m}_1} ({^{s}\varphi^{m_0}_{m_1}}(x), {^{cu}\varphi^{f(m_0)}_{f(m_1)}}(y) )| \\
		& \leq \widetilde{C} (|\overline{x}^s_1 - {^{s}\varphi^{m_0}_{m_1}}(x)| + |\overline{x}^{cu}_2 - {^{cu}\varphi^{f(m_0)}_{f(m_1)}}(y)|) \leq \widetilde{C} d(m_1, m_0) (|x|+|y|),
		\end{align*}
		and similarly $ |\overline{x}^{s}_2 - x_2| \leq \widetilde{C} d(m_1, m_0) (|x|+|y|) $. Therefore,
		\begin{align*}
		|\hat{x} - (^{s}\varphi^{f(m_0)}_{f(m_1)})^{-1} x_2| & \leq |\hat{x} - x'_2| + |x'_2 - (^{s}\varphi^{f(m_0)}_{f(m_1)})^{-1} x_2| \\
		& \leq \widetilde{C} d(m_1, m_0) (|x|+|y|) + |(^{s}\varphi^{f(m_0)}_{f(m_1)})^{-1}||{^{s}\varphi^{f(m_0)}_{f(m_1)}}(x'_2) - x_2| \\
		& \leq \widetilde{C} d(m_1, m_0) (|x|+|y|) + \widetilde{C} (|{^{s}\varphi^{f(m_0)}_{f(m_1)}}(x'_2) - \overline{x}^s_2| + |\overline{x}^s_2 - x_2|) \\
		& \leq \widetilde{C} d(m_1, m_0) (|x|+|y|),
		\end{align*}
		and analogously $ |\hat{y} - (^{cu}\varphi^{m_0}_{m_1})^{-1} y_1 | \leq \widetilde{C} d(m_1, m_0)|z| $. This completes the proof.
	\end{proof}

	Now, all the assumptions in \autoref{smoothbase} are satisfied for $ \widehat{H} $, $ \widehat{X}^{s} \oplus \widehat{X}^{cu} $, $ f $, $ \widehat{M} $ and the natural $ 0 $-section of $ \widehat{X}^{s} \oplus \widehat{X}^{cu} $, which yields the unique invariant graph $ \widehat{W}^{s} $ of $ \widehat{X}^s \to \widehat{X}^{cu} $ such that $ \widehat{H} \widehat{W}^{s} \subset \widehat{W}^{s} $ is $ C^1 $. Let $ \widehat{W}^{s}_m(\sigma) = \widehat{W}^{s} \cap (\widehat{X}^{s}_m(\sigma) \oplus \widehat{X}^{cu}_m) $. From $ \sup_m\|T^{s}_m f\| < 1 $, $ \widehat{H}_m|_{\widehat{W}^{s}_m(\sigma)} = \widehat{f}_m|_{\widehat{W}^{s}_m(\sigma)} $ and the characterization of the strong stable foliation $ W^{s} $ of $ f $, if $ \sigma $ is small, then $ \exp^{-1}_m \widehat{W}^{s}_m(\sigma) $ is open in $ W^s_m $, the leaf of $ W^{s} $ through $ m $. The $ C^1 $ dependence of $ W^{cs}_{m_0} \ni m \mapsto \exp^{-1}_m \widehat{W}^{s}_m(\sigma) $ shows that $ W^{s} $ is a $ C^1 $ foliation inside each leaf of $ W^{cs} $.

	For the last conclusion, by noting that under $ \| T^c_m f \|^2 \cdot \| (T^u_m f)^{-1} \| < 1, ~ m \in M $, each leaf of $ W^{cs} $ is $ C^2 $ and consequently $ \widehat{F}, \widehat{G} \in C^{1,1} $ (i.e. $ \widetilde{F}_m(\cdot,\cdot,\cdot), \widetilde{G}_m(\cdot,\cdot,\cdot) \in C^{1,1} $ uniformly for $ m $), the conclusion follows from \autoref{lem:final} (1).
\end{proof}

\begin{rmk}
	\begin{enumerate}[(a)]
		\item If $ f \in C^{1,\gamma} $ ($ \gamma > 0 $) and the center bunching condition is replaced by the following
		\[
		\| T^{c}_mf \| \cdot \|T^{s}_m f\|^{\gamma} \cdot \| (T^{c}_m f)^{-1} \| < 1,~m \in M,
		\]
		then one also gets $ W^{s} $ is a $ C^1 $ foliation inside each leaf of $ W^{cs} $; see \autoref{weakFG}.

		\item Assume $ TM $ has a $ Df $-invariant continuous splitting $ TM = X^s \oplus X^c $ with
		\[
		\|T^s_mf\|< 1, ~\|T^s_mf\| \cdot \|(T^c_mf)^{-1}\| < 1, ~m \in M.
		\]
		Let $ X^{c_1} $ be a $ Df $-invariant continuous subbundle of $ TM $ such that $ X^{c_1} \subset X^{c} $. Assume there is a $ C^0 $ foliation $ W^{c_1} $ tangent to $ X^s \oplus X^{c_1} $ which is invariant under $ f $, i.e. $ f(W^{c_1}_m) \subset W^{c_1}_{f(m)} $, where $ W^{c_1}_m $ is the $ C^1 $ leaf of $ W^{c_1} $ through $ m $. Suppose the following bunching condition holds:
		\[
		\| T^{c_1}_mf \| \cdot \|T^{s}_m f\| \cdot \| (T^{c_1}_m f)^{-1} \| < 1,~m \in M.
		\]
		Then each leaf of $ W^{c_1} $ is $ C^1 $-foliated by the strong stable foliation $ W^s $ if $ f \in C^{1,1} $. In the above, we use the notation $ T^\kappa f = Df|_{X^\kappa} $, $ \kappa = s, c, c_1 $. The proof is the same as for \autoref{holfoliation} (by using $ W^{c_1} $ instead of $ W^{cs} $).

		\item This argument given in \autoref{holfoliation} was also used in the proof that the center-stable manifold of a normally hyperbolic invariant manifold is $ C^1 $ foliated by the strong stable foliation in the Banach space setting; see \cite{Che18b}.
	\end{enumerate}
\end{rmk}

\begin{thm}[Strong stable laminations I: Lipschitz case]\label{thm:ssl}
	Let $ M $ be a metric space, $ u: M \to M $ a continuous map, and $ \Lambda \subset M $ a subset such that $ u(\Lambda) \subset \Lambda $. Assume that there is an $ \epsilon > 0 $ such that, for every $ m \in \Lambda $, there is an open ball $ U_m(\epsilon') \subset M $, where $ U_{m}(\epsilon) = \{ m' \in M : d(m',m) < \epsilon \} $, such that $ U_m(\epsilon') \subset X_m \times Y_m $, where $ X_m, Y_m $ are complete metric spaces. (This means that there is $ \phi_m:  U_m(\epsilon') \to \phi_m(U_m(\epsilon')) \subset X_m \times Y_m $ which is bi-Lipschitz with (bi-)Lipschitz constant no more than a fixed constant. The metric in $ \phi_m(U_m(\epsilon')) $ is now induced from $ U_m(\epsilon') $ through $ \phi_m $.)

	For each $ m \in \Lambda $, there is a correspondence $ \tilde{u}_m: X_m \times Y_m \to X_{u(m)} \times Y_{u(m)} $ satisfying the (A$ ' $)($ \alpha $, $ \lambda_u(m) $) (B)($ \beta $, $ \lambda_s(m) $) condition, such that $ u|_{U_m(\epsilon')} = \tilde{u}_m|_{U_m(\epsilon')} $, where $ \alpha, \beta > 0 $ are constants (for simplicity) and $ \lambda_s, \lambda_u $ are bounded functions of $ \Lambda \to \mathbb{R}_+ $.
	Assume $ \alpha \beta < 1 $ and $ \sup_m(\lambda_s(m) + \epsilon)\lambda_u(m) < 1 - \alpha \beta $, where $ \epsilon > 0 $.

	Then there are unique $ \Lambda^s_{loc}(m, \epsilon') \subset U_m(\epsilon') $, $ m \in \Lambda $, such that \textnormal{(1) (2)} below hold.
	\begin{enumerate}[(1)]
		\item $ m \in \Lambda^s_{loc}(m, \epsilon') $ and every $ \Lambda^s_{loc}(m, \epsilon') $ is a Lipschitz graph of $ X_m \to Y_m $ with Lipschitz constant no more than $ \beta $.
		\item $ u(\Lambda^s_{loc}(m, \epsilon')) \cap U_{u(m)}(\epsilon') \subset \Lambda^s_{loc}(u(m), \epsilon') $.

		Moreover, $ \Lambda^s_{loc}(m, \epsilon') $, $ m \in \Lambda $, have the following properties:
		\item $ \{ x \in U_m(\epsilon'): u^n(x) \in U_{u^n(m)}(\epsilon'), | u^n(x) - u^n(m) | \lesssim \varepsilon^{(n)}(m), n \to \infty \} \subset \Lambda^s_{loc}(m,\epsilon') $, where $ \varepsilon: \Lambda \to \mathbb{R}_+ $ is such that $ \lambda_s(m) + \epsilon < \varepsilon(m) < \lambda^{-1}_u(m) (1 - \alpha \beta) $; if $ u(\Lambda^s_{loc}(m, \epsilon')) \subset \Lambda^s_{loc}(u(m), \epsilon') $ for all $ m \in \Lambda $, then the two sets are equal.
		\item  Suppose that (i) $ \sup_m \lambda_s(m) + \epsilon < 1 $, and  (ii) if $ m_i \in \Lambda $, $ i = 1, 2 $ and $ d(m_1 , m_2) < \delta $, then $ |\lambda_s(m_1) - \lambda_s(m_2)| < \epsilon $. Let
		\[
		\Lambda^s_m \triangleq \bigcup_{n \geq 0} u^{-n} (\Lambda^s_{loc}(u^n(m), \epsilon')) = \{ x \in M: | u^n(x) - u^n(m) | \lesssim \varepsilon^{(n)}_1(m), n \to \infty \}.
		\]
		Then $ \Lambda^s \triangleq \{ \Lambda^s_m: m \in \Lambda \} $ is a ($ u $-invariant) lamination called the (global) \textbf{strong stable lamination} of $ \Lambda $ for $ u $, where $ \lambda_s(m) + \epsilon < \varepsilon_1(m) < \min\{1, \lambda^{-1}_u(m) (1 - \alpha \beta) \} - \varsigma $, and $ \varsigma > 0 $ is sufficiently small.
	\end{enumerate}
\end{thm}
Write $ \Lambda^s_{loc}(m_1, \epsilon_1) = \Lambda^s_{loc}(m_1, \epsilon') \cap U_m(\epsilon_1) $ where $ \epsilon_1 \leq \min \{ \delta, \epsilon' \} / 4 $. In the setting of \autoref{thm:ssl} (4), $ \Lambda^s_{loc}(\epsilon_1) \triangleq \{ \Lambda^s_{loc}(m,\epsilon_1): m \in \Lambda \} $ is called the \emph{local strong stable lamination} of $ \Lambda $ for $ u $; $ \Lambda^s_m $ (resp. $ \Lambda^s_{loc}(m,\epsilon_1) $) is called the strong stable manifold (resp. local strong stable manifold) of $ m $ for $ u $.

\begin{proof}
	Let $ X \times Y = \{ (m,x,y): (x,y) \in X_m \times Y_m, m \in \Lambda \} $ be a bundle over $ \Lambda $. Let $ H(m,z) = (u(m), \tilde{u}_m(z)): X \times Y \to X \times Y $ be a bundle correspondence over $ u $. Now the existence of $ \Lambda^s_{loc}(m, \epsilon) $ and (1) (2) follow from \autoref{thmA} applied to $ H $, $ X \times Y $, $ u $, $ \Lambda $; and (3) follows from \autoref{properties}.

	To show (4), note that $ \sup_m \lambda_s(m) + \epsilon < 1 $, and so $ u(\Lambda^s_{loc}(m, \epsilon'')) \subset \Lambda^s_{loc}(u(m), \epsilon'') $ for all $ \epsilon'' \leq \epsilon' $.
	Let $ \epsilon_1 \leq \min \{ \delta, \epsilon' \} / 4 $ and write $ \Lambda^s_{loc}(m_1, \epsilon_1) = \Lambda^s_{loc}(m_1, \epsilon') \cap U_m(\epsilon_1) $.
	If $ x \in \Lambda^s_{loc}(m_1, \epsilon_1) \cap \Lambda^s_{loc}(m_2, \epsilon_1) \neq \emptyset $, then $ u^n(x) \in \Lambda^s_{loc}(u^n(m_1), \epsilon_1) \cap \Lambda^s_{loc}(u^n(m_2), \epsilon_1) \subset U_{u^n(m_2)} (\epsilon') $. So by (3), one gets $ m_1 \in \Lambda^s_{loc}(m_2, \epsilon') $. If $ z \in \Lambda^s_{loc}(m_1, \epsilon_1) $, then $ u^n(z) \in \Lambda^s_{loc}(u^n(m_1), \epsilon_1) \subset U_{u^n(m_2)} (\epsilon') $ and $ | u^n(z) - u^n(m_2) | \lesssim \varepsilon^{(n)}_1(m_2) $, $ n \to \infty $, which yields $ \Lambda^s_{loc}(m_1, \epsilon_1) \subset \Lambda^s_{m_2} $. We conclude that if $ \Lambda^s_{m_1} \cap \Lambda^s_{m_2} \neq \emptyset $, then $ \Lambda^s_{m_1} = \Lambda^s_{m_2} $, i.e., $ \Lambda^s $ is a lamination. 
\end{proof}

\begin{cor}[Holonomy map]\label{corr:hol}
	Let \autoref{thm:ssl} (4) hold. Assume the following uniform continuity of $ m \mapsto X_m, Y_m $ holds:
	\begin{enumerate}
		\item [\textnormal{(UC)}] Let $ (\phi^{m_1, m}_1, \phi^{m_1, m}_2) = \phi_m \circ \phi^{-1}_{m_1}: U_{m_1}(\epsilon) \subset X_{m_1} \times Y_{m_1} \to U_{m}(\epsilon') \subset X_{m} \times Y_{m} $, $ m_1, m \in \Lambda $, where $ \epsilon < \epsilon' / 2 $. Then for any $ \eta > 0 $, there is a $ \rho > 0 $ such that $ \lip\phi^{m_1, m}_i(\cdot) < 1 + \eta $, $ i = 1, 2 $, provided $ d(m_1, m) < \rho < \epsilon' /2 $.
	\end{enumerate}
	Let $ \tilde{\Lambda} \subset U_m(\epsilon_1) \cap \Lambda $. Consider $ \Lambda^s_{m,\epsilon_1} = \bigcup_{m' \in \tilde{\Lambda}} \Lambda^s_{loc}(m', \epsilon_1) \subset U_m(\epsilon') $, a stack of leaves over $ \tilde{\Lambda} $.
	For any Lipschitz graphs $ \varSigma_1, \varSigma_2 \subset U_m(\epsilon') $ of $ Y_m \to X_m $ with Lipschitz constants no more than $ \beta $, and small $ \epsilon_1 > 0 $, we have:
	\begin{enumerate}[(1)]
	 	\item For any $ z \in \Lambda^s_{m,\epsilon_1} \cap \varSigma_1 $, there is a unique $ m' = m'(z) \in \tilde{\Lambda} $ such that $ \Lambda^s_{loc}(m', \epsilon_1) \cap \varSigma_1 = \{z\} $.
	 	\item If $ p(z) $ denotes the unique point of $ \Lambda^s_{loc}(m'(z),\epsilon_1) \cap \varSigma_2 $, then $ z \to p(z) $ is continuous. The map $ p(\cdot) $ is called a \textbf{holonomy map} for the strong stable lamination $ \Lambda^s $.
	 	\item
	 	If $ u $ is Lipschitz on $ \Lambda_{\epsilon'/2} = \{m : d(m, \Lambda) < \epsilon'/2 \} $, then $ z \mapsto p(z) $ is also H\"older.
	\end{enumerate}
\end{cor}
\begin{proof}
	Choose $ \epsilon_1 $ small. Then by assumption (UC) we know that $ m' \mapsto  \Lambda^s_{loc}(m',\epsilon_1) $ is continuous pointwise (\autoref{lem:continuity_f}), and for any $ m' \in \tilde{\Lambda} $, $ \Lambda^s_{loc}(m',\epsilon_1) \cap \varSigma_i $ consists of one point $ z_i(m') $, $ i = 1, 2 $, and $ m' \mapsto z_i(m') $ is continuous (see e.g. \autoref{lem:fconx}).
	By the lamination property, $ m' \to z_i(m') $ is injective.
	But one cannot expect that in general $ z \mapsto z^{-1}_i(z) $ is also continuous (where it can be defined) except when $ \tilde{\Lambda} $ is \emph{compact}. In contrast, the map $ p(z) = z_2(z^{-1}_1(z)) $ is continuous. This can be argued as follows (see also \cite[proof of Corollary 4.4]{PSW97}). Note that $ u(\Lambda^s_{loc}(\epsilon_2)) \subset \Lambda^s_{loc}(\epsilon_2) $ and the assumptions on $ \Lambda $ are also satisfied by $ \Lambda^s_{loc}(\epsilon_2) $, where $ \epsilon_2 < \epsilon_1 $ is sufficiently small. So there also exists the local strong stable lamination $ \{ \tilde{\Lambda}^s_{loc}(m', \epsilon_2) : m' \in \Lambda^s_{loc}(\epsilon_2) \} $ of $ \Lambda^s_{loc}(\epsilon_2) $. By the characterization of the strong stable lamination, if $ \epsilon_2 $ is small, one has $ \tilde{\Lambda}^s_{loc}(m', \epsilon_2) \subset \Lambda^s_{loc}(m, \epsilon_1) $ if $ m' \in \Lambda^s_{loc}(m, \epsilon_2) $. Also, note that $ m' \mapsto \tilde{\Lambda}^s_{loc}(m', \epsilon_2) $ is continuous pointwise. Now
	\[
	\{p(z)\} = \Lambda^s_{loc}(m'(z),\epsilon_2) \cap \varSigma_2 = \tilde{\Lambda}^s_{loc}(z,\epsilon_2) \cap \varSigma_2
	\]
	if $ z \in \tilde{\Lambda}^s_{loc}(z,\epsilon_2) \cap \varSigma_1 \subset \Lambda^s_{m,\epsilon_2} \cap \varSigma_1 $. Taking $ \epsilon_2 $ (small) instead of $ \epsilon_1 $, we finish the proof of (1) (2).

	To show (3), by \autoref{lem:sheaf}, if $ u $ is Lipschitz on $ \Lambda_{\epsilon'/2} $, then also $ m' \mapsto \tilde{\Lambda}^s_{loc}(m',\epsilon_1) $ is H\"older if $ \epsilon_1 $ is small, which yields the H\"older continuity of $ m' \mapsto z_i(m') $ (see e.g. \autoref{lem:fconx}). Thus, we obtain the H\"olderness of $ p(\cdot) $, completing the proof.
\end{proof}

If $ \Lambda = M $, \autoref{thm:ssl} means $ u $ admits an $ s $-lamination $ \{ \Lambda^s_m: m \in M \} $ (see e.g. \cite{AV10}). In addition, if $ u $ is a homeomorphism and $ \tilde{u}_m: X_m \times Y_m \to X_{u(m)} \times Y_{u(m)} $ satisfies (A)($ \alpha $, $ \lambda_u(m) $) (B)($ \beta $, $ \lambda_s(m) $) condition with $ \sup_{m}\lambda_s(m) + \epsilon < 1 $ and $ \sup_{m}\lambda_u(m) + \epsilon < 1 $, then $ u $ admits an $ u $-lamination. \autoref{corr:hol} also shows that the homeomorphism $ u $ is \emph{(uniformly) hyperbolic} (see also \cite{Via08, AV10}).

Now, let us consider the smooth case.

\begin{enumerate}[({SS}1)]
	\item Let $ X $ be a $ C^1 $ Finsler manifold with Finsler metric $ d $ in its components, which is $ C^{0,1} $-uniform (i.e. assumption $ (\blacksquare) $ on \autopageref{C1MM} holds for $ X = M = M_1 $); see also \autoref{examples} for some examples. Let $ \chi_m: U_m(\epsilon') \to T_{m}X $ be the local chart given by ($ \blacksquare $), and $ \mathcal{M} $ the {canonical} bundle atlas of $ TX $ induced by $ \chi_{m_0} $, $ m_0 \in X $.
	Let $ u: X \rightarrow X $ be a $ C^1 $ map, $ \Lambda \subset X $ and $ u(\Lambda) \subset \Lambda $.
	\item Let $ T_mX = X^{s}_m \oplus X^u_m $ with projections $ \Pi^{\kappa}_m $, so that $ \Pi^s_m + \Pi^u_m= \id $, $ R(\Pi^\kappa_m) = X^\kappa_m $, $ \kappa = s, u $, $ m \in \Lambda $. Suppose $ \sup_{m}|\Pi^s_m| < \infty $ and $ m \mapsto X^{\kappa}_m : X \to \mathbb{G}(TX) $ is uniformly continuous around $ \Lambda $, $ \kappa = s, u $, where $ \mathbb{G}(TX) $ is the Grassmann manifold of $ TX $; see \autoref{upVector}.
	\item Assume that $ Du(m): X^{s}_m \oplus X^u_m \rightarrow X^{s}_{u(m)} \oplus X^u_{u(m)} $, $ m \in \Lambda $, satisfy
	\begin{enumerate}[(i)]
		\item $ \|\Pi^s_{u(m)}Du(m)|_{X^{s}_m}\| = \lambda_{s}(m) $,
		\item $ \Pi^u_{u(m)}Du(m): X^u_{m} \rightarrow X^u_{u(m)} $ is invertible, $ \| (\Pi^u_{u(m)}Du(m)|_{X^u_m})^{-1} \| = \lambda_u(m) $,
		\item $ \| \Pi^{\kappa_1}_{u(m)} Du(m)\Pi^{\kappa_2}_m \| \leq \eta $, $ \kappa_1 \neq \kappa_2 $, $ \kappa_1, \kappa_2 \in \{s,u\} $,
		\item $ \sup_m \lambda_{s}(m) \lambda_u(m) < 1 $, $ \sup_m \lambda_s(m) < 1 $, $ \sup_m \lambda_{u}(m) < \infty $,
		\item $ Du $ is uniformly continuous around $ \Lambda $ with respect to $ \mathcal{M} $; see \autoref{def:manifoldMap}.
	\end{enumerate}
\end{enumerate}
\begin{cor}[Strong stable laminations II: smooth case]\label{corr:smoothcase}
	Let \textnormal{(SS1), (SS2), (SS3)} hold. If $ \eta > 0 $ is small, then there are $ \epsilon' > 0 $ and $ \eta_1 > \eta $ small such that there are unique $ \Lambda^s_{loc}(m, \epsilon') \subset U_m(\epsilon') $, $ m \in \Lambda $, such that the following \textnormal{(1) (2)} hold:
	\begin{enumerate}[(1)]
		\item $ m \in \Lambda^s_{loc}(m, \epsilon') $ and every $ \Lambda^s_{loc}(m, \epsilon') $ is a Lipschitz graph of $ X^s_m(\epsilon') \to X^u_m $ (through $ \chi_m $) with the Lipschitz constant no more than $ \eta_1 $.
		\item $ u(\Lambda^s_{loc}(m, \epsilon')) \subset \Lambda^s_{loc}(u(m), \epsilon') $.

		Moreover, $ \Lambda^s_{loc}(m, \epsilon') $, $ m \in \Lambda $, have the following properties:
		\item $ \{ x \in U_m(\epsilon'): u^n(x) \in U_{u^n(m)}(\epsilon'), | u^n(x) - u^n(m) | \lesssim \varepsilon^{(n)}(m), n \to \infty \} = \Lambda^s_{loc}(m,\epsilon') $, where $ \varepsilon: \Lambda \to \mathbb{R}_+ $ such that $ \lambda_s(m) + \varsigma < \varepsilon(m) < \lambda^{-1}_u(m) - \varsigma $ (for small $ \varsigma > 0 $).
		\item Each $ \Lambda^s_{loc}(m, \epsilon') $ is a $ C^1 $ graph. If $ \eta = 0 $, then $ T_m \Lambda^s_{loc}(m, \epsilon') = X^s_m $. Moreover, if $ u \in C^{k,\gamma} $, then $ \Lambda^s_{loc}(m, \epsilon') $ is a $ C^{k,\gamma} $ graph. Furthermore, $ m \mapsto \Lambda^s_{loc}(m,\epsilon') $ is continuous in $ C^1 $ topology on  bounded subset.
		\item Let
		\[
		\Lambda^s_m \triangleq \bigcup_{n \geq 0} u^{-n} (\Lambda^s_{loc}(u^n(m), \epsilon')) = \{ x \in X: | u^n(x) - u^n(m) | \lesssim \varepsilon^{(n)}_1(m), n \to \infty \}.
		\]
		Then $ \Lambda^s \triangleq \{ \Lambda^s_m: m \in \Lambda \} $ is a ($ u $-invariant) lamination called the \textbf{strong stable lamination} of $ \Lambda $ for $ u $, where $ \lambda_s(m) + \varsigma < \varepsilon_1(m) < \min\{1, \lambda^{-1}_u(m) \} - \varsigma $ (for small $ \varsigma > 0 $).
		\item The holonomy maps for $ \Lambda^s $ are uniformly (locally) H\"older if $ \sup_{m \in \Lambda} |Du(m)| < \infty $.
	\end{enumerate}
\end{cor}
\begin{proof}
	One can reduce the smooth case to the Lipschitz case.
	Let $ m \in \Lambda $. Consider
	\[
	\hat{u}_{m} = \chi_{u(m)} \circ u_m \circ \chi^{-1}_{m}: X^s_{m}(r_1) \oplus X^u_{m}(r'_1) \to X^s_{u(m)}(r_1) \oplus X^u_{u(m)}(r'_2),
	\]
	where $ r_1, r'_2 $ are chosen such that the maps are well defined. By taking $ r_1, r'_2 $ further smaller, we can assume that $ \hat{u}_{m} $ satisfies the (A$ ' $)($ \alpha $, $ \lambda'_u(m) $) (B$ ' $)($ \beta $, $ \lambda'_s(m) $) condition, where $ \lambda'_s(m) = \lambda_s(m) + \epsilon $, $ \lambda'_u(m) = \lambda_u(m) + \epsilon $, $ \epsilon > 0 $ (sufficiently small), and $ \alpha, \beta \to 0 $ as $ r_1, r'_1, r'_2, \eta \to 0 $.
	Using the radial retractions in $ X^s_m, X^u_m $, one gets
	\[
	\tilde{u}_{m} \sim (F_m, G_m) :  X^s_{m} \oplus X^u_{m} \to X^s_{u(m)} \oplus X^u_{u(m)}
	\]
	such that $ \tilde{u}_m|_{U_m(\epsilon')} = \hat{u}_m $ and it satisfies the (A$ ' $) ($ \alpha $, $ \lambda'_u(m) $) (B$ ' $) ($ \beta $, $ \lambda'_s(m) $) (see e.g. \autoref{lem:general} and \autoref{lem:exampleH}). Now (1), (2), (3) and (5) are consequences of \autoref{thm:ssl}. Note that $ \sup_{m \in \Lambda} |Du(m)| < \infty $ implies that $ |Du(m)| $ is bounded in a neighborhood of $ \Lambda $ for $ Du(m) $ is uniformly continuous around $ \Lambda $. So (6) follows from \autoref{corr:hol}. The regularity results in (4) are consequences of \autoref{lem:leaf1}, \autoref{lem:leaf1a}, \autoref{lem:continuity_f} and \autoref{lem:baseK1} (by taking $ \epsilon' $ smaller); see also \autoref{thm:LocalRegularity}.
\end{proof}

If $ \Lambda = X $, we also call $ \Lambda^s $, obtained in \autoref{corr:smoothcase}, the \emph{strong stable foliation} for $ u $.

\subsection{Fake invariant foliations}\label{fake}

In the following, we give a more general result on the existence and regularity of fake invariant foliations than that of \cite{BW10} (see also \cite{Wil13}). Also, with a little effort, we can deal with smooth maps in Finsler manifolds.

Let $ X $ be a Banach space, a compact Riemannian manifold without boundary, a Riemannian manifold having bounded geometry, or a Finsler manifold satisfying the following general setting (see e.g. \autoref{examples}).

\begin{enumerate}
	\item [(HF1)] Let $ X $ be a $ C^1 $ Finsler manifold with Finsler metric $ d $ in its components, modeled on a Banach space $ \mathbb{E} $, which is $ C^{0,1} $-uniform (i.e. assumption $ (\blacksquare) $ on \autopageref{C1MM} holds for $ X = M = M_1 $). Let $ \chi_m: U_m(\epsilon') \to T_{m}X $ be the local chart given by ($ \blacksquare $), and $ \mathcal{M} $ is the {canonical} bundle atlas of $ TX $ induced by $ \chi_{m_0} $, $ m_0 \in X $.

    Let $ u: X \rightarrow X $ be a $ C^1 $ map with $ \lip u < \infty $.
	\item [(HF2)] Let $ T_mX = X^{s}_m \oplus X^c_m \oplus X^u_m $ with projections $ \Pi^{\kappa}_m $, $ \kappa = c, s, u $, such that $ \Pi^c_m + \Pi^s_m + \Pi^u_m= \id $, $ R(\Pi^\kappa_m) = X^\kappa_m $, $ m \in X $. Set $ \Pi^{\kappa_1 \kappa_2}_m = \Pi^{\kappa_1}_m + \Pi^{\kappa_2}_m $ and $ X^{\kappa_1\kappa_2}_m = X^{\kappa_1}_m + X^{\kappa_2}_m $, $ \kappa_1 \neq \kappa_2 $, $ \kappa_i = c, s, u $, $ i = 1,2 $. Suppose $ \sup_{m}|\Pi^{\kappa}_m| < \infty $, and $ m \mapsto X^{\kappa}_m : X \to \mathbb{G}(TX) $ is uniformly continuous, $ \kappa = c, s, u $, where $ \mathbb{G}(TX) $ is the Grassmann manifold of $ TX $; see \autoref{upVector}.
    \item  [(HF3)] Assume that $ Du(m): X^{cs}_m \oplus X^u_m \rightarrow X^{cs}_{u(m)} \oplus X^u_{u(m)} $, $ m \in X $, satisfy
    \begin{enumerate}[(i)]
    	\item $ \|\Pi^s_{u(m)}Du(m)|_{X^{s}_m}\| = \lambda'_{s}(m) $,
    	$ \|\Pi^{cs}_{u(m)}Du(m)|_{X^{cs}_m}\| = \lambda'_{cs}(m) $,
    	$ \Pi^u_{u(m)}Du(m): X^u_{m} \rightarrow X^u_{u(m)} $ is invertible, $ \| (\Pi^u_{u(m)}Du(m)|_{X^u_m})^{-1} \| = \lambda'_u(m) $,
    	\item $ \| \Pi^{\kappa_1}_{u(m)} Du(m)\Pi^{\kappa_2}_m \| \leq \eta $, $ \kappa_1 \neq \kappa_2 $, $ \kappa_1, \kappa_2 \in \{s,c,u\} $,
    	\item $ \sup_m \lambda'_{cs}(m) \lambda'_u(m) < 1 $, $ \sup_m \lambda'_u(m) < \infty $, $ \sup_m \lambda'_{cs}(m) < \infty $.
    	\item $ Du $ is uniformly continuous with respect to $ \mathcal{M} $; see \autoref{def:manifoldMap}.
    \end{enumerate}
\end{enumerate}

\begin{thm}[Fake invariant foliations I]\label{thm:fake1}
	Let \textnormal{(HF1), (HF2), (HF3)} hold. Then the following conclusions hold:

	\textbf{Existence}.
	For any sufficiently small $ \varepsilon > 0 $, if $ \eta $ is small enough, then there are $ r > r_0 > 0 $ such that the following results hold. There exist (in general not unique) Lipschitz graphs $ W^{cs}_m(q, r) $, which can be parameterized as Lipschitz graphs of $ X^{cs}_m(r) \to X^{u}_m $ through $ \chi_m $ (i.e. the sets $ \chi_m (W^{cs}_m(q, r)) $), in a neighborhood of $ m \in X $ with Lipschitz constants no more than a fixed small constant (independent of $ m,q $), such that the following hold:
	\begin{enumerate}[(1)]
		\item (local invariance) $ u(W^{cs}_m(q, r_0)) \subset W^{cs}_{u(m)} (u(q), r) $, $ q \in {U}_m(r_0) $; $ q \in W^{cs}_m(q,r_0) \subset U_{m}(\epsilon') $. Moreover, if $ \sup_m \lambda'_{cs}(m) < 1 $, then $ u(W^{cs}_m(q, r_0)) \subset W^{cs}_{u(m)} (u(q), r_0) $ for $ q \in {U}_m(r_0) $.
		\item (tangency) For each $ m \in X $, $ W^{cs}_m(m,r) $ is differentiable at $ m $ with $ T_m W^{cs}_m(m,r) $ a (linear) graph of $ X^{cs}_m \to X^{u}_m $.
		Moreover, if for all $ m \in X $, $ Du(m) X^{cs}_m \subset X^{cs}_{u(m)} $, then $ T_m W^{cs}_m(m, r) = X^{cs}_m $.
		\item (exponential growth bounds) Let $ p_j = u^j(p) $, $ q_j = u^{j}(q) $, $ m_j = u^j(m) $, $ j = 1, 2, \ldots, n $. If $ q_j \in {U}_{m_j}(r_0) $, and $ p_j \in W^{cs}_{m_j}(q_j, r_0) $, $ j = 1, 2, \ldots, n-1 $, then $ p_n \in W^{cs}_{m_n}(q_n, r) $ and $ d(p_n, q_n) \leq \lambda^{(n)}_{cs}(m)d(p,q) $, where $ \lambda_{cs}(m) = \lambda'_{cs}(m) + \varepsilon $.
		\item If $ \sup_m\lambda'_{cs}(m) < 1 $, then for $ m \in X $, $ W^{cs}_m(m,r) $ is the (local) strong stable manifold of $ m $ for $ u $.
	\end{enumerate}

	\textbf{Regularity}.
	\begin{enumerate}[(1)]
		\item (H\"older Regularity) $ \mathcal{W}^{cs}_m(r) \triangleq \bigsqcup_{q \in U_{m}(r_0)} W^{cs}_m(q, r) $, $ m \in X $, are (uniformly) H\"older `foliations'.

		\item ($ C^1 $ leaves) Assume that $ \mathbb{E} $ has a $ C^1 \cap C^{0,1} $ bump function. Then for all $ m \in X $, all leaves $ W^{cs}_m(q,r_0) $, $ q \in U_{m}(r_0) $, are $ C^1 $ graphs (and so $ \mathcal{W}^{cs}_m(r) $ is indeed a foliation in $ U_{m}(\epsilon') $). Moreover, the tangent distribution $ T_qW^{cs}_m(q,r) $ depends in a uniformly (locally) H\"older fashion on $ q \in U_{m}(r_0) $ if $ u $ is uniformly (locally) $ C^{1,\gamma} $ ($ \gamma > 0 $) and in addition the bump function is $ C^{1,1} $; the H\"older constants are no more than a fixed constant independent of $ m $.

		If \textnormal{(i)} $ \sup_m\lambda'_{cs}(m) < 1 $, or
		\textnormal{(ii)} every $ X^{c}_m $ admits a $ C^1 \cap C^{0,1} $ bump function and $ \sup_m \lambda'_s(m) < 1 $, then for every $ m \in X $, the leaf $ W^{cs}_m(m,r_0) $ is a $ C^1 $ graph.
	\end{enumerate}
\end{thm}

\begin{enumerate}
	\item [(HF3$ ' $)] 
	(i) $ \Pi^{c}_{u(m)}Du(m): X^c_{m} \rightarrow X^c_{u(m)} $ is invertible, $ \| (\Pi^{c}_{u(m)}Du(m)|_{X^c_m})^{-1} \| = \lambda'_{1,c}(m) $,
	
	(ii) $ \sup_m \lambda'_{s}(m) \lambda'_{1,c}(m) < 1 $, $ \sup_m \lambda'_{1,c}(m) < \infty $.
\end{enumerate}

\begin{thm}[Fake invariant foliations II]\label{thm:fake2}
	Let \textnormal{(HF1), (HF2), (HF3), (HF3$ ' $)} hold. Let $ W^{cs}_m(q, r) $, $ W^{s}_m(q, r) $, $ q \in U_{m_0}(r_0) $, $ m \in X $, be obtained in \autoref{thm:fake1} such that they are parameterized as Lipschitz graphs of $ X^{cs}_m(r) \to X^{u}_m $ and $ X^{s}_m(r) \to X^{cu}_m $ through $ \chi_m $, respectively. Let $ \mathcal{W}^{\kappa}_m(r) \triangleq \bigsqcup_{q \in U_{m}(r_0)} W^{cs}_m(q, r) $, $ m \in X $, $ \kappa = cs, s $.
	Then we get the following conclusions:
	\begin{enumerate}[(1)]
		\item (coherence) Every leaf $ W^{cs}_m(q, r) $ is `foliated' by the leaves of $ \mathcal{W}^s_m(r) $; in particular $ W^s_m(p, r) \subset W^{cs}_m(q, r) $, if $ p \in W^{cs}_m(q, r_0) $. So $ \mathcal{W}^{s}_m(r) $ `foliates' $ \mathcal{W}^{cs}_m(r) $ as well.
		\item (regularity) Moreover suppose that
		$ \mathbb{E} $ has a $ C^1 \cap C^{0,1} $ bump function, $ u $ is $ C^{1,1} $, and $ \sup_m \lambda'_s(m) \lambda'_{1,c}(m) \lambda'_{cs}(m) < 1 $ (center bunching condition). Then each leaf of $ \mathcal{W}^{cs}_m(r) $ is $ C^{1} $ foliated by $ \mathcal{W}^{s}_m(r) $, $ m \in X $.
	\end{enumerate}
\end{thm}

\begin{proof}[Proof of \autoref{thm:fake1} and \autoref{thm:fake2}]
We write $ a \ll b $ if $ a $ is sufficiently smaller than $ b $, where $ a, b > 0 $.
\autoref{thm:fake1} is the local version of \autoref{thm:bundlemaps}. We also give a direct proof.

\textbf{(I)} Let $ \varepsilon > 0 $ be sufficiently small. As the norm in $ X_1 \oplus X_2 $ we take the maximal norm, i.e., $ |(x_1, x_2)| \triangleq |x_1 + x_2| = \max\{ |x_1|, |x_2| \} $, where $ x_i \in X_i $, $ i = 1,2 $.
Let
\[
{H}_m(x) = \chi_{u(m)} \circ u \circ \chi^{-1}_m(x) : X^{cs}_m(r_1) \oplus X^u_m(r'_1) \to X^{cs}_{u(m)}(r_2) \oplus X^u_{u(m)}(r'_2),
\]
which satisfies the (A)$(\alpha, \lambda_u(m))$ (B)$(\beta, \lambda_{cs}(m))$ condition by (HF3) (see \autoref{lem:a3} and \autoref{lem:general}),
where $ \lambda_{cs}(m) = \lambda'_{cs}(m) + \varepsilon $, $ \lambda_u(m) = \lambda'_u(m) + \varepsilon $, and $ \alpha , \beta \rightarrow 0 $ as $ r_1, r'_2, \eta \rightarrow 0 $. Here $ r_1, r'_1 $ are independent of $ m $ due to (HF2) (ii) and (HF3) (iv) and such that $ X^{cs}_m(r_1) \oplus X^u_m(r'_1) \subset \chi_m (U_m(\epsilon')) $ and $ u \circ \chi^{-1}_m ( X^{cs}_m(r_1) \oplus X^u_m(r'_1)) \subset \chi_{u(m)} (U_{u(m)}(\epsilon') ) $ for all $ m \in X $.

Take $ 0 < \epsilon_0 \ll r' \ll r \ll \min\{r_i , r'_i: i = 1, 2\} / 2 $.
Using the radial retractions in $ X^{\kappa}_m $ (see \eqref{radial}), one gets
\[
\widetilde{H}_m: T_m X  \rightarrow T_{u(m)}X
\]
such that $ \widetilde{H}_m |_{T_mX(\epsilon_0)} = H_m $ and $ \widetilde{H}_m |_{T_mX} \subset T_mX(r') $; here $ T_mX(\epsilon_0) = X^{s}_m(\epsilon_0) \oplus X^{c}_m(\epsilon_0) \oplus X^{u}_m(\epsilon_0) $.

Let $ H_m \sim (F_m, G_m) $, where $ F_m: X^{cs}_m(r_1) \oplus X^{u}_{u(m)}(r'_2) \to X^{cs}_{u(m)}(r_2) $ and $ G_m: X^{cs}_m(r_1) \oplus X^{u}_{u(m)}(r'_2) \to X^{u}_{m}(r'_1) $. As before, using the radial retractions in $ X^{\kappa}_m $ (see \eqref{radial}), one defines $ \widehat{F}_m, \widehat{G}_m $ in all $ X^{cs}_m \oplus X^{u}_{u(m)} $ such that $ \widehat{F}_m|_{X^{cs}_m(r') \oplus X^{u}_{u(m)}(r')} = {F}_m $, $ \widehat{G}_m|_{X^{cs}_m(r') \oplus X^{u}_{u(m)}(r')} = {G}_m $ (see e.g. \eqref{equ:rad}). Set
\[
\widehat{H}_m \sim (\widehat{F}_m, \widehat{G}_m): X^{cs}_m \oplus X^u_m \to X^{cs}_{u(m)} \oplus X^u_{u(m)}.
\]

Define
\[
\widehat{M} = \{  (m, q): q \in \widehat{U}_{m}(r) , m \in X \},~
\widehat{M}_1 = \{  (m, q): q \in \widehat{U}_{m}(r') , m \in X \},
\]
and
\[
\widehat{U}_m(\epsilon) = \chi_m^{-1}(X^{s}_m(\epsilon) \oplus X^{c}_m(\epsilon) \oplus X^u_m(\epsilon)), ~0 < \epsilon \leq r,
\]
with metric $ d_m(p_1,p_2) = |\chi_m(p_1) - \chi_m(p_2)|_m \approx d(p_1, p_2) $, where $ d $ is the Finsler metric in each component of $ X $. Let
\[
H^\sim_m = \chi^{-1}_{u(m)} \circ \widetilde{H}_m \circ \chi_m,
\]
and
\[
\widehat{H}_{(m,q)} (x) = \widehat{H}_m (x + \chi_m(q)) - \widehat{H}_m (\chi_m(q)), ~ ~ q \in \widehat{U}_m(r).
\]
Note that $ H^\sim_m|_{\widehat{U}_m(\epsilon_0)} = u $ and $ H^\sim_m (\widehat{U}_m(r)) \subset \widehat{U}_{u(m)} (r') $.

Let $ \overline{X} $, $ \overline{Y} $ be bundles over $ \widehat{M} $ with fibers $ \overline{X}_{(m, q)} = X^{cs}_m $, $
\overline{Y}_{(m, q)} = X^u_m $.
Then $ \widehat{M} \cong X \times X $ is indeed a $ C^1 $ bundle (over $ X $). But $ \overline{X} $, $ \overline{Y} $ are not even $ C^1 $ by our construction. We can give another topology in $ \widehat{M} $ such that $ \overline{X} $, $ \overline{Y} $ are $ C^1 $. This can be done if we consider the fiber topology in $ \widehat{M} $ (see \autoref{immersed}); i.e., consider $ \widehat{M} $ as a bundle over $ X $ with fibers $ \widehat{U}_{m}(r) $, $ m \in X $. Then $ \widehat{M} $ (with $ \widehat{M}_1 $) trivially satisfies (H1c) (of \autoref{settingOverview}). Now $ \overline{X}|_{m \times \widehat{U}_{m}(r)} $, $ \overline{Y}|_{m \times \widehat{U}_{m}(r)} $ are trivial, so $ \overline{X} $, $ \overline{Y} $ are $ C^1 $ and satisfy (H2c) (of \autoref{settingOverview}).
Let
\[
g(m, q) = (u(m), H^\sim_m(q)) : \widehat{M} \to \widehat{M}_1 \subset \widehat{M},
\]
and
\[
\overline{H}(m, q, x) = (g(m, q), \widehat{H}_{(m,q)}(x)) : \overline{X} \times \overline{Y} \to \overline{X} \times \overline{Y}.
\]
Now
\[
\overline{H}_{(m,q)} \sim ( \overline{F}_{(m,q)}, \overline{G}_{(m,q)} ) : \overline{X}_{(m,q)} \times \overline{Y}_{(m,q)} \to \overline{X}_{g(m,q)} \times \overline{Y}_{g(m,q)}
\]
satisfies the (A)$(\alpha, \lambda_u(m))$ (B)$(\beta, \lambda_{cs}(m))$ condition for all $ q \in \widehat{U}_m(r) $.

Let $ i(m,q) = 0 $ be the $ 0 $-section of $ \overline{X} \times \overline{Y} $.
Then $ i $ is invariant under $ \overline{H} $.
By applying \autoref{thmA} to $ \overline{H} $, $ \overline{X} \times \overline{Y} $, $ g $, $ \widehat{M} $, $ i $, we have $ f^{cs}_{(m,q)} : \overline{X}_{(m,q)} \to \overline{Y}_{(m,q)} $ with $ \lip f^{cs}_{(m,q)} \leq \widetilde{\beta}' < 1 $, and $ \graph f^{cs}_{(m,q)} \subset \overline{H}^{-1}_{(m,q)}  \graph f^{cs}_{( u(m), H^\sim_m(q) )} $, or more especially,
\[
( h^{cs}_{(m,q)}(p), f^{cs}_{( u(m), H^\sim_m(q) )} (h^{cs}_{(m,q)}(p))  ) \in \widehat{H}_{(m,q)} ( p , f^{cs}_{(m,q)} (p) ) ,~
p \in {X}^{cs}_{m},
\]
where $ h^{cs}_{(m,q)}: {X}^{cs}_{m} \to {X}^u_{u(m)} $ with $ \lip h^{cs}_{(m,q)} \leq \widetilde{\lambda}'_{cs}(m) $, $ q \in \widehat{U}_m(r) $.
Set
\[
f^{cs\sim}_{(m,q)} (\cdot) = \chi^{-1}_{m} \circ ( \Pi^{cs}_m \chi_m(\cdot) + f^{cs}_{(m,q)} (\Pi^{cs}_m \chi_m(\cdot) ) +  \chi_m(q) ) ,~
W^{cs}_m(q, \varepsilon) = f^{cs\sim}_{(m,q)} ( \widehat{U}_m(\varepsilon) ).
\]
\begin{asparaitem}
	\item If we take $ r_0 \ll \epsilon_0 $, $ q \in \widehat{U}_m(r_0) $, then $ W^{cs}_m(q, r_0) \subset \widehat{U}_{m} (\epsilon_0) $. So $ u(W^{cs}_m(q, r_0)) \subset W^{cs}_{u(m)} (u(q), r) $.

	\item $ W^{cs}_m(m,r) $ is differentiable at $ m $, i.e. $ f^{cs}_{(m,m)} $ is differentiable at $ m $ (see \autoref{00diff}). If $ Du(m) X^{cs}_m \subset X^{cs}_{u(m)} $, then $ Df^{cs}_{(m,q)}(0) = 0 $ and so $ T_m W^{cs}_m(m, r) = X^{cs}_m $.

	\item By our construction, for every $ m \in X $, $ W^{cs}_m(q,r) $, $ q \in \widehat{U}_{m}(r) $, are Lipschitz graphs with Lipschitz constants approximately no more than $ \widetilde{\beta}' $. By \autoref{lem:sheaf}, 
	\[
	\mathcal{W}^{cs}_m(r) \triangleq \bigsqcup_{q \in U_{m}(r_0)} W^{cs}_m(q, r),~m \in X,
	\]
	are uniformly (locally) H\"older foliations; the H\"older constants are no more than a fixed constant independent of $ m $. The disjointness property of $ W^{cs}_m(q,r) $, $ q \in \widehat{U}_{m}(\epsilon_0) $, follows from the characterization of $ \graph f^{cs}_{(m,q)} $ (see e.g. \autoref{properties}). Also under $ \sup_m\lambda'_{cs}(m) < 1 $, the assertion that for every $ m \in X $, $ W^{cs}_m(m,r) $ is the (local) strong stable manifold of $ m $ for $ u $, follows from the characterization of the (local) strong stable foliation for $ u $; see also \autoref{corr:smoothcase}.

	\item The characterization of the exponential growth bounds (\autoref{thm:fake1} (3)) follows directly from the (A) (B) condition.

	\item Let us consider the regularity of the leaves. The radial retraction is not smooth, and neither is $ \overline{H}_{(m,q)}(\cdot) $. Thus, \autoref{lem:leaf1} cannot be applied. However, if $ \mathbb{E} $ has a $ C^1 \cap C^{0,1} $ bump function, then we can choose a suitable $ C^1 \cap C^{0,1} $ bump function instead of the radial retractions so that $ \overline{H}_{(m,q)}(\cdot) $ is $ C^1 $. Now \autoref{lem:leaf1} implies that all leaves of $ \mathcal{W}^{cs}_m(r) $ are $ C^1 $. Moreover, by \autoref{lem:base0}, the tangent distribution $ T_qW^{cs}_m(q,r) $ depends in a uniformly (locally) H\"older fashion on $ q \in U_{m}(r) $ if $ u $ is $ C^{1,\gamma} $ ($ \gamma > 0 $) and in addition the bump function is $ C^{1,1} $; the H\"older constants are no more than a fixed constant independent of $ m $.

	Set
	\[
	W^{cs\sim}_{(m,q)} (\varepsilon) = \{ (x^s, x^c, f^{cs}_{(m,q)} (x^s, x^c)): (x^s, x^c) \in X^s_{m}(\varepsilon) \oplus X^c_{m}(\varepsilon) \}.
	\]
	Suppose that every $ X^{c}_m $ has a $ C^1 \cap C^{0,1} $ bump function. Now we can choose a suitable $ C^1 \cap C^{0,1} $ bump function instead of the radial retraction in $ X^c_m $.
	Since $ \sup_m \lambda'_{s}(m) < 1 $, now $ \overline{H}_{(m,m)} W^{cs\sim}_{(m,m)}(\epsilon_0) \subset W^{cs\sim}_{(u(m),u(m))}(\epsilon_0) $. Since $ \overline{H}_{(m,m)} = \widehat{H}_{m} $ and $ \widehat{F}_m(\cdot), \widehat{G}_m(\cdot) $ are now $ C^1 $ at $ W^{cs\sim}_{(m,m)}(\epsilon_0) $, by \autoref{thm:LocalRegularity}, $ W^{cs\sim}_{(m,m)}(\epsilon_0) $ is $ C^1 $ and so too is $ W^{cs}_m(m,\epsilon_0) $. If $ \sup_m \lambda'_{cs}(m) < 1 $, we can reduce this case to $ X^c_m = \{0\} $.
\end{asparaitem}

This completes the proof of \autoref{thm:fake1}.

\textbf{(II)} If \textnormal{(HF3$ ' $)} holds, then for every $ m \in X $ and $ q \in \widehat{U}_{m}(r) $, we get $ W^s_m(q, r) $, which is parameterized as a Lipschitz graph of $ X^{s}_m(r) \to X^{cu}_m $ through $ \chi_m $.

By the characterization (for $ \overline{H} $), one has $ W^s_m(p, r) \subset W^{cs}_m(q, r) $, if $ p \in W^{cs}_m(q, r_0) $, i.e. \autoref{thm:fake2} (1) holds. Next we will show for every $ m \in X $, $ \mathcal{W}^s_m(r) $ is a $ C^1 $ foliation inside every leaf of $ \mathcal{W}^{cs}_m(r) $.
Let
\begin{gather*}
\widehat{M}^{cs} = \{ (m,q, p): p \in W^{cs}_m(q, r), q \in \widehat{U}_m(r), m\in X \},\\
\widehat{M}^{cs}_1 = \{ (m,q, p): p \in W^{cs}_m(q, r'), q \in \widehat{U}_m(r'), m\in X \},
\end{gather*}
and
\[
g^{cs}(m,q, p) = ( g(m,q), H^\sim_m(p) ): \widehat{M}^{cs} \to \widehat{M}^{cs}_1 \subset \widehat{M}^{cs}.
\]
Consider $ \widehat{X} $, $ \widehat{Y} $ as bundles over $ \widehat{M}^{cs} $ with fibers $ \widehat{X}^{s}_{(m, q, p)} = X^{s}_m $, $ \widehat{Y}^{cu}_{(m, q, p)} = X^{cu}_m $, respectively.
Set
\[
\overline{H}^{cs}(m, q, p, x) = (g^{cs}(m, q, p), \widehat{H}_{(m,p)}(x)) : \widehat{X} \times \widehat{Y} \to \widehat{X} \times \widehat{Y}.
\]
Note that the unique invariant graph of $ \overline{H}^{cs} $, represented in the direct $ \widehat{X} \to \widehat{Y} $, contains $ \cup_{m,p} \chi_m( W^{s}_m(p, r) ) $, and $ \lip H^\sim_m|_{W^{cs}_m(q, r)} \leq \widetilde{\lambda}'_{cs}(m) $.

If $ \mathbb{E} $ has a $ C^1 \cap C^{0,1} $ bump function, then we can assume $ \widetilde{H}_m(\cdot) $ and $ \widehat{F}_m(\cdot), \widehat{G}_m(\cdot) $ are $ C^1 $, and so too are $ H^\sim_m(\cdot) $, $ W^{cs}_m(p,r) $ for every $ p $, and $ (q,x) \mapsto \overline{F}_{(m,q)}(x), \overline{G}_{(m,q)}(x) $.
Thus we can give the $ C^1 $ leaf topology in $ \widehat{M}^{cs} $ (by considering it as a bundle over $ \widehat{M} $, see \autoref{immersed}), making it satisfy (H1b). (Note that $ \widehat{M}^{cs} $ is not a $ C^1 $ bundle if we consider it locally as $ X \times X \times W^{cs}_m(p,r) $.)
Since $ \widehat{X}|_{(m,q) \times W^{cs}_m(q,r)} $, $ \widehat{Y}|_{(m,q) \times W^{cs}_m(q,r)} $ are trivial, we see $ \widehat{X} $, $ \widehat{Y} $ are $ C^1 $ and satisfy (H2c).
If $ u $ is $ C^{1,1} $ (i.e. $ H_m(\cdot) \in C^{1,1} $ uniformly for $ m $) and the center bunching condition
\[
\sup_m \lambda'_s(m) \lambda'_{1,c}(m) \lambda'_{cs}(m) < 1
\]
holds, then all the assumptions in \autoref{smoothbase} are satisfied for $ \overline{H}^{cs} $, $ \widehat{X} \times \widehat{Y} $, $ g^{cs} $, $ \widehat{M}^{cs} $ and the natural $ 0 $-section. So by \autoref{smoothbase}, $ \mathcal{W}^s_m(r) $ is a $ C^1 $ foliation inside every leaf of $ \mathcal{W}^{cs}_m(r) $. Here note that in general $ (p,x) \mapsto \overline{H}^{cs}(m,q,p,x) = \widehat{H}_{(m,p)}(x) $ is not $ C^{1,1} $ (as $ W^{cs}_m(p,r) $ is only $ C^1 $) but satisfies \eqref{ccc0} (since $ H_m(\cdot) \in C^{1,1} $ uniformly for $ m $).
\end{proof}

Here note that the fake invariant foliations are in general not unique. Our proof of the existence of fake invariant foliations depends on the particular choice of the bump functions; in fact, one can further discuss how the `fake' invariant foliations depend on this choice. See \autoref{bump} for a review of bump functions (and blid maps); actually, we can use blid maps instead of bump functions; note that smooth blid maps exist in $ C[0,1] $.

\subsection{Holonomy over a lamination for a bundle map}\label{holonomyL}

Let $ M $ be a set. It is called a \textbf{lamination} with metric leaves (or for short a lamination) if for every $ m \in M $, there is a set $ \mathcal{F}_m \subset M $ with the following properties.
\begin{enumerate}[(a)]
	\item $ m \in \mathcal{F}_m $, and if $ \mathcal{F}_{m_1} \cap \mathcal{F}_{m_2} \neq \emptyset $, then $ \mathcal{F}_{m_1} = \mathcal{F}_{m_2} $; $ \mathcal{F}_m $ is called a \textbf{leaf} of $ M $.
	\item Each leaf $ \mathcal{F}_m $ is a complete metric space with metric $ d_m $. $ \mathcal{F}_m (r) \triangleq \{ m' \in \mathcal{F}_m: d_m(m', m) < r \} $ is called a \textbf{plaque} of $ M $ (at $ m $).
\end{enumerate}
$ M / \mathcal{F} $ denotes the leaf space of $ M $ which is the set of all leaves of $ M $. Using $ M / \mathcal{F} $, one can define a leaf topology in $ M $, i.e., consider $ M $ as a bundle over $ M / \mathcal{F} $ with fibers the leaves of $ M $ and the projection $ \pi $ sending $ m \in M $ to $ \mathcal{F}_m $. The $ C^1 $ regularity of laminations can be defined as follows. Let $ M $ be a subset of a $ C^1 $ Riemannian manifold. $ M $ is a $ C^1 $ lamination if the disjoint leaves are $ C^1 $-injectively (and connected) immersed submanifolds and $ \mathcal{F}_m, T_{m}\mathcal{F}_m $ are $ C^0 $ with respect $ m $ (see \cite{HPS77}).

Let $ f: M_1 \to M_2 $ be a map, where $ M_1, M_2 $ are two laminations. Then $ f $ is called a \textbf{lamination map} if it sends leaves of $ M_1 $ to leaves of $ M_2 $, i.e. $ f: (M_1, M_1/\mathcal{F}, \pi) \to (M_2, M_2/\mathcal{F}, \pi) $ is a bundle map. If $ f: M \to M $ is a lamination map then $ M $ is called an $ f $-invariant lamination.
\begin{exa}
	A typical lamination is the strong stable lamination of a $ C^1 $ map on a Finsler manifold which is partially hyperbolic on a compact set; for a more general setting, see \autoref{thm:ssl}.
\end{exa}

\begin{defi}\label{holonomy}
	Let $ M $ be a lamination and $ \mathcal{E} $ a bundle over $ M $ with metric fibers. Let $ f : M \to M $ be a lamination map and $ H: \mathcal{E} \to \mathcal{E} $ a bundle map over $ f $. A family $ \mathcal{H} $ of $ C^0 $ maps $ h_{x,y} : \mathcal{E}_x \to \mathcal{E}_y $ for all $ x, y $ in the same leaf of $ M $ is called a \textbf{holonomy over the lamination $ M $ for $ H $} if the following properties are satisfied:
	\begin{enumerate}[(hl1)]
		\item (transition) $ h_{x,x} = \id $, $ h_{y,z} \circ h_{x,y} = h_{x,z} $ for all $ x, y, z $ in the same leaf of $ M $. So $ h_{x,y} $ is a homeomorphism.
		\item (invariance) $ H_y \circ h_{x,y} = h_{f(x), f(y)} \circ H_x $.
		\item (regularity) $ (x,y, \xi) \mapsto h_{x,y} (\xi) $ is continuous where $ x,y \in \mathcal{F}_m(\epsilon) $ if $ \mathcal{E} $ is a $ C^0 $ topology bundle (see \autoref{bundleP} \eqref{topologyBundle}).
	\end{enumerate}
\end{defi}
See also \cite{Via08,AV10} for the definitions of $ s $-holonomy and $ u $-holonomy which are special cases of \autoref{holonomy} (by taking $ M $ as strong (un)stable laminations).
Let us consider the existence of the holonomy over a lamination for a bundle map in a more general setting so that one could apply it to infinite-dimensional dynamical systems. We need the following assumptions.
\begin{enumerate}[({HS}1)]
	\item Let $ M $ be a lamination with plaques $ \{ \mathcal{F}_m(\epsilon): m \in M, \epsilon > 0 \} $. Let $ \phi: M \to M $ be a lamination map. Assume for small $ \epsilon < \epsilon' $, $ \phi_m \triangleq \phi: \mathcal{F}_m(\epsilon) \to  \mathcal{F}_{\phi(m)}(\epsilon) $ with $ \lip \phi_m \leq \lambda_s(m) $, and  $ \sup_m\lambda_s(m) < \infty $.
	\item Let $ \mathcal{E} $ be a \emph{uniform $ C^{0,1} $-fiber} bundle (see \autoref{def:lipbundle}) over $ M $ with respect to the bundle atlas $ \mathcal{A} $, where
	\[
	\mathcal{A} = \{ (\mathcal{F}_m(\epsilon), \varphi^{m}) ~\text{is a bundle chart of $ \mathcal{E} $ at}~ m: \varphi^{m}_m = \id, m \in M, 0 < \epsilon < \epsilon' \}.
	\]
	\item Let $ H: \mathcal{E} \to \mathcal{E} $ be a bundle map over $ \phi $. Assume the following conditions hold:
	\begin{enumerate}[(i)]
		\item The fiber maps $ H_m(\cdot) $, $ m \in M $, have \textbf{left invertible} maps $ H^{-1}_m(\cdot) $, i.e., $ H^{-1}_m \circ H_m = \id $, and $ H_m(\mathcal{E}_m) $ is open.
		\item $ \sup_m \lip H^{-1}_m(\cdot) < \infty $. $ m \mapsto H^{- 1}_m $ depends in a uniformly (locally) $ C^{0,1} $ fashion on the base points (in each small plaque $ \mathcal{F}_m(\epsilon') $), i.e., 
		\[
		\sup_{m} \sup_x \lip \widehat{H}^{-1}_m(\cdot, x) < \infty,
		\]
		where $ \lip\widehat{H}^{-1}_m(\cdot, x) $ is the Lipschitz constant of the vertical part $ H^{-1} $ at $ m $ in local representations with respect to $ \mathcal{A} $ (see \autoref{vsII}).
		\item $ \sup_m \lip H^{-1}_m(\cdot) \lambda_s(m) < 1 $ (fiber bunching condition) and $ m' \mapsto \lambda_s(m') $, $ m' \mapsto \lip H^{-1}_{m'}(\cdot) $ are $ C^0 $ (or $ \zeta $-almost $ C^0 $ with $ \zeta > 0 $ small (see \autoref{ucontinuous})) in each $ \mathcal{F}_m(\epsilon') $.
	\end{enumerate}
\end{enumerate}
\begin{thm}\label{thm:hol}
	Let \textnormal{(HS1), (HS2), (HS3)} hold with sufficiently small $ \epsilon $ (in \textnormal{(HS1)}). Then there is a unique (local) partition $ \mathcal{W}_{\epsilon} = \{ \mathcal{W}_{\epsilon}(m,x): (m,x) \in \mathcal{E} \} $ of $ \mathcal{E} $ such that:
	\begin{enumerate}[(1)]
		\item $ (m,x) \in \mathcal{W}_{\epsilon}(m,x) $ and every $ \mathcal{W}_{\epsilon}(m,x) $ is a Lipschitz graph of $ \mathcal{F}_m(\epsilon) \to \mathcal{E}_m $ (through the bundle chart in $ \mathcal{A} $) with Lipschitz constant no more than a fixed constant (independent of $ m,x $).
		\item $ H\mathcal{W}_{\epsilon}(m,x) \subset \mathcal{W}_{\epsilon}(H(m,x)) $ for all $ (m,x) \in \mathcal{E} $.
		\item Let $ h_{m_1, m_2} (x) = y $. If $ (m_2,y) \in \mathcal{W}_{\epsilon}(m_1,x) $, then $ h_{m_1, m_2} : \mathcal{E}_{m_1} \to \mathcal{E}_{m_2} $ is a $ C^0 $ map; in fact, $ (m_1, m_2, x) \mapsto h_{m_1, m_2}(x) $ is $ C^0 $ when $ m_1, m_2 \in \mathcal{F}_m(\epsilon) $. Moreover,
		\begin{equation}\label{hollip}
		|h_{m_1, m_2}(x) - x| \leq C d_{m_1}(m_1, m_2),~ m_2 \in \mathcal{F}_{m_1}(\epsilon), ~x \in \mathcal{E}_{m_1},
		\end{equation}
		where $ C $ is a constant independent of $ m_1, x $.
		If $ \sup_m \lip H_m(\cdot) < \infty $, then $ h_{m_1, m_2} $ is also a H\"older map with the H\"older constant no more than a fixed constant (independent of $ m_1, m_2 $); when the fibers of $ \mathcal{E} $ are uniformly bounded, the H\"older exponent $ \theta > 0 $ can be chosen uniformly such that
		\[
		\sup_m (\lip H_m(\cdot))^\theta \lip H^{-1}_m(\cdot) \lambda_s(m) < 1.
		\]
		\item $ \mathcal{H}_{\epsilon} = \{ h_{m_1,m_2}:  m_1 \in \mathcal{F}_{m_2}(\epsilon), m_2 \in M \} $ is a local holonomy over $ M $ for $ H $, i.e., $ \mathcal{H}_{\epsilon} $ satisfies \textnormal{(hl1), (hl2), (hl3)} in \autoref{holonomy}.
		\item In addition, if $ \sup_m \lambda_s(m) < 1 $, define $ \mathcal{W}(m,x) = \bigcup_{n \geq 0} H^{-n} \mathcal{W}_{\epsilon}(H^{n}(m,x)) $ and $ \mathcal{W} = \{ \mathcal{W}(m,x): (m,x) \in \mathcal{E} \} $. Then $ \mathcal{W} $ is a partition of $ \mathcal{E} $ which is invariant under $ H $. By using $ \mathcal{W} $, the maps $ h_{m_1, m_2} : \mathcal{E}_{m_1} \to \mathcal{E}_{m_2} $ can be defined for all $ m_1, m_2 $ in the same leaf, and they give a holonomy over $ M $ for $ H $.
	\end{enumerate}
\end{thm}
\begin{proof}
	Let
	\[
	\widehat{\mathcal{E}}_{\overline{m}} = \{ (m, x): x \in \mathcal{E}_m, m \in \overline{\mathcal{F}_{\overline{m}}(\epsilon)} \},
	\]
	which is a bundle over $ \overline{\mathcal{F}_{\overline{m}}(\epsilon)} = \{ m' \in \mathcal{F}_m: d_m(m', m) \leq \epsilon \} $. By the assumption (HS2), using the bundle charts in $ \mathcal{A} $, one can assume $ \widehat{\mathcal{E}}_{\overline{m}} \cong \overline{\mathcal{F}_{\overline{m}}(\epsilon)} \times \mathcal{E}_{\overline{m}} $.
	Set
	\[
	\widehat{M} = \{ (\widehat{m},m): m \in \overline{\mathcal{F}_{\widehat{m}} (\epsilon)}, \widehat{m} \in M \}.
	\]
	Define the bundle $ \widetilde{\mathcal{E}} $ as
	\[
	\widetilde{\mathcal{E}} = \{ (\widehat{m}, m, x): x \in \mathcal{E}_m, (\widehat{m},m) \in \widehat{M} \},
	\]
	where the topology in $ \widetilde{\mathcal{E}} $ is the leaf topology; here the bundle $ \widetilde{\mathcal{E}} $ will be considered over $ \widehat{M} $ or $ M $. (We define the bundle $ \widetilde{\mathcal{E}} $ only in order to give another immersed topology in $ \mathcal{E} $.)
	The maps $ H $ and $ \phi $ have natural extensions to $ \widetilde{\mathcal{E}} $ and $ \widehat{M} $, respectively, i.e.
	\[
	\phi^{\sim}(\widehat{m}, m) = ( \phi(\widehat{m}), \phi_{\widehat{m}}(m) ): \widehat{M} \to \widehat{M},
	\]
	\[
	H^{\sim} (\widehat{m}, m , x) = (\phi^{\sim}(\widehat{m}, m), H_m(x)) = (\phi(\widehat{m}), H(m,x)): \widetilde{\mathcal{E}} \to \widetilde{\mathcal{E}}.
	\]

	Define a bundle $ \varSigma $ over $ \widetilde{\mathcal{E}} $ as
	\[
	\varSigma = \{ (\widehat{m}, \overline{m}, \overline{x}, m, x): (m,x) \in \widehat{\mathcal{E}}_{\overline{m}}, (\widehat{m}, \overline{m}, \overline{x}) \in \widetilde{\mathcal{E}} \},
	\]
	and a bundle map $ {H}: \varSigma \to \varSigma $ over $ H^{\sim} $ with fiber maps
	\[
	{H}_{(\widehat{m}, \overline{m}, \overline{x})} (m, x) = ( \phi_{\overline{m}} (m), H_m(x) ) = H(m,x), ~(m,x) \in \widehat{\mathcal{E}}_{\overline{m}}.
	\]
	Through the bundle charts in $ \mathcal{A} $, we assume
	\[
	{H}_{(\widehat{m}, \overline{m}, \overline{x})} (m, x) = ( \phi_{\overline{m}} (m), \widehat{H}_{\overline{m}}(m, x) ): \overline{\mathcal{F}_{\overline{m}}(\epsilon)} \times \mathcal{E}_{\overline{m}} \to \overline{\mathcal{F}_{\phi(\overline{m})}(\epsilon)} \times \mathcal{E}_{\phi(\overline{m})},
	\]
	where $ \widehat{H}_{\overline{m}}( \cdot, \cdot ) $ is the local representation of $ H $ at $ \overline{m} $ with respect to $ \mathcal{A} $.
	Let $ \widehat{H}_{\overline{m}}^{-1}(m, \cdot) $ be the local representation of $ H^{-1}_{\overline{m}}(\cdot) $ at $ \overline{m} $ with respect to $ \mathcal{A} $.

	Now by \autoref{lem:a3}, $ \widetilde{H}_{(\widehat{m}, \overline{m}, \overline{x})} \sim ( \phi_{\overline{m}} (\cdot), \widehat{H}_{\overline{m}}^{-1}(\cdot, \cdot) ) $ satisfies the (A)($ 0 $, $ \lambda_s(\overline{m}) $) (B)($ \beta $, $ \lambda_u(\overline{m}) $) condition, where $ \lambda_u(\overline{m}) \to \lip H^{-1}_{\overline{m}}(\cdot) $ as $ \epsilon \to 0 $, and $ \beta $ is a large number depending on $ \sup_{m} \sup_x \lip \widehat{H}^{-1}_m(\cdot, x) $, $ \sup_m\lambda_s(m) $ and $ \sup_m\lambda_u(m) $.
	Let $ i(\widehat{m}, \overline{m}, \overline{x}) = (\overline{m}, \overline{x}) $ be a section of $ \varSigma $ which is invariant under $ \widetilde{H} $. The existence of $ \mathcal{W}_{\epsilon} $ is now a consequence of \autoref{thmA} applied to  $ \varSigma $, $ \widetilde{\mathcal{E}} $, $ \widetilde{H} $, $ H^{\sim} $, $ i $.
	More precisely, if $ \epsilon $ is small, then there are $ f_{(\overline{m}, \overline{x})} = f_{(\widehat{m}, \overline{m}, \overline{x})}: \mathcal{F}_{\overline{m}}(\epsilon) \to \mathcal{E}_{\overline{m}} $, $ (\overline{m}, \overline{x}) \in \mathcal{E} $, with $ \lip f_{(\overline{m}, \overline{x})} \leq \beta $ such that $ f_{(\overline{m}, \overline{x})} (\overline{m}) = \overline{x} $ and
	\[
	f_{H(\overline{m}, \overline{x})}( \mathcal{F}_{\phi(\overline{m})} (\epsilon) ) \subset \widehat{H}_{\overline{m}}( \mathcal{F}_{\overline{m}}(\epsilon') \times \mathcal{E}_{\overline{m}} ),
	\]
	\[
	{H}_{(\widehat{m}, \overline{m}, \overline{x})} \graph f_{(\overline{m}, \overline{x})} \subset \graph f_{H(\overline{m}, \overline{x})};
	\]
	here we use the fact that $ H_m(\mathcal{E}_m) $ is open ((HS3) (i)).
	Let $ \mathcal{W}_{\epsilon}(\overline{m}, \overline{x}) = \graph f_{(\overline{m}, \overline{x})} $, which is a Lipschitz graph of $ \mathcal{F}_{\overline{m}}(\epsilon) \to \mathcal{E}_{\overline{m}} $. Thus, we obtain (1) (2). Note that $ h_{m_1, m_2} (x) = f_{(m_1,x)}(m_2) $. Since $ f_{(m_1,x)}(m_1) = x $ and $ (m_2, x') \in \mathcal{W}_{\epsilon}(m_1, x) $ if and only if $ (m_1, x) \in \mathcal{W}_{\epsilon}(m_2, x') $, we have (hl1). The invariance of $ \mathcal{W}_\epsilon $ gives (hl2). \eqref{hollip} follows from the Lipschitz continuity of $ f_{(\overline{m}, \overline{x})}(\cdot) $.
	By \autoref{lem:continuity_f}, $ (m_1, m_2, x) \mapsto h_{m_1, m_2}(x) $ is $ C^0 $ by regarding $ \widetilde{\mathcal{E}} $ as a bundle over $ {M} $; the continuity of $ m' \mapsto \lambda_s(m'), \lip H^{-1}_{m'}(\cdot) $ was used here.
	By \autoref{lem:sheaf}, we know $ h_{m_1, m_2} $ is H\"older uniformly for $ m_1, m_2 $, under the condition $ \sup_m \lip H_m(\cdot) < \infty $ (considering $ \widetilde{\mathcal{E}} $ as a bundle over $ \widehat{M} $).
	Thus, we obtain conclusions (3) (4). Finally, (5) follows from (1)--(4) and the fact that for every $ m_1 \in \mathcal{F}_{m} $ there is an integer $ N $ such that $ \phi^N(m_1) \in \mathcal{F}_{\phi^N(m)} (\epsilon) $. 
\end{proof}

\begin{rmk}
	\begin{enumerate}[(a)]
		\item As the lamination $ M $ one usually takes the strong stable lamination of a partial hyperbolic (compact) set for the Lipschitz (or smooth) map $ \phi $ (see also \cite[Section 4]{AV10} and \cite{Via08}); so in this case $ \sup_m\lambda_s(m) < 1 $. For applications, the bundle $ \mathcal{E} $ is usually the center-stable foliation for $ \phi $.

		\item If for each $ m \in M $, $ H_m $ is invertible, then the Lipschitz continuity of $ m \mapsto H_m $ and $ \sup_m\lip H^{-1}_m(\cdot) < \infty $ imply the Lipschitz continuity of $ m \mapsto H^{-1}_m $; using this fact, \autoref{thm:hol} recovers the existence of $ s $($ u $)-holonomies for cocycles given in \cite{AV10, ASV13}.

		\item If one only focuses on the existence of $ \mathcal{W} $, the Lipschitz continuity of $ m \mapsto H^{-1}_m $ can be replaced by $ m \mapsto H^{-1}_m $ depending in a uniformly (locally) H\"older fashion on the base points, i.e., $ \sup_{m} \sup_x \hol_{\nu} \widehat{H}^{-1}_m(\cdot, x) < \infty $. But in this case the `bunching condition' has to be $ \sup_m \lip H^{-1}_m(\cdot) \lambda^\nu_s(m) < 1 $.
		Indeed, if we redefine the metric in $ \mathcal{F}_m(\epsilon) $ as $ d'_m(x, y) = d_m(x,y)^\nu $, then in this new metric the `bunching condition' is $ \sup_m \lip H^{-1}_m(\cdot) \lambda^\nu_s(m) < 1 $. So one can apply \autoref{thm:hol}. See also \cite[Section 3]{ASV13}.

		\item Since $ h_{m_1,m_2} $ satisfies \eqref{hollip}, an easy computation shows that
		\[
		h_{m_1, m_2} (x) = \lim_{n \to \infty} (H^{-1})^{(n)}_{m_2} \circ H^{(n)}_{m_1}(x),
		\]
		where the convergence is uniform for $ x, m_1, m_2 $. Using this fact, one can prove more regularities of $ x \mapsto h_{m_1, m_2}(x) $ (or $ (m_1, m_2, x) \mapsto h_{m_1, m_2}(x) $). See also \cite[Section 3]{ASV13} and \cite[Section 5]{AV10}. Moreover, if the bundle map in (HS3) depends on an external parameter $ \lambda $, i.e. $ \lambda \mapsto H^\lambda $, then the partition of $ \mathcal{E} $ obtained in \autoref{thm:hol} also depends on $ \lambda $, i.e. $ \lambda \mapsto \mathcal{W}^{\lambda} $. One can consider the continuity of $ \lambda \mapsto \mathcal{W}^{\lambda} $ by using the results in \autoref{continuityf} (left to the readers).
	\end{enumerate}
\end{rmk}

\begin{cor}[\cite{Via08, ASV13}] \label{corr:hol2}
	Assume \textnormal{(HS1)} and the following \textnormal{(HS3$ ' $)} hold.
	\begin{enumerate}
		\item [$(\mathrm{HS}3 ' )$] Let $ \mathcal{E} = M \times X $, where $ X $ is a Banach space. Suppose $ m \mapsto A(m) \in GL(X, X) $ (the space of all invertible continuous linear maps in $ X $) is uniformly $ \nu $-H\"older in each small plaque of $ M $, i.e.,
		\[
		\|A(m_1) - A(m_2)\| \leq C_0 d_m(m_1, m_2)^\nu, ~m_1 , m_2 \in \mathcal{F}_m(\epsilon),
		\]
		where $ C_0 > 0 $ is a constant independent of $ m $ and $ 0 < \nu \leq 1 $, and $ \sup_m\|A^{\pm 1}(m)\| < \infty $. Suppose the following fiber bunching condition holds:
		\[
		\sup_m\|A(m)\| \cdot \| A^{-1}(m) \| \lambda^\nu_s(m) < 1.
		\]
	\end{enumerate}

	Let $ H(m,x) = (\phi(m), A(m)x) $ be a bundle map on $ \mathcal{E} $. Then there is a unique holonomy over $ M $ for $ H $ satisfying \textnormal{(hl1) (hl2)} in \autoref{holonomy}. In fact, $ h_{m_1, m_2} $ is a linear map. Moreover, if $ m' \mapsto \lambda_s(m') $ is continuous in each small $ \mathcal{F}_m(\epsilon) $, then the holonomy over $ M $ for $ H $ also satisfies \textnormal{(hl3)} in \autoref{holonomy}.
\end{cor}
\begin{proof}
	Without loss of generality, we assume $ \nu = 1 $.
	Note that $ \hol_{\nu} H(\cdot,x) \leq C_0|x| $, so in general $ \sup_x\hol_\nu H(\cdot,x) = \infty $, which is not the case in \autoref{thm:hol}. In order to deal with this case, let us consider the map
	\[
	H_1(m,x) = (\phi(m), \|A(m)\|^{-1} A(m)x ):  M \times \mathbb{B}_1 \to M \times \mathbb{B}_1,
	\]
	where $ \mathbb{B}_1 $ is the closed unit ball of $ X $. Let $ r_1(\cdot) $ be the radial retraction (see \eqref{radial}). Then $ H_1(m,\cdot) $ has a natural left invertible map $ H_1^{-1}(m,x) = r_1( \|A(m)\| A^{-1}(m)x ) $.

	We can assume $ \sup_m\|A(m)\| \cdot \| A^{-1}(m) \| \lambda_s(m) < 1/2 $ if we consider $ H^n_1 $ instead of $ H_1 $ for large $ n $. Note that $ \lip H^{-1}_1(m, \cdot) \leq 2\|A(m)\| \cdot \| A^{-1}(m) \|  $, $ \sup_{x \in \mathbb{B}_1}\hol_{\nu} H_1(\cdot,x) < \infty $; furthermore $ H_1(m, \mathbb{B}_1) $ is open (as $ A(m) $ is invertible), and
	\[
	\|A^{-1}_{m_1} - A^{-1}_{m_2}\| = \|A^{-1}_{m_1} A_{m_2} A^{-1}_{m_2} - A^{-1}_{m_1}A_{m_1}A^{-1}_{m_2}\| \leq \|A^{-1}_{m_1}\| \cdot\|A_{m_1} - A_{m_2}\|\cdot\|A^{-1}_{m_2}\|.
	\]
	So applying \autoref{thm:hol} to $ H_1, M \times \mathbb{B}_1 $, one obtains $ f_{(m,x)} : \mathcal{F}_m(\epsilon) \to \mathbb{B}_1  $ such that $ H_1 \graph f_{(m,x)} \subset \graph f_{H_1(m,x)}  $ for all $ (m,x) \in M \times \mathbb{B}_1 $.
	By the uniqueness property, $ f_{(m,ax)} = a f_{(m,x)} $ for $ |a| \leq 1 $, $ x \in \mathbb{B}_1 $. Using this one can define $ f_{(m,x)} $ for all $ x \in X $ and $ H \graph f_{(m,x)} \subset \graph f_{H(m,x)} $. The proof is complete.
\end{proof}

We refer the readers to \cite{Via08, AV10, ASV13} for more results and applications.

%% file: app.tex
\chapter{Appendix. A key argument in the proof of regularity}\label{Appbb}

Let us give a general argument in order to prove the regularity results in \autoref{stateRegularity}. Since we also deal with the `pointwise dependence' and the `unbounded' case (for example the bundles are not usually assumed to be uniformly bounded), some additional preliminaries are needed.
\begin{defi}\label{def:lypnum}
		\begin{enumerate}[(a)]
		\item Let $ \overline{M} $ be a set, $ u: \overline{M} \to \overline{M} $ and $ \theta: \overline{M} \to \mathbb{R}_+ $. If we write $ \theta $ over $ u $, then we use the notation $ \theta^{(k)}(m) = \theta(m) \theta(u(m)) \cdots \theta(u^{k-1}(m)) $.

		\item A function sequence $ \theta_n : \overline{M} \to \mathbb{R}_+ $, $ n = 1,2,\ldots, $ is \textbf{ln-subadditive} (relative to $ u $) if $ \theta_n(m) > 0 $ and $ \ln \theta_{n_1 + n_2} (m) \leq \ln\theta_{n_1}(m) + \ln\theta_{n_2}(u^{n_1}(m)) $ for all $ n, n_1, n_2 \in \mathbb{N} $ and $ m \in \overline{M} $.

		\item We write $ \{ \theta_n \} \in \mathcal{E}(u) $ if $ \{ \theta_n \} $ is ln-subadditive (relative to $ u $) and 
		\[
		\overline{\theta}(m) \triangleq \lim_{n \to \infty} \theta_n(m)^{1/n}
		\]
		exists for every $ m \in \overline{M} $. In this case, $ \overline{\theta}(m) $ (or $ \ln \overline{\theta}(m) $) is called an \textbf{(exact) Lyapunov number} of $ \{ \theta_n \} $ (at $ m $). Note that in general $ \overline{\theta}(m) \leq \overline{\theta}(u(m)) $; $ \overline{\theta}(m) = \overline{\theta}(u(m)) $ if $ u $ is invertible. The Kingman Subadditive Ergodic Theorem (\cite{Kin68}) ensures that some function sequences belong to $ \mathcal{E}(u) $.

		\item A function $ \theta: \overline{M} \to \mathbb{R}_+ $ (over $ u $) is \textbf{orbitally decreased} (resp. \textbf{orbitally increased}) (relative to $ u $) if $ \theta(u(m)) \leq \theta(m) $ (resp. $ \theta(u(m)) \geq \theta(m) $); and $ \theta $ is (positively) \textbf{orbitally bounded} (relative to $ u $) if $ \sup_{N \geq 0}\theta(u^N(m)) < \infty $ for every $ m \in \overline{M} $.

		\item We write $ \theta \in \mathcal{E}(u) $ if $ \{ \theta^{(n)} \} \in \mathcal{E}(u) $. Note that in this case, $ \overline{\theta}(u(m)) = \overline{\theta}(m) $. If $ u $ is a period function or $ \theta $ is orbitally decreased, then $ \theta \in \mathcal{E}(u) $.

		\item Given a ln-subadditive (relative to $ u $) and strictly positive function sequence $ \{ \theta_n \} $, assume $ \theta_1 $ is orbitally bounded relative to $ u $. Note that by ln-subadditivity, $ \theta_n $, $ n = 1,2,\ldots, $ are all orbitally bounded. Set $ \theta^*_n(m) = \sup_{N \geq 0} \theta_n(u^N(m)) $. Then $ \{\theta^*_n(m)\} $ is also ln-subadditive (relative to $ u $). Since $ \theta^*_{kn}(m)^{1/(kn)} \leq \theta^*_{n}(m)^{1/n} $, we know that $ \{\theta^*_n(m)\} \in \mathcal{E}(u) $. Let $ \theta^*(m) $, $ m \in \overline{M} $ be its (exact) Lyapunov numbers. We call them the \textbf{sup Lyapunov numbers} of $ \{ \theta_n \} $. Note that $ \theta^*(u(m)) = \theta^*(m) $ and $ (\theta^*)^n(m) \leq \theta^*_n(m) $. Also, $ \theta^*(m) = \inf_{n \geq 1} ( \limsup_{N  \to \infty} \theta_n (u^N(m)) )^{1/n} $. The above notations will be used for the function $ \theta $ if we consider the function sequence $ \{ \theta^{(n)} \} $.
	\end{enumerate}
\end{defi}

\begin{exa}
	\begin{enumerate}[(a)]
		\item Let $ A $ be a bounded linear operator from a Banach space into itself. Let $ \theta_n = \| A^n \| $. Then by the Gelfand Theorem, $ \overline{\theta} = r(A) $, the spectral radius of $ A $.

		\item Let $ \overline{M} $ be a probability space with probability measure $ P $, and $ u : \overline{M} \to \overline{M} $ an invertible function preserving $ P $. Let $ A: \overline{M} \to L( X, X ) $ be strongly measurable (over $ u $) where $ X $ is a Banach space. Set $ \theta^n(m) = \|A^{(n)}(m)\| $. Assume $ \ln^+ \|A(\cdot)\| \in L^1(P) $. Then $ \overline{\theta}(m) $ exists for $ m \in \overline{M}_1 \subset \overline{M} $, where $ P(\overline{M}_1) = 1 $ and $ u(\overline{M}_1) = \overline{M}_1 $. Moreover, if $ u $ is ergodic, then $ \overline{\theta} $ is constant. This is a direct consequence of the Kingman Subadditive Ergodic Theorem \cite{Kin68}. $ \overline{\theta}(m) $, $ m \in \overline{M}_1 $, are the largest Lyapunov numbers. See \cite[Section 8.1]{CL99} for more characterizations of $ \theta^* $.
		See \cite{LL10} for some new important development of the Multiplicative Ergodic Theorem in infinite dimensions.

		\item Let $ \overline{X} $ be a bundle with base space $ \overline{M} $ and fibers $ \overline{X}_m $, $ m \in \overline{M} $, being metric spaces. Let $ w: \overline{M} \to \overline{M} $ and $ f : \overline{X} \to \overline{X} $ over $ w $. Assume that $ \lip f^{(n)}_m(\cdot) = \mu_n(m) $ and $ \sup_{N \geq 0}\mu_1(u^N(m)) < \infty $. Note that $ \{ \mu_n \} $ is ln-subadditive. Consider the sup Lyapunov numbers of $ \{\mu^{(n)}_1 \} $ and $ \{ \mu_n \} $, respectively, denoted by $ \mu^*_1 $ and $ \mu^* $, respectively. Then $ \mu^*(m) \leq \mu^*_1(m) $. Note that $ \mu^*_1 $, $ \mu^* $ do not depend on the choice of the uniform equivalent metrics in fibers of $ \overline{X} $, i.e. any $ \tilde{d}_m $ in $ \overline{X}_m $ satisfies $ C^{-1} d_m(x,y) \leq \tilde{d}_m(x,y) \leq C d_m(x,y) $ for all $ (x,y) \in \overline{X}_m $, where $ C $ is a constant independent of $ m \in \overline{M} $ and $ d_m $ is the original metric in $ \overline{X}_m $. For this reason, one can say $ \mu^*(m) $ is the (non-linear) \emph{spectral radius of $ f $ along the orbit from $ m $}.
	\end{enumerate}
\end{exa}

\begin{lem}
	Given two strictly positive function sequences $ \{ \mu_n \}, \{ \theta_n \} $ (over $ u $) which are ln-subadditive and orbitally bounded relative to $ u $, i.e., 
	\[
	\sup_{N \geq 0} \mu_1(u^N(m)) < \infty, ~\sup_{N \geq 0} \theta_1(u^N(m)) < \infty.
	\]
	Let us consider some different types of `spectral gap condition':
	\begin{enumerate}[(i)]
		\item $ \sup_{m} \mu_1(m) \sup_{N \geq 0} \theta_1(u^N(m)) < 1 $. $ (\mathrm{i})_k $ $ \sup_{m} \mu_k(m) \sup_{N \geq 0} \theta_k(u^N(m)) < 1 $.
		\item $ \sup_{m} \mu_1(m) \theta^*(m) < 1 $. $ (\mathrm{ii})_k $ $ \sup_{m} \mu_k(m) (\theta^*)^k(m) < 1 $.
		\item $ \sup_{m} \mu^*(m) \theta^*(m) < 1 $.
		\item $ \sup_m \mu_1(m) \theta_1(m) < 1 $. $ (\mathrm{iv})_k $ $ \sup_m \mu_k(m) \theta_k(m) < 1 $.
		\item $ \sup_m (\mu\theta)^*(m) < 1 $, where $ (\mu\theta)_n(m) = \mu_n(m)\theta_n(m) $.
		\item $ \sup_m (\overline{\mu\theta})(m) < 1 $ if $ \{ (\mu\theta)_n \} \in \mathcal{E}(u)  $.
	\end{enumerate}
	Then \textnormal{(i)} $ \Rightarrow $ $ (\mathrm{i})_k $, \textnormal{(ii)} $ \Rightarrow $ $ (\mathrm{ii})_k $, \textnormal{(iv)} $ \Rightarrow $ $ (\mathrm{iv})_k $ for all $ k \geq 1 $. $ (\mathrm{i})_k $ $ \Rightarrow $ $ (\mathrm{ii})_k $. \textnormal{(ii)} $ \Rightarrow $ $ (\mathrm{i})_k $ for large $ k $ (depending on $ m $). $ (\mathrm{ii})_k $ $ \Rightarrow $ \textnormal{(iii)} and \textnormal{(iii)} $ \Rightarrow $ $ (\mathrm{ii})_k $ for large $ k $ (depending on $ m $). \textnormal{(iii)} $ \Rightarrow $ \textnormal{(v)}. $ (\mathrm{iv})_k $ $ \Rightarrow $ \textnormal{(v)} and \textnormal{(v)} $ \Rightarrow $ $ (\mathrm{iv})_k $ for large $ k $ (depending on $ m $). \textnormal{(v)} $ \Rightarrow $ \textnormal{(vi)} and \textnormal{(vi)} $ \Rightarrow $ $ (\mathrm{iv})_k $ for large $ k $ (depending on $ m $). Thus, \textnormal{(v)} $ \Leftrightarrow $ \textnormal{(vi)} if $ \{ (\mu\theta)_n \} \in \mathcal{E}(u) $.
\end{lem}
The proof is easy, so we omit it.

In the literature, each of the above types of `spectral gap condition' (or `bunching condition') has been used, e.g.,
\cite{PSW97, PSW12,BLZ99,BLZ08} use (iv) (so (v) (vi)), and \cite{Fen74, Has97, BLZ98, BLZ00, LYZ13} use (iii). In our general settings, we cannot use type (iv) (and so (v) (vi)) of spectral gap condition to give our regularity results. However, our results also recover some classical results when using type (iv) (and so (v) (vi)). This is not surprising since in the classical setting the fibers are assumed to be uniformly bounded; now case (ii) in (R3) (see below) can be applied. Note that in this case, what we actually need is that the Lipschitz functions are also H\"older.

First, consider a special example as a motivation.
\begin{exa}\label{appexb}
	Let
	\begin{equation}\label{appf}
	f(y) \leq x a^t + y b^t,~\forall t >0,
	\end{equation}
	where $ f(y) \geq 0$, $ x, y, a , b > 0 $.
	\begin{enumerate}[(1)]
		\item If $ 0 < a, b < 1 $, then $ f(y) = 0 $. If $ a < 1 = b $, then $ f(y) \leq y $.
		\item Assume $ a < 1 < b $. Take
		\[
		t_0(y) = \frac{1}{\ln b - \ln a} \ln \left( \frac{-x \ln a}{y \ln b} \right).
		\]
		Substituting in \eqref{appf} for $ t $, then we see that
		\begin{equation}\label{appf1}
		f(y) \leq C x^{1-\alpha}y^{\alpha}, ~\text{if}~ y \leq r x,
		\end{equation}
		where $ \alpha = \frac{r}{r+1} < 1 $, $ r = \frac{-\ln a}{\ln b} > 0 $, and $ C  = \sup_{r > 0}\{ e^{-\frac{r}{r+1} \ln r} + e^{\frac{1}{r+1}\ln r} \} $.

		If \eqref{appf} holds only for $ t \in \mathbb{N} $, then letting $ t = [t_0] + 1 $ in \eqref{appf}, we also have \eqref{appf1} for a different constant $ C $.
		If \eqref{appf} further holds for $ t \in \mathbb{R} $, the restriction $ y \leq r x $ is not needed.
	\end{enumerate}
\end{exa}

Let us give the following setting.

\begin{enumerate}[(R1)]
	\item Assume $ \overline{X} , \overline{Y} $, $ \widehat{Y} $ are bundles with base space $ \overline{M} $ and fibers $ \overline{X}_m, \overline{Y}_m $, $ \widehat{Y}_m $, $ m \in \overline{M} $, being metric spaces. We write all the metrics $ d(z,z') = |z-z'| $. Let $ \imath: \overline{M} \to \overline{Y} $, $ \jmath: \overline{M} \to \overline{X} $, $ \imath_0 : \overline{M} \to \widehat{Y} $ be fixed sections. We also write $ |y| = d(y,\imath (m)) $ if $ y \in \overline{Y}_m $; $ |x| = d(x,\jmath (m)) $ if $ x \in \overline{X}_m $; $ |y_0| = d(y_0,\imath_0 (m)) $ if $ y_0 \in \widehat{Y}_m $.

	\item Let $ w: \overline{M} \to \overline{M} $, and $ g: \overline{X} \times \overline{Y} \to  \widehat{Y} $ a bundle map over $ w $.
	Let $ u: \overline{X} (\epsilon_1) \to \overline{X} $, $ v: \overline{X} \times \overline{Y} \to \overline{Y} $ be two bundle maps over $ w $ satisfying
	\begin{gather}
	|u_{m}(x_1) - u_{m}(x_2)| \leq \mu(m)|x_1 - x_2|, ~x_1, x_2 \in \overline{X}_m (\epsilon_1), \label{ulip}\\
	|v_m(x, y)| \leq \kappa(m) |y|, ~ \forall (x,y) \in \overline{X}_m \times \overline{Y}_m, ~
	u_m(\jmath(m)) = \jmath(w(m)),~ m \in \overline{M}, \notag
	\end{gather}
	where
	\[
	\overline{X}(\epsilon) = \{ (m, x): x \in \overline{X}_m(\epsilon) , m \in \overline{M} \}, ~\overline{X}_m(\epsilon) = \{ x \in X_m: |x - \jmath(m)| \leq \epsilon \}.
	\]
	Assume
	\begin{multline}\label{appg}
	|g_m(x_1,y) - g_m(x_2,y)| \leq  C_1 |x_1 - x_2|^{\gamma_1}|y|^{\zeta_1} + C_2 |x_1- x_2|^{\gamma_2}|y|^{\zeta_2}  \\
	 + \theta(m)|g_{w(m)}(u_m(x_1), v_m(x_1,y)) - g_{w(m)}(u_m(x_2), v_m(x_1,y))|,
	\end{multline}
	for some $ 0 \leq \gamma_i, \zeta_i \leq 1 $, $ C_i \geq 0 $, $ i = 1,2 $.

	\item We use the notation $ (\kappa^{\zeta_2} \mu^{\gamma_2})^{*\alpha} \theta < 1 $ which means the following in different settings (similar for $ ( \kappa^{\zeta_2 } \mu^{\gamma_2} / \kappa)^{*\alpha} \theta \kappa < 1 $):
	\begin{enumerate}[(i)]
		\item $ C_2 = 0 $, delete the condition.
		\item $ \sup_m  (\kappa^{\zeta_2 }(m) \mu^{\gamma_2}(m))^\alpha \theta(m) < 1 $ or $ \sup_m  ((\kappa^{\zeta_2 } \mu^{\gamma_2})^\alpha \theta)^*(m) < 1 $ if $ u_m(\cdot)|_{X_{m}(\epsilon_1)} $, $ v_m(\cdot,\cdot) $ are bounded uniformly for $ m \in \overline{M} $; in particular when $ \epsilon_1 < \infty $, and $ \overline{Y}_m $ are bounded uniformly for $ m \in \overline{M} $ if $ \zeta_2 > 0 $.
		\item $ \sup_m  (\kappa^{\zeta_2 }(m) \mu^{\gamma_2}(m))^\alpha \theta(m) < 1 $ or $ \sup_m  ((\kappa^{\zeta_2 } \mu^{\gamma_2})^\alpha \theta)^*(m) < 1 $ if $ \theta \in \mathcal{E}(w) $.
		\item $ \sup_m  (\kappa^{\zeta_2 }(m) \mu^{\gamma_2}(m))^\alpha \theta(m) < 1 $ or $ \sup_m  ((\kappa^{\zeta_2 } \mu^{\gamma_2})^\alpha \theta)^*(m) < 1 $ if $ \alpha = 1 $.
		\item $ \sup_m (\kappa^{\zeta_2} \mu^{\gamma_2})^{*\alpha}(m) \theta^*(m) < 1 $ otherwise.
	\end{enumerate}
\end{enumerate}

\begin{rmk}
	\begin{enumerate}[(a)]
		\item Let $ \lip u^{(n)}_m (\cdot) = \mu_n (m) $ and $ \kappa_n(m) = \inf\{ \kappa': |v^{(n)}_m(x,y)| \leq \kappa'|y|  \} $. One can use $ \{ \mu_n \} $, $ \{ \kappa_n \} $ instead of $ \{ \mu^{(n)} \} $, $ \{ \kappa^{(n)} \} $ to give a better `spectral gap condition'. The proof is essentially the same.
		\item In \autoref{lem:rglobal} and \autoref{lem:rlocal}, the condition $ \sup_m \kappa^{\zeta_1}(m)\mu^{\gamma_1}(m) \theta(m) < 1 $ can be replaced by $ \sup_m  ((\kappa^{\zeta_1 } \mu^{\gamma_1})^\alpha \theta)^* (m) < 1 $.
	\end{enumerate}
\end{rmk}

First, consider the following two cases when $ \epsilon_1 = \infty $ or $ \sup_m \mu(m) \leq 1 $.
\begin{lem}\label{lem:rglobal}
	Let \textnormal{(R1), (R2), (R3)} hold. Assume $ \epsilon_1 = \infty $, or $ \sup_m \mu(m) \leq 1 $, or 
	\[
	u_m(X_m(\epsilon_1)) \subset X_{u(m)}(\epsilon_1)~\text{for all $ m $};
	\]
	$ \mu, \kappa, \theta $ are bounded functions; \eqref{ulip} and \eqref{appg} hold for all $ x_1, x_2 \in \overline{X}(\epsilon_1) $, \textbf{or} $ x_1 \in \overline{X}(\epsilon_1) $ and $ x_2 = \jmath(m) $.
	\begin{enumerate}[(a)]
		\item \label{g1} Suppose
		\begin{enumerate}[(i)]
			\item $ g $ is bounded in the sense that $ \sup_m \sup_{(x,y)\in \overline{X}_m \times \overline{Y}_m} |g_m(x,y)| < \infty $,
			\item $ \sup_m\theta(m) < 1 $, $ \sup_m \kappa^{\zeta_1}(m)\mu^{\gamma_1}(m) \theta(m) < 1 $,
			$ (\kappa^{\zeta_2} \mu^{\gamma_2})^{*\alpha} \theta < 1 $.
		\end{enumerate}
		Then
		\[
		|g_m(x_1,y) - g_m(x_2,y)| \leq CC_1|x_1 - x_2|^{\gamma_1}|y|^{\zeta_1} + CC_2(|x_1 - x_2|^{\gamma_2}|y|^{\zeta_2})^\alpha,
		\]
		if $ \alpha < 1 $ under a restriction $  |x_1 - x_2|^{\gamma_2}|y|^{\zeta_2} \leq \hat{r} $, and if $ \alpha = 1 $ for $ y \in \overline{Y}_m $, $ m \in \overline{M} $, where $ C $ is a constant depending on the constant $ \hat{r} > 0 $, but not on $ m \in \overline{M} $.

		\item \label{g2} Suppose
		\begin{enumerate}[(i)]
			\item $ |g_m(x,y)| \leq M_0|y| $ for all $ x \in \overline{X}_m $, $ y \in \overline{Y}_m $, where $ M_0 $ is a constant independent of $ m, x $,
			\item $ \sup_m \theta(m) \kappa(m) < 1 $, $ \sup_m \kappa^{\zeta_1}(m)\mu^{\gamma_1}(m) \theta(m) < 1 $, $ ({\kappa^{\zeta_2 } \mu^{\gamma_2}} / {\kappa})^{*\alpha} \theta \kappa < 1 $.
		\end{enumerate}
		Then
		\[
		|g_m(x_1,y) - g_m(x_2,y)| \leq CC_1|x_1 - x_2|^{\gamma_1}|y|^{\zeta_1} + CC_2(|x_1 - x_2|^{\gamma_2}|y|^{\zeta_2})^\alpha |y|^{1-\alpha},
		\]
		if $ \alpha < 1 $ under a restriction $  |x_1 - x_2|^{\gamma_2}|y|^{\zeta_2 - 1} \leq \hat{r} $, and if $ \alpha = 1 $ for $ y \in \overline{Y}_m $, $ m \in \overline{M} $, where $ C $ is a constant depending on the constant $ \hat{r} > 0 $, but not on $ m \in \overline{M} $.
	\end{enumerate}
\end{lem}
\begin{proof}
	We only give the proof of (a); it is similar for (b).
	First, assume $ \alpha < 1 $. Set
	\[
	\nu_1 = \sup_m \kappa^{\zeta_1}(m)\mu^{\gamma_1}(m) \theta(m),~
	\nu_2^{(k)}(m) = (\kappa^{(k)}(m))^{\zeta_2 } (\mu^{(k)}(m))^{\gamma_2}\theta^{(k)}(m).
	\]
	Take $ \theta_0 $ such that $ \sup_m \theta(m) < \theta_0 < 1 $. Note that $ (\kappa^{(k)}(m))^{\zeta_1} (\mu^{(k)}(m))^{\gamma_1} \theta^{(k)}(m) \leq \nu^k_1 $.

	First, we consider the case where $ (\kappa^{\zeta_2} \mu^{\gamma_2})^{*\alpha} \theta < 1 $ means $ \sup_m  ((\kappa^{\zeta_2 } \mu^{\gamma_2})^\alpha \theta)^*(m) < 1 $ under $ \theta \in \mathcal{E}(w) $. Choose $ \hat{\alpha} $ such that
	\[
	\max \left\{ 1, \zeta_2 \gamma_2 \sup_m \limsup_{k \to \infty} \frac{\ln \kappa^{(k)}(m) \ln \mu^{(k)}(m)}{-\ln \theta^{(k)}(m)} \right\} < \hat{\alpha}^{-1} < \alpha^{-1}.
	\]
	Since $ \theta \in \mathcal{E}(w) $, we have $ \overline{\theta} : \overline{M} \to \mathbb{R}_+ $ such that
	\[
	\mathrm{(a)}' ~(1+\varepsilon(m))^{-k}\overline{\theta}^k(m) \leq \theta^{(k)}(m),~~ \mathrm{(b)}'~  \theta^{(k)}(m) \leq (1+\varepsilon(m))^k \overline{\theta}^k(m),
	\]
	for $ k \geq N \geq N(m) $. As $ \sup_m  ((\kappa^{\zeta_2 } \mu^{\gamma_2})^\alpha \theta)^*(m) < 1 $, we can further get
	\[
	\mathrm{(c)}'~ (   \kappa^{(k)} (m) )^{\zeta_2 \alpha} (\mu^{(k)} (m))^{\gamma_2\alpha}  \theta^{(k)}(m) < 1,
	\]
	for $ k \geq N \geq N(m) $.
	The function $ \varepsilon: M \to \mathbb{R}_+ $ we choose satisfies $ (1+\varepsilon(m))\overline{\theta}^\beta(m) = 1 $, where $ \beta = \frac{\hat{\alpha} - \alpha}{\alpha+\hat{\alpha} - 2\alpha\hat{\alpha}} > 0 $. Set $ \overline{\theta}_1(m) = \left( \frac{\overline{\theta}(m)}{1+\varepsilon(m)}  \right)^{1-1/\hat{\alpha}} > 1 $ and $ \overline{\theta}_2(m) = (1+\varepsilon(m)) \overline{\theta}(m) $. Note that $ \overline{\theta}_2(m) < 1 $ and $ \overline{\theta}_1(m) = \overline{\theta}^{1 - 1/\alpha}_2(m) $.
	Also, we can assume that $ \sup_m\overline{\theta}_2(m) < \theta_0 $ and that $ v^{(k)}_2(m) \leq (\theta^{(k)} (m))^{1-1/\hat{\alpha}} $ for $ k \geq N \geq N(m) $.
	Assume $ C_2 > 0 $.
	Note that $ v^{(k)}_2(m) \leq (\theta^{(k)} (m))^{1-1/\hat{\alpha}} \leq \overline{\theta}^k_1(m) $ if $ k \geq N \geq N(m) $.

	By iterating \eqref{appg} $ kN - 1 $ times, $ k \geq 1 $, we have
	\begin{align}\label{appb}
	& |g_m(x_1,y) - g_m(x_2,y)| \notag\\
	\leq & C_1 \left\lbrace 1 + \nu_1 + \cdots + \nu^{kN-1}_1 \right\rbrace |x_1 - x_2|^{\gamma_1}|y|^{\zeta_1} \notag\\
	& + C_2 \left\lbrace 1 +  \nu_2^{(1)}(m) + \cdots + \nu_2^{(kN-1)}(m) \right\rbrace |x_1 - x_2|^{\gamma_2}|y|^{\zeta_2} \notag\\
	& + \theta^{(kN)}(m)  | g_{w^{kN}(m)}(u^{(kN)}_{m}(x_1), v^{(kN)}_{m}(x_1,y)) - g_{w^{kN}(m)}(u^{(kN)}_{m}(x_2), v^{(kN)}_{m}(x_1,y)) | \notag\\
	\leq & \widehat{C}_1 |x_1 - x_2|^{\gamma_1}|y|^{\zeta_1} + \widehat{C}_2(m,k,N) |x_1 - x_2|^{\gamma_2}|y|^{\zeta_2}  \notag\\
	& + \theta^{(kN)}(m)  | g_{w^{kN}(m)}(u^{(kN)}_{m}(x_1), v^{(kN)}_{m}(x_1,y)) - g_{w^{kN+1}(m)}(u^{(kN)}_{m}(x_2), v^{(kN)}_{m}(x_1,y)) |,
	\end{align}
	where
	\[
	\widehat{C}_1 = \frac{C_1}{1-v_1}, ~~\widehat{C}_2(m,k,N) =
	C_2 V(m) + C_2 \frac{\overline{\theta}_1^{kN}(m) - \overline{\theta}_1^{N}(m)}{\overline{\theta}_1(m) - 1 },
	\]
	and $ V(m) = 1 + \nu^{(1)}_2(m) + \cdots + \nu^{(N-1)}_2(m) $.

	Consider $ g $ as a bundle map over $ w^{kN} $. By iterating \eqref{appb} $ n - 1 $ times, we get
	\begin{align}\label{appb00}
	& |g_m(x_1,y) - g_m(x_2,y)| \notag\\
	\leq & \widehat{C}_1 \left\lbrace 1 + \nu^{kN}_1 + \cdots + \nu^{(n-1)kN}_1 \right\rbrace |x_1 - x_2|^{\gamma_1}|y|^{\zeta_1} \notag\\
	& + \widehat{C}_2(m,k,N) \left\lbrace 1 +  \nu_2^{(kN)}(m) + \cdots + \nu_2^{(n-1)kN}(m) \right\rbrace |x_1 - x_2|^{\gamma_2}|y|^{\zeta_2} + C_0 (\overline{\theta}_2^{kN}(m))^n \notag\\
	\leq & \widehat{C}'_1 |x_1 - x_2|^{\gamma_1}|y|^{\zeta_1} \notag\\
	& + \widehat{C}_2(m,k,N) \left\lbrace 1 + \overline{\theta}_1^{kN}(m) + \cdots + \overline{\theta}_1^{(n-1)kN}(m) \right\rbrace |x_1 - x_2|^{\gamma_2}|y|^{\zeta_2} + C_0 (\overline{\theta}_2^{kN}(m))^n \notag\\
	\leq & \widehat{C}'_1 |x_1 - x_2|^{\gamma_1}|y|^{\zeta_1} + C''_2(m,k,N)|x_1 - x_2|^{\gamma_2}|y|^{\zeta_2} (\overline{\theta}_1^{kN}(m))^{n} + C_0 (\overline{\theta}_2^{kN}(m))^n,
	\end{align}
	where $ \widehat{C}'_1 = \widehat{C}_1 /(1 - \nu_1^{kN}) $, and
	\[
	C_0 \geq 2\sup_m \sup_{(x,y)\in X_m \times Y_m} |g_m(x,y)|,~
	C''_2(m,k,N) = \widehat{C}_2(m,k,N) / ( \overline{\theta}_1^{kN}(m) - 1 ).
	\]
	Note that $ \overline{\theta}_1(m) > 1 $. We can choose $ k $ so large that $ C''_2(m,k,N) \leq \widehat{C}_2 \triangleq 2C_2 / (\theta^{1-1/\hat{\alpha}}_0 - 1) $ and also $ \widehat{C}'_1 \leq 2\widehat{C}_1 $.
	Let
	\[
	n = \left[ \frac{-\alpha}{\ln \overline{\theta}_2(m)} \ln \left( \frac{\alpha}{1-\alpha}\frac{C_0/\widehat{C}_2}{|x_1 - x_2|^{\gamma_2}|y|^{\zeta_2}} \right) \right] + 1,
	\]
	and $ |x_1 - x_2|^{\gamma_2}|y|^{\zeta_2} \leq \hat{r} $, where $ \hat{r}  $ is taken such that $ n \geq 1 $. Then
	\[
	|g_m(x_1,y) - g_m(x_2,y)| \leq 2 \widehat{C}_1 |x_1 - x_2|^{\gamma_1}|y|^{\zeta_1} + CC_2(|x_1 - x_2|^{\gamma_2}|y|^{\zeta_2})^\alpha.
	\]
	(Here we take $ a = \overline{\theta}_2(m) < 1 $, $ b = \overline{\theta}_1(m) > 1 $, $ f =  C_0 b^n + \widehat{C}_2 |x_1 - x_2|^{\gamma_2}|y|^{\zeta_2}a^n $; see \autoref{appexb}.) The above inequality also holds for $ \alpha = 1 $, but then the restriction $ |x_1 - x_2|^{\gamma_2}|y|^{\zeta_2} \leq \hat{r} $ is not needed anymore.

	If $ C_2 = 0 $, by iterating \eqref{appg} $ n $ times, the result follows. Also in this case we only need $ \sup_m  ((\kappa^{\zeta_1 } \mu^{\gamma_1})^\alpha \theta)^* (m) < 1 $, so case (iv) in (R3) follows. For case (ii) in (R3), we have
	\begin{equation*}
	|u^{(n)}_m(x_1) - u^{(n)}_m(x_2)| \leq C' (\mu^{(n)}(m))^\alpha|x_1 - x_2|^\alpha, ~
	|\nu^{(n)}_m(x, y)| \leq C' (\kappa^{(n)}(m))^\alpha |y|^\alpha,
	\end{equation*}
	where $ C' = C'(\alpha) $ is a constant independent of $ m $. (Here note that in the `unbounded' case Lipschitz functions are usually not H\"older.) Now we can reduce this case to $ C_2 = 0 $.
	The case (v) in (R3) is easier than case (iii) by using $ \theta^* $ instead of $ \theta $.
\end{proof}

  Next, consider another two cases when $ \epsilon_1 < \infty $ and $ \sup_m \mu(m) > 1 $.

\begin{enumerate}[({R}3$ ' $)]
	\item The notation $ ( \mu^{\gamma_2} )^{*\alpha} \theta < 1 $ will mean the following in different settings (similar for $ (  \mu^{\gamma_2} )^{*\alpha} \theta \kappa < 1 $):
	\begin{enumerate}[(i)]
		\item $ \sup_m  ( \mu^{\gamma_2}(m) )^\alpha \theta(m) < 1 $ or $ \sup_m  (( \mu^{\gamma_2})^\alpha \theta)^*(m) < 1 $ if (ii) in (R3) holds.

		\item $ \sup_m  ( \mu^{\gamma_2}(m) )^\alpha \theta(m) < 1 $ or $ \sup_m  (( \mu^{\gamma_2})^\alpha \theta)^*(m) < 1 $ if $ \theta \in \mathcal{E}(w) $.
		\item $ \sup_m  ( \mu^{\gamma_2}(m))^\alpha \theta(m) < 1 $ or $ \sup_m  ((\mu^{\gamma_2})^\alpha \theta)^*(m) < 1 $ if $ \alpha = 1 $.
		\item $ \sup_m ( \mu^{\gamma_2})^{*\alpha}(m) \theta^*(m) < 1 $ otherwise.
	\end{enumerate}
\end{enumerate}

\begin{lem}\label{lem:rlocal}
	Let \textnormal{(R1), (R2), (R3), (R3$ ' $)} hold. Assume $ \epsilon_1 < \infty $ and $ \sup_m \mu(m) > 1 $; $ \mu, \kappa, \theta $ are bounded functions; \eqref{ulip} and \eqref{appg} hold for all $ x_1 \in \overline{X}(\epsilon_1) $, $ x_2 = \jmath(m) $; and $ C_2 > 0 $.
\begin{enumerate}[(a)]
	\item \label{r1} Suppose
	\begin{enumerate}[(i)]
		\item $ g $ is bounded in the sense that $ \sup_m \sup_{(x,y)\in \overline{X}_m \times \overline{Y}_m} |g_m(x,y)| < \infty $,
		\item $ \sup_m\theta(m) < 1 $, $ \sup_m \kappa^{\zeta_1}(m)\mu^{\gamma_1}(m) \theta(m) < 1 $,
		$ (\kappa^{\zeta_2 } \mu^{\gamma_2})^{*\alpha} \theta < 1 $, $   (\mu^{\gamma_2})^{*\alpha} \theta< 1 $.
	\end{enumerate}
	If $ 0 < \alpha \leq 1 $, then 
		\[
		|g_m(x_1,y) - g_m(\jmath(m),y)| \leq CC_1|x_1 - \jmath(m)|^{\gamma_1}|y|^{\zeta_1} + CC_2(|x_1 - \jmath(m)|^{\gamma_2}|y|^{\zeta_2})^\alpha
		\]
		whenever $ |x_1 - \jmath(m)|^{\gamma_2}|y|^{\zeta_2} \leq \hat{r}\min\{ |y|^{\zeta_2c}, 1 \}$ for $ y \in \overline{Y}_m $, $ m \in \overline{M} $, and some constants $ \hat{r} > 0 $, $ c > 1 $, where $ C $ is a constant depending on $ \hat{r} > 0 $, but independent of $ m $.

	\item \label{r2} Suppose
	\begin{enumerate}[(i)]
		\item $ |g_m(x,y)| \leq M_0|y| $ for all $ x \in \overline{X}_m $, $ y \in \overline{Y}_m $, where $ M_0 $ is a constant independent of $ m, x $,
		\item $ \sup_m \theta(m) \kappa(m) < 1 $, $ \sup_m \kappa^{\zeta_1}(m)\mu^{\gamma_1}(m) \theta(m) < 1 $, $   (\kappa^{\zeta_2 } \mu^{\gamma_2} / {\kappa} )^{*\alpha} \theta \kappa < 1 $, 
		
		$  (\mu^{\gamma_2})^{*\alpha} \theta \kappa < 1 $.
	\end{enumerate}
	If $ 0 < \alpha \leq 1 $, then 
		\[
		|g_m(x_1,y) - g_m(\jmath(m),y)| \leq CC_1|x_1 - \jmath(m)|^{\gamma_1}|y|^{\zeta_1} + CC_2(|x_1 - \jmath(m)|^{\gamma_2}|y|^{\zeta_2})^\alpha |y|^{1-\alpha},
		\]
		whenever $ |x_1 - \jmath(m)|^{\gamma_2}|y|^{\zeta_2} \leq \hat{r}\min\{ |y|^{\zeta_2 c - (c-1)}, |y|^{} \}$ for $ y \in \overline{Y}_m $, $ m \in \overline{M} $, and some constants $ \hat{r} > 0 $, $ c > 1 $, where $ C $ is a constant depending on $ \hat{r} > 0 $ but not on $ m \in \overline{M} $.
\end{enumerate}
\end{lem}
\begin{rmk}
	If $ C_2 = 0 $, then we can reduce the case to $ C_1 = 0 $ and $ \alpha = 1 $. The corresponding `spectral gap condition' is $ \sup_m\theta(m) < 1 $, $ \sup_m \kappa^{\zeta_1}(m)\mu^{\gamma_1}(m) \theta(m) < 1 $, $  \sup_m \mu^{\gamma_1}(m) \theta(m)< 1 $ for case (a), and $ \sup_m \theta(m) \kappa(m) < 1 $, $ \sup_m \kappa^{\zeta_1}(m)\mu^{\gamma_1}(m) \theta(m) < 1 $, $ \sup_m \mu^{\gamma_1}(m) \theta(m) \kappa(m) < 1 $ for case (b), respectively.
\end{rmk}
\begin{proof}[Proof of \autoref{lem:rlocal}]
	We only give the proof of (a); it is similar for (b).

	\textbf{Case} $ \alpha < 1 $. We first consider the case where $ (\kappa^{\zeta_2} \mu^{\gamma_2})^{*\alpha} \theta < 1 $, $  (\mu^{\gamma_2})^{*\alpha} \theta< 1 $ mean $ \sup_m  ((\kappa^{\zeta_2 } \mu^{\gamma_2})^\alpha \theta)^*(m) < 1 $, $ \sup_m ((\mu^{\gamma_2})^\alpha \theta)^* (m) < 1 $, respectively, and $ \theta \in \mathcal{E}(w) $. Let $ v_1, v^{(k)}_2(m), \theta_0 $ be as in the proof of \autoref{lem:rglobal}.

	Since $ \theta \in \mathcal{E}(w) $, we have $ \overline{\theta} : \overline{M} \to \mathbb{R}_+ $ such that
	\[
	\mathrm{(a)}' ~(1+\varepsilon(m))^{-n}\overline{\theta}^n(m) \leq \theta^{(n)}(m),~~ \mathrm{(b)}'~  \theta^{(k)}(m) \leq (1+\varepsilon(m))^k \overline{\theta}^k(m),
	\]
	for $ k \geq N \geq N(m) $. As $ \sup_m  ((\kappa^{\zeta_2 } \mu^{\gamma_2})^\alpha \theta)^*(m) < 1 $ and $ \sup_m ((\mu^{\gamma_2})^\alpha \theta)^* (m) < 1 $, we can further get
	\[
	\mathrm{(c)}'~ (   \kappa^{(k)} (m) )^{\zeta_2 \alpha} (\mu^{(k)} (m))^{\gamma_2\alpha'}  \theta^{(k)}(m) < 1, ~~ \mathrm{(d)}'~ (\mu^{(k)}(m))^{\gamma_2\alpha'} (1+\varepsilon(m))^k \overline{\theta}^k(m) < 1,
	\]
	for $ k \geq N \geq N(m) $, where $ \alpha' $ is chosen so that
	\begin{multline*}
	\max\left\{1, \zeta_2 \gamma_2 \sup_m \limsup_{k \to \infty} \frac{\ln \kappa^{(k)}(m) \ln \mu^{(k)}(m)}{-\ln \theta^{(k)}(m)}, \gamma_2 \sup_m \limsup_{k \to \infty} \frac{\ln \mu^{(k)}(m)}{-\ln \theta^{(k)}(m)}\right\} \\
	< {\alpha'}^{-1} < \alpha^{-1}.
	\end{multline*}
	The function $ \varepsilon: M \to \mathbb{R}_+ $ is chosen so that it satisfies $ (1+\varepsilon(m))\overline{\theta}^\beta(m) = 1 $, where $ \beta = \frac{{\alpha'} - \alpha}{\alpha+{\alpha'} - 2\alpha{\alpha'}} > 0 $. Set $ \overline{\theta}_1(m) = \left( \frac{\overline{\theta}(m)}{1+\varepsilon(m)}  \right)^{1-1/{\alpha'}} > 1 $ and $ \overline{\theta}_2(m) = (1+\varepsilon(m)) \overline{\theta}(m) $. Note that $ \overline{\theta}_2(m) < 1 $ and $ \overline{\theta}_1(m) = \overline{\theta}^{1 - 1/\alpha}_2(m) $.

	When $ |x_1 - \jmath(m)| \leq (\mu^{(kN-1)}(m))^{-1} \epsilon_1 $,
	by iterating \eqref{appg} $ kN - 1 $ times, we see \eqref{appb} holds for $ x_2 = \jmath(m) $.
	Consider $ g $ as a bundle map over $ w^{kN} $.
	Whenever $ |x_1 - \jmath(m)| \leq (\mu^{((n-1)kN)}(m))^{-1} \epsilon_1 $, by iterating \eqref{appb} $ n-1 $ times, we also see that \eqref{appb00} holds for $ x_2 = \jmath(m) $.

	Let
	\[
	n(m) \triangleq \frac{-\alpha}{\ln \overline{\theta}_2(m)} \ln \left( \frac{\alpha}{1-\alpha}\frac{C_0/\widehat{C}_2}{|x_1 - x_2|^{\gamma_2}|y|^{\zeta_2}} \right),
	\]
	and $ n = [n(m)]+1 $.
	Now $ |x_1 - \jmath(m)| \leq (\mu^{((n-1)kN)}(m))^{-1} \epsilon_1 $ can be satisfied if
	\[
	|x_1 - \jmath(m)| \leq (\overline{\theta}_2(m))^{n(m)kN/(\gamma_2 \alpha')}\epsilon_1,
	\]
	or equivalently
	\[
	|x_1 - \jmath(m)|^{\gamma_2}|y|^{\zeta_2} \leq \left( \frac{\alpha}{1-\alpha} \right)^{-\tilde{\alpha}/(1-\tilde{\alpha})}|y|^{\zeta_2/(1-\tilde{\alpha})}  \widehat{M}_0,
	\]
	where $ \tilde{\alpha} = \alpha / \alpha' < 1 $, $ \widehat{M}_0 = (\widehat{C}_2 / C_0)^{1/(1/\tilde{\alpha} - 1)} \epsilon_1^{\gamma_2/ (1-\tilde{\alpha})} $, and then
	\[
	|g_m(x_1,y) - g_m(\jmath(m),y)| \leq 2\widehat{C}_1 |x_1 - \jmath(m)|^{\gamma_1}|y|^{\zeta_1} + CC_2(|x_1 - \jmath(m)|^{\gamma_2}|y|^{\zeta_2})^{\alpha},
	\]
	for some constant $ C $ independent of $ m $ (and $ \alpha $).

	\textbf{Case} $ \alpha = 1 $. In this case, the condition is given by $ \sup_m \kappa^{\zeta_2}(m)\mu^{\gamma_2}(m) \theta(m) < 1 $, $ \sup_m \mu^{\gamma_2}(m) \theta(m) < 1 $. (Similarly $ \sup_m (\kappa^{\zeta_2}\mu^{\gamma_2} \theta)^* (m) < 1 $ and $ \sup_m (\mu^{\gamma_2} \theta)^* (m) < 1 $.) Let $ \nu_2 = \sup_m \kappa^{\zeta_2}(m)\mu^{\gamma_2}(m) \theta(m)~(<1) $. Now we have $ \nu^{(k)}_2(m) \leq \nu^k_2 $. Without loss of generality, assume $ C_1 = 0 $. Since $ \sup_m \mu^{\gamma_2}(m) \theta(m) < 1 $, we can choose $ \alpha'' > 1 $ such that $ \sup_m (\mu^{\gamma_2}(m))^{\alpha''} \theta(m) < 1 $. When $ |x_1 - \jmath(m)| \leq (\mu^{(n-1)}(m))^{-1} \epsilon_1 $, by iterating \eqref{appg} $ n-1 $ times, we get
	\begin{align*}
	|g_m(x_1,y) - g_m(\jmath(m),y)|
	\leq \widehat{C}_3  |x_1 - \jmath(m)|^{\gamma_2}|y|^{\zeta_2} + C_0 \theta^{(n)} (m),
	\end{align*}
	where $ \widehat{C}_3 = C_2 / (1-\nu_2) $. The above inequality can be satisfied if 
	\[
	|x_1 - \jmath(m)| \leq (\theta^{(n)}(m))^{1/(\gamma_2\alpha'')}\epsilon_1.
	\]
	If $ |x_1 - \jmath(m)|^{\gamma_2}|y|^{\zeta_2} \neq 0 $, choose $ n $ so that
	\[
	C_0 \theta^{(n)} (m) \leq \widehat{C}_3  |x_1 - \jmath(m)|^{\gamma_2}|y|^{\zeta_2}.
	\]

	In order to do this, we need
	\[
	|x_1 - \jmath(m)| \leq (\theta^{(n)}(m))^{1/(\gamma_2\alpha'')}\epsilon_1 \leq (C_0^{-1} \widehat{C}_3  |x_1 - \jmath(m)|^{\gamma_2}|y|^{\zeta_2})^{1/(\gamma_2\alpha'')} \epsilon_1,
	\]
	i.e.,
	\[
	|x_1 - \jmath(m)|^{\gamma_2} |y|^{\zeta_2} \leq \widehat{M}_1 |y|^{\zeta_2 \alpha'' / (\alpha'' -1 ) },
	\]
	where $ \widehat{M}_1 = (C_0^{-1} \widehat{C}_3)^{ \gamma_2 / (\alpha'' - 1) } \epsilon^{\gamma_2 / (1-1/\alpha'')}_1 $. In more detail, when
	\[
	|x_1 - \jmath(m)|^{\gamma_2} |y|^{\zeta_2} \leq \min\{\widehat{M}_1 |y|^{\zeta_2 \alpha'' / (\alpha'' -1 ) }, \hat{r}\}, \tag{$ \blacklozenge $}
	\]
	where $ \hat{r} > 0 $ is small, one can always choose $ n $ such that
	\[
	|x_1 - \jmath(m)| \leq (\mu^{(n-1)}(m))^{-1} \epsilon_1,~C_0 \theta^{(n)} (m) \leq \widehat{C}_3  |x_1 - \jmath(m)|^{\gamma_2}|y|^{\zeta_2}. \tag{$ \blacklozenge \blacklozenge $}
	\]
	Since we have assumed $ \mu $ is a bounded function, without loss of generality let $ \inf_m\theta(m) \geq \tilde{\theta} > 0 $ (otherwise take $ \theta(m) + \zeta $ instead of $ \theta(m) $ where $ \zeta > 0 $ is sufficiently small). Reselect $ \alpha'' $ such that $ \sup_m (\mu^{\gamma_2}(m))^{\rho} \theta(m) < 1 $ for all $ \rho \in [1,\tau\alpha''] $, where $ \tau > 1 $ satisfies $ \theta_0 < \tilde{\theta}^{1/\tau} $. Then there exists a $ \delta > 0 $ such that
	\[
	(0,\delta] \subset \bigcup_{n \geq 0} [(\theta^{(n)}(m))^{1/(\gamma_2\alpha'')}, (\theta^{(n)}(m))^{1/(\tau\gamma_2\alpha'')}].
	\]
	Indeed, since $ \theta_0^n < \tilde{\theta}^{(n+1)/ \tau} $ for large $ n $, we have $ a_n \leq  b_{n+1} $ where $ a_n = (\theta^{(n)}(m))^{1/(\gamma_2\alpha'')} $ and $ b_n = (\theta^{(n)}(m))^{1/(\tau\gamma_2\alpha'')} $.
	If $ \hat{r} $ is small and ($ \blacklozenge $) holds, then there is an $ n $ such that
	\[
	|x_1 - \jmath(m)| \leq (\theta^{(n)}(m))^{1/(\gamma_2\rho(n))}\epsilon_1 \leq (C_0^{-1} \widehat{C}_3  |x_1 - \jmath(m)|^{\gamma_2}|y|^{\zeta_2})^{1/(\gamma_2\alpha'')} \epsilon_1,
	\]
	where $ \rho(n) \in [\alpha'', \tau\alpha''] $, which yields ($ \blacklozenge \blacklozenge $).

	By using $ \theta^* $ instead of $ \theta $, one can get case (v) in (R3) or (iv) in (R3$ ' $). The final case is (ii) in (R3) and (i) in (R3$ ' $), which can be reduced to the case $ \alpha = 1 $. Thus, the proof is complete.
\end{proof}

\begin{rmk}\label{moregeneral}
	The readers can consider the more general case when \eqref{appg} is replaced by
	\begin{multline*}
	|g_m(x_1,y) - g_m(x_2,y)| \leq  C_1 |x_1 - x_2|^{\gamma_1}(|y|^{\zeta_1} + c_1) + C_2 |x_1- x_2|^{\gamma_2}(|y|^{\zeta_2} + c_2)  \\
	 + \theta(m)|g_{w(m)}(u_m(x_1), v_m(x_1,y)) - g_{w(m)}(u_m(x_2), v_m(x_1,y))|,
	\end{multline*}
	where $ c_i = 0 ~\text{or}~ 1 $, and give the corresponding results. The proof only needs a minor change, so we omit it.
\end{rmk}

\begin{rmk}\label{argumentapp}
	In the following, we show how the results obtained in this appendix can be applied to prove the regularity results in Lemmas \ref{lem:leaf1a}, \ref{lem:sheaf}, \ref{lem:base0}, \ref{lem:baseleaf}, \ref{lem:holversheaf}, \ref{lem:final}.
	\begin{asparaenum}[(a)]
		\item\label{lem1} ($ x \mapsto K^1_m(x) $) In the proof of \autoref{lem:leaf1a}, we need to consider \eqref{k11}.
		Let $ \overline{M} = M $, $ \overline{X}_m = X_m $, $ \overline{Y}_m = \{m\} $, $ \widehat{Y}_m = L(X_m, Y_m) $, $ g_m(x,y) = K^1_m(x) $, $ w(m) = u(m) $, $ u_m(x) = x_m(x) $, $ \jmath = \id $, $ \imath_0 = 0 $. Then apply \autoref{lem:rglobal} \eqref{g1}.

		\item\label{lem2} ($ m \mapsto f_m(x) $) In the proof of \autoref{lem:sheaf}, we have \eqref{basef}.
		Let $ \overline{M} = M_1 $, $ \overline{X}_{m_0} = U_{m_0}(\hat{\mu}^{-2}\varepsilon_1) $, $ \overline{Y}_{m_0} = X_{m_0} $, $ \widehat{Y}_m = Y_m $, $ g_{m_0}(m,x) = \widehat{f}_{m_0}(m,x) $, $ w(m_0) = u(m_0) $, $ u_{m_0}(m) = u(m) $, $ v_{m_0}(m,x) =  \widehat{x}_{m_0} (m,x) $, $ \jmath = \id $, $ \imath = i_X $, $ \imath_0 = i_Y $. Then apply \autoref{lem:rlocal} \eqref{r2}.

		Let us consider the proof of \autoref{h3*}. In this case, \eqref{basef} becomes
			\begin{multline}\label{basef1}
			|\widehat{f}_{m_0}(m_1,x) - \widehat{f}_{m_0}(m_0,x)| \leq \widetilde{C} |m_1 - m_0|^\gamma \\
			 + \theta_1 (m_0) | \widehat{f}_{u(m_0)} (u(m_1), \widehat{x}_{m_0} (m_1,x) ) - \widehat{f}_{u(m_0)} (u(m_0), \widehat{x}_{m_0} (m_1,x) ) |,
			\end{multline}
			where $ \theta_1 (m_0) = \frac{\lambda''_u(m_0)}{1-\alpha''(m_0)\beta''(u(m_0))} $.
			Although $ \widehat{f}_{m_0}(m_1,x) $ now does not satisfy the condition (i) in \autoref{lem:rlocal} \eqref{r2}, the same argument can apply. Note that
			\begin{align}
			& | \widehat{f}_{m_0}(m_1,x) - \widehat{f}_{m_0}(m_0,x) | \notag\\
			\leq & |\widehat{f}_{m_0}(m_1,x) - \widehat{f}_{m_0}(m_1, \widehat{i}^{m_0}_X(m_1))| \notag\\
			& + |\widehat{f}_{m_0}(m_0,x) - \widehat{f}_{m_0}(m_0, \widehat{i}^{m_0}_X(m_0))| + |\widehat{f}_{m_0}(m_1, \widehat{i}^{m_0}_X(m_1)) - \widehat{f}_{m_0}(m_0, \widehat{i}^{m_0}_X(m_0))| \notag\\
			\leq & \beta''(m_0) (|x - \widehat{i}^{m_0}_X(m_1)| + |x - \widehat{i}^{m_0}_X(m_0)|) + |\widehat{i}^{m_0}_X(m_1) - \widehat{i}^{m_0}_X(m_0)| \notag\\
			\leq & \beta''(m_0) (|x - \widehat{i}^{m_0}_X(m_1)| + |x - \widehat{i}^{m_0}_X(m_0)|) + c_0|m_1 - m_0| \label{appfff}.
			\end{align}

			Iterating \eqref{basef1} $ n-1 $ times, we have
			\begin{align*}
			&|\widehat{f}_{m_0}(m_1,x) - \widehat{f}_{m_0}(m_0,x)| \\
			\leq & \widetilde{C} \left\lbrace 1 + (\mu(m_0))^{\gamma}\theta_1(m_0) + \cdots + (\mu^{(n-1)}(m_0))^{\gamma}\theta_1^{(n-1)}(m_0)  \right\rbrace |m_1 - m_0|^\gamma \\
			& + \theta^{(n)}_1 (m_0) |\widehat{f}_{u^n(m_0)} (u^n(m_1), \widehat{x}^{(n)}_{m_0} (m_1,x) ) - \widehat{f}_{u^n(m_0)} (u^n(m_0), \widehat{x}^{(n)}_{m_0} (m_1,x) ) |,
			\end{align*}
			where $ \widehat{x}^{(n)} $ is the $ n $th composition of $ \widehat{x} $. From \eqref{appfff}, we see that
			\begin{align*}
			&|\widehat{f}_{m_0}(m_1,x) - \widehat{f}_{m_0}(m_0,x)| \\
			\leq & \widetilde{C} \left\lbrace 1 + (\mu(m_0))^{\gamma}\theta_1(m_0) + \cdots + (\mu^{(n-1)}(m_0))^{\gamma}\theta_1^{(n-1)}(m_0)  \right\rbrace |m_1 - m_0|^\gamma \\
			& + \beta''(u^n(m_0)) \theta^{(n)}_1(m_0) (\lambda''_s)^{(n)}(m_0) \left\{ |x - \widehat{i}^{m_0}_X(m_1)|+ |x - \widehat{i}^{m_0}_X(m_0)| \right\} \\
			& + c_0 \mu^{(n)}(m_0)\theta^{(n)}_1(m_0)|m_1 - m_0|.
			\end{align*}
			Now we can use the same argument as in the proof of \autoref{lem:rlocal} \eqref{r1} to show
			\[
			|\widehat{f}_{m_0}(m_1,x) - \widehat{f}_{m_0}(m_0,x)| \leq \widetilde{C}|m_1 - m_0|^{\gamma\alpha},
			\]
			if $ |x - \widehat{i}^{m_0}_X(m_1)|+ |x - \widehat{i}^{m_0}_X(m_0)| \leq r_1 $, and $ |m_1 - m_0| \leq \hat{r}(r_1) $, for some suitable constant $ \hat{r}(r_1) $ depending on $ r_1 $, where $ \widetilde{C} $ is a constant depending on $ r_1 $, but not on $ m_0 \in M_1 $, and $ r_1 $ is chosen arbitrarily.

			\item \label{lem3} ($ m \mapsto K^1_m(x) $) To obtain \autoref{lem:base0}, consider \eqref{K110}.
			Let $ \overline{M} = M_1 $, $ \overline{X}_{m_0} = U_{m_0}(\hat{\mu}^{-2}\varepsilon_1) $, $ \overline{Y}_{m_0} = \{i_X(m_0)\} $, $ \widehat{Y}_{m_0} = L(X_{m_0}, Y_{m_0}) $, $ g_{m_0}(m,y) = \widehat{K}^1_{m_0} (m, i_{X}(m_0)) $, $ w(m) = u(m) $, $ u_{m_0}(m) = u(m) $, $ \jmath = \id $, $ \imath_0 = 0 $. Then apply \autoref{lem:rlocal} \eqref{r1}.

			For the final step of the proof of \autoref{lem:baseleaf}, we need to consider \eqref{K111}.
			Let $ \overline{M} = M_1 $, $ \overline{X}_{m_0} = U_{m_0}(\hat{\mu}^{-2}\varepsilon_1) $, $ \overline{Y}_{m_0} = X_{m_0} $, $ \widehat{Y}_{m_0} = L(X_{m_0}, Y_{m_0}) $, $ g_{m_0}(m,x) = \widehat{K}^1_{m_0} (m, x) $, $ w(m) = u(m) $, $ u_{m_0}(m) = u(m) $, $ v_{m_0}(m,x) =  \widehat{x}_{m_0} (m,x) $, $ \jmath = \id $, $ \imath = i_X $, $ \imath_0 = 0 $. Apply \autoref{lem:rlocal} \eqref{r1}.

			\item \label{lem4} ($ x \mapsto K_m(x) $) In the proof of \autoref{lem:holversheaf}, we have \eqref{Kxx}. Let $ \overline{M} = M $, $ \overline{X}_m = \{m\} $, $ \overline{Y}_m = X_m $, $ \widehat{Y}_{m} = L(T_mM, Y_m) $, $ g_m(m,x) = K_m(x) $, $ w(m) = u(m) $, $ u_m(m) = m $, $ v_m(m, x) = x_m(x) $, $ \imath = i_X $, $ \imath_0 = 0 $. Now apply \autoref{lem:rglobal} \eqref{g2}.

			\item \label{lem5} ($ m \mapsto K_m(x) $) To prove \autoref{lem:final}, we need to consider \eqref{k011}. Let $ \overline{M} = M_1 $, $ \overline{X}_{m_0} = U_{m_0}(\epsilon_*)  $, $ \overline{Y}_{m_0} = X_{m_0} $, $ \widehat{Y}_{m_0} = L(T_{m_0}M, Y_{m_0}) $, $ g_{m_0}(m,x) = \widehat{K}_{m_0} (m,x) $, $ w(m) = u(m) $, $ u_{m_0}(m) = u(m) $, $ v_{m_0}(m,x) = \widehat{x}_{m_0} (m,x) $, $ \jmath = \id $, $ \imath = i_X $, $ \imath_0 = 0 $. Apply \autoref{lem:rlocal} \eqref{r2} (and \autoref{moregeneral}).
	\end{asparaenum}
\end{rmk}

\begin{rmk}\label{rmk:simple}
	Assume there is an $ \varepsilon_1 > 0 $ (including $ \varepsilon_1 = \infty $) such that for any $ m_0 \in M_1 $, $ (U_{m_0}(\varepsilon_1), \varphi^{m_0}) \in \mathcal{A} $, $ (U_{m_0}(\varepsilon_1), \phi^{m_0}) \in \mathcal{B} $ (in (H2)) and $ u(U_{m_0}(\varepsilon_1)) \subset U_{u(m_0)} (\varepsilon_1) $ (in (H4a)).
	\begin{enumerate}[(a)]
		\item Under the same conditions as in \autoref{lem:sheaf} except that the spectral gap condition (c) is replaced by $ \lambda_s \lambda_u < 1 $, $ (\lambda^{\zeta - 1}_s \mu^{\gamma})^{*\alpha} \lambda_s \lambda_u < 1 $, where $ 0 < \alpha \leq 1 $, then \eqref{cc3} also holds for $ m_1 \in U_{m_0}(\varepsilon_1) $ and if $ \alpha < 1 $ under $ |m_1 - m_0|^\gamma |x|^{\zeta - 1} \leq \hat{r} $. (By \autoref{lem:rglobal} \eqref{g2}.)
		\item The conditions and conclusions in \autoref{lem:base0} and \autoref{lem:baseleaf} are the same but the conclusions hold for $ m_1 \in U_{m_0}(\varepsilon^*_1) $, where if $ \alpha < 1 $ then $ \varepsilon^*_1 $ is small, and otherwise $ \varepsilon^*_1 = \varepsilon_1 $. (By \autoref{lem:rglobal} \eqref{g1}.)
		\item In addition, if for all $ m_0 \in M_1 $, $ (U_{m_0}(\varepsilon_1), (\id \times D\chi_{m_0}^{-1}(\chi_{m_0}(\cdot))) ) \in \mathcal{M} $ (in (H1c)), then under the same conditions as in \autoref{lem:sheaf} except that in the spectral gap condition (c), $ \max\{\mu/\lambda_s, \mu \}^{*\alpha} \lambda_s \lambda_u \mu < 1 $ is replaced by $ (\mu/\lambda_s)^{*\alpha} \lambda_s \lambda_u \mu < 1 $, \autoref{lem:sheaf} (1) holds for $ m_1 \in U_{m_0}(\varepsilon^*_1) $, where if $ \alpha < 1 $ then $ \varepsilon^*_1 $ is small, and otherwise $ \varepsilon^*_1 = \varepsilon_1 $. (By \autoref{lem:rglobal} \eqref{g2}.)
	\end{enumerate}
\end{rmk}

%% file: app0.tex
\chapter{Appendix. Bundles and bundle maps with uniform properties: part II} \label{bundleII}

This appendix continues \autoref{bundle}. We only give the definitions to express our main ideas.

\subsection{Base-regularity of bundle maps, $ C^{0} $-uniform bundles: $ C^0 $ case}\label{uniformC}

We give a description of $ C^0 $ continuity respecting the base points for a bundle map $ f $ in the uniform sense: $ m \mapsto f_m(\cdot) $ is \emph{continuous} or \emph{uniformly continuous} in \emph{$ C^0 $-topology on bounded sets or in the whole space}.
Consider the following four different limits:
\begin{equation*}
\left\{
\begin{gathered}
\mathcal{L}^{base}_1: \sup_{m_0 \in M_1} \limsup_{m' \to m_0} \sup_{x \in A_{m_0}}, ~\mathcal{L}^{base}_{1,u}: \limsup_{\epsilon \to 0} \sup_{m_0 \in M_1} \sup_{m' \in U_{m_0}(\epsilon)} \sup_{x \in A_{m_0}},\\
\mathcal{L}^{base}_2: \sup_{m_0 \in M_1} \limsup_{m' \to m_0} \sup_{x \in X_{m_0}}, ~\mathcal{L}^{base}_{2,u}: \limsup_{\epsilon \to 0} \sup_{m_0 \in M_1} \sup_{m' \in U_{m_0}(\epsilon)} \sup_{x \in X_{m_0}},
\end{gathered}
\right.
\end{equation*}
where for $ \mathcal{L}^{base}_1 $, $ A_{m_0} $ is any bounded subset of $ X_{m_0} $, and for $ \mathcal{L}^{base}_{1,u} $, $ \sup_{m_0 \in M_1}\diam A_{m_0} < \infty $.

\begin{defi}\label{def:ucontinuity}
	Assume that $ X, Y $ are $ C^0 $ topology bundles over $ M, N $ with $ C^0 $ (\emph{open regular}) bundle atlases $ \mathcal{A}, \mathcal{B} $, respectively, and that $ u: M \to N $ is $ C^0 $ and $ M_1 \subset M $. Let $ f: X \to Y $ be a bundle map over $ u $.
	\begin{enumerate}[(a)]
		\item If every (regular) local representation $ \widehat{f}_{m_0} $ of $ f $ at $ m_0 \in M $ with respect to $ \mathcal{A}, \mathcal{B} $ satisfies $ \mathcal{L} |\widehat{f}_{m_0}(m',x) - \widehat{f}_{m_0}(m_0,x)| = 0 $, where $ \mathcal{L} = \mathcal{L}^{base}_1 $ (resp. $ \mathcal{L} = \mathcal{L}^{base}_2 $), then we say $ m \mapsto f_m $ is \emph{continuous around $ M_1 $} \emph{in $ C^0 $-topology on bounded sets} (resp. \emph{in the whole space}) with respect to $ \mathcal{A}, \mathcal{B} $.

		\item \label{uc:bb} In addition, let $ M, N $ be two locally metrizable spaces associated with open covers $ \{ U_{m}: m \in M \} $ and $ \{ V_{n}: n \in N \} $ respectively, and $ X, Y $ have uniform size trivializations on $ M_1 $ with respect to $ \mathcal{A}, \mathcal{B} $. Moreover, $ u $ is uniformly continuous around $ M_1 $. If $ \mathcal{L} |\widehat{f}_{m_0}(m',x) - \widehat{f}_{m_0}(m_0,x)| = 0 $, where $ \mathcal{L} = \mathcal{L}^{base}_{1, u} $ (resp. $ \mathcal{L} = \mathcal{L}^{base}_{2,u} $), then we say $ m \mapsto f_m $ is \emph{uniformly continuous around $ M_1 $ in $ C^0 $-topology on bounded sets} (resp. \emph{in the whole space}) with respect to $ \mathcal{A}, \mathcal{B} $.

		\item \label{uc:cc} (vector case) Under \eqref{uc:bb}, let $ X, Y, \mathcal{A}, \mathcal{B}, f $ be $ C^0 $ vector. We also say $ m \mapsto f_m : M \to L_{u}(X, Y) $ is \emph{uniformly continuous} around $ M_1 $ (with respect to $ \mathcal{A}, \mathcal{B} $) for short if $ m \mapsto f_m $ is uniformly continuous around $ M_1 $ in $ C^0 $-topology on bounded sets (with respect to $ \mathcal{A}, \mathcal{B} $).
	\end{enumerate}
	As usual, if $ M_1 = M $, the words `around $ M_1 $' will be omitted. Also, the words `with respect to $ \mathcal{A}, \mathcal{B} $' will be omitted if $ \mathcal{A}, \mathcal{B} $ are predetermined.
\end{defi}

In analogy with \autoref{def:C01Uniform} and \autoref{lem:uniformC01}, we have the following:

\begin{defi}[$ C^{0} $-uniform bundle]\label{def:C0Uniform}
	Recall \autoref{def:C01Uniform}. $ \mathcal{A} $ is said to be \emph{$ C^{0} $-uniform} around $ M_1 $ if the transition maps $ \varphi^{m_0, m_1} $, $ m_0, m_1 \in M_1 $, are equicontinuous on bounded-fiber sets, i.e.,
	\[
	\varphi^{m_0, m_1}_{m'}(x) \rightrightarrows \varphi^{m_0, m_1}_{m_0}(x) ~\text{as}~ m' \to m_0,
	\]
	uniformly for $ x \in A_{m_0} $, $ m_0, m_1 \in M_1 $, where $ \sup_{m_0 \in M_1} \dim A_{m_0} < \infty $. If $ \mathcal{A} $ is \emph{$ C^{0} $-uniform} around $ M_1 $ and $ X $ has \emph{$ \varepsilon $-almost uniform $ C^{0,1} $-fiber trivializations} on $ M_1 $ with respect to $ \mathcal{A} $, then we say $ X $ has $ \varepsilon $-almost $ C^{0} $-uniform trivializations on $ M_1 $ with respect to $ \mathcal{A} $; in addition, if $ M_1 = M $ and $ \varepsilon = 0 $, we also call $ X $ a \emph{$ C^{0} $-uniform bundle} (with respect to $ \mathcal{A} $).
\end{defi}

\begin{lem}\label{lem:uniformC0}
	Assume $ (X, M, \pi_1), (Y, N, \pi_2) $ have $ \varepsilon $-almost $ C^{0} $-uniform trivializations on $ M^{\epsilon_1}_1 $, $ u(M^{\epsilon_1}_1) $ with respect to preferred $ C^{0} $-uniform bundle atlases $ \mathcal{A}, \mathcal{B} $ respectively (see \autoref{def:C0Uniform}), where $ M_1 \subset M $ and $ M^{\epsilon_1}_1 $ is the $ \epsilon_1 $-neighborhood of $ M_1 $. Suppose $ f: X \to Y $ is uniformly continuous-fiber (resp. equicontinuous-fiber) (see \autoref{fiberR}) over a map $ u $ which is uniformly continuous around $ M^{\epsilon_1}_1 $ (see \autoref{hcontinuous}). If (i) $ m \mapsto f_m $ is continuous (resp. uniformly continuous) around $ M^{\epsilon_1}_1 $ in $ C^0 $-topology on bounded sets (see \autoref{def:ucontinuity}), and (ii) for each $ m \in M^{\epsilon_1}_1 $, $ f_m $ maps bounded subsets of $ X_{m} $ into bounded subsets of $ X_{u(m)} $ (resp. $ f $ maps bounded-fiber sets on $ M^{\epsilon_1}_1 $ into bounded-fiber sets on $ u(M^{\epsilon_1}_1) $),
	then the local representation $ \widehat{f}_{m_0, m'_0} $ of $ f $ at $ m_0, m'_0 $ (see \autoref{lem:uniformC01}) satisfies $ \mathcal{L} |\widehat{f}_{m_0, m'_0}(m',x) - \widehat{f}_{m_0, m'_0}(m_0,x)| = 0 $, where $ \mathcal{L} = \mathcal{L}^{base}_1 $ (resp. $ \mathcal{L} = \mathcal{L}^{base}_{1,u} $). Condition (ii) can be removed if the fibers of $ X, Y $ are length spaces (see \autoref{length} and \autoref{lem:BTB}).

	If $ m \mapsto f_m $ is continuous (resp. uniformly continuous) around $ M^{\epsilon_1}_1 $ in $ C^0 $-topology in the whole space (see \autoref{def:ucontinuity}), then $ \mathcal{L} |\widehat{f}_{m_0, m'_0}(m',x) - \widehat{f}_{m_0, m'_0}(m_0,x)| = 0 $, where $ \mathcal{L} = \mathcal{L}^{base}_2 $ (resp. $ \mathcal{L} = \mathcal{L}^{base}_{2,u} $).
\end{lem}

\begin{rmk}[Uniform properties of vector bundle maps and subbundles]\label{upVector}
	In \autoref{vsII} \eqref{vs:cc} and \autoref{def:ucontinuity} \eqref{uc:cc}, we in fact give a description of H\"older continuity and uniform $ C^0 $ continuity of a vector bundle map respecting base points. Let $ X, Y $ be vector bundles with vector bundle atlases $ \mathcal{A}, \mathcal{B} $. Here any vector bundle map $ L \in L(X, Y) $ over $ u $ can be considered as a section $ m \mapsto L_{m}: M \to L_{u}(X, Y) $. A particular situation is $ L = Du \in L_{u}(TM, TN) $ if $ u: M \to N $ is $ C^1 $ and $ M,N $ are $ C^1 $ Finsler manifolds.
	Note that $ m \mapsto L_m $ is $ C^0 $ in $ C^0 $-topology on bounded sets, which means that $ L $ is a $ C^0 $ vector bundle map if $ X, Y, \mathcal{A}, \mathcal{B} $ are $ C^0 $.

	Consider a special case. Suppose each $ \Pi_m \in L(X_m, X_m) $ is a projection (i.e. $ \Pi^2_{m} = \Pi_{m} $). Let $ X^c_m = R(\Pi_{m}) $, $ X^h_m = R(\id - \Pi_{m}) $. We say $ m \mapsto X^{\kappa}_m $, $ \kappa = c, h $, are \emph{uniformly $ C^0 $} (resp. \emph{uniformly (locally) $ C^{0,\theta} $}) around $ M_1 $ if $ m \mapsto \Pi_{m} $ is uniformly $ C^0 $ (resp. uniformly (locally) $ C^{0,\theta} $) around $ M_1 $. This is equivalent to saying that $ m \mapsto X^{\kappa}_m: M \to \mathbb{G}(X) $, $ \kappa = c, h $, are uniformly $ C^0 $ (resp. uniformly (locally) $ C^{0,\theta} $) around $ M_1 $, where $ \mathbb{G}(X) = \bigsqcup_{m} \mathbb{G}(X_m) $ is the \emph{Grassmann manifold} of $ X $ (see e.g. \cite{AMR88} for the definition of the Grassmann manifold of a Banach space). These concepts are frequently used in \autoref{foliations}.
	We also write $ X^{\kappa} = \bigsqcup_{m} X^{\kappa}_m $, $ \kappa = c, h $, which are subbundles of $ X $, and $ X^c \oplus X^h = X $.
\end{rmk}

\subsection{Base-regularity of $ C^1 $-fiber bundle maps}\label{tensorV}

For a $ C^1 $-fiber bundle map $ f: X \to Y $ over $ u $, $ D^{v} f \in L(\Upsilon_X^V,  \Upsilon_Y^V) $ (see \eqref{HVspace}) over $ f $; see \autoref{fiberR}. So a description of the H\"older continuity of $ m \mapsto K^1_m(x) $, where $ K^1 \in L_f(\Upsilon_X^V,  \Upsilon_Y^V) $ (see \autoref{vBundle}) and $ K^1_m(x) = K^1_{(m,x)} $, will give the H\"older continuity of $ m \mapsto Df_m(x) $. \emph{Assume the fibers of $ X, Y $ are Banach spaces (or open subsets of Banach spaces)} for simplicity (see also \autoref{rmk:general}).

Take a bundle atlas $ \mathcal{A} $ of the bundle $ X $. If each bundle chart belonging to $ \mathcal{A} $ is $ C^1 $-fiber, then there is a \emph{canonical bundle atlas} $ \mathcal{A}_1 $ (with respect to $ \mathcal{A} $) of $ \Upsilon_X^V $ (see \eqref{HVspace}), i.e.,
\[
U_{m_0} \times X_{m_0} \times X_{m_0} \to \Upsilon_X^V,~ ((m,x), y) \mapsto ((m,\varphi^{m_0}_m(x)), D\varphi^{m_0}_m(x)y),
\]
where $ (U_{m_0}, \varphi^{m_0}) \in \mathcal{A} $ at $ m_0 $, or $ ( \varphi^{m_0} (U_{m_0} \times X_{m_0}), (D^{v}\varphi^{m_0}) \circ (\varphi^{m_0})^{-1} ) \in \mathcal{A}_1 $. In particular, if $ X $ is a $ C^1 $ topology bundle (see \autoref{def:C1fiber}), then $ \Upsilon_X^V $ is a $ C^0 $ vector bundle over $ X $.
Thus, for a vector bundle map $ K^1 \in L_f(\Upsilon_X^V,  \Upsilon_Y^V) $, as the \emph{canonical local representation} of $ K^1 $ with respect to bundle atlases $ \mathcal{A}, \mathcal{B} $ one can take
\begin{multline*}
\widehat{K}^1_{m_0} (m, x) \triangleq ( D \phi^{u(m_0)}_{u(m)} )^{-1} ( f_m( \varphi^{m_0}_{m}(x) ) ) K^1_m( \varphi^{m_0}_{m}(x) ) D\varphi^{m_0}_{m}(x): \\
U_{m_0} \times X_{m_0} \to L (X_{m_0}, Y_{u(m_0)}),
\end{multline*}
where $ (U_{m_0}, \varphi^{m_0}) \in \mathcal{A} $ at $ m_0 $ and $ (V_{u(m_0)}, \phi^{u(m_0)}) \in \mathcal{B} $ at $ u(m_0) $.
So if $ \widehat{f}_{m_0} $ is a (regular) local representation of $ f $ at $ m_0 $ with respect to $ \mathcal{A}, \mathcal{B} $, as the canonical (regular) local representation $ \widehat{D}_x f_{m_0} $ of $ D^{v}f $ one can take
\begin{multline*}
\widehat{D}_x f_{m_0} (m,x)  \triangleq ( D \phi^{u(m_0)}_{u(m)} )^{-1} ( f_m( \varphi^{m_0}_{m}(x) ) ) Df_m( \varphi^{m_0}_{m}(x) ) D\varphi^{m_0}_{m}(x): \\
U_{m_0} \times X_{m_0} \to L (X_{m_0}, Y_{u(m_0)}),
\end{multline*}
i.e. $ \widehat{D}_x f_{m_0} (m,x) = D_x\widehat{f}_{m_0} (m,x) $. Now we have the following definition similar to \autoref{def:ucontinuity} and \autoref{vsII}.

\begin{defi}\label{def:uD}
	Let $ X, Y $ be $ C^1 $ topology bundles over $ M, N $ with (\emph{open regular}) $ C^1 $-fiber bundle atlases $ \mathcal{A}, \mathcal{B} $ respectively (see \autoref{def:C1fiber}). Let $ u: M \to N $ be $ C^0 $ and $ M_1 \subset M $. Let $ f: X \to Y $ be a $ C^1 $-fiber and $ C^0 $ bundle map over $ u $. \emph{Assume the fibers of $ X, Y $ are Banach spaces (or open subsets of Banach spaces)} for simplicity. Take a $ K^1 \in L_f(\Upsilon_X^V,  \Upsilon_Y^V) $. Let $ \widehat{f}_{m_0} $ be a (regular) local representation of $ f $, and $ \widehat{K}^1_{m_0} $ a canonical local representation $ K^1 $, at $ m_0 \in M $ with respect to $ \mathcal{A}, \mathcal{B} $.

	(a) If $ \mathcal{L} |\widehat{K}^1_{m_0}(m',x) - \widehat{K}^1_{m_0}(m_0,x)| = 0 $, where $ \mathcal{L} = \mathcal{L}^{base}_{1} $ (resp. $ \mathcal{L} = \mathcal{L}^{base}_{2} $), then we say $ m \mapsto K^1_m(\cdot) $ is \emph{continuous around $ M_1 $ in $ C^0 $-topology on bounded sets} (resp. \emph{in the whole space}) with respect to $ \mathcal{A}, \mathcal{B} $. In addition, if $ Df_{m}(x) = K^1_{m}(x) $ and $ \mathcal{L} |\widehat{f}_{m_0}(m',x) - \widehat{f}_{m_0}(m_0,x)| = 0 $, then we say $ m \mapsto f_m $ is \emph{continuous around $ M_1 $ in $ C^1 $-topology on bounded sets} (resp. \emph{in the whole space}) with respect to $ \mathcal{A}, \mathcal{B} $.

	Suppose that (i) $ M, N $ are locally metrizable spaces associated with open covers $ \{ U_{m}: m \in M \} $ and $ \{ V_{n}: n \in N \} $ respectively, (ii) $ \mathcal{A}, \mathcal{B} $ have uniform size domains (see \autoref{def:uSize}), and (iii) $ u $ is uniformly continuous around $ M_1 $ (see \autoref{hcontinuous}).

	(b) Under (i)--(iii), if $ \mathcal{L} |\widehat{K}^1_{m_0}(m',x) - \widehat{K}^1_{m_0}(m_0,x)| = 0 $, where $ \mathcal{L} = \mathcal{L}^{base}_{1, u} $ (resp. $ \mathcal{L} = \mathcal{L}^{base}_{2,u} $), then we say $ m \mapsto K^1_m(\cdot) $ is \emph{uniformly continuous around $ M_1 $ in $ C^0 $-topology on bounded sets} (resp. \emph{in the whole space}) with respect to $ \mathcal{A}, \mathcal{B} $. In addition, if $ Df_{m}(x) = K^1_{m}(x) $ and $ \mathcal{L} |\widehat{f}_{m_0}(m',x) - \widehat{f}_{m_0}(m_0,x)| = 0 $, then we say $ m \mapsto f_m $ is \emph{uniformly continuous around $ M_1 $ in $ C^1 $-topology on bounded sets} (resp. \emph{in the whole space}) with respect to $ \mathcal{A}, \mathcal{B} $.

	(c) Under (i), suppose $ \widehat{K}^1_{m_0} $ satisfies
	\[
	|\widehat{K}^1_{m_0} (m,x) - \widehat{K}^1_{m_0} (m_0,x)| \leq c^3_{m_0}(x) d^1_{m_0}(m,m_0)^{\theta}, ~ m \in U_{m_0}(\varepsilon_{m_0}),
	\]
	where $  c^3_{m_0}: X_{m_0} \to \mathbb{R}_+ $, $ m_0 \in M_1 $. Then we say \emph{$ K^1 $ depends in a (locally) $ C^{0,\theta} $ fashion on the base points around $ M_1 $}, or $ m \mapsto K^1_m(\cdot) $ is \emph{(locally) $ C^{0,\theta} $} around $ M_1 $, with respect to $ \mathcal{A} $, $ \mathcal{B} $.

	(d) Under (i)--(iii) (in particular, one can choose $ \varepsilon_{m_0} $ such that $ \inf_{m_0 \in M_1} \varepsilon_{m_0} > 0 $), if $ c^3_{m_0} $, $ m_0 \in M_1 $, are bounded on $ M_1 $ on bounded-fiber sets (see \autoref{def:boundedf}), then we say \emph{$ K^1 $ depends in a uniformly (locally) $ C^{0,\theta} $ fashion on the base points around $ M_1 $ {`uniformly for bounded-fiber sets'}}, or $ m \mapsto K^1_m(\cdot) $ is \emph{uniformly (locally) $ \theta $-H\"older around $ M_1 $ {`uniformly for bounded-fiber sets'}}, with respect to $ \mathcal{A}, \mathcal{B} $. In addition, if $ Df_m(x) = K^1_{m}(x) $ and $ f $ depends in a uniformly (locally) $ C^{0,\theta} $ fashion on the base points around $ M_1 $ uniformly for bounded-fiber sets, then we say $ m \mapsto f_{m} $ is \emph{uniformly (locally) H\"older around $ M_1 $ in $ C^{1} $-topology on bounded sets}.
	Usually, the words in `...' are omitted, especially when $ c^3_{m_0} $ is the class of functions in case (a) or (b) in \autoref{def:boundedf}.
	As usual, if $ M_1 = M $, the words `around $ M_1 $' will be omitted.
\end{defi}

\begin{defi}[$ C^{1} $-fiber-uniform bundle, $ C^{1,1} $-fiber-uniform]\label{def:C11uf}
	Recall \autoref{def:C01Uniform}. $ \mathcal{A} $ is said to be \emph{$ C^{1} $-fiber-uniform} around $ M_1 $ if the transition maps $ \varphi^{m_0, m_1} $, $ m_0, m_1 \in M_1 $, are $ C^1 $-fiber equicontinuous on bounded-fiber sets, i.e.,
	\[
	\varphi^{m_0, m_1}_{m'}(x) \rightrightarrows \varphi^{m_0, m_1}_{m_0}(x), ~D\varphi^{m_0, m_1}_{m'}(x) \rightrightarrows D\varphi^{m_0, m_1}_{m_0}(x), ~\text{as}~ m' \to m_0,
	\]
	uniformly for $ x \in A_{m_0} $, $ m_0, m_1 \in M_1 $, where $ \sup_{m_0 \in M_1} \dim A_{m_0} < \infty $.
	\emph{Assume the fibers of $ X, Y $ are Banach spaces (or open subsets of Banach spaces)} for simplicity (see also \autoref{rmk:general}).
	$ \mathcal{A} $ is said to be \emph{$ C^{1,1} $-fiber-uniform} around $ M_1 $ if it is $ C^{0,1} $-uniform around $ M_1 $ and the transition map $ \varphi^{m_0, m_1} $ satisfies further
	\[
	\sup_{x \in X_{m_0}}\lip D_x\varphi^{m_0,m_1}_{(\cdot)} (x)  \leq C,
	\]
	where $ C > 0 $ is a constant independent of $ m_0, m_1 \in M_1 $.
	If $ \mathcal{A} $ is {$ C^{1} $-fiber-uniform} (resp. {$ C^{1,1} $-fiber-uniform}) around $ M_1 $ and $ X $ has $ \varepsilon $-almost uniform $ C^{1,1} $-fiber trivializations (see \autoref{def:C11bundle}) on $ M_1 $ with respect to $ \mathcal{A} $, then we say $ X $ has \emph{$ \varepsilon $-almost $ C^{1} $-fiber-uniform trivializations} (resp. \emph{$ \varepsilon $-almost $ C^{1,1} $-fiber-uniform trivializations}) on $ M_1 $ with respect to $ \mathcal{A} $; in addition, if $ M_1 = M $ and $ X $ is uniform $ C^{1,1} $-fiber (see \autoref{def:C11bundle}), then we also call $ X $ a \emph{$ C^{1} $-fiber-uniform bundle} (resp. \emph{$ C^{1,1} $-fiber-uniform bundle}) with respect to $ \mathcal{A} $.
\end{defi}

Using this definition, one can obtain similar results on the regularity of $ (m,x) \mapsto Df_{m}(x) $ (and $ (m,x) \mapsto K^1_m(x) $) to \autoref{lem:uniformC0} and \autoref{lem:uniformC01}.

\subsection{Base-regularity of $ C^1 $ bundle maps, $ C^{1,1} $-uniform bundles, $ C^{1} $-uniform bundles}\label{tensorH}

As usual, for a $ C^1 $ manifold $ M $ and a $ C^1 $ atlas $ \mathcal{A}_0 $ of $ M $, $ \mathcal{A}_0 $ will induce a natural $ C^0 $ bundle atlas $ \mathcal{M} $ for $ TM $, called a \emph{canonical} bundle atlas of $ TM $. That is, for every local chart $ \chi_{m}: U_{m} \to T_{m}M $ with $ \chi_{m}(m) = 0 $ and $ D\chi_{m}(m) = \id $, let $ \psi^{m}_{m'}(x) = D\chi^{-1}_{m}(\chi_{m}(m'))x $; then $ (U_{m}, \psi^{m}) $ is a (canonical) bundle chart of $ TM $.

Let $ M, N $ be $ C^1 $ Finsler manifolds with $ C^1 $ (regular) atlases $ \mathcal{A}_0, \mathcal{B}_0 $.
Let $ (X, M, \pi_1) $ and $ (Y, N, \pi_2) $ be $ C^1 $ bundles (see \autoref{fiberbundle}) and $ f: X \to Y $ a $ C^1 $ bundle map over $ u $. Consider the H\"older continuity of $ Df $. As before, we will consider it in local representations. Take $ C^1 $ (regular) bundle atlases $ \mathcal{A}, \mathcal{B} $ of $ X, Y $, respectively.
Here we assume the charts belonging to $ \mathcal{A}_0 $ share the same domains as those belonging to $ \mathcal{A} $; similarly for $ \mathcal{B}_0, \mathcal{B} $.

For a precise presentation of $ Df $ respecting the base points, one needs the \emph{connection} structures in $ X, Y $ and the \emph{covariant derivative} of $ f $; see \autoref{connections} for a quick review of these notions. If $ f: X \to Y $ is $ C^1 $, now we have $ \nabla f \in L_{f}(\Upsilon^H_X, \Upsilon^V_Y) $ (see \eqref{HVspace} and \autoref{coderivative}). So a description of the H\"older continuity of $ m \mapsto K_m(x) $, where $ K \in L_f(\Upsilon_X^H,  \Upsilon_Y^V) $ (see \autoref{vBundle}) and $ K_m(x) = K_{(m,x)} $, will give the same for $ m \mapsto \nabla_{m}f_m(x) $.
\emph{Assume the fibers of $ X, Y $ are Banach spaces (or open subsets of Banach spaces)} for simplicity (see also \autoref{rmk:general}).

Take a $ C^0 $ (regular) bundle atlas $ \mathcal{M} $ for $ TM $ and a $ C^1 $ (regular) bundle atlas $ \mathcal{A} $ for a $ C^1 $ bundle $ X $ (over $ M $); in many cases, $ \mathcal{M} $ is the canonical bundle atlas induced from $ \mathcal{A}_0 $, the $ C^1 $ (regular) atlas of $ M $. Then there is a \emph{canonical bundle atlas} $ \mathcal{A}_2 $ (with respect to $ \mathcal{M} $ (and $ \mathcal{A} $)) for $ \Upsilon_X^H $ (see \eqref{HVspace}), i.e.,
\[
U_{m_0} \times X_{m_0} \times T_{m_0}M \to \Upsilon_X^H,~ ((m,x), v) \mapsto ((m,\varphi^{m_0}_m(x)), \psi^{m_0}_m v ),
\]
where $ (U_{m_0}, \varphi^{m_0}) \in \mathcal{A} $ at $ m_0 $ and $ (U_{m_0}, \psi^{m_0}) \in \mathcal{M} $ at $ m_0 $, or $ ( \varphi^{m_0} (U_{m_0} \times X_{m_0}), \id \times \psi^{m_0} ) \in \mathcal{A}_2 $. In particular, $ \Upsilon_X^H $ is a $ C^0 $ vector bundle over $ X $.
Thus, for a vector bundle map $ K \in L_f(\Upsilon_X^H,  \Upsilon_Y^V) $, as \emph{canonical local representation} of $ K $ with respect to bundle atlases $ \mathcal{A}, \mathcal{B}, \mathcal{M} $ one can take
\begin{multline*}
\widehat{K}_{m_0}(m,x)\triangleq(D\phi^{u(m_0)}_{u(m)})^{-1}(f_m(\varphi^{m_0}_{m}(x)))K_m(\varphi^{m_0}_{m}(x))\psi^{m_0}_m :\\
U_{m_0}\times X_{m_0}\to L(T_{m_0}M,Y_{u(m_0)}),
\end{multline*}
where $ (U_{m_0}, \varphi^{m_0}) \in \mathcal{A} $ at $ m_0 $ and $ (V_{u(m_0)}, \phi^{u(m_0)}) \in \mathcal{B} $ at $ u(m_0) $.
So as the canonical (regular) local representation $ \widehat{D}_m f_{m_0} $ of $ \nabla f $ (at $ m_0 $) one can take
\begin{multline*}
\widehat{D}_m f_{m_0} (m,x)  \triangleq ( D \phi^{u(m_0)}_{u(m)} )^{-1} ( f_m( \varphi^{m_0}_{m}(x) ) ) \nabla_m f_m( \varphi^{m_0}_{m}(x) ) \psi^{m_0}_{m}: \\
U_{m_0} \times X_{m_0} \to L (T_{m_0}M, Y_{u(m_0)}).
\end{multline*}

In analogy with \autoref{def:uD}, one can give precise definitions of the H\"older continuity of $ m \mapsto K_{m}(\cdot) $.

\begin{defi}\label{def:cd}
	Take a vector bundle map $ K \in L(\Upsilon_X^H,  \Upsilon_Y^V) $ over $ f $.
	We say $ K $ depends in a (locally) $ C^{0,\theta} $ fashion on the base points around $ M_1 $ with respect to $ \mathcal{A}, \mathcal{B}, \mathcal{M} $ if the canonical local representation $ \widehat{K}_{m_0} $ of $ K $ satisfies
	\[
	|\widehat{K}_{m_0}(m_1,x) - \widehat{K}_{m_0}(m_0,x)| \leq c^4_{m_0}(x) d(m_1, m_0), ~m_1 \in U_{m_0}(\varepsilon_{m_0}),
	\]
	where $ c^4_{m_0}: X_{m_0} \to \mathbb{R}_+ $, $ m_0 \in M_1 $, and $ d $ is the Finsler metric in $ M $.

	Assume $ u $ is uniformly continuous around $ M_1 $ (see \autoref{hcontinuous}).
	Suppose $ X, Y $ have \emph{uniform size trivializations} on $ M_1 $ with respect to $ \mathcal{A}, \mathcal{B} $ respectively (see \autoref{def:uSize}) (and so $ TM$ with respect to $ \mathcal{M} $), i.e., one can choose $ \varepsilon_{m_0} $ such that $ \inf_{m_0 \in M_1} \varepsilon_{m_0} > 0 $.
	If $ c^4_{m_0} $, $ m_0 \in M_1 $, are bounded on $ M_1 $ on bounded-fiber sets (see \autoref{def:boundedf}), then we say \emph{$ K $ depends in a uniformly (locally) $ C^{0,\theta} $ fashion on the base points around $ M_1 $ {`uniformly for bounded-fiber sets'}}, or $ m \mapsto K_m(\cdot) $ is \emph{uniformly (locally) $ \theta $-H\"older around $ M_1 $ {`uniformly for bounded-fiber sets'}}, with respect to $ \mathcal{A}, \mathcal{B}, \mathcal{M} $. Usually, the words in `...' are omitted, especially when $ c^4_{m_0} $ is the class of functions in case (a) or (b) in \autoref{def:boundedf}.
\end{defi}

Under the assumption that $ X, Y $ have $ \varepsilon $-almost $ C^{1} $-fiber-uniform (resp. $ C^{1,1} $-fiber-uniform) trivializations (see \autoref{def:C11uf}) and $ TM $ has $ \varepsilon $-almost $ C^{0} $-uniform (resp. $ C^{0,1} $-uniform) trivializations on $ M_1 $ (see \autoref{def:C0Uniform} and \autoref{def:C01Uniform}), one can obtain similar results on the regularity of $ (m,x) \mapsto \nabla_{m} f_{m}(x) $ (and $ (m,x) \mapsto K_{m}(x) $) to \autoref{lem:uniformC0} (resp. \autoref{lem:uniformC01}). Next let us focus on the $ C^{1,\gamma} $ continuity of $ f $. See \autoref{connections} for a quick review of some notions related to connections.

\begin{defi}[Uniform properties of connections] \label{lipcon}
	Let $ \mathcal{C}^X $ be a $ C^0 $ connection of a $ C^1 $ bundle $ X $ over a $ C^1 $ Finsler manifold $ M $. Let $ \mathcal{A} $ and $ \mathcal{M} $ be $ C^1 $ bundle atlases of $ X $ and $ TM $, respectively. Let $ M_1 \subset M $.
	Take $ (U_{m_0}(\epsilon'_{m_0}), \varphi^{m_0}) \in \mathcal{A} $ and $ (U_{m_0}(\epsilon'_{m_0}), \psi^{m_0}) \in \mathcal{M} $ at $ m_0 $. Let
	\begin{align*}
	\widehat{D}_m \varphi^{m_0} (m,x) & \triangleq (D\varphi^{m_0}_{m})^{-1} (\varphi^{m_0}_m(x)) \nabla_m \varphi^{m_0}_m(x) \psi^{m_0}_{m} \\
	& = \widehat{\varGamma}^{m_0}_{(m,x)} \psi^{m_0}_{m} :~ U_{m_0}(\epsilon'_{m_0}) \times X_{m_0} \to L(T_{m_0}M, X_{m_0}),
	\end{align*}
	where $ \widehat{\varGamma}^{m_0}_{(m,x)} $ is the Christoffel map in the bundle chart $ \varphi^{m_0} $ (see \autoref{def:connection}). We call $ \widehat{D}_m \varphi^{m_0} $ a local representation of $ \mathcal{C}^X $ at $ m_0 $ with respect to $ \mathcal{A}, \mathcal{M} $.
	We say the connection $ \mathcal{C}^X $ is \emph{locally Lipschitz} (with respect to $ \mathcal{A}, \mathcal{M} $) if
	\begin{gather*}
	\sup_{m \in U_{m_0}(\epsilon'_{m_0})}\lip_{x} \widehat{\varGamma}^{m_0}_{(m,\cdot)} < C_{m_0}, \\
	|\widehat{D}_{m} \varphi^{m_0} (m_1,x) - \widehat{D}_{m} \varphi^{m_0} (m_0,x) | \leq c^0_{m_0}(x) |m_1 - m_0|,~ m_1 \in U_{m_0}(\epsilon'_{m_0}),
	\end{gather*}
	where $ c^0_{m_0}: X_{m_0} \to \mathbb{R}_+ > 0 $ and $ C_{m_0} > 0 $. The connection $ \mathcal{C}^X $ is said to be \textbf{uniformly (locally) Lipschitz} around $ M_1 $ (with respect to $ \mathcal{A}, \mathcal{M} $) if $ \inf_{m_0 \in M_1} \epsilon'_{m_0} > 0 $, $ \sup_{m_0 \in M_1}C_{m_0} < \infty $ and $ c^0_{m_0} $ are bounded on $ M_1 $ on bounded-fiber sets (see \autoref{def:boundedf}). We say $ \mathcal{C}^X $ is \emph{uniformly $ C^0 $} around $ M_1 $ (with respect to $ \mathcal{A}, \mathcal{M} $) if $ \inf_{m_0 \in M_1} \epsilon'_{m_0} > 0 $, and
	\[
	\widehat{D}_{m} \varphi^{m_0} (m',x') \rightrightarrows \widehat{D}_{m} \varphi^{m_0} (m_0,x) ~\text{as}~ (m',x') \to (m_0,x),
	\]
	uniformly for $ x \in A_{m_0} $, $ m_0 \in M_1 $, where $ \sup_{m_0 \in M_1} \dim A_{m_0} < \infty $.
\end{defi}

\begin{defi}[$ C^1 $-uniform bundle, $ C^{1,1} $-uniform bundle]\label{def:C11Uniform}
	Let $ M $ be a $ C^1 $ Finsler manifold with a $ C^1 $ (regular) atlas $ \mathcal{A}_0 $. Let $ \mathcal{M} $ be the \emph{canonical} bundle atlas of $ TM $ induced by $ \mathcal{A}_0 $. Let $ X $ be a $ C^1 $ bundle over $ M $ with $ C^1 $ (regular) bundle atlas $ \mathcal{A} $ and $ C^0 $ connection $ \mathcal{C}^{X} $. Let $ M_1 \subset M $. The charts belonging to $ \mathcal{A} $ and $ \mathcal{A}_0 $ share the same domains and have uniform size domains on $ M_1 $. Suppose that $ TM $ has \emph{$ \varepsilon $-almost $ C^{0} $-uniform trivializations} on $ M_1 $ with respect to $ \mathcal{M} $ (see \autoref{def:C0Uniform}).
	Then $ \mathcal{A} $ (with $ \mathcal{M} $) is said to be \emph{$ C^{1} $-uniform} around $ M_1 $, if the transition maps $ \varphi^{m_0, m_1} $ (with respect to $ \mathcal{A} $), $ m_0, m_1 \in M_1 $, are $ C^1 $ equicontinuous on bounded-fiber sets, i.e.
	\[
	D_{m}\varphi^{m_0, m_1}_{m'}(x') \rightrightarrows D_{m}\varphi^{m_0, m_1}_{m_0}(x), ~D_x\varphi^{m_0, m_1}_{m'}(x) \rightrightarrows D_x\varphi^{m_0, m_1}_{m_0}(x), ~\text{as}~ (m',x') \to (m_0,x),
	\]
	uniformly for $ x \in A_{m_0} $, $ m_0, m_1 \in M_1 $, where $ \sup_{m_0 \in M_1} \dim A_{m_0} < \infty $. If $ \mathcal{A} $ (with $ \mathcal{M} $) is \emph{$ C^{1} $-uniform} around $ M_1 $, $ X $ has \emph{$ \varepsilon $-almost uniform $ C^{1,1} $-fiber trivializations} on $ M_1 $ with respect to $ \mathcal{A} $ and $ \mathcal{C}^{X} $ is uniformly $ C^0 $ around $ M_1 $ (\autoref{lipcon}), then we say $ X $ has \textbf{$ \varepsilon $-almost $ C^{1} $-uniform trivializations} on $ M_1 $ with respect to $ \mathcal{A}, \mathcal{M} $; in addition, if $ M_1 = M $ and $ \varepsilon = 0 $, we also call $ X $ a \emph{$ C^{1} $-uniform bundle} (with respect to $ \mathcal{A}, \mathcal{M} $).

	Assume that (i) $ TM $ has \emph{$ \varepsilon $-almost $ C^{0,1} $-uniform trivializations} (see \autoref{def:C01Uniform}) on $ M_1 $ with respect to $ \mathcal{M} $ (where in \autoref{def:C01Uniform} we need $ c^1_{m_0, m_1} $ thereof satisfies $ c^1_{m_0, m_1}(x) \leq C|x| $ for some constant $ C > 0 $ independent of $ m_0, m_1 $), and (ii) \emph{the fibers of $ X, Y $ are Banach spaces (or open subsets of Banach spaces)} for simplicity (see also \autoref{rmk:general}). We say $ \mathcal{A} $ (with $ \mathcal{M} $) is $ C^{1,1} $-uniform around $ M_1 $ if $ X $ has $ \varepsilon $-almost $ C^{1,1} $-fiber-uniform trivializations (see \autoref{def:C11uf}) on $ M_1 $ with respect to $ \mathcal{A} $, and for the transition map $ \varphi^{m_0,m_1} $ with respect to $ \mathcal{A} $ (see \autoref{def:C01Uniform}), $ m_0, m_1 \in M_1 $, and the local chart $ \chi_{m_0} \in \mathcal{A}_0 $ at $ m_0 $, they satisfy
	\[
	\lip_x (D_m\varphi^{m_0,m_1}_{m} (x)) \leq C, ~\lip_{m} (D_m\varphi^{m_0,m_1}_{m} (x)) D\chi^{-1}_{m_0}(\chi_{m_0}(m))\leq c^5_{m_0, m_1}(x),
	\]
	where $ c^5_{m_0, m_1}(x) $, $ m_1 \in M_1 $, are bounded on $ M_1 $ on bounded-fiber sets (see \autoref{def:boundedf}) and $ C > 0 $ is a constant independent of $ m_0, m_1 $.
	We say $ X $ has \textbf{$ \varepsilon $-almost $ C^{1,1} $-uniform trivializations} on $ M_1 $ with respect to $ \mathcal{A}, \mathcal{M} $ if $ \mathcal{A} $ with $ \mathcal{M} $ is $ C^{1,1} $-uniform and $ \mathcal{C}^{X} $ is uniformly (locally) Lipschitz around $ M_1 $ (see \autoref{lipcon}); in addition, if $ M_1 = M $ and $ \varepsilon = 0 $, then we call $ X $ a \emph{$ C^{1,1} $-uniform bundle} (with respect to $ \mathcal{A}, \mathcal{M} $).
\end{defi}

\begin{rmk}[Vector bundle]\label{rmk:vector}
	For the vector bundle case, we require the bundle atlas $ \mathcal{A} $ is vector, and $ c^1_{m_0, m_1} $ (in \autoref{def:C01Uniform}) and $ c^5_{m_0, m_1} $ (in \autoref{def:C11Uniform}) satisfy $ c^1_{m_0, m_1}(x) \leq C|x| $ and $ c^5_{m_0, m_1}(x) \leq C|x| $ where $ C > 0 $ independent of $ m_0, m_1 $. Now we have different classes of vector bundles with local uniform property (of the base space), named \emph{$ C^0 $-uniform} (\autoref{def:C0Uniform}), \emph{$ C^{0,1} $-uniform} (\autoref{def:C01Uniform}), \emph{$ C^{1} $-uniform} (\autoref{def:C11Uniform}), and \emph{$ C^{1,1} $-uniform} (\autoref{def:C11Uniform}) vector bundle.
	\[
	\begin{CD}
	\text{$ C^{1,1} $-uniform vector bundle} @>>> \text{$ C^{1} $-uniform vector bundle}\\
	@VVV  @VVV \\
	\text{$ C^{0,1} $-uniform vector bundle} @>>> \text{$ C^{0} $-uniform vector bundle}
	\end{CD}
	\]
	Moreover, for the local uniform property of the fiber space, there is a class of vector bundles called uniform $ C^{0,1} $-fiber bundles (\autoref{def:lipbundle}). Note that in the vector bundle setting, uniform $ C^{0,1} $-fiber = uniform $ C^{1,1} $-fiber, $ C^1 $-fiber-uniform (\autoref{def:C11uf}) = $ C^{0} $-uniform, and $ C^{1,1} $-fiber-uniform (\autoref{def:C11uf}) = $ C^{0,1} $-uniform. The notion of $ C^{1} $-uniform vector bundle is the same as in \cite[Chapter 6]{HPS77}. The above notions are related to vector bundles having bounded geometry (see \autoref{vectorB}).

	More generally, if the bundle $ X $ has a $ 0 $-section $ i $ with respect to the bundle atlas $ \mathcal{A} $ (see \autoref{0-section}), we also require $ c^1_{m_0, m_1} $ (in \autoref{def:C01Uniform}) and $ c^5_{m_0, m_1} $ (in \autoref{def:C11Uniform}) satisfy $ c^1_{m_0, m_1}(x) \leq C|x| $ and $ c^5_{m_0, m_1}(x) \leq C|x| $ where $ C > 0 $ independent of $ m_0, m_1 $. Here $ |x| = d(x, i(m)) $ if $ x \in X_m $.
\end{rmk}

Here note that for the (regular) local representation $ \widehat{f}_{m_0} $ of $ f $ at $ m_0 $ with respect to $ \mathcal{A}, \mathcal{B} $, one has
\begin{multline*}
D_{m} \widehat{f}_{m_0} (m,x)  = \nabla_{u(m)} (\phi^{u(m_0)}_{u(m)})^{-1}(x'') Du(m) \\
+ D(\phi^{u(m_0)}_{u(m)})^{-1}(x'') \left\lbrace  \nabla_{m} f_{m}(x') + Df_{m}(x') \nabla_{m}\varphi^{m_0}_{m} (x) \right\rbrace,
\end{multline*}
where $ x' = \varphi^{m_0}_m(x) $, $ x'' = f_m(x') $.
Once we know the (uniform) Lipschitz continuity of the connections in $ X, Y $ (see \autoref{lipcon}), $ m \mapsto Du(m) $ (see \autoref{def:manifoldMap}), $ m \mapsto f_m(\cdot) $ and $ x \mapsto \nabla_{m} f_m(x), Df_m(x) $, the H\"older continuity of $ m \mapsto D_{m} \widehat{f}_{m_0} (m,x) $ is equivalent to the H\"older continuity of $ m \mapsto \nabla_{m} f_{m}(x) $.
We think that a natural setting for discussing the uniform $ C^{1,\gamma} $ continuity of a bundle map $ f $ in the \emph{classical} sense is that the bundles are $ C^{1,1} $-uniform.
\emph{Through the study of $ (m,x) \mapsto Df_m(x) $ and $ (m,x) \mapsto \nabla_{m} f_{m}(x) $ to understand the properties of $ Df $}, this is an important idea we used in \autoref{stateRegularity}.

\subsection{Summary and extension}\label{summary}

Let $ X, Y $ be bundles with base spaces $ M, N $, respectively. Assume the fibers of $ X, Y $ are Banach spaces (or open subsets of Banach spaces) for simplicity (see also \autoref{rmk:general}). Take (regular) bundle atlases $ \mathcal{A}, \mathcal{B} $ of $ X, Y $, respectively.

Let us consider a special type of vector bundle over $ X $. Suppose $ \Theta $ is a vector bundle over $ X $ with a (regular) vector bundle atlas $ \mathcal{A}_1 $. The fibers of $ \Theta $ are $ \Theta_{(m_0, x_0)} = \Theta_{m_0} $, $ (m_0, x_0) \in X $.
Every bundle chart at $ (m_0, x_0) $ belonging to the bundle atlas $ \mathcal{A}_1 $ of $ \Theta $ satisfies
\[
U_{m_0} \times X_{m_0} \times \Theta_{m_0} \to \Theta, ~((m,x), v) \mapsto ((m, \varphi^{m_0}_mx), \psi^{m_0}_{m}(x) v),
\]
and $ \psi^{m_0}_{m_0}(x) = \id $ for all $ x \in X_{m_0} $, where $ (U_{m_0}, \varphi^{m_0}) \in \mathcal{A} $ at $ m_0 \in M $. That is, $ ( \varphi^{m_0}(U_{m_0} \times X_{m_0}), (\id \times \psi^{m_0}) \circ ( (\varphi^{m_0})^{-1}, \id) ) \in \mathcal{A}_1 $, where $ \id \times \psi^{m_0} (m,x, v) = ((m,x), \psi^{m_0}_m(x) v) $. We call $ \Theta $ a \textbf{pre-tensor bundle} over $ X $ (with respect to $ \mathcal{A}_1 $). For example $ \Upsilon^H_X $ and $ \Upsilon^V_X $ are pre-tensor bundles (see \autoref{tensorH} and \autoref{tensorV}). If $ \Theta_1, \Theta_2 $ are pre-tensor bundles over $ X $, then so is $ \Theta_1 \times \Theta_2 $.

Take a bundle map $ f: X \to Y $ over $ u $.
Let $ \Theta, \Omega $ be pre-tensor bundles over $ X, Y $ with respect to bundle atlases $ \mathcal{A}_1, \mathcal{B}_1 $, respectively.
Take a vector bundle map $ E \in L(\Theta, \Omega) $ over $ f $; we write $ E_{m}(x) = E_{(m,x)} \in L(\Theta_{m}, \Omega_{u(m)}) $ and consider it as
\[
(m,x) \mapsto E_{m}(x): X \to L_{f}(\Theta, \Omega).
\]
\begin{enumerate}[$ \bullet $]
	\item (fiber-regularity) Similar to the bundle map case, the same terms as in \autoref{fiberR}, such as $ C^{0} $-fiber, uniformly continuous-fiber, equicontinuous-fiber, $ C^{k,\theta} $-fiber, uniform $ C^{k,\theta} $-fiber, uniformly locally $ C^{k,\theta} $-fiber, fiber derivative (i.e. $ D^v E $), etc., can be used for the $ C^{k,\theta} $-fiber continuity of $ x \mapsto E_{m}(x): X_{m} \to L(\Theta_{m}, \Omega_{u(m)}) $.
\end{enumerate}

As the (regular) canonical local representation of $ E $ at $ m_0 $ with respect to $ \mathcal{A}_1, \mathcal{B}_1 $ one can take
\[
\widehat{E}_{m_0} (m, x) v = (\phi^{u(m_0)}_{u(m)})^{-1}(f_{m}(\varphi^{m_0}_{m}(x))) E_{m}( \varphi^{m_0}_{m}(x) ) \psi^{m_0}_{m} (x) v,
\]
where $ \phi^{u(m_0)} \in \mathcal{B}_1 $, $ \varphi^{m_0} \in \mathcal{A} $ and $ \psi^{m_0} \in \mathcal{A}_1 $.

\begin{enumerate}[$ \bullet $]
	\item (base-regularity: H\"older case) Similar to \autoref{vsII}, \autoref{def:uD} and \autoref{def:cd}, one can define the H\"older continuity of $ m \mapsto E_{m}(\cdot) $ by using the canonical local representation $ \widehat{E}_{m_0} $.
	We will use the following expressions: $ E $ depends in a (locally) $ C^{0,\theta} $ fashion on the base points around $ M_1 $, or $ m \mapsto E_m(\cdot) $ is (locally) $ C^{0,\theta} $ around $ M_1 $ with respect to $ \mathcal{A}_1, \mathcal{B}_2 $, and $ E $ depends in a uniformly (locally) $ C^{0,\theta} $ fashion on the base points around $ M_1 $ `uniformly for bounded-fiber sets', or $ m \mapsto E_m(\cdot) $ is uniformly (locally) $ \theta $-H\"older around $ M_1 $ {`uniformly for bounded-fiber sets'} with respect to $ \mathcal{A}_1, \mathcal{B}_1 $.

	\item (base-regularity: $ C^0 $ case) Also, one can discuss the $ C^0 $ continuity of $ m \mapsto E_{m}(\cdot) $ as in \autoref{uniformC} and \autoref{tensorV}; the following expressions will be used: $ m \mapsto E_m(\cdot) $ is continuous (resp. uniformly continuous) around $ M_1 $ in $ C^0 $-topology (resp. $ C^1 $-topology) on bounded sets (resp. in the whole space) with respect to $ \mathcal{A}_1, \mathcal{B}_1 $.
\end{enumerate}

Until now, we have given a way of describing the (uniform) H\"older continuity and $ C^{0} $ continuity of $ f $ and $ Df $ in appropriate bundles $ X, Y $. For the base-regularity, this is done by using regular local representations of $ f $ with respect to preferred bundle atlases. Although the regularity of regular local representations can yield the classical description of regularity in some sense under the stronger regularity of the bundles (see \autoref{lem:uniformC0} and \autoref{lem:uniformC01}), in fact it does not depend on the choice of bundle atlases with the same properties (i.e. \emph{equivalent} bundle atlases); this is simple but we do not give the details and refer the reader to \cite{Ama15} for a discussion in the manifold setting.

Using pre-tensor bundles, one can further study higher order derivatives of $ f $. For example, assume $ X, Y $ are paracompact $ C^{k} $ bundles with $ C^{k-1} $ connections $ \mathcal{C}^X, \mathcal{C}^Y $, respectively, and $ f \in C^{k}(X, Y) $. Now $ TX \cong \Upsilon_X = \Upsilon^H_X \times \Upsilon^V_X $,
so we can write $ Df = (\nabla f, D^v f) $ and
\[
(m,x) \mapsto Df(m,x) = (\nabla_{m} f_{m}(x), Df_m(x)): X \to L_f(TX, \Upsilon^V_Y) \cong L_f(\Upsilon_X, \Upsilon^V_Y),
\]
where we ignore the base map of $ f $.
Here $ L_f(\Upsilon_X, \Upsilon^V_Y) $ is a $ C^{k-1} $ pre-tensor bundle over $ X $ (and so we can give it a connection). Taking the derivative of $ Df $, we have
\[
D^2 f =
\left(  \begin{matrix}
\nabla \nabla f & D^v \nabla f  \\
\nabla D^v f & D^v D^v f
\end{matrix} \right) \in L_{f}(\Upsilon_X \times \Upsilon_X, \Upsilon^V_Y).
\]
Similarly, $ D^i f \in L^{i}_{f}(\Upsilon_X, \Upsilon^{V}_Y) $, $ i = 1,2,\ldots,k $. Here $ L^k_{u}(Z_1, Z_2) \triangleq L_{u}(\underbrace{Z_1 \times \cdots \times Z_1}_{k}, Z_2) $ if $ Z_i $ is a vector bundle over $ M_i $, $ i = 1,2 $, and $ u: M_1 \to M_2 $.
Although our regularity results only concern $ C^{k,\alpha} $ continuity of $ f $, $ k = 0, 1 $, $ 0 \leq \alpha \leq 1 $, by using almost the same strategy in \autoref{stateRegularity} with induction and the above idea, one can obtain higher order smoothness of $ f $ (but the statements are more complicated due to the non-triviality of $ X, Y $); the details are omitted in this paper (see also \cite{HPS77}).

\begin{rmk}\label{rmk:general}
	In order to discuss the $ C^{k,\theta} $-fiber continuity ($ k \geq 1 $) of a bundle map, in general, we assume the fibers of $X, Y$ are Banach spaces (or open subsets of Banach spaces), i.e., each fiber can be represented by a single chart. In most cases, this is sufficient for us to apply our regularity results. However, under the above preliminaries, one can further extend it to the fibers being general (connected) Finsler manifolds with local uniform properties, as we do for the extension of the base space in more general settings. A model for these Finsler manifolds is e.g. Riemannian manifolds having bounded geometry (see e.g. \autoref{defi:bounded}), or more general, Banach-manifold-like manifolds in \autoref{def:C11um}. Note that in this case, if $ k \geq 2 $, the $ k $th order (covariant) derivative of $ f_m(\cdot) $ is considered as an element of the pre-tensor bundle over $ X_m $, $ m \in M $. (A minor modification of the definition of pre-tensor bundle is needed, left to the readers.)
\end{rmk}

\chapter{Appendix. Examples of manifolds and bundles} \label{examples}

In this appendix, we give some examples of manifolds and bundles satisfying our (uniformity) hypotheses: immersed manifolds in Banach spaces studied in \cite{BLZ99, BLZ08} and Riemannian manifolds having bounded geometry introduced in e.g. \cite{Ama15, Eic91}.

\subsection{Immersed manifolds}\label{immeresedM}

\begin{exa}[Trivial example]
	Let $ M $ be any open subset of a Banach space and $ M_1 \subset \{ m \in M: d(m, \partial M) > \epsilon \} $ ($ \epsilon > 0 $). Then $ M $ is trivially $ C^{1,1} $-uniform around $ M_1 $ (see \autoref{def:C11um}).
\end{exa}

\begin{exa}[Immersed manifolds in Banach spaces I] \label{immersionI}
	Let $ M \subset X $ be an immersed manifold where $ X $ is a Banach space for simplicity. That is, there are a $ C^1 $ manifold $ \widehat{M} $ and an immersion $ \phi : \widehat{M} \to X $ with $ \phi(\widehat{M}) = M $; the latter means that $ \phi $ is $ C^1 $ with $ T_m \phi $ injective and $ R (T_m \phi) $ closed splitting in $ X $ (i.e. $ R (T_m \phi) \oplus X_m = X $ with $ R (T_m \phi) $ and $ X_m $ closed) for all $ m \in \widehat{M} $; see e.g. \cite{AMR88}. (We do not assume $ \phi $ is injective here.) First assume $ \widehat{M} $ is boundaryless. Let $ \Pi^c_m $ be the projection associated to $ R (T_m \phi) \oplus X_m = X $ with $ R(\Pi^c_m) = R (T_m \phi) $. Let $ \Lambda_{m} = \phi^{-1}(m) $ and $ X^c_m = R (T_m \phi) $, $ m \in M $. Let us represent $ \widehat{M} $ in $ X $. Take $ U_{m_c}(\epsilon_{m}) $ to be the component of the set $ \phi^{-1} (B_m(\epsilon_m)) $ containing $ m_c $, where $ m_c \in \Lambda_m $ and $ B_m(\epsilon_{m}) = \{ m': |m' - m| < \epsilon_{m} \} $. Let
	\begin{equation}\label{injectivity}
	\sup_{ \substack{ m_1 , m_2 \in U_{m_c}(\epsilon_{m}) \\ m_1 \neq m_2 } }\frac{|\phi(m_1) - \phi(m_2) - \Pi^c_m(\phi(m_1) - \phi(m_2))|}{|\phi(m_1) - \phi(m_2)|} \leq r_{m}.
	\end{equation}
	Set $ U^c_{m_c}(\epsilon_{m}) = \phi(U_{m_c}(\epsilon_{m})) $ with the metric $ d(m_1, m_2) = d_{m_c}(m_1, m_2) = |m_1 - m_2| $. Since $ r_{m} \to 0 $ as $ \epsilon_{m} \to 0 $, we know
	\begin{equation}\label{chart}
	\chi_{m_c}: U^c_{m_c}(\epsilon_{m}) \to X^c_{m}, m' \mapsto \Pi^c_m (m' - m),
	\end{equation}
	is a $ C^1 $ diffeomorphism and $ X^c_{m}(\epsilon'_{m}) \subset \chi_{m_c} (U^c_{m_c}(\epsilon_{m})) $ for some small $ \epsilon'_{m} > 0 $. Now $ \{ U^c_{m_c}(\epsilon_{m}): m_c \in \Lambda_{m}, m \in M \} $ is an open cover of $ M $ (endowed with immersed topology), and so $ M $ is a \emph{uniformly locally metrizable space} (see \autoref{def:ulms}). Note that if $ m_i \in U_{m^i_c}(\epsilon_{m}) $, $ \phi(m^i_c) = m $, $ i = 1,2 $, in general there is no metric between $ m_1, m_2 $ though in $ X $, $ |m_1 - m_2| < \epsilon_{m} $. Then $ \{ (U^c_{m_c}(\epsilon_{m}), \chi_{m_c}) \} $ gives a $ C^1 $ atlas for $ M $. Let $ \mathcal{M} $ be the canonical bundle atlas of $ TM $ induced by $ M $.

	Also, a set $ M $ in $ X $ might have different immersed representations; see e.g. Figure 1, where (a) (b) are non-injectively immersed representations with locally uniform size neighborhoods on themselves (see \autoref{uniformSize}), and (c) is an injectively immersed representation but in this case it does not have locally uniform size neighborhoods on itself.
	If $ \widehat{M} $ has boundary $ \partial \widehat{M} $, one can consider $ \widehat{M} / \partial \widehat{M} $ and $ \partial \widehat{M} $ separately. But unlike the boundaryless case, $ M $ would look very non-smooth (e.g. the closure of a homoclinic orbit).
	Also, note that for a map $ u: M \to M $, even if $ u $ is Lipschitz considered as a map in $ X $, it might not be Lipschitz when $ M $ is endowed with the immersed topology.
\end{exa}

\begin{figure}
	\centering
	\subfigure[original set $ M \subset \mathbb{R}^2 $]{
		\label{fig:f1}
		\includegraphics[height = 4cm]{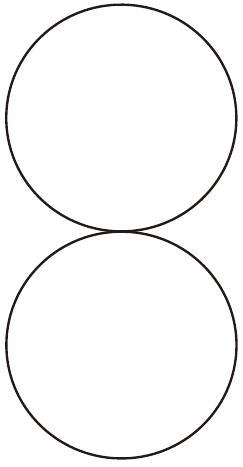}
	}
	\hspace{0.5cm}
	\subfigure[immersed representation $ \widehat{M}_1 $ ($ \cong \mathbb{S}^1 \cup \mathbb{S}^1 $) $ \subset \mathbb{R}^2 $]{
		\includegraphics[height=4cm]{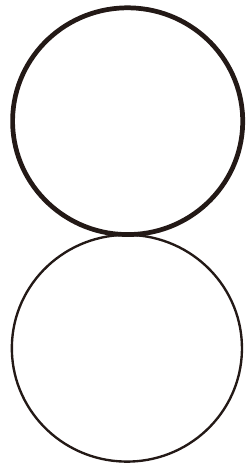}
	}
	\hspace{0.5cm}
	\subfigure[immersed representation $ \widehat{M}_2 $ ($ \cong \mathbb{S}^1  $) $ \subset \mathbb{R}^3 $]{
		\includegraphics[height=4cm]{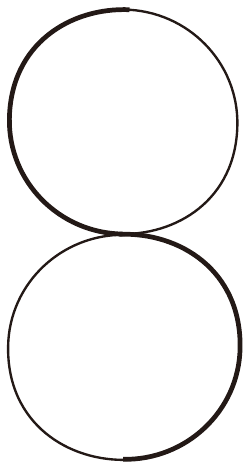}
	}
	\hspace{0.5cm}
	\subfigure[immersed representation $ \widehat{M}_3 $ ($ \cong \mathbb{R}  $) $ \subset \mathbb{R}^2 $]{
		\includegraphics[height=4cm]{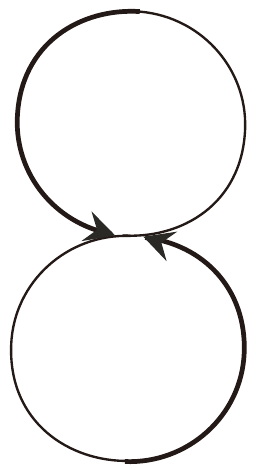}
	}
	\label{fig:f}
	\caption{Different immersed representations of $ M $}
\end{figure}

\begin{exa}[Immersed manifolds in Banach spaces II] \label{immersionII}
	We continue \autoref{immersionI}. We will identify $ M = \widehat{M} $ if we endow $ M $ with the immersed smooth structure. Note that the norms $ |\cdot|_{m_c} = |\cdot| $ in $ X^c_m = T_{m_c}M $, ($ m_c \in \Lambda_m $), $ m \in M $, give the natural Finsler structure for $ TM $. Furthermore, if $ m' \mapsto \Pi^c_{m'} $, $ m' \in U^c_{m_c}(\epsilon_{m}) $, is continuous for each $ m_c \in \Lambda_m $, $ m \in M $, then $ M $ is a Finsler manifold in the sense of Palais (see \autoref{finsler} \eqref{finD}); in particular now $ M $ satisfies the assumption (H1b) on \autopageref{baseSpaceAA}.

	Let $ M_1 \subset M $. Consider the following assumptions.
	\begin{enumerate}[($ \bullet $1)]
		\item Let $ \inf_{m_0 \in M_1} \epsilon'_{m_0} > 0  $ so that \eqref{chart} holds. This is characterized in \cite{BLZ08} as follows: There is an $ r_0 > 0 $ such that (a) $ \phi(\overline{U_{m_c}(r_0)}) $ is closed in $ X $ for all $ m_c \in \Lambda_{m_0} $, $ m_0 \in M_1 $, and (b) $ \sup_{m_0 \in M_1} r_{m_0} < 1/2 $ (see \eqref{injectivity}) if $ \sup_{m_0 \in M_1}\epsilon_{m_0} < r_0 $. That is, $ M $ has locally uniform size neighborhoods around $ M_1 $ (see \autoref{uniformSize}). See also \autoref{lem:usd}.

		\item (almost uniform continuity case) Let ($ \bullet $1) hold. Assume that there are an $ \varepsilon > 0 $ and a $ \delta > 0 $ such that for each $ m_c \in \Lambda_{m_0} $, $ m_0 \in M_1 $, $ |\Pi^c_{m_1} - \Pi^c_{m_2}| < \varepsilon $ provided $ m_1, m_2 \in U^c_{m_c}(\delta) $. Now $ TM $ has $ \varepsilon $-almost uniform $ C^{0,1} $-fiber trivializations on $ M_1 $ with respect to $ \mathcal{M} $ (see \autoref{def:lipbundle}), and $ M $ is $ C^{0,1} $-uniform around $ M_1 $ (see \autoref{def:C11um}).

		\item (Lipschitz continuity case) Let ($ \bullet $1) hold. Assume $ m \mapsto \Pi^c_m $ is uniformly Lipschitz in the immersed topology in the following sense: There is a constant $ C_0 > 0 $ such that for each $ m_c \in \Lambda_{m_0} $, $ m_0 \in M_1 $, $ |\Pi^c_{m_1} - \Pi^c_{m_2}| \leq C_0|m_1 - m_2| $, $ m_1, m_2 \in U^c_{m_c}(\delta) $ (i.e. the assumption (H2) in \cite{BLZ08}). Then $ TM $ has $ 0 $-almost $ C^{0,1} $-uniform trivializations on $ M_1 $ (see \autoref{def:C01Uniform}), and $ M $ is $ C^{1,1} $-uniform around $ M_1 $ (see \autoref{def:C11um}).
	\end{enumerate}

	For instance, the immersed representations (b) (c) in Figure 1 give $ M $ $ C^{1,1} $-uniformity but the immersed representation (d) does not.
\end{exa}

\begin{exa}[$ C^2 $ compact submanifold]
	We assume that in \autoref{immersionI} $ \phi $ is a $ C^2 $ injective map and $ M $ is a compact subset of $ X $, i.e. $ M $ is a $ C^2 $ (closed) compact submanifold of $ X $. Then assumption ($ \bullet $3) in \autoref{immersionII} is satisfied, which is shown as follows.

	Here note that since $ M $ is $ C^{2} $ and compact, $ m \mapsto T_mM \in \mathbb{G}(X) $ is $ C^1 $; through $ C^1 $ partitions of unity, one can further construct the $ C^1 $ normal bundle over $ M $, i.e. $ m \mapsto X^h_{m} \in \mathbb{G}(X) $ is $ C^1 $ with $ X^h_{m} \oplus T_{m}M = X $. Now we have projections $ \Pi^c_m $, $ m \in M $, associated with $ X^h_{m} \oplus T_{m}M = X $ such that $ R(\Pi^c_m) = T_{m}M $, and moreover $ m \mapsto \Pi^c_m \in L(X,X) $ is $ C^1 $. Then it follows from \cite{BLZ98} that assumption ($ \bullet $3) in \autoref{immersionII} holds.

	$ \phi \in C^2 $ can be relaxed to $ \phi \in C^1 $ by smooth approximation (see e.g. \cite[Theorem 6.9]{BLZ08} for details), but now $ M $ is strongly $ C^{0,1} $-uniform in the sense of \autoref{def:uManifold}.
\end{exa}

\begin{exa}[Vector bundles over immersed manifolds]\label{vectorBB}
	We continue \autoref{immersionII} with $ (\bullet 3) $. Let $ \Pi^h_m $, $ m \in M $, be projections of $ X $ (i.e. $ \Pi^h_m \circ \Pi^h_m = \Pi^h_m $) and $ X^h_m = R(\Pi^h_m) $. Let us consider the following bundle over $ M $:
	\[
	X^h = \{ (m, x): x \in X^h_m, m \in M \}.
	\]
	The natural Finsler structure in $ X^h $ is given by $ |x|_m = |x| $, $ x \in X^h_m $. The natural bundle atlas $ \mathcal{A} $ for $ X^h $ is given (formally) by $ \varphi^{m_c}(m',x) = (m', \Pi^h_{m'} x) $, $ m' \in U^c_{m_c}(\epsilon_{m}) $, $ x \in X^h_{m} $, $ m_c \in \Lambda_{m} $; note that under (a) below, it is indeed a bundle atlas.
	\begin{enumerate}[(a)]
		\item If $ m \mapsto \Pi^h_m $, $ m \in M $, is continuous in the immersed topology of $ M $, i.e., for each $ m \in M $, $ m_c \in \Lambda_{m} $, $ m' \mapsto \Pi^h_{m'} $, $ m' \in U^c_{m_c}(\epsilon_{m}) $, is continuous, then $ X^h $ is a $ C^0 $ vector bundle with strongly $ C^0 $ Finsler structure (see \autoref{finsler} \eqref{finA}). A particular case where the assumption holds is that $ m \mapsto \Pi^h_m $ is continuous in the topology induced from $ X $.

		\item If $ m \mapsto \Pi^h_m $ is uniformly continuous (resp. $ \varepsilon $-almost uniformly continuous) around $ M_1 $ (see \autoref{hcontinuous}) in the immersed topology of $ M $, then $ X^h $ is $ C^0 $-uniform around $ M_1 $ (see \autoref{def:C0Uniform}) (resp. $ X^h $ has $ \varepsilon $-almost uniform $ C^{1,1} $-fiber trivializations on $ M_1 $ (see \autoref{def:lipbundle})) with respect to $ \mathcal{A} $. Similar to $ (\bullet 3) $ in \autoref{immersionII}, if $ m \mapsto \Pi^h_m $ is uniformly Lipschitz in the sense of $ (\bullet 3) $, then $ X^h $ has $ 0 $-almost $ C^{0,1} $-uniform trivializations on $ M_1 $ (see \autoref{def:C01Uniform}) with respect to $ \mathcal{A} $.

		\item Suppose for every $ m \in M $, $ m_c \in \Lambda_{m} $, $ \Pi^h_{(\cdot)}|_{U^c_{m_c}(\epsilon_{m})} $ is differentiable at $ m $, and let $ \mathcal{C}_{m_c} = D_m \Pi^h_{m}|_{U^c_{m_c}(\epsilon_{m})} \in L(X^c_m \times X^h_m; X^h_m) $. (Note that since $ X^c_m, X^h_m $ are closed splitting, we can consider $ \mathcal{C}_{m_c} \in L(X\times X; X) $.) Assume $ m' \mapsto \mathcal{C}_{m'} $, $ m' \in U^c_{m_c}(\epsilon_{m}) $, is continuous. Then $ X^h $ is a $ C^1 $ bundle. The readers should notice that $ m \mapsto \Pi^h_m $ might not be differentiable in $ X $ and $ \mathcal{C}_{m_c} \neq \mathcal{C}_{m'_c} $ when $ m_c, m'_c \in \Lambda_{m} $, for $ M $ is only immersed. Moreover, $ \{ \mathcal{C}_{m_c}\} $ gives a natural $ C^0 $ linear connection $ \mathcal{C} $ of $ X^h $ such that $ \varphi^{m_c} $ is a normal bundle chart at $ m_c $ with respect to $ \mathcal{C} $, i.e., $ \mathcal{C} D \varphi^{m_c}(m_c, x) = \id $.

		In addition, if $ m' \mapsto \Pi^c_{m'} $, $ m' \mapsto \Pi^h_{m'} $ and $ m' \mapsto \mathcal{C}_{m'} $ are uniformly continuous (see \autoref{hcontinuous}) in the immersed topology, then $ X^h $ is a $ C^1 $-uniform vector bundle (see \autoref{def:C11Uniform}); what is more, if $ m' \mapsto \mathcal{C}_{m'} $ is uniformly (locally) Lipschitz in the immersed topology, then $ X^h $ is a $ C^{1,1} $-uniform vector bundle (see \autoref{def:C11Uniform}).
	\end{enumerate}
\end{exa}

\subsection{Bounded geometry}\label{bgemetry}

Let $ ( M, g ) $ be a connected Riemannian manifold modeled on a Hilbert space $ \mathbb{H} $ and equipped with the \emph{Levi-Civita connection} where $ g $ is the \emph{Riemannian metric}. Let $ R $ be the \emph{curvature tensor} induced by the Levi-Civita connection and $ r(m) $ the \emph{injectivity radius} at $ m \in M $ (i.e. the supremum radius for which the exponential map at $ m $ is a diffeomorphism). See e.g. \cite{Kli95} for these basic concepts.

\begin{defi}[Bounded geometry]\label{defi:bounded}
	A $ C^{k+2} $ complete Riemannian manifold $ ( M, g ) $ has $ k $th order \textbf{bounded geometry} if the following hold:
	\begin{enumerate}[(i)]
		\item the injectivity radius of $ M $ is positive, i.e., $ r(M) = \inf_{m \in M} r(m) > 0 $;
		\item the covariant derivatives of $ R $ up to $ k $th order and $ R $ are uniformly bounded, i.e.,
		\[
		\sup_{m \in M} \|\nabla^i_{m} R(m)\| < \infty,~0 \leq i \leq k.
		\]
	\end{enumerate}
\end{defi}
See also \cite{Ama15, DSS16, Eld13, Eic91} where the readers can find many \emph{concrete} examples of Riemannian manifolds having bounded geometry. Note that \emph{every $ C^{k+2} $ compact Riemannian manifold has $ k $th order bounded geometry}.
The global definition of bounded geometry does not give a very clear relation to our assumptions. In \cite{Ama15}, the author introduced uniformly regular Riemannian manifolds, and proved that a Riemannian manifold admitting bounded geometry is uniformly regular.

Following \cite{Ama15} we introduce uniformly regular Riemannian manifolds.
Take a $ C^{k} $ atlas $ \mathcal{M} = \{ (U_{\gamma}, \varphi_\gamma): \gamma \in \Lambda \} $ of a $ C^{k} $ manifold $ M $ where $ \Lambda $ is an index set. Take $ S \subset M $ and $ \Lambda_S = \{ \gamma \in \Lambda: U_{\gamma} \cap S \neq \emptyset \} $. We say that (i) $ \mathcal{M} $ is \emph{normalized} on $ S $ if $ \varphi_{\gamma}(U_\gamma) = \mathbb{B}_1 $, $ \gamma \in \Lambda_S $ where $ \mathbb{B}_1 $ is the unit (open) ball of $ \mathbb{H} $, (ii) $ \mathcal{M} $ has \emph{finite multiplicity} on $ S $ if there is an $ N \in \mathbb{N} $ such that any intersection of  $ N+1 $ elements of $ \{U_\gamma: \gamma \in \Lambda_S\} $ is empty, and (iii) $ \mathcal{M} $ is \emph{uniformly shrinkable} on $ S $ if there is an $ r \in (0,1) $ such that $ \{ \varphi^{-1}_\gamma(r\mathbb{B}_1) : \gamma \in \Lambda_S \} $ is also an open cover of $ S $.

\begin{defi}[Uniformly regular Riemannian manifolds]\label{UR}
	We say a $ C^{k+1} $ Riemannian manifold $ (M,g) $ is a $ k $th order \textbf{uniformly regular Riemannian manifold} on $ S $ if there is a $ C^{k+1} $ atlas $ \mathcal{M} = \{ (U_{\gamma}, \varphi_\gamma): \gamma \in \Lambda \} $ such that the following hold:
	\begin{enumerate}[(UR1)]
		\item $ \mathcal{M} $ is normalized and uniformly shrinkable and has finite multiplicity on $ S $;
		\item $ |\varphi_{\gamma_1} \circ \varphi^{-1}_\gamma|_{k+1} \leq C(k+1) $, $ \gamma_1, \gamma \in \Lambda_S $ where $ C(k+1) $ is a constant;
		\item $ (\varphi^{-1}_\gamma)^{*} g \sim d_{\mathbb{H}} $, i.e., there is $ C \geq 1 $ such that for all $ \gamma \in \Lambda_S $, $ m \in \mathbb{B}_1 $, $ x \in \mathbb{H} $,
		$ C^{-1}|x|^2 \leq ((\varphi^{-1}_\gamma)^{*} g) (m)(x, x) \leq C|x|^2 $,
		where $ (\varphi^{-1}_\gamma)^{*} g $ denotes the pull-back metric of $ g $ by $ \varphi^{-1}_\gamma $.
		\item $ |(\varphi^{-1}_\gamma)^{*} g|_{k} \leq C_1(k) $ for all $ \gamma \in \Lambda_S $ where $ C_1(k) $ is a constant.
	\end{enumerate}
	If $ S = M $, we also say $ M $ is a $ k $th order \emph{uniformly regular Riemannian manifold}.
	Here $ |u|_{k} = \sup_{0\leq i \leq k} \sup_{x} |\partial^{i}_xu(x)|  $.
\end{defi}

\begin{exa}[$ 0 $th and $ 1 $st order uniformly regular Riemannian manifolds]
	If $ M $ is $ 0 $th order uniformly regular, then $ M $ is a $ C^{0,1} $-uniform manifold. Indeed, the atlas $ \mathcal{M} $ in \autoref{UR} provides us the desired atlas needed in assumption $ (\blacksquare) $ (\autopageref{C1MM}). By the uniformly shrinkable condition of $ \mathcal{M} $, one has an $ r \in (0,1) $ such that $ \{ \varphi^{-1}_\gamma(r\mathbb{B}_1) : \gamma \in \Lambda \} $ is also an open cover of $ M $, and invoking (UR3) one finds that there is a $ c \geq 1 $ such that
	\[
	U_{m}(\delta/c) \subset \varphi^{-1}_\gamma (\mathbb{B}(x,\delta)) \subset U_{m}(c\delta),
	\]
	where $ x = \varphi_{\gamma}(m) \in r\mathbb{B}_1 $, $ \delta < 1-r $, $ \mathbb{B}(x,\delta) = x + \delta\mathbb{B}_1 $; see e.g. \cite{DSS16} and \autoref{lem:usd}. Now for any $ m \in M $, choose a map $ \varphi_{\gamma} $ such that $ \varphi_{\gamma}(m) = x \in r\mathbb{B}_1 $, and let $ \chi_{m}(m') = (D\varphi_{\gamma}(m))^{-1}(\varphi_{\gamma}(m') - x) $. (Note that $ \sup_\gamma \sup_{x \in \mathbb{B}_1}|D\varphi^{-1}_{\gamma}(x)| < \infty $ by (UR3).) Upon choosing smaller $ \delta > 0 $, the atlas $ \mathcal{M}_1 = \{ (U_{m}(\delta/c), \chi_{m}): m \in M \} $ satisfies $ (\blacksquare) $ (with $ M_1 = M $).
	Similarly, if $ M $ is a $ 0 $th order uniformly regular Riemannian manifold on $ M_1 $, then $ M $ is $ C^{0,1} $-uniform around $ M_1 $ (see $ (\blacksquare) $), and if $ M $ is a $ 1 $st order uniformly regular Riemannian manifold on $ M_1 $, then $ M $ is $ C^{1,1} $-uniform around $ M_1 $ (see \autoref{def:C11um}); for the latter case, the Levi-Civita connection is uniformly (locally) Lipschitz in the sense of \autoref{lipcon}.
\end{exa}

\begin{exa}[$ 2 $nd order bounded geometry]
	When $ M $ has $ 2 $nd order bounded geometry, $ M $ is a $ C^{0,1} $-uniform manifold (see \autoref{def:C11um}).

	The most important thing following from $ M $ having $ k $th order bounded geometry is that there is an $ r_0 > 0 $ such that the Riemannian metric up to $ k $th order derivatives and the Christoffel maps in the normal charts (i.e. the normal coordinates) up to $ (k-1) $th order derivatives are all uniformly bounded in normal charts of radius $ r_0 $ around each $ m \in M $ with bounds independent of $ m \in M $. See e.g. \cite{Eic91} and note that the proof given there also makes sense in the infinite-dimensional setting. Let
	\[
	\chi_{m} = \mathrm{ exp }^{-1}_{m}: U_{m}(r_0) \to \mathbb{B}_{r_0} = T_{m}M(r_0) \subset \mathbb{H}
	\]
	be the normal chart at $ m \in M $, where $ U_{m}(r_0) = \{ m' \in M: d(m,m') < r_0 \} $. Now the atlas $ \mathcal{M}' $ given by the normal charts satisfies (UR3), (UR4) (for $ k = 2 $) and (UR2) (for $ \gamma_1 = m', \gamma = m \in \Lambda = M $ with $ d(m',m) < r_0 $ and for $ k = 0 $). (See also \cite[Chapter 2]{Eld13}.)
	So assumption $ (\blacksquare) $ (\autopageref{C1MM}) holds.
\end{exa}

It is proved in \cite{Ama15, DSS16} that a $ C^\infty $ Riemannian manifold has any order bounded geometry if and only if it is any order uniformly regular when $ \mathbb{H} = \mathbb{R}^n $.

\begin{exa}[Vector bundles having bounded geometry]\label{vectorB}
	Consider a special class of vector bundles, \emph{vector bundles having $ k $th order bounded geometry}, introduced in e.g. \cite[p. 45]{Eld13} (for finite $ k $) and \cite[p. 65]{Shu92} (for all $ k \in \mathbb{N} $). Let $ M $ be a Riemannian manifold having $ k $th order bounded geometry and $ X $ a vector bundle over $ M $. We say $ X $ has $ k $th order bounded geometry if there exist preferred bundle charts $ \varphi^{m_0}: U_{m_0}(\delta) \times X_{m_0} \to X $ $ (\varphi^{m_0}_{m_0} = \id $), $ m_0 \in M $, such that all the transition maps $ {\varphi}^{m_0, m_1} (m) = (\varphi^{m_1}_{m})^{-1} \circ \varphi^{m_0}_{m} \in L(X_{m_0}, X_{m_1}) $, $ U_{m_0}(\delta) \cap U_{m_1}(\delta) \neq \emptyset $, satisfy $ |{\varphi}^{m_0, m_1} |_{k} \leq C_0(k) $, where $ C_0(k) > 0 $ is a constant independent of $ m_0, m_1 \in M $.
	In particular, \emph{$ TM $ has $ (k-2) $th order bounded geometry if $ M $ has $ k $th order bounded geometry}.
	Now a vector bundle having $ 1 $st (resp. $ 2 $nd) order bounded geometry is $ C^{0,1} $-uniform (resp. $ C^{1,1} $-uniform) by definition; see \autoref{def:C01Uniform} and \autoref{def:C11Uniform}.
\end{exa}

We refer the readers to similar assumptions on manifolds and bundles in \cite[Section 2.2]{Cha04}, \cite[Section 8]{PSW12} and \cite[Chapter 6]{HPS77}. It seems that J. Eldering first introduced bounded geometry to investigate normal hyperbolicity theory for flows (see \cite{Eld13}); also, he provided a detailed exposition of bounded geometry. \autoref{immersionII} and \autoref{vectorBB} are essentially due to \cite{BLZ99, BLZ08}, where the authors studied normal hyperbolicity theory in the setting of \autoref{immersionII} and \autoref{vectorBB} for semiflows.

%% file: app1.tex
\chapter{Appendix. Miscellanea}\label{misc}

\section{A parameter-dependent fixed point theorem}\label{fixedp}

The following well known results are frequently used in this paper. We state them (together with a very simple proof) for the convenience of the readers. Some \emph{redundant} conditions are removed in order for us to be able to apply them.

Let $ (X,d) $ be a metric space with metric $ d $, $ Y $ a topological space, and $ G: X \times Y \to X $. We also write $ d(x,y) = |x - y| $.

1. Existence of $ x(\cdot) $

\begin{lem}
	Suppose $ \lip G(\cdot, y) < 1 $ for all $ y \in Y $, and $ X $ is complete. Then by the Banach Fixed Point Theorem, there is a unique $ x(y) \in X $ such that
	\begin{equation}\label{app0}
	G(x(y), y) = x(y).
	\end{equation}
\end{lem}

2. Continuity of $ x(\cdot) $

\begin{lem}\label{lem:fconx}
	Let $ \lip G(\cdot, y) < 1 $ and $ x(\cdot) $ satisfy \eqref{app0} (without assuming that $ X $ is complete). Then the following assertions hold.
	\begin{enumerate}[(1)]
		\item If $ G(x_2, y_2) = x_2 $, $ y \mapsto G(x_2, y) $ is continuous at $ y_2 $, and there is a neighborhood $ U_{y_2} $ of $ y_2 $ in $ Y $ such that $ \sup_{y \in U_{y_2}}\lip G(\cdot,y) < 1 $, then $ x(\cdot) $ is continuous at $ x_2 $.

		In particular, if $ \sup_{y} \lip G(\cdot,y) < 1 $ (or $ y \mapsto \lip G(\cdot,y) $ is continuous) and $ G $ is continuous, then $ x(\cdot) $ is continuous.

		\item $ G(x, \cdot) $ is continuous uniformly for $ x $ in the following sense: $ \forall y > 0 $, $ \forall \epsilon > 0 $, $ \exists U_y $ a neighborhood of $ y $ in $ Y $ such that $ |G(x,y') - G(x,y)| < \epsilon $ for all $ x \in X $ provided $ y' \in U_{y} $. Then $ x(\cdot) $ is continuous. Note that in this case, also $ y \mapsto \lip G(\cdot,y) $ and $ G $ are continuous.

		Similarly, if $ Y $ is a uniform space and $ G(x, \cdot) $, $ x \in X $, are equicontinuous, then $ x(\cdot) $ is uniformly continuous.

		\item Let $ Y $ be a metric space. If $ \lip G(x, \cdot) \leq \alpha(x) < \infty $, then $ x(\cdot) $ is locally Lipschitz. Furthermore, suppose $ \lip G(\cdot, y) \leq \beta < 1 $ and $ \sup_x\alpha(x) < \alpha $. Then $ |x(y_1) - x(y_2)| \leq \frac{\alpha}{1- \beta} |y_1 - y_2| $, which yields $ x(\cdot) $ is globally Lipschitz.

		\item Let $ Y $ be a metric space. If $ \sup_{y} \lip G(\cdot, y)  < 1 $ and $ G(x, \cdot) $ is $ \alpha $-H\"older uniformly for $ x $ (i.e. $ \sup_{x} \hol_{\alpha} G(x,\cdot) < \infty $), then $ x(\cdot) $ is $ \alpha $-H\"older.

		Similarly, when $ G(x, \cdot) $ is $ \alpha $-H\"older (or locally $ \alpha $-H\"older) uniformly for $ x \in X $ (or for $ x $ belonging to any bounded subset of $ X $), then $ x(\cdot) $ is either $ \alpha $-H\"older, locally $ \alpha $-H\"older, or $ \alpha $-H\"older on any bounded subset of $ Y $.
	\end{enumerate}
\end{lem}
\begin{proof}
	Consider
	\begin{align*}
	|x(y_1) - x(y_2)| & \leq |G(x(y_1), y_1) - G(x(y_2), y_1)| + |G(x(y_2), y_1) - G(x(y_2), y_2)| \\
	& \leq \lip G(\cdot, y_1)|x(y_1) - x(y_2)| + |G(x(y_2), y_1) - G(x(y_2), y_2)|,
	\end{align*}
	which yields
	\begin{equation*}
	|x(y_1) - x(y_2)| \leq \frac{|G(x(y_2), y_1) - G(x(y_2), y_2)|}{1-\lip G(\cdot, y_1)}.
	\end{equation*}
	All the assertions follow from the above inequality.
\end{proof}

3. Differentiability of $ x(\cdot) $

\begin{lem}\label{lem:fdiifx}
	In order to make sense of the differentiability of $ x(\cdot) $, first assume $ X, Y $ are open subsets of normed spaces. If $ x(\cdot) $ satisfying \eqref{app0} in a neighborhood of $ y $ is continuous at $ y $, $ G(\cdot,\cdot) $ is differentiable at $ (x(y),y) $, and $ |D_1G(x(y),y)| < 1 $, then $ x(\cdot) $ is differentiable at $ y $ with $ Dx(y) = A(y) $, where
	\[
	A(y) = (I - D_1G(x(y),y))^{-1}D_2G(x(y),y).
	\]
\end{lem}
\begin{proof}
	The proof is standard.
	Since $ |D_1G(x(y),y)| < 1 $, $ A(y) $ is well defined.
	By the differentiability of $ G $ at $ (x(y),y) $, for every $ \epsilon > 0 $, there is a $ \delta' > 0 $ such that
	\[
	|G(x',y') - G(x(y),y) - DG(x(y),y)(x'-x(y),y'-y)| \leq \epsilon(|x' - x(y)|+ |y' - y|)
	\]
	provided $ |x' -x(y)| + |y' - y| < \delta' $.
	Moreover, by the continuity of $ x(\cdot) $ at $ y $, there is a $ \delta > 0$ such that $ |x(y+a) - x(y)| \leq \delta' $, if $ |a| \leq \delta $.
	Thus, if $ \epsilon < \frac{1}{2} (1-|D_1G(x(y),y)|) $ and $ |a| \leq \delta $, we have
	\begin{align*}
	& |x(y+a) - x(y) - A(y)a| = |G(x(y+a),y+a) - G(x(y),y) - DG(x(y),y)(A(y)a,a)| \\
	\leq & |DG(x(y),y)(x(y+a)-x(y), a) - DG(x(y),y)(A(y)a,a)| + \epsilon (|x(y+a)-x(y)| + |a|) \\
	\leq & |D_1G(x(y),y) ||x(y+a)- x(y) -A(y)a| + \epsilon (|x(y+a)-x(y)| + |a|),
	\end{align*}
	which yields
	\[
	|x(y+a) - x(y) - A(y)a| \leq \frac{\epsilon}{1- |D_1G(x(y),y)|} (|x(y+a)-x(y)| + |a|)
	\]
	and
	\[
	|x(y+a) - x(y)| \leq 2(|A(y)|+1)|a|.
	\]
	Therefore,
	\[
	\frac{|x(y+a) - x(y) - A(y)a|}{|a|} \leq \frac{\epsilon}{1- |D_1G(x(y),y)|} (2|A(y)|+3),
	\]
	which shows $ Dx(y) = A(y) $.
\end{proof}

Unlike the classical situation, $ Y $ need not be complete (which is important in \autoref{stateRegularity}), and the condition $ G \in C^1 $ is reduced to $ G $ only differentiable at $ (x(y), y) $. We usually apply \autoref{lem:fconx} and \autoref{lem:fdiifx} in \autoref{stateRegularity} in the following way: $ x(\cdot) $ is known by the global construction and the regularity of $ x(\cdot) $ is obtained by applying the above results to the representation of $ x(\cdot) $ in local coordinates.

Note that if $ \lip G(\cdot, y) < 1 $ and $ X $ is an open subset of a normed space, then 
\[
|D_1G(x(y),y)| < 1.
\]

\begin{cor}\label{1111}
	If $ \lip G(\cdot, y) < 1 $ for $ y \in Y $ where $ X, Y $ are normed spaces, and if $ G \in C^{k} $ ($ k = 1,2,\ldots,\infty,\omega $) and $ x(\cdot) $ satisfying \eqref{app0} is continuous, then $ x(\cdot) \in  C^{k} $.
\end{cor}

Consider the general case.

\begin{cor}
	Assume that $ X $ is a $ C^k $ Finsler manifold (possibly incomplete) and $ Y $ is a $ C^k $ Banach manifold. Also suppose $ G \in C^k(X \times Y, X) $ and 
	\[
	|D_1G(x(y),y)| < 1.
	\]
	If $ x(\cdot) $ satisfying \eqref{app0} is continuous, then it is also $ C^k $.
\end{cor}
\begin{proof}
	Choose $ C^k $ local charts $ \varphi: U' \subset U_{x(y)} \to T_{x(y)}X $ and $ \phi: V_{y} \to T_{y}Y $, such that $ \varphi(x(y)) = 0 $, $ \phi(y) = 0 $, $ G(U', V_{y}) \subset U_{x(y)} $, and $ D\varphi(x(y)) = \id $. Set
	\[
	G'(x',y') = \varphi \circ G (\varphi^{-1}(x'), \phi^{-1}(y')) : \varphi(U') \times \phi(V_y) \to \varphi (U_{x(y)}).
	\]
	Note that $ D_1 G'(0,0) = D_1G(x(y),y): T_{x(y)}X \to T_{x(y)}X $. So $ |D_1 G'(0,0)| < 1 $.
	Let $ x'(y') = \varphi\circ x\circ \phi^{-1}(y') $. Then
	\[
	G'(x'(y'), y') = x'(y').
	\]
	Therefore, $ x'(\cdot) $ is $ C^k $.
\end{proof}

\begin{cor}
	Assume that $ X $ is a $ C^k $ Finsler manifold with a compatible metric $ d $ such that for every $ f \in C^1(X, X) $, $ \|Df\| \leq K \lip f $ where $ K $ is a constant, and $ Y $ is a $ C^k $ Banach manifold. Suppose $ G \in C^k(X \times Y, X) $ and $ \lip G(\cdot, y) < 1 $ for all $ y \in Y $. If $ x(\cdot) $ satisfying \eqref{app0} is continuous, then it is also $ C^k $.
\end{cor}

See \autoref{finsler} for the situation: $ f \in C^1(X, X) \Rightarrow \|Df\| \leq K \lip f $.

\section{Finsler structure in the sense of Neeb--Upmeier}\label{finsler}
We give some basic definitions of Finsler structure for the convenience of the readers. In the infinite-dimensional setting, the following definitions are not equivalent.
\begin{asparaenum}[(a)]
	\item \label{finA} We say a $ C^0 $ vector bundle $ (X, M, \pi) $ has a $ C^0 $ \textbf{Finsler structure} if there is a $ C^0 $ map $ |(\cdot, \cdot)|: X \to \mathbb{R}_+ $ such that $ x \mapsto |x| \triangleq |x|_m \triangleq |(m,x)| $ is a norm on $ X_m $. Moreover, for every $ m \in M $, there is a $ C^0 $ (vector) bundle chart $ (U,\psi) $ at $ m $ such that $ \sup_{m' \in U}| \psi^{\pm1}_{m'} | < \infty $, where the norm of $ | \psi_{m'} | $ (and similarly for $ |\psi^{-1}_{m'}| $) is defined by
	\[
	| \psi_{m'} | = \sup \{ |\psi_{m'} x |_{m'}: x \in X_{m}, |x|_{m} \leq 1 \}.
	\]

	Here, $ |(\cdot, \cdot)|$ is $ C^0 $ in the following sense: For every $ m \in M $ and every $ C^0 $ (vector) bundle chart $ (U,\psi) $ at $ m $, $ (m',x) \mapsto | \psi_{m'}x |_{m'} $ is $ C^0 $. We say the Finsler structure is uniformly $ C^0 $ (or $ X $ has a \textbf{uniformly $ C^0 $ Finsler structure}) if for every $ m \in M $ and every $ C^0 $ (vector) bundle chart $ (U,\psi) $ at $ m $, $ m' \mapsto | \psi^{\pm1}_{m'} | $ is continuous at $ m $.

	\item \label{finB} We say a $ C^0 $ vector bundle $ (X, M, \pi) $ has a $ C^0 $ \textbf{uniform Finsler structure (with constant $ \Xi $)} if there is a $ C^0 $ map $ |(\cdot, \cdot)|: X \to \mathbb{R}_+ $ such that (i) $ x \mapsto |x| \triangleq |x|_m \triangleq |(m,x)| $ is a norm on $ X_m $, and (ii) there exists a constant $ \Xi > 0 $ such that for every $ m \in M $, there is a $ C^0 $ (vector) bundle chart $ (U,\psi) $ at $ m $ such that $ \sup_{m' \in U}|\psi_{m'}| \leq \Xi, \sup_{m' \in U}|\psi^{-1}_{m'}| \leq \Xi $.

	\item \label{finC} Let $ X $ be a $ C^0 $ vector bundle $ (X, M, \pi) $ with fibers modeled on $ \mathbb{F} $, i.e., every $ X_m $ is isomorphic to $ \mathbb{F} $ with an isomorphism $ \mathbb{I}_m: X_m \to \mathbb{F} $. We say $ X $ has a $ C^0 $ \textbf{strong uniform Finsler structure (with constant $ \varXi $)} if there is a $ C^0 $ map $ |(\cdot, \cdot)|: X \to \mathbb{R}_+ $ such that (i) $ x \mapsto |x| \triangleq |x|_m \triangleq |(m,x)| $ is a norm on $ X_m $, and (ii) there exists a constant $ \varXi > 0 $ such that for every $ m \in M $, there is a $ C^0 $ (vector) bundle chart $ (U,\psi) $ at $ m $ such that $ \sup_{m' \in U}| \mathbb{I}_{m'} \psi_{m'} | \leq \varXi, \sup_{m' \in U}|\psi^{-1}_{m'}\mathbb{I}^{-1}_{m'}| \leq \varXi $.

	\item \label{finD} A $ C^k $ Banach manifold $ M $ is called a $ C^k $ \textbf{Finsler(--Banach) manifold} (in the sense of Neeb--Upmeier) if there is a $ C^0 $ Finsler structure in $ TM $.
	Moreover it is called a \textbf{$ \Xi $-weak uniform $ C^k $ Finsler manifold in the sense of Neeb--Upmeier} if there is a $ C^0 $ uniform Finsler structure (with constant $ \Xi $) in $ TM $.

	Finally, $ M $ is called a $ C^k $ \textbf{Finsler manifold in the sense of Palais} (see \cite{Pal66}) if the $ C^0 $ Finsler structure is uniformly $ C^0 $; or equivalently, for any constant $ \Xi > 1 $ and every $ m \in M $, there is a $ C^0 $ (vector) bundle chart $ (U,\psi) $ at $ m $ such that $ \sup_{m' \in U}|\psi_{m'}| \leq \Xi |\psi_{m}| $, $ \sup_{m' \in U}|\psi^{-1}_{m'}| \leq \Xi |\psi^{-1}_{m}| $. Note that every \emph{Riemannian manifolds} belongs to this case.

	\item Let $ M $ be a $ C^k $ Banach manifold modeled on $ \mathbb{E} $. Note that in this case, $ TM $ is modeled on $ \mathbb{E} \times \mathbb{E} $. $ M $ is called a  \textbf{$ \varXi $-uniform $ C^k $ Finsler manifold in the sense of Neeb--Upmeier} if there is a $ C^0 $ strong uniform Finsler structure (with constant $ \varXi $) in $ TM $.

	\item The \textbf{Finsler metric} in Finsler manifolds can be defined in the usual way; see e.g. \cite{Upm85}. Let $ M $ be a $ C^1 $ Finsler manifold. The associated Finsler metric $ d_M $ in each component of $ M $ is defined by
	\begin{multline*}
	d_M(p,q) \\
	= \inf\left\{ \int_{0}^{t} |c'(t)|_{c(t)} ~ \mathrm{ d } t: c ~\text{is a}~ C^1 ~\text{path}~ [0,1] \to M ~\text{such that}~ c(0) = p ,  c(1) = q \right\}.
	\end{multline*}

	For any $ f \in C^1(M, N) $, the following basic facts hold: If $ M $, $ N $ are Finsler manifolds \emph{in the sense of Palais}, then $ \sup_{m}|Df(m)| = \lip f $ (see e.g. \cite{GJ07});
	if $ M $, $ N $ are Finsler manifolds (in the sense of Neeb--Upmeier), then $ |\lip f| \leq \sup_m|Df(m)| $; if $ M $, $ N $ are respectively $ \Xi_1 $-\emph{weak uniform} and $ \Xi_2 $-\emph{weak uniform} Finsler manifolds in the sense of Neeb--Upmeier, then $ \sup_{m}|Df(m)| \leq \Xi_1 \Xi_2 \lip f $; if $ M $, $ N $ are respectively $ \Xi_1 $-\emph{uniform} and $ \Xi_2 $-\emph{uniform} Finsler manifolds in the sense of Neeb--Upmeier, then $ \sup_{m}|Df(m)| \leq \Xi^2_1 \Xi^2_2 \lip f $. For a proof of the last three results, see \cite{JS11}.
\end{asparaenum}

Note that in the finite-dimensional setting, $ C^0 $ Finsler structure is the same as uniform $ C^0 $ Finsler structure; in particular, the definitions given in (a)--(c) are identical.
The definitions of Finsler manifold in the sense of Neeb--Upmeier are taken from \cite{JS11}; see also \cite{Pal66, Upm85, Nee02}.

\section{Length metric and Lipschitz continuity} \label{length}
In this appendix, we discuss some versions of `local Lipschitz (H\"older) continuity' in the pointwise sense that may imply `global Lipschitz (H\"older) continuity' in bounded sets.

Let $ M $ be a metric space with metric $ d $. A path $ s $ from $ x $ to $ y $ in $ M $ is a continuous map $ s: [a,b] \to M $ such that $ s(a) = x $, $ s(b) = y $, and its length is defined by
\[
l(s) = \sup \left\lbrace  \sum^{k}_{i=1} d( s(t_{i-1}), s(t_{i}) ): a = t_0 < t_1 < \cdots < t_k = b, k=1,2,3,\ldots \right\rbrace .
\]
If $ l(s) < \infty $, then $ s $ is called a rectifiable path (or bounded variation path). If for every $ x, y \in M $,
\[
d(x,y) = \inf \{ l(s): s ~\text{is a path from $ x $ to $ y $} \},
\]
then $ M $ is said to be a \textbf{length space} and $ d $ a \textbf{length metric}. Some examples of length spaces are the following (see \cite{GJ07} for details): (a) normed spaces, and convex subsets of normed spaces; (b) connected Finsler manifolds in the sense of Palais (see \autoref{finsler}) with Finsler metrics, in particular, connected Riemannian manifolds with Riemannian metrics; (c) geodesic spaces.

Let $ M, N $ be metric spaces, and $ f: M \to N $. Define the \textbf{upper scalar derivative} of $ f $ at $ x $ by (see \cite{GJ07})
\begin{equation}\label{scalarD}
D^{+}_x f = \limsup_{y \to x} \frac{d(f(y), f(x))}{d(y,x)}.
\end{equation}
If $ f: M \to N $ is continuous, $ s $ is a rectifiable path in $ M $, and $ t = f \circ s $, then (see \cite[Proposition 3.8]{GJ07})
\[
l(t) \leq \sup_{x \in \mathrm{Im}(s)} D^{+}_x f \cdot l(s).
\]
In particular, we have the following property.
\begin{lem}\label{liplength}
	If $ M, N $ are length spaces and $ \sup_{x} D^{+}_x f < \infty $, then $ f $ is Lipschitz with $ \lip f = \sup_{x} D^{+}_x f $.
\end{lem}
See \cite{GJ07} for more properties of upper scalar derivative.
Let us consider the H\"older case. Define the \textbf{upper scalar uniform $ \alpha $-H\"older constant} of $ f $ by
\[
D^{+}_{\alpha} f = \lim_{r \to 0 } \sup_{x} \sup_{d(z,x) < r} \frac{d(f(z), f(x))}{d(z,x)^\alpha},
\]
where $ 0 < \alpha \leq 1 $.
\begin{lem}\label{lem:H}
	If $ M, N $ are length spaces and $ D^{+}_{\alpha} f < \infty $, then for any given bounded subset $ A $ of $ X $,
	\[
	d(f(x),f(y)) \leq \Omega(\diam(A))(D^{+}_{\alpha} f + 1) d(x,y)^{\alpha},~ \forall x,y \in A,
	\]
	for some $ \Omega : \mathbb{R}_+ \to \mathbb{R}_+ $.
\end{lem}

In contrast to $ D^{+}_x f $, $ D^{+}_{\alpha} f $ is defined uniformly, so the proof is simpler than in the Lipschitz case.

\begin{proof}
	Since $ D^{+}_{\alpha} f \triangleq M_0 < \infty $, there is a $ \delta > 0 $ such that $ d(f(z), f(x)) \leq (M_0 + 1) d(z,x)^{\alpha} $ provided $ d(z,x) \leq \delta $. Let $ x, y \in A $ and $ n = [\diam(A)/\delta]+1 $. For any given $ \epsilon > 0 $, there is a path $ s: [a,b] \to X $ from $ x $ to $ y $ such that $ l(s) < d(x,y) + \epsilon $ and $ d(x,y) + \epsilon \leq n \delta $. Now one can choose $ a = t_0 < t_1 < \cdots < t_m =b $, $ m \leq n $, such that $ d( s(t_{i-1}), s(t_{i}) ) = \delta $, $ i = 1,2,\ldots, m-1 $, and $ d( s(t_{m-1}), s(t_{m}) ) \leq \delta $. Let $ r = f \circ s $. Then $ d( r(t_{i-1}), r(t_{i}) ) \leq (M_0 + 1) d(s(t_{i-1}), s(t_{i}))^\alpha $. Hence,
	\begin{multline*}
	d(f(x),f(y)) \leq \sum^{m}_{i=1} d( r(t_{i-1}), r(t_{i}) ) \leq (M_0 + 1) \sum^{m}_{i=1}d(s(t_{i-1}), s(t_{i}))^\alpha \\
	\leq (M_0 + 1) n^{1-\alpha} (d(x,y) + \epsilon )^{\alpha},
	\end{multline*}
	where the H\"older inequality is used. Since $ \epsilon $ is arbitrary, we have $ d(f(x),f(y)) \leq (M_0 + 1) n^{1-\alpha} d(x,y)^{\alpha} $.
\end{proof}

One can define the pointwise version of the upper scalar $ \alpha $-H\"older constant of $ f $ at $ x $ by
\[
D^{+}_{\alpha, x} f = \limsup_{z \to x} \frac{d(f(z), f(x))}{d(z,x)^\alpha}.
\]
However, $ \sup_{x} D^{+}_{\alpha, x} f < \infty $ does not imply the $ \alpha $-H\"olderness of $ f $ even in compact metric spaces, as the following example shows.
Let $ f(x)  =  x \sin x^{-1} $ if $ x \in (0,1] $, and $ f(0) = 0 $. One can easily verify that for any $ 0 < \alpha < 1 $, $ D^{+}_{\alpha, x} f = 0 $ for all $ x \in [0,1] $. However, $ \hol_{\alpha}f|_{[0,1] }= \infty $ if $ \alpha > \frac{1}{2} $. To see this, consider $ x_n = \frac{1}{2n\pi} $, $ y_n = \frac{1}{2n\pi+ \frac{\pi}{2}}  $.

\begin{lem}\label{lem:BTB}
	If $ M, N $ are length spaces and $ f: M \to N $ is uniformly continuous, then for any given bounded set $ A $ of $ X $, we have $ \diam(f(A)) \leq \Omega(\diam(A)) $, where $ \Omega: \mathbb{R}_+ \to \mathbb{R}_+ $.
\end{lem}
\begin{proof}
	Since $ f: M \to N $ is uniformly continuous, there is a $ \delta > 0 $ such that 
	\[
	d(f(z), f(x)) \leq 1
	\]
	provided $ d(z,x) \leq \delta $. Let $ x, y \in A $, $ n = [\frac{\diam(A)}{\delta}]+1 $. The same argument in \autoref{lem:H} shows that $ d(f(x),f(y)) \leq n $.
\end{proof}

\section{Construction of normal bundle charts from given bundle charts}\label{app:normal}

Consider the construction of a normal bundle chart from a given bundle chart $ (V, \phi) $ at $ m_0 $. Locally, this is quite easy as the following shows formally:
\[
\psi_{m}(x) = \phi_m(x) - \nabla_{m_0} \phi_{m_0}(x)(m - m_0).
\]
We need to give an explicit meaning to the above expression. It is essentially the same as the construction of \emph{parallel translation} for vector bundles, which we recall briefly.

Let $ M $ be a $ C^1 $ manifold, and $ (Y, N, \pi_2) $ a $ C^1 $ bundle with $ C^0 $ connection $ \mathcal{C} $.
Let $ f: M \to Y $ be a $ C^1 $ bundle map over $ u $, i.e., $ f(m) \in Y_{u(m)} $. Take a $ C^1 $ bundle chart $ \phi : U \times Y_{n_0} \to Y $ and $ \varphi = \phi^{-1} $. We say $ f $ is \emph{parallel} (with respect to $ \mathcal{C} $) if $ \nabla_{m}f(m) = 0 $ for all $ m \in M $. Consider its meaning in the local bundle chart. Let $ \varGamma_{(m',x')} $ be the Christoffel map in the local bundle chart $ \phi $.
Set $ g(m) = \varphi_{u(m)} (f(m)) : M \to Y_{n_0} $. Now by the chain rule (see \autoref{lem:chain}), we have
\[
Dg(m) = \nabla_{u(m)} \varphi_{u(m)} (f(m)) Du(m) = -D\varphi_{u(m)}(g(m))\varGamma_{(u(m), g(m))} Du(m).
\]
Set $ \widehat{\varGamma}_{(m,x)} = D\varphi_{m}(x)\varGamma_{(m, x)} $.

Let us construct the normal bundle chart from $ \phi $. The construction \emph{loses} smoothness.
Consider a special case. Using the local chart of $ N $ at $ n_0 $, we can identify $ U = T_{n_0}N $. Assume $ M = [-2,2] $, $ u(t) = (1-t)n_0 + t n $, $ n \in U $. For any $ x \in Y_{n_0} $, $ n \in U $, we will show that there is a unique $ f(t) \in Y_{u(t)} $ such that $ f(0) = x $, $ f'(t) \triangleq \nabla_{u(t)} f(t) \dot{u}(t) = 0 $, $ t \in M $, once we assume $ x \mapsto \widehat{\varGamma}_{(m,x)} $ is locally Lipschitz. (Note that $ U $ might be smaller around $ n_0 $ depending on $ x $ due to $ \widehat{\varGamma}_{(m,\cdot)} $ being locally Lipschitz.) In this case $ g(t) = \varphi_{u(t)}(f(t)) $. Now $ f'(t) = 0 $ is equivalent to
\begin{equation}\label{parallel}
g'(t) = -\widehat{\varGamma}_{(u(t),g(t))}(n-n_0).
\end{equation}
The above differential equation in $ Y_{n_0} $ has a unique solution $ g $ with $ g(0) = x $. Take
\[
\psi_n (x) = f(1): U \times Y_{n_0} \to Y_{n}.
\]
Next we show $ \psi_n $ is the desired normal bundle chart. Note that $ f(t) = \psi_{u(t)}(x) $. Since $ f'(t) = 0 $, in particular, one gets $ \nabla_{u(t)}\psi_{u(t)}(x) (n-n_0) = 0 $ for all $ n \in U $. For $ t = 0 $, it follows that $ \nabla_{n_0}\psi_{n_0}(x) (n-n_0) = 0 $, `yielding' $ \nabla_{n_0} \psi_{n_0}(x) = 0 $.
Since $ \mathcal{C} $ is only locally Lipschitz, $ x \mapsto \psi_{n}(x) $ is also only locally Lipschitz. In general, if $ \phi \in C^{k+1} $ and $ \mathcal{C} \in C^k $, then $ \psi \in C^k $.
By the construction, if $ Y $ is a vector bundle, $ \phi $ is a vector bundle chart and $ \mathcal{C} $ is a linear connection, then $ \psi $ is also a vector bundle chart. Altogether we obtain the following:

\begin{lem}[Existence of a normal bundle chart]
	Let $ M $ be a $ C^{2} $ manifold and $ M_1 \subset M $. Let $ (X,M,\pi) $ be a $ C^2 $ bundle with a $ C^2 $ bundle atlas $ \mathcal{A} $ and a $ C^1 $ connection $ \mathcal{C} $. (We can further assume each fiber of $ X $ is a Banach space or an open subset of a Banach space for simplicity.)
	\begin{enumerate}[(1)]
		\item If $ X $ is a vector bundle, $ \mathcal{C} $ is a linear connection, and $ (U, \phi) \in \mathcal{A} $ at $ m_0 $ is a vector bundle chart, then there is an open $ V \subset U $ and $ \psi: V \times X_{m_0} \to X $ such that $ (V, \psi) $ is a $ C^1 $ vector normal bundle chart at $ m_0 $.
		\item Take $ (U_{m_0}, \phi^{m_0}) \in \mathcal{A} $ at $ m_0 \in M_1 $. Suppose $ \chi_{m_0}: U_{m_0} \to T_{m_0} M $ is a $ C^2 $ chart with $ \chi(m_0) = 0 $ and $ V_{m_0} = \chi^{-1}_{m_0}(T_{m_0} M(\epsilon_{m_0})) \subset U_{m_0} $,  where $ \epsilon_{m_0} > 0 $. Assume the connection $ \mathcal{C} $ is uniformly Lipschitz around $ M_1 $ in the following sense:
		\[
		\lip \widehat{\varGamma}^{m_0}_{(m,\cdot)} \leq C_{m_0} < \infty, ~m \in V_{m_0},
		\]
		where $ C_{m_0} > 0 $ is a constant independent of $ m \in V_{m_0} $. Then we have $ \widehat{\varGamma}^{m_0}_{(m,x)} = (D\phi^{m_0}_{m}(x))^{-1} \varGamma^{m_0}_{(m, x)} $, where $ \varGamma^{m_0}_{(m, x)} $ is the Christoffel map of $ \mathcal{C} $ in $ \phi^{m_0} $. Now there are $ C^1 $ maps $ \psi^{m_0}: V_{m_0} \times X_{m_0} \to X $, $ m_0 \in M_1 $, that form a family of $ C^1 $ normal bundle charts on $ M_1 $, denoted by $ \mathcal{A}_0 $. Moreover, if $ X $ has $ \varepsilon $-almost uniform $ C^{0,1} $-fiber trivializations on $ M_1 $ with respect to $ \mathcal{A} $ (see \autoref{uniform lip bundle}), and $ \sup_{m_0 \in M_1} C_{m_0} < \infty $, $ \inf_{m_0 \in M_1}\epsilon_{m_0} > 0 $, then $ X $ also has $ C \varepsilon $-almost uniform $ C^{0,1} $-fiber trivializations on $ M_1 $ with respect to $ \mathcal{A}_0 $ where $ C > 0 $.
	\end{enumerate}
\end{lem}
\begin{proof}
	All we need is to consider \eqref{parallel}. Also, since $ \mathcal{C} $ is $ C^1 $ and $ \phi $ is $ C^2 $, one indeed sees that $ \psi $ is $ C^1 $, so the covariant derivative of $ \psi $ exists.
\end{proof}

\section{Bump functions and blid maps}\label{bump}

A \emph{bump function} on a Banach space $ X $ is a function with nonempty bounded support. The existence of a special bump function will indicate a special geometrical property of $ X $. As usual, one uses bump functions to truncate a map. A Lipschitz bump function always exists in every Banach space. In order to apply our regularity results in \autoref{stateRegularity} to the local case, a $ C^{0,1} \cap C^{1} $ or $ C^{1,1} $ bump function is needed (see e.g. \autoref{thm:fake1} or \autoref{thm:fake2}). In general Banach spaces, such bump functions may not exist. We list some classes of Banach spaces which have such bump functions; see \cite{DGZ93} for more details.

$ \bullet $ If a Banach space has a $ C^1 $ (resp. $ C^{1,1} $) norm away from the origin, then it admits a $ C^{0,1} \cap C^1 $ (resp. $ C^{1,1} $) bump function. For example (i) every Banach space $ X $ with separable dual has a $ C^1 $ norm away from the origin; (ii) the Hilbertian norm in each Hilbert space is a $ C^\infty \cap C^{k,1} $ norm away from the origin; (iii) the canonical norms on $ L^{p}(\mathbb{R}^n) $ (or $ l^p(\mathbb{R}^n) $), $ 1 \leq p < \infty $, are $ C^\infty \cap C^{k,1} $ away from $ 0 $ if $ p $ is even, $ C^{p-1,1} $ if $ p $ is odd, and $ C^{[p], p - [p]} $ if $ p $ is not an integer.

Smooth bump functions do not exist in the continuous space $ C[0,1] $, but this space is sometimes important in the study of differential equations such as delay equations and reaction-diffusion equations. In order to avoid using bump functions, following \cite{BR17} we introduce blid maps.

$ \bullet $ A $ C^{k,\alpha} $ \emph{blid map} for a Banach space $ X $ is a global \textbf{b}ounded \textbf{l}ocal \textbf{id}entity-at-zero $ C^{k,\alpha} $ map $ b_{\varepsilon}: X \to X $ where $ b_{\varepsilon}(x) = x $ if $ x \in X(\varepsilon) $.

Obviously, \textbf{(a)} any Banach space admitting $ C^{k,\alpha} $ bump functions also has a $ C^{k,\alpha} $ blid map. \textbf{(b)} $ C(\Omega) $ admits $ C^{k,1} $ blid maps ($ k = 0,1,\ldots,\infty $), where $ \Omega $ is a compact Hausdorff space; they are constructed in \cite{BR17} as
\[
b_{\varepsilon} (x) (t) = h(x(t))x(t), ~t \in \Omega,
\]
where $ h: \mathbb{R} \to [0,1] $ is a $ C^{\infty} $ bump function in $ \mathbb{R} $ with $ h(t) = t $ if $ |t| \leq \varepsilon $. \textbf{(c)} If $ X $ is complemented in $ C(\Omega) $ or is an ideal of $ C(\Omega) $, then $ X $ has a $ C^{k,1} $ blid map.

$ C^{k,\alpha} $ blid maps are used to truncate maps in the same way as bump functions do; for more details, see \cite{BR17}. So in this paper, one can use $ C^{k,\alpha} $ blid maps instead of $ C^{k,\alpha} $ bump functions where needed (in particular, in \autoref{thm:fake1} and \autoref{thm:fake2}), which makes our results applicable to the spaces listed in \textbf{(a)}--\textbf{(c)}.